\documentclass[a4paper, intlimits, reqno]{amsbook}
\DeclareRobustCommand{\SkipTocEntry}[5]{}

\pdfoutput=1 

\usepackage[english]{babel}
\usepackage[T1]{fontenc}
\usepackage[utf8]{inputenc}

\usepackage{amsmath}
\usepackage{amssymb}
\usepackage{MnSymbol}
\usepackage{amsthm}
\allowdisplaybreaks
\usepackage{amsfonts}
\usepackage{mathrsfs} 
\usepackage{mathtools}
\usepackage{nicefrac}
\usepackage{enumitem}
\usepackage{multicol}
\usepackage[hyphens]{url} 
\usepackage{hyperref}     
\usepackage[hyphenbreaks]{breakurl} 
\hypersetup{breaklinks=true} 
\usepackage{dsfont}
\usepackage[numbers]{natbib}
\usepackage{epsfig}
\usepackage{epstopdf}
\usepackage{prettyref}
\usepackage{doi}
\usepackage{appendix}
\usepackage{xcolor}
\usepackage{orcidlink}
\counterwithout{footnote}{chapter} 

\usepackage[record,index,abbreviations,symbols,toc=false]{glossaries-extra}
\usepackage{glossary-mcols}
\glsxtrsetgrouptitle{sets}{Sets and systems of sets}
\glsxtrsetgrouptitle{misc}{Miscellaneous}
\glsxtrsetgrouptitle{lchs_linear}{Locally convex Hausdorff spaces \& spaces of continuous linear operators}
\glsxtrsetgrouptitle{maps}{Maps}
\glsxtrsetgrouptitle{spaces}{Spaces of functions}

\GlsXtrLoadResources[
 src={sets},
 type=symbols,
 group={sets},
 set-widest=true,
 save-locations=true
]
\GlsXtrLoadResources[
 src={lchs_linear},
 type=symbols,
 group={lchs_linear},
 set-widest=true,
 save-locations=true
]
\GlsXtrLoadResources[
 src={spaces},
 type=symbols,
 group={spaces},
 set-widest=true,
 save-locations=true
]
\GlsXtrLoadResources[
 src={index},
 type=index,
 set-widest,
 save-locations=true
]
\GlsXtrLoadResources[
 src={maps},
 type=symbols,
 group={maps},
 save-locations=true
]
\GlsXtrLoadResources[
 src={misc},
 type=symbols,
 group={misc},
 save-locations=true
]

\newrefformat{def}{Definition \ref{#1}}
\newrefformat{rem}{Remark \ref{#1}}
\newrefformat{sect}{Section \ref{#1}}
\newrefformat{sub}{Section \ref{#1}}
\newrefformat{prop}{Proposition \ref{#1}}
\newrefformat{thm}{Theorem \ref{#1}}
\newrefformat{lem}{Lemma \ref{#1}}
\newrefformat{chap}{Chapter \ref{#1}}
\newrefformat{cor}{Corollary \ref{#1}}
\newrefformat{par}{Paragraph \ref{#1}}
\newrefformat{ex}{Example \ref{#1}}
\newrefformat{part}{Part \ref{#1}}
\newrefformat{fig}{Figure \ref{#1}}
\newrefformat{app}{Appendix \ref{#1}}
\newrefformat{que}{Question \ref{#1}}
\newrefformat{cond}{Condition \ref{#1}}
\newrefformat{conv}{Convention \ref{#1}}

\swapnumbers
\newtheoremstyle{dotless}{}{}{\itshape}{}{\bfseries}{}{}{}
\theoremstyle{dotless}
\theoremstyle{plain}
\numberwithin{section}{chapter}
\numberwithin{subsection}{section}
\newtheorem{thm}{Theorem}
\numberwithin{thm}{section}
\newtheorem{lem}[thm]{Lemma}
\newtheorem{prop}[thm]{Proposition}
\newtheorem{cor}[thm]{Corollary}
\theoremstyle{definition}
\newtheorem{defn}[thm]{Definition}
\newtheorem{rem}[thm]{Remark}
\newtheorem{exa}[thm]{Example}

\newtheorem{cond}[thm]{Condition}

\newtheorem{que}[thm]{Question}
\newtheorem*{que*}{Question}

\newcommand{\N} {\mathbb{N}}
\newcommand{\Z} {\mathbb{Z}}
\newcommand{\Q} {\mathbb{Q}}
\newcommand{\R} {\mathbb{R}}
\newcommand{\C} {\mathbb{C}}
\newcommand{\K} {\mathbb{K}}

\newcommand{\D} {\mathbb{D}}

\newcommand{\FV} {\mathcal{FV}(\Omega)}
\newcommand{\FVE} {\mathcal{FV}(\Omega,E)}
\newcommand{\Fv} {\mathcal{F}\nu(\Omega)}
\newcommand{\FvE} {\mathcal{F}\nu(\Omega,E)}
\newcommand{\acx} {\operatorname{acx}}
\newcommand{\oacx} {\overline{\operatorname{acx}}}
\newcommand{\F} {\mathcal{F}(\Omega)}
\newcommand{\FE} {\mathcal{F}(\Omega,E)}
\newcommand{\f} {F(\Omega)}
\newcommand{\fe} {F(\Omega,E)}
\newcommand{\Feps} {\mathcal{F}_{\varepsilon}\nu(\Omega,E)}

\DeclareMathOperator{\id}{id}
\DeclareMathOperator{\re}{Re}

\DeclareMathOperator{\dom}{dom}
\DeclareMathOperator{\Span}{span}
\providecommand{\differential}{\mathrm{d}}
\renewcommand{\d}{\differential}
\newcommand{\e}{\mathrm{e}}
\newcommand{\iu}{\mathrm{i}}

\newcommand\llim{
\mathchoice{\vcenter{\hbox{${\scriptstyle{-}}$}}}
{\vcenter{\hbox{$\scriptstyle{-}$}}}
{\vcenter{\hbox{$\scriptscriptstyle{-}$}}}
{\vcenter{\hbox{$\scriptscriptstyle{-}$}}}}
\newcommand\rlim{
\mathchoice{\vcenter{\hbox{${\scriptstyle{+}}$}}}
{\vcenter{\hbox{$\scriptstyle{+}$}}}
{\vcenter{\hbox{$\scriptscriptstyle{+}$}}}
{\vcenter{\hbox{$\scriptscriptstyle{+}$}}}}

\newcommand{\vertiii}[1]{{\left\vert\kern-0.25ex\left\vert\kern-0.25ex\left\vert #1 
    \right\vert\kern-0.25ex\right\vert\kern-0.25ex\right\vert}}  

\makeatletter
\newcommand{\fakephantomsection}{%
  \Hy@GlobalStepCount\Hy@linkcounter%
  \Hy@MakeCurrentHref{\@currenvir.\the\Hy@linkcounter}
  \Hy@raisedlink{\hyper@anchorstart{\@currentHref}\hyper@anchorend}%
}
\makeatother
  
\glsxtrRevertMarks

\hypersetup{
pdftitle={On vector-valued functions and the epsilon-product},
pdfsubject={MSC2020: Primary 46E10, 46E15, 46E40; Secondary 35A01, 35B30, 46A03, 46A32, 46A63},
pdfauthor={Karsten Kruse, ORCID: 0000-0003-1864-4915},
pdfkeywords={vector-valued functions, weighted space, epsilon-product, linearisation, tensor product, parameter dependence, extension, weak-strong principle, Blaschke, Wolff, Schauder decomposition, series representation, sequence space}
}

\begin{document}
\begin{titlepage}
\begin{center}
\includegraphics[scale=0.25]{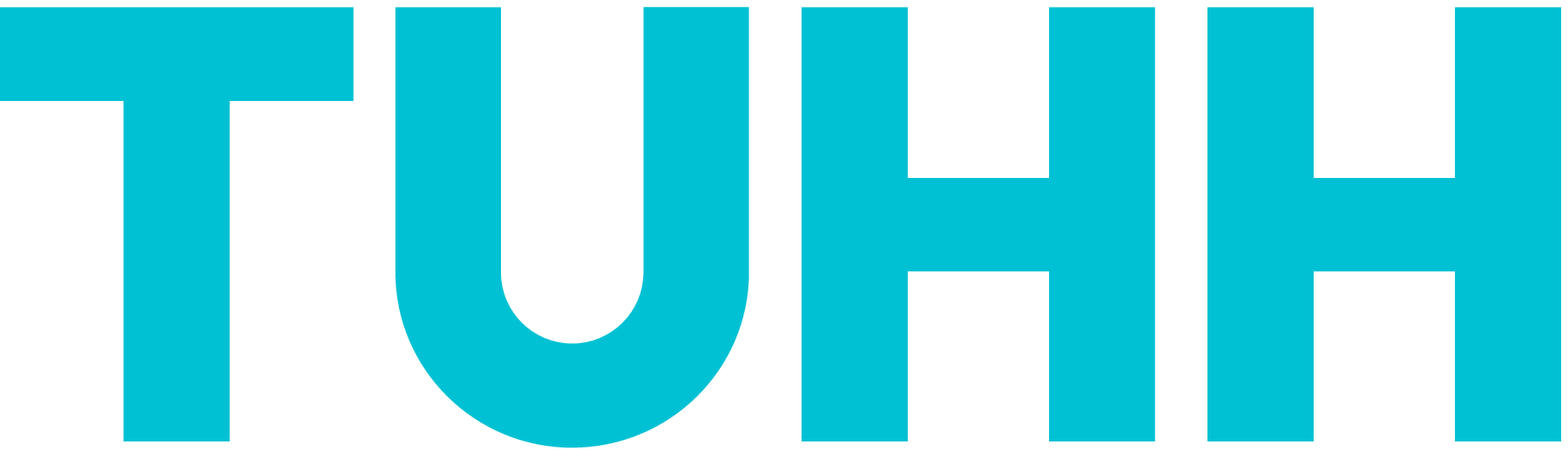}
\end{center}
\vspace{2cm}
\begin{center}
\huge{{\textbf{On vector-valued functions and the $\varepsilon$-product}}}
\end{center}
\vspace{4cm}
\begin{center}
\large{{\textbf{Habilitationsschrift}}}
\end{center}
\vspace{4cm}
\normalsize
\begin{center}
vorgelegt am 31.01.2022\\
der Technischen Universit\"at Hamburg\\
von\\
Dr.~rer.~nat.~Karsten Kruse,\\
geboren am 19.11.1984 in Papenburg.
\end{center}
\vspace{2cm}
\begin{center}
Die Habilitationsschrift wurde in der Zeit von Juli 2020 bis Januar 2022 im Institut f\"ur
Mathematik der Technischen Universit\"at Hamburg angefertigt.
\end{center}
\end{titlepage}
\thispagestyle{empty}
\noindent Gutachter: Prof.~Jos\'e Bonet\,\orcidlink{0000-0002-9096-6380}\\ 
\phantom{Gutachter: }Prof.~Dr.~Leonhard Frerick \\
\phantom{Gutachter: }Prof.~Dr.~Thomas Kalmes\,\orcidlink{0000-0001-7542-1334}\\
\phantom{Gutachter: }PD Dr.~Christian Seifert\,\orcidlink{0000-0001-9182-8687} 
\vskip 5mm
\noindent eingereicht: 31.~Januar 2022; \"uberarbeitet: 30.~Januar 2023
\vskip 5mm
\noindent Tag des Habilitationskolloquiums: 01.~Juli 2022

\vfill
\noindent DOI: \href{https://doi.org/10.15480/882.4898}{10.15480/882.4898}\\
ORCID: \orcidlink{0000-0003-1864-4915} 0000-0003-1864-4915
\vskip 5mm
\noindent Creative Commons Lizenz:\\
Diese Arbeit steht unter der Creative Commons Lizenz Namensnennung 4.0 (CC BY 4.0). 
Das bedeutet, dass sie vervielf\"altigt, verbreitet und \"offentlich zug\"anglich gemacht 
werden darf, auch kommerziell, sofern dabei stets der Urheber, die Quelle des Textes
und o.~g.~Lizenz genannt werden. Die genaue Formulierung der Lizenz kann unter 
\url{https://creativecommons.org/licenses/by/4.0/legalcode.de} auf\-gerufen werden.

\setcounter{page}{0}
\setcounter{page}{2}

\chapter*{Acknowledgement}
It is quite hard to express how grateful I am to the people who helped, in one way or the 
other, to finish this thesis which spans a part of my work between 2016 and 2022. 
But I will give it a try. 

First of all, I am deeply indebted to Marko Lindner and Christian Seifert 
who always supported and encouraged me since I joined the TUHH in 2014 and gave me a home 
so that I could work on the kind of mathematics I love. 
I know how lucky I was to meet you both. 
Second, I would like to thank the two people who taught me probably the most I know about
complex analysis and functional analysis, namely, Andreas Defant and Michael Langenbruch
(Oldenburg). Third, I am utterly grateful to Jos\'e Bonet and Enrique Jord\'a 
(Valencia) for many helpful suggestions and comments, improving some of the papers 
this thesis is based on, as well as enduring the quite abstract setting. 

Let us come to the honorable mentions, $\ldots$ just kidding. 
In 2015 I was lucky again because Jan Meichsner joined the TUHH as a PhD student. 
Despite him being a physicist and a dispraiser of green cabbage, it was a real pleasure 
to share an office, do mathematics or just spend time with him.  
Further, let me thank Dennis Gallaun with whom I spent a lot of effort and gaffer tape 
setting up the mobile e-assessment center at the TUHH between 2018 and 2020. 

Apart from the people mentioned above I would like to thank my other mathe\-matical co-authors 
Hans Daduna, Ruslan Krenzler, Felix Schwenninger and Lin Xie whose work is not physically present 
in this thesis but whose mathematical influence or spirit probably is. 
Furthermore, I am thankful to the whole Institute of Mathematics of the TUHH, in particular, 
the group of Applied Analysis for their support. Moreover, I am grateful to the reviewers of this thesis 
Leonhard Frerick (Trier) and Thomas Kalmes (Chemnitz) besides Jos\'e Bonet and Christian Seifert, 
and the anonymous reviewers of the papers it is based on for their work, helpful comments and corrections.   

Finally, I would like to thank my family for their continuous support and my love Sonja for sharing my 
mathematical interests, bearing my kind of humour and so much more which I cannot put into words.

These words only roughly express my gratitude and I hope that the minimum that remains after 
reading the acknowledgement is the thought `At least, he gave it a try.' and a smile. 
\chapter*{Abstract}
This habilitation thesis centres on linearisation of vector-valued functions which means that vector-valued
functions are represented by continuous linear operators. The first question we face is which vector-valued 
functions may be represented by continuous linear operators. 
We study this problem in the framework of $\varepsilon$-products and give sufficient conditions 
in \prettyref{chap:linearisation} and \ref{chap:consistency} when a space $\FE$ of vector-valued functions 
on a set $\Omega$ coincides (up to an isomorphism) with the $\varepsilon$-product $\F\varepsilon E$ 
of a corresponding space of scalar-valued functions $\F$ and the codomain $E$ 
which is usually an infinite-dimensional locally convex Hausdorff space. 
The $\varepsilon$-product $\F\varepsilon E$ is a space of continuous linear operators from the dual space $\F'$ 
to $E$. 

Once we have a representation of a space $\FE$ of vector-valued functions by an 
$\varepsilon$-product $\F\varepsilon E$, we have access to the rich theory of continuous linear operators 
which allows us to lift results that are known for the scalar-valued case to the vector-valued case. 
The whole \prettyref{chap:applications}, which spans more than half of this thesis, is dedicated to this 
lifting mechanism. But we should point out that this is not only about transferring results from the scalar-valued 
to the vector-valued case. The results in the vector-valued case encode additional information for the 
scalar-valued case as well, e.g.\ we may deduce from the solvability of a linear partial differential equation 
in the vector-valued case affirmative answers on the parameter dependence of solutions in the scalar-valued case 
(see \prettyref{sect:lifting}). 

In \prettyref{sect:extension} we give a unified approach to handle the problem of extending  
functions with values in $E$, which have weak extensions in $\F$, 
to functions in the vector-valued counterpart $\FE$ of $\F$. We present different extension theorems depending 
on the topological properties of the spaces $\F$ and $E$. These theorems also cover weak-strong principles. 
In particular, we study weak-strong principles for continuously partially differentiable functions of finite order 
in \prettyref{sect:weak_strong_finite_order} and improve the well-known weak-strong principles of Grothendieck 
and Schwartz.   
We use our results on the extension of vector-valued functions to derive Blaschke's convergence theorem for 
several spaces of vector-valued functions and Wolff's theorem for the description of dual spaces of several 
function spaces $\F$ in \prettyref{sect:blaschke} and \ref{sect:wolff_type}. 
Starting from the observation that every scalar-valued holomorphic function has a local power series expansion 
and that this is still true for holomorphic functions with values in $E$ if $E$ is locally complete, we 
develop a machinery which is based on linearisation and Schauder decomposition to transfer known series 
expansions from scalar-valued to vector-valued functions in \prettyref{sect:schauder}. 
Especially, we apply this machinery to derive Fourier expansions for $E$-valued Schwartz functions 
and $\mathcal{C}^{\infty}$-smooth functions on $\R^{d}$ that 
are $2\pi$-periodic in each variable. The last section of \prettyref{chap:applications} is devoted 
to the representation of spaces $\FE$ of vector-valued functions by sequence spaces, 
which can be used to identify the coefficient spaces of the series expansions from the preceding section, if 
one knows the coefficient space in the scalar-valued case. Furthermore, we give several new conditions on the 
Pettis-integrability of vector-valued functions in \prettyref{app:pettis}, which are, for instance, 
needed for the Fourier expansions in \prettyref{sect:schauder}.
\chapter*{Kurzfassung}
Im Mittelpunkt dieser Habilitationsschrift steht die Linearisierung vektorwertiger Funktionen, d.~h.\ 
vektorwertige Funktionen sollen durch stetige lineare Opera\-toren dargestellt werden. 
Die erste Frage, der man sich stellen muss, ist, welche vektorwertigen 
Funktionen durch stetige lineare Operatoren dargestellt werden k\"onnen. 
Wir untersuchen dieses Problem im Rahmen von $\varepsilon$-Produkten und geben hinreichende Bedingungen 
in Kapitel \ref{chap:linearisation} und \ref{chap:consistency} an, wann ein Raum $\FE$ 
von vektorwertigen Funktionen auf einer Menge $\Omega$ mit dem $\varepsilon$-Produkt $\F\varepsilon E$ 
eines entsprechenden Raums skalarwertiger Funktionen $\F$ und 
des Wertebereichs $E$ (bis auf Isomorphie) \"ubereinstimmt. Hierbei ist $E$ 
\"ublicherweise ein unendlich-dimensionaler lokalkonvexer Hausdorff Raum. 
Das $\varepsilon$-Produkt $\F\varepsilon E$ ist ein Raum stetiger linearer Operatoren, die vom Dualraum $\F'$ 
nach $E$ abbilden. 

Sobald wir eine Darstellung eines Raums $\FE$ von vektorwertigen Funktionen durch ein 
$\varepsilon$-Produkt $\F\varepsilon E$ gewonnen haben, ist es uns m\"oglich die reich\-haltige Theorie 
der stetigen linearen Operatoren zu nutzen, die es uns erlaubt, Ergebnisse, die f\"ur den skalarwertigen Fall 
bekannt sind, auf den vektorwertigen Fall zu \"ubertragen. 
Das gesamte Kapitel \ref{chap:applications}, das mehr als die H\"alfte dieser Arbeit einnimmt, widmet sich diesem 
\"Ubertragungsmechanismus. Es sei jedoch darauf hingewiesen, 
dass es hier nicht nur um die \"Ubertragung von Ergebnissen aus dem skalarwertigen 
auf den vektorwertigen Fall geht. Die Ergebnisse im vektorwertigen Fall beinhalten auch zus\"atzliche Informationen 
f\"ur den skalarwertigen Fall, z.~B.\ k\"onnen wir aus der L\"osbarkeit einer linearen partiellen Differentialgleichung
im vektorwertigen Fall Antworten auf die Frage nach der Parameterabh\"angigkeit der L\"osungen im skalarwertigen Fall 
ableiten (siehe Abschnitt \ref{sect:lifting}). 

In Abschnitt \ref{sect:extension} stellen wir einen einheitlichen Ansatz zur L\"osung des
Fortsetzungsproblems von Funktionen mit Werten in $E$, die schwache Fortsetzungen in $\F$ haben, 
zu Funktionen im vektorwertigen Gegenst\"uck $\FE$ von $\F$ vor. 
Wir pr\"asentieren verschiedene Fortsetzungss\"atze in Abh\"angigkeit von den topologi\-schen Eigenschaften der 
R\"aume $\F$ und $E$. Diese S\"atze decken auch schwach-stark Prinzipien ab. 
Insbesondere untersuchen wir schwach-stark Prinzipien f\"ur endlich oft stetig partiell differenzierbare Funktionen 
in Abschnitt \ref{sect:weak_strong_finite_order} und verbes\-sern die bekannten schwach-starken Prinzipien 
von Grothendieck und Schwartz.   
Zudem leiten wir von unseren Ergebnissen zur Fortsetzung vektorwertiger Funktionen den Konvergenzsatz von Blaschke 
f\"ur diverse R\"aume vektorwertiger Funktionen ab und \"ubertragen den Satz von Wolff auf Dualr\"aume mehrerer 
Funktionenr\"aume $\F$ in den Abschnitten \ref{sect:blaschke} und \ref{sect:wolff_type}. 
Ausgehend von der Beobachtung, dass jede skalar\-wer\-tige holomorphe Funktion eine lokale Potenzreihenentwicklung hat 
und dass dies auch f\"ur holomorphe Funktionen mit Werten in $E$ gilt, wenn $E$ lokal vollst\"andig ist,  
entwickeln wir einen Mechanismus, der auf Linearisierung und Schauder-Zerlegung basiert, 
um in Abschnitt \ref{sect:schauder} bekannte Reihenentwicklungen von skalarwertigen auf vektorwertige Funktionen 
zu erweitern. 
Insbesondere wenden wir diesen Mechanismus an, um Fourier-Entwicklungen f\"ur $E$-wertige 
Schwartz-Funktionen und $\mathcal{C}^{\infty}$-glatte Funktionen auf $\R^{d}$, die 
$2\pi$-periodisch in jeder Variablen sind, zu erhalten. 
Der letzte Abschnitt von Kapitel \ref{chap:applications} ist der Darstellung von R\"aumen $\FE$ 
vektorwertiger Funktionen durch Folgenr\"aume gewidmet, 
was man dazu nutzen kann, die Koeffizientenr\"aume der Reihenentwicklungen aus dem vorangegangenen Abschnitt 
zu bestimmen, sofern man den Koeffizientenraum im skalarwertigen Fall kennt. 
Au{\ss}erdem legen wir mehrere neue Bedingungen f\"ur die 
Pettis-Integrierbarkeit von vektorwertigen Funktionen in Anhang \ref{app:pettis} dar, 
die z.~B.\ f\"ur die Fourier-Entwicklungen in Abschnitt \ref{sect:schauder} ben\"otigt werden.
\setcounter{tocdepth}{2}
\tableofcontents
\chapter{Introduction}
\label{chap:intro}
This work is dedicated to a classical topic, namely, the linearisation of weighted spaces of vector-valued functions. 
The setting we are interested in is the following. Let $\F$ be a locally convex Hausdorff space of 
functions from a non-empty set $\Omega$ to a field $\K$ and $E$ be a locally convex Hausdorff space over $\K$. 
The $\varepsilon$-product of $\F$ and $E$ is defined as the space of linear continuous operators
\[
\F\varepsilon E:=L_{e}(\F_{\kappa}',E)
\]
equipped with the topology of uniform convergence on equicontinuous subsets of the dual $\F'$ which 
itself is equipped with the topology of uniform convergence on absolutely convex compact subsets of $\F$.
Suppose that the point-evaluation functionals $\delta_{x}$, $x\in\Omega$, belong to $\F'$ and that 
there is a locally convex Hausdorff space $\FE$ of $E$-valued functions 
on $\Omega$ such that the map 
\begin{equation}\label{eq:intro_0}
S\colon \F\varepsilon E \to \FE,\; u\longmapsto [x\mapsto u(\delta_{x})],
\end{equation}
is well-defined. 
The main question we want to answer reads as follows. When is $\F\varepsilon E$ a 
linearisation of $\FE$, i.e.\ when is $S$ an isomorphism?

In \cite{B3,B1,B2} Bierstedt treats the space $\mathcal{CV}(\Omega,E)$ 
of continuous functions on a completely regular Hausdorff space $\Omega$ weighted with 
a Nachbin-family $\mathcal{V}$ and its topological subspace $\mathcal{CV}_{0}(\Omega,E)$ of functions 
that vanish at infinity in the weighted topology. He derives sufficient conditions on $\Omega$, $\mathcal{V}$ and 
$E$ such that the answer to our question is affirmative, i.e.\ $S$ is an isomorphism. 
Schwartz answers this question for several weighted spaces of $k$-times continuously partially differentiable 
functions on $\Omega=\R^{d}$ like the Schwartz space in \cite{Schwartz1955,Sch1} for quasi-complete $E$
with regard to vector-valued distributions. 
Grothendieck treats the question in \cite{Gro}, mainly for nuclear $\F$ and complete $E$. 
In \cite{Kom7,Kom8,Kom9} Komatsu gives a positive answer for ultradifferentiable 
functions of Beurling or Roumieu type and sequentially complete $E$ with regard to vector-valued ultradistributions.
For the space of $k$-times continuously partially differentiable functions on open subsets $\Omega$ of 
infinite dimensional spaces equipped with the topology of uniform convergence of all partial derivatives 
up to order $k$ on compact subsets of $\Omega$ sufficient conditions for an affirmative answer are deduced 
by Meise in \cite{meise1977}. For holomorphic functions on open subsets of infinite dimensional spaces 
a positive answer is given in \cite{dineen1981} by Dineen. 
Bonet, Frerick and Jord{\'a} show in \cite{B/F/J} that $S$ is an isomorphism for 
certain closed subsheaves of the sheaf $\mathcal{C}^{\infty}(\Omega,E)$ of smooth functions on an open subset 
$\Omega\subset\R^{d}$ with the topology of uniform convergence of all partial derivatives on compact subsets 
of $\Omega$ and locally complete $E$ which, in particular, covers the spaces of harmonic and holomorphic functions. 

An important application of linearisation is within the field of partial differential equations. 
Let $E$ be a linear space of functions on a set $U$ and 
$P(\partial)\colon\mathcal{C}^{\infty}(\Omega)\to\mathcal{C}^{\infty}(\Omega)$ a 
linear partial differential operator with $\mathcal{C}^{\infty}$-smooth coefficients 
where $\mathcal{C}^{\infty}(\Omega):=\mathcal{C}^{\infty}(\Omega,\K)$. 
We call the elements of $U$ parameters and say that a family 
$(f_{\lambda})_{\lambda\in U}$ in $\mathcal{C}^{\infty}(\Omega)$ depends on a parameter
w.r.t.\ $E$ if the map $\lambda\mapsto f_{\lambda}(x)$ is an element of $E$ for every $x\in\Omega$. 
The question of parameter dependence is whether for every family 
$(f_{\lambda})_{\lambda\in U}$ in $\mathcal{C}^{\infty}(\Omega)$ depending on a parameter w.r.t.\ $E$ 
there is a family $(u_{\lambda})_{\lambda\in U}$ in $\mathcal{C}^{\infty}(\Omega)$ 
with the same kind of parameter dependence 
which solves the partial differential equation
\[
P(\partial)u_{\lambda}=f_{\lambda},\quad \lambda\in U.
\]
In particular, it is the question of $\mathcal{C}^{k}$-smooth (holomorphic, distributional, etc.) 
parameter dependence if $E$ is the space $\mathcal{C}^{k}(U)$ of $k$-times 
continuously partially differentiable functions 
on an open set $U\subset\R^{d}$ (the space $\mathcal{O}(U)$ of holomorphic functions on an open set $U\subset\C$, 
the space of distributions $\mathcal{D}(V)'$ on an open set $V\subset\R^{d}$ where $U=\mathcal{D}(V)$, etc.). 
The question of parameter dependence w.r.t.\ $E$ has an affirmative answer 
for several locally convex Hausdorff spaces $E$ due to 
tensor product techniques and splitting theory. Indeed, the answer is affirmative if
the topology of $E$ is stronger than the topology of pointwise convergence on $U$ and 
\[
P(\partial)^{E}\colon\mathcal{C}^{\infty}(\Omega,E)\to\mathcal{C}^{\infty}(\Omega,E)
\] 
is surjective where $P(\partial)^{E}$ is the version of $P(\partial)$ for $E$-valued functions. 
The operator $P(\partial)^{E}$ is surjective if its version $P(\partial)$ for scalar-valued functions 
is surjective, for instance, if $P(\partial)$ is elliptic, and $E$ is a Fr\'echet space. 
This is a consequence of Grothendieck's theory of tensor products \cite{Gro}, 
the nuclearity of $\mathcal{C}^{\infty}(\Omega)$ and the 
isomorphism $\mathcal{C}^{\infty}(\Omega,E)\cong\mathcal{C}^{\infty}(\Omega)\varepsilon E$ 
for locally complete $E$. 
Thanks to the splitting theory of Vogt for Fr\'{e}chet spaces \cite{vogt1983} and of 
Bonet and Doma\'nski for PLS-spaces \cite{D/L} we even have in case of an elliptic $P(\partial)$ 
that $P(\partial)^{E}$ for $d>1$ is surjective if $E:=F_{b}'$ where $F$ is a Fr\'{e}chet space satisfying 
the condition $(DN)$ or if $E$ is an ultrabornological PLS-space having the property $(PA)$ 
since $\operatorname{ker}P(\partial)$ has the property $(\Omega)$. 
In particular, these three results cover the cases that $E =\mathcal{C}^{k}(U)$, $\mathcal{O}(U)$ 
or $\mathcal{D}(V)'$. Of course, this technique to answer the question of parameter dependence 
is not restricted to linear partial differential operators or the space $\mathcal{C}^{\infty}(\Omega)$.

Another application of linearisation lies in the problem of extending a vector-valued 
function $f\colon \Lambda\to E$ from a subset $\Lambda\subset\Omega$ to a locally convex Hausdorff space $E$ 
if the scalar-valued functions $e'\circ f$ are extendable for each continuous linear functional 
$e'$ from certain linear subspaces $G$ of $E'$ under 
the constraint of preserving the properties, like holomorphy, of the scalar-valued extensions. 
This problem was considered, among others, by Grothendieck \cite{Grothendieck1953,Gro}, Bierstedt \cite{B2}, 
Gramsch \cite{Gramsch1977}, Grosse-Erdmann \cite{grosse-erdmann1992,grosse-erdmann2004}, 
Arendt and Nikolski \cite{Arendt2016,Arendt2000,Arendt2006}, 
Bonet, Frerick, Jord\'a and Wengenroth \cite{B/F/J,F/J,F/J/W,jorda2005,jorda2013}.
Even the simple case $\Lambda=\Omega$ and $G=E'$ is interesting and an affirmative answer is called a 
weak-strong principle. 

Our goal is to give a unified and flexible approach to linearisation 
which is able to handle new examples and covers the already known examples. 

\addtocontents{toc}{\SkipTocEntry}
\section*{Organisation of this thesis}

After fixing some notions and preliminaries on locally convex Hausdorff spaces, 
continuous linear operators and continuously partially differentiable functions 
in \textbf{\prettyref{chap:notation}}, we study the problem of linearisation in 
\textbf{\prettyref{chap:linearisation}}. In \prettyref{sect:eps-prod_into} we introduce 
our standard example of spaces $\FE$ that we consider. Namely, 
spaces of functions $\FVE$ from $\Omega$ to $E$ which are subspaces of sections of domains of linear operators 
$T^{E}$ on $E^{\Omega}$, and whose topology is generated by a family of weight functions $\mathcal{V}$.
These spaces cover many examples of classical spaces of functions appearing in analysis like the mentioned ones 
and an example of the operators $T^{E}$ are the partial derivative operators. Then we exploit the structure of 
our spaces to describe a sufficient condition, which we call consistency, 
on the interplay of the pairs of operators $(T^{E},T^{\K})$ and the map $S$ 
such that $S$ becomes an isomorphism into, i.e.\ an isomorphism 
to its range (see \prettyref{thm:linearisation}). 

In \prettyref{sect:eps-prod} 
we tackle the problem of surjectivity of $S$. In our main \prettyref{thm:full_linearisation} 
and its \prettyref{cor:full_linearisation} we give several sufficient conditions on the pairs of operators 
$(T^{E},T^{\K})$ and the spaces involved such that $S\colon\FV\varepsilon E \to \FVE$ is an isomorphism. 
Looking at the pair of partial differential operators $(P(\partial)^{E},P(\partial))$ considered above, 
these conditions allow us to express $P(\partial)^{E}$ as 
$P(\partial)^{E}=S\circ (P(\partial)\varepsilon\id_{E})\circ S^{-1}$ where 
$P(\partial)\varepsilon\id_{E}$ is the $\varepsilon$-product of $P(\partial)$ and the identity $\id_{E}$ on $E$. 
Hence it becomes obvious that the surjectivity of $P(\partial)^{E}$ is equivalent to the surjectivity of 
$P(\partial)\varepsilon\id_{E}$. This is used in \cite{ich,kruse2019_5,kruse2019_1,kruse2018_5,kruse2019_2}
in the case of the Cauchy--Riemann operator $P(\partial)=\overline{\partial}$ on spaces of smooth functions 
with exponential growth. 

In \textbf{\prettyref{chap:consistency}} we take a closer look at the notion of consistency 
of $(T^{E},T^{\K})$. In \prettyref{sect:consistency} we characterise several properties of the
functions $S(u)$ for $u\in\FV\varepsilon E$ that are inherited from the elements 
of $\FV$. 

\prettyref{sect:examples} is devoted to several concrete examples of 
spaces of vector-valued functions that may be linearised by $S$ and which we use for our 
applications in the forthcoming sections and chapters. 

In \prettyref{sect:riesz_markov_kakutani} we answer in several cases the question 
whether given a continuous linear functional $T^{\K}$ on $\F$ there is always 
a continuous linear operator $T^{E}$ on $\FE$ such that $(T^{E},T^{\K}$) is consistent. 
This is closely related to Riesz--Markov--Kakutani theorems for $T^{\K}$, 
which we transfer to the vector-valued case.  
 
\textbf{\prettyref{chap:applications}} is dedicated to applications of linearisation. 
In \prettyref{sect:lifting} we come back to our problem of parameter dependence. We show in our main 
\prettyref{thm:eps_prod_surj_inj} of this section how to use 
linearisations to transfer properties like injectivity, surjectivity or bijectivity from a map 
$T^{\K}\colon \mathcal{F}_{1}(\Omega_{1})\to \mathcal{F}_{2}(\Omega_{2})$ to the corresponding map 
$T^{E}\colon \mathcal{F}_{2}(\Omega_{1},E)\to \mathcal{F}_{2}(\Omega_{2},E)$ if the pair $(T^{E},T^{\K})$ is 
consistent under suitable assumptions on the spaces involved. Besides the problem of parameter dependence 
for (hypo)elliptic linear partial differential operators (see \prettyref{cor:surjectivity_hypo_elliptic}), we deduce 
a vector-valued version of the Borel--Ritt theorem (see \prettyref{thm:Borel_Ritt}) from this main theorem 
and give sufficient conditions under which the Fourier transformation 
$\mathfrak{F}^{\C}\colon \mathcal{S}_{\mu}(\R^{d})\to\mathcal{S}_{\mu}(\R^{d})$ on the Beurling--Bj\"orck space 
is still an isomorphism in the vector-valued case and may be decomposed as 
$\mathfrak{F}^{E}=S\circ (\mathfrak{F}^{\C}\varepsilon\id_{E})\circ S^{-1}$ (see \prettyref{thm:Bjoerck_Fourier}).

In \prettyref{sect:extension} we present a general approach to the extension problem considered above 
for a large class of function spaces $\FE$ if the map $S$ is an isomorphism into. 
The spaces we treat are of the kind 
that $\F$ belongs to the class of semi-Montel, Fr\'echet--Schwartz or Banach spaces, 
or that $E$ is a semi-Montel space. Apart from linearisation and consistency, the main ingredient of this approach 
is to view the set $\Lambda\subset\Omega$ from which we want to extend our functions as a set of functionals 
$\{\delta_{x}\;|\;x\in\Lambda\}$. This view allows us to generalise the extension problem 
in \prettyref{que:weak_strong} by swapping this set of functionals by other functionals, 
which opens up new possibilities in applications that we explore in \prettyref{sect:weak_strong_finite_order},
\prettyref{sect:blaschke}, \prettyref{sect:wolff_type} and \prettyref{sect:sequence_space}.
In the extension problem we always have to balance the sets $\Lambda$ 
from which we extend our functions and the subspaces $G\subset E'$ 
with which we test. The case of `thin' sets $\Lambda$ and `thick' subspaces $G$ 
is handled 
in \prettyref{sub:thin_0} with main theorems \prettyref{thm:ext_F_semi_M}, \prettyref{thm:ext_FS_set_uni} and 
\prettyref{thm:ext_B_unique}, the converse case of `thick' sets $\Lambda$ and `thin' subspaces $G$ 
is handled in \prettyref{sub:thick} with main theorems \prettyref{thm:fix_topo_E_semi_M},
\prettyref{thm:ext_FS_fix_top} and \prettyref{thm:ext_B_fix_top}.

In \prettyref{sect:weak_strong_finite_order} we consider weak-strong principles for continuously 
partially differentiable functions of finite order. For locally complete $E$ it is well-known that 
a function $f$ belongs to $\mathcal{C}^{\infty}(\Omega,E)$ if and only if $e'\circ f\in\mathcal{C}^{\infty}(\Omega)$ 
for all $e'\in E'$ (see e.g.\ \cite[Theorem 9, p.\ 232]{B/F/J}). 
If $k\in\N_{0}$, then it is still true that $f\in\mathcal{C}^{k}(\Omega,E)$ implies 
$e'\circ f\in\mathcal{C}^{k}(\Omega)$ for all $e'\in E'$. 
But the converse is not true anymore. Only a weaker version of this weak-strong principle 
holds which is due to Grothendieck \cite{Grothendieck1953} and Schwartz \cite{Schwartz1955} 
(see \prettyref{thm:schwartz_weak_strong}). 
Namely, if $k\in\N_{0}$, $E$ is sequentially complete and $f\colon\Omega\to E$ is such that 
$e'\circ f\in\mathcal{C}^{k+1}(\Omega)$ for all $e'\in E'$, then $f\in\mathcal{C}^{k}(\Omega,E)$. 
Using the results from \prettyref{sect:extension}, we improve this weaker version of the weak-strong principle 
by allowing $E$ to be locally complete, only testing with less functionals from certain linear subspaces 
$G\subset E'$ and getting that $f$ does not only belong to $\mathcal{C}^{k}(\Omega,E)$ 
but that all partial derivatives of order $k$ are actually locally Lipschitz continuous 
(see \prettyref{cor:ext_B_unique_loc_Hoelder}). If we restrict to semi-Montel spaces $E$, 
then even a `full' weak-strong principle \prettyref{thm:weak_strong_finite_order} 
holds as in the $\mathcal{C}^{\infty}$-case.

In \prettyref{sect:blaschke} we derive vector-valued Blaschke theorems like \prettyref{cor:Blaschke_vector_valued} 
for several function spaces. This generalises 
results of Arendt and Nikolski \cite{Arendt2000} for bounded holomorphic functions 
and Frerick, Jord\'a and Wengenroth \cite{F/J/W} for bounded functions in the kernel of a 
hypoelliptic linear partial differential operator. These are results of the form: 
given a bounded net $(f_{\iota})_{\iota\in I}$ in some space 
$\mathcal{F}_{1}(\Omega,E)$ of Banach-valued functions which converges pointwise on a certain subset of 
$\Omega$ there is a limit $f\in\mathcal{F}_{1}(\Omega,E)$ of this net w.r.t.\ a weaker topology 
of a linear superspace $\mathcal{F}_{2}(\Omega,E)$ of $\mathcal{F}_{1}(\Omega,E)$. 
In Blaschke's classical convergence theorem \cite[Theorem 7.4, p.\ 219]{burckel1979} 
we have $E=\C$, $\mathcal{F}_{1}(\Omega,E)$ is the space of bounded holomorphic functions on the open 
unit disc $\D\subset\C$, $\mathcal{F}_{2}(\Omega,E)$ is the space of holomorphic functions on 
$\D$ and the weaker topology is the topology of compact convergence. 

In \prettyref{sect:wolff_type} we present Wolff type descriptions of the dual space of several function spaces 
$\F$ using linearisation (see \prettyref{thm:wolff}). 
Wolff's theorem \cite[p.\ 1327]{wolff1921} (cf.\ \cite[Theorem (Wolff), p.\ 402]{grosse-erdmann2004}) 
phrased in a functional analytic way (see \cite[p.\ 240]{F/J/W}) says: if $\Omega\subset\C$
is a domain, then for each $\mu\in\mathcal{O}(\Omega)'$ there are a sequence 
$(z_{n})_{n\in\N}$ which is relatively compact in $\Omega$ and a sequence
$(a_{n})_{n\in\N}$ in the space $\ell^{1}$ of absolutely summable sequences 
such that $\mu=\sum_{n=1}^{\infty}a_{n}\delta_{z_{n}}$.

In \prettyref{sect:schauder} we derive a general result for Schauder decompositions of the $\varepsilon$-product 
$F\varepsilon E$ for locally convex Hausdorff spaces $F$ and $E$ if $F$ has an equicontinuous Schauder basis 
(see \prettyref{thm:schauder_decomp}). In combination with linearisation and consistency this can be used 
for $F=\F$ to lift series representations like the power series expansion of holomorphic functions 
from scalar-valued functions to vector-valued functions (see \prettyref{cor:schauder_decomp}). 
We present several examples in \prettyref{sub:Schauder_examples}, for instance, 
Fourier expansions in the Schwartz space $\mathcal{S}(\R^{d},E)$ and 
in the space $\mathcal{C}^{\infty}_{2\pi}(\R^{d},E)$ of functions in 
$\mathcal{C}^{\infty}(\R^{d},E)$ 
that are $2\pi$-periodic in each variable.
In particular, we combine these expansions for locally complete $E$ with the results from \prettyref{sect:lifting} 
to identify the coefficient spaces 
of the Fourier expansions in $\mathcal{S}(\R^{d},E)$ and 
$\mathcal{C}^{\infty}_{2\pi}(\R^{d},E)$ (see \prettyref{thm:fourier.rap.dec} and 
\prettyref{thm:fourier_periodic}). 

In \prettyref{sect:sequence_space} an application of our extension results from \prettyref{sect:extension} 
is given to represent function spaces $\FE$ by sequence spaces if one knows such a representation for $\F$ 
(see  \prettyref{thm:Schauder_coeff_space}).
As examples we treat the space $\mathcal{O}(\D_{R}(0),E)$ of $E$-valued holomorphic functions 
on the disc $\D_{R}(0)\subset\C$ around $0$ with radius $0<R\leq\infty$ and the multiplier space 
$\mathcal{O}_{M}(\R,E)$ of the Schwartz space for locally complete $E$ 
(see \prettyref{cor:Schauder_coeff_space_holom}, \prettyref{cor:Schauder_coeff_space_multiplier} 
and \prettyref{rem:Schauder_coeff_space_multiplier}).

The first section \prettyref{app:clos_abs_conv_compact} of the \textbf{\prettyref{app:appendix}} 
is devoted to the question when the closure of an absolutely convex hull of a set is compact in 
a locally convex Hausdorff space $E$ and \prettyref{app:pettis} to the related question of Pettis-integrability 
of an $E$-valued function. 

\newpage
\addtocontents{toc}{\SkipTocEntry}
\section*{Concerning originality}

We note that some parts of chapters or sections are based on our papers and preprints. 

\begin{itemize}
\item \prettyref{chap:linearisation}, \prettyref{sect:consistency} and \prettyref{sect:examples} are based on 
our paper \emph{Weighted spaces of vector-valued functions and the $\varepsilon$-product} \cite{kruse2017} and 
its extended preprint \cite{kruse2017a}. Furthermore, \prettyref{sect:examples} contains results from 
Sections 3 and 6 of our accepted preprint \emph{Extension of weighted vector-valued functions and 
sequence space representation} \cite{kruse2018_3} and our paper
\emph{Extension of weighted vector-valued functions and weak--strong principles for differentiable functions 
of finite order} \cite{kruse2019_3} and its extended preprint \cite{kruse2019_3a}.
\item \prettyref{sect:lifting} generalises some results of our papers 
\emph{Surjectivity of the $\overline{\partial}$-operator between weighted spaces of smooth vector-valued functions}
\cite{kruse2018_5} and \emph{Parameter dependence of solutions of the Cauchy--Riemann equation on weighted spaces of smooth functions} \cite{kruse2019_1} and its extended preprint \cite{kruse2019_1a}.
\item \prettyref{sect:extension}, \prettyref{sect:weak_strong_finite_order}, \prettyref{sect:blaschke},
\prettyref{sect:wolff_type} and \prettyref{sect:sequence_space} are based on our accepted preprint \cite{kruse2018_3} 
and our paper \cite{kruse2019_3} (and its extended preprint \cite{kruse2019_3a}).
\item \prettyref{sect:schauder} is based on our paper \emph{Series representations in spaces of vector-valued
functions via Schauder decompositions} \cite{kruse2018_1}.
\end{itemize}

Moreover, the introduction \prettyref{chap:intro} and \prettyref{chap:notation} on notation and preliminaries 
are based on the corresponding sections in our papers and preprints 
\cite{kruse2017a,kruse2017,kruse2019_1,kruse2018_1,kruse2018_3,kruse2018_5,kruse2019_3,kruse2019_3a}. 
However, not all of the results given in this thesis are already contained in our preprints or papers. 

In \prettyref{chap:linearisation} the new, i.e.\ not contained in our preprints or papers, 
results are \prettyref{cor:full_linearisation} (ii), \prettyref{ex:weighted_diff} e)+f),
\prettyref{ex:weighted_C_1_diff} and \prettyref{cor:Schwartz}. 

In \prettyref{sect:examples} the new examples and results are \prettyref{ex:sequence_vanish_infty},
\prettyref{cor:sequence_vanish_infty}, \prettyref{ex:cont_loc_comp}, \prettyref{ex:disc_algebra}, 
\prettyref{prop:frechet_bierstedt}, \prettyref{ex:subspace_bierstedt}, \prettyref{ex:diff_vanish_at_infinity} 
which extends \cite[Proposition 3.17 a), p.\ 244]{kruse2018_2} of our paper
\emph{The approximation property for weighted spaces of differentiable function} \cite{kruse2018_2}, 
\prettyref{prop:Fourier-trafo_Bjoerck} which extends \cite[Proposition 4.8, p.\ 370]{kruse2018_1} 
from sequentially complete $E$ to locally complete $E$, \prettyref{ex:Bjoerck} 
and \prettyref{ex:diff_ext_boundary} (ii).
All the results of \prettyref{sect:riesz_markov_kakutani} are new except for \prettyref{def:cons_strong} which is 
\cite[2.2 Definition, p.\ 4]{kruse2018_3} (and also not a result). 

The main theorem of \prettyref{sect:lifting}, \prettyref{thm:eps_prod_surj_inj}, 
is new even though special cases appeared in \cite{kruse2019_1,kruse2018_5}. 
\prettyref{thm:Borel_Ritt} and \prettyref{thm:Bjoerck_Fourier} are new as well. 
\prettyref{cor:Hoelder_vanish_Blaschke} extends \cite[7.3 Corollary, p.\ 22]{kruse2019_3a} 
from metric spaces with finite diameter to arbitrary metric spaces. 
\prettyref{thm:fourier.rap.dec} and \prettyref{thm:fourier_periodic} b) extend 
\cite[Theorem 4.9, p.\ 371--372]{kruse2018_1} and \cite[Theorem 4.11, p.\ 375]{kruse2018_1}
from sequentially complete $E$ to locally complete $E$. \prettyref{cor:Schauder_coeff_space_holom} is new 
in the sense that there is only a sketch how to prove it in \cite[p.\ 31]{kruse2018_3}.

The results of \prettyref{app:appendix} are also new except for \prettyref{prop:cadlag_precomp}, 
\prettyref{prop:vanish_at_infinity_precomp}, which are contained in 
\cite[5.2 Proposition, p.\ 24]{kruse2017a} and \cite[3.13 Lemma d), p.\ 10]{kruse2017a}, 
and \prettyref{lem:pettis.loc.complete} which is \cite[Lemma 4.7, p.\ 369]{kruse2018_1}.
\chapter{Notation and preliminaries}
\label{chap:notation}
\addtocontents{toc}{\SkipTocEntry}
\section*{Basics of topology}

We equip the spaces $\R^{d}$, $d\in\N$, and $\C$ with the usual Euclidean norm $|\cdot|$, denote by 
$\gls{br}:=\{w\in\R^{d}\;|\;|w-x|<r\}$ the ball around $x\in\R^{d}$ and by 
$\gls{dr}:=\{w\in\C\;|\;|w-z|<r\}$ the disc around $z\in\C$
with radius $r>0$. 
Furthermore, for a subset $M$ of a topological space $(X,t)$ we denote the closure of $M$ by $\gls{cl}$ 
and the boundary of $M$ by $\gls{bo}$. If we want to emphasize that we take the closure in $X$ resp.\ w.r.t.\ the 
topology $t$, then we write $\gls{clX}$ resp.\ $\gls{clt}$.
For a subset $M$ of a vector space $X$ we denote by $\gls{cih}$ the circled hull, 
by $\gls{coh}$ the convex hull and by $\gls{abcx}$ the 
absolutely convex hull of $M$. If $X$ is a topological vector space, we write $\gls{cabcx}$ 
for the closure of $\acx(M)$ in $X$.

\addtocontents{toc}{\SkipTocEntry}
\section*{Locally convex Hausdorff spaces and continuous linear operators}

By $\gls{E}$ we always denote a non-trivial, i.e.\ $E\neq\{0\}$, locally convex Hausdorff space over the field 
$\K=\R$ or $\C$ equipped with a directed fundamental system of 
seminorms $\gls{pa}$ and, in short, we write that $E$ is an \gls{lchs}. 
If $E=\K$, then we set $(p_{\alpha})_{\alpha\in \mathfrak{A}}:=\{|\cdot|\}.$ 

By $\gls{XhochO}$ we denote the set of maps from a non-empty set $\Omega$ to a non-empty set $X$, 
by $\gls{chiK}$ we mean the characteristic function of $K\subset\Omega$,  
by $\gls{c}$ the space of continuous functions from a topological space $\Omega$ 
to a topological space $X$, and by $\gls{c0}$ its subspace of 
continuous functions that vanish at infinity if $X$ is a locally convex Hausdorff space.

We denote by $\gls{LFE}$ the space of continuous linear operators from $F$ to $E$ 
where $F$ and $E$ are locally convex Hausdorff spaces. 
If $E=\K$, we just write $\gls{F'}:=L(F,\K)$ for the dual space and 
$\gls{Gcirc}$ for the \emph{polar set} of $G\subset F$. 
If $F$ and $E$ are linearly topologically isomorphic, we just write that $F$ and $E$ are isomorphic, 
in symbols $\gls{FcE}$.
We denote by $\gls{LtFE}$ the space $L(F,E)$ equipped with the locally convex topology $t$ of uniform convergence 
on the finite subsets of $F$ if $t=\gls{LsFE}$, on the absolutely convex, compact subsets of $F$ if $t=\gls{LkFE}$,
on the absolutely convex, $\sigma(F,F')$-compact subsets of $F$ if $t=\gls{LtauFE}$, 
on the precompact (totally bounded) subsets of $F$ 
if $t=\gls{LgFE}$ and on the bounded subsets of $F$ if $t=\gls{LbFE}$. 
We use the symbols $\gls{tF'F}$ for the corresponding topology on $F'$ 
and $\gls{tF}$ for the corresponding bornology on $F$. 
We say that a subspace $G\subset F'$ is \emph{\gls{separating}} (the points of $F$) if for every $x\in F$ 
it follows from $y(x)=0$ for all $y\in G$ that $x=0$. Clearly, this is equivalent to $G$ being $\sigma(F',F)$-dense in $F'$. 
For details and notions on the theory of locally convex spaces not explained 
in this thesis see \cite{F/W/Buch,Jarchow,meisevogt1997,Bonet}.

\addtocontents{toc}{\SkipTocEntry}
\section*{\texorpdfstring{$\varepsilon$}{epsilon}-products and tensor products}

The so-called \emph{\gls{eps_product}} of Schwartz is defined by 
\begin{equation}\label{notation0}
\gls{FepsE}:=L_{e}(F_{\kappa}',E)
\end{equation}
where $L(F_{\kappa}',E)$ is equipped with the topology of uniform convergence on equicontinuous subsets of $F'$. 
This definition of the $\varepsilon$-product coincides with the original 
one by Schwartz \cite[Chap.\ I, \S1, D\'{e}finition, p.\ 18]{Sch1}. 
It is symmetric which means that $F\varepsilon E\cong E\varepsilon F$. In the literature the definition of the 
$\varepsilon$-product is sometimes done the other way around, 
i.e.\ $E\varepsilon F$ is defined by the right-hand side 
of \eqref{notation0} but due to the symmetry these definitions are equivalent and for our purpose 
the given definition is more suitable. 
If we replace $F_{\kappa}'$ by $F_{\gamma}'$, we obtain Grothendieck's definition of the 
$\varepsilon$-product and  we remark that the two $\varepsilon$-products coincide 
if $F$ is quasi-complete because then $F_{\gamma}'=F_{\kappa}'$ holds. However, we stick to Schwartz' definition. 

For locally convex Hausdorff spaces $F_{i}$, $E_{i}$ and $T_{i}\in L(F_{i},E_{i})$, $i=1,2$, 
we define the $\varepsilon$-product 
$\gls{T1epsT2}\in L(F_{1}\varepsilon F_{2},E_{1}\varepsilon E_{2})$ of the operators $T_{1}$ and $T_{2}$ by 
\[
 (T_{1}\varepsilon T_{2})(u):=T_{2}\circ u\circ T_{1}^{t},\quad u\in F_{1}\varepsilon F_{2},
\]
where $T_{1}^{t}\colon E_{1}'\to F_{1}'$, $e'\mapsto e'\circ T_{1}$, is the \emph{dual map} of $T_{1}$. 
If $T_{1}$ is an isomorphism and $F_{2}=E_{2}$, then $T_{1}\varepsilon \id_{E_{2}}$ 
is also an isomorphism with inverse $T_{1}^{-1}\varepsilon\id_{E_{2}}$ by 
\cite[Chap.\ I, \S1, Proposition 1, p.\ 20]{Sch1} 
(or \cite[16.2.1 Proposition, p.\ 347]{Jarchow} if the $F_{i}$ are complete).

As usual we consider the tensor product $F\otimes E$ as a linear subspace of $F\varepsilon E$ for 
two locally convex Hausdorff spaces $F$ and $E$ by means of the linear injection 
\begin{equation}\label{eq:tensor_into_eps_product}
\gls{Theta}\colon F\otimes E\to F\varepsilon E,\; 
\sum^{k}_{n=1}{f_{n}\otimes e_{n}}\longmapsto\bigl[y\mapsto \sum^{k}_{n=1}{y(f_{n}) e_{n}}\bigr].
\end{equation}
Via $\Theta$ the space $\gls{FotimesE}$ is identified with the space of operators with finite rank in $F\varepsilon E$ 
and a locally convex topology is induced on $F\otimes E$. 
We write $\gls{FotimesepsE}$ for $F\otimes E$ equipped 
with this topology and $\gls{FotimesepsEcompl}$ for the completion 
of the \emph{\gls{injtenprod}} $F\otimes_{\varepsilon}E$.
For more information on the theory of $\varepsilon$-products 
and tensor products see \cite{Defant,Jarchow,Kaballo}. 

\addtocontents{toc}{\SkipTocEntry}
\section*{Several degrees of completeness}

The sufficient conditions for surjectivity of the map $S\colon \F\varepsilon E \to\FE$ from the introduction, 
which we derive in the forthcoming, depend on assumptions on different types of completeness of $E$. 
For this purpose we recapitulate some definitions which are connected to completeness. 
We start with local completeness. For a \emph{\gls{disk}} $D\subset E$, i.e.\ a bounded, absolutely convex set, 
the linear space $\gls{E_D}:=\bigcup_{n\in\N}nD$ becomes a normed space if it is equipped with the
gauge functional of $D$ as a norm (see \cite[p.\ 151]{Jarchow}). The space $E$ is called \emph{\gls{loccompl}} 
if $E_{D}$ is a Banach space for every closed disk $D\subset E$ (see \cite[10.2.1 Proposition, p.\ 197]{Jarchow}). 
We call a non-empty subset $A$ of an lcHs $E$ \emph{\gls{locclosed}} if every local limit point of $A$ 
belongs to $A$.
Here, a point $x\in E$ is called a \emph{\gls{loclimitp}} of $A$ if there is a sequence $(x_{n})_{n\in\N}$ in $A$ 
that converges locally to $x$ (see \cite[Definition 5.1.14, p.\ 154--155]{Bonet}), i.e.\ there is a disk $D\subset E$
such that $(x_{n})$ converges to $x$ in $E_{D}$ (see \cite[Definition 5.1.1, p.\ 151]{Bonet}). 
The \emph{\gls{locclosure}} of a subset $A$ of $E$ is defined as the smallest locally closed subset of $E$ 
which contains $A$ (see \cite[Definition 5.1.18, p.\ 155]{Bonet}).
Moreover, we note that every locally complete linear subspace of $E$ is locally closed and a locally closed linear
subspace of a locally complete space is locally complete by \cite[Proposition 5.1.20 (i), p.\ 155]{Bonet}. 

Moreover, a locally convex Hausdorff space is locally complete if and only 
if it is convenient by \cite[2.14 Theorem, p.\ 20]{kriegl}.
In particular, every complete locally convex Hausdorff space is quasi-complete, 
every quasi-complete space is sequentially complete and every 
sequentially complete space is locally complete and all these implications are strict. The first two by 
\cite[p.\ 58]{Jarchow} and the third by \cite[5.1.8 Corollary, p.\ 153]{Bonet} and 
\cite[5.1.12 Example, p.\ 154]{Bonet}. 

Now, let us recall the following definition from \cite[9-2-8 Definition, p.\ 134]{Wilansky} 
and \cite[p.\ 259]{J.Voigt}. A locally convex Hausdorff space is said 
to have the \emph{[\gls{metricconvcompactprop}] \gls{convcompactprop}} ([metric] ccp) if 
the closure of the absolutely convex hull of every [metrisable] compact set is compact. 
Sometimes this condition is phrased with the term convex hull instead of absolutely convex hull but these 
definitions coincide. Indeed, the first definition implies the second since every convex hull of a set 
$A\subset E$ is contained in its absolutely convex hull. 
On the other hand, we have $\acx(A)=\operatorname{cx}(\operatorname{ch}(A))$ by 
\cite[6.1.4 Proposition, p.\ 103]{Jarchow} and the circled hull $\operatorname{ch}(A)$ of a [metrisable] compact set
$A$ is compact by \cite[Chap.\ I, 5.2, p.\ 26]{schaefer} [and metrisable by 
\cite[Chap.\ IX, \S2.10, Proposition 17, p.\ 159]{bourbakiII} since $\D\times A$ is metrisable and 
$\operatorname{ch}(A)=M_{E}(\D\times A)$ where $M_{E}\colon\K\times E\to E$ is the continuous scalar 
multiplication and $\gls{unitdisc}:=\D_{1}(0)$ the open unit disc], which yields the other implication.

In particular, every locally convex Hausdorff space with ccp has obviously metric ccp, 
every quasi-complete locally convex Hausdorff space has ccp by \cite[9-2-10 Example, p.\ 134]{Wilansky}, 
every sequentially complete locally convex Hausdorff space 
has metric ccp by \cite[A.1.7 Proposition (ii), p.\ 364]{bogachev} and 
every locally convex Hausdorff space with metric ccp is locally complete by \cite[Remark 4.1, p.\ 267]{J.Voigt}. 
All these implications are strict. The second by \cite[9-2-10 Example, p.\ 134]{Wilansky} 
and the others by \cite[Remark 4.1, p.\ 267]{J.Voigt}. 
For more details on the [metric] convex compactness property and local completeness see \cite{Bonet2002,J.Voigt}.
In addition, we remark that every semi-Montel space is semi-reflexive by \cite[11.5.1 Proposition, p.\ 230]{Jarchow} 
and every semi-reflexive locally convex Hausdorff space is 
quasi-complete by \cite[Chap.\ IV, 5.5, Corollary 1, p.\ 144]{schaefer} and these implications are strict as well.
Summarizing, we have the following diagram of strict implications:
\begin{align*}
\text{semi-Montel}\;\Rightarrow\;&\phantom{q}\text{semi-reflexive}\\
&\phantom{semi}\Downarrow\\
\text{complete}\;\Rightarrow\;&\text{quasi-complete}\;\Rightarrow\;\text{sequentially complete}\;\Rightarrow\;\text{locally complete}\\
&\phantom{semi}\Downarrow\quad\phantom{complete sequentia}\Downarrow\phantom{complete}\; \rotatebox[origin=c]{29}{$\Longrightarrow$}\\
&\phantom{sem}\text{ccp}\phantom{omplete}\Rightarrow\phantom{seque}\text{metric ccp}
\end{align*}

\addtocontents{toc}{\SkipTocEntry}
\section*{Vector-valued continuously partially differentiable functions}

Since weighted spaces of continuously partially differentiable resp.\ holomorphic vector-valued functions will 
serve as our standard examples, we recall the definition of the spaces $\mathcal{C}^{k}(\Omega,E)$ resp.\ 
$\mathcal{O}(\Omega,E)$.
A function $f\colon\Omega\to E$ on an open set $\Omega\subset\R^{d}$ to an lcHs $E$ is called 
\emph{\gls{cpd}} ($f$ is $\mathcal{C}^{1}$) 
if for the $n$-th unit vector $e_{n}\in\R^{d}$ the limit
\[
(\partial^{e_{n}})^{E}f(x):=\lim_{\substack{h\to 0\\ h\in\R,h\neq 0}}\frac{f(x+he_{n})-f(x)}{h}
\]
exists in $E$ for every $x\in\Omega$ and $(\partial^{e_{n}})^{E}f$ 
is continuous on $\Omega$ ($(\partial^{e_{n}})^{E}f$ is $\mathcal{C}^{0}$) for every $1\leq n\leq d$.
For $k\in\N$ a function $f$ is said to be $k$-times continuously partially differentiable 
($f$ is $\mathcal{C}^{k}$) if $f$ is $\mathcal{C}^{1}$ and all its first partial derivatives are $\mathcal{C}^{k-1}$.
A function $f$ is called infinitely continuously partially differentiable ($f$ is $\mathcal{C}^{\infty}$) 
if $f$ is $\mathcal{C}^{k}$ for every $k\in\N$.
For $k\in\gls{N_infty}:=\N\cup\{\infty\}$ the functions $f\colon\Omega\to E$ which are $\mathcal{C}^{k}$ 
form a linear space which is denoted by $\gls{ck}$. 
For $\beta\in\N_{0}^{d}$ with $|\beta|:=\sum_{n=1}^{d}\beta_{n}\leq k$ and a function $f\colon\Omega\to E$ 
on an open set $\Omega\subset\R^{d}$ to an lcHs $E$ we set $(\partial^{\beta_{n}})^{E}f:=f$ if $\beta_{n}=0$, and
\[
(\partial^{\beta_{n}})^{E}f(x)
:=\underbrace{(\partial^{e_{n}})^{E}\cdots(\partial^{e_{n}})^{E}}_{\beta_{n}\text{-times}}f(x)
\]
if $\beta_{n}\neq 0$ and the right-hand side exists in $E$ for every $x\in\Omega$.
Further, we define 
\[
\gls{partial_b_f_E}(x)
:=\bigl((\partial^{\beta_{1}})^{E}\cdots(\partial^{\beta_{d}})^{E}\bigr)f(x)
\]
if the right-hand side exists in $E$ for every $x\in\Omega$. 
If $E=\K$, we often just write $\gls{partial_b_f}:=(\partial^{\beta})^{\K}f$ 
for $\beta\in\N_{0}^{d}$, $|\beta|\leq k$, and $f\in\mathcal{C}^{k}(\Omega)$.
Furthermore, we define the space of bounded continuously partially differentiable functions by 
\[
\gls{c1b}:=\{f\in\mathcal{C}^{1}(\Omega,E)\;|\;\forall\;\alpha\in\mathfrak{A}:\;
|f|_{\mathcal{C}^{1}_{b}(\Omega),\alpha}:=\sup_{\substack{x\in\Omega\\ \beta\in\N_{0}^{d},|\beta|\leq 1}}p_{\alpha}((\partial^{\beta})^{E}f(x))<\infty\}.
\]

\addtocontents{toc}{\SkipTocEntry}
\section*{Vector-valued holomorphic functions}

A function $f\colon\Omega\to E$ on an open set $\Omega\subset\C$ to an lcHs $E$ over $\C$ 
is called \emph{\gls{holomorphic}} if the limit
\[
 (\partial^{1}_{\C})^{E}f(z):=\lim_{\substack{h\to 0\\ h\in\C,h\neq 0}}\frac{f(z+h)-f(z)}{h},\quad z\in \Omega,
\]
exists in $E$. We denote by $\gls{OE}$ the linear space of holomorphic functions 
$f\colon\Omega\to E$. Defining the vector-valued version of the \emph{\gls{CRoperator}} by 
\[
\gls{dbar_f}:=\frac{1}{2}((\partial^{e_{1}})^{E}+\iu(\partial^{e_{2}})^{E})f
\]
for $f\in\mathcal{C}(\Omega,E)$ such that the partial derivatives $(\partial^{e_{n}})^{E}f$, $n=1,2$, 
exist in $E$, we remark that 
\begin{equation}\label{eq:holomorphic_coincide_0}
\mathcal{O}(\Omega,E)=\{f\in\mathcal{C}(\Omega,E)\;|\;f\in\ker\overline{\partial}^{E}\}
                     =\{f\in\mathcal{C}^{\infty}(\Omega,E)\;|\;f\in\ker\overline{\partial}^{E}\}
\end{equation}
by \cite[Theorem 6.1, p.\ 267]{kruse2019_4} if $E$ is locally complete. 
Further, we set $(\partial^{0}_{\C})^{E}f:=f$ and note that the $(n+1)$-th complex derivative
$(\partial^{n+1}_{\C})^{E}f:=(\partial^{1}_{\C})^{E}(\gls{partial_nC_f_E})$ exists 
for all $n\in\N_{0}$ and $f\in \mathcal{O}(\Omega,E)$ by 
\cite[2.1 Theorem and Definition, p.\ 17--18]{grosse-erdmann1992} and 
\cite[5.2 Theorem, p.\ 35]{grosse-erdmann1992} if $E$ is locally complete. If $E=\C$, we often just write 
$\gls{partial_nC_f}:=(\partial^{n}_{\C})^{\C}f$ for $n\in\N_{0}$ and 
$f\in\mathcal{O}(\Omega):=\mathcal{O}(\Omega,\C)$. 
We note that the real and complex derivatives are related by
\begin{equation}\label{eq:complex.real.deriv}
 (\partial^{\beta})^{E}f(z)=\iu^{\beta_{2}}(\partial^{|\beta|}_{\C})^{E}f(z),\quad z\in\Omega,
\end{equation}
for every $f\in\mathcal{O}(\Omega,E)$ and $\beta=(\beta_{1},\beta_{2})\in\N_{0}^{2}$ 
by \cite[Proposition 7.1, p.\ 270]{kruse2019_4} if $E$ is locally complete.  
\chapter{The \texorpdfstring{$\varepsilon$}{epsilon}-product for weighted function spaces}
\label{chap:linearisation}
\section{\texorpdfstring{$\varepsilon$}{epsilon}-into-compatibility}
\label{sect:eps-prod_into}
In the introduction we already mentioned that linearisations of spaces of vector-valued functions by means of 
$\varepsilon$-products are essential for our approach. Here, one of the important questions 
is which spaces of vector-valued functions can be represented by $\varepsilon$-products. 
Let $\Omega$ be a non-empty set and $E$ an lcHs. If $\F\subset\K^{\Omega}$
is an lcHs such that $\gls{delta_x}\in\F'$ for all $x\in\Omega$, then the map 
\[
S\colon \F\varepsilon E\to E^{\Omega},\;u\longmapsto [x\mapsto u(\delta_{x})],
\]
is well-defined and linear. This leads to the following definition.

\begin{defn}[{$\varepsilon$-into-compatible}]
Let $\Omega$ be a non-empty set and $E$ an lcHs. Let $\F\subset\K^{\Omega}$ and $\FE\subset E^{\Omega}$ 
be lcHs such that $\delta_{x}\in\F'$ for all $x\in\Omega$.
We call the spaces $\F$ and $\FE$ \emph{\gls{eps_into_comp}} if the map
\[
\gls{S}\colon \F\varepsilon E\to \FE,\;u\longmapsto [x\mapsto u(\delta_{x})],
\]
is a well-defined isomorphism into, i.e.\ to its range. 
We call $\F$ and $\FE$ \emph{\gls{eps_comp}} if $S$ is an isomorphism. 
We write $\gls{S_F}$ if we want to emphasise the dependency on $\F$.
\end{defn}

In this section we introduce the weighted space $\FVE$ of $E$-valued functions on $\Omega$ 
as a subspace of sections of domains in $E^{\Omega}$ of linear operators $T^{E}_{m}$ 
equipped with a generalised version of a weighted graph topology. 
This space is the role model for many function spaces and an example for these operators 
are the partial derivative operators.
Then we treat the question whether $\FVE$ and $\FV\varepsilon E$ are $\varepsilon$-into-compatible.
This is deeply connected with the interplay of the pair of operators $(T^{E}_{m},T^{\K}_{m})$ with the map $S$ 
(see \prettyref{def:consist_strong}). In our main theorem of this section we give sufficient conditions such that 
$S\colon\FV\varepsilon E\to\FVE$ is an isomorphism into (see \prettyref{thm:linearisation}). In the next section we provide conditions such that 
$S$ becomes surjective (see \prettyref{thm:full_linearisation}). 
We start with the well-known example $\mathcal{C}^{k}(\Omega,E)$ of $k$-times continuously 
partially differentiable $E$-valued functions to motivate our definition of $\FVE$.

\begin{exa}\label{ex:k_smooth_functions}
Let $k\in\N_{\infty}$ and $\Omega\subset\R^{d}$ be open. Consider the space $\mathcal{C}(\Omega,E)$ 
of continuous functions $f\colon\Omega\to E$ with the \emph{topology} $\gls{tau_c}$ \emph{of compact convergence}, 
i.e.\ the topology given by the seminorms 
\[
\|f\|_{K,\alpha}:=\sup_{x\in K}p_{\alpha}(f(x)),\quad f\in\mathcal{C}(\Omega,E),
\]
for compact $K\subset\Omega$ and $\alpha\in\mathfrak{A}$. 
The usual topology on the space $\mathcal{C}^{k}(\Omega,E)$ of $k$-times continuously partially differentiable 
functions is the graph topology generated by the partial derivative operators 
$
(\partial^{\beta})^{E}\colon \mathcal{C}^{k}(\Omega,E)\to \mathcal{C}(\Omega,E)
$
for $\beta\in\N_{0}^{d}$, $|\beta|\leq k$, i.e.\ the topology given by the seminorms 
\[
\|f\|_{K,\beta,\alpha}:=\max(\|f\|_{K,\alpha},\|(\partial^{\beta})^{E}f\|_{K,\alpha}),
\quad f\in\mathcal{C}^{k}(\Omega,E),
\]
for compact $K\subset\Omega$, $\beta\in\N_{0}^{d}$, $|\beta|\leq k$, and $\alpha\in\mathfrak{A}$. 
The same topology is induced by the directed system of seminorms given by 
\[
 |f|_{K,m,\alpha}:=\sup_{\beta\in\N_{0}^{d},|\beta|\leq m}\|f\|_{K,\beta,\alpha}
 =\sup_{\substack{x\in K\\ \beta\in\N_{0}^{d},|\beta|\leq m}}
 p_{\alpha}\bigl((\partial^{\beta})^{E}f(x)\bigr),\quad f\in\mathcal{C}^{k}(\Omega,E),
\]
for compact $K\subset\Omega$, $m\in\N_{0}$, $m\leq k$, and $\alpha\in\mathfrak{A}$ 
and may also be seen as a weighted topology induced by the family $(\chi_{K})$ of characteristic 
functions of the compact sets $K\subset\Omega$ by writing 
\[
 |f|_{K,m,\alpha}=\sup_{\substack{x\in \Omega\\ \beta\in\N_{0}^{d},|\beta|\leq m}}
 p_{\alpha}\bigl((\partial^{\beta})^{E}f(x)\bigr)\chi_{K}(x),\quad f\in\mathcal{C}^{k}(\Omega,E).
\]
This topology is inherited by linear subspaces of functions having additional properties like being holomorphic 
or harmonic. 
\end{exa}

We turn to the weight functions which we use to define a kind of weighted graph topology.

\begin{defn}[{weight function}]\label{def:weight} 
Let $J$ be a non-empty set and $(\omega_{m})_{m\in M}$ a family of non-empty sets. 
We call $\gls{V}:=(\nu_{j,m})_{j\in J,m\in M}$  
a family of \emph{\gls{weight}s} on $(\omega_{m})_{m\in M}$ if it fulfils 
$\nu_{j,m}\colon \omega_{m}\to [0,\infty)$ 
for all $j\in J$, $m\in M$ and 
\begin{equation}\label{loc3}
  \forall \; m\in M,\, x\in\omega_{m}\;\exists\;j\in J:\;0< \nu_{j,m}(x).
\end{equation}
\end{defn}

From the structure of \prettyref{ex:k_smooth_functions} we arrive at the following definition 
of the weighted spaces of vector-valued functions we want to consider. 

\begin{defn}\label{def:weighted_space}
Let $\Omega$ be a non-empty set, $\mathcal{V}:=(\nu_{j,m})_{j\in J,m\in M}$ 
a family of weight functions on $(\omega_{m})_{m\in M}$ and
$T^{E}_{m}\colon E^{\Omega}\supset\dom T^{E}_{m} \to E^{\omega_{m}}$ a linear map for every $m\in M$. 
Let $\gls{APE}$ be a linear subspace of $E^{\Omega}$ and define the space of intersections 
\[
\gls{FE}:=\operatorname{AP}(\Omega,E)\cap\bigl(\bigcap_{m\in M}\dom T^{E}_{m}\bigr)
\]
as well as
\[
\gls{FVE}:=\bigl\{f\in F(\Omega,E)\;|\; 
 \forall\;j\in J,\, m\in M,\,\alpha\in \mathfrak{A}:\; |f|_{j,m,\alpha}<\infty\bigr\}
\]
where
\[
\gls{f_jma}:=\sup_{x \in \omega_{m}}
p_{\alpha}\bigl(T^{E}_{m}(f)(x)\bigr)\nu_{j,m}(x)
=\sup_{e\in N_{j,m}(f)}p_{\alpha}(e)
\]
with
\[
\gls{Njmf}:=\{T^{E}_{m}(f)(x)\nu_{j,m}(x)\;|\;x\in\omega_{m}\}.
\]
Further, we write $\gls{F_Om}:=F(\Omega,\K)$ and $\gls{FV}:=\mathcal{FV}(\Omega,\K)$. 
If we want to emphasise dependencies, we write $M(E)$ instead of $M$, 
$\gls{APE_FV}$ instead of $\operatorname{AP}(\Omega,E)$
and $\gls{f_FVjma}$ instead of $|f|_{j,m,\alpha}$. 
If $J$, $M$ or $\mathfrak{A}$ are singletons, we omit the index $j$, $m$ resp.\ $\alpha$ in $|f|_{j,m,\alpha}$.
\end{defn}

Note that $\omega_{m}$ need not be a subset of $\Omega$. 
The space $\operatorname{AP}(\Omega,E)$ is a placeholder where we collect 
\emph{additional properties} ($\operatorname{AP}$) 
of our functions not being reflected by the operators $T^{E}_{m}$ which we integrated in the topology. 
However, these additional properties might come from being in the domain or kernel of additional operators, e.g.\  
harmonicity means being in the kernel of the Laplacian.
But often $\operatorname{AP}(\Omega,E)$ can be chosen as $E^{\Omega}$ 
or $\mathcal{C}(\Omega,E)$. 
The space $\FVE$ is locally convex but need not be Hausdorff. 
Since it is easier to work with Hausdorff spaces and a directed family 
of seminorms plus the point evaluation functionals $\delta_{x}\colon \FV\to \K$, 
$f\mapsto f(x)$, for $x\in \Omega$ and their continuity play a big role, we introduce the following definition. 

\begin{defn}[{$\dom$-space and $T^{E}_{m,x}$}]
We call $\FVE$ a \emph{\gls{domspace}} if it is a locally convex Hausdorff space, the system of 
seminorms $(|f|_{j,m,\alpha})_{j\in J, m\in M, \alpha\in\mathfrak{A}}$ 
is directed and, in addition, $\delta_{x}\in \FV'$ for every $x\in \Omega$ if $E=\K$. 
We define the point evaluation of $T^{E}_{m}$ by $\gls{TE_mx}\colon \dom T^{E}_{m} \to E$,
$T^{E}_{m,x}(f):=T^{E}_{m}(f)(x)$, for $m\in M$ and $x\in\omega_{m}$.
\end{defn}

\begin{rem}\fakephantomsection\label{rem:weights_Hausdorff_directed}
\begin{enumerate}
\item[a)] It is easy to see that $\FVE$ is Hausdorff if there is $m\in M$ such that $\omega_{m}=\Omega$ and 
$T^{E}_{m}=\operatorname{id}_{E^{\Omega}}$ since $E$ is Hausdorff. 
\item[b)] If $E=\K$, then $T^{\K}_{m,x}\in\FV'$ for every $m\in M$ and $x\in\omega_{m}$. 
Indeed, for $m\in M$ and $x\in\omega_{m}$ there exists $j\in J$ 
 such that $\nu_{j,m}(x)>0$ by \eqref{loc3},
 implying for every $f\in\FV$ that
 \[ 
 |T^{\K}_{m,x}(f)|=\frac{1}{\nu_{j,m}(x)}|T^{\K}_{m}(f)(x)|\nu_{j,m}(x)
 \leq \frac{1}{\nu_{j,m}(x)}|f|_{j,m}.
 \]
In particular, this implies $\delta_{x}\in \FV'$ for all $x\in \Omega$ if there is $m\in M$ 
such that $\omega_{m}=\Omega$ and $T^{\K}_{m}=\operatorname{id}_{\K^{\Omega}}$. 
\item[c)] Let the family of weight functions $\mathcal{V}$ be \emph{\gls{directed}}, i.e.\
\begin{flalign*}
  \forall \; j_{1},j_{2}\in J,\, &m_{1},m_{2}\in M\;\exists\;j_{3}\in J,\, m_{3}\in M,\,C>0\;
  \forall\;i\in\{1,2\}:\\ 
  &(\omega_{m_{1}}\cup \omega_{m_{2}})\subset \omega_{m_{3}}
   \quad\text{and}\quad 
   \nu_{j_{i},m_{i}}\leq C\nu_{j_{3},m_{3}}.
\end{flalign*} 
Then the system of seminorms $(|f|_{j,m,\alpha})_{j\in J, m\in M, \alpha\in\mathfrak{A}}$ is directed 
if $\mathcal{V}$ is directed and additionally it holds with $m_{i}$, $i\in\{1,2,3\}$, from above that
\[
\forall\;f\in\FVE,\,i\in\{1,2\},\,x\in\omega_{m_{i}}:\; T^{E}_{m_{i}}(f)(x)=T^{E}_{m_{3}}(f)(x),
\]
since the system $(p_{\alpha})_{\alpha\in \mathfrak{A}}$ of $E$ is already directed.
\end{enumerate}
\end{rem}

We point out that the additional condition in \prettyref{rem:weights_Hausdorff_directed} c) is missing in 
\cite[Remark 5 c), p.\ 1516]{kruse2017} (resp.\ \cite[3.5 Remark, p.\ 6]{kruse2017a}), which we correct here. 

For the lcHs $E$ over $\K$ we want to define a natural $E$-valued version of a $\dom$-space 
$\FV=\mathcal{FV}(\Omega,\K)$. The natural $E$-valued version of $\FV$ should be a $\dom$-space $\FVE$ 
such that there is a canonical relation between the families $(T^{\K}_{m})$ and $(T^{E}_{m})$. 
This canonical relation will be explained in terms of their interplay with the map 
\[
S\colon \FV\varepsilon E \to E^{\Omega},\; u\longmapsto [x\mapsto u(\delta_{x})].
\]
Further, the elements of our $E$-valued version $\FVE$ of $\FV$ should be compatible with a weak definition 
in the sense that $e'\circ f\in\FV$ should hold for every $e'\in E'$ and $f\in\FVE$.

\begin{defn}[{generator, consistent, strong}]\label{def:consist_strong} 
Let $\FV$ and $\FVE$ be $\dom$-spaces such that $M:=M(\K)=M(E)$.
\begin{enumerate}
 \item [a)] We call $(T^{E}_{m},T^{\K}_{m})_{m\in M}$ a \emph{\gls{generator}} for $(\FV,E)$, 
 in short, $(\mathcal{FV},E)$.
 \item [b)] We call $(T^{E}_{m},T^{\K}_{m})_{m\in M}$ \emph{\gls{consistent}} if we have
 for all $u\in\FV\varepsilon E$ that $S(u)\in F(\Omega,E)$ and 
 \[
 \forall\;m\in M,\,x\in\omega_{m}:\;\bigl(T^{E}_{m}S(u)\bigr)(x)=u(T^{\K}_{m,x}).
 \]
 \item[c)] We call $(T^{E}_{m},T^{\K}_{m})_{m\in M}$ \emph{\gls{strong}} if we have 
 for all $e'\in E'$, $f\in \FVE$ that $e'\circ f\in F(\Omega)$ and 
 \[
   \forall\; m\in M,\,x\in\omega_{m}:\;T^{\K}_{m}(e'\circ f)(x)=\bigl(e'\circ T^{E}_{m}(f)\bigr)(x). 
   \]
\end{enumerate}
\end{defn}

More precisely, $T^{\K}_{m,x}$ in b) means the restriction of $T^{\K}_{m,x}$ 
to $\FV$ and the term $u(T^{\K}_{m,x})$ is well-defined by \prettyref{rem:weights_Hausdorff_directed} b). 
Consistency will guarantee that the map $S\colon \FV\varepsilon E\to \FVE$ is a well-defined isomorphism into,
i.e.\ $\varepsilon$-into-compatibility, and strength will help us to prove its surjectivity 
under some additional assumptions on $\FV$ and $E$.
Let us come to a lemma which describes the topology of $\FV\varepsilon E$ in terms of the operators 
$T^{\K}_{m}$ with $m\in M$. 
It was the motivation for the definition of consistency and allows us to consider $\FV\varepsilon E$ as a 
topological subspace of $\FVE$ via $S$, assuming consistency. 

\begin{lem}\label{lem:topology_eps}
Let $\FV$ be a $\dom$-space. Then the topology of $\FV\varepsilon E$ is given by the system of seminorms 
defined by
	\[
	\|u\|_{j,m,\alpha}:=\sup_{x\in\omega_{m}}p_{\alpha}\bigl(u(T^{\K}_{m,x})\bigr)\nu_{j,m}(x),
	\quad u\in \FV\varepsilon E,
	\]
	for $j\in J$, $m\in M$ and $\alpha\in \mathfrak{A}$.
\end{lem}
\begin{proof}
 We define the sets $D_{j,m}:=\{T^{\K}_{m,x}(\cdot)\nu_{j,m}(x)\;|\;x\in\omega_{m}\}$ 
 and $B_{j,m}:=\{f\in\FV\;|\;|f|_{j,m}\leq 1\}$ for every $j\in J$ and $m\in M$.
 We claim that $\acx(D_{j,m})$ is dense in the polar $B^{\circ}_{j,m}$ 
 with respect to $\kappa(\FV',\FV)$.
 The observation 
 \begin{align*}
 D_{j,m}^{\circ}
 &=\{T^{\K}_{m,x}(\cdot)\nu_{j,m}(x)\;|\;x\in\omega_{m}\}^{\circ}\\
 &=\{f\in\FV\;|\;\forall x\in\omega_{m}: |T^{\K}_{m}(f)(x)|
 \nu_{j,m}(x)\leq 1\}\\
 &=\{f\in\FV\;|\; |f|_{j,m}\leq 1\}
 =B_{j,m}
 \end{align*}
 yields
\[
 \oacx(D_{j,m})^{\kappa(\FV',\FV)}
=(D_{j,m})^{\circ\circ}=B^{\circ}_{j,m}
\]
by the bipolar theorem. 
By \cite[8.4, p.\ 152, 8.5, p.\ 156--157]{Jarchow} the system of seminorms defined by 
\[
q_{j,m,\alpha}(u):=\sup_{y\in B_{j,m}^{\circ}}p_{\alpha}\bigl(u(y)\bigr),\quad u\in \FV\varepsilon E,
\]
for $j\in J$, $m\in M$ and $\alpha\in \mathfrak{A}$ gives the topology on $\FV\varepsilon E$ (here it is used 
that the system of seminorms $(|\cdot|_{j,m})$ of $\FV$ is directed). 
As every $u\in\FV\varepsilon E$ is continuous on $B_{j,m}^{\circ}$, we may replace $B_{j,m}^{\circ}$ 
by a $\kappa(\FV',\FV)$-dense subset. Therefore we obtain
\[
q_{j,m,\alpha}(u)=\sup\bigl\{p_{\alpha}\bigl(u(y)\bigr)\;|\;
y\in\acx(D_{j,m})\bigr\}.
\]
For $y\in\acx(D_{j,m})$ there are $n\in\N$, $\lambda_{k}\in\K$, $x_{k}\in\omega_{m}$, $1\leq k\leq n$, with
$\sum^{n}_{k=1}{|\lambda_{k}|}\leq 1$ such that 
$y=\sum^{n}_{k=1}\lambda_{k}T^{\K}_{m,x_{k}}(\cdot)\nu_{j,m}(x_{k})$. 
Then we have for every $u\in\FV\varepsilon E$
\[
     p_{\alpha}\bigl(u(y)\bigr)
\leq \sum^{n}_{k=1}|\lambda_{k}|p_{\alpha}\bigl(u(T^{\K}_{m,x_{k}})\bigr)\nu_{j,m}(x_{k})
\leq \|u\|_{j,m,\alpha},
\]
thus $q_{j,m,\alpha}(u)\leq \|u\|_{j,m,\alpha}$. On the other hand, we derive
\[
q_{j,m,\alpha}(u)\geq\sup_{y\in D_{j,m}}p_{\alpha}\bigl(u(y)\bigr)
=\sup_{x\in\omega_{m}}p_{\alpha}\bigl(u(T^{\K}_{m,x})\bigr)\nu_{j,m}(x)=\|u\|_{j,m,\alpha}.\qedhere
\]
\end{proof}

Let us turn to a more general version of \prettyref{ex:k_smooth_functions}, namely, 
to weighted spaces of $k$-times continuously partially differentiable functions and 
kernels of linear partial differential operators in these spaces.

\begin{exa}\label{ex:weighted_smooth_functions}
Let $k\in\N_{\infty}$ and $\Omega\subset\R^{d}$ be open. 
We consider the cases
\begin{enumerate}
\item [(i)] $\omega_{m}:=M_{m}\times\Omega$ with $\gls{M_m}:=\{\beta\in\N_{0}^{d}\;|\;|\beta|\leq \min(m,k)\}$ 
for all $m\in\N_{0}$, or
\item [(ii)] $\omega_{m}:=\N_{0}^{d}\times\Omega$ for all $m\in\N_{0}$ and $k=\infty$,
\end{enumerate}
and let $\mathcal{V}^{k}:=(\nu_{j,m})_{j\in J, m\in\N_{0}}$ be a directed family 
of weights on $(\omega_{m})_{m\in\N_{0}}$. 

a) We define the weighted space of $k$-times continuously partially differentiable functions with values 
in an lcHs $E$ as
\[
 \gls{CV_k_OE}:=\{f\in\mathcal{C}^{k}(\Omega,E)\;|\;\forall\;j\in J,\,m\in\N_{0},\,
 \alpha\in\mathfrak{A}:\;|f|_{j,m,\alpha}<\infty\} 
\]
where 
\[
 |f|_{j,m,\alpha}:=\sup_{(\beta,x)\in\omega_{m}}
 p_{\alpha}\bigl((\partial^{\beta})^{E}f(x)\bigr)\nu_{j,m}(\beta,x).
\]
Setting $\dom T^{E}_{m}:=\mathcal{C}^{k}(\Omega,E)$ and 
\[
 T^{E}_{m}\colon\mathcal{C}^{k}(\Omega,E)\to E^{\omega_{m}},\; 
 f\longmapsto [(\beta,x)\mapsto (\partial^{\beta})^{E}f(x)], 
\]
as well as $\operatorname{AP}(\Omega,E):=E^{\Omega}$, we observe that $\mathcal{CV}^{k}(\Omega,E)$ 
is a $\dom$-space by \prettyref{rem:weights_Hausdorff_directed} and 
\[
 |f|_{j,m,\alpha}=\sup_{x\in\omega_{m}}p_{\alpha}\bigl(T^{E}_{m}f(x)\bigr)\nu_{j,m}(x).
\]

b) The space $\mathcal{C}^{k}(\Omega,E)$ with its usual topology given in \prettyref{ex:k_smooth_functions} 
is a special case of a)(i) with $J:=\{K\subset\Omega\;|\;K\;\text{compact}\}$, 
$\nu_{K,m}(\beta,x):=\chi_{K}(x)$, $(\beta,x)\in\omega_{m}$, for all $m\in\N_{0}$ and $K\in J$ 
where $\chi_{K}$ is the characteristic function of $K$.  
In this case we write $\gls{W_k}:=\mathcal{V}^{k}$ for the family of weight functions.

c) The \emph{\gls{Schwartz_space}} is defined by
\[
\gls{SRE}
:=\{f\in \mathcal{C}^{\infty}(\R^{d},E)\;|\;
\forall\;m\in\N_{0},\,\alpha\in \mathfrak{A}:\;|f|_{m,\alpha}<\infty\}
\]
where 
\[
 |f|_{m,\alpha}:=\sup_{\substack{x\in\R^{d}\\ \beta\in\N_{0}^{d},|\beta|\leq m}}
 p_{\alpha}\bigl((\partial^{\beta})^{E}f(x)\bigr)(1+|x|^{2})^{m/2}.
\]
This is a special case of a)(i) with $k:=\infty$, $\Omega:=\R^{d}$, $J:=\{1\}$ and 
$\nu_{1,m}(\beta,x):=(1+|x|^{2})^{m/2}$, $(\beta,x)\in\omega_{m}$, for all $m\in\N_{0}$.

d) The \emph{\gls{multiplier_space}} for the Schwartz space is defined by 
\[
\gls{O_MRE}
:=\{f\in \mathcal{C}^{\infty}(\R^{d},E)\;|\;\forall\;g\in\mathcal{S}(\R^{d}),\,
m\in\N_{0},\,\alpha\in \mathfrak{A}:\;\|f\|_{g,m,\alpha}<\infty\}
\]
where
\[
 \|f\|_{g,m,\alpha}:=\sup_{\substack{x\in\mathbb{R}^{d}\\
  \beta\in\N_{0}^{d},|\beta|\leq m}}
 p_{\alpha}\bigl((\partial^{\beta})^{E}f(x)\bigr)|g(x)|
\]
(see \cite[$4^{0}$), p.\ 97]{Schwartz1955}).
This is a special case of a)(i) with $k:=\infty$, $\Omega:=\R^{d}$, 
$J:=\{j\subset\mathcal{S}(\R^{d})\;|\; j\;\text{finite}\}$ and 
$\nu_{j,1,m}(\beta,x):=\max_{g\in j}|g(x)|$, 
$(\beta,x)\in\omega_{m}$, for all $m\in\N_{0}$. This choice of $J$ guarantees 
that the family $\mathcal{V}^{\infty}$ is directed and does not change the topology.

e) Let $\mathfrak{K}:=\{K\subset\Omega\;|\;K\;\text{compact}\}$ 
and $(M_{p})_{p\in\N_{0}}$ be a sequence of positive real numbers. 
The space $\gls{EMpOE}$ of \emph{\gls{ultradifferentiable} functions of class 
$(M_{p})$ of Beurling-type} is defined as 
\[
\mathcal{E}^{(M_{p})}(\Omega,E)
:=\{f\in \mathcal{C}^{\infty}(\Omega,E)\;|\;\forall\;K\in\mathfrak{K},\,h>0,\,\alpha\in \mathfrak{A}:\;
|f|_{(K,h),\alpha}<\infty\}
\]
where 
\[
 |f|_{(K,h),\alpha}:=\sup_{\substack{x\in K \\ \beta\in\N_{0}^{d}}}
 p_{\alpha}\bigl((\partial^{\beta})^{E}f(x)\bigr)\frac{1}{h^{|\beta|}M_{|\beta|}}.
\]
This is a special case of a)(ii) with $J:=\mathfrak{K}\times\R_{>0}$ and 
$\nu_{(K,h),m}(\beta,x):=\chi_{K}(x)\frac{1}{h^{|\beta|}M_{|\beta|}}$, $(\beta,x)\in\omega_{m}$, 
for all $(K,h)\in J$ and $m\in\N_{0}$ where $\R_{>0}:=(0,\infty)$.

f) Let $\mathfrak{K}$ and $(M_{p})_{p\in\N_{0}}$ be as in e). 
The space $\gls{EM_pOE}$ of \emph{ultradifferentiable functions of class 
$\{M_{p}\}$ of Roumieu-type} is defined as
\[
\mathcal{E}^{\{M_{p}\}}(\Omega,E)
:=\{f\in \mathcal{C}^{\infty}(\Omega,E)\;|\;\forall\;(K,H)\in J,\,\alpha\in \mathfrak{A}:\;
|f|_{(K,H),\alpha}<\infty\}
\]
where 
\[
J:=\mathfrak{K}\times\{H=(H_{n})_{n\in\N}\;|\;\exists\;(h_{k})_{k\in\N},\,h_{k}>0,\,h_{k}\nearrow\infty
\;\forall\;n\in\N:\;
H_{n}=h_{1}\cdot\ldots\cdot h_{n}\}
\]
and
\[
 |f|_{(K,H),\alpha}:=\sup_{\substack{x\in K\\ \beta\in\N_{0}^{d}}}
 p_{\alpha}\bigl((\partial^{\beta})^{E}f(x)\bigr)\frac{1}{H_{|\beta|}M_{|\beta|}}
\]
(see \cite[Proposition 3.5, p.\ 675]{Kom9}). Again, this is a special case of a)(ii) with 
$\nu_{(K,H),m}(\beta,x):=\chi_{K}(x)\frac{1}{H_{|\beta|}M_{|\beta|}}$, $(\beta,x)\in\omega_{m}$, 
for all $(K,H)\in J$ and $m\in\N_{0}$.

g) Let $n\in\N$, $\beta_{i}\in\N_{0}^{d}$ with $|\beta_{i}|\leq k$ and 
$a_{i}\colon\Omega\to\K$ for $1\leq i\leq n$. We set 
\[
 P(\partial)^{E}\colon \mathcal{C}^{k}(\Omega,E)\to E^{\Omega},\;
 P(\partial)^{E}(f)(x):=\sum_{i=1}^{n}a_{i}(x)(\partial^{\beta_{i}})^{E}(f)(x)
\]
and obtain the (topological) subspace of $\mathcal{CV}^{k}(\Omega,E)$ given by 
\[
 \gls{CVP_k_OE}:=\{f\in\mathcal{CV}^{k}(\Omega,E)\;|\;f\in\ker P(\partial)^{E}\}.
\]
Choosing $\operatorname{AP}(\Omega,E):=\ker P(\partial)^{E}$, we see that this is also a $\dom$-space by a). 
If $P(\partial)^{E}$ is the Cauchy--Riemann operator (and $E$ locally complete) or the Laplacian, 
we obtain the weighted space of holomorphic resp.\ harmonic functions. 
\end{exa}

Let us show that the generators of these spaces are strong and consistent.
In order to obtain consistency for their generators 
we have to restrict to directed families of weights  
which are \emph{\gls{loc_b_away_0}} on $\Omega$, i.e.\ 
\[
 \forall\;K\subset\Omega\;\text{compact},\,m\in\N_{0}\;\exists\; 
 j\in J\;\forall\;\beta\in\N_{0}^{d},\,|\beta|\leq\min(m,k):\;\inf_{x\in K}\nu_{j,m}(\beta,x)>0.
\]
This condition on $\mathcal{V}^{k}$ guarantees that the map 
$I\colon\mathcal{CV}^{k}(\Omega)\to\mathcal{CW}^{k}(\Omega)$, 
$f\mapsto f$, is continuous which is needed for consistency. 

\begin{prop}\label{prop:weighted_diff_strong_cons}
Let $E$ be an lcHs, $k\in\N_{\infty}$, $\mathcal{V}^{k}$ be a directed family of weights 
which is locally bounded away from zero on an open set $\Omega\subset\R^{d}$. 
The generator of $(\mathcal{CV}^{k},E)$ resp.\ $(\mathcal{CV}^{k}_{P(\partial)},E)$ 
from \prettyref{ex:weighted_smooth_functions} 
is strong and consistent if $\mathcal{CV}^{k}(\Omega)$ resp.\ 
$\mathcal{CV}^{k}_{P(\partial)}(\Omega$) is barrelled.
\end{prop} 
\begin{proof}
We recall the definitions from \prettyref{ex:weighted_smooth_functions}. 
We have $\omega_{m}:=M_{m}\times\Omega$ with $M_{m}:=\{\beta\in\N_{0}^{d}\;|\;|\beta|\leq\min(m,k)\}$ 
for all $m\in\N_{0}$ or $\omega_{m}:=\N_{0}^{d}\times\Omega$ for all $m\in\N_{0}$. 
Further, $\operatorname{AP}_{\mathcal{CV}^{k}}(\Omega,E)=E^{\Omega}$, 
$\operatorname{AP}_{\mathcal{CV}^{k}_{P(\partial)}}(\Omega,E)=\ker P(\partial)^{E}$, 
$\dom T^{E}_{m}:=\mathcal{C}^{k}(\Omega,E)$ and 
\[
 T^{E}_{m}\colon\mathcal{C}^{k}(\Omega,E)\to E^{\omega_{m}},\; 
 f\longmapsto [(\beta,x)\mapsto (\partial^{\beta})^{E}f(x)], 
\]
for all $m\in\N_{0}$ and the same with $\K$ instead of $E$. 
The family $(T^{E}_{m},T^{\K}_{m})_{m\in\N_{0}}$ is a strong generator for $(\mathcal{CV}^{k},E)$ because 
\[
 (\partial^{\beta})^{\K}(e'\circ f)(x)=e'\bigl((\partial^{\beta})^{E}f(x)\bigr),\quad (\beta,x)\in\omega_{m},
\]
for all $e'\in E'$, $f\in\mathcal{CV}^{k}(\Omega,E)$ and $m\in\N_{0}$ due to the linearity and 
continuity of $e'\in E'$. 
In addition, $e'\circ f\in\ker P(\partial)^{\K}$ for all $e'\in E'$ and 
$f\in\mathcal{CV}^{k}_{P(\partial)}(\Omega,E)$, 
which implies that $(T^{E}_{m},T^{\K}_{m})_{m\in\N_{0}}$ is also a strong generator 
for $(\mathcal{CV}^{k}_{P(\partial)},E)$.

For consistency we need to prove that 
\[
(\partial^{\beta})^{E}S(u)(x)=u(\delta_{x}\circ(\partial^{\beta})^{\K}), \quad (\beta,x)\in\omega_{m},
\]
for all $u\in\mathcal{CV}^{k}(\Omega)\varepsilon E$ resp.\ 
$u\in\mathcal{CV}^{k}_{P(\partial)}(\Omega)\varepsilon E$. 
This follows from the subsequent \prettyref{prop:diff_cons_barrelled} b) since 
$\FV=\mathcal{CV}^{k}(\Omega)$ resp.\ $\FV=\mathcal{CV}^{k}_{P(\partial)}(\Omega)$ is barrelled and 
$\mathcal{V}^{k}$ locally bounded away from zero on $\Omega$.
Thus $(T^{E}_{m},T^{\K}_{m})_{m\in\N_{0}}$ is a consistent generator for $(\mathcal{CV}^{k},E)$. In addition, 
we have with $P(\partial)^{E}$ from \prettyref{ex:weighted_smooth_functions} g) that
\begin{align}\label{eq:cons_partial_diff}
  P(\partial)^{E}(S(u))(x)
&=\sum_{i=1}^{n}a_{i}(x)(\partial^{\beta_{i}})^{E}(S(u))(x)
 =u\bigl(\sum_{i=1}^{n}a_{i}(x)(\delta_{x}\circ (\partial^{\beta_{i}})^{\K})\bigr)\nonumber\\
&=u(\delta_{x}\circ P(\partial)^{\K})=0,\quad x\in\Omega,
\end{align}
for every $u\in \mathcal{CV}^{k}_{P(\partial)}(\Omega)\varepsilon E$. This yields $S(u)\in\ker P(\partial)^{E}$ 
for all $u\in \mathcal{CV}^{k}_{P(\partial)}(\Omega)\varepsilon E$. 
Therefore $(T^{E}_{m},T^{\K}_{m})_{m\in\N_{0}}$ is a consistent generator for 
$(\mathcal{CV}^{k}_{P(\partial)},E)$ as well.
\end{proof}

Let us turn to the postponed part in the proof of consistency. 
We denote by $\mathcal{CW}(\Omega)$ the space of scalar-valued continuous functions on 
an open set $\Omega\subset\R^{d}$ with the topology of uniform convergence on compact subsets, 
i.e.\ the weighted topology given by the family of weights
$\mathcal{W}:=\mathcal{W}^{0}:=\{\chi_{K}\;|\;K\subset\Omega\;\text{compact}\}$, and we set 
$\delta(x):=\delta_{x}$ for $x\in\Omega$.

\begin{prop}\label{prop:diff_cons_barrelled}
Let $\Omega\subset\R^{d}$ be open, $k\in\N_{\infty}$ and $\FV$ a $\dom$-space. 
\begin{enumerate}
\item[a)] If $T\in L(\FV,\mathcal{CW}(\Omega))$, then $\delta\circ T\in\mathcal{C}(\Omega,\FV_{\gamma}')$. 
\item[b)] If $T\in L(\FV,\mathcal{CW}^{1}(\Omega))$ and $\FV$ is barrelled, then 
\begin{align*}
  (\partial^{e_{n}})^{\FV_{\kappa}'}(\delta\circ T)(x)
&=\lim_{h\to 0}\frac{\delta_{x+he_{n}}\circ T-\delta_{x}\circ T}{h}\\
& =\delta_{x}\circ (\partial^{e_{n}})^{\K}\circ T,\quad x\in\Omega,\;1\leq n\leq d,
\end{align*}
and $\delta\circ T\in\mathcal{C}^{1}(\Omega,\FV_{\kappa}')$.
\item[c)] If the inclusion $I\colon\FV\to\mathcal{CW}^{k}(\Omega)$, $f\mapsto f$, 
is continuous and $\FV$ barrelled, then $S(u)\in\mathcal{C}^{k}(\Omega,E)$ and 
\[
(\partial^{\beta})^{E}S(u)(x)=u(\delta_{x}\circ(\partial^{\beta})^{\K}),\quad 
\beta\in\N_{0}^{d},\;|\beta|\leq k,\; x\in\Omega,
\]
for all $u\in\FV\varepsilon E$. 
\end{enumerate}
\end{prop}
\begin{proof}
a) First, if $x\in\Omega$ and $(x_{\tau})_{\tau\in\mathcal{T}}$ is a net in $\Omega$ converging to $x$, then 
we observe that 
\[
(\delta_{x_{\tau}}\circ T)(f)=T(f)(x_{\tau})\to T(f)(x)=(\delta_{x}\circ T)(f)
\]
for every $f\in\FV$ as $T(f)$ is continuous on $\Omega$. 
Second, let $K\subset\Omega$ be compact. Then there are $j\in J$, $m\in M$ and $C>0$ such that
   \[
     \sup_{x\in K}|(\delta_{x}\circ T)(f)|=\sup_{x\in K}|T(f)(x)|
     \leq C|f|_{j,m}
   \]
for every $f\in\FV$. This means that $\{\delta_{x}\circ T\;|\;x\in K\}$ is equicontinuous in $\FV'$. 
The topologies $\sigma(\FV',\FV)$ and $\gamma(\FV',\FV)$ coincide on equicontinuous subsets of $\FV'$, 
implying that the restriction $(\delta\circ T)_{\mid_{K}}\colon K\to \FV_{\gamma}'$ 
is continuous by our first observation. As $\delta\circ T$ is continuous on every compact subset of 
the open set $\Omega\subset\R^{d}$, it follows that 
$\delta\circ T\colon \Omega\to \FV_{\gamma}'$ is well-defined and continuous.

b) Let $x\in\Omega$ and $1\leq n\leq d$. Then there is $\varepsilon>0$ such that $x+he_{n}\in\Omega$ 
for all $h\in\R$ with $0<|h|<\varepsilon$. We note that $\delta\circ T\in\mathcal{C}(\Omega,\FV_{\kappa}')$ 
by part a), which implies $\frac{\delta_{x+he_{n}}\circ T-\delta_{x}\circ T}{h}\in\FV'$. 
For every $f\in\FV$ we have 
\[
\lim_{h\to 0}\frac{\delta_{x+he_{n}}\circ T-\delta_{x}\circ T}{h}(f)
=\lim_{h\to 0}\frac{T(f)(x+he_{n})-T(f)(x)}{h}=(\partial^{e_{n}})^{\K}T(f)(x)
\]
in $\K$ as $T(f)\in\mathcal{C}^{1}(\Omega)$. Therefore $\tfrac{1}{h}(\delta_{x+he_{n}}\circ T-\delta_{x}\circ T)$ 
converges to $\delta_{x}\circ(\partial^{e_{n}})^{\K}\circ T$ in $\FV_{\sigma}'$ and thus in 
$\FV_{\kappa}'$ by the Banach--Steinhaus theorem as well. In particular, we obtain 
\[
\delta_{x}\circ (\partial^{e_{n}})^{\K}\circ T
=\lim_{h\to 0}\frac{\delta_{x+he_{n}}\circ T-\delta_{x}\circ T}{h}
=(\partial^{e_{n}})^{\FV_{\kappa}'}(\delta\circ T)(x)
\]
in $\FV_{\kappa}'$. Moreover, $\delta\circ(\partial^{e_{n}})^{\K}\circ T\in\mathcal{C}(\Omega,\FV_{\kappa}')$ 
by part a) as $(\partial^{e_{n}})^{\K}\circ T\in L(\FV,\mathcal{CW}(\Omega))$. 
Hence we deduce that $\delta\circ T\in \mathcal{C}^{1}(\Omega,\FV_{\kappa}')$.

c) We prove our claim by induction on the order of differentiation. 
Let $u\in\FV\varepsilon E$. For $\beta\in\N_{0}^{d}$ 
with $|\beta|=0$ we get $S(u)=u\circ\delta\in\mathcal{C}(\Omega,E)$ 
from part a) with $T=I$. Further, 
\[
(\partial^{\beta})^{E}S(u)(x)=S(u)(x)=u(\delta_{x})=u(\delta_{x}\circ(\partial^{\beta})^{\K}),\quad x\in\Omega.
\]
Let $m\in\N_{0}$, $m< k$, such that $S(u)\in\mathcal{C}^{m}(\Omega,E)$ and 
\begin{equation}\label{eq:diff_induction}
(\partial^{\beta})^{E}S(u)(x)=u(\delta_{x}\circ(\partial^{\beta})^{\K}),\quad x\in\Omega,
\end{equation}
for all $\beta\in\N_{0}^{d}$ with $|\beta|\leq m$. 
Let $\beta\in\N_{0}^{d}$ with $|\beta|=m+1\leq k$. Then there is $1\leq n\leq d$ and 
$\widetilde{\beta}\in\N_{0}^{d}$ with $|\widetilde{\beta}|=m$ such 
that $\beta=e_{n}+\widetilde{\beta}$.
The barrelledness of $\FV$ yields that 
$\tfrac{1}{h}(\delta_{x+he_{n}}\circ(\partial^{\widetilde{\beta}})^{\K}
-\delta_{x}\circ(\partial^{\widetilde{\beta}})^{\K})$ converges to 
$\delta_{x}\circ (\partial^{e_{n}})^{\K}\circ(\partial^{\widetilde{\beta}})^{\K}$ 
in $\FV_{\kappa}'$ for every $x\in\Omega$ by part b)
with $T:=(\partial^{\widetilde{\beta}})^{\K}$. 
Therefore we derive from $\delta_{x}\circ(\partial^{e_{n}})^{\K}
\circ(\partial^{\widetilde{\beta}})^{\K}=\delta_{x}\circ(\partial^{\beta})^{\K}$ by Schwarz' theorem that
\begin{align*}
u(\delta_{x}\circ(\partial^{\beta})^{\K})
&=\lim_{h\to 0}\frac{1}{h}\bigl(u(\delta_{x+he_{n}}\circ(\partial^{\widetilde{\beta}})^{\K})
 -u(\delta_{x}\circ(\partial^{\widetilde{\beta}})^{\K})\bigr)\\
&\underset{\mathclap{\eqref{eq:diff_induction}}}{=}\;
 \lim_{h\to 0}\frac{1}{h}\bigl((\partial^{\widetilde{\beta}})^{E}S(u)(x+he_{n})
 -(\partial^{\widetilde{\beta}})^{E}S(u)(x)\bigr)\\
&=(\partial^{e_{n}})^{E}(\partial^{\widetilde{\beta}})^{E}S(u)(x)
\end{align*}
for every $x\in\Omega$. 
Moreover, 
$\delta\circ(\partial^{\beta})^{\K}
=(\partial^{e_{n}})^{\FV_{\kappa}'}(\delta\circ T)\in\mathcal{C}(\Omega,\FV_{\kappa}')$ 
for $T=(\partial^{\widetilde{\beta}})^{\K}$ by part b). 
Hence we have $S(u)\in\mathcal{C}^{m+1}(\Omega,E)$ and it follows from Schwarz' theorem again that 
\[
u(\delta_{x}\circ(\partial^{\beta})^{\K})
=(\partial^{e_{n}})^{E}(\partial^{\widetilde{\beta}})^{E}S(u)(x)
=(\partial^{\beta})^{E}S(u)(x),\quad x\in\Omega.\qedhere
\]
\end{proof}

Part a) of the preceding proposition is just a modification of \cite[4.1 Lemma, p.\ 198]{B1},
where $\FV=\mathcal{CV}(\Omega)$ is the Nachbin-weighted space of continuous functions and 
$T=\id$, and holds more general for $k_{\R}$-spaces $\Omega$ (see \prettyref{lem:bier}).

\begin{thm}\label{thm:linearisation}
Let $(T^{E}_{m},T^{\K}_{m})_{m\in M}$ be a consistent generator for $(\mathcal{FV},E)$. 
Then the map $S\colon \FV\varepsilon E\to\FVE$ is an isomorphism into, i.e.\ 
the spaces $\FV$ and $\FVE$ are $\varepsilon$-into-compatible.
\end{thm}
\begin{proof}
First, we show that $S(\FV\varepsilon E)\subset \FVE$. Let $u\in\FV\varepsilon E$. Due to the consistency of 
$(T^{E}_{m},T^{\K}_{m})_{m\in M}$ we have $S(u)\in\operatorname{AP}(\Omega,E)\cap\dom T^{E}_{m}$ and
\[
 \bigl(T^{E}_{m}S(u)\bigr)(x)=u(T^{\K}_{m,x}),\quad m\in M\; x\in \omega_{m}.
\]
Furthermore, we get by \prettyref{lem:topology_eps} for every $j\in J$, $m\in M$ and $\alpha\in\mathfrak{A}$
\begin{equation}\label{thm15.1}
|S(u)|_{j,m,\alpha}=\sup_{x\in\omega_{m}}p_{\alpha}\bigl(T^{E}_{m}(S(u))(x)\bigr)\nu_{j,m}(x)
=\|u\|_{j,m,\alpha}<\infty,
\end{equation}
implying $S(u)\in \FVE$ and the continuity of $S$.
Moreover, we deduce from \eqref{thm15.1} that $S$ is injective and that the inverse of $S$ on the range of 
$S$ is also continuous.
\end{proof}

\begin{rem}\label{rem:isometry}
If $J$, $M$ and $\mathfrak{A}$ are countable, then $S$ is an isometry with respect to the induced metrics on 
$\FVE$ and $\FV\varepsilon E$ by \eqref{thm15.1}. 
\end{rem}

The basic idea for \prettyref{thm:linearisation} was derived from analysing the proof of an analogous statement 
for Bierstedt's weighted spaces $\mathcal{CV}(\Omega,E)$ and $\mathcal{CV}_{0}(\Omega,E)$ 
of continuous functions already mentioned in the introduction 
(see \cite[4.2 Lemma, 4.3 Folgerung, p.\ 199--200]{B1} and \cite[2.1 Satz, p.\ 137]{B2}). 
\section{\texorpdfstring{$\varepsilon$}{epsilon}-compatibility}
\label{sect:eps-prod}
Now, we try to answer the natural question. When is $S$ surjective? 
The strength of a generator and a weaker concept to define a natural 
$E$-valued version of $\FV$ come into play to answer the question on the surjectivity of our key map $S$.
Let $\FV$ be a $\dom$-space. 
We define the linear space of \emph{\gls{Evalued_weak_FV_function}} by
\[
\gls{FVE_sigma}:=\{f\colon \Omega\to E\;|\;\forall\;e'\in E':\;e'\circ f\in \FV\}.
\]
Moreover, for $f\in\FVE_{\sigma}$ we define the linear map
\[
\gls{R_f}\colon E'\to \FV,\; R_{f}(e'):=e'\circ f, 
\]
and the dual map 
\[
\gls{R_ft}\colon \FV'\to E'^{\star}, \; f'\longmapsto \bigl[ e'\mapsto f'\bigl(R_{f}(e')\bigr) \bigr],
\]
where $\gls{E'star}$ is the algebraic dual of $E'$. Furthermore, we set
\[
\gls{FVE_kappa}:=\{f\in \FVE_{\sigma}\;|\;\forall\; \alpha\in \mathfrak{A}:
\;R_{f}(B_{\alpha}^{\circ})\;\text{relatively compact in}\;\FV\}
\]
where $B_{\alpha}:=\{x\in E\;|\; p_{\alpha}(x)<1\}$ for $\alpha\in\mathfrak{A}$.
Next, we give a sufficient condition for the inclusion $\FVE\subset\FVE_{\sigma}$ 
by means of the family $(T^{E}_{m},T^{\K}_{m})_{m\in M}$. 

\begin{lem}\label{lem:strong_is_weak}
If $(T^{E}_{m},T^{\K}_{m})_{m\in M}$ is a strong generator for $(\mathcal{FV},E)$, 
then we have $\FVE\subset \FVE_{\sigma}$ and
\begin{equation}\label{T.3.1}
\sup_{e'\in B_{\alpha}^{\circ}}|R_{f}(e')|_{j,m}=|f|_{j,m,\alpha}
\end{equation}
for every $f\in\FVE$, $j\in J$, $m\in M$ and $\alpha\in\mathfrak{A}$.
\end{lem}
\begin{proof}
Let $f\in \FVE$. We have $e'\circ f\in\mathcal{F}(\Omega)$ for every $e'\in E'$ 
since $(T^{E}_{m},T^{\K}_{m})_{m\in M}$ is a strong generator. Moreover, we have
\begin{align}\label{T.3.2}
|R_{f}(e')|_{j,m}&=|e'\circ f|_{j,m}=\sup_{x\in\omega_{m}}
\bigl|T^{\K}_{m}(e'\circ f)(x)\bigr|\nu_{j,m}(x)\notag\\
& =\sup_{x\in\omega_{m}}\bigl|e'\bigl(T^{E}_{m}(f)(x)\bigr)\bigr|\nu_{j,m}(x)
=\sup_{x\in N_{j,m}(f)}|e'(x)|
\end{align}
for every $j\in J$ and $m\in M$ with the set $N_{j,m}(f)$ from \prettyref{def:weighted_space}. We note that $N_{j,m}(f)$ is bounded in $E$ by \prettyref{def:weighted_space} 
and thus weakly bounded, implying that the right-hand side 
of \eqref{T.3.2} is finite. Hence we conclude $f\in\FVE_{\sigma}$.
Further, we observe that
\[
\sup_{e'\in B_{\alpha}^{\circ}}|R_{f}(e')|_{j,m}=|f|_{j,m,\alpha}
\]
for every $j\in J$, $m\in M$ and $\alpha\in\mathfrak{A}$ due to \cite[Proposition 22.14, p.\ 256]{meisevogt1997}. 
\end{proof}

Now, we phrase some sufficient conditions for $\FVE\subset\FVE_{\kappa}$ 
to hold which is one of the key points regarding the surjectivity of $S$.

\begin{lem}\label{lem:FVE_rel_comp}
If $(T^{E}_{m},T^{\K}_{m})_{m\in M}$ is a strong generator for $(\mathcal{FV},E)$ 
and one of the following conditions is fulfilled, then $\FVE\subset\FVE_{\kappa}$.
\begin{enumerate}
\item [a)] $\FV$ is a semi-Montel space.
\item [b)] $E$ is a semi-Montel or Schwartz space.
\item [c)] $\forall\;f\in\FVE,\, j\in J,\,m\in M\;\exists\; K\in\gamma(E):\;N_{j,m}(f)\subset K$.
\end{enumerate}
\end{lem}
\begin{proof}
Let $f\in \FVE$. By virtue of \prettyref{lem:strong_is_weak} we already have $f\in \FVE_{\sigma}$.

a) For every $j\in J$, $m\in M$ and $\alpha\in \mathfrak{A}$ 
we derive from 
\[
\sup_{e'\in B_{\alpha}^{\circ}}|R_{f}(e')|_{j,m}\underset{\eqref{T.3.1}}{=}|f|_{j,m,\alpha}<\infty
\]
that $R_{f}(B^{\circ}_{\alpha})$ is bounded and thus relatively compact in the semi-Montel space $\FV$.

c) It follows from \eqref{T.3.2} that $R_{f}\in L(E_{\gamma}',\FV)$. 
Further, the polar $B_{\alpha}^{\circ}$ is relatively compact in $E_{\gamma}'$ for every $\alpha\in\mathfrak{A}$ 
by the Alao\u{g}lu--Bourbaki theorem. 
The continuity of $R_{f}$ implies that $R_{f}(B_{\alpha}^{\circ})$ is relatively compact as well.

b) Let $j\in J$ and $m\in M$. The set $K:=N_{j,m}(f)$ is bounded in $E$ by \prettyref{def:weighted_space}. 
We deduce that $K$ is already precompact in $E$ by \cite[10.4.3 Corollary, p.\ 202]{Jarchow} 
if $E$ is a Schwartz space resp.\ since it is relatively compact if $E$ is a semi-Montel space. 
Hence the statement follows from c). 
\end{proof}

Let us turn to sufficient conditions for $\FVE\cong \FV\varepsilon E$. 
For the lcHs $E$ we denote by $\gls{J}\colon E\to E'^{\star}$, $x\longmapsto [e'\mapsto e'(x)]$, 
the canonical injection.

\begin{cond}\label{cond:surjectivity_linearisation}
Let $(T^{E}_{m},T^{\K}_{m})_{m\in M}$ be a strong generator for $(\mathcal{FV},E)$. Define the following conditions:
\begin{enumerate}
\item [a)] $E$ is complete.
\item [b)] $E$ is quasi-complete and for every $f\in\FVE$ and 
$f'\in\FV'$ there is a bounded net $(f'_{\tau})_{\tau\in\mathcal{T}}$ in $\FV'$ 
converging to $f'$ in $\FV_{\kappa}'$ such that $R_{f}^{t}(f'_{\tau})\in \mathcal{J}(E)$ 
for every $\tau\in\mathcal{T}$.
\item [c)] $E$ is sequentially complete and for every $f\in\FVE$ and 
$f'\in\FV'$ there is a sequence 
$(f'_{n})_{n\in\N}$ in $\FV'$ converging to $f'$ in $\FV_{\kappa}'$ such that 
$R_{f}^{t}(f'_{n})\in \mathcal{J}(E)$ for every $n\in\N$.
\item [d)] $E$ is locally complete and for every $f\in\FVE$ and $f'\in\FV'$ there is a sequence 
$(f'_{n})_{n\in\N}$ in $\FV'$ locally converging to $f'$ in $\FV_{\kappa}'$
such that $R_{f}^{t}(f'_{n})\in \mathcal{J}(E)$ for every $n\in\N$.
\item [e)] 
$
\forall\;f\in\FVE,\,j\in J,\,m\in M\;\exists\; K\in\tau(E):\;N_{j,m}(f)\subset K.
$
\end{enumerate}
\end{cond}

\begin{thm}\label{thm:full_linearisation}
Let $(T^{E}_{m},T^{\K}_{m})_{m\in M}$ be a consistent generator for $(\mathcal{FV},E)$ and let
$\FVE\subset\FVE_{\kappa}$. If one of the Conditions \ref{cond:surjectivity_linearisation} is fulfilled, 
then the map $S\colon\FV\varepsilon E\to \FVE$ is an isomorphism, 
i.e.\ $\FV$ and $\FVE$ are $\varepsilon$-compatible.
The inverse of $S$ is given by the map 
\[
R^{t}\colon \FVE \to \FV\varepsilon E,\;f\mapsto \mathcal{J}^{-1}\circ R_{f}^{t},
\]
where $\mathcal{J}\colon E\to E'^{\star}$ is the canonical injection and 
\[
R_{f}^{t}\colon \FV'\to E'^{\star}, \; f'\longmapsto \bigl[ e'\mapsto f'\bigl(R_{f}(e')\bigr) \bigr],
\]
with $R_{f}(e')=e'\circ f$.
\end{thm}
\begin{proof}
Due to \prettyref{thm:linearisation} we only have to show that $S$ is surjective. 
We equip $\mathcal{J}(E)$ with the system of seminorms given by 
\begin{equation}\label{wT.1.1}
p_{B^{\circ}_{\alpha}}(\mathcal{J}(x)):=\sup_{e'\in B^{\circ}_{\alpha}}|\mathcal{J}(x)(e')|
=p_{\alpha}(x),\quad x\in E,
\end{equation}
for every $\alpha\in \mathfrak{A}$.
Let $f\in \FVE$. We consider the dual map $R_{f}^{t}$ and claim that 
$R_{f}^{t}\in L(\FV_{\kappa}',\mathcal{J}(E))$. 
Indeed, we have
\begin{equation}\label{wT.1.2}
p_{B^{\circ}_{\alpha}}\bigl(R_{f}^{t}(y)\bigr)
=\sup_{e'\in B^{\circ}_{\alpha}}\bigl|y\bigl(R_{f}(e')\bigr)\bigr|
=\sup_{x\in R_{f}(B^{\circ}_{\alpha})}|y(x)|
\leq\sup_{x\in K_{\alpha}}|y(x)|
\end{equation}
for all $y\in\FV'$ where $K_{\alpha}:=\overline{R_{f}(B^{\circ}_{\alpha})}$. 
Since $\FVE\subset\FVE_{\kappa}$, the set $R_{f}(B^{\circ}_{\alpha})$ is absolutely convex 
and relatively compact, implying that $K_{\alpha}$ is absolutely convex and compact in 
$\FV$ by \cite[6.2.1 Proposition, p.\ 103]{Jarchow}. 
Further, we have for all $e'\in E'$ and $x\in\Omega$
\begin{equation}\label{wT.1.2a} 
R_{f}^{t}(\delta_{x})(e')=\delta_{x}(e'\circ f)=e'\bigl(f(x)\bigr)=\mathcal{J}\bigl(f(x)\bigr)(e')
\end{equation}
and thus $R_{f}^{t}(\delta_{x})\in\mathcal{J}(E)$.

a) Let $E$ be complete and $f'\in \FV'$. 
Since the span of $\{\delta_{x}\;|\; x\in \Omega\}$ is dense in
$\mathcal{F}(\Omega)_{\kappa}'$ by the bipolar theorem, there is a net $(f_{\tau}')$ 
converging to $f'$ in $\FV_{\kappa}'$ with $R_{f}^{t}(f_{\tau}')\in\mathcal{J}(E)$ by \eqref{wT.1.2a}. As 
\begin{equation}\label{wT.1.3} 
p_{B^{\circ}_{\alpha}}\bigl(R_{f}^{t}(f_{\tau}')-R_{f}^{t}(f')\bigr)
 \underset{\eqref{wT.1.2}}{\leq} \sup_{x\in K_{\alpha}}|(f_{\tau}'-f')(x)|\to 0,
\end{equation}
for all $\alpha\in \mathfrak{A}$, we gain that $(R_{f}^{t}(f_{\tau}'))$ is a Cauchy net 
in the complete space $\mathcal{J}(E)$.
Hence it has a limit $g\in\mathcal{J}(E)$ which coincides with $R_{f}^{t}(f')$ since
\begin{align*}
\qquad p_{B^{\circ}_{\alpha}}\bigl(g-R_{f}^{t}(f')\bigr)
&\leq p_{B^{\circ}_{\alpha}}\bigl(g-R_{f}^{t}(f_{\tau}')\bigr)
 +p_{B^{\circ}_{\alpha}}\bigl(R_{f}^{t}(f_{\tau}')-R_{f}^{t}(f')\bigr)\\
 &\;\;\mathclap{\underset{\eqref{wT.1.3}}{\leq}}\;\;\; p_{B^{\circ}_{\alpha}}\bigl(g-R_{f}^{t}(f_{\tau}')\bigr)
 + \sup_{x\in K_{\alpha}}\bigl|(f_{\tau}'-f')(x)\bigr|\to 0
\end{align*}
for all $\alpha\in \mathfrak{A}$. We conclude that $R_{f}^{t}(f')\in\mathcal{J}(E)$ for every $f'\in \FV'$. 

b) Let \prettyref{cond:surjectivity_linearisation} b) hold and $f'\in\FV'$. Then there is a bounded net 
$(f'_{\tau})_{\tau\in\mathcal{T}}$ in $\FV'$ converging to $f'$ in $\FV_{\kappa}'$ 
such that $R_{f}^{t}(f'_{\tau})\in \mathcal{J}(E)$ for every 
$\tau\in\mathcal{T}$. Due to \eqref{wT.1.2} we obtain that $(R_{f}^{t}(f_{\tau}'))$ 
is a bounded Cauchy net in the quasi-complete space $\mathcal{J}(E)$ 
converging to $R_{f}^{t}(f')\in\mathcal{J}(E)$.

c) Let \prettyref{cond:surjectivity_linearisation} c) hold and $f'\in\FV'$. Then there is a sequence 
$(f'_{n})_{n\in\N}$ in $\FV'$ converging to $f'$ in $\FV_{\kappa}'$ such that 
$R_{f}^{t}(f'_{n})\in \mathcal{J}(E)$ for every $n\in\N$. Again \eqref{wT.1.2} implies that 
$(R_{f}^{t}(f_{n}'))$ is a Cauchy sequence in the sequentially complete space $\mathcal{J}(E)$ 
which converges to $R_{f}^{t}(f')\in\mathcal{J}(E)$.

d) Let \prettyref{cond:surjectivity_linearisation} d) hold and $f'\in\FV'$. Then there is an absolutely convex, 
bounded subset $D\subset\FV_{\kappa}'$ and a sequence $(f'_{n})_{n\in\N}$ in $\FV'$ converging to $f'$ 
in $(\FV_{\kappa}')_{D}$ such that $R_{f}^{t}(f_{n})\in \mathcal{J}(E)$ for every $n\in\N$.
Let $r>0$ and $f'_{n}-f'_{k}\in rD$. Then $R_{f}^{t}(f'_{n}-f'_{k})\in r(R_{f}^{t}(D)\cap\mathcal{J}(E))$, implying
\[
\{r>0\;|\; f'_{n}-f'_{k}\in rD\}\subset 
\bigl\{r>0\;|\; R_{f}^{t}(f'_{n}-f'_{k})\in r \overline{(R_{f}^{t}(D)\cap\mathcal{J}(E))}^{\mathcal{J}(E)}\bigr\}.
\]
Setting $B:=\overline{R_{f}^{t}(D)\cap\mathcal{J}(E)}^{\mathcal{J}(E)}$, we derive
\[
 q_{B}(R_{f}^{t}(f'_{n}-f'_{k}))\leq q_{D}(f'_{n}-f'_{k})
\]
where $q_{B}$ and $q_{D}$ are the gauge functionals of $B$ resp.\ $D$. 
The set $R_{f}^{t}(D)\cap\mathcal{J}(E)$ is absolutely convex as the intersection of two absolutely convex sets 
and it is bounded by \eqref{wT.1.2} and the boundedness of $D$. 
So $B$, being the closure of a disk, is a disk as well. 
Since $(f'_{n})$ is a Cauchy sequence in $(\FV_{\kappa}')_{D}$, we conclude
that $(R_{f}'(f'_{n}))$ is a Cauchy sequence in $\mathcal{J}(E)_{B}$. 
The set $B$ is a closed disk in the locally complete space $\mathcal{J}(E)$ and 
hence a Banach disk by \cite[10.2.1 Proposition, p.\ 197]{Jarchow}. 
Thus $\mathcal{J}(E)_{B}$ is a Banach space and $(R_{f}^{t}(f'_{n}))$ has a limit 
$g\in\mathcal{J}(E)_{B}$. The continuity of the canonical injection 
$\mathcal{J}(E)_{B}\hookrightarrow\mathcal{J}(E)$ implies that $(R_{f}^{t}(f'_{n}))$ converges to 
$g$ in $\mathcal{J}(E)$ as well. As in a) we obtain that $R_{f}^{t}(f')=g\in\mathcal{J}(E)$.

e) Let \prettyref{cond:surjectivity_linearisation} e) be fulfilled. Let $f\in \FVE$ and $e'\in E'$. 
For every $f'\in \FV'$ there are $j\in J$, $m\in M$ and $C>0$ such that
\[
|R_{f}^{t}(f')(e')|\leq C |R_{f}(e')|_{j,m}
\underset{\eqref{T.3.2}}{=}C\sup_{x\in N_{j,m}(f)}|e'(x)|
\]
because $(T^{E}_{m},T^{\K}_{m})_{m\in M}$ is a strong generator.
Since there is $K\in\tau(E)$ such that $N_{j,m}(f)\subset K$, we have
\[
|R_{f}^{t}(f')(e')|\leq C \sup_{x\in K}|e'(x)|,
\]
implying $R_{f}^{t}(f')\in (E'_{\tau})'=\mathcal{J}(E)$ by the Mackey--Arens theorem.

Therefore we obtain that $R_{f}^{t}\in L(\FV_{\kappa}',\mathcal{J}(E))$.
So we get for all $\alpha\in \mathfrak{A}$ and $y\in \mathcal{F}(\Omega)'$ 
\[
p_{\alpha}\bigl((\mathcal{J}^{-1}\circ R_{f}^{t})(y)\bigr)
\underset{\eqref{wT.1.1}}{=}p_{B^{\circ}_{\alpha}}\bigl(\mathcal{J}((\mathcal{J}^{-1}\circ R_{f}^{t})(y))\bigr)
= p_{B^{\circ}_{\alpha}}\bigl(R_{f}^{t}(y)\bigr)
\underset{\eqref{wT.1.2}}{\leq}\sup_{x\in K_{\alpha}}|y(x)|.
\]
This implies $\mathcal{J}^{-1}\circ R_{f}^{t}\in L(\FV_{\kappa}', E)=\FV\varepsilon E$ (as linear spaces)
and we gain
\[
S(\mathcal{J}^{-1}\circ R_{f}^{t})(x)=\mathcal{J}^{-1}\bigl(R_{f}^{t}(\delta_{x})\bigr)
\underset{\eqref{wT.1.2a}}{=}\mathcal{J}^{-1}\bigl(\mathcal{J}(f(x))\bigr)=f(x)
\]
for every $x\in \Omega$. Thus $S(\mathcal{J}^{-1}\circ R_{f}^{t})=f$, proving the surjectivity of $S$.
\end{proof}

Further sufficient conditions for $S$ being a topological isomorphism can be found in 
\prettyref{prop:weak_strong_principle}, \prettyref{prop:reverse_Schauder} 
and \prettyref{thm:Schauder_coeff_space}. 
In particular, we get the following corollary as a special case of \prettyref{thm:full_linearisation}.

\begin{cor}\label{cor:full_linearisation}
Let $(T^{E}_{m},T^{\K}_{m})_{m\in M}$ be a strong, consistent generator for $(\mathcal{FV},E)$. If 
\begin{enumerate}
\item[(i)] $\FV$ is a semi-Montel space and $E$ complete, or
\item[(ii)] $\FV$ is a Fr\'echet--Schwartz space and $E$ locally complete, or
\item[(iii)] $E$ is a semi-Montel space, or
\item[(iv)] $\forall\;f\in\FVE,\,j\in J,\,m\in M\;\exists\; K\in\kappa(E):\;N_{j,m}(f)\subset K$, 
\end{enumerate}
then $\FV$ and $\FVE$ are $\varepsilon$-compatible, in particluar, $\FVE\cong\FV\varepsilon E$.
\end{cor}
\begin{proof}
(i) Follows from \prettyref{lem:FVE_rel_comp} a) and \prettyref{thm:full_linearisation} 
with \prettyref{cond:surjectivity_linearisation} a). 

(ii) If $\FV$ is a Fr\'echet--Schwartz space, then we have 
\[
\overline{\operatorname{span}\{\delta_{x}\;|\;x\in\Omega\}}^{\operatorname{lc}}
=\overline{\operatorname{span}\{\delta_{x}\;|\;x\in\Omega\}}^{\FV_{b}'}
=\overline{\operatorname{span}\{\delta_{x}\;|\;x\in\Omega\}}^{\FV_{\kappa}'}
=\FV'
\]
by \cite[Lemma 6 (b), p.\ 231]{B/F/J} and the bipolar theorem 
where $\overline{\operatorname{span}\{\delta_{x}\;|\;x\in\Omega\}}^{\operatorname{lc}}$ is the local closure of 
$\operatorname{span}\{\delta_{x}\;|\;x\in\Omega\}$ in $\FV_{b}'$. Hence for every $f'\in\FV'$ there is a 
sequence $(f_{n}')$ in the span of $\{\delta_{x}\;|\;x\in\Omega\}$ which converges locally to $f'$ in 
$\FV_{\kappa}'$. 
Due to \eqref{wT.1.2a} we know that $R_{f}^{t}(f_{n}')\in\mathcal{J}(E)$ for every $f\in\FVE$ and $n\in\N$.
Since Fr\'echet--Schwartz spaces are also semi-Montel spaces, the statement follows from 
\prettyref{lem:FVE_rel_comp} a) and \prettyref{thm:full_linearisation} 
with \prettyref{cond:surjectivity_linearisation} d). 

(iv) Follows from \prettyref{lem:FVE_rel_comp} c) and \prettyref{thm:full_linearisation} 
with \prettyref{cond:surjectivity_linearisation} e). 

(iii) Is a special case of (iv) since the set $K:=\oacx(N_{j,m}(f))$ 
is absolutely convex and compact in the semi-Montel space $E$ by \cite[6.2.1 Proposition, p.\ 103]{Jarchow} 
and \cite[6.7.1 Proposition, p.\ 112]{Jarchow} for every $f\in\FVE$, $j\in J$ and $m\in M$. 
\end{proof}

\begin{rem}
Linearisations of spaces $\FVE_{\sigma}$ of weak $E$-valued functions, where $\FV$ need not be a $\dom$-space, are 
treated in \cite{kruse_2022}.
\end{rem}

Let us apply our preceding results to our weighted spaces of $k$-times continuously 
partially differentiable functions on an open set $\Omega\subset\R^{d}$ with $k\in\N_{\infty}$.

\begin{exa}\label{ex:weighted_diff}
Let $E$ be an lcHs, $k\in\N_{\infty}$, $\mathcal{V}^{k}$ be a directed family of weights 
which is locally bounded away from zero on an open set $\Omega\subset\R^{d}$.
\begin{enumerate}
\item [a)] $\mathcal{CV}^{k}(\Omega,E)\cong \mathcal{CV}^{k}(\Omega)\varepsilon E$ if $E$ is a semi-Montel space 
and $\mathcal{CV}^{k}(\Omega)$ barrelled.
\item [b)] $\mathcal{CV}^{k}_{P(\partial)}(\Omega,E)\cong \mathcal{CV}^{k}_{P(\partial)}(\Omega)\varepsilon E$ if 
$E$ is a semi-Montel space and $\mathcal{CV}^{k}_{P(\partial)}(\Omega)$ barrelled.
\item [c)] $\mathcal{CV}^{k}(\Omega,E)\cong \mathcal{CV}^{k}(\Omega)\varepsilon E$ if $E$ is complete and 
$\mathcal{CV}^{k}(\Omega)$ a Montel space.
\item [d)] $\mathcal{CV}^{k}_{P(\partial)}(\Omega,E)\cong \mathcal{CV}^{k}_{P(\partial)}(\Omega)\varepsilon E$ 
if $E$ is complete and $\mathcal{CV}^{k}_{P(\partial)}(\Omega)$ a Montel space.
\item [e)] $\mathcal{CV}^{k}(\Omega,E)\cong \mathcal{CV}^{k}(\Omega)\varepsilon E$ if $E$ is locally complete and 
$\mathcal{CV}^{k}(\Omega)$ a Fr\'echet--Schwartz space.
\item [f)] $\mathcal{CV}^{k}_{P(\partial)}(\Omega,E)\cong \mathcal{CV}^{k}_{P(\partial)}(\Omega)\varepsilon E$ 
if $E$ is locally complete and $\mathcal{CV}^{k}_{P(\partial)}(\Omega)$ a Fr\'echet--Schwartz space.
\end{enumerate} 
\end{exa} 
\begin{proof}
The generator of $(\mathcal{CV}^{k},E)$ and $(\mathcal{CV}^{k}_{P(\partial)},E)$ is strong and 
consistent by \prettyref{prop:weighted_diff_strong_cons}.
From \prettyref{cor:full_linearisation} (iii) we deduce part a) and b), from (i) part c) and d) 
and from (ii) part e) and f). 
\end{proof}

Closed subspaces of Fr\'echet--Schwartz spaces are also Fr\'echet--Schwartz spaces 
by \cite[Proposition 24.18, p.\ 284]{meisevogt1997}. 
The space $\mathcal{CV}^{\infty}_{P(\partial)}(\Omega)$ is closed in $\mathcal{CV}^{\infty}(\Omega)$ 
if there is an lcHs $Y$ such that 
$P(\partial)_{\mid\mathcal{CV}^{\infty}(\Omega)}\colon\mathcal{CV}^{\infty}(\Omega)\to Y$ is continuous. 
For example, this is fulfilled if the coefficients of $P(\partial)$ belong to $\mathcal{C}(\Omega)$, 
in particular, if $P(\partial):=\Delta$ or $\overline{\partial}$, with $Y:=(\mathcal{C}(\Omega),\tau_{c})$ 
due to $\mathcal{V}^{\infty}$ being locally bounded away from zero.
The spaces $\mathcal{CV}^{k}(\Omega)$ from \prettyref{ex:weighted_smooth_functions} a)(i) 
with $\omega_{m}:=M_{m}\times\Omega$ for all $m\in\N_{0}$, 
where $M_{m}:=\{\beta\in\N_{0}^{d}\;|\;|\beta|\leq \min(m,k)\}$, 
are Fr\'{e}chet spaces and thus barrelled if the $J$ in 
$\mathcal{V}^{k}:=(\nu_{j,m})_{j\in J, m\in\N_{0}}$ is countable by \cite[Proposition 3.7, p.\ 240]{kruse2018_2}. 
Sufficient conditions on the weights that guarantee that $\mathcal{CV}^{\infty}(\Omega)$ is a nuclear 
Fr\'echet space and hence a Schwartz space as well can be found in \cite[Theorem 3.1, p.\ 188]{kruse2018_4}.
For the case $\omega_{m}=\N_{0}^{d}\times\Omega$ see the references given in \cite[p.\ 1]{kruse2018_4}.

If $\mathcal{V}^{k}=\mathcal{W}^{k}$, i.e.\ $\mathcal{C}^{k}(\Omega,E)$ is equipped with its usual topology 
of uniform convergence of all partial derivatives up to order $k$ on compact subsets of $\Omega$, 
\prettyref{ex:weighted_diff} c)+d) can be improved to quasi-complete $E$. 
For $\Omega=\R^{d}$ this can be found in \cite[Proposition 9, p.\ 108, Th\'{e}or\`{e}me 1, p.\ 111]{Schwartz1955}
and for general open $\Omega\subset\R^{d}$ it is already mentioned in \cite[(9), p.\ 236]{Kaballo} 
(without a proof) that $\mathcal{CW}^{k}(\Omega,E)\cong \mathcal{CW}^{k}(\Omega)\varepsilon E$ 
for $k\in\N_{\infty}$ and quasi-complete $E$. 
For $k=\infty$ we even have $\mathcal{CW}^{\infty}(\Omega,E)\cong \mathcal{CW}^{\infty}(\Omega)\varepsilon E$ 
for locally complete $E$ by \cite[p.\ 228]{B/F/J}. Our technique allows us to generalise the first result 
and to get back the second result.

\begin{exa}\label{ex:diff_usual}
Let $E$ be an lcHs, $k\in\N_{\infty}$ and $\Omega\subset\R^{d}$ open. If $k<\infty$ and $E$ has metric ccp, 
or if $k=\infty$ and $E$ is locally complete, then
\begin{enumerate}
\item [a)] $\mathcal{CW}^{k}(\Omega,E)\cong \mathcal{CW}^{k}(\Omega)\varepsilon E$, and
\item [b)] $\mathcal{CW}^{k}_{P(\partial)}(\Omega,E)\cong \mathcal{CW}^{k}_{P(\partial)}(\Omega)\varepsilon E$ if 
$\mathcal{CW}^{k}_{P(\partial)}(\Omega)$ is closed in $\mathcal{CW}^{k}(\Omega)$.
\end{enumerate}
\end{exa}
\begin{proof}
We recall from \prettyref{ex:weighted_smooth_functions} b) that $\mathcal{W}^{k}$ is the family of weights given 
by $\nu_{K,m}(\beta,x):=\chi_{K}(x)$, $(\beta,x)\in M_{m}\times\Omega$, for all $m\in\N_{0}$ 
and compact $K\subset\Omega$ where $M_{m}:=\{\beta\in\N^{d}_{0}\;|\;|\beta|\leq\min(m,k)\}$ and 
$\chi_{K}$ is the characteristic function of $K$. 
We already know that the generator for $(\mathcal{CW}^{k},E)$ and $(\mathcal{CW}^{k}_{P(\partial)},E)$ 
is strong and consistent by \prettyref{prop:weighted_diff_strong_cons}
because $\mathcal{W}^{k}$ is locally bounded away from zero on $\Omega$, and $\mathcal{CW}^{k}(\Omega)$ 
and its closed subspace $\mathcal{CW}^{k}_{P(\partial)}(\Omega)$ are Fr\'echet spaces.  
Let $f\in\mathcal{CW}^{k}(\Omega,E)$, $K\subset\Omega$ be compact, $m\in\N_{0}$ and  
consider
\[
 N_{K,m}(f)=\{(\partial^{\beta})^{E}f(x)\nu_{K,m}(\beta,x)\;|\;x\in \Omega,\,\beta\in M_{m}\}
            =\{0\}\cup\bigcup_{\beta\in M_{m}}(\partial^{\beta})^{E}f(K). 
\]
$N_{K,m}(f)$ is compact since it is a finite union of compact sets. 
Furthermore, the compact sets $\{0\}$ and $(\partial^{\beta})^{E}f(K)$ are metrisable by 
\cite[Chap.\ IX, \S2.10, Proposition 17, p.\ 159]{bourbakiII} and thus their finite union $N_{K,m}(f)$ 
is metrisable as well by \cite[Theorem 1, p.\ 361]{stone} since the compact set $N_{K,m}(f)$ 
is collectionwise normal and locally countably compact by \cite[5.1.18 Theorem, p.\ 305]{engelking}.  
If $E$ has metric ccp, then the set $\oacx(N_{K,m}(f))$ is 
absolutely convex and compact. Thus \prettyref{cor:full_linearisation} (iv) 
settles the case for $k<\infty$. 
If $k=\infty$ and $E$ is locally complete, we observe that 
$K_{\beta}:=\oacx((\partial^{\beta})^{E}f(K))$ for 
$f\in\mathcal{CW}^{\infty}(\Omega,E)$ is absolutely convex and compact 
by \cite[Proposition 2, p.\ 354]{Bonet2002}. Then we have
\[
 N_{K,m}(f)\subset \acx\bigl(\bigcup_{\beta\in M_{m}}K_{\beta}\bigr)
\]
and the set on the right-hand side is absolutely convex and compact 
by \cite[6.7.3 Proposition, p.\ 113]{Jarchow}. 
Again, the statement follows from \prettyref{cor:full_linearisation} (iv). 
\end{proof}

The statement above for $k=\infty$ follows from \prettyref{ex:weighted_diff} e)+f) as well because 
$\mathcal{CW}^{\infty}(\Omega)$ and its closed subspaces are Fr\'echet--Schwartz spaces. 
In the context of differentiability on infinite dimensional spaces the preceding example a) 
remains true for an open subset $\Omega$ of a Fr\'{e}chet space or DFM-space and quasi-complete $E$ by 
\cite[3.2 Corollary, p.\ 286]{meise1977}. 
Like here this can be generalised to $E$ with [metric] ccp. 
A special case of example b) is already known to be a consequence of \cite[Theorem 9, p.\ 232]{B/F/J}, 
namely, if $k=\infty$ and $P(\partial)^{\K}$ is hypoelliptic with constant coefficients. 
In particular, this covers the space of holomorphic functions and the space of harmonic functions. 
Holomorphy on infinite dimensional spaces is treated in \cite[Corollary 6.35, p.\ 332--333]{dineen1981} 
where $\mathcal{V}=\mathcal{W}^{0}$, $\Omega$ is an open subset of a locally convex Hausdorff $k$-space and
$E$ a quasi-complete locally convex Hausdorff space, both over $\C$, which can be generalised to 
$E$ with [metric] ccp in a similar way.

For a second improvement of \prettyref{ex:weighted_diff} for $k=\infty$ to locally complete $E$ without 
the condition that $\mathcal{CV}^{\infty}(\Omega)$ resp.\ $\mathcal{CV}^{\infty}_{P(\partial)}(\Omega)$ 
is a Fr\'echet--Schwartz space we introduce the following conditions on the family 
$\mathcal{V }^{\infty}$ on $(M_{m}\times\Omega)_{m\in\N_{0}}$.
We say that a family $\mathcal{V}^{\infty}$ of weights on $(M_{m}\times\Omega)_{m\in\N_{0}}$ is 
\emph{\gls{C_1controlled}} if 
\begin{align*}
&\phantom{ii}(i)\;\forall\;j \in J,\, m\in\N_{0},\,\beta\in M_{m}:\;
 \nu_{j,m}(\beta,\cdot)\in\mathcal{C}^{1}(\Omega),\\
&\phantom{i}(ii)\;\forall\;j \in J,\, m\in\N_{0},\, \beta,\gamma\in M_{m},x\in\Omega:\;
\nu_{j,m}(\beta,x)=\nu_{j,m}(\gamma,x),\\
&(iii)\;\forall\;j \in J,\, m\in\N_{0}\;\exists\; i\in J,\, k\in\N_{0},\, k\geq m,\, C>0\;
        \forall\;\beta\in M_{m},\,x\in\Omega,\, 1\leq n\leq d:\\
&\phantom{(iii)\;} \bigl|\partial^{e_{n}}\nu_{j,m}(\beta,\cdot)\bigr|(x)
\leq C\nu_{i,k}(\beta,x).
\end{align*}
We say that family $\mathcal{V}^{k}$, $k\in\N_{\infty}$, fulfils condition $\gls{Vinfty}$ if
\begin{align*}
\forall\;m\in\N_{0},\,j\in J\;\exists\;n\in\N_{\geq m},\,i\in J\;&\forall\;\varepsilon >0 \;\exists\; 
K\subset\Omega\;\text{compact}\; \forall\;\beta\in M_{m},\,x\in\Omega\setminus K: \\
&\nu_{j,m}(\beta,x)\leq \varepsilon \nu_{i,n}(\beta,x)
\end{align*}
where $\N_{\geq m}:=\{n\in\N_{0}\;|\;n\geq m\}$.
Here $(V_{\infty})$ stands for \emph{vanishing at infinity} and the condition was introduced in 
\cite[Remark 3.4, p.\ 239]{kruse2018_2} and for $k=0$ in \cite[1.3 Bemerkung, p.\ 189]{B1}.

\begin{exa}\label{ex:weighted_C_1_diff}
Let $E$ be an lcHs and $\mathcal{V}^{\infty}$ a directed $\mathcal{C}^{1}$-controlled family of weights on 
an open convex set $\Omega\subset\R^{d}$ which fulfils $(V_{\infty})$.
If $E$ is locally complete, then
\begin{enumerate}
\item [a)] $\mathcal{CV}^{\infty}(\Omega,E)\cong\mathcal{CV}^{\infty}(\Omega)\varepsilon E$ if 
$\mathcal{CV}^{\infty}(\Omega)$ is barrelled, and
\item [b)] $\mathcal{CV}^{\infty}_{P(\partial)}(\Omega,E)\cong
\mathcal{CV}^{\infty}_{P(\partial)}(\Omega)\varepsilon E$ if $\mathcal{CV}^{\infty}_{P(\partial)}(\Omega)$ 
is barrelled. 
\end{enumerate} 
\end{exa} 
\begin{proof}
We already know that the generator for $(\mathcal{CV}^{\infty},E)$ and $(\mathcal{CV}^{\infty}_{P(\partial)},E)$ 
is strong and consistent by \prettyref{prop:weighted_diff_strong_cons}
because $\mathcal{V}^{\infty}$ is locally bounded away from zero on $\Omega$ as $\nu_{j,m}(\beta,\cdot)$ 
is continuous for all $j\in J$, $m\in\N_{0}$ and $\beta\in M_{m}$.

Let $f\in\mathcal{CV}^{\infty}(\Omega,E)$, $j\in J$, $m\in\N_{0}$ and
$\beta\in M_{m}$. 
We set $g\colon\Omega\to E$, $g(x):=(\partial^{\beta})^{E}f(x)\nu_{j,m}(\beta,x)$, and note that 
\[
(\partial^{e_{n}})^{E}g(x)=(\partial^{\beta+e_{n}})^{E}f(x)\nu_{j,m}(\beta,x)
+(\partial^{\beta})^{E}f(x)\bigl((\partial^{e_{n}})^{\R}\nu_{j,m}(\beta,\cdot)\bigr)(x),\quad x\in\Omega,
\]
for all $1\leq n\leq d$. Since $\mathcal{V}^{\infty}$ is directed and $\mathcal{C}^{1}$-controlled
there are $i_{1},i_{2}\in J$, $k_{1},k_{2}\in\N_{0}$, $k_{1}>m$, $k_{2}\geq m$, and $C_{1},C_{2}>0$ such that
\begin{flalign*}
&\hspace*{0.43cm} p_{\alpha}((\partial^{e_{n}})^{E}g(x))\\
&\leq p_{\alpha}((\partial^{\beta+e_{n}})^{E}f(x))\nu_{j,m}(\beta,x)
 +p_{\alpha}((\partial^{\beta})^{E}f(x))\bigl|(\partial^{e_{n}})^{\R}\nu_{j,m}(\beta,\cdot)\bigr|(x)\\
&\leq C_{1}p_{\alpha}((\partial^{\beta+e_{n}})^{E}f(x))\nu_{i_{1},k_{1}}(\beta,x)
 +C_{2}p_{\alpha}((\partial^{\beta})^{E}f(x))\nu_{i_{2},k_{2}}(\beta,x)\\
&=C_{1}p_{\alpha}((\partial^{\beta+e_{n}})^{E}f(x))\nu_{i_{1},k_{1}}(\beta+e_{n},x)
 +C_{2}p_{\alpha}((\partial^{\beta})^{E}f(x))\nu_{i_{2},k_{2}}(\beta,x)
\end{flalign*}
for all $1\leq n\leq d$ and $\alpha\in\mathfrak{A}$, which implies 
\[
\sup_{\substack{x\in\Omega\\ \gamma\in\N_{0}^{d},|\gamma|\leq 1}}p_{\alpha}((\partial^{\gamma})^{E}g(x))
\leq |f|_{j,m,\alpha}+C_{1}|f|_{i_{1},k_{1},\alpha}+C_{2}|f|_{i_{2},k_{2},\alpha}.
\]
Thus $g$ is (weakly) $\mathcal{C}^{1}_{b}$. 

Due to $(V_{\infty})$ there are $n\in\N_{\geq m}$ and $i\in J$ such that for all $\varepsilon >0$ there is 
a compact set $K\subset\Omega$ such that for all $\beta\in M_{m}$ and $x\in\Omega\setminus K$ we have 
\[
\nu_{j,m}(\beta,x)\leq \varepsilon \nu_{i,n}(\beta,x).
\]
Since $\mathcal{V}^{\infty}$ is directed, we may assume w.l.o.g.\ that 
$\nu_{j,m}(\beta,x)\leq \nu_{i,n}(\beta,x)$ for all $x\in\Omega$. 
This implies that the zeros of $\nu_{i,n}(\beta,\cdot)$ are zeros of $\nu_{j,m}(\beta,\cdot)$. 
We define $h\colon\Omega\to[0,\infty)$ by $h(x):=\nu_{i,n}(\beta,x)/\nu_{j,m}(\beta,x)$ for 
$x\in\Omega$ with $\nu_{j,m}(\beta,x)\neq 0$ and $h(x):=1$ if $\nu_{j,m}(\beta,x)=0$. 
We note that $h(x)>0$ for all $x\in\Omega$ as the zeros of $\nu_{i,n}(\beta,\cdot)$ are contained in the 
zeros of $\nu_{j,m}(\beta,\cdot)$.
It follows that
\[
(\partial^{\beta})^{E}f(x)\nu_{j,m}(\beta,x)h(x)=(\partial^{\beta})^{E}f(x)\nu_{i,n}(\beta,x)
\]
for $x\in\Omega$ with $\nu_{j,m}(\beta,x)\neq 0$ and 
$(\partial^{\beta})^{E}f(x)\nu_{j,m}(\beta,x)h(x)=0$ for $x\in\Omega$ with $\nu_{j,m}(\beta,x)=0$. 
Therefore $(\partial^{\beta})^{E}f\nu_{j,m}(\beta,\cdot)h$ is bounded on $\Omega$. Further, 
\[
\varepsilon h(x)=\varepsilon\nu_{i,n}(\beta,x)/\nu_{j,m}(\beta,x)\geq 1
\]
for $x\in\Omega\setminus K$ with $\nu_{j,m}(\beta,x)\neq 0$ because $(V_{\infty})$ is fulfilled. 
Further, the zeros of $\nu_{j,m}(\beta,\cdot)$ are contained in 
$N:=\{x\in\Omega\;|\;(\partial^{\beta})^{E}f(x)\nu_{j,m}(\beta,x)=0\}$.
This yields that $K_{\beta}:=\oacx((\partial^{\beta})^{E}f\nu_{j,m}(\beta,\cdot)(\Omega))$ 
is absolutely convex and compact 
by \prettyref{prop:abs_conv_comp_hoelder} and \ref{prop:abs_conv_comp_C_1_b}. Furthermore, 
\[
 N_{j,m}(f)=\{(\partial^{\beta})^{E}f(x)\nu_{j,m}(\beta,x)\;|\;x\in \Omega,\,\beta\in M_{m}\}
           \subset \acx\bigl(\bigcup_{\beta\in M_{m}}K_{\beta}\bigr)
\]
and the set on the right-hand side is absolutely convex and compact 
by \cite[6.7.3 Proposition, p.\ 113]{Jarchow}. 
Finally, our statement follows from \prettyref{cor:full_linearisation} (iv). 
\end{proof}

For the Schwartz space $\mathcal{S}(\R^{d},E)$ and the multiplier space $\mathcal{O}_{M}(\R^{d},E)$ 
from \prettyref{ex:weighted_smooth_functions} c) and d) 
an improvement of \prettyref{ex:weighted_diff} c) to quasi-complete $E$ is already known, 
see e.g.\ \cite[Proposition 9, p.\ 108, Th\'{e}or\`{e}me 1, p.\ 111]{Schwartz1955}. 
However, due to \prettyref{ex:weighted_C_1_diff} it is even allowed that $E$ is only locally complete.

\begin{cor}\label{cor:Schwartz}
If $E$ is a locally complete lcHs, then $\mathcal{S}(\R^{d},E)\cong\mathcal{S}(\R^{d})\varepsilon E$ and 
$\mathcal{O}_{M}(\R^{d},E)\cong \mathcal{O}_{M}(\R^{d})\varepsilon E$.
\end{cor}
\begin{proof}
We start with the Schwartz space. 
Due to \prettyref{ex:weighted_C_1_diff} a) and the barrelledness of the Fr\'{e}chet space $\mathcal{S}(\R^{d})$ 
we only need to check that its
directed family $\mathcal{V}^{\infty}:=(\nu_{1,m})_{m\in\N_{0}}$ of weights given by 
$\nu_{1,m}(\beta,x):=(1+|x|^{2})^{m/2}$, $x\in\R^{d}$, for $m\in\N_{0}$ and $\beta\in M_{m}$ 
is $\mathcal{C}^{1}$-controlled and fulfils $(V_{\infty})$. 
Obviously, condition (i) and (ii) are fulfilled. 
Since 
\[
\bigl|\partial^{e_{n}}\nu_{1,m}(\beta,\cdot)\bigr|(x)
=(m/2)(1+|x|^{2})^{(m/2)-1}2|x_{n}|\leq m (1+|x|^{2})^{m/2}=m\nu_{1,m}(\beta,x)
\]
for all $x\in\R^{d}$ and $1\leq n\leq d$, condition (iii) is also fulfilled. 
Thus $\mathcal{V}^{\infty}$ is $\mathcal{C}^{1}$-controlled. 
Noting that for every $m\in\N$ and $\varepsilon>0$ there is $r>0$ such 
that
\[
\frac{(1+|x|^{2})^{m/2}}{(1+|x|^{2})^{m}}=(1+|x|^{2})^{-m/2}\leq\varepsilon
\]
for all $x\notin\overline{\mathbb{B}_{r}(0)}$, we obtain that
\[
\nu_{1,m}(\beta,x)\leq\varepsilon \nu_{1,2m}(\beta,x)
\] 
for all $x\notin\overline{\mathbb{B}_{r}(0)}$ and $\beta\in M_{m}$. 
Hence $\mathcal{V}^{\infty}$ fulfils condition $(V_{\infty})$.

Now, let us consider the multiplier space. We already know that the generator for $(\mathcal{O}_{M},E)$ 
is strong and consistent by \prettyref{prop:weighted_diff_strong_cons}
because $\mathcal{O}_{M}(\R)$ is a Montel space, thus barrelled, by 
\cite[Chap.\ II, \S4, n$^\circ$4, Th\'{e}or\`{e}me 16, p.\ 131]{Gro} 
and its family of weights is continuous on $\R^{d}$, thus locally bounded away from zero.

Let $f\in\mathcal{O}_{M}(\R ,E)$, $g\in\mathcal{S}(\R^{d})$, $m\in\N_{0}$ and
$\beta\in M_{m}$. Then $(\partial^{\beta})^{E}f\in\mathcal{O}_{M}(\R ,E)$ and hence 
$((\partial^{\beta})^{E}f)g\in\mathcal{S}(\R^{d},E)$, which implies that 
$((\partial^{\beta})^{E}f)g\in\mathcal{C}^{1}_{b}(\R^{d},E)$. 
Moreover, we choose $h\colon\R^{d}\to (0,\infty)$, $h(x):=1+|x|^2$.
Then $((\partial^{\beta})^{E}f)gh$ is bounded on $\R^{d}$ and for $\varepsilon>0$ there is $r>0$ 
such that $(1+|x|^2)^{-1}\leq \varepsilon$ for all $x\notin\overline{\mathbb{B}_{r}(0)}$,
yielding that $K_{\beta,g}:=\oacx(((\partial^{\beta})^{E}f)g(\R^{d}))$ 
is absolutely convex and compact 
by \prettyref{prop:abs_conv_comp_hoelder} and \ref{prop:abs_conv_comp_C_1_b}.
Let $j\subset\mathcal{S}(\R^{d})$ be finite. Since for each $x\in\R^{d}$ we have 
$(\partial^{\beta})^{E}f(x)\max_{g\in j}|g(x)|=\e^{\iu\theta}(\partial^{\beta})^{E}f(x)\widetilde{g}(x)$ 
for some $\widetilde{g}\in j$ and $\theta\in [0,2\pi)$, we get 
\[
 N_{j,m}(f)=\{(\partial^{\beta})^{E}f(x)\max_{g\in j}|g(x)|\;|\;x\in \R^{d},\,\beta\in M_{m}\}
           \subset \acx\bigl(\bigcup_{\beta\in M_{m},g\in j}K_{\beta,g}\bigr).
\]
The set on the right-hand side is absolutely convex and compact by \cite[6.7.3 Proposition, p.\ 113]{Jarchow}. 
Finally, our statement follows from \prettyref{cor:full_linearisation} (iv). 
\end{proof}

For an alternative proof in the case of the Schwartz space we may also use \prettyref{ex:weighted_diff} e) 
since $\mathcal{S}(\R^{d})$ is a Fr\'echet--Schwartz space. 
\prettyref{ex:weighted_C_1_diff} can also be used for an alternative proof of \prettyref{ex:diff_usual} 
if $k=\infty$ by observing that $\mathcal{CW}^{\infty}(\Omega,E)=\mathcal{CV}^{\infty}(\Omega,E)$ for 
any lcHs $E$ where $\mathcal{V}^{\infty}:=\{\nu\in\mathcal{C}^{\infty}_{c}(\Omega)\;|\;\nu\geq 0\}$ and $\mathcal{C}^{\infty}_{c}(\Omega)$ is the space of functions in 
$\mathcal{C}^{\infty}(\Omega)$ with compact support.

Now, we improve \prettyref{ex:weighted_diff} for the special case of spaces of ultradifferentiable functions 
$\mathcal{E}^{(M_{p})}(\Omega,E)$ and $\mathcal{E}^{\{M_{p}\}}(\Omega,E)$ 
from \prettyref{ex:weighted_smooth_functions} e) and f)
where $\omega_{m}:=\N_{0}^{d}\times\Omega$ for all $m\in\N_{0}$. 
For this purpose we recall the following conditions of Komatsu for the sequence $(M_{p})_{p\in\N_{0}}$ (see 
\cite[p.\ 26]{Kom7} and \cite[p.\ 653]{Kom9}):
\begin{enumerate}
 \item [(M.0)$\phantom{'}$] $M_{0}=M_{1}=1$,
 \item [(M.1)$\phantom{'}$] $\forall\; p\in\N:\; M_{p}^{2}\leq M_{p-1}M_{p+1}$,
 \item [(M.2)'] $\exists\; A,C>0\;\forall\;p\in\N_{0}:\;M_{p+1}\leq AC^{p+1}M_{p}$,
 \item [(M.3)'] $\sum_{p=1}^{\infty}\frac{M_{p-1}}{M_{p}}<\infty$.
\end{enumerate}

\begin{exa}\label{ex:ultradifferentiable}
Let $E$ be an lcHs, $\Omega\subset\R^{d}$ open
and $(M_{p})_{p\in\N_{0}}$ a sequence of positive real numbers.
\begin{enumerate}
 \item [a)] $\mathcal{E}^{(M_{p})}(\Omega,E)\cong\mathcal{E}^{(M_{p})}(\Omega)\varepsilon E$ 
 if $E$ is locally complete.
 \item [b)] $\mathcal{E}^{\{M_{p}\}}(\Omega,E)\cong\mathcal{E}^{\{M_{p}\}}(\Omega)\varepsilon E$ 
 if $E$ is complete or semi-Montel
 and in both cases $(M_{p})_{p\in\N_{0}}$ fulfils (M.1) and (M.3)'.
 \item [c)] $\mathcal{E}^{\{M_{p}\}}(\Omega,E)\cong\mathcal{E}^{\{M_{p}\}}(\Omega)\varepsilon E$ 
 if $E$ is sequentially complete and  $(M_{p})_{p\in\N_{0}}$ fulfils (M.0), (M.1), (M.2)' and (M.3)'.
\end{enumerate}
\end{exa}
\begin{proof}
The generator is strong and consistent by \prettyref{prop:weighted_diff_strong_cons}
since the family of weights given in \prettyref{ex:weighted_smooth_functions} e) resp.\ f) 
is locally bounded away from zero on $\Omega$ and 
$\mathcal{E}^{(M_{p})}(\Omega)$ is a Fr\'{e}chet--Schwartz space in a) by \cite[Theorem 2.6, p.\ 44]{Kom7} whereas 
$\mathcal{E}^{\{M_{p}\}}(\Omega)$ is a Montel space in b) and c) by \cite[Theorem 5.12, p.\ 65--66]{Kom7}. 
Hence the statements a) and b) follow from \prettyref{ex:weighted_diff}. 

Let us turn to c). We note that $\mathcal{E}^{\{M_{p}\}}(\Omega,E)\subset\mathcal{E}^{\{M_{p}\}}(\Omega,E)_{\kappa}$ 
by \prettyref{lem:FVE_rel_comp} a) for any lcHs $E$. 
Further, we claim that \prettyref{cond:surjectivity_linearisation} c) is fulfilled. 
Let $f'\in\mathcal{E}^{\{M_{p}\}}(\Omega)'$. 
Due to \cite[Proposition 3.7, p.\ 677]{Kom9} there is a sequence $(f_{n})_{n\in\N}$ 
in the space $\mathcal{D}^{\{M_{p}\}}(\Omega)$ of ultradifferentiable functions 
of class $\{M_{p}\}$ of Roumieu-type with compact support which converges to 
$f'$ in $\mathcal{E}^{\{M_{p}\}}(\Omega)_{b}'$. 
Let $f\in\mathcal{E}^{\{M_{p}\}}(\Omega,E)$. We observe that for every $e'\in E'$
\[
|R_{f}^{t}(f_{n})(e')|=\bigl|\int_{\Omega}f_{n}(x)e'(f(x))\d x\bigr|
\leq \lambda\bigl(\operatorname{supp}(f_{n})\bigr)\sup_{y\in K_{n}(f)}|e'(y)|
\]
where $\lambda$ is the Lebesgue measure, $\operatorname{supp}(f_{n})$ is the support of $f_{n}$ and 
$K_{n}(f):=\{f_{n}(x)f(x)\;|\;x\in\operatorname{supp}(f_{n})\}$. 
The set $K_{n}(f)$ is compact and metrisable by 
\cite[Chap.\ IX, \S2.10, Proposition 17, p.\ 159]{bourbakiII} 
and thus the closure of its absolutely convex hull is 
compact in $E$ as the sequentially complete space $E$ has metric ccp. 
We conclude that $R_{f}^{t}(f_{n})\in (E_{\kappa}')'=\mathcal{J}(E)$ for 
every $n\in\N$. Therefore \prettyref{cond:surjectivity_linearisation} c) is fulfilled, implying 
statement c) for sequentially complete $E$ by \prettyref{thm:full_linearisation}. 
\end{proof}

The results a) and b) in this example are new whereas c) is already proved in
\cite[Theorem 3.10, p.\ 678]{Kom9} in a different way. In particular, part a)
improves \cite[Theorem 3.10, p.\ 678]{Kom9} since Komatsu's conditions (M.0), (M.1), (M.2)' and (M.3)' 
are not needed and the condition that $E$ is sequentially complete is weakened to local completeness. 
We included c) to demonstrate an application of \prettyref{cond:surjectivity_linearisation} c). 
\chapter{Consistency}
\label{chap:consistency}
\section{The spaces \texorpdfstring{$\operatorname{AP}(\Omega,E)$}{AP(Omega,E)} and consistency}
\label{sect:consistency}
This section is dedicated to the properties of functions which are compatible with the $\varepsilon$-product 
in the sense that the space of functions having these properties can be chosen as the space 
$\operatorname{AP}(\Omega,E)$ or $\bigcap_{m\in M}\dom T^{E}_{m}$ in the \prettyref{def:consist_strong} b) 
of consistency. This is done in a quite general way so that we are not tied to certain spaces and have to 
redo our argumentation, for example, if we consider the same generator $(T^{E}_{m},T^{\K}_{m})_{m\in M}$ 
for two different spaces of functions. 

Due to the linearity and continuity of $u\in\FV\varepsilon E$ for a $\dom$-space $\FV$ and $S(u)=u\circ \delta$ 
with $\delta\colon\Omega\to\FV'$, $x\mapsto \delta_{x}$, these are properties which are
purely pointwise or given by pointwise approximation. 
Among such properties of functions are continuity by \prettyref{prop:stetig.cons}, 
Cauchy continuity by \prettyref{prop:c-stetig.cons}, uniform continuity by \prettyref{prop:u-stetig.cons}, 
continuous extendability by \prettyref{prop:cont_ext}, 
continuous differentiability by \prettyref{prop:weighted_diff_strong_cons}, 
vanishing at infinity by \prettyref{prop:van.at.inf0}
and purely pointwise properties of a function like vanishing on a set by \prettyref{prop:zeros}.

We collect these properties in propositions and in follow-up lemmas we handle properties which can be described 
by compositions of defining operators $T^{E}_{m_{1}}\circ T^{E}_{m_{2}}$ 
like continuous differentiability (of higher order) of Fourier transformations
(see \prettyref{ex:Bjoerck}). 
We fix the following notation for this section. For a $\dom$-space $\FV$ and linear 
$T\colon\FV\to\K^{\Omega}$ we set 
$(\delta\circ T)(x)(f):=(\delta_{x}\circ T)(f):=T(f)(x)$ for all $x\in\Omega$ and $f\in\FV$.

\begin{prop}[{continuity}]\label{prop:stetig.cons}
Let $\Omega$ be a topological Hausdorff space and $\FV$ a $\dom$-space 
such that $\FV\subset\mathcal{C}(\Omega)$ as a linear subspace. 
Then $S(u)\in\mathcal{C}(\Omega,E)$ for all $u\in\FV\varepsilon E$ if 
$\delta\in\mathcal{C}(\Omega,\FV_{\kappa}')$.  
\end{prop}
\begin{proof}
Let $u\in\FV\varepsilon E$. Since $S(u)=u\circ\delta$ and $\delta\in\mathcal{C}(\Omega,\FV_{\kappa}')$, we obtain 
that $S(u)$ is in $\mathcal{C}(\Omega,E)$.
\end{proof}

Now, we tackle the problem of the continuity of $\delta\colon\Omega\to\FV_{\kappa}'$ 
in the proposition above and phrase our solution in a way such that it can be applied to show the 
continuity of the partial derivative $(\partial^{\beta})^{E}(S(u))$ as well 
(see \prettyref{prop:diff_cons_barrelled}). 
We recall that a topological space $\Omega$ is called \emph{\gls{completelyregular}} 
if for any non-empty closed subset $A\subset\Omega$ and 
$x\in\Omega\setminus A$ there is $f\in\mathcal{C}(\Omega,[0,1])$ 
such that $f(x)=0$ and $f(z)=1$ for all $z\in A$ (see \cite[Definition 11.1, p.\ 180]{james}). 
Examples of completely regular spaces are uniformisable, particularly metrisable, spaces 
by \cite[Proposition 11.5, p.\ 181]{james} and locally convex Hausdorff spaces by 
\cite[Proposition 3.27, p.\ 95]{fabian}.
A completely regular space $\Omega$ is a \emph{\gls{kRspace}} if for any 
completely regular space $Y$ and any map $f\colon \Omega \to Y$, 
whose restriction to each compact $K\subset\Omega$ is continuous, the map is already continuous on $\Omega$ 
(see \cite[(2.3.7) Proposition, p.\ 22]{buchwalter}). 
Examples of $k_{\R}$-spaces are completely regular $k$-spaces by \cite[3.3.21 Theorem, p.\ 152]{engelking}.
A topological space $\Omega$ is called \emph{\gls{kspace}} (compactly generated space) 
if it satisfies the following condition:
$A\subset \Omega$ is closed if and only if $A\cap K$ is closed in $K$ for every compact $K\subset\Omega$.  
Every locally compact Hausdorff space is a completely regular $k$-space. 
Further, every sequential Hausdorff space is a $k$-space by \cite[3.3.20 Theorem, p.\ 152]{engelking}, 
in particular, every first-countable Hausdorff space. 
Thus metrisable spaces are completely regular Hausdorff $k$-spaces. 
Moreover, the dual space $(X',\tau_{c})$ with the topology of compact convergence 
$\tau_{c}$ is an example of a completely regular Hausdorff $k$-space that 
is neither locally compact nor metrisable by \cite[p.\ 267]{warner1958}
if $X$ is an infinite-dimensional Fr\'{e}chet space. 

We denote by $\mathcal{CW}(\Omega)$ the space of scalar-valued continuous functions on a 
topological Hausdorff space $\Omega$ with the topology $\tau_{c}$ of compact convergence, i.e.\ the topology 
of uniform convergence on compact subsets, which itself is the weighted topology given by the family of weights
$\mathcal{W}:=\mathcal{W}^{0}:=\{\chi_{K}\;|\;K\subset\Omega\;\text{compact}\}$, 
and by $\mathcal{C}_{b}(\Omega)$ the space of scalar-valued bounded, continuous functions on $\Omega$ 
with the topology of uniform convergence on $\Omega$.

\begin{lem}\label{lem:bier}
 Let $\Omega$ be a topological Hausdorff space, $\FV$ a $\dom$-space and 
 $T\colon\FV\to\mathcal{C}(\Omega)$ linear.
 Then $\delta\circ T\in\mathcal{C}(\Omega,\FV_{\gamma}')$ in 
 each of the subsequent cases:
 \begin{enumerate}
 \item [(i)] $\Omega$ is a $k_{\R}$-space and 
 $T\colon\FV\to\mathcal{CW}(\Omega)$ is continuous.
 \item [(ii)] $T\colon\FV\to\mathcal{C}_{b}(\Omega)$ is continuous.
 \end{enumerate}
\end{lem}
\begin{proof}
First, if $x\in\Omega$ and $(x_{\tau})_{\tau\in\mathcal{T}}$ is a net in $\Omega$ converging to $x$, then 
we observe that 
\[
(\delta_{x_{\tau}}\circ T)(f)=T(f)(x_{\tau})\to T(f)(x)=(\delta_{x}\circ T)(f)
\]
for every $f\in\FV$ as $T(f)$ is continuous on $\Omega$. 

(i) Verbatim as in \prettyref{prop:diff_cons_barrelled} a).

(ii) There are $j\in J$, $m\in M$ and $C>0$ such that
   \[
     \sup_{x\in  \Omega}|(\delta_{x}\circ T)(f)|=\sup_{x\in \Omega}|T(f)(x)|
     \leq C|f|_{\FV,j,m}
   \]
for every $f\in\FV$. This means that $\{\delta_{x}\circ T\;|\;x\in \Omega\}$ is equicontinuous 
in $\FV'$, yielding the statement like before.
\end{proof}

The preceding lemma is just a modification of \cite[4.1 Lemma, p.\ 198]{B1} 
where $\FV=\mathcal{CV}(\Omega)$, the Nachbin-weighted space of continuous functions, and 
$T=\id$. 

Next, we turn to Cauchy continuity. A function $f\colon \Omega \to E$ 
from a metric space $\Omega$ to an lcHs $E$ is called \emph{\gls{Cauchycontinuous}} if it maps Cauchy sequences 
to Cauchy sequences. We write $\gls{CCE}$ for 
the space of Cauchy continuous functions from $\Omega$ to $E$ and 
set $\mathcal{CC}(\Omega):=\mathcal{CC}(\Omega,\K)$.
 
\begin{prop}[{Cauchy continuity}]\label{prop:c-stetig.cons}
Let $\Omega$ be a metric space and $\FV$ a $\dom$-space such that $\FV\subset\mathcal{CC}(\Omega)$ 
as a linear subspace. 
Then $S(u)\in\mathcal{CC}(\Omega,E)$ for all $u\in\FV\varepsilon E$ if 
$\delta\in\mathcal{CC}(\Omega,\FV_{\kappa}')$.  
\end{prop}
\begin{proof}
Let $u\in\FV\varepsilon E$ and $(x_{n})$ a Cauchy sequence in $\Omega$. Then $(\delta_{x_{n}})$ is a Cauchy 
sequence in $\FV_{\kappa}'$ since $\delta\in\mathcal{CC}(\Omega,\FV_{\kappa}')$. 
It follows that $(S(u)(x_{n}))$ is a 
Cauchy sequence in $E$ because $u$ is uniformly continuous and $u(\delta_{x_{n}})=S(u)(x_{n})$. 
Hence we conclude that $S(u)\in\mathcal{CC}(\Omega,E)$. 
\end{proof}

For the next lemma we equip the space $\mathcal{CC}(\Omega)$ with the topology 
of uniform convergence on precompact subsets of $\Omega$.

\begin{lem}\label{lem:bier3}
 Let $\FV$ be a $\dom$-space and $T\in L(\FV,\mathcal{CC}(\Omega))$ for a metric space $\Omega$.
 Then $\delta\circ T\in\mathcal{CC}(\Omega,\FV_{\gamma}')$.
\end{lem}
\begin{proof}
 Let $(x_{n})$ be a Cauchy sequence in $\Omega$.
 We have $(\delta_{x_{n}}\circ T)(f)=T(f)(x_{n})$
 for every $f\in\FV$, which implies that 
 $((\delta_{x_{n}}\circ T)(f))$ is a Cauchy sequence in $\K$ 
 because $T(f)\in\mathcal{CC}(\Omega)$ by assumption. 
 Since $\K$ is complete, it has a unique limit $T_{\infty}(f):=\lim_{n\to\infty}(\delta_{x_{n}}\circ T)(f)$ 
 defining a linear functional in $f$. 
 The set $N:=\{x_{n}\;|\;n\in\N\}$ is precompact in $\Omega$ since Cauchy sequences are precompact. 
 Hence there are $j\in J$, $m\in M$ and $C>0$ such that
   \[
     \sup_{n\in\N}|(\delta_{x_{n}}\circ T)(f)|
     =\sup_{x\in N}|T(f)(x)|
     \leq C|f|_{\FV,j,m}
   \]
 for every $f\in\FV$. Therefore the set $\{\delta_{x_{n}}\circ T\;|\;n\in\N\}$ is 
 equicontinuous in $\FV'$, which implies that $T_{\infty}\in\FV'$
 and the convergence of $(\delta_{x_{n}}\circ T)$ to 
 $T_{\infty}$ in $\FV_{\gamma}'$ 
 due to the observation in the beginning and the fact that $\gamma(\FV',\FV)$ 
 and $\sigma(\FV',\FV)$ coincide on equicontinuous sets. 
 In particular, $(\delta_{x_{n}}\circ T)$ is a Cauchy sequence in $\FV_{\gamma}'$. 
 Furthermore, for every $x\in\Omega$ we obtain from the choice $x_{n}=x$ for all $n\in\N$ that 
 $\delta_{x}\circ T\in\FV'$. Thus the map $\delta\circ T\colon\Omega\to\FV_{\gamma}'$ is well-defined and 
 Cauchy continuous.
\end{proof}

The subsequent proposition and lemma handle the analogous statements for uniform continuity. 
For a metric space $\Omega$ we denote by $\gls{Cu}$ the space of uniformly continuous 
functions from $\Omega$ to $E$ and set $\mathcal{C}_{u}(\Omega):=\mathcal{C}_{u}(\Omega,\K)$.

\begin{prop}[{uniform continuity}]\label{prop:u-stetig.cons}
Let $(\Omega,\d)$ be a metric space and $\FV$ a $\dom$-space such that 
$\FV\subset\mathcal{C}_{u}(\Omega)$ 
as a linear subspace. Then $S(\mathsf{u})\in\mathcal{C}_{u}(\Omega,E)$ for all $\mathsf{u}\in\FV\varepsilon E$ if 
$\delta\in\mathcal{C}_{u}(\Omega,\FV_{\kappa}')$.\footnote{Here, we use the symbol $\mathsf{u}$ for elements in 
$\FV\varepsilon E$ instead of the usual $u$ to avoid confusion with the index $u$ of $\mathcal{C}_{u}(\Omega)$ resp.\ 
$\mathcal{C}_{u}(\Omega,E)$.} 
\end{prop}
\begin{proof}
Let $(z_{n})$, $(x_{n})$ be sequences in $\Omega$ with 
$\lim_{n\to\infty}\d(z_{n},x_{n})=0$ and $\mathsf{u}\in\FV\varepsilon E$. 
Then $(\delta_{z_{n}}-\delta_{x_{n}})$ converges to $0$ in $\FV_{\kappa}'$ 
because $\delta\in\mathcal{C}_{u}(\Omega,\FV_{\kappa}')$.
As a consequence $(S(\mathsf{u})(z_{n})-S(\mathsf{u})(x_{n}))$ converges to 
$0$ in $E$ since $\mathsf{u}$ is uniformly continuous and $\mathsf{u}(\delta_{z_{n}}-\delta_{x_{n}})=S(\mathsf{u})(z_{n})-S(\mathsf{u})(x_{n})$. 
Hence we conclude that $S(\mathsf{u})\in\mathcal{C}_{u}(\Omega,E)$.
\end{proof}

For the next lemma we mean by $\mathcal{C}_{bu}(\Omega)$ 
the space of scalar-valued bounded, uniformly continuous functions 
equipped with the topology of uniform convergence on a metric space $\Omega$.

\begin{lem}\label{lem:bier4}
 Let $\FV$ be a $\dom$-space
 and $T\in L(\FV,\mathcal{C}_{bu}(\Omega))$ for a metric space $(\Omega,\d)$.
 Then $\delta\circ T\in\mathcal{C}_{u}(\Omega,\FV_{\gamma}')$.
\end{lem}
\begin{proof}
Let $(z_{n})$ and $(x_{n})$ be sequences in $\Omega$ such that $\lim_{n\to\infty}\d(z_{n},x_{n})=0$.
We have
  \[
  (\delta_{z_{n}}\circ T-\delta_{x_{n}}\circ T)(f)=T(f)(z_{n})-T(f)(x_{n})
  \]
for every $f\in\FV$, which implies that 
$(\delta_{z_{n}}\circ T-\delta_{x_{n}}\circ T)(f)$ converges 
to $0$ in $\K$ for every $f\in\FV$ because $T(f)\in\mathcal{C}_{u}(\Omega)$.
There exist $j\in J$, $m\in M$ and $C>0$ such that
 \[
   \sup_{n\in\N}|(\delta_{z_{n}}\circ T-\delta_{x_{n}}\circ T)(f)|
   \leq 2\sup_{x\in \Omega}|T(f)(x)|
   \leq 2C|f|_{\FV,j,m}
 \]
for every $f\in\FV$. Therefore the set 
$\{\delta_{z_{n}}\circ T-\delta_{x_{n}}\circ T\;|\;n\in\N\}$ is 
equicontinuous in $\FV'$ and
we conclude the statement like before.
\end{proof}

Let us turn to continuous extensions. Let $X$ be a metric space and $\Omega\subset X$. We write 
$\gls{CextE}$ for the space of functions $f\in\mathcal{C}(\Omega,E)$ 
which have a continuous extension to $\overline{\Omega}$ and set 
$\mathcal{C}^{ext}(\Omega):=\mathcal{C}^{ext}(\Omega,\K)$.

\begin{prop}[{continuous extendability}]\label{prop:cont_ext}
Let $X$ be a metric space, $\Omega\subset X$ and $\FV$ a $\dom$-space such that 
$\FV\subset\mathcal{C}^{ext}(\Omega)$ as a linear subspace. 
Then $S(u)\in\mathcal{C}^{ext}(\Omega,E)$ for all $u\in\FV\varepsilon E$ if 
$\delta\in\mathcal{C}^{ext}(\Omega,\FV_{\kappa}')$.
\end{prop} 
\begin{proof}
Let $u\in\FV\varepsilon E$. There is $\delta^{ext}\in\mathcal{C}(\overline{\Omega},\FV_{\kappa}')$ such that 
$\delta^{ext}=\delta$ on $\Omega$ since $\delta\in\mathcal{C}^{ext}(\Omega,\FV_{\kappa}')$. 
Moreover, $u\circ\delta^{ext}\in\mathcal{C}(\overline{\Omega},E)$ and equal to $S(u)=u\circ\delta$ on $\Omega$,
yielding $S(u)\in\mathcal{C}^{ext}(\Omega,E)$.
\end{proof}

For the next lemma we equip $\mathcal{C}^{ext}(\Omega)$ with the topology of uniform convergence on 
compact subsets of $\Omega$.

\begin{lem}\label{lem:cont_ext}
Let $X$ be a metric space, $\Omega\subset X$, $\FV$ a $\dom$-space and
$T\in L(\FV,\mathcal{C}^{ext}(\Omega))$. 
Then $\delta\circ T\in\mathcal{C}^{ext}(\Omega,\FV_{\gamma}')$ if 
$\FV$ is barrelled.
\end{lem}
\begin{proof}
From \prettyref{lem:bier} (i) we derive that $\delta\circ T\in\mathcal{C}(\Omega,\FV_{\gamma}')$.
Let $x\in\partial\Omega$ and $(x_{n})$ be a sequence in $\Omega$ with $x_{n}\to x$.
Then $(\delta_{x_{n}}\circ T)$ is a sequence in $\FV'$ and 
\[
 \lim_{n\to\infty}(\delta_{x_{n}}\circ T)(f)=\lim_{n\to\infty}T(f)(x_{n})=:(\delta_{x}^{ext}\circ T)(f)
\]
in $\K$ for every $f\in\FV$, which implies that $(\delta_{x_{n}}\circ T)$ converges to 
$\delta^{ext}_{x}\circ T$ pointwise on $\FV$ because 
$T(f)\in\mathcal{C}^{ext}(\Omega)$.
As a consequence of the Banach--Steinhaus theorem 
we get $(\delta^{ext}_{x}\circ T)\in\FV'$ and
the convergence in $\FV_{\gamma}'$.
\end{proof}

Let $\FVE$ be a $\dom$-space, $X$ a set, $\mathfrak{K}$ a family of sets and 
$\pi\colon\bigcup_{m\in M}\omega_{m} \to X$ such that $\bigcup_{K\in\mathfrak{K}}K\subset X$. 
We say that a function $f\in\bigcap_{m\in M}\dom T^{E}_{m}$ \emph{\gls{vanish_infty_piK}} if 
\begin{align}\label{van.a.inf}
&\forall\;\varepsilon >0,\, j\in J,\, m\in M,\,\alpha\in\mathfrak{A}\;\exists\;K\in\mathfrak{K}:\;
\sup_{\substack{x\in\omega_{m},\\ \pi(x)\notin K}}p_{\alpha}\bigl(T^{E}_{m}(f)(x)\bigr)\nu_{j,m}(x)<\varepsilon .
\end{align}
Further, we set
\[
\operatorname{AP}_{\pi,\mathfrak{K}}(\Omega,E):=\{f\in\bigcap_{m\in M}\dom T^{E}_{m}\;|\; 
f\;\text{fulfils}\;\eqref{van.a.inf}\}.
\]

\begin{prop}[{vanishing at $\infty$ w.r.t.\ to $(\pi,\mathfrak{K})$}]\label{prop:van.at.inf0}
Let $(T^{E}_{m},T^{\K}_{m})_{m\in M}$ be the generator for $(\mathcal{FV},E)$, let
$\mathcal{FV}(\Omega,Y)\subset\operatorname{AP}_{\pi,\mathfrak{K}}(\Omega,Y)$ as a linear subspace for 
$Y\in\{\K,E\}$ and $\mathfrak{K}$ be closed under taking finite unions.
\begin{enumerate}
\item[(i)] If for all $u\in\FV\varepsilon E$ it holds that $S(u)\in\bigcap_{m\in M}\dom(T^{E}_{m})$ and 
 \begin{equation}\label{eq:van.at.inf_cons}
 \forall\;m\in M,\,x\in\omega_{m}:\;\bigl(T^{E}_{m}S(u)\bigr)(x)=u(T^{\K}_{m,x}),
 \end{equation}
then $S(u)\in\operatorname{AP}_{\pi,\mathfrak{K}}(\Omega,E)$ for all $u\in\FV\varepsilon E$.
\item[(ii)] If for all $e'\in E'$ and $f\in \FVE$ it holds that $e'\circ f\in \bigcap_{m\in M}\dom(T^{\K}_{m})$ and 
 \begin{equation}\label{eq:van.at.inf_strong}
   \forall\; m\in M,\,x\in\omega_{m}:\;T^{\K}_{m}(e'\circ f)(x)=\bigl(e'\circ T^{E}_{m}(f)\bigr)(x),
 \end{equation}
then $e'\circ f\in\operatorname{AP}_{\pi,\mathfrak{K}}(\Omega)$ for all $e'\in E'$ and $f\in\FVE$. 
\end{enumerate}
\end{prop}
\begin{proof}
(i) We set $B_{j,m}:=\{f\in\FV\;|\;|f|_{j,m}\leq 1\}$ for $j\in J$ and $m\in M$. 
Let $u\in\FV\varepsilon E$. 
The topologies $\sigma(\FV',\FV)$ and 
$\kappa(\FV',\FV)$ coincide on the equicontinuous set $B^{\circ}_{j,m}$ and we deduce that 
the restriction of $u$ to $B_{j,m}^{\circ}$ is 
$\sigma(\FV',\FV)$-continuous.

Let $\varepsilon>0$, $j\in J$, $m\in M$, $\alpha\in\mathfrak{A}$ and set 
$U_{\alpha,\varepsilon}:=\{x\in E\;|\;p_{\alpha}(x)< \varepsilon\}$. 
Then there are a finite set $N\subset\FV$ and $\eta>0$ such that
$u(f')\in U_{\alpha,\varepsilon}$ for all $f'\in V_{N,\eta}$ where
\[
V_{N,\eta}:=\{f'\in\FV'\;|\; \sup_{f\in N}|f'(f)|<\eta\}\cap B_{j,m}^{\circ}
\]
because the restriction of $u$ to $B_{j,m}^{\circ}$ is 
$\sigma(\FV',\FV)$-continuous. 
Since $N\subset\FV$ is finite, $\FV\subset\operatorname{AP}_{\pi,\mathfrak{K}}(\Omega)$ 
and $\mathfrak{K}$ is closed under taking finite unions, 
there is $K\in\mathfrak{K}$ such that
\begin{equation}\label{van.at.inf1}
 \sup_{\substack{x\in \omega_{m}\\ \pi(x)\notin K}}|T^{\K}_{m}(f)(x)|
 \nu_{j,m}(x)<\eta
\end{equation}
for every $f\in N$. It follows from \eqref{van.at.inf1} and (the proof of) \prettyref{lem:topology_eps}
that
\[
D_{\pi \nsubset K,j,m}:=\{T^{\K}_{m,x}(\cdot)\nu_{j,m}(x)
\;|\;x\in \omega_{m},\,\pi(x)\notin K\}\subset V_{N,\eta}
\]
and thus $u(D_{\pi \nsubset K,j,m})\subset U_{\alpha,\varepsilon}$. Therefore we have
\[
 \sup_{\substack{x\in\omega_{m}\\ \pi(x)\notin K}}
 p_{\alpha}\bigl(T^{E}_{m}(S(u))(x)\bigr)\nu_{j,m}(x)
\underset{\eqref{eq:van.at.inf_cons}}{=}\sup_{\substack{x\in \omega_{m}\\ \pi(x)\notin K}}
 p_{\alpha}\bigl(u(T^{\K}_{m,x})\bigr)\nu_{j,m}(x)
<\varepsilon .
\]
Hence we conclude that $S(u)\in \operatorname{AP}_{\pi,\mathfrak{K}}(\Omega,E)$. 

(ii) Let $\varepsilon>0$, $f\in\FVE$ and $e'\in E'$. 
Then there exist $\alpha\in\mathfrak{A}$ and $C>0$ such that
$|e'(x)|\leq C p_{\alpha}(x)$ for every $x\in E$. 
For $j\in J$ and $m\in M$ there is 
$K\in\mathfrak{K}$ such that 
\[
\sup_{\substack{x\in \omega_{m}\\ \pi(x)\notin K}}p_{\alpha}\bigl(T^{E}_{m}(f)(x)\bigr)
 \nu_{j,m}(x)<\frac{\varepsilon}{C}
\]
since $\FVE\subset\operatorname{AP}_{\pi,\mathfrak{K}}(\Omega,E)$.
It follows that
\[
 \sup_{\substack{x\in\omega_{m}\\ \pi(x)\notin K}}
 |T^{\K}_{m}(e'\circ f)(x)|\nu_{j,m}(x)
\underset{\eqref{eq:van.at.inf_strong}}{=}\sup_{\substack{x\in\omega_{m} \\ \pi(x)\notin K}}
 \bigl|e'\bigl(T^{E}_{m}(f)(x)\bigr)\bigr|\nu_{j,m}(x)
<C\frac{\varepsilon}{C}=\varepsilon,
\]
yielding $e'\circ f\in\operatorname{AP}_{\pi,\mathfrak{K}}(\Omega)$.
\end{proof}

The first part of the proof above adapts an idea in the proof of \cite[4.4 Theorem, p.\ 199--200]{B1} 
where $(T^{E}_{m},T^{\K}_{m})_{m\in M}
=(\id_{E^{\Omega}},\id_{\K^{\Omega}})$ which is a special case of our proposition.

Our last proposition of this section is immediate. 
For $\omega\subset\Omega$ we set $\operatorname{AP}_{\omega}(\Omega,E):=\{f\in E^{\Omega}\;|\;\forall\;x\in\omega:\;f(x)=0\}$ 
and $\operatorname{AP}_{\omega}(\Omega):=\operatorname{AP}_{\omega}(\Omega,\K)$.

\begin{prop}[{vanishing on a subset}]\label{prop:zeros}
Let $\omega\subset\Omega$ and $\FV$ a $\dom$-space such that $\FV\subset\operatorname{AP}_{\omega}(\Omega)$ as a linear subspace. 
Then $S(u)\in\operatorname{AP}_{\omega}(\Omega,E)$ for all $u\in\FV\varepsilon E$. 
\end{prop}
\section{Further examples of \texorpdfstring{$\varepsilon$}{epsilon}-products}
\label{sect:examples}
In \prettyref{chap:linearisation} we dealt with weighted spaces of continuously partially differentiable functions.
Now, we treat many examples of weighted spaces $\FVE$ of functions with less regularity on a set $\Omega$ 
with values in a locally convex Hausdorff space $E$ over the field $\K$. 
Applying the results of the preceding sections, we give conditions on $E$ such that 
$\FV$ and $\FVE$ are $\varepsilon$-compatible, in particular, that
\[
 \FVE\cong \FV\varepsilon E
\]
holds. 
We start with the simplest example of all. 
Let $\Omega$ be a non-empty set and equip the space $E^{\Omega}$ with the topology 
of pointwise convergence, i.e.\ the locally convex topology given by the seminorms 
\[
|f|_{K,\alpha}:=\sup_{x\in K}p_{\alpha}(f(x))\chi_{K}(x),\quad f\in E^{\Omega},
\]
for finite $K\subset\Omega$ and $\alpha\in \mathfrak{A}$.
To prove $E^{\N_{0}}\cong \K^{\N_{0}}\varepsilon E$ for complete $E$ is given as an exercise in 
\cite[Aufgabe 10.5, p.\ 259]{Kaballo}, which we generalise now.

\begin{exa}\label{ex:space_of_all_functions}
Let $\Omega$ be a non-empty set and $E$ an lcHs. Then $E^{\Omega}\cong\K^{\Omega}\varepsilon E$.
\end{exa}
\begin{proof}
The strength and consistency of the generator $(\id_{E^{\Omega}},\id_{\K^{\Omega}})$ is obvious. 
Let $f\in E^{\Omega}$, $K\subset\Omega$ be finite and set $N_{K}(f):=f(\Omega)\chi_{K}(\Omega)$. Then 
we have $N_{K}(f)=f(K)\cup\{0\}$ if $K\neq\Omega$, and $N_{K}(f)=f(K)$ if $K=\Omega$. Thus $N_{K}(f)$
is finite, hence compact, $N_{K}(f)\subset\acx(f(K))$ and $\acx(f(K))$ 
is a subset of the finite dimensional subspace $\operatorname{span}(f(K))$ of $E$. It follows that 
$\acx(f(K))$ is compact by \cite[6.7.4 Proposition, p.\ 113]{Jarchow}, implying 
our statement by virtue of \prettyref{cor:full_linearisation} (iv).
\end{proof}

The next example will give us the counterpart of \prettyref{ex:weighted_C_1_diff} a)
on the level of sequence spaces. 
Let $\Omega$ be a set, $E$ an lcHs and $\mathcal{V}:=(\nu_{j})_{j\in J}$ a directed family of 
weights $\nu_{j}\colon\Omega\to [0,\infty)$ on $\Omega$. We set 
\[
\gls{ellVE}:=\{f\in E^{\Omega}\;|\;\forall\;j\in J,\,\alpha\in\mathfrak{A}:\;
|f|_{j,\alpha}:=\sup_{x\in\Omega}p_{\alpha}(f(x))\nu_{j}(x)<\infty\}
\]
and $\ell\mathcal{V}(\Omega):=\ell\mathcal{V}(\Omega,\K)$.

\begin{exa}\label{ex:sequence_vanish_infty}
Let $E$ be an lcHs, $(\Omega,\d)$ a \emph{\gls{uniformlyDMS}}, i.e.\ 
there is $r>0$ such that $\d(x,y)\geq r$ for all $x,y\in\Omega$, $x\neq y$, and 
$\mathcal{V}:=(\nu_{j})_{j\in J}$ a directed family of weights on $\Omega$ such that 
\begin{equation}\label{eq:sequence_vanish_infty}
\forall\;j\in J\;\exists\;i\in J\;\forall\;\varepsilon>0\;\exists\;K\subset\Omega\;\text{compact}\;\forall\;
x\in\Omega\setminus K:\;
\nu_{j}(x)\leq \varepsilon \nu_{i}(x).
\end{equation}
If $E$ is locally complete, then $\ell\mathcal{V}(\Omega,E)\cong\ell\mathcal{V}(\Omega)\varepsilon E$.
\end{exa}
\begin{proof}
Let $f\in\ell\mathcal{V}(\Omega,E)$ and $j\in J$. Then $f\nu_{j}$ is bounded on $\Omega$ by definition 
of $\ell\mathcal{V}(\Omega,E)$.
Since $(\Omega,\d)$ is uniformly discrete, there is $r>0$ such that
\[
\frac{p_{\alpha}(f(x)\nu_{j}(x)-f(y)\nu_{j}(y))}{\d(x,y)}\leq \frac{2}{r}|f|_{j,\alpha}<\infty,\quad 
x,y\in\Omega,\,x\neq y,
\]
for every $\alpha\in\mathfrak{A}$. Therefore $f\nu_{j}\in\mathcal{C}^{[1]}_{b}(\Omega,E)$ where 
\[
\mathcal{C}_{b}^{[1]}(\Omega,E):=\Bigl\{g\in E^{\Omega}\;|\;\forall\alpha\in\mathfrak{A}:\;\sup_{x\in \Omega}p_{\alpha}(g(x))<\infty\;\text{and}\;\sup_{\substack{x,y\in \Omega\\x\neq y}}\frac{p_{\alpha}(g(x)-g(y))}{\d (x,y)}<\infty\Bigr\}.
\]
Due to \eqref{eq:sequence_vanish_infty} there is $i\in J$ such that for all $\varepsilon>0$ there exists 
a compact set $K\subset\Omega$ such that $\nu_{j}(x)\leq \varepsilon \nu_{i}(x)$ for all $x\in\Omega\setminus K$.
As $\mathcal{V}$ is directed, we may assume w.l.o.g.\ that 
$\nu_{j}(x)\leq \nu_{i}(x)$ for all $x\in\Omega$. 
This implies that the zeros of $\nu_{i}$ are zeros of $\nu_{j}$. 
We define $h\colon\Omega\to[0,\infty)$ by $h(x):=\nu_{i}(x)/\nu_{j}(x)$ for 
$x\in\Omega$ with $\nu_{j}(x)\neq 0$ and $h(x):=1$ if $\nu_{j}(x)=0$. 
We observe that $h(x)>0$ for all $x\in\Omega$ as the zeros of $\nu_{i}$ are contained in the 
zeros of $\nu_{j}$.
It follows that
\[
f(x)\nu_{j}(x)h(x)=f(x)\nu_{i}(x)
\]
for $x\in\Omega$ with $\nu_{j}(x)\neq 0$ and 
$f(x)\nu_{j}(x)h(x)=0$ for $x\in\Omega$ with $\nu_{j}(x)=0$. 
Hence $f\nu_{j}h$ is bounded on $\Omega$. Further, 
\[
\varepsilon h(x)=\varepsilon\nu_{i}(x)/\nu_{j}(x)\geq 1,
\]
for $x\in\Omega\setminus K$ with $\nu_{j}(x)\neq 0$ because \eqref{eq:sequence_vanish_infty} is fulfilled. 
Moreover, the zeros of $\nu_{j}$ are contained in 
$N:=\{x\in\Omega\;|\;f(x)\nu_{j}(x)=0\}$.
This yields that $\oacx(f\nu_{j}(\Omega))$ is absolutely convex and compact 
by \prettyref{prop:abs_conv_comp_hoelder}. 
So our statement follows from \prettyref{cor:full_linearisation} (iv). 
\end{proof}

Let us apply the preceding result to some known sequence spaces. We recall that 
a matrix $A:=\left(a_{k,j}\right)_{k,j\in\N}$ of non-negative numbers is called \emph{\gls{KoetheM}} if 
it fulfils:
\begin{enumerate}
	\item [(1)] $\forall\;k\in\N\;\exists\;j\in\N:\;a_{k,j}>0$,
	\item [(2)] $\forall\;k,j\in\N:\;a_{k,j}\leq a_{k,j+1}$.
\end{enumerate}
We note that what we call $k$ is usually called $j$ and vice-versa 
(see e.g.\ \cite[Definition, p.\ 326]{meisevogt1997}). 
But the notation we chose is more in line with the meaning of $j$ in our 
\prettyref{def:weight} of a weight function and therefore we prefer to keep our notation consistent. 
For an lcHs $E$ we define the \emph{\gls{KoetheS}}
\[
\gls{lambda_inftyAE}:=\{x=(x_{k})\in E^{\N}\;|\;\forall\;j\in\N,\,\alpha\in \mathfrak{A}:\; 
|x|_{j,\alpha}:=\sup_{k\in\N}p_{\alpha}(x_{k})a_{k,j}<\infty\}
\]
and the \emph{spaces of $E$-valued rapidly decreasing sequences} 
which we need for some theorems on Fourier expansions (see 
\prettyref{thm:fourier.rap.dec}, \prettyref{thm:fourier_periodic}) by 
\[
\gls{sOE}:=\{x=(x_{k})\in E^{\Omega}\;|\;\forall\;j\in\N,\,\alpha\in\mathfrak{A}:\; 
|x|_{j,\alpha}:=\sup_{k\in\Omega}p_{\alpha}(x_{k})(1+|k|^{2})^{j/2}<\infty\}
\]
with $\Omega=\N^{d}$, $\N_{0}^{d}$, $\Z^{d}$. Further, we set $\lambda^{\infty}(A):=\lambda^{\infty}(A,\K)$ and 
$s(\Omega):=s(\Omega,\K)$.

\begin{cor}\label{cor:sequence_vanish_infty}
Let $E$ be a locally complete lcHs.
\begin{enumerate}
\item[a)] If $A:=\left(a_{k,j}\right)_{k,j\in\N}$ is a K\"othe matrix such that
\begin{equation}\label{eq:sequence_vanish_infty_koethe}
\forall\;j\in\N\;\exists\;i\in\N\;\forall\;\varepsilon>0\;\exists\;K\in\N\;\forall\;k\in\N,\,k>K:\;
a_{k,j}\leq \varepsilon a_{k,i},
\end{equation}
then $\lambda^{\infty}(A,E)\cong \lambda^{\infty}(A)\varepsilon E$.
\item[b)] $s(\Omega,E)\cong s(\Omega)\varepsilon E$ for $\Omega=\N^{d}$, $\N_{0}^{d}$, $\Z^{d}$.
\end{enumerate}
\end{cor}
\begin{proof}
We observe that $\N$ and $\Omega$ are uniformly discrete metric spaces if they are equipped 
with the metric induced by the absolute value. Further, a set in a discrete space is compact if and only if 
it is finite. In case b) we set $\nu_{j}\colon\Omega\to (0,\infty)$, $\nu_{j}(k):=(1+|k|^{2})^{j/2}$ 
for $j\in\N$. Then for $\varepsilon>0$ there is $K\in\N$ such that 
\[
\frac{(1+|k|^{2})^{j/2}}{(1+|k|^{2})^{j}}=(1+|k|^{2})^{-j/2}\leq\varepsilon
\]
for all $k\in\Omega$ with $|k|>K$. In both cases the family of weights are directed, in case a) due 
to condition (2) of the definition of a K\"othe matrix. 
Hence we can apply \prettyref{ex:sequence_vanish_infty} in both cases.
\end{proof}

Due to \cite[Proposition 27.10, p.\ 330--331]{meisevogt1997} condition \eqref{eq:sequence_vanish_infty_koethe} 
is equivalent to $\lambda^{\infty}(A)$ being a Schwartz space. 
Since $\lambda^{\infty}(A)$ is also a Fr\'echet space by \cite[Lemma 27.1, p.\ 326]{meisevogt1997}, 
another way to prove \prettyref{cor:sequence_vanish_infty} a) (and b) as well) 
is given by \prettyref{cor:full_linearisation} (ii). 

Our next examples are \emph{Favard-spaces}.
Let $E$ be an lcHs, $0<\gamma\leq 1$, $\Omega$ a compact Hausdorff space, 
$\varphi\colon [0,\infty)\times\Omega\to\Omega$ a continuous \emph{\gls{semiflow}}, i.e.\ 
\[
\varphi(t+r,s)=\varphi(t,\varphi(r,s))\quad\text{and}\quad\varphi(0,s)=s,\quad t,r\in[0,\infty),\,s\in\Omega,
\]
and $(\widetilde{T}^{E}_{t})_{t\geq 0}$ the induced semigroup given by 
$\widetilde{T}^{E}_{t}\colon\mathcal{C}(\Omega,E)\to\mathcal{C}(\Omega,E)$, 
$\widetilde{T}^{E}_{t}(f):=f(\varphi(t,\cdot))$. 
The semigroup $(\widetilde{T}^{\K}_{t})_{t\geq 0}$ is (equi-)bounded and strongly continuous 
by \cite[Chap.\ II, 3.31 Exercises (1), p.\ 95]{engel2000}. 
The vector-valued \emph{\gls{FavardSpace}} of order $\gamma$ of the semigroup 
$(\widetilde{T}^{E}_{t})_{t\geq 0}$ is defined by 
\[
F_{\gamma}(\Omega,E):=\{f\in \mathcal{C}(\Omega,E)\;|\;\forall\;\alpha\in\mathfrak{A}:\;
\sup_{x\in\Omega,t>0}p_{\alpha}(\widetilde{T}^{E}_{t}(f)(x)-f(x))t^{-\gamma}<\infty\}
\]
equipped with the system of seminorms given by 
\[
|f|_{\alpha}:=\max\bigl(\sup_{x\in\Omega}p_{\alpha}(f(x)),
                        \sup_{x\in\Omega,t>0}p_{\alpha}(\widetilde{T}^{E}_{t}(f)(x)-f(x))t^{-\gamma}\bigr),\quad
              f\in F_{\gamma}(\Omega,E),
\]
for $\alpha\in\mathfrak{A}$ (see \cite[Definition 3.1.2, p.\ 160]{butzer1967} and 
\cite[Proposition 3.1.3, p.\ 160]{butzer1967}). Further, we set $F_{\gamma}(\Omega):=F_{\gamma}(\Omega,\K)$. 
$F_{\gamma}(\Omega,E)$ is a $\dom$-space, 
which follows from the setting $\omega:=[0,\infty)\times\Omega$, 
$\dom T^{E}:=\mathcal{C}(\Omega,E)$ and $T^{E}\colon\mathcal{C}(\Omega,E)\to E^{\omega}$ 
given by 
\[
 T^{E}(f)(0,x):=f(x)\;\;\text{and}\;\;
 T^{E}(f)(t,x):=\widetilde{T}^{E}_{t}(f)(x)-f(x),\quad t>0,\,x\in\Omega,
\]
as well as $\operatorname{AP}(\Omega,E):=E^{\Omega}$ and 
the weight given by $\nu(0,x):=1$ and $\nu(t,x):=t^{-\gamma}$ for $t>0$ and $x\in\Omega$.

\begin{exa}\label{ex:Favard}
Let $E$ be a semi-Montel space, $0<\gamma\leq 1$, $\Omega$ a compact Hausdorff space, 
$\varphi\colon [0,\infty)\times\Omega\to\Omega$ a continuous semiflow. Then 
$F_{\gamma}(\Omega,E)\cong F_{\gamma}(\Omega)\varepsilon E$ holds for the
Favard space of order $\gamma$ of the induced semigroup $(\widetilde{T}^{E}_{t})_{t\geq 0}$.
\end{exa}
\begin{proof}
The generator $(T^{E},T^{\K})$ for $(\mathcal{F_{\gamma}},E)$ is consistent 
by \prettyref{prop:stetig.cons} and \prettyref{lem:bier} b)(ii). Its strength is clear. 
Thus our statement follows from \prettyref{cor:full_linearisation} (iii).
\end{proof}

The \emph{space of \gls{cadlagf}} on a set $\Omega\subset\R$ with values in an lcHs $E$ is defined by 
\[
  \gls{DOE}:=\{f\in E^{\Omega}\;|\;\forall\;x\in\Omega:\;\lim_{w\to x\rlim}f(w)=f(x)\;\text{and}\;
  \lim_{w\to x\llim}f(w)\;\text{exists}\}.\footnote{We note that for $x\in\Omega$ we only demand $\lim_{w\to x\rlim}f(w)=f(x)$ if $x$ is an 
  accumulation point of $[x,\infty)\cap\Omega$, and the existence of the limit $\lim_{w\to x\llim}f(w)$ if $x$ is an 
  accumulation point of $(-\infty,x]\cap\Omega$.}
\]
Further, we set $D(\Omega):=D(\Omega,\K)$. Due to \prettyref{prop:cadlag_precomp} the maps given by 
\[
|f|_{K,\alpha}:=\sup_{x\in \Omega}p_{\alpha}(f(x))\chi_{K}(x),\quad f\in D(\Omega,E),
\]
for compact $K\subset\Omega$ and $\alpha\in\mathfrak{A}$ form a system of seminorms 
inducing a locally convex Hausdorff topology on $D(\Omega,E)$. 

\begin{exa}\label{ex:cadlag}
Let $E$ be an lcHs and $\Omega\subset\R$ locally compact. If $E$ is quasi-complete, then 
$D(\Omega)\varepsilon E\cong D(\Omega,E)$.
\end{exa}
\begin{proof}
First, we show that the generator $(\id_{E^{\Omega}},\id_{\K^{\Omega}})$ for $(D,E)$ is strong and 
consistent. The strength is a consequence of a simple calculation, so we only prove the consistency explicitly. 
We have to show that $S(u)\in D(\Omega,E)$ for all $u\in  D(\Omega)\varepsilon E$. 
Let $x\in\Omega$ be an accumulation point of $[x,\infty)\cap\Omega$ resp.\ $(-\infty,x]\cap\Omega$, 
$(x_{n})$ be a sequence in $\Omega$ such that $x_{n}\to x\rlim$ resp.\ $x_{n}\to x\llim$.
We have
\[
\delta_{x_{n}}(f)=f(x_{n})\to f(x)=\delta_{x}(f),\quad x_{n}\to x\rlim,
\]
and
\[
\delta_{x_{n}}(f)=f(x_{n})\to \lim_{n\to\infty}f(x_{n})=:T(f)(x),\quad x_{n}\to x\llim,
\]
for every $f\in D(\Omega)$, which implies that $(\delta_{x_{n}})$ converges to $\delta_{x}$ if $x_{n}\to x\rlim$, 
and to $\delta_{x}\circ T$ if $x_{n}\to x\llim$ in $D(\Omega)_{\sigma}'$.
Since $\Omega$ is locally compact, there are a compact neighbourhood $U(x)\subset\Omega$ of $x$ 
and $n_{0}\in\N$ such that $x_{n}\in U(x)$ for all $n\geq n_{0}$. Hence we deduce
\[
\sup_{n\geq n_{0}}|\delta_{x_{n}}(f)|\leq |f|_{U(x)}
\]
for every $f\in D(\Omega)$. Therefore the set $\{\delta_{x_{n}}\;|\;n\geq n_{0}\}$ is 
equicontinuous in $D(\Omega)'$, which implies that $(\delta_{x_{n}})$ converges to $\delta_{x}$ 
if $x_{n}\to x\rlim$ and to $\delta_{x}\circ T$ if $x_{n}\to x\llim$ in $D(\Omega)_{\gamma}'$ 
and thus in $D(\Omega)_{\kappa}'$.
From
\[
S(u)(x)=u(\delta_{x})=\lim_{n\to\infty}u(\delta_{x_{n}})=\lim_{n\to\infty}S(u)(x_{n}),\quad x_{n}\to x\rlim,
\]
and 
\[
u(\delta_{x}\circ T)=\lim_{n\to\infty} u(\delta_{x_{n}})=\lim_{n\to\infty}S(u)(x_{n}),\quad x_{n}\to x\llim,
\]
for every $u\in D(\Omega)\varepsilon E$ follows the consistency. 
Second, let $f\in D(\Omega,E)$, $K\subset\Omega$ be compact and consider $N_{K}(f)=f(\Omega)\chi_{K}(\Omega)$.
We observe that $N_{K}(f)=f(K)\cup\{0\}$ if $K\neq \Omega$, and 
$N_{K}(f)=f(K)$ if $K=\Omega$. 
We note that $N_{K}(f)\subset \oacx(\overline{f(K)})$ and $\oacx(\overline{f(K)})$ is absolutely convex 
and compact by \prettyref{prop:cadlag_precomp} because $E$ is quasi-complete. 
Thus we derive our statement from \prettyref{cor:full_linearisation} (iv).
\end{proof}

We turn to Cauchy continuous functions. Let $\Omega$ be a metric space, $E$ an lcHs and 
the space $\mathcal{CC}(\Omega,E)$ of Cauchy continuous functions from $\Omega$ to $E$ be 
equipped with the system of seminorms given by
\[
|f|_{K,\alpha}:=\sup_{x\in K}p_{\alpha}(f(x))\chi_{K}(x),\quad f\in\mathcal{CC}(\Omega,E),
\]
for $K\subset\Omega$ precompact and $\alpha\in\mathfrak{A}$. 

\begin{exa}\label{ex:cauchy_cont}
Let $E$ be an lcHs and $\Omega$ a metric space. If $E$ is a Fr\'{e}chet space or a semi-Montel space, then 
$\mathcal{CC}(\Omega,E)\cong\mathcal{CC}(\Omega)\varepsilon E$.
\end{exa}
\begin{proof}
The generator $(\id_{E^{\Omega}},\id_{\K^{\Omega}})$ for $(\mathcal{CC},E)$ 
is consistent by \prettyref{prop:c-stetig.cons} with \prettyref{lem:bier3}. 
Its strength follows from the uniform continuity of every $e'\in E'$.
First, we consider the case that $E$ is a Fr\'{e}chet space. 
Let $f\in\mathcal{CC}(\Omega,E)$, $K\subset\Omega$ be precompact and consider 
$N_{K}(f)=f(\Omega)\chi_{K}(\Omega)$. Then $N_{K}(f)=f(K)\cup\{0\}$ if $K\neq\Omega$, and 
$N_{K}(f)=f(K)$ if $K=\Omega$.
The set $f(K)$ is precompact in the metrisable space $E$ by \cite[Proposition 4.11, p.\ 576]{beer2009}. 
Thus we obtain $\mathcal{CC}(\Omega,E)\subset\mathcal{CC}(\Omega,E)_{\kappa}$ 
by virtue of \prettyref{lem:FVE_rel_comp} c). 
Since $E$ is complete, the first part of the statement
follows from \prettyref{thm:full_linearisation} with \prettyref{cond:surjectivity_linearisation} a). 
If $E$ is a semi-Montel space, then it is a 
consequence of \prettyref{cor:full_linearisation} (iii).
\end{proof}

Let $(\Omega,\d)$ be a metric space, $E$ an lcHs and the space $\gls{Cbu}$ 
of bounded uniformly continuous functions from $\Omega$ to $E$ 
be equipped with the system of seminorms given by
\[
|f|_{\alpha}:=\sup_{x\in\Omega}p_{\alpha}(f(x)),\quad f\in\mathcal{C}_{bu}(\Omega,E),
\]
for $\alpha\in\mathfrak{A}$.

\begin{exa}\label{ex:uniformly_cont}
Let $E$ be an lcHs and $(\Omega,\d)$ a metric space. If $E$ is a semi-Montel space, then 
$\mathcal{C}_{bu}(\Omega,E)\cong\mathcal{C}_{bu}(\Omega)\varepsilon E$.
\end{exa}
\begin{proof}
The generator $(\id_{E^{\Omega}},\id_{\K^{\Omega}})$ for $(\mathcal{C}_{bu},E)$ 
is consistent by \prettyref{prop:u-stetig.cons} with \prettyref{lem:bier4}. It is also strong 
due to the uniform continuity of every $e'\in E'$, 
yielding our statement by \prettyref{cor:full_linearisation} (iii).
\end{proof}

\begin{rem}
If $\N$ is equipped with the metric induced by the absolut value, then 
$\mathcal{C}_{bu}(\N,E)=\ell^{\infty}(\N,E)$ where $\ell^{\infty}(\N,E)$ is the space of 
bounded $E$-valued sequences. If $E$ is a separable infinite-dimensional 
Hilbert space, then the map $S\colon\mathcal{C}_{bu}(\N)\varepsilon E 
\to\mathcal{C}_{bu}(\N,E)$ is not surjective by \cite[2.8 Beispiel, p.\ 140]{B2} 
and \cite[Satz 10.5, p.\ 235--236]{Kaballo}. 
Hence one cannot drop the condition that $E$ is a semi-Montel space in 
\prettyref{ex:uniformly_cont}.
\end{rem}

Let $(\Omega,\d)$ be a metric space, $z\in\Omega$, $E$ an lcHs, $0<\gamma\leq 1$ and define the 
space of $E$-valued $\gamma$-\gls{hoelderCont} functions on $\Omega$ that vanish at $z$ by
\[
\gls{Cgamma_z}:=\{f\in E^{\Omega}\;|\;f(z)=0\;\text{and}\; 
\forall\;\alpha\in \mathfrak{A}:\;|f|_{\alpha}<\infty\}
\]
where
\[
 |f|_{\alpha}:=\sup_{\substack{x,w\in\Omega\\x\neq w}}\frac{p_{\alpha}\bigl(f(x)-f(w)\bigr)}{\d(x,w)^{\gamma}}.
\]
The topological subspace $\gls{Cgamma_z0}$ of 
$\gamma$-H\"older continuous functions that vanish at infinity 
consists of all $f\in\mathcal{C}^{[\gamma]}_{z}(\Omega,E)$ such that for all $\varepsilon>0$ 
there is $\delta>0$ with
\[
 \sup_{\substack{x,w\in\Omega\\0<\d(x,w)<\delta}}\frac{p_{\alpha}\bigl(f(x)-f(w)\bigr)}{\d(x,w)^{\gamma}}
 <\varepsilon.
\]
Further, we set $\mathcal{C}^{[\gamma]}_{z}(\Omega):=\mathcal{C}^{[\gamma]}_{z}(\Omega,\K)$ and 
$\mathcal{C}^{[\gamma]}_{z,0}(\Omega):=\mathcal{C}^{[\gamma]}_{z,0}(\Omega,\K)$. 
Moreover, we define $M:=J:=\{1\}$, $\omega_{1}:=\Omega^{2}\setminus\{(x,x)\;|\;x\in\Omega\}$ 
and $T^{E}_{1}\colon E^{\Omega}\to E^{\omega_{1}}$, $T^{E}_{1}(f)(x,w):=f(x)-f(w)$, and 
\[
\nu_{1,1}\colon\omega_{1}\to [0,\infty),\;\nu_{1,1}(x,w)
:=\frac{1}{\d(x,w)^{\gamma}}.
\]
Then we have for every $\alpha\in\mathfrak{A}$ that
\[
 |f|_{\alpha}=\sup_{(x,w)\in\omega_{1}}p_{\alpha}\bigl(T^{E}_{1}(f)(x,w)\bigr)\nu_{1,1}(x,w),
 \quad f\in\mathcal{C}^{[\gamma]}_{z}(\Omega,E).
\]

\begin{exa}\label{ex:hoelder}
Let $E$ be an lcHs, $(\Omega,\d)$ a metric space, $z\in\Omega$ and $0<\gamma\leq 1$. Then 
\begin{enumerate}
\item [a)] $\mathcal{C}^{[\gamma]}_{z}(\Omega,E)\cong \mathcal{C}^{[\gamma]}_{z}(\Omega)\varepsilon E$ 
if $E$ is a semi-Montel space,
\item [b)] $\mathcal{C}^{[\gamma]}_{z,0}(\Omega,E)\cong \mathcal{C}^{[\gamma]}_{z,0}(\Omega)\varepsilon E$ 
if $\Omega$ is precompact and $E$ quasi-complete.
\end{enumerate}
\end{exa}
\begin{proof}
Let us start with a). From \prettyref{prop:zeros} for vanishing at $z$ and 
a simple calculation follows that $(T^{E}_{1},T^{\K}_{1})$ is a strong 
and consistent generator for $(\mathcal{C}^{[\gamma]}_{z},E)$. 
This proves part a) by \prettyref{cor:full_linearisation} (iii). 
Concerning part b), we set 
$\mathfrak{K}:=\bigl\{\{(x,w)\in\Omega^{2}\;|\;\d(x,w)\geq \delta\}\;|\;\delta>0\bigr\}$, 
and let $\pi\colon\omega_{1}\to\omega_{1}$ be the identity. 
Then 
$\mathcal{C}^{[\gamma]}_{z,0}(\Omega,E)
=\mathcal{C}^{[\gamma]}_{z}(\Omega,E)\cap\operatorname{AP}_{\pi,\mathfrak{K}}(\Omega,E)$ 
with $\operatorname{AP}_{\pi,\mathfrak{K}}(\Omega,E)$ 
from \prettyref{prop:van.at.inf0}
and the generator $(T^{E}_{1},T^{\K}_{1})$ for $(\mathcal{C}^{[\gamma]}_{z,0},E)$ is strong and consistent by 
\prettyref{prop:zeros} for vanishing at $z$ and 
\prettyref{prop:van.at.inf0} for vanishing at infinity w.r.t.\ $(\pi,\mathfrak{K})$. 

Let $f\in\mathcal{C}^{[\gamma]}_{z,0}(\Omega,E)$ and 
$K_{\delta}:=\{(x,w)\in\Omega^{2}\;|\;\d(x,w)\geq \delta\}$ for $\delta>0$. For
\[
  N_{\pi\subset K_{\delta},1,1}(f)
 =\{T^{E}_{1}(f)(x,w)\nu_{1,1}(x,w)\;|\;(x,w)\in K_{\delta}\}
 =\bigl\{\tfrac{f(x)-f(w)}{\d(x,w)^{\gamma}}\;|\;(x,w)\in K_{\delta}\bigr\}
\]
we have 
\begin{align*}
  N_{\pi\subset K_{\delta},1,1}(f)
&\subset \delta^{-\gamma}\{c(f(x)-f(w))\;|\;x,w\in\Omega,\,|c|\leq 1\}\\
&=\delta^{-\gamma}\operatorname{ch}\bigl(f(\Omega)-f(\Omega)\bigr).
\end{align*}
The set $f(\Omega)$ is precompact because 
$\Omega$ is precompact and the $\gamma$-H\"{o}lder continuous function 
$f$ is uniformly continuous. It 
follows that the linear combination $f(\Omega)-f(\Omega)$ is precompact and 
the circled hull of a precompact set is still precompact 
by \cite[Chap.\ I, 5.1, p.\ 25]{schaefer}. 
Therefore $N_{\pi\subset K_{\delta},1,1}(f)$ is precompact for every $\delta>0$, 
giving the precompactness of 
\[
N_{1,1}(f)=\{T^{E}_{1}(f)(x,w)\nu_{1,1}(x,w)\;|\;(x,w)\in\omega_{1}\}
\]
by \prettyref{prop:vanish_at_infinity_precomp}. 
Hence statement b) is a consequence of \prettyref{cor:full_linearisation} (iv),
\prettyref{prop:vanish_at_infinity_precomp} and the quasi-completeness of $E$. 
\end{proof}

Let $\Omega$ be a topological Hausdorff space and $\mathcal{V}:=(\nu_{j})_{j\in J}$ a directed family of weights 
$\nu_{j}\colon \Omega\to [0,\infty)$. The weighted space of continuous functions on $\Omega$ 
with values in an lcHs $E$ is given by
\[
\gls{CVE}:=\{f\in\mathcal{C}(\Omega,E)\;|\;\forall\;j\in J,\,\alpha\in\mathfrak{A}:\;
|f|_{j,\alpha}<\infty\}
\]
where
\[
|f|_{j,\alpha}:=\sup_{x\in\Omega}p_{\alpha}(f(x))\nu_{j}(x).
\]
Its topological subspace of functions that vanish at infinity in the weighted topology is defined by 
 \begin{align*}
  \gls{CVE_0}:=\{f\in\mathcal{CV}(\Omega,E)\;|\;&\forall\;j\in J,\,\alpha\in\mathfrak{A},
  \,\varepsilon>0\\
  &\exists\;K\subset \Omega\;\text{compact}:\;|f|_{\Omega\setminus K,j,\alpha}<\varepsilon\} 
 \end{align*}
where 
 \[
  |f|_{\Omega\setminus K,j,\alpha}:=\sup_{x\in\Omega\setminus K}
  p_{\alpha}(f(x))\nu_{j}(x).
 \]
Further, we define $\mathcal{CV}(\Omega):=\mathcal{CV}(\Omega,\K)$ and $\mathcal{CV}_{0}(\Omega):=\mathcal{CV}(\Omega,\K)$.
In particular, we set $\gls{Cb}:=\mathcal{CV}(\Omega,E)$, 
i.e.\ the space of bounded continuous functions, and have
$\mathcal{CV}_{0}(\Omega,E)=\mathcal{C}_{0}(\Omega,E)$ if $\mathcal{V}:=\{1\}$.
In \cite{B3,B1,B2} Bierstedt studies these spaces in the case that 
$\mathcal{V}$ is a \emph{\gls{Nachbin-family}} which means that the functions $\nu_{j}$ are upper semi-continuous 
for all $j\in J$ and directed in the sense that for $j_{1},j_{2}\in J$ and $\lambda\geq 0$ there 
is $j_{3}\in J$ such that $\lambda\nu_{j_{1}},\lambda\nu_{j_{2}}\leq\nu_{j_{3}}$. Formally this is 
stronger than our definition of being directed in \prettyref{rem:weights_Hausdorff_directed} c).
The notion $\mathcal{U}\leq\mathcal{V}$ for two Nachbin-families means that for every $\mu\in\mathcal{U}$ there 
is $\nu\in\mathcal{V}$ such that $\mu\leq \nu$.
One of his main results from \cite{B2} is the following theorem.

\begin{thm}[{\cite[2.4 Theorem (2), p.\ 138--139]{B2}}]\label{thm:bierstedt}
Let $E$ be a quasi-complete lcHs, $\Omega$ a completely regular Hausdorff space 
and $\mathcal{V}$ a Nachbin-family on $\Omega$. If
\begin{enumerate}
\item [(i)] $\mathcal{Z}:=\left\{v\colon \Omega\to\R\;|\; v\;\text{constant},\;v\geq 0\right\}
\leq \mathcal{V}$, or
\item [(ii)] $\widetilde{\mathcal{W}}:=\left\{\mu \chi_{K}\;|\; \mu>0,\;K\subset\Omega\;\text{compact}\right\}
\leq \mathcal{V}$ 
 and $\Omega$ is a $k_{\R}$-space, 
\end{enumerate}
then $\mathcal{CV}_{0}(\Omega,E)\cong \mathcal{CV}_{0}(\Omega)\varepsilon E$.
\end{thm} 

We note that $\mathcal{C}\widetilde{\mathcal{W}}(\Omega,E)=\mathcal{CW}(\Omega,E)$ with our definition 
of $\mathcal{W}=\{\chi_{K}\;|\; K\subset\Omega\;\text{compact}\}$ from above \prettyref{lem:bier}. 
The only difference is that $\mathcal{W}$ is not a Nachbin-family because it is not directed in the sense of 
Nachbin-families but in the sense of \prettyref{rem:weights_Hausdorff_directed} c).
We improve this result by strengthening the conditions on $\Omega$ and $\mathcal{V}$ which allows us 
to weaken the assumptions on $E$.
 
\begin{exa}\label{ex:cont_loc_comp}
Let $E$ be an lcHs, $\Omega$ a locally compact topological Hausdorff space
and $\mathcal{V}$ a directed family of continuous weights on $\Omega$. 
\begin{enumerate}
\item[(i)] If $E$ has ccp, or 
\item[(ii)] if $E$ has metric ccp and $\Omega$ is second-countable,
\end{enumerate}
then $\mathcal{CV}_{0}(\Omega,E)\cong \mathcal{CV}_{0}(\Omega)\varepsilon E$.
\end{exa}
\begin{proof}
We set $\mathfrak{K}:=\{K\subset\Omega\;|\;K\;\text{compact}\}$ and $\pi\colon\Omega\to\Omega$, $\pi(x):=x$. 
It follows from \prettyref{prop:stetig.cons} combined with \prettyref{lem:bier} (i) (continuity) 
and \prettyref{prop:van.at.inf0} (vanish at infinity w.r.t.\ $(\pi,\mathfrak{K})$) that the generator 
$(\id_{E^{\Omega}},\id_{\K^{\Omega}})$ is strong and consistent since $\mathcal{V}$ is a family of 
continuous weights and $\Omega$ a $k_{\R}$-space due to local compactness.

Let $f\in\mathcal{CV}_{0}(\Omega,E)$, $j\in J$ and consider $N_{j}(f)=(f\nu_{j})(\Omega)$. 
By \prettyref{prop:abs_conv_comp_C_0} the set $K:=\oacx(N_{j}(f))$ is absolutely convex and compact 
as $f\nu_{j}\in\mathcal{C}_{0}(\Omega,E)$, implying our statement by \prettyref{cor:full_linearisation} (iv).
\end{proof}

\begin{exa}\label{ex:cont_usual}
Let $E$ be an lcHs and $\Omega$ a [metrisable] $k_{\R}$-space. 
If $E$ has [metric] ccp, then $\mathcal{CW}(\Omega,E)\cong \mathcal{CW}(\Omega)\varepsilon E$.
\end{exa}
\begin{proof}
First, we observe that the generator $(\id_{E^{\Omega}},\id_{\K^{\Omega}})$ for $(\mathcal{CW},E)$ 
is consistent by \prettyref{prop:stetig.cons} and \prettyref{lem:bier} b)(i). Its strength is obvious. 
Let $f\in\mathcal{CW}(\Omega,E)$, $K\subset\Omega$ be compact and consider $N_{K}(f)=f(\Omega)\nu_{K}(\Omega)$. 
Then $N_{K}(f)=f(K)\cup\{0\}$ if $K\neq\Omega$, and $N_{K}(f)=f(K)$ if $K=\Omega$, 
which yields that $N_{K}(f)$ is compact in $E$. 
If $\Omega$ is even metrisable, then $f(K)$ is also metrisable by 
\cite[Chap.\ IX, \S2.10, Proposition 17, p.\ 159]{bourbakiII} and thus the finite union $N_{K}(f)$ as well
by \cite[Theorem 1, p.\ 361]{stone} since the compact set $N_{K}(f)$ is collectionwise normal and locally 
countably compact by \cite[5.1.18 Theorem, p.\ 305]{engelking}.  
Further, $\oacx(N_{K}(f))$ is absolutely convex and compact in $E$ if $E$ has ccp 
resp.\ if $\Omega$ is metrisable and $E$ has metric ccp. 
We conclude that
$
\mathcal{CW}(\Omega,E)\cong \mathcal{CW}(\Omega)\varepsilon E
$
if $E$ has ccp resp.\ if $\Omega$ is metrisable and $E$ has metric ccp by 
\prettyref{cor:full_linearisation} (iv).
\end{proof}

Bierstedt also considers closed subspaces of $\mathcal{CV}(\Omega)$ and $\mathcal{CV}_{0}(\Omega)$, for 
instance subspaces of holomorpic functions on open $\Omega$, and of holomorpic functions on the inner points of $\Omega$ 
which are continuous on the boundary in \cite[3.1 Bemerkung, p.\ 141]{B2} and \cite[3.7 Satz, p.\ 144]{B2}. 

Let $\Omega\subset\C$ be open and bounded and $E$ an lcHs over $\C$. 
We denote by $\gls{AOE}$ the space of continuous functions 
from $\overline{\Omega}$ to an lcHs $E$ which are holomorphic on $\Omega$
and equip $\mathcal{A}(\overline{\Omega},E)$ with the system of seminorms given by
\[
|f|_{\alpha}:=\sup_{x\in\overline{\Omega}}p_{\alpha}(f(x)),\quad f\in\mathcal{A}(\overline{\Omega},E),
\]
for $\alpha\in\mathfrak{A}$. We set $\mathcal{A}(\overline{\Omega}):=\mathcal{A}(\overline{\Omega},\C)$, $J:=M:=\{1\}$ 
and $\nu_{1,1}:=1$ on $\overline{\Omega}$.

\begin{exa}\label{ex:disc_algebra}
Let $E$ be an lcHs and $\Omega\subset\C$ open and bounded. 
Then $\mathcal{A}(\overline{\Omega},E)\cong\mathcal{A}(\overline{\Omega})\varepsilon E$ if $E$ has metric ccp.
\end{exa}
\begin{proof}
The space $\mathcal{A}(\overline{\Omega})$ is a Banach space and hence barrelled.
The inclusion $I\colon\mathcal{A}(\overline{\Omega})
\to\mathcal{CW}^{\infty}_{\overline{\partial}}(\Omega)$ is continuous
due to the Cauchy inequality ($I$ is an inclusion due to the identity theorem). 
It follows from \prettyref{prop:stetig.cons}, \prettyref{lem:bier} b)(i),
\prettyref{prop:diff_cons_barrelled} c) and \eqref{eq:holomorphic_coincide_0} 
that the generator 
$(\id_{E^{\overline{\Omega}}},\id_{\C^{\overline{\Omega}}})$ is consistent 
and as in \prettyref{prop:weighted_diff_strong_cons} that it is strong, too. 

Let $f\in\mathcal{A}(\overline{\Omega},E)$ and $N_{1,1}(f)=f(\overline{\Omega})$. 
The set $K:=\oacx(N_{1,1}(f))$ is absolutely convex and compact by \prettyref{prop:abs_conv_comp_C_0} 
since $f\in\mathcal{C}(\overline{\Omega},E)=\mathcal{C}_{0}(\overline{\Omega},E)$, 
implying our statement by \prettyref{cor:full_linearisation} (iv).
\end{proof}

For quasi-complete $E$ this is already covered by \cite[3.1 Bemerkung, p.\ 141]{B2}.
More general than holomorphic functions, we may also consider kernels of hypoelliptic linear partial 
differential operators in $\mathcal{CV}(\Omega)$ and $\mathcal{CV}_{0}(\Omega)$.
For an open set $\Omega\subset\R^{d}$, a directed family $\mathcal{V}:=(\nu_{j})_{j\in\N}$ of weights 
$\nu_{j}\colon\Omega\to[0,\infty)$, an lcHs $E$ and a linear partial differential operator 
$P(\partial)^{E}$ which is hypoelliptic if $E=\K$ we define the space of zero solutions 
\[
\gls{CVPE}
:=\{f\in\mathcal{C}^{\infty}_{P(\partial)}(\Omega,E)\;|\;\forall\;j\in\N,\,\alpha\in\mathfrak{A}:\;
|f|_{j,\alpha}<\infty\},
\]
where $\gls{C_infty_P}$ is the kernel of $P(\partial)^{E}$ 
in $\mathcal{C}^{\infty}(\Omega,E)$, 
\[
|f|_{j,\alpha}:=\sup_{x\in\Omega}p_{\alpha}(f(x))\nu_{j}(x),
\]
and its topological subspace 
\[
\gls{CVPE_0}:=\mathcal{CV}_{P(\partial)}(\Omega,E)\cap\mathcal{CV}_{0}(\Omega,E).
\]
Further, we set $\mathcal{CV}_{P(\partial)}(\Omega):=\mathcal{CV}_{P(\partial)}(\Omega,\K)$ and 
$\mathcal{CV}_{0,P(\partial)}(\Omega):=\mathcal{CV}_{0,P(\partial)}(\Omega,\K)$. 
We say that $\mathcal{V}$ is \emph{\gls{loc_b_away_0}} on $\Omega$ if
\[
 \forall\;K\subset\Omega\;\text{compact}\;\exists\; j\in\N :\;\inf_{x\in K}\nu_{j}(x)>0.
\]
This is an extension of the definition of being \emph{locally bounded away from zero} from $\mathcal{V}^{k}$ 
with $k\in\N_{\infty}$ to the case $k=0$ (see \prettyref{prop:weighted_diff_strong_cons}). If $\mathcal{V}$ 
is a Nachbin-family, this means that $\widetilde{\mathcal{W}}\leq\mathcal{V}$ (see \prettyref{thm:bierstedt} (ii)).

\begin{prop}\label{prop:frechet_bierstedt}
Let $\Omega\subset\R^{d}$ be open, $\mathcal{V}:=(\nu_{j})_{j\in\N}$ an increasing family of weights 
which is locally bounded away from zero on $\Omega$ and $P(\partial)^{\K}$ a 
hypoelliptic linear partial differential operator. 
Then $\mathcal{CV}_{P(\partial)}(\Omega)$ and $\mathcal{CV}_{0,P(\partial)}(\Omega)$
are Fr\'echet spaces.
\end{prop}
\begin{proof}
We note that $\mathcal{CV}_{P(\partial)}(\Omega)$ is metrisable 
as $\mathcal{V}$ is countable. Let $(f_{n})$ be a Cauchy sequence in $\mathcal{CV}_{P(\partial)}(\Omega)$. 
From $\mathcal{V}$ being locally bounded away from zero it follows that 
for every compact $K\subset\Omega$ there is $j\in\N$ such that
\begin{equation}\label{eq:hypo_weighted_Frechet}
     \sup_{x\in K}|f(x)|
\leq \sup_{z\in K}\nu_{j}(z)^{-1}\sup_{x\in K}|f(x)|\nu_{j}(x)
\leq \sup_{z\in K}\nu_{j}(z)^{-1}|f|_{j}, 
\quad f\in\mathcal{CV}_{P(\partial)}(\Omega),
\end{equation}
which means that the inclusion $I\colon \mathcal{CV}_{P(\partial)}(\Omega)\to\mathcal{CW}_{P(\partial)}(\Omega)$ 
is continuous. 
Thus $(f_{n})$ is also a Cauchy sequence in $\mathcal{CW}_{P(\partial)}(\Omega)$ and has a limit $f$ there 
as $\mathcal{CW}_{P(\partial)}(\Omega)$ is complete due to the hypoellipticity of $P(\partial)^{\K}$. 
Let $j\in\N$, $\varepsilon>0$ and $x\in\Omega$. Then there is $m_{j,\varepsilon,x}\in\N$ such that 
for all $m\geq m_{j,\varepsilon,x}$ it holds that
\[
|f_{m}(x)-f(x)|<\frac{\varepsilon}{2\nu_{j}(x)}
\]
if $\nu_{j}(x)\neq 0$.
Further, there is $m_{j,\varepsilon}\in\N$ such that for all $n,m\geq m_{j,\varepsilon}$ it holds that
\[
|f_{n}-f_{m}|_{j}<\frac{\varepsilon}{2}.
\]
Hence for $n\geq m_{j,\varepsilon}$ we choose $m\geq\max(m_{j,\varepsilon},m_{j,\varepsilon,x})$ and derive 
\[
|f_{n}(x)-f(x)|\nu_{j}(x)\leq |f_{n}(x)-f_{m}(x)|\nu_{j}(x)+|f_{m}(x)-f(x)|\nu_{j}(x)
< \frac{\varepsilon}{2}+\frac{\varepsilon}{2\nu_{j}(x)}\nu_{j}(x)=\varepsilon.
\]
It follows that $|f_{n}-f|_{j}\leq \varepsilon$ and $|f|_{j}\leq \varepsilon+|f_{n}|_{j}$ 
for all $n\geq m_{j,\varepsilon}$, implying the convergence of $(f_{n})$ to $f$ 
in $\mathcal{CV}_{P(\partial)}(\Omega)$.
Therefore $\mathcal{CV}_{P(\partial)}(\Omega)$ is a Fr\'echet space.
$\mathcal{CV}_{0,P(\partial)}(\Omega)$ is a closed subspace of $\mathcal{CV}_{P(\partial)}(\Omega)$ and so 
a Fr\'echet space as well.
\end{proof}

Due to the proposition above the spaces 
$\mathcal{CV}_{P(\partial)}(\Omega)$ and $\mathcal{CV}_{0,P(\partial)}(\Omega)$ are closed subspaces of 
$\mathcal{CV}(\Omega)$ resp.\ $\mathcal{CV}_{0}(\Omega)$. Hence we have the following consequence of 
\prettyref{thm:bierstedt} (ii), \cite[2.12 Satz (1), p.\ 141]{B2} and \cite[3.1 Bemerkung, p.\ 141]{B2}. 

\begin{cor}\label{cor:subspace_bierstedt}
Let $E$ be an lcHs, $\Omega\subset\R^{d}$ open, $\mathcal{V}$ a Nachbin-family on $\Omega$ 
which is locally bounded away from zero and $P(\partial)^{\K}$ a hypoelliptic linear partial differential operator. 
\begin{enumerate}
\item [a)] $\mathcal{CV}_{P(\partial)}(\Omega,E)\cong\mathcal{CV}_{P(\partial)}(\Omega)\varepsilon E$ 
if $E$ is a semi-Montel space.
\item [b)] $\mathcal{CV}_{0,P(\partial)}(\Omega,E)\cong\mathcal{CV}_{0,P(\partial)}(\Omega)\varepsilon E$ if 
$E$ is quasi-complete. 
\end{enumerate} 
\end{cor} 

Like before we may improve this result by strengthening the conditions on $\mathcal{V}$ 
and $\mathcal{CV}_{P(\partial)}(\Omega)$ resp.\ $\mathcal{CV}_{0,P(\partial)}(\Omega)$ which allows us 
to weaken the assumptions on $E$.

\begin{exa}\label{ex:subspace_bierstedt}
Let $E$ be an lcHs, $\Omega\subset\R^{d}$ open, $\mathcal{V}:=(\nu_{j})_{j\in\N}$ an increasing family of weights 
which is locally bounded away from zero on $\Omega$ and $P(\partial)^{\K}$ a 
hypoelliptic linear partial differential operator. 
\begin{enumerate}
\item [a)] $\mathcal{CV}_{P(\partial)}(\Omega,E)\cong\mathcal{CV}_{P(\partial)}(\Omega)\varepsilon E$ 
if $E$ is complete and $\mathcal{CV}_{P(\partial)}(\Omega)$ a semi-Montel space.
\item [b)] $\mathcal{CV}_{P(\partial)}(\Omega,E)\cong\mathcal{CV}_{P(\partial)}(\Omega)\varepsilon E$ 
if $E$ is locally complete and $\mathcal{CV}_{P(\partial)}(\Omega)$ a Schwartz space.
\item [c)] $\mathcal{CV}_{0,P(\partial)}(\Omega,E)\cong\mathcal{CV}_{0,P(\partial)}(\Omega)\varepsilon E$ 
if $E$ has metric ccp and $\nu_{j}\in\mathcal{C}(\Omega)$ for all $j\in\N$.
\item [d)] $\mathcal{CV}_{0,P(\partial)}(\Omega,E)\cong\mathcal{CV}_{0,P(\partial)}(\Omega)\varepsilon E$ if 
$E$ is locally complete and $\mathcal{CV}_{0,P(\partial)}(\Omega)$ a Schwartz space.
\end{enumerate} 
\end{exa} 
\begin{proof}
Let $\mathcal{F}$ stand for $\mathcal{CV}_{P(\partial)}$ or $\mathcal{CV}_{0,P(\partial)}$. 
The space $\mathcal{F}(\Omega)$ is a Fr\'echet space and hence barrelled by \prettyref{prop:frechet_bierstedt}.
The inclusion $I\colon\mathcal{F}(\Omega)\to\mathcal{CW}_{P(\partial)}(\Omega)$ is continuous since 
$\mathcal{V}$ is locally bounded away from zero on $\Omega$. 
The hypoellipticity of $P(\partial)^{\K}$ (see e.g.\ \cite[p.\ 690]{F/J/W}) yields that 
$\mathcal{CW}_{P(\partial)}(\Omega)=\mathcal{CW}^{\infty}_{P(\partial)}(\Omega)$ as locally convex spaces. 
Thus the inclusion $I\colon\mathcal{F}(\Omega)\to\mathcal{CW}_{P(\partial)}^{\infty}(\Omega)$ is continuous. 
It follows from \prettyref{prop:diff_cons_barrelled} c) that the generator 
$(\id_{E^{\Omega}},\id_{\K^{\Omega}})$ is consistent if $\mathcal{F}=\mathcal{CV}_{P(\partial)}$, and 
combined with \prettyref{prop:van.at.inf0} (vanish at infinity w.r.t.\ $(\pi,\mathfrak{K})$) if 
$\mathcal{F}=\mathcal{CV}_{0,P(\partial)}$ where $\mathfrak{K}$ and $\pi$ are chosen as in \prettyref{ex:cont_loc_comp}.
The strength of the generator follows as in \prettyref{prop:weighted_diff_strong_cons} and, 
if $\mathcal{F}=\mathcal{CV}_{0,P(\partial)}$, in combination with \prettyref{prop:van.at.inf0} b).
This proves part a), b) and d) due to \prettyref{cor:full_linearisation} (i) and (ii). 

Let us turn to part c). Let $f\in\mathcal{CV}_{0,P(\partial)}(\Omega,E)$, $j\in\N$ and $N_{j}(f):=(f\nu_{j})(\Omega)$. 
The set $K:=\oacx(N_{j}(f))$ is absolutely convex compact by \prettyref{prop:abs_conv_comp_C_0} 
as $f\nu_{j}\in\mathcal{C}_{0}(\Omega,E)$, implying our statement by \prettyref{cor:full_linearisation} (iv).
\end{proof}

At least for some weights and operators $P(\partial)$ we can show that $\mathcal{CV}_{P(\partial)}(\Omega,E)$ 
coincides with a corresponding space $\mathcal{CV}_{P(\partial)}^{\infty}(\Omega,E)$ 
from \prettyref{ex:weighted_smooth_functions} if $E$ is locally complete.

\begin{prop}\label{prop:co_top_isomorphism}
Let $E$ be a locally complete lcHs, $\Omega\subset\R^{d}$ and $P(\partial)^{\K}$ a hypoelliptic 
linear partial differential operator.
Then we have $\mathcal{CW}_{P(\partial)}(\Omega)\varepsilon E
\cong \mathcal{CW}_{P(\partial)}(\Omega,E)$ 
and $\mathcal{CW}_{P(\partial)}(\Omega,E)=\mathcal{CW}^{\infty}_{P(\partial)}(\Omega,E)$ 
as locally convex spaces. 
\end{prop}
\begin{proof}
We already know that 
\[
S_{\mathcal{CW}^{\infty}_{P(\partial)}(\Omega)}\colon
\mathcal{CW}^{\infty}_{P(\partial)}(\Omega)\varepsilon E \to
\mathcal{CW}^{\infty}_{P(\partial)}(\Omega,E)
\]
is an isomorphism by \prettyref{ex:diff_usual} b). 
The hypoellipticity of $P(\partial)^{\K}$ (see e.g.\ \cite[p.\ 690]{F/J/W}) yields that 
$\mathcal{CW}_{P(\partial)}(\Omega)\varepsilon E
=\mathcal{CW}^{\infty}_{P(\partial)}(\Omega)\varepsilon E$. 
Thus $S_{\mathcal{CW}_{P(\partial)}(\Omega)}(u)
=S_{\mathcal{CW}^{\infty}_{P(\partial)}(\Omega)}(u)
\in\mathcal{C}^{\infty}_{P(\partial)}(\Omega,E)$ 
for all $u\in\mathcal{CW}_{P(\partial)}(\Omega)\varepsilon E $. In particular, we obtain that 
\[
S_{\mathcal{CW}_{P(\partial)}(\Omega)}\colon
\mathcal{CW}_{P(\partial)}(\Omega)\varepsilon E \to
\mathcal{CW}^{\infty}_{P(\partial)}(\Omega,E)
\]
is an isomorphism. From \prettyref{prop:diff_cons_barrelled} c) and \prettyref{thm:linearisation} with 
$(T^{E},T^{\K}):=(\id_{E^{\Omega}},\id_{\K^{\Omega}})$ we deduce that
\[
S_{\mathcal{CW}_{P(\partial)}(\Omega)}\colon
\mathcal{CW}_{P(\partial)}(\Omega)\varepsilon E \to
\mathcal{CW}_{P(\partial)}(\Omega,E)
\]
is an isomorphism into, and from
\[
S_{\mathcal{CW}_{P(\partial)}(\Omega)}
\bigl(\mathcal{CW}_{P(\partial)}(\Omega)\varepsilon E\bigr)
=\mathcal{C}^{\infty}_{P(\partial)}(\Omega,E)
\]
that $\mathcal{CW}_{P(\partial)}(\Omega,E)=\mathcal{CW}^{\infty}_{P(\partial)}(\Omega,E)$ 
as locally convex spaces, which proves our statement. 
\end{proof}

Hence the topology $\tau_{c}$ of compact convergence induced by $\mathcal{C}(\Omega,E)$ 
and the usual topology from \prettyref{ex:k_smooth_functions} induced by $\mathcal{C}^{\infty}(\Omega,E)$ 
coincide on $\mathcal{C}_{P(\partial)}(\Omega,E)$ if $P(\partial)^{\K}$ is hypoelliptic and $E$ locally complete 
by \prettyref{prop:co_top_isomorphism}. In particular, we have 
\begin{equation}\label{eq:holomorphic_coincide_1}
 (\mathcal{O}(\Omega,E),\tau_{c})
\underset{\eqref{eq:holomorphic_coincide_0}}{=}\mathcal{CW}_{\overline{\partial}}(\Omega,E)
=\mathcal{CW}^{\infty}_{\overline{\partial}}(\Omega,E)
\end{equation}
if $E$ is locally complete. 
For more interesting weights than $\mathcal{W}$ we introduce the following condition. 

\begin{cond}\label{cond:weights}
Let $\mathcal{V}:=(\nu_{j})_{j\in\N}$ be an increasing family of continuous weights on $\R^{d}$. 
Let there be $r\colon\R^{d}\to (0,1]$ and for any $j\in\N$ let there be $\psi_{j}\in L^{1}(\R^{d})$, $\psi_{j}>0$, 
and $\N\ni I_{m}(j)\geq j$ and $A_{m}(j)>0$, $m\in\{1,2,3\}$, such that for any $x\in\R^{d}$:
\begin{enumerate}
  \item [$(\alpha.1)$] $\sup_{\zeta\in\R^{d},\,\|\zeta\|_{\infty}\leq r(x)}\nu_{j}(x+\zeta)
  \leq A_{1}(j)\inf_{\zeta\in\R^{d},\,\|\zeta\|_{\infty}\leq r(x)}\nu_{I_{1}(j)}(x+\zeta)$,
  \item [$(\alpha.2)$] $\nu_{j}(x)\leq A_{2}(j)\psi_{j}(x)\nu_{I_{2}(j)}(x)$,
  \item [$(\alpha.3)$] $\nu_{j}(x)\leq A_{3}(j)r(x)\nu_{I_{3}(j)}(x)$.
\end{enumerate}
\end{cond}

Here, $\|\zeta\|_{\infty}:=\sup_{1\leq n\leq d}|\zeta_{n}|$ 
for $\zeta=(\zeta_{n})\in\R^{d}$.
The preceding condition is a special case of \cite[Condition 2.1, p.\ 176]{kruse2018_4} with 
$\Omega:=\Omega_{n}:=\R^{d}$ for all $n\in\N$.
If $\mathcal{V}$ fulfils \prettyref{cond:weights} 
and we set $\mathcal{V}^{\infty}:=(\nu_{j,m})_{j\in\N,m\in\N_{0}}$ where 
$\nu_{j,m}\colon\{\beta\in\N_{0}^{d}\;|\;|\beta|\leq m\}\times\R^{d}\to[0,\infty)$, $\nu_{j,m}(\beta,x):=\nu_{j}(x)$, 
then $\mathcal{CV}^{\infty}(\R^{d})$ and its closed subspace $\mathcal{CV}^{\infty}_{P(\partial)}(\R^{d})$ 
for $P(\partial)$ with continuous coefficients are nuclear by 
\cite[Theorem 3.1, p.\ 188]{kruse2018_4} in combination with 
\cite[Remark 2.7, p.\ 178--179]{kruse2018_4} and Fr\'echet spaces 
by \cite[Proposition 3.7, p.\ 240]{kruse2018_2}. 

\begin{prop}\label{prop:AV_CV_coincide_hol_har}
Let $E$ be a locally complete lcHs, $\mathcal{V}:=(\nu_{j})_{j\in\N}$ an increasing family of continuous weights 
on $\R^{d}$ and $\mathcal{V}^{\infty}$ defined as above. 
If $\mathcal{V}$ fulfils \prettyref{cond:weights}, then $\mathcal{CV}_{\overline{\partial}}(\C)$ 
and $\mathcal{CV}_{\Delta}(\R^{d})$ are nuclear Fr\'echet spaces and 
$\mathcal{CV}_{\overline{\partial}}(\C,E)=\mathcal{CV}^{\infty}_{\overline{\partial}}(\C,E)$ and 
$\mathcal{CV}_{\Delta}(\R^{d},E)=\mathcal{CV}^{\infty}_{\Delta}(\R^{d},E)$ as locally convex spaces.
\end{prop}
\begin{proof}
Let $P(\partial):=\overline{\partial}$ ($d:=2$ and $\K:=\C$) or $P(\partial):=\Delta$. 
First, we show that $\mathcal{CV}_{P(\partial)}(\R^{d})=\mathcal{CV}^{\infty}_{P(\partial)}(\R^{d})$ 
as locally convex spaces, which implies that 
$\mathcal{CV}_{P(\partial)}(\R^{d})$ is a nuclear Fr\'echet space as 
$\mathcal{CV}^{\infty}_{P(\partial)}(\R^{d})$ is such a space. 
Let $f\in\mathcal{CV}_{\overline{\partial}}(\C)$, $j\in\N$, $m\in\N_{0}$, $z\in\C$ 
and $\beta:=(\beta_{1},\beta_{2})\in\N_{0}^{2}$. 
Then it follows from $\|\cdot\|_{\infty}\leq |\cdot|$ and Cauchy's inequality that
\begin{align*}
 |\partial^{\beta}f(z)|\nu_{j}(z)
 &\underset{\mathclap{\eqref{eq:complex.real.deriv}}}{=}\;|\iu^{\beta_{2}}\partial^{|\beta|}_{\C}f(z)|\nu_{j}(z)
  \leq \frac{|\beta|!}{r(z)^{|\beta|}}\sup_{|w-z|=r(z)}|f(w)|\nu_{j}(z)\\
 &\underset{\mathclap{(\alpha.3)}}{\leq}\;|\beta|!C(j,|\beta|)\sup_{|w-z|=r(z)}|f(w)|\nu_{B_{3}(j)}(z)\\
 &\underset{\mathclap{(\alpha.1)}}{\leq}\;|\beta|!C(j,|\beta|)A_{1}(B_{3}(j))
  \sup_{|w-z|=r(z)}|f(w)|\nu_{I_{1}B_{3}(j)}(w)\\
 &\leq |\beta|!C(j,|\beta|)A_{1}(B_{3}(j))|f|_{\mathcal{CV}_{\overline{\partial}}(\C),I_{1}B_{3}(j)}
\end{align*}
where $C(j,|\beta|):=A_{3}(j)A_{3}(I_{3}(j))\cdots A_{3}((B_{3}-1)(j))$ and 
$B_{3}-1$ is the $(|\beta|-1)$-fold composition of $I_{3}$. Choosing $k:=\max_{|\beta|\leq m}I_{1}B_{3}(j)$, 
it follows that 
\[
 |f|_{\mathcal{CV}_{\overline{\partial}}^{\infty}(\C),j,m}
 \leq\sup_{|\beta|\leq m}|\beta|!C(j,|\beta|)A_{1}(B_{3}(j))|f|_{\mathcal{CV}_{\overline{\partial}}(\C),k}<\infty
\]
and thus $f\in\mathcal{CV}^{\infty}_{\overline{\partial}}(\C)$ and 
$\mathcal{CV}_{\overline{\partial}}(\C)=\mathcal{CV}^{\infty}_{\overline{\partial}}(\C)$ as locally convex spaces.
In the case $P(\partial)=\Delta$ an analogous proof works due to Cauchy's inequality for harmonic functions, i.e.\ 
for all $f\in\mathcal{CV}_{\Delta}(\R^{d})$, $j\in\N$, $x\in\R^{d}$ and $\beta\in\N_{0}^{d}$ it holds that
\[
 |\partial^{\beta}f(x)|\nu_{j}(x)\leq \Bigl(\frac{d|\beta|}{r(x)}\Bigr)^{|\beta|}\sup_{|w-x|<r(x)}|f(w)|\nu_{j}(x)
\]
(see e.g.\ \cite[Theorem 2.10, p.\ 23]{gilbarg_trudinger2001}).

Nuclear Fr\'echet spaces are Fr\'echet--Schwartz spaces and hence we have
\[
\mathcal{CV}^{\infty}_{P(\partial)}(\R^{d},E)\cong \mathcal{CV}^{\infty}_{P(\partial)}(\R^{d})\varepsilon E
\cong\mathcal{CV}_{P(\partial)}(\R^{d})\varepsilon E
\cong\mathcal{CV}_{P(\partial)}(\R^{d},E)
\]
by \prettyref{ex:weighted_diff} f) and \prettyref{ex:subspace_bierstedt} b). 
The isomorphism $\mathcal{CV}^{\infty}_{P(\partial)}(\R^{d},E)\cong\mathcal{CV}_{P(\partial)}(\R^{d},E)$ 
is 
\[
S_{\mathcal{CV}^{\infty}_{P(\partial)}(\R^{d})}^{-1}\circ 
\id_{\mathcal{CV}^{\infty}_{P(\partial)}(\R^{d})\varepsilon E}\circ S_{\mathcal{CV}_{P(\partial)}(\R^{d})}
=S_{\mathcal{CV}^{\infty}_{P(\partial)}(\R^{d})}^{-1}\circ S_{\mathcal{CV}_{P(\partial)}(\R^{d})}
=\id_{\mathcal{CV}^{\infty}_{P(\partial)}(\R^{d},E)}
\] 
by \prettyref{thm:full_linearisation}.
\end{proof}

\begin{rem}\label{rem:ex_NF_cond_weights}
Let $E$ be a locally complete lcHs and $P(\partial)^{\K}$ a hypoelliptic linear partial differential operator.
\begin{enumerate}
\item[a)] Let $0\leq \tau<\infty$. Then $\mathcal{V}:=(\nu_{j})_{j\in\N}$ given by 
$\nu_{j}(x):=\exp(-(\tau+\tfrac{1}{j})|x|)$, $x\in\R^{d}$, fulfils \prettyref{cond:weights} 
by \cite[Example 2.8 (iii), p.\ 179]{kruse2018_4}. 
Then $\mathcal{CV}_{P(\partial)}(\R^{d},E)$ is the space of smooth functions of \emph{\gls{exponential_type}} 
$\tau$ in the kernel of $P(\partial)$. If $\tau=0$, then the elements of these spaces are also called functions 
of \emph{\gls{infra-exponential_type}}. In particular, 
if $P(\partial)=\overline{\partial}$, $d=2$ and $\K=\C$, or $P(\partial)=\Delta$, then 
$\gls{A_dbarE}:=\mathcal{CV}_{\overline{\partial}}(\C,E)$ is the space 
of entire and $\gls{A_LapE}:=\mathcal{CV}_{\Delta}(\R^{d},E)$ 
the space of harmonic functions of exponential type $\tau$. 
\item[b)] Further examples of families of weights fulfilling \prettyref{cond:weights} 
can be found in \cite[Example 2.8, p.\ 179]{kruse2018_4} and \cite[1.5 Examples, p.\ 205]{meise1987}.
\end{enumerate}
\end{rem}

Next, we take a look at $k$-times continuously partially differentiable functions 
that vanish with all their derivatives when weighted at infinity. 
Let $k\in\N_{\infty}$, $\Omega\subset\R^{d}$ be open, 
$\omega_{m}:=M_{m}\times\Omega$ with $M_{m}:=\{\beta\in\N_{0}^{d}\;|\;|\beta|\leq \min(m,k)\}$ 
for all $m\in\N_{0}$
and $\mathcal{V}^{k}:=(\nu_{j,m})_{j\in J, m\in\N_{0}}$ be a directed family of weights on 
$(\omega_{m})_{m\in\N_{0}}$. 
We define the topological subspace of 
$\mathcal{CV}^{k}(\Omega,E)$ from \prettyref{ex:weighted_smooth_functions} a)(i) 
consisting of the functions that vanishs with all their derivatives when weighted at infinity by 
 \begin{align*}
  \gls{CV_k_OE_0}:=\{f\in\mathcal{CV}^{k}(\Omega,E)\;|\;&\forall\;j\in J,\,
  m\in\N_{0},\,\alpha\in\mathfrak{A},\,\varepsilon>0\\
  &\exists\;K\subset \Omega\;\text{compact}:\;|f|_{\Omega\setminus K,j,m,\alpha}<\varepsilon\} 
 \end{align*}
where 
 \[
  |f|_{\Omega\setminus K,j,m,\alpha}:=\sup_{\substack{x\in\Omega\setminus K\\ \beta\in M_{m}}}
  p_{\alpha}\bigl((\partial^{\beta})^{E}f(x)\bigr)\nu_{j,m}(\beta,x).
 \]
Further, we define its subspace
$\gls{CV_k_OE_0P}
:=\{f\in\mathcal{CV}^{k}_{0}(\Omega,E)\;|\;f\in\operatorname{ker}P(\partial)^{E}\}$ 
where 
\[
 P(\partial)^{E}\colon \mathcal{C}^{k}(\Omega,E)\to E^{\Omega},\;
 P(\partial)^{E}(f)(x):=\sum_{i=1}^{n}a_{i}(x)(\partial^{\beta_{i}})^{E}(f)(x),
\]
with $n\in\N$, $\beta_{i}\in\N_{0}^{d}$ such that $|\beta_{i}|\leq k$ and 
$a_{i}\colon\Omega\to\K$ for $1\leq i\leq n$. 

\begin{rem}\label{rem:CV=CV_0}
If $\mathcal{V}^{k}$ fulfils condition $(V_{\infty})$ from \prettyref{ex:weighted_C_1_diff}, 
then we have $\mathcal{CV}^{k}_{0}(\Omega,E)=\mathcal{CV}^{k}(\Omega,E)$ (see \cite[Remark 3.4, p.\ 239]{kruse2018_2}).
\end{rem}

So $\mathcal{CW}^{k}(\Omega,E)$, $\mathcal{S}(\R^{d},E)$ and $\mathcal{O}_{M}(\R^{d},E)$ 
are concrete examples of spaces $\mathcal{CV}^{k}_{0}(\Omega,E)$ (see \prettyref{cor:Schwartz}).
We present the counterpart for differentiable functions to Bierstedt's
\prettyref{thm:bierstedt} for the space $\mathcal{CV}_{0}(\Omega,E)$ 
of continuous functions from a completely regular Hausdorff space 
$\Omega$ to an lcHs $E$ weighted with a Nachbin-family $\mathcal{V}$ 
that vanish at infinity in the weighted topology. For this purpose we need the following definition. 
We call $\mathcal{V}^{k}$ \emph{\gls{locally_bounded}} on $\Omega$ if 
\[
\forall\;K\subset\Omega\;\text{compact},\,j\in J,\,m\in\N_{0},\,\beta\in M_{m}:\;
\sup_{x\in K}\nu_{j,m}(\beta,x)<\infty.
\]

\begin{exa}\label{ex:diff_vanish_at_infinity}
Let $E$ be an lcHs, $k\in\N_{\infty}$, $\mathcal{V}^{k}$ be a directed family of weights 
which is locally bounded away from zero on an open set $\Omega\subset\R^{d}$. 
\begin{enumerate}
\item [a)] $\mathcal{CV}^{k}_{0}(\Omega,E)\cong\mathcal{CV}^{k}_{0}(\Omega)\varepsilon E$ if $E$ is quasi-complete, $\mathcal{V}^{k}$ locally bounded and 
$\mathcal{CV}^{k}_{0}(\Omega)$ barrelled.
\item [b)] $\mathcal{CV}^{k}_{0}(\Omega,E)\cong\mathcal{CV}^{k}_{0}(\Omega)\varepsilon E$ if $E$ has metric ccp, 
$\mathcal{CV}^{k}_{0}(\Omega)$ is barrelled and 
$\nu_{j,m}(\beta,\cdot)\in\mathcal{C}(\Omega)$ for all $j\in J$, $m\in\N_{0}$,
$\beta\in\N_{0}^{d}$, $|\beta|\leq \min(m,k)$.
\item [c)] $\mathcal{CV}^{k}_{0}(\Omega,E)\cong\mathcal{CV}^{k}_{0}(\Omega)\varepsilon E$ if $E$ is locally complete and 
$\mathcal{CV}^{k}_{0}(\Omega)$ a Fr\'echet--Schwartz space.
\item [d)] $\mathcal{CV}^{k}_{0,P(\partial)}(\Omega,E)\cong\mathcal{CV}^{k}_{0,P(\partial)}(\Omega)\varepsilon E$ if 
$E$ is quasi-complete, $\mathcal{V}^{k}$ loc.\ bounded and $\mathcal{CV}^{k}_{0,P(\partial)}(\Omega)$ barrelled. 
\item [e)] $\mathcal{CV}^{k}_{0,P(\partial)}(\Omega,E)\cong\mathcal{CV}^{k}_{0,P(\partial)}(\Omega)\varepsilon E$ 
if $E$ has metric ccp, $\mathcal{CV}^{k}_{0,P(\partial)}(\Omega)$ is barrelled and 
$\nu_{j,m}(\beta,\cdot)\in\mathcal{C}(\Omega)$ for all $j\in J$, $m\in\N_{0}$,
$\beta\in\N_{0}^{d}$, $|\beta|\leq \min(m,k)$.
\item [f)] $\mathcal{CV}^{k}_{0,P(\partial)}(\Omega,E)\cong\mathcal{CV}^{k}_{0,P(\partial)}(\Omega)\varepsilon E$ if 
$E$ is locally complete and $\mathcal{CV}^{k}_{0,P(\partial)}(\Omega)$ a Fr\'echet--Schwartz space.
\end{enumerate} 
\end{exa} 
\begin{proof}
The generator $(T^{E}_{m},T^{\K}_{m})_{m\in\N_{0}}$ for $(\mathcal{CV}^{k}_{0},E)$
and $(\mathcal{CV}^{k}_{0,P(\partial)},E)$ is given 
by $\dom T^{E}_{m}:=\mathcal{C}^{k}(\Omega,E)$ and 
\[
 T^{E}_{m}\colon\mathcal{C}^{k}(\Omega,E)\to E^{\omega_{m}},\; 
 f\longmapsto [(\beta,x)\mapsto (\partial^{\beta})^{E}f(x)], 
\]
for all $m\in\N_{0}$ and the same with $\K$ instead of $E$.  

Set $X:=\Omega$, $\mathfrak{K}:=\{K\subset\Omega\;|\;K\;\text{compact}\}$ and 
$\pi\colon\bigcup_{m\in\N_{0}}\omega_{m}\to X$, $\pi(\beta,x):=x$. 
We have 
\[
|f|_{\Omega\setminus K,j,m,\alpha}=\sup_{\substack{x\in\omega_{m}\\ \pi(x)\notin K}}
  p_{\alpha}\bigl(T^{E}_{m}(f)(x)\bigr)\nu_{j,m}(x),
\]
for $f\in\mathcal{CV}^{k}_{0}(\Omega,E)$, $K\in\mathfrak{K}$, $j\in J$ and 
$m\in\N_{0}$, implying that \eqref{van.a.inf} is satisfied. 
With $\operatorname{AP}_{\pi,\mathfrak{K}}(\Omega,E)$ 
from \prettyref{prop:van.at.inf0} we note that
\[
  \mathcal{CV}^{k}_{0}(\Omega,E)
 =\mathcal{CV}^{k}(\Omega,E)\cap\operatorname{AP}_{\pi,\mathfrak{K}}(\Omega,E).
\]
As in \prettyref{prop:weighted_diff_strong_cons} it follows that the generator 
$(T^{E}_{m},T^{\K}_{m})_{m\in\N_{0}}$ fulfils \eqref{eq:van.at.inf_cons} 
and \eqref{eq:van.at.inf_strong}
where we use \prettyref{prop:diff_cons_barrelled}, the barrelledness of 
$\mathcal{CV}^{k}_{0}(\Omega)$ resp.\ $\mathcal{CV}^{k}_{0,P(\partial)}(\Omega)$ 
and the assumption that $\mathcal{V}^{k}$ is locally bounded away from zero on 
$\Omega$. Therefore the generator is strong and consistent by virtue of 
\prettyref{prop:van.at.inf0}. 

a)+d) Let $f\in\mathcal{CV}^{k}_{0}(\Omega,E)$, $K\in\mathfrak{K}$, $j\in J$ and $m\in\N_{0}$. 
We claim that the set 
\[
  N_{j,m}(f)
=\{\partial^{\beta})^{E}f(x)\nu_{j,m}(\beta,x)\;|\;x\in\Omega,\,\beta\in M_{m}\}
\]
is precompact in $E$ by \prettyref{prop:vanish_at_infinity_precomp}. 
Since $f$ vanishes at infinity in the weighted topology, 
condition (i) of \prettyref{prop:vanish_at_infinity_precomp} is fulfilled. 
Hence we only need to show that condition (ii) is satisfied as well, i.e.\ we 
have to show that
\[
N_{\pi\subset K,j,m}(f)=\bigcup_{\beta\in M_{m}}(\partial^{\beta})^{E}f\nu_{j,m}(\beta,\cdot)(K)
\]
is precompact in $E$. 
Thus we only have to prove that the sets $(\partial^{\beta})^{E}f\nu_{j,m}(\beta,\cdot)(K)$ are 
precompact since $N_{\pi\subset K,j,m}(f)$ is a finite union of these sets. But this is a consequence of 
the proof of \cite[\S1, 16.\ Lemma, p.\ 15]{B3} using the continuity of $(\partial^{\beta})^{E}f$ 
and the boundedness of $\nu_{j,m}(\beta,K)$, 
which follows from $\mathcal{V}^{k}$ being locally bounded.
So we deduce statements a) and d) from \prettyref{cor:full_linearisation} (iv), 
\prettyref{prop:vanish_at_infinity_precomp} and the quasi-completeness of $E$.

b)+e) The set $K_{\beta}:= \oacx ((\partial^{\beta})^{E}f\nu_{j,m}(\beta,\cdot)(\Omega))$
is absolutely convex and compact by \prettyref{prop:abs_conv_comp_C_0} (ii) 
for every $f\in\mathcal{CV}^{k}_{0}(\Omega,E)$, $j\in J$, $m\in\N_{0}$ and $\beta\in M_{m}$ 
as $E$ has metric ccp and $\nu_{j,m}(\beta,\cdot)\in\mathcal{C}(\Omega)$. 
We have
\[
 N_{j,m}(f)=\{(\partial^{\beta})^{E}f(x)\nu_{j,m}(\beta,x)\;|\;x\in \Omega,\,\beta\in M_{m}\}
    \subset\acx\bigl(\bigcup_{\beta\in M_{m}}K_{\beta}\bigr)
\]
and the set on the right-hand side is absolutely convex and compact by \cite[6.7.3 Proposition, p.\ 113]{Jarchow}. 
Now, statements b)+e) follow from \prettyref{cor:full_linearisation} (iv).

c)+f) They follow from \prettyref{cor:full_linearisation} (ii).
\end{proof}

The spaces $\mathcal{CV}^{k}_{0}(\Omega)$ are Fr\'{e}chet spaces and thus barrelled if $J$ is countable 
by \cite[Proposition 3.7, p.\ 240]{kruse2018_2}. In \cite[Theorem 5.2, p.\ 255]{kruse2018_2} 
the question is answered when they have the approximation property.
The spaces $\mathcal{CV}^{\infty}_{0}(\Omega)$ and $\mathcal{CV}^{\infty}_{P(\partial),0}(\Omega)$ 
are closed subspaces of $\mathcal{CV}^{\infty}(\Omega)$ and $\mathcal{CV}^{\infty}_{P(\partial)}(\Omega)$, 
respectively. For conditions that they are Fr\'echet--Schwartz spaces see the remarks 
below \prettyref{ex:weighted_diff}.

We already saw different choices for $\mathfrak{K}$ in \prettyref{ex:hoelder} b) and 
\prettyref{ex:diff_vanish_at_infinity}. For holomorphic functions on an open subset $\Omega$ 
of an infinite dimensional Banach space $X$ the family $\mathfrak{K}$ of \emph{$\Omega$-bounded} sets, 
i.e.\ bounded sets $K\subset\Omega$ with positive distance to $X\setminus\Omega$, 
is used in \cite[p.\ 2]{garcia2000} and \cite[p.\ 2]{jorda2013}. 
This family is clearly closed under taking finite unions, 
so \prettyref{prop:van.at.inf0} is applicable as well. 

Now, we consider an example of weighted smooth functions where the corresponding space of scalar-valued 
functions may not be barrelled. For an open set $\Omega\subset\R^{d}$, an lcHs $E$ and 
a linear partial differential operator $P(\partial)^{E}$ 
which is hypoelliptic if $E=\K$ we define the space of bounded zero solutions 
\[
\gls{C_infty_P_b}
:=\{f\in\mathcal{C}^{\infty}_{P(\partial)}(\Omega,E)\;|\;\forall\;\alpha\in\mathfrak{A}:\;
\|f\|_{\infty,\alpha}:=\sup_{x\in\Omega}p_{\alpha}(f(x))<\infty\}
\]
where $\mathcal{C}^{\infty}_{P(\partial)}(\Omega,E)$ is the kernel of $P(\partial)^{E}$ 
in $\mathcal{C}^{\infty}(\Omega,E)$. 
Further, we set $\mathcal{C}^{\infty}_{P(\partial),b}(\Omega):=\mathcal{C}^{\infty}_{P(\partial),b}(\Omega,\K)$. 
Apart from the topology given by $(\|\cdot\|_{\infty,\alpha})_{\alpha\in\mathfrak{A}}$ 
there is another weighted locally convex topology 
on $\mathcal{C}^{\infty}_{P(\partial),b}(\Omega,E)$ which is of interest, namely, the one induced by the seminorms 
\[
|f|_{\nu,\alpha}:=\sup_{x\in\Omega}p_{\alpha}(f(x))|\nu(x)|,\quad f\in\mathcal{C}^{\infty}_{P(\partial),b}(\Omega,E),
\]
for $\nu\in\mathcal{C}_{0}(\Omega)$ and $\alpha\in\mathfrak{A}$.
We denote by $(\mathcal{C}^{\infty}_{P(\partial),b}(\Omega,E),\beta)$ 
the space $\mathcal{C}^{\infty}_{P(\partial),b}(\Omega,E)$ 
equipped with the topology $\beta$ induced by the seminorms 
$(|\cdot|_{\nu,\alpha})_{\nu\in\mathcal{C}_{0}(\Omega),\alpha\in\mathfrak{A}}$. 
The topology $\beta$ is called the \emph{\gls{strict_topology}}.
It is a bit tricky to prove the $\varepsilon$-compatibility of 
$(\mathcal{C}^{\infty}_{P(\partial),b}(\Omega),\beta)$ 
and $(\mathcal{C}^{\infty}_{P(\partial),b}(\Omega,E),\beta)$  
because $(\mathcal{C}^{\infty}_{P(\partial),b}(\Omega),\beta)$ 
may not be barrelled. 

\begin{rem}\label{rem:strict_top_non_barrelled}
Let $\Omega\subset\R^{d}$ be open and $P(\partial)^{\K}$ a hypoelliptic linear partial differential operator. 
Then $(\mathcal{C}^{\infty}_{P(\partial),b}(\Omega),\beta)$ is non-barrelled if $\tau_{c}$ does not coincide with 
the $\|\cdot\|_{\infty}$-topology by \cite[Section I.1, 1.15 Proposition, p.\ 12]{cooper1978}, e.g.\ 
$(\mathcal{C}^{\infty}_{\overline{\partial},b}(\D),\beta)$ is non-barrelled.
\end{rem}

Hence we cannot use \prettyref{prop:diff_cons_barrelled} c) directly. 

\begin{prop}\label{prop:strict_top_isomorphism}
Let $\Omega\subset\R^{d}$ be open, $P(\partial)^{\K}$ a hypoelliptic linear  partial differential operator 
and $E$ an lcHs.
Then $(\mathcal{C}^{\infty}_{P(\partial),b}(\Omega),\beta)\varepsilon E
\cong (\mathcal{C}^{\infty}_{P(\partial),b}(\Omega,E),\beta)$ if $E$ has metric ccp.
\end{prop}
\begin{proof}
We set $\operatorname{AP}(\Omega,E):=\mathcal{C}^{\infty}_{P(\partial),b}(\Omega,E)$ 
and observe that $(\id_{E^{\Omega}},\id_{\Omega^{\K}})$
is the generator of $((\mathcal{C}^{\infty}_{P(\partial),b}(\Omega),\beta),E)$. 
First, we prove that the generator is consistent. 
Clearly, we only need to show that $S(u)\in\operatorname{AP}(\Omega,E)$ 
for every $u\in(\mathcal{C}^{\infty}_{P(\partial),b}(\Omega),\beta)\varepsilon E$. Let 
$u\in(\mathcal{C}^{\infty}_{P(\partial),b}(\Omega),\beta)\varepsilon E$. 
Next, we show that $u\in\mathcal{CW}^{\infty}_{P(\partial)}(\Omega)\varepsilon E$ 
with $\mathcal{CW}^{\infty}_{P(\partial)}(\Omega)$ from \prettyref{ex:weighted_smooth_functions} b).
For $\alpha\in\mathfrak{A}$ there are an absolutely convex, compact 
$K\subset(\mathcal{C}^{\infty}_{P(\partial),b}(\Omega),\beta)$ 
and $C>0$ such that for all $f'\in(\mathcal{C}^{\infty}_{P(\partial),b}(\Omega),\beta)'$ it holds that
\begin{equation}\label{eq:strict_top}
p_{\alpha}(u(f'))\leq C\sup_{f\in K}|f'(f)|.
\end{equation}
We denote by $\tau_{c}$ the topology of compact convergence on $\mathcal{C}^{\infty}_{P(\partial),b}(\Omega)$, i.e.\ 
the topology of uniform convergence on compact subsets of $\Omega$. 
From the compactness of $K$ in $(\mathcal{C}^{\infty}_{P(\partial),b}(\Omega),\beta)$ it follows that 
$K$ is $\|\cdot\|_{\infty}$-bounded and $\tau_{c}$-compact by 
\cite[Proposition 1 (viii), p.\ 586]{cooper1971} since $(\mathcal{C}^{\infty}_{P(\partial),b}(\Omega),\beta)$ 
carries the induced topology of $(\mathcal{C}_{b}(\Omega),\beta)$ and the strict topology $\beta$ 
is the mixed topology $\gamma(\tau_{c},\|\cdot\|_{\infty})$ by \cite[Proposition 3, p.\ 590]{cooper1971}. 
Let $f'\in(\mathcal{C}^{\infty}_{P(\partial)}(\Omega),\tau_{c})'$. Then there are $M\subset\Omega$ compact 
and $C_{0}>0$ such that 
\[
|f'(f)|\leq C_{0}\sup_{x\in M}|f(x)|
\]
for all $f\in \mathcal{C}^{\infty}_{P(\partial)}(\Omega)$.
Choosing a compactly supported cut-off function $\nu\in\mathcal{C}^{\infty}_{c}(\Omega)$ 
with $\nu=1$ near $M$, we obtain 
\[
|f'(f)|\leq C_{0}\sup_{x\in \Omega}|f(x)||\nu(x)|=C_{0}|f|_{\nu}
\] 
for all $f\in \mathcal{C}^{\infty}_{P(\partial)}(\Omega)$. 
Therefore $f'\in(\mathcal{C}^{\infty}_{P(\partial)}(\Omega),\beta)'$. 
In combination with the $\tau_{c}$-compactness of $K$ it follows from \eqref{eq:strict_top} that 
$u\in(\mathcal{C}^{\infty}_{P(\partial)}(\Omega),\tau_{c})\varepsilon E$. Using that 
$(\mathcal{C}^{\infty}_{P(\partial)}(\Omega),\tau_{c})=\mathcal{CW}^{\infty}_{P(\partial)}(\Omega)$ 
as locally convex spaces by the hypoellipticity of $P(\partial)^{\K}$ (see e.g.\ \cite[p.\ 690]{F/J/W}), 
we obtain that $u\in\mathcal{CW}^{\infty}_{P(\partial)}(\Omega)\varepsilon E$. 
Due to \prettyref{prop:diff_cons_barrelled} c) this yields that 
$S(u)\in\mathcal{C}^{\infty}_{P(\partial)}(\Omega,E)$. Furthermore, we note that 
\begin{align*}
\|S(u)\|_{\infty,\alpha}&=\sup_{x\in\Omega}p_{\alpha}(S(u)(x))=\sup_{x\in\Omega}p_{\alpha}(u(\delta_{x}))
\underset{\eqref{eq:strict_top}}{\leq} C\sup_{x\in\Omega}\sup_{f\in K}|\delta_{x}(f)|\\
&=C\sup_{f\in K}\|f\|_{\infty}<\infty
\end{align*}
as $K$ is $\|\cdot\|_{\infty}$-bounded, 
implying that $S(u)\in\mathcal{C}^{\infty}_{P(\partial),b}(\Omega,E)=\operatorname{AP}(\Omega,E)$. 
Hence the generator $(\id_{E^{\Omega}},\id_{\Omega^{\K}})$ is consistent. 

It is easily seen that $e'\circ f\in\mathcal{C}^{\infty}_{P(\partial),b}(\Omega)=\operatorname{AP}(\Omega)$ for all 
$e'\in E'$ and $f\in\mathcal{C}^{\infty}_{P(\partial),b}(\Omega,E)$ 
(see the proof of \prettyref{prop:weighted_diff_strong_cons}), which proves that the generator is strong as well. 
Moreover, we define $N_{\nu}(f):=\{f(x)|\nu(x)|\;|\;x\in\Omega\}$ for 
$f\in(\mathcal{C}^{\infty}_{P(\partial),b}(\Omega,E),\beta)$ and $\nu\in\mathcal{C}_{0}(\Omega)$. 
The set $K:=\oacx(N_{\nu}(f))$ is absolutely convex and compact in $E$ by \prettyref{prop:abs_conv_comp_C_0} (ii)
because $f|\nu|\in\mathcal{C}_{0}(\Omega,E)$ and $\Omega$ second-countable, yielding our statement 
by \prettyref{cor:full_linearisation} (iv).
\end{proof}

If $\Omega\subset\C$ is an open, simply connected set, $P(\partial)=\overline{\partial}$ and $E$ is complete, then 
the preceding result is also a consequence of \cite[3.10 Satz, p.\ 146]{B2}.

Next, we consider the vector-valued \emph{\gls{Beurling_Bjoerck_space}} $\mathcal{S}_{\mu}(\R^{d},E)$
which generalises the Schwartz space and whose scalar-valued counterpart was studied by Bj\"orck 
in \cite{bjoerck1965}, by Schmeisser and Triebel in \cite{schmeisser1987} 
(see \cite[Definition 1.8.1, p.\ 375]{bjoerck1965}, \cite[1.2.1.2 Definition, p.\ 15]{schmeisser1987}) 
whereas semigroups on its toplogical dual space were treated by Alvarez et al.\ in \cite{alvarez2008}. 
Since Fourier transformation is involved in the definition 
of $\mathcal{S}_{\mu}(\R^{d},E)$, we start with the following statement.

\begin{prop}\label{prop:Fourier-trafo_Bjoerck}
Let $E$ be a locally complete lcHs over $\C$, $f\in\mathcal{S}(\R^{d},E)$ and $x\in\R^{d}$.
Then $f\e^{-\iu\langle x,\cdot\rangle}$ is Pettis-integrable on $\R^{d}$ where $\langle\cdot,\cdot\rangle$ 
is the usual scalar product on $\R^{d}$.
\end{prop}
\begin{proof}
We choose $m:=d+1$ and set $\psi\colon\R^{d}\to [0,\infty)$, $\psi(\zeta):=(1+|\zeta|^{2})^{-m/2}$, as well as 
$g\colon\R^{d}\to [0,\infty)$, $g(\zeta):=\psi(\zeta)^{-1}$. 
Then $\psi\in\mathcal{L}^{1}(\R^{d},\lambda)$ and $\psi g=1$. 
Moreover, let $x=(x_{i})\in\R^{d}$ and set $u\colon\R^{d}\to E$, 
$u(\zeta):=f(\zeta)\e^{-\iu\langle x,\zeta\rangle}g(\zeta)$. 
We note that
\begin{flalign*}
&\hspace*{0.43cm}(\partial^{e_{n}})^{E}u(\zeta)\\
&=(\partial^{e_{n}})^{E}f(\zeta)\e^{-\iu\langle x,\zeta\rangle}g(\zeta)
 -\iu x_{n}f(\zeta)\e^{-\iu\langle x,\zeta\rangle}g(\zeta)
 +mf(\zeta)\e^{-\iu\langle x,\zeta\rangle}(1+|\zeta|^{2})^{(m/2)-1}\zeta_{n}
\end{flalign*}
for all $\zeta=(\zeta_{i})\in\R^{d}$ and $1\leq n\leq d$, which implies 
\[
     p_{\alpha}((\partial^{e_{n}})^{E}u(\zeta))\\
\leq p_{\alpha}((\partial^{e_{n}})^{E}f(\zeta))g(\zeta)
 +|x_{n}|p_{\alpha}(f(\zeta))g(\zeta)
 +mp_{\alpha}(f(\zeta))g(\zeta)\\
\]
for all $\alpha\in\mathfrak{A}$ and hence 
\[
\sup_{\substack{\zeta\in\R^{d}\\ \beta\in\N_{0}^{d},|\beta|\leq 1}}p_{\alpha}((\partial^{\beta})^{E}u(\zeta))
\leq (1+|x_{n}|+m)|f|_{\mathcal{S}(\R^{d}),m,\alpha}.
\]
Therefore $u=f\e^{-\iu\langle x,\cdot\rangle}g$ is (weakly) $\mathcal{C}^{1}_{b}$, which yields 
$u\in\mathcal{C}_{b}^{[1]}(\R^{d},E)$ by \prettyref{prop:abs_conv_comp_C_1_b}. 
Now, we choose $h\colon\R^{d}\to (0,\infty)$, $h(\zeta):=1+|\zeta|^{2}$. 
Then  
\[
     \sup_{\zeta\in\R^{d}}p_{\alpha}(u(\zeta)h(\zeta))
\leq \sup_{\zeta\in\R^{d}}p_{\alpha}(f(\zeta))(1+|\zeta|^{2})^{(m+2)/2}
\leq |f|_{\mathcal{S}(\R^{d}),m+2,\alpha}<\infty
\]
for all $\alpha\in\mathfrak{A}$, and for every $\varepsilon>0$ there is $r>0$ such that 
$1\leq \varepsilon h(\zeta)$ for all $\zeta\notin\overline{\mathbb{B}_{r}(0)}=:K$.
We deduce from \prettyref{prop:pettis.ccp.to.loc.complete} (iii) 
that $f\e^{-\iu\langle x,\cdot\rangle}$ is Pettis-integrable on $\R^{d}$.
\end{proof}

Thus, for $f\in\mathcal{S}(\R^{d},E)$ with locally complete $E$ the \emph{\gls{Fourier_transformation}} 
\[
  \mathfrak{F}^{E}(f)\colon\R^{d}\to E,\;\mathfrak{F}^{E}(f)(x)
:=(2\pi)^{-d/2}\int_{\R^{d}}f(\zeta)\e^{-\iu\langle x,\zeta\rangle}\d\zeta,
\]
is defined. From the Pettis-integrability we get $(e'\circ\mathfrak{F}^{E})(f)=\mathfrak{F}^{\C}(e'\circ f)$ for 
every $e'\in E'$. As $\mathfrak{F}^{\C}(e'\circ f)\in\mathcal{S}(\R^{d})$ for every $e'\in E'$ 
by \cite[Proposition 1.8.2, p.\ 375]{bjoerck1965}, 
we obtain from the weak-strong principle \prettyref{cor:weak_strong_CV} 
(or \cite[Theorem 9, p.\ 232]{B/F/J} and \cite[Mackey's theorem 23.15, p.\ 268]{meisevogt1997}) 
that $\mathfrak{F}^{E}(f)\in\mathcal{S}(\R^{d},E)$.

For a locally complete lcHs $E$ over $\C$ and a continuous function $\mu\colon\R^{d}\to[0,\infty)$ such that 
\begin{enumerate}
\item[$\gls{gamma}$] there are $a\in\R$, $b>0$ with $\mu(x)\geq a+b\ln(1+|x|)$ for all $x\in\R^{d}$,
\end{enumerate}
we set 
\[
\gls{SRE_mu}:=\{f\in\mathcal{C}^{\infty}(\R^{d},E)\;|\;\forall\;m,j\in\N_{0},\,
\alpha\in\mathfrak{A}:\;|f|_{j,m,\alpha}<\infty\}
\]
where $|f|_{m,j,\alpha}:=\max(q_{m,j,\alpha}(f),q_{m,j,\alpha}(\mathfrak{F}^{E}(f)))$ with
\[
q_{m,j,\alpha}(f):=\sup_{\substack{x\in\R^{d}\\\beta\in\mathbb{N}_{0}^{d},|\beta|\leq m}}
                   p_{\alpha}((\partial^{\beta})^{E}f(x))\e^{j\mu(x)}.
\]
We note that from $q_{m,j,\alpha}(f)<\infty$ for all $m,j\in\N_{0}$, $\alpha\in\mathfrak{A}$ 
and condition $(\gamma)$ it follows that $f\in\mathcal{S}(\R^{d},E)$ 
and hence $q_{m,j,\alpha}(\mathfrak{F}^{E}(f))$ is defined. 
Further, we set $\mathcal{S}_{\mu}(\R^{d}):=\mathcal{S}_{\mu}(\R^{d},\C)$. 
We observe that $\mathcal{S}_{\mu}(\R^{d},E)$ is a $\dom$-space. 
Indeed, let $\omega_{m}:=\widetilde{\omega}_{m}\cup\widetilde{\omega}_{m,1}$ where 
$\widetilde{\omega}_{m}:=M_{m}\times\R^{d}$ with 
$M_{m}:=\{\beta\in\N_{0}^{d}\;|\;|\beta|\leq m\}$ and 
$\widetilde{\omega}_{m,1}:=\widetilde{\omega}_{m}\times\{1\}$ for all $m\in\N_{0}$. 
Setting $\dom T^{E}_{m}:=\mathcal{S}(\R^{d},E)$ and $T^{E}_{m}\colon\mathcal{S}(\R^{d},E)\to E^{\omega_{m}}$ 
by 
\[
 T^{E}_{m}(f)(\beta,x):=(\partial^{\beta})^{E}f(x)\;\;\text{and}\;\;
 T^{E}_{m}(f)(\beta,x,1):=((\partial^{\beta})^{E}\circ\mathfrak{F}^{E})f(x),
 \quad (\beta,x)\in\widetilde{\omega}_{m},
\]
for every $m\in\N_{0}$ as well as $\operatorname{AP}(\R^{d},E):=E^{\R^{d}}$, 
we have that $\mathcal{S}_{\mu}(\R^{d},E)$ is a $\dom$-space with weights given by 
$\nu_{j,m}(\beta,x):=\nu_{j,m}(\beta,x,1):=\e^{j\mu(x)}$ 
for all $(\beta,x)\in\widetilde{\omega}_{m}$ and $m,j\in\N_{0}$.

The condition $(\gamma)$ is introduced in \cite[p.\ 363]{bjoerck1965}. 
Choosing $\mu(x):=\ln(1+|x|)$, $x\in\R^{d}$, we get the Schwartz space 
$\mathcal{S}_{\mu}(\R^{d},E)=\mathcal{S}(\R^{d},E)$ back. 

\begin{exa}\label{ex:Bjoerck}
Let $E$ be a locally complete lcHs over $\C$ and $\mu\colon\R^{d}\to[0,\infty)$ continuous such that 
condition $(\gamma)$ is fulfilled. 
\begin{enumerate}
\item[(i)] If $E$ has metric ccp, or
\item[(ii)] if $\mu\in\mathcal{C}^{1}(\R^{d})$ and there are $k\in\N_{0}$, $C>0$ such that 
$|\partial^{e_{n}}\mu(x)|\leq C\e^{k\mu(x)}$ for all $x\in\R^{d}$ and $1\leq n\leq d$,
\end{enumerate}
then $\mathcal{S}_{\mu}(\R^{d},E)\cong\mathcal{S}_{\mu}(\R^{d})\varepsilon E$. 
\end{exa}
\begin{proof}
First, we show that the generator $(T^{E}_{m},T^{\C}_{m})_{m\in\N_{0}}$ for $(\mathcal{S}_{\mu},E)$ 
is strong and consistent. From 
\[
 (\partial^{\beta})^{\C}(e'\circ f)(x)=e'\bigl((\partial^{\beta})^{E}f(x)\bigr),
 \quad (\beta,x)\in\widetilde{\omega}_{m},
\]
where $\widetilde{\omega}_{m}=\{\beta\in\N_{0}^{d}\;|\;|\beta|\leq m\}\times\R^{d}$, 
we get in combination with the Pettis-integrability by \prettyref{prop:Fourier-trafo_Bjoerck} that
\begin{equation}\label{eq:Bjoerck}
 \bigl((\partial^{\beta})^{\C}\circ\mathfrak{F}^{\C}\bigr)(e'\circ f)(x)
=e'\bigl(\bigl((\partial^{\beta})^{E}\circ\mathfrak{F}^{E}\bigr)f(x)\bigr),
 \quad (\beta,x)\in\widetilde{\omega}_{m}
\end{equation}
for all $e'\in E'$, $f\in\mathcal{S}_{\mu}(\R^{d},E)$ and $m\in\N_{0}$, which means 
that the generator is strong. 
For consistency we consider the case $\mu(x)=\ln(1+|x|)$, $x\in\R^{d}$, i.e.\ 
the Schwartz space, first. Due to \prettyref{cor:Schwartz} the map 
$S\colon\mathcal{S}(\R^{d})\varepsilon E\to \mathcal{S}(\R^{d},E)$ is an isomorphism 
and according to \prettyref{thm:full_linearisation} its inverse is given by 
\[
 R^{t}\colon \mathcal{S}(\R^{d},E)\to\mathcal{S}(\R^{d})\varepsilon E,
 \;f\mapsto \mathcal{J}^{-1}\circ R_{f}^{t}.
\]
Let $u\in\mathcal{S}(\R^{d})\varepsilon E$. Thanks to the proof of \prettyref{cor:Schwartz} 
we only need to show that 
\[
 u(\delta_{x}\circ(\partial^{\beta}\circ\mathfrak{F}^{\C}))
=(\partial^{\beta})^{E}\mathfrak{F}^{E}(S(u))(x),\quad x\in\R^{d}.
\]
We set $f:=S(u)\in\mathcal{S}(\R^{d},E)$ and from \eqref{eq:Bjoerck}
we obtain
\[
  R_{f}^{t}\bigl(\delta_{x}\circ((\partial^{\beta})^{\C}\circ\mathfrak{F}^{\C})\bigr)(e')
=(\partial^{\beta})^{\C}\mathfrak{F}^{\C}(e'\circ f)(x)
=e'\bigl((\partial^{\beta})^{E}\mathfrak{F}^{E}(f)(x)\bigr),\quad e'\in E',
\]
for all $x\in\R^{d}$ and $\beta\in\N_{0}^{d}$, which results in
\begin{align}\label{eq:bjoerck_fourier_eps}
  u\bigl(\delta_{x}\circ((\partial^{\beta})^{\C}\circ\mathfrak{F}^{\C})\bigr)
&=S^{-1}(f)\bigl(\delta_{x}\circ((\partial^{\beta})^{\C}\circ\mathfrak{F}^{\C})\bigr)
 =\mathcal{J}^{-1}\bigl(R_{f}^{t}(\delta_{x}\circ((\partial^{\beta})^{\C}\circ\mathfrak{F}^{\C}))\bigr)
  \nonumber\\
&=(\partial^{\beta})^{E}\mathfrak{F}^{E}(f)(x)=(\partial^{\beta})^{E}\mathfrak{F}^{E}(S(u))(x).
\end{align}
Thus $(T^{E}_{m},T^{\C}_{m})_{m\in\N_{0}}$ is a consistent generator for $(\mathcal{S},E)$. 

Let us turn to general $\mu$. Let $u\in\mathcal{S}_{\mu}(\R^{d})\varepsilon E$. 
We show that $u\in\mathcal{S}(\R^{d})\varepsilon E$. Then it follows from the first part of the proof that 
$(T^{E}_{m},T^{\C}_{m})_{m\in\N_{0}}$ is a consistent generator for $(\mathcal{S}_{\mu},E)$.
For $\alpha\in\mathfrak{A}$ there are an absolutely convex compact set $K\subset\mathcal{S}_{\mu}(\R^{d})$ 
and $C>0$ such that for all $f'\in\mathcal{S}_{\mu}(\R^{d})'$ it holds
\begin{equation}\label{eq:bjoerck_eps}
p_{\alpha}(u(f'))\leq C\sup_{f\in K}|f'(f)|.
\end{equation}
The compactness of $K$ in $\mathcal{S}_{\mu}(\R^{d})$ and the estimate 
\begin{align*}
 \sup_{\substack{x\in \R^{d}\\\beta\in\N_{0}^{d},|\beta|\leq m}}|(\partial^{\beta})^{\C}f(x)|(1+|x|^{2})^{j/2}
&\leq \sup_{\substack{x\in \R^{d}\\\beta\in\N_{0}^{d},|\beta|\leq m}}|(\partial^{\beta})^{\C}f(x)
      |\e^{(j/2)(2/b)(\mu(x)-a)}\\
&\leq \e^{-(aj)/b}|f|_{j,m},\quad f\in \mathcal{S}_{\mu}(\R^{d}),
\end{align*}
for all $j,m\in\N_{0}$ by condition $(\gamma)$ imply that the inclusion 
$\mathcal{S}_{\mu}(\R^{d})\hookrightarrow \mathcal{S}(\R^{d})$ is continuous 
and thus that $K$ is compact in $\mathcal{S}(\R^{d})$. 
Let $f'\in\mathcal{S}(\R^{d})'$. Then there are $j,m\in\N_{0}$ and $C_{0}>0$ such that 
\[
|f'(f)|\leq C_{0}\sup_{\substack{x\in \R^{d}\\\beta\in\N_{0}^{d},|\beta|\leq m}}|(\partial^{\beta})^{\C}f(x)|
            (1+|x|^{2})^{j/2}
       \leq C_{0}\e^{-(aj)/b}|f|_{j,m}
\]
for all $f\in\mathcal{S}_{\mu}(\R^{d})$. Hence $f'\in\mathcal{S}_{\mu}(\R^{d})'$ and 
from \eqref{eq:bjoerck_eps} we obtain that $u\in\mathcal{S}(\R^{d})\varepsilon E$ 
because $K$ is absolutely convex and compact in $\mathcal{S}(\R^{d})$. 

Condition $(\gamma)$ implies that $\mu(x)\to\infty$ for $|x|\to\infty$. 
Noting that for every $j\in\N$ and $\varepsilon>0$ there is $r>0$ such 
that
\begin{equation}\label{eq:bjoerck_vanish_infty}
\frac{\e^{j\mu(x)}}{\e^{2j\mu(x)}}=\e^{-j\mu(x)}<\varepsilon
\end{equation}
for all $x\notin\overline{\mathbb{B}_{r}(0)}$, we deduce 
$|f|_{\R^{d}\setminus \overline{\mathbb{B}_{r}(0)},m,j,\alpha}\leq\varepsilon |f|_{m,2j,\alpha}$
for every $f\in\mathcal{S}_{\mu}(\R^{d},E)$, $m\in\N_{0}$ and $\alpha\in\mathfrak{A}$. 

(i) Thus, if $E$ has metric ccp, then the sets $K_{\beta}:= \oacx ((\partial^{\beta})^{E}f\e^{j\mu}(\R^{d}))$
and $K_{\beta,1}:= \oacx ((\partial^{\beta})^{E}\mathfrak{F}^{E}(f)\e^{j\mu}(\R^{d}))$ are absolutely convex 
and compact by \prettyref{prop:abs_conv_comp_C_0} (ii) 
for every $f\in \mathcal{S}_{\mu}(\R^{d},E)$, $j,m\in\N_{0}$ and $\beta\in M_{m}$ as 
$(\partial^{\beta})^{E}f\e^{j\mu}\in\mathcal{C}_{0}(\R^{d},E)$ and 
$(\partial^{\beta})^{E}\mathfrak{F}^{E}(f)\e^{j\mu}\in\mathcal{C}_{0}(\R^{d},E)$.

(ii) We set $g_{0}\colon\R^{d}\to E$, $g_{0}(x):=(\partial^{\beta})^{E}f(x)\e^{j\mu(x)}$, and 
$g_{1}\colon\R^{d}\to E$, $g_{1}(x):=(\partial^{\beta})^{E}\mathfrak{F}^{E}(f)(x)\e^{j\mu(x)}$,
for $j,m\in\N_{0}$ and $\beta\in M_{m}$.
We observe that
\[
 (\partial^{e_{n}})^{E}g_{0}(x)
=(\partial^{\beta+e_{n}})^{E}f(x)\e^{j\mu(x)}+j(\partial^{\beta})^{E}f(x)\e^{j\mu(x)}\partial^{e_{n}}\mu(x)
\]
and 
\[
 (\partial^{e_{n}})^{E}g_{1}(x)
=(\partial^{\beta+e_{n}})^{E}\mathfrak{F}^{E}(f)(x)\e^{j\mu(x)}
 +j(\partial^{\beta})^{E}\mathfrak{F}^{E}(f)(x)\e^{j\mu(x)}\partial^{e_{n}}\mu(x)
\]
for all $x\in\R^{d}$ and $1\leq n\leq d$. 
As in \prettyref{ex:weighted_C_1_diff} it follows from condition (ii) that there are $k\in\N_{0}$, $C>0$ such that 
\[
\sup_{\substack{x\in\R^{d}\\ \gamma\in\N_{0}^{d},|\gamma|\leq 1}}p_{\alpha}((\partial^{\gamma})^{E}g_{i}(x))
\leq |f|_{m+1,j,\alpha}+Cj|f|_{m,j+k,\alpha}
\]
for all $\alpha\in\mathfrak{A}$ and $i=0,1$.
Thus $g_{0}$ and $g_{1}$ are (weakly) $\mathcal{C}^{1}_{b}$. We set $h:=\e^{j\mu}$ and note that 
\[
\sup_{x\in\R^{d}}p_{\alpha}(g_{i}(x)h(x))\leq |f|_{m,2j,\alpha}<\infty
\]
for all $\alpha\in\mathfrak{A}$ and $i=0,1$. 
This yields that $K_{\beta}=\oacx(g_{0}(\R^{d}))$ and $K_{\beta,1}=\oacx(g_{1}(\R^{d}))$
are absolutely convex and compact 
by \prettyref{prop:abs_conv_comp_hoelder} with \eqref{eq:bjoerck_vanish_infty} 
and \prettyref{prop:abs_conv_comp_C_1_b}. 

Then we have in both cases
\begin{align*}
 N_{j,m}(f)=&\phantom{\cup}\bigl(\{(\partial^{\beta})^{E}f(x)\e^{j\mu(x)}\;|\;x\in \R^{d},\,\beta\in M_{m}\}\\
 &\cup\:\{(\partial^{\beta})^{E}\mathfrak{F}^{E}(f)(x)\e^{j\mu(x)}\;|\;x\in \R^{d},\,\beta\in M_{m}\}\bigr)\\
 \subset& \acx\bigl(\bigcup_{\beta\in M_{m}}(K_{\beta}\cup K_{\beta,1})\bigr)
\end{align*}
and the set on the right-hand side is absolutely convex and compact 
by \cite[6.7.3 Proposition, p.\ 113]{Jarchow}, 
which implies that $\mathcal{S}_{\mu}(\R^{d},E)\cong\mathcal{S}_{\mu}(\R^{d})\varepsilon E$
by \prettyref{cor:full_linearisation} (iv).
\end{proof}

We come back to these spaces in \prettyref{thm:Bjoerck_Fourier}.
Another example that is related to Fourier transformation is the space 
of vector-valued smooth functions that are $2\pi$-periodic in each variable. 
We equip the space $\mathcal{C}^{\infty}(\R^{d},E)$ for an lcHs $E$ 
with the system of seminorms generated by 
\[
|f|_{K,m,\alpha}:=\hspace{-0.2cm}\sup_{\substack{x\in \R^{d}\\ \beta\in\N_{0}^{d},|\beta|\leq m}}\hspace{-0.2cm}
p_{\alpha}((\partial^{\beta})^{E}f(x))\chi_{K}(x)
=\hspace{-0.2cm}\sup_{\substack{x\in K\\ \beta\in\N_{0}^{d},|\beta|\leq m}}\hspace{-0.2cm}
p_{\alpha}((\partial^{\beta})^{E}f(x)),\quad f\in\mathcal{C}^{\infty}(\R^{d},E),
\]
for $K\subset\R^{d}$ compact, $m\in\N_{0}$ and $\alpha\in\mathfrak{A}$, 
i.e.\ we consider $\mathcal{CW}^{\infty}(\R^{d},E)$. 
By $\gls{C_infty_2pi}$ we denote the topological subspace of 
$\mathcal{C}^{\infty}(\R^{d},E)$ consisting of the functions 
which are $2\pi$-periodic in each variable. 
Further, we set $\mathcal{C}^{\infty}_{2\pi}(\R^{d}):=\mathcal{C}^{\infty}_{2\pi}(\R^{d},\K)$.

\begin{exa}\label{ex:smooth_periodic_eps_compat}
If $E$ is a locally complete lcHs, then 
$\mathcal{C}^{\infty}_{2\pi}(\R^{d},E)\cong \mathcal{C}^{\infty}_{2\pi}(\R^{d})\varepsilon E$.
\end{exa}
\begin{proof}
First, we note that for each $x\in\R^{d}$ and $1\leq n\leq d$ we have $\delta_{x}=\delta_{x+2\pi e_{n}}$ 
in $\mathcal{C}^{\infty}_{2\pi}(\R^{d})'$ and thus
\[
 S_{\mathcal{C}^{\infty}_{2\pi}(\R^{d})}(u)(x)-S_{\mathcal{C}^{\infty}_{2\pi}(\R^{d})}(u)(x+2\pi e_{n})
 =u(\delta_{x}-\delta_{x+2\pi e_{n}})=0,\quad u\in\mathcal{C}^{\infty}_{2\pi}(\R^{d})\varepsilon E,
\]
implying that $S_{\mathcal{C}^{\infty}_{2\pi}(\R^{d})}(u)$ is $2\pi$-periodic in each variable. 
In addition, we observe that $e'\circ f$ is $2\pi$-periodic in each 
variable for all $e'\in E'$ and $f\in\mathcal{C}^{\infty}_{2\pi}(\R^{d},E)$. 
Now, we obtain as in \prettyref{ex:diff_usual} a) for $k=\infty$ that 
$S_{\mathcal{C}^{\infty}_{2\pi}(\R^{d})}\colon \mathcal{C}^{\infty}_{2\pi}(\R^{d})\varepsilon E
\to \mathcal{C}^{\infty}_{2\pi}(\R^{d},E)$ is an isomorphism.
\end{proof}

We return to $\mathcal{C}^{\infty}_{2\pi}(\R^{d},E)$ in \prettyref{thm:Borel_Ritt} 
and \prettyref{thm:fourier_periodic}. 
Now, we direct our attention to spaces of continuously partially differentiable functions 
on an open bounded set such that all derivatives can be continuously extended to the boundary.
Let $E$ be an lcHs, $k\in\N_{\infty}$ and $\Omega\subset\R^{d}$ open and bounded. 
The space $\mathcal{C}^{k}(\overline{\Omega},E)$ is given by 
\[
 \gls{C_k_onclosure}:=\{f\in\mathcal{C}^{k}(\Omega,E)\;|\;(\partial^{\beta})^{E}f\;
 \text{cont.\ extendable on}\;\overline{\Omega}\;\text{for all}\;\beta\in\N^{d}_{0},\,|\beta|\leq k\}
\]
and equipped with the system of seminorms given by 
\[
 |f|_{\alpha}:=\sup_{\substack{x\in \Omega\\ \beta\in\N^{d}_{0}, |\beta|\leq k}}
               p_{\alpha}\bigl((\partial^{\beta})^{E}f(x)\bigr),\quad f\in\mathcal{C}^{k}(\overline{\Omega},E),
\]
for $\alpha\in \mathfrak{A}$ if $k<\infty$, and by 
\[
 |f|_{m,\alpha}:=\sup_{\substack{x\in \Omega\\ \beta\in\N^{d}_{0}, |\beta|\leq m}}
 p_{\alpha}\bigl((\partial^{\beta})^{E}f(x)\bigr),\quad f\in\mathcal{C}^{\infty}(\overline{\Omega},E),
\]
for $m\in\N_{0}$ and $\alpha\in \mathfrak{A}$ if $k=\infty$. 
Further, we set $\mathcal{C}^{k}(\overline{\Omega}):=\mathcal{C}^{k}(\overline{\Omega},\K)$.

\begin{exa}\label{ex:diff_ext_boundary}
Let $E$ be an lcHs, $k\in\N_{\infty}$ and $\Omega\subset\R^{d}$ open and bounded. 
\begin{enumerate}
\item[(i)] If $E$ has metric ccp, or
\item[(ii)] if $E$ is locally complete, $k=\infty$ and there exists $C>0$ such that for each $x,y\in\Omega$ 
there is a continuous path from $x$ to $y$ in $\Omega$ whose length is bounded by $C|x-y|$, 
\end{enumerate}
then $\mathcal{C}^{k}(\overline{\Omega},E)\cong\mathcal{C}^{k}(\overline{\Omega})\varepsilon E$. 
\end{exa}
\begin{proof}
The generator coincides with the one of \prettyref{ex:diff_vanish_at_infinity}. 
Due to \prettyref{prop:diff_cons_barrelled} we have $S(u)\in\mathcal{C}^{k}(\Omega,E)$ and 
\[
(\partial^{\beta})^{E}S(u)(x)=u(\delta_{x}\circ(\partial^{\beta})^{\K}),
\quad \beta\in\N_{0}^{d},\;|\beta|\leq k,\; x\in\Omega,
\]
for all $u\in\mathcal{C}^{k}(\overline{\Omega})\varepsilon E$ 
since $\mathcal{C}^{k}(\overline{\Omega})$ is a Banach space if $k<\infty$, and a Fr\'{e}chet space if $k=\infty$, 
in particular, both are barrelled. 
As a consequence of \prettyref{prop:cont_ext} and \prettyref{lem:cont_ext} with $T=(\partial^{\beta})^{\K}$ 
for $\beta\in\N^{d}_{0}$, $|\beta|\leq k$, 
we obtain that $(\partial^{\beta})^{E}S(u)\in\mathcal{C}^{ext}(\Omega,E)$ 
for all $u\in\mathcal{C}^{k}(\overline{\Omega})\varepsilon E$. 
Thus the generator is consistent. It is easy to check that it is strong, too. 
This yields (ii) by \prettyref{cor:full_linearisation} (ii) since 
$\mathcal{C}^{\infty}(\overline{\Omega})$ is a nuclear Fr\'echet space 
by \cite[Examples 28.9 (5), p.\ 350]{meisevogt1997} under the conditions on $\Omega$.

Let us turn to part (i). Let $f\in\mathcal{C}^{k}(\overline{\Omega},E)$, $J:=\{1\}$, $m\in\N_{0}$ and 
set $M_{m}:=\{\beta\in\N^{d}_{0}\;|\;|\beta|\leq k\}$ if $k<\infty$,
and $M_{m}:=\{\beta\in\N^{d}_{0}\;|\;|\beta|\leq m\}$ if $k=\infty$. 
We denote by $f_{\beta}$ the continuous extension of $(\partial^{\beta})^{E}f$ on the 
compact metrisable set $\overline{\Omega}$. The set
\[
 N_{1,m}(f)=\{(\partial^{\beta})^{E}f(x)\;|\;x\in \Omega,\,\beta\in M_{m}\}
 \subset\bigcup_{\beta\in M_{m}}f_{\beta}(\overline{\Omega}) 
\]
is relatively compact and metrisable since it is a subset of a finite union of the compact metrisable sets 
$f_{\beta}(\overline{\Omega})$ as in \prettyref{ex:diff_usual}. 
Due to \prettyref{cor:full_linearisation} (iv) we obtain our statement (i) as $E$ has metric ccp.
\end{proof}

We close this section by an examination of the topological subspace
\[
\gls{E_0_E}:=\{f\in\mathcal{C}^{\infty}([0,1],E)\;|\;\forall\;k\in\N_{0}:\;(\partial^{k})^{E}f(1)=0\}
\]
where $(\partial^{k})^{E}f(1):=\lim_{x\to 1\rlim}(\partial^{k})^{E}f(x)$. 
Further, we set $\mathcal{E}_{0}:=\mathcal{E}_{0}(\K)$.

\begin{exa}\label{ex:E_0}
Let $E$ be a locally complete lcHs. Then $\mathcal{E}_{0}\varepsilon E\cong\mathcal{E}_{0}(E)$.
\end{exa}
\begin{proof}
We note that $\Omega:=(0,1)$ satisfies the condition on $\Omega$ in \prettyref{ex:diff_ext_boundary} (ii) 
with $C:=1$ and thus $\mathcal{C}^{\infty}([0,1])$ and its closed subspace $\mathcal{E}_{0}$ are nuclear Fr\'echet 
spaces. The generator coincides with the one of \prettyref{ex:diff_ext_boundary}.
From the proof of \prettyref{ex:diff_ext_boundary} we know that 
\[
 \lim_{x\to 1\rlim}(\partial^{k})^{E}S(u)(x)
=u(\delta_{1}\circ(\partial^{k})^{\K})=u(0)=0,\quad k\in\N_{0},
\]
for all $u\in\mathcal{E}_{0}\varepsilon E$. In combination with \prettyref{ex:diff_ext_boundary} 
this yields the consistency of the generator. Again, its strength is easy to check. 
Therefore our statement is valid by \prettyref{cor:full_linearisation} (ii).
\end{proof}
\section{Riesz--Markov--Kakutani representation theorems}
\label{sect:riesz_markov_kakutani}
In this subsection we generalise the concept of strength and consistency such that it is not strictly 
bounded to $\dom$-spaces and their generators anymore. 
This allows us to answer the question: Given $T^{\K}_{m}\in\F'$ is there $T^{E}_{m}\in L(\FE,E)$ such that 
$(T^{E}_{m},T^{\K}_{m})$ is strong and consistent? 
Furthermore, we will see that the operators $T^{E}_{m}$ are usually the ones that 
can be obtained from integral representations of $T^{\K}_{m}$, 
i.e.\ we transfer Riesz--Markov--Kakutani theorems from the scalar-valued to the vector-valued case. 
We recall that the \emph{\gls{rieszmarkovkakutani}} for compact topological Hausdorff spaces $\Omega$ says that 
for every $T^{\R}\in\mathcal{C}_{b}(\Omega)'$ there is a unique regular $\R$-valued Borel measure $\mu$ 
on $\Omega$ such that
\begin{equation}\label{eq:riesz_markov_kakutani_original}
T^{\R}(f)=\int_{\Omega}f(x)\d\mu(x),\quad f\in \mathcal{C}_{b}(\Omega),
\end{equation}
which was proved by Riesz \cite[p.\ 976]{riesz1909} in the case $\Omega:=[0,1]$ and by Kakutani 
\cite[Theorem 9, p.\ 1009]{kakutani1941} for general compact Hausdorff $\Omega$ 
(see Saks \cite[Eq.\ (1.1), 6.,  p.\ 408, 411]{saks1938} for compact metric $\Omega$). 
Markov treated the case where $\Omega$ is a normal (not necessarily Hausdorff) topological space and 
the $T^{\R}$ are positive linear functionals on $\mathcal{C}_{b}(\Omega)$ such that $T^{\R}(1)=1$ 
\cite[Definition 2, p.\ 167]{markov1938}. In this case, for every such $T^{\R}$ there is a unique exterior density 
$\mu$ on $\Omega$ in the sense of \cite[Definition 3, p.\ 167]{markov1938} 
such that \eqref{eq:riesz_markov_kakutani_original} holds by \cite[Theorem 22, p.\ 184]{markov1938} 
and the right-hand side is read in the sense of \cite[Eq.\ (71), (72), (80), p.\ 180--181]{markov1938} 
(see also \cite[IV.6.2 Theorem, p.\ 262]{dunford1958} for a more familiar version 
with regular (finitely) additive bounded Borel measures $\mu$). 

\begin{defn}[{strong, consistent}]\label{def:cons_strong}
Let $E$ be an lcHs and $\Omega$ a non-empty set. Let $\F\subset\K^{\Omega}$ and $\FE\subset E^{\Omega}$ 
be lcHs such that $\delta_{x}\in\F'$ for all $x\in\Omega$. 
Let $(\omega_m)_{m\in M}$ be a family of non-empty sets, 
$T^{\K}_{m}\colon\dom T^{\K}_{m}\to \K^{\omega_{m}}$ 
and $T^{E}_{m}\colon\dom T^{E}_{m}\to E^{\omega_{m}}$ be linear with 
$\F\subset\dom T^{\K}_{m}\subset\K^{\Omega}$ and $\FE\subset\dom T^{E}_{m}\subset E^{\Omega}$ for all $m\in M$.
\begin{enumerate}
\item[a)] We call $(T^{E}_{m},T^{\K}_{m})_{m\in M}$ a \emph{\gls{consistent_fam}} for $(\F,E)$, 
in short $(\mathcal{F},E)$, if we have for every $u\in\F\varepsilon E$, $m\in M$ and $x\in\omega_{m}$ that
\begin{enumerate}
\item[(i)] $S(u)\in\FE$ and $T^{\K}_{m,x}:=\delta_{x}\circ T^{\K}_{m}\in\F'$,
\item[(ii)] $T^{E}_{m}S(u)(x)=u(T^{\K}_{m,x})$.
\end{enumerate}
\item[b)] We call $(T^{E}_{m},T^{\K}_{m})_{m\in M}$ a \emph{\gls{strong_fam}} for $(\F,E)$, 
in short $(\mathcal{F},E)$, if we have for every $e'\in E'$, $f\in\FE$, $m\in M$ and $x\in\omega_{m}$ that
\begin{enumerate}
\item[(i)] $e'\circ f\in\F$,
\item[(ii)] $T^{\K}_{m}(e'\circ f)(x)=\bigl(e'\circ T^{E}_{m}(f)\bigr)(x)$.
\end{enumerate}
\end{enumerate}
\end{defn}

Note that $\omega_{m}$ need not be a subset of $\Omega$. 
As a convention we omit the index $m$ of the set $\omega_{m}$, the operators $T^{E}_{m}$ and $T^{\K}_{m}$ 
if $M$ is a singleton. 
The following remark shows that the preceding definition of a consistent resp.\ strong family 
coincides with the usual definition in the case of generators of $\dom$-spaces (see \prettyref{def:consist_strong}). 

\begin{rem}
Let $(T^{E}_{m},T^{\K}_{m})_{m\in M}$ be a generator for $(\mathcal{FV},E)$. 
We note that the condition $T^{\K}_{m,x}\in\FV'$ for all $m\in M$ and $x\in\omega_{m}$ in a)(i) 
of \prettyref{def:cons_strong} is always satisfied for generators 
by \prettyref{rem:weights_Hausdorff_directed} b). Moreover, 
if $S(u)\in\operatorname{AP}(\Omega,E)\cap\dom T^{E}_{m}$ for $u\in\FV\varepsilon E$ and all $m\in M$
and a)(ii) of \prettyref{def:cons_strong} is fulfilled, then $S(u)\in\FVE$ 
by \prettyref{lem:topology_eps}, implying that a)(i) is satisfied. 
Further, if $f\in\FVE$ and $e'\circ f\in \operatorname{AP}(\Omega)\cap\dom T^{\K}_{m}$ 
for all $e'\in E'$ and $m\in M$ and b)(ii) of \prettyref{def:cons_strong} is fulfilled, 
then $e'\circ f\in\FV$ by \prettyref{lem:strong_is_weak}, implying that b)(i) is satisfied. 
\end{rem}

The next proposition is the key result in transferring Riesz--Markov--Kakutani theorems 
from the scalar-valued to the vector-valued case. To state this proposition we need that our
map $S\colon \F\varepsilon E\to \FE$ is an isomorphism and that its inverse is given as in 
\prettyref{thm:full_linearisation}, i.e.\ that 
\[
 R^{t}\colon \FE\to\F\varepsilon E,\;f\mapsto \mathcal{J}^{-1}\circ R_{f}^{t},
\]
is the inverse of $S$ where $R_{f}^{t}(f')(e')=f'(e'\circ f)$, for $f'\in\F'$ and $e'\in E'$, and 
$\mathcal{J}\colon E\to E'^{\star}$ is the canonical injection in the algebraic dual $E'^{\star}$ of $E'$.

\begin{prop}\label{prop:pettis_consistent}
Let $E$ be an lcHs, $(\Omega,\Sigma,\mu)$ a measure space and 
$\F$ and $\FE$ $\varepsilon$-compatible with inverse $R^{t}$ of $S$ 
and $(T^{E}_{0},T^{\K}_{0})$ a strong family for $(\mathcal{F},E)$ with $\omega_{0}:=\Omega$.
If $T^{\K}_{0}(f)$ is integrable for every $f\in\F$ and $T^{E}_{0}(f)$ is Pettis-integrable on $\Omega$ 
for every $f\in\FE$ and 
\[
T^{\K}\colon \F\to\K,\;T^{\K}(f):=\int_{\Omega}T^{\K}_{0}(f)(x)\d\mu(x),
\]
is continuous, then 
\[
u(T^{\K})=\int_{\Omega}T^{E}_{0}S(u)(x)\d\mu(x),\quad u\in\F\varepsilon E.
\] 
\end{prop}
\begin{proof}
Let $u\in\F\varepsilon E$ and set $f:=S(u)\in\FE$. We have
\[
 R_{f}^{t}(T^{\K})(e')
=T^{\K}(e'\circ f)
=\int_{\Omega}T^{\K}_{0}(e'\circ f)(x)\d\mu(x)
=\langle e',\int_{\Omega}T^{E}_{0}f(x)\d\mu(x)\rangle,\quad e'\in E',
\]
by the strength of $(T^{E}_{0},T^{\K}_{0})$ and the Pettis-integrability of $T^{E}_{0}(f)$, which yields
\[
 u(T^{\K})
=S^{-1}(f)(T^{\K})
=\mathcal{J}^{-1}\bigl(R_{f}^{t}(T^{\K})\bigr)
=\int_{\Omega}T^{E}_{0}f(x)\d\mu(x)=\int_{\Omega}T^{E}_{0}S(u)(x)\d\mu(x)
\]
due to $R^{t}$ being the inverse of $S$.
\end{proof}

\begin{prop}\label{prop:pettis_consistent_rev}
Let $E$ be an lcHs, $(\Omega,\Sigma,\mu)$ a measure space, $(T^{E}_{0},T^{\K}_{0})$ a strong family for $(\mathcal{F},E)$ 
with $\omega_{0}:=\Omega$ such that $T^{E}_{0}(f)$ is Pettis-integrable on $\Omega$ for every $f\in\FE$, and 
$(T^{E},T^{\K})$ a consistent family for $(\mathcal{F},E)$ such that 
\[
T^{E}(f)=\int_{\Omega}T^{E}_{0}f(x)\d\mu(x),\quad f\in\FE.
\]
Then $(T^{E},T^{\K})$ is a strong family for $(\mathcal{F},E)$, $T^{\K}_{0}(f)$ is integrable for every $f\in\F$ and
\[
T^{\K}=\int_{\Omega}T^{\K}_{0}(f)(x)\d\mu(x),\quad f\in\F.
\]
\end{prop}
\begin{proof}
We set $f\cdot e\colon \Omega\to E$, $(f\cdot e)(x):=f(x)e$, for $e\in E$ and $f\in\F$.
Since $(T^{E},T^{\K})$ is a consistent family for $(\mathcal{F},E)$, we get 
$f\cdot e = S(\Theta(e\otimes f))\in\FE$ and 
\begin{equation}\label{eq:pettis_cons_rev}
 T^{E}(f\cdot e)
=\Theta(e\otimes f)(T^{\K})
=T^{\K}(f)e
\end{equation}
with the map $\Theta$ from \eqref{eq:tensor_into_eps_product}.
From the strength of $(T^{E}_{0},T^{\K}_{0})$ we deduce that 
\[
 e'\circ T^{E}_{0}(f\cdot e)
=T^{\K}_{0}(e'\circ (f\cdot e))
=T^{\K}_{0}(e'(e)f)=e'(e)T^{\K}_{0}(f)
\]
and from the Pettis-integrability of $T^{E}_{0}(f\cdot e)$ that
\[
 T^{\K}(f)e'(e)
\underset{\eqref{eq:pettis_cons_rev}}{=}\langle e', T^{E}(f\cdot e)\rangle
=\int_{\Omega}e'(e)T^{\K}_{0}(f)(x)\d\mu(x)
\]
for all $e'\in E'$. This implies that $e'(e)T^{\K}_{0}(f)$ is integrable for all $e'\in E'$. 
Further, since $E$ is non-trivial by our assumptions in \prettyref{chap:notation}, 
there is some $e_0\in E$, $e_0\neq 0$. By the Hahn--Banach theorem there is some 
$e_0'\in E'$ with $e_0'(e_0)\neq 0$, which yields that $T^{\K}_{0}(f)$ is integrable and 
\[
T^{\K}(f)=\int_{\Omega}T^{\K}_{0}(f)(x)\d\mu(x).
\]
Furthermore, we conclude in combination with the strength of $(T^{E}_{0},T^{\K}_{0})$ and the 
Pettis-integrability of $T^{E}_{0}(f)$ for all $f\in\FE$ that
\[
 \langle e',T^{E}(f)\rangle
=\int_{\Omega}T^{\K}_{0}(e'\circ f)(x)\d\mu(x)
=T^{\K}(e'\circ f)
\]
for all $f\in\FE$ and $e'\in E'$, which means that $(T^{E},T^{\K})$ is a strong family for $(\mathcal{F},E)$.
\end{proof}

Let us apply the preceding propositions to the space $D([0,1],E)$ of $E$-valued c\`{a}dl\`{a}g functions on $[0,1]$. For $f\in D([0,1],E)$ we set $f(x\llim):=\lim_{w\to x\llim}f(w)$ 
if $x\in(0,1]$, and $f(0\llim):=0$.

\begin{prop}\label{prop:representation_cadlag}
Let $E$ be a quasi-complete lcHs. 
Then for every $T^{\K}\in D([0,1])'$ there is $T^{E}\in L(D([0,1],E),E)$ such that 
$(T^{E},T^{\K})$ is a consistent family for $(D,E)$ and 
there are a unique regular $\K$-valued Borel measure $\mu$ on $[0,1]$ and a unique
$\varphi\in\ell^{1}([0,1],\K)$ such that 
\begin{equation}\label{eq:representation_cadlag}
T^{E}(f)=\int_{[0,1]}f(x)\d\mu(x)+\sum_{x\in [0,1]}(f(x)-f(x\llim))\overline{\varphi(x)},\quad f\in D([0,1],E).
\end{equation}
On the other hand, if $(T^{E},T^{\K})$ is a consistent family, there is a unique regular $\K$-valued
Borel measure $\mu$ on $[0,1]$ such that \eqref{eq:representation_cadlag} holds 
and $T^{E}\in L(D([0,1],E),E)$.
\end{prop}
\begin{proof}
Due to the representation theorem \cite[Theorem 1, p.\ 383]{pestman1995} there 
are a unique regular $\K$-valued Borel measure $\mu$ on $[0,1]$ and a unique
$\varphi\in\ell^{1}([0,1],\K)$ such that 
\begin{equation}\label{eq:unique_meas_cadlag}
T^{\K}(f)=\int_{[0,1]}f(x)\d\mu(x)+\sum_{x\in [0,1]}(f(x)-f(x\llim))\overline{\varphi(x)},\quad f\in D([0,1],\K).
\end{equation}
By \prettyref{ex:cadlag} $S\colon D([0,1])\varepsilon E\to D([0,1],E)$ 
is an isomorphism with inverse $R^{t}\colon f\mapsto \mathcal{J}\circ R^{t}_{f}$. 

The next part is the analogon of \prettyref{prop:pettis_consistent} for $\sum_{x\in[0,1]}$. We set 
\[
T^{\K}_{1}\colon D([0,1])\to \K,\;T^{\K}_{1}(f):=\sum_{x\in [0,1]}(f(x)-f(x\llim))\overline{\varphi(x)},
\]
and note that $T^{\K}_{1}\in D([0,1])'$. 
Let $u\in D([0,1])\varepsilon E$ and set $f:=S(u)\in D([0,1],E)$. We have
\begin{align*}
  R_{f}^{t}(T^{\K}_{1})(e')
&=T^{\K}_{1}(e'\circ f)
 =\sum_{x\in [0,1]}((e'\circ f)(x)-(e'\circ f)(x\llim))\overline{\varphi(x)}\\
&=\langle e',\sum_{x\in [0,1]}(f(x)-f(x\llim))\overline{\varphi(x)}\rangle,\quad e'\in E',
\end{align*}
due to the Pettis-summability of $x\mapsto (f(x)-f(x\llim))\overline{\varphi(x)}$ on $[0,1]$
by \prettyref{prop:cadlag_pettis}, which yields
\begin{align}\label{eq:pettis_sum_cons}
  u(T^{\K}_{1})
&=S^{-1}(f)(T^{\K}_{1})
 =\mathcal{J}^{-1}\bigl(R_{f}^{t}(T^{\K}_{1})\bigr)
 =\sum_{x\in [0,1]}(f(x)-f(x\llim))\overline{\varphi(x)}\notag\\
&=\sum_{x\in [0,1]}(S(u)(x)-S(u)(x\llim))\overline{\varphi(x)}.
\end{align}

We note that every $f\in D([0,1],E)$ is Pettis-integrable and that  
$x\mapsto (f(x)-f(x\llim))\overline{\varphi(x)}$ is Pettis-summable on $[0,1]$
by \prettyref{prop:cadlag_pettis}. Further, 
\[
p_{\alpha}\bigl(\int_{[0,1]}f(x)\d\mu(x)+\sum_{x\in[0,1]}(f(x)-f(x\llim))\overline{\varphi(x)}\bigr)
\leq (|\mu|([0,1])+2\|\varphi\|_{\ell^{1}})\sup_{x\in [0,1]}p_{\alpha}(f(x))
\]
for all $f\in D([0,1],E)$ and $\alpha\in\mathfrak{A}$. 
The rest follows from \prettyref{prop:pettis_consistent} 
with $(T^{E}_{0},T^{\K}_{0}):=(\id_{E^{[0,1]}},\id_{\K^{[0,1]}})$ combined with \eqref{eq:pettis_sum_cons}. 
For the uniqueness of $\mu$ in \eqref{eq:representation_cadlag} use that the $\mu$ in \eqref{eq:unique_meas_cadlag} is 
unique and \prettyref{prop:pettis_consistent_rev} (and for the uniqueness of $\varphi$ use an analogon 
of \prettyref{prop:pettis_consistent_rev} for $(T^{E}_{1},T^{\K}_{1})$).
\end{proof}

Let us turn to continuous functions that vanish at infinity.

\begin{prop}\label{prop:representation_C_0}
Let $\Omega$ be a locally compact [second countable] topological Hausdorff space and $E$ an lcHs with [metric] ccp. 
Then for every $T^{\K}\in\mathcal{C}_{0}(\Omega)'$ there is $T^{E}\in L(\mathcal{C}_{0}(\Omega,E),E)$ such that 
$(T^{E},T^{\K})$ is a consistent family for $(\mathcal{C}_{0},E)$ and 
there is a unique regular $\K$-valued Borel measure $\mu$ on $\Omega$ such that 
\begin{equation}\label{eq:representation_C_0}
T^{E}(f)=\int_{\Omega}f(x)\d\mu(x),\quad f\in \mathcal{C}_{0}(\Omega,E).
\end{equation}
On the other hand, if $(T^{E},T^{\K})$ is a consistent family, then there is a unique regular $\K$-valued 
Borel measure $\mu$ on $\Omega$ such that \eqref{eq:representation_C_0} holds 
and $T^{E}\in L(\mathcal{C}_{0}(\Omega,E),E)$.
\end{prop}
\begin{proof}
Due to the Riesz--Markov--Kakutani representation theorem (see \cite[6.19 Theorem, p.\ 130]{rudin1970}) there 
is a unique regular $\K$-valued Borel measure $\mu$ on $\Omega$ such that 
\begin{equation}\label{eq:unique_meas_C_0}
T^{\K}(f)=\int_{\Omega}f(x)\d\mu(x),\quad f\in \mathcal{C}_{0}(\Omega).
\end{equation}
By \prettyref{ex:cont_loc_comp} $S\colon\mathcal{C}_{0}(\Omega)\varepsilon E\to\mathcal{C}_{0}(\Omega,E)$ 
is an isomorphism with inverse $R^{t}\colon f\mapsto \mathcal{J}\circ R^{t}_{f}$. 
We note that every $f\in\mathcal{C}_{0}(\Omega,E)$ is Pettis-integrable 
by \prettyref{prop:pettis.ccp.to.loc.complete} (i) resp.\ (ii) with $\psi:=g:=1$ since 
\[
\int_{\Omega}|\psi(x)|\d|\mu|(x)=|\mu|(\Omega)<\infty
\]
and 
\[
p_{\alpha}\bigl(\int_{\Omega}f(x)\d\mu(x)\bigr)\leq |\mu|(\Omega)\sup_{x\in\Omega}p_{\alpha}(f(x)),
\quad f\in \mathcal{C}_{0}(\Omega,E),
\]
for all $\alpha\in\mathfrak{A}$. The rest follows from \prettyref{prop:pettis_consistent} 
with $(T^{E}_{0},T^{\K}_{0}):=(\id_{E^{\Omega}},\id_{\K^{\Omega}})$. For the uniqueness of $\mu$ in 
\eqref{eq:representation_C_0} use that the $\mu$ in \eqref{eq:unique_meas_C_0} is 
unique and \prettyref{prop:pettis_consistent_rev}.
\end{proof}

Next, we consider the space of bounded continuous $E$-valued functions on a locally compact topological 
Hausdorff space $\Omega$, i.e.\ 
\[
\mathcal{C}_{b}(\Omega,E)=\{f\in\mathcal{C}(\Omega,E)\;|\;\forall\;\alpha\in\mathfrak{A}:\;
\sup_{x\in\Omega}p_{\alpha}(f(x))<\infty\},
\]
but equipped with the strict topology $\beta$ (see \prettyref{rem:strict_top_non_barrelled}) 
which is induced by the seminorms 
\[
|f|_{\nu,\alpha}:=\sup_{x\in\Omega}p_{\alpha}(f(x))|\nu(x)|,\quad f\in\mathcal{C}_{b}(\Omega,E),
\]
for $\nu\in\mathcal{C}_{0}(\Omega)$ and $\alpha\in\mathfrak{A}$. 

\begin{prop}\label{prop:representation_C_b_strict}
Let $\Omega$ be a locally compact [second countable] topological Hausdorff space and $E$ an lcHs with [metric] ccp. 
Then for every $T^{\K}\in(\mathcal{C}_{b}(\Omega),\beta)'$ there is 
$T^{E}\in L((\mathcal{C}_{b}(\Omega,E),\beta),E)$ such that 
$(T^{E},T^{\K})$ is a consistent family for $((\mathcal{C}_{b}(\Omega),\beta),E)$ and 
there is a unique regular $\K$-valued Borel measure $\mu$ on $\Omega$ such that 
\begin{equation}\label{eq:representation_C_b_strict}
T^{E}(f)=\int_{\Omega}f(x)\d\mu(x),\quad f\in \mathcal{C}_{b}(\Omega,E).
\end{equation}
On the other hand, if $(T^{E},T^{\K})$ is a consistent family, then there is a unique regular $\K$-valued
Borel measure $\mu$ on $\Omega$ such that \eqref{eq:representation_C_b_strict} holds 
and $T^{E}\in L((\mathcal{C}_{b}(\Omega,E),\beta),E)$.
\end{prop}
\begin{proof}
Due to the Riesz--Markov--Kakutani representation theorem 
\cite[7.6.3 Theorem, p.\ 141]{Jarchow} for the strict topology 
there is a unique regular $\K$-valued Borel measure $\mu$ on $\Omega$ such that 
\[
T^{\K}(f)=\int_{\Omega}f(x)\d\mu(x),\quad f\in \mathcal{C}_{b}(\Omega).
\]
Since $T^{\K}$ is continuous, there are $\nu\in\mathcal{C}_{0}(\Omega)$ and $C>0$ such that 
\[
 \bigl|\int_{\Omega}\langle e',f(x)\rangle\d\mu(x)\bigr|
=|T^{\K}(e'\circ f)|
\leq C\sup_{x\in\Omega}|(e'\circ f)(x)\nu(x)|
\leq C\sup_{x\in K}|e'(x)|,\quad e'\in E',
\]
for $f\in \mathcal{C}_{b}(\Omega,E)$ with $K:=\oacx(f\nu(\Omega))$. As $K$ is absolutely convex and compact by 
\prettyref{prop:abs_conv_comp_C_0}, $f$ is Pettis-integrable on $\Omega$ w.r.t.\ $\mu$ by the Mackey--Arens theorem. 
The remaining parts of the proof follow from \prettyref{ex:cont_loc_comp} and \prettyref{prop:pettis_consistent}  
as in \prettyref{prop:representation_C_0}. 
\end{proof}

\begin{prop}\label{prop:representation_C_compact_open}
Let $\Omega$ be a locally compact [second countable] topological Hausdorff space and $E$ an lcHs with [metric] ccp. 
Then for every $T^{\K}\in\mathcal{CW}(\Omega)'$ there is 
$T^{E}\in L(\mathcal{CW}(\Omega,E),E)$ such that 
$(T^{E},T^{\K})$ is a consistent family for $(\mathcal{CW},E)$ and 
there is a unique regular $\K$-valued Borel measure $\mu$ on $\Omega$ with compact support such that 
\begin{equation}\label{eq:representation_C_compact_open}
T^{E}(f)=\int_{\Omega}f(x)\d\mu(x),\quad f\in \mathcal{C}(\Omega,E).
\end{equation}
On the other hand, if $(T^{E},T^{\K})$ is a consistent family, then there is a unique regular $\K$-valued 
Borel measure $\mu$ on $\Omega$ with compact support such that \eqref{eq:representation_C_compact_open} holds 
and $T^{E}\in L(\mathcal{CW}(\Omega,E),E)$.
\end{prop}
\begin{proof} 
By the Riesz--Markov--Kakutani representation theorem given in the remark after 
\cite[Chap.\ 4, \S4.8, Proposition 14, p.\ INT IV.48]{bourbakiI} for the topology of compact convergence
there is a unique regular $\K$-valued Borel measure $\mu$ on $\Omega$ with compact support such that 
\[
T^{\K}(f)=\int_{\Omega}f(x)\d\mu(x),\quad f\in \mathcal{C}(\Omega).
\]
Since $T^{\K}$ is continuous, there are a compact set $M\subset\Omega$ and $C>0$ such that 
\[
 \bigl|\int_{\Omega}\langle e',f(x)\rangle\d\mu(x)\bigr|
=|T^{\K}(e'\circ f)|
\leq C\sup_{x\in M}|(e'\circ f)(x)|
\leq C\sup_{x\in K}|e'(x)|,\quad e'\in E',
\]
for $f\in \mathcal{C}(\Omega,E)$ with the absolutely convex and compact set $K:=\oacx(f(M))$, 
implying that $f$ is Pettis-integrable on $\Omega$ w.r.t.\ $\mu$ by the Mackey--Arens theorem. 
The rest of the proof is identical to the one of \prettyref{prop:representation_C_b_strict}.
\end{proof}

\begin{prop}\label{prop:representation_C_infty}
Let $\Omega\subset\R^{d}$ be open and $E$ a locally complete lcHs. 
Then for every $T^{\K}\in\mathcal{CW}^{\infty}(\Omega)'$ there is 
$T^{E}\in L(\mathcal{CW}^{\infty}(\Omega,E),E)$ such that 
$(T^{E},T^{\K})$ is a consistent family for $(\mathcal{CW}^{\infty},E)$. 
Given any open neighbourhood $U\subset\Omega$ of the compact distributional support 
$\operatorname{supp} T^{\K}$ of $T^{\K}$
there are $m\in\N_{0}$ and a family of $\K$-valued Radon measures 
$(\mu_{\beta})_{\beta\in\N_{0}^{d},|\beta|\leq m}$ on $\Omega$ such that 
$\operatorname{supp} \mu_{\beta}\subset U$ for all $\beta\in\N_{0}^{d}$, $|\beta|\leq m$, and
\begin{equation}\label{eq:representation_C_infty}
T^{E}(f)=\sum_{|\beta|\leq m}\int_{\Omega}(\partial^{\beta})^{E}f(x)\d\mu_{\beta}(x),
\quad f\in \mathcal{C}^{\infty}(\Omega,E).
\end{equation}
On the other hand, if $(T^{E},T^{\K})$ is a consistent family, then given any open neighbourhood 
$U\subset\Omega$ of the compact distributional support 
$\operatorname{supp} T^{\K}$ of $T^{\K}$
there are $m\in\N_{0}$ and a family of $\K$-valued Radon measures 
$(\mu_{\beta})_{\beta\in\N_{0}^{d},|\beta|\leq m}$ on $\Omega$ such that 
\eqref{eq:representation_C_infty} holds, 
$\operatorname{supp} \mu_{\beta}\subset U$ for all $\beta\in\N_{0}^{d}$, $|\beta|\leq m$ 
and $T^{E}\in L(\mathcal{CW}^{\infty}(\Omega,E),E)$.
\end{prop}
\begin{proof}
$T^{\K}\in\mathcal{CW}^{\infty}(\Omega)'$ is a distribution with compact support 
and thus has finite order by \cite[Corollary, p.\ 259]{Treves}. Denote by $m\in\N_{0}$ the 
order of $T^{\K}$. Given any open neighbourhood $U\subset\Omega$ 
of $\operatorname{supp} T^{\K}$ there is a family of $\K$-valued Radon measures 
$(\mu_{\beta})_{\beta\in\N_{0}^{d},|\beta|\leq m}$ on $\Omega$ such that 
\[
T^{\K}(f)=\sum_{|\beta|\leq m}\int_{\Omega}(\partial^{\beta})^{\K}f(x)\d\mu_{\beta}(x),
\quad f\in \mathcal{C}^{\infty}(\Omega),
\]
and $\operatorname{supp} \mu_{\beta}\subset U$ for all $\beta\in\N_{0}^{d}$, $|\beta|\leq m$ 
by \cite[Theorem 24.4, p.\ 259]{Treves}.
Since the support $K_{\beta}:=\operatorname{supp} \mu_{\beta}$ of $\mu_{\beta}$ is compact
\[
\int_{\Omega}(\partial^{\beta})^{\K}f(x)\d\mu_{\beta}(x)=\int_{K_{\beta}}(\partial^{\beta})^{\K}f(x)\d\mu_{\beta}(x),
\quad f\in \mathcal{C}^{\infty}(\Omega),
\]
and $(\partial^{\beta})^{E}f\in\mathcal{C}^{1}(\Omega,E)$ for $f\in\mathcal{C}^{\infty}(\Omega,E)$, 
it follows from \prettyref{lem:pettis.loc.complete} 
and \prettyref{rem:pettis.loc.complete} that $(\partial^{\beta})^{E}f$ is Pettis-integrable on 
$\Omega$ w.r.t.\ $\mu_{\beta}$ for all $\beta$ and that 
\[
     p_{\alpha}\bigl(\int_{\Omega}(\partial^{\beta})^{E}f(x)\d\mu_{\beta}(x)\bigr)
\leq |\mu_{\beta}|(K_{\beta})\sup_{x\in K_{\beta}}p_{\alpha}((\partial^{\beta})^{E}f(x)),\quad\alpha\in\mathfrak{A}.
\]
By \prettyref{ex:diff_usual} a) the map 
$S\colon \mathcal{CW}^{\infty}(\Omega)\varepsilon E\to \mathcal{CW}^{\infty}(\Omega,E)$ 
is an isomorphism with inverse $R^{t}\colon f\mapsto \mathcal{J}\circ R^{t}_{f}$. 
The remaining parts of the proof follow from \prettyref{prop:pettis_consistent} 
with $(T^{E}_{0},T^{\K}_{0}):=((\partial^{\beta})^{E},(\partial^{\beta})^{\K})$.
\end{proof}

\begin{prop}\label{prop:representation_Schwartz}
Let $E$ be a locally complete lcHs. 
Then for every $T^{\K}\in\mathcal{S}(\R^{d})'$ there is 
$T^{E}\in L(\mathcal{S}(\R^{d},E),E)$ such that 
$(T^{E},T^{\K})$ is a consistent family for $(\mathcal{S},E)$
and there are $m\in\N_{0}$ and a family of continuous functions 
$(g_{\beta})_{\beta\in\N_{0}^{d},|\beta|\leq m}$ on $\R^{d}$ 
growing at infinity slower than some polynomial such that
\begin{equation}\label{eq:representation_Schwartz}
T^{E}(f)=\sum_{|\beta|\leq m}\int_{\R^{d}}g_{\beta}(x)(\partial^{\beta})^{E}f(x)\d x,
\quad f\in \mathcal{S}(\R^{d},E).
\end{equation}
On the other hand, if $(T^{E},T^{\K})$ is a consistent family, 
then there are $m\in\N_{0}$ and a family of continuous functions 
$(g_{\beta})_{\beta\in\N_{0}^{d},|\beta|\leq m}$ on $\R^{d}$ 
growing at infinity slower than some polynomial such that
\eqref{eq:representation_Schwartz} holds
and $T^{E}\in L(\mathcal{S}(\R^{d},E),E)$.
\end{prop}
\begin{proof}
Let $T^{\K}\in\mathcal{S}(\R^{d})'$. Then there are $m\in\N_{0}$ 
and a family of continuous functions 
$(g_{\beta})_{\beta\in\N_{0}^{d},|\beta|\leq m}$ on $\R^{d}$ 
growing at infinity slower than some polynomial such that
\[
T^{\K}(f)=\sum_{|\beta|\leq m}\int_{\R^{d}}g_{\beta}(x)(\partial^{\beta})^{\K}f(x)\d x,
\quad f\in \mathcal{S}(\R^{d}),
\]
by \cite[Theorem 25.4, p.\ 272]{Treves}. Here, $g_{\beta}$ \emph{growing at infinity slower than some polynomial}
means that there are $k\in\N_{0}$ and $C>0$ such that $|g_{\beta}(x)|\leq C(1+|x|^{2})^{k/2}$ for all $x\in\R^{d}$. 
Since the family $(g_{\beta})$ is finite, we can take one $k$ and one $C$ for all $\beta$. 
Due to the proof of \prettyref{ex:weighted_C_1_diff} and \prettyref{cor:Schwartz} we know that 
$K_{\beta}:=\oacx(((\partial^{\beta})^{E}f)(1+|\cdot|^{2})^{k/2}(\R^{d}))$ is absolutely convex and compact 
for $f\in \mathcal{S}(\R^{d},E)$. 
The estimate 
\[
\bigl|\int_{\R^{d}}\langle e',g_{\beta}(x)(\partial^{\beta})^{E}f(x)\rangle\d x\bigr|
\leq C\sup_{x\in\R^{d}}|e'((\partial^{\beta})^{E}f(x))|(1+|x|^{2})^{k/2}
=C\sup_{x\in K_{\beta}}|e'(x)|
\]
for all $e'\in E'$ and $f\in \mathcal{S}(\R^{d},E)$ yields that $g_{\beta}(\partial^{\beta})^{E}f$ 
is Pettis-integrable on $\R^{d}$ w.r.t.\ the Lebesgue measure by the Mackey--Arens theorem. 
Further, it implies that 
\[
     p_{\alpha}\bigl(\int_{\R^{d}}g_{\beta}(x)(\partial^{\beta})^{E}f(x)\d x\bigr)
\leq C\sup_{x\in\R^{d}}p_{\alpha}((\partial^{\beta})^{E}f(x))(1+|x|^{2})^{k/2},\quad\alpha\in\mathfrak{A},
\]
as in \prettyref{lem:pettis.loc.complete}.
By \prettyref{cor:Schwartz} the map 
$S\colon \mathcal{S}(\R^{d})\varepsilon E\to \mathcal{S}(\R^{d},E)$ 
is an isomorphism with inverse $R^{t}\colon f \mapsto \mathcal{J}\circ R^{t}_{f}$. 
The remaining parts of the proof follow from \prettyref{prop:pettis_consistent} 
with $(T^{E}_{0},T^{\K}_{0}):=(g_{\beta}(\partial^{\beta})^{E},g_{\beta}(\partial^{\beta})^{\K})$.
\end{proof}

\begin{rem}\fakephantomsection\label{rem:representation_subspaces}
\begin{enumerate}
\item[a)] Let $\Omega\subset\R^{d}$ be open and $E$ a locally complete lcHs. Then 
\prettyref{prop:representation_C_infty} is still valid with $\mathcal{CW}^{\infty}$ 
replaced by $\mathcal{CW}^{\infty}_{P(\partial)}$ due to the Hahn--Banach theorem and  
\prettyref{ex:diff_usual} b). If $P(\partial)^{\K}$ is a hypoelliptic linear partial differential operator, 
then one can represent $T^{E}$ as in \eqref{eq:representation_C_compact_open} 
due to \prettyref{prop:co_top_isomorphism} but the measure $\mu$ need not be unique anymore. 
\item[b)] Let $\Omega\subset\R^{d}$ be open and $E$ an lcHs with metric ccp. 
Then \prettyref{prop:representation_C_b_strict} is still valid with $\mathcal{C}_{b}$ replaced 
by $\mathcal{C}^{\infty}_{P(\partial),b}$ for a hypoelliptic linear partial differential operator $P(\partial)^{\K}$
due to the Hahn--Banach theorem and \prettyref{prop:strict_top_isomorphism} 
but the measure $\mu$ need not be unique anymore. 
\item[c)] All families $(T^{E},T^{\K})$ considered in this section are strong 
which is a consequence of \prettyref{prop:pettis_consistent_rev} 
(and of Pettis-summability in \prettyref{prop:representation_cadlag}). 
\end{enumerate}
\end{rem}
\chapter{Applications}
\label{chap:applications}
\section{Lifting the properties of maps from the scalar-valued case}
\label{sect:lifting}
In this section we briefly show how to use the $\varepsilon$-compatibility of 
spaces $\F$ and $\FE$ to lift properties like injectivity, surjectivity, 
bijectivity and continuity from a map $T^{\K}$ to a map $T^{E}$ if 
$(T^{E},T^{\K})$ forms a consistent family. 
Especially, we pay attention to surjectivity whose transfer to the vector-valued case is 
accomplished by Grothendieck's classical theory of tensor products of Fr\'echet 
spaces \cite{Gro} and by the splitting theory of Vogt for Fr\'{e}chet spaces \cite{vogt1983} and of 
Bonet and Doma\'nski  for PLS-spaces \cite{D/L}.
In order to apply splitting theory, we recall the definitions of the topological invariants $(\Omega)$, 
$(DN)$ and $(PA)$.

Let us recall that a Fr\'echet space $F$ with an increasing fundamental system of 
seminorms $(\vertiii{\cdot}_{k})_{k\in\N}$ satisfies $\gls{Omega}$ if
\[
\forall\; p\in\N\; \exists\; q\in\N\;\forall\; k\in\N\;\exists\; n\in\N,\,C>0\;\forall\; r>0:\;
 U_{q}\subset Cr^n U_k + \frac{1}{r} U_p
\]
where $U_{k}:=\{x\in F \; | \; \vertiii{x}_{k}\leq 1\}$ (see \cite[Chap.\ 29, Definition, p.\ 367]{meisevogt1997}). 

We recall that a Fr\'echet space $(F,(\vertiii{\cdot}_{k})_{k\in\N})$ satisfies $\gls{DN}$
by \cite[Chap.\ 29, Definition, p.\ 359]{meisevogt1997} if
\[
\exists\;p\in\N\;\forall\;k\in\N\;\exists\;n\in\N,\,C>0\;\forall\;x\in F:\;
\vertiii{x}^{2}_{k}\leq C\vertiii{x}_{p}\vertiii{x}_{n}.
\]
A \emph{\gls{PLS_space}} is a projective limit $X=\lim\limits_{\substack{\longleftarrow\\N\in\N}}X_{N}$, where the $X_{N}$
given by inductive limits $X_{N}=\lim\limits_{\substack{\longrightarrow\\n\in \N}}(X_{N,n},\vertiii{\cdot}_{N,n})$ 
are DFS-spaces (which are also called LS-spaces), and it satisfies $\gls{PA}$ if
\begin{gather*}
\forall \;N\;\exists\; M\; \forall\; K\; \exists\; n\; \forall\; m\; \forall\; \eta >0\; \exists\; k,C,r_0 >0\;  
\forall\; r>r_0\; \forall\; x'\in X'_{N}:\\
	\vertiii{x'\circ i^{M}_{N}}^{\ast}_{M,m}
	\leq C\bigl(r^{\eta}\vertiii{x'\circ i^{K}_{N}}^{\ast}_{K,k}+\frac{1}{r}\vertiii{x'}^{\ast}_{N,n}\bigr)
\end{gather*}
where $\vertiii{\cdot}^{\ast}$ denotes the dual norm of $\vertiii{\cdot}$ and 
$i^{M}_{N}$, $i^{K}_{N}$ the linking maps (see \cite[Section 4, Eq.\ (24), p.\ 577]{Dom1}).

Examples of Fr\'{e}chet spaces with $(DN)$ are the spaces of rapidly decreasing sequences 
$s(\N^{d})$, $s(\N_{0}^{d})$ and $s(\Z^{d})$, the space $\mathcal{C}^{\infty}([a,b])$ of all
$\mathcal{C}^{\infty}$-smooth functions on $(a,b)$ such that all derivatives 
can be continuously extended to the boundary and the space of smooth
functions $\mathcal{C}^{\infty}_{2\pi}(\R^{d})$ that are $2\pi$-periodic 
in each variable. 
Examples of ultrabornological PLS-space with $(PA)$ are Fr\'{e}chet--Schwartz spaces, 
the space of tempered distributions $\mathcal{S}(\R^{d})_{b}'$, 
the space of distributions $\mathcal{D}(\Omega)_{b}'$ and ultradistributions of Beurling type 
$\mathcal{D}_{(\omega)}(\Omega)_{b}'$ on an open set $\Omega\subset\R^{d}$. 
These and many more examples may be found in \cite{Dom1}, \cite[Corollary 4.8, p.\
1116]{D/L} and \cite[Example 3, p.\ 7]{kruse2019_1}.

\begin{prop}\fakephantomsection\label{prop:isom_op_spaces}
\begin{enumerate}
\item [a)] Let $Y$ be a Fr\'echet space and $X$ a semi-reflexive lcHs. 
Then $L_{b}(X_{b}',Y_{b}')\cong L_{b}(Y,(X_{b}')_{b}')$ via taking adjoints.
\item [b)] Let $E$ be an lcHs and $X$ a Montel space. 
Then $L_{b}(X_{b}',E)\cong X\varepsilon E$ where the isomorphism is the identity map.
\end{enumerate}
\end{prop}
\begin{proof}
a) We consider the map
\[
{^{t}(\cdot)}\colon L_{b}(X_{b}',Y_{b}')\to L_{b}(Y,(X_{b}')_{b}'),\; u\mapsto {^{t}u},
\]
defined by ${^{t}u}(y)(x'):=u(x')(y)$ for $y\in Y$ and $x'\in X'$.
First, we prove that ${^{t}(\cdot)}$ is well-defined. Let $u\in L(X_{b}',Y_{b}')$ and $y\in Y$. 
Since $u\in L(X_{b}',Y_{b}')$ and 
$\{y\}$ is bounded in $Y$, there are a bounded set $B\subset X$ and $C>0$ such that
\[
|{^{t}u}(y)(x')|=|u(x')(y)|\leq C\sup_{x\in B}|x'(x)|
\]
for all $x'\in X'$, implying ${^{t}u}(y)\in(X_{b}')'$.
  
Let us denote by $(\|\cdot\|_{Y,n})_{n\in\N}$ the (directed) system of seminorms generating the metrisable locally convex topology of $Y$. 
The canonical embedding $J\colon Y\to (Y_{b}')_{b}'$ is an isomorphism between $Y$ and $J(Y)$ 
by \cite[Corollary 25.10, p.\ 298]{meisevogt1997} because $Y$ is a Fr\'echet space. 
For a bounded set $M\subset X_{b}'$ we note that
\[
 \sup_{x'\in M}|{^{t}u}(y)(x')|
=\sup_{x'\in M}|u(x')(y)|
=\sup_{x'\in M}|\langle J(y),u(x')\rangle|.
\]
The next step is to prove that $u(M)$ is bounded in $Y_{b}'$. 
Let $N\subset Y$ be bounded. Since $u\in L(X_{b}',Y_{b}')$, there are again a bounded set $B\subset X$ 
and a constant $C>0$ such that
\[
\sup_{x'\in M}\sup_{y\in N}|u(x')(y)|\leq C\sup_{x'\in M}\sup_{x\in B}|x'(x)|<\infty
\]
where the last estimate follows from the boundedness of $M\subset X_{b}'$. 
Hence $u(M)$ is bounded in $Y_{b}'$.
By the remark about the canonical embedding there are $n\in\N$ and $C_{0}>0$ such that
\[
 \sup_{x'\in M}|{^{t}u}(y)(x')|
=\sup_{y'\in u(M)}|\langle J(y),y'\rangle|
\leq C_{0}\|y\|_{Y,n},
\]
so ${^{t}u}\in L(Y,(X_{b}')_{b}')$ and the map ${^{t}(\cdot)}$ is well-defined.

Let us turn to injectivity. Let $u,v\in L(X_{b}',Y_{b}')$ with ${^{t}u}={^{t}v}$. This is equivalent to
\[
u(x')(y)={^{t}u}(y)(x')={^{t}v}(y)(x')=v(x')(y)
\]
for all $y\in Y$ and $x'\in X'$. This implies $u(x')=v(x')$ for all $x'\in X'$, hence $u=v$.

Next, we turn to surjectivity. We consider the map
\[
{^{\mathrm{t}}(\cdot)}\colon L_{b}(Y,(X_{b}')_{b}')\to L_{b}(X_{b}',Y_{b}'),\; u\mapsto {^{\mathrm{t}}u},
\]
defined by ${^{\mathrm{t}}u}(x')(y):=u(y)(x')$ for $x'\in X'$ and $y\in Y$. We show that this map is well-defined.
Let $u\in L_{b}(Y,(X_{b}')_{b}')$ and $x'\in X'$. 
Since $u\in L_{b}(Y,(X_{b}')_{b}')$ and $\{x'\}$ is bounded in $X'$, there are $n\in\N$ and $C>0$ such that
\[
|{^{\mathrm{t}}u}(x')(y)|=|u(y)(x')|\leq C\|y\|_{Y,n}
\]
for all $y\in Y$, yielding ${^{\mathrm{t}}u}(x')\in Y'$.
Let $B\subset Y$ be bounded. The semi-reflexivity of $X$ implies that for every $u(y)$, $y\in B$, 
there is a unique $x_{u(y)}\in X$ such that $u(y)(x')=x'(x_{u(y)})$ for all $x'\in X'$. Then we get
\[
 \sup_{y\in B}|{^{\mathrm{t}}u}(x')(y)|
=\sup_{y\in B}|u(y)(x')|
=\sup_{y\in B}|x'(x_{u(y)})|.
\]
We claim that $D:=\{x_{u(y)}\;|\;y\in B\}$ is a bounded set in $X$. Let $N\subset X'$ be finite. 
Then the set $M:=\{{^{\mathrm{t}}u}(x')\;|\;x'\in N\}\subset Y'$ is finite. 
We have
\[
 \sup_{y\in B}\sup_{x'\in N}|x'(x_{u(y)})|
=\sup_{y\in B}\sup_{x'\in N}|{^{\mathrm{t}}u}(x')(y)|
=\sup_{y\in B}\sup_{y'\in M}|y'(y)|<\infty
\]
where the last estimate follows from the fact that the bounded set $B$ is weakly bounded. 
Thus $D$ is weakly bounded and by \cite[Mackey's theorem 23.15, p.\ 268]{meisevogt1997} bounded in $X$. 
Therefore it follows from
\[
 \sup_{y\in B}|{^{\mathrm{t}}u}(x')(y)|
=\sup_{y\in B}|x'(x_{u(y)})|
=\sup_{x\in D}|x'(x)|
\]
for all $x'\in X'$ that ${^{\mathrm{t}}u}\in L(X_{b}',Y_{b}')$, which means that ${^{\mathrm{t}}(\cdot)}$ is well-defined.
Let $u\in L(Y,(X_{b}')_{b}')$. Then we have ${^{\mathrm{t}}u}\in L_{b}(X_{b}',Y_{b}')$. 
In addition, for all $y\in Y$ and all $x'\in X'$
\[
{^{t}({^{\mathrm{t}}u})}(y)(x')={^{\mathrm{t}}u}(x')(y)=u(y)(x')
\]
is valid and so ${^{t}({^{\mathrm{t}}u})}(y)=u(y)$ for all $y\in Y$, proving the surjectivity.

The last step is to prove the continuity of ${^{t}(\cdot)}$ and its inverse. 
Let $M\subset Y$ and $B\subset X_{b}'$ be bounded sets. Then
\begin{align*}
  \sup_{y\in M}\sup_{x'\in B}|{^{t}u}(y)(x')|
&=\sup_{y\in M}\sup_{x'\in B}|u(x')(y)|
 =\sup_{x'\in B}\sup_{y\in M}|u(x')(y)|\\
&=\sup_{x'\in B}\sup_{y\in M}|{^{\mathrm{t}}({^{t}u})}(x')(y)|
\end{align*}
holds for all $u\in L(X_{b}',Y_{b}')$. Therefore, ${^{t}(\cdot)}$ and its inverse are continuous.
	 
b) Let $T\in L(X_{b}',E)$. For $\alpha\in\mathfrak{A}$ there are a bounded set $B\subset X$ and $C>0$ such that
\[
p_{\alpha}(T(x'))\leq C\sup_{x\in B}|x'(x)|\leq C\sup_{x\in \oacx{(B)}}|x'(x)|
\]
for every $x'\in X'$ where $\oacx{(B)}$ is the closure of the absolutely convex hull of $B$. 
The set $\oacx{(B)}$ is absolutely convex and compact by \cite[6.2.1 Proposition, p.\ 103]{Jarchow} 
and \cite[6.7.1 Proposition, p.\ 112]{Jarchow} since $B$ is bounded in the Montel space $X$. 
Hence we gain $T\in L(X_{\kappa}',E)$.

Let $M\subset X'$ be equicontinuous. Due to \cite[8.5.1 Theorem (a), p.\ 156]{Jarchow} $M$ is bounded in $X_{b}'$. 
Therefore,
\[
\id\colon L_{b}(X_{b}',E)\to L_{e}(X_{\kappa}',E)=X\varepsilon E
\]
is continuous.

Let $T\in L(X_{\kappa}',E)$. For $\alpha\in\mathfrak{A}$ there are an absolutely convex compact set $B\subset X$ 
and $C>0$ such that
\[
p_{\alpha}(T(x'))\leq C\sup_{x\in B}|x'(x)|
\]
for every $x'\in X'$. Since the compact set $B$ is bounded, we get $T\in L(X_{b}',E)$.

Let $M$ be a bounded set in $X_{b}'$. Then $M$ is equicontinuous by virtue of \cite[Theorem 33.2, p.\ 349]{Treves}, 
as $X$, being a Montel space, is barrelled. 
Thus
\[
\id\colon L_{e}(X_{\kappa}',E)\to L_{b}(X_{b}',E)
\]
is continuous.
\end{proof}

For part e) of the next theorem we need that our
map $S_{\mathcal{F}_{2}(\Omega_{2})}\colon \mathcal{F}_{2}(\Omega_{2})\varepsilon E\to 
\mathcal{F}_{2}(\Omega_{2},E)$ is an isomorphism and that its inverse is given as in 
\prettyref{thm:full_linearisation}, i.e.\ that 
\[
 R^{t}\colon \mathcal{F}_{2}(\Omega_{2},E)\to \mathcal{F}_{2}(\Omega_{2})\varepsilon E,\;f\mapsto \mathcal{J}^{-1}\circ R_{f}^{t},
\]
is the inverse of $S_{\mathcal{F}_{2}(\Omega_{2})}$ where $R_{f}^{t}(f')(e')=f'(e'\circ f)$ 
for $f'\in\mathcal{F}_{2}(\Omega_{2})'$ and $e'\in E'$, and $\mathcal{J}\colon E\to E'^{\star}$ is the canonical injection 
in the algebraic dual $E'^{\star}$ of $E'$.

\begin{thm}\label{thm:eps_prod_surj_inj} 
Let $E$ be an lcHs, $\mathcal{F}_{1}(\Omega_{1})$ and $\mathcal{F}_{1}(\Omega_{1},E)$ as well as 
$\mathcal{F}_{2}(\Omega_{2})$ and $\mathcal{F}_{2}(\Omega_{2},E)$ be $\varepsilon$-into-compatible. 
Let $(T^{E},T^{\K})$ be a consistent family for $(\mathcal{F}_{1}, E)$ such that 
$T^{\K}\colon \mathcal{F}_{1}(\Omega_{1})\to \mathcal{F}_{2}(\Omega_{2})$ is continuous and 
$T^{E}\colon \mathcal{F}_{1}(\Omega_{1},E)\to \mathcal{F}_{2}(\Omega_{2},E)$. 
Then the following holds:
\begin{enumerate}
\item[a)] $T^{E}\circ S_{\mathcal{F}_{1}(\Omega_{1})}
=S_{\mathcal{F}_{2}(\Omega_{2})}\circ(T^{\K}\varepsilon\id_{E})$.
\item[b)] If $S_{\mathcal{F}_{1}(\Omega_{1})}$ is surjective and $T^{\K}$ is injective, 
then $T^{E}$ is injective, continuous and 
\[
T^{E}=S_{\mathcal{F}_{2}(\Omega_{2})}\circ(T^{\K}\varepsilon\id_{E})\circ S^{-1}_{\mathcal{F}_{1}(\Omega_{1})}. 
\]
If in addition $S_{\mathcal{F}_{2}(\Omega_{2})}$ is surjective and $T^{\K}$ an isomorphism, then 
$T^{E}$ is an isomorphism with inverse
\[
 (T^{E})^{-1}
=S_{\mathcal{F}_{1}(\Omega_{1})}\circ((T^{\K})^{-1}\varepsilon\id_{E})\circ S^{-1}_{\mathcal{F}_{2}(\Omega_{2})}. 
\]
\item[c)] If $S_{\mathcal{F}_{2}(\Omega_{2})}$ and $T^{\K}\varepsilon\id_{E}$ are surjective, then 
$T^{E}$ is surjective.
\item[d)] If $S_{\mathcal{F}_{2}(\Omega_{2})}$ and $T^{\K}$ are surjective, 
$\mathcal{F}_{1}(\Omega_{1})$, $\mathcal{F}_{2}(\Omega_{2})$ and $E$ are Fr\'echet spaces and 
\begin{enumerate}
\item[(i)] $\mathcal{F}_{1}(\Omega_{1})$ and $\mathcal{F}_{2}(\Omega_{2})$ are nuclear, or
\item[(ii)] $E$ is nuclear, 
\end{enumerate}
then $T^{E}$ is surjective. 
\item[e)] If $S_{\mathcal{F}_{2}(\Omega_{2})}$ is surjective with inverse $R^{t}$,  
$T^{\K}$ is surjective, 
$\mathcal{F}_{1}(\Omega_{1})$ and $\mathcal{F}_{2}(\Omega_{2})$ are Fr\'echet spaces, 
$\ker T^{\K}$ is nuclear and has $(\Omega)$, and 
\begin{enumerate}
\item[(i)] $\mathcal{F}_{1}(\Omega_{1})$ and $\mathcal{F}_{2}(\Omega_{2})$ are Montel spaces, 
$E=F_{b}'$ where $F$ is a Fr\'echet space satisfying $(DN)$, or
\item[(ii)] $\mathcal{F}_{1}(\Omega_{1})$ and $\mathcal{F}_{2}(\Omega_{2})$ are Schwartz spaces, 
$E$ is an ultrabornological PLS-space satisfying $(PA)$, 
\end{enumerate}
then $T^{E}$ is surjective. 
\end{enumerate}
\end{thm}
\begin{proof}
a) Let $u\in\mathcal{F}_{1}(\Omega_{1})\varepsilon E$. Then  
\begin{align*}
 (T^{E}\circ S_{\mathcal{F}_{1}(\Omega_{1})})(u)(x)
&=u(\delta_{x}\circ T^{\K})=(u\circ (T^{\K})^{t})(\delta_{x})
 =(T^{\K}\varepsilon\id_{E})(u)(\delta_{x})\\
&=S_{\mathcal{F}_{2}(\Omega_{2})}\bigl((T^{\K}\varepsilon\id_{E})(u)\bigr)(x),\quad x\in \Omega_{2},
\end{align*}
as $(T^{E},T^{\K})$ is consistent for $(\mathcal{F}_{1}, E)$, which proves part a). 

b) If $S_{\mathcal{F}_{1}(\Omega_{1})}$ is surjective, then $S_{\mathcal{F}_{1}(\Omega_{1})}$ is 
an isomorphism, because it is an isomorphism into, and we have 
\[
T^{E}=S_{\mathcal{F}_{2}(\Omega_{2})}\circ(T^{\K}\varepsilon\id_{E})\circ S^{-1}_{\mathcal{F}_{1}(\Omega_{1})}
\]
by part a). If $T^{\K}$ is injective, then $T^{\K}\varepsilon\id_{E}$ 
is also injective by \cite[Chap.\ I, \S1, Proposition 1, p.\ 20]{Sch1} and thus $T^{E}$ by the formula above as well 
since $S_{\mathcal{F}_{1}(\Omega_{1})}$ is an isomorphism and $S_{\mathcal{F}_{2}(\Omega_{2})}$ an isomorphism into. 
If $S_{\mathcal{F}_{2}(\Omega_{2})}$ is surjective and $T^{\K}$ an isomorphism, 
then $S_{\mathcal{F}_{2}(\Omega_{2})}$ and $T^{\K}\varepsilon\id_{E}$ are isomorphisms, the latter by 
\cite[Chap.\ I, \S1, Proposition 1, p.\ 20]{Sch1} and its inverse is $(T^{\K})^{-1}\varepsilon\id_{E}$. 
The rest of part b) follows from the formula for $T^{E}$ above. 

c) Let $f\in\mathcal{F}_{2}(\Omega_{2},E)$. Then there is $g\in\mathcal{F}_{1}(\Omega_{1})\varepsilon E$ such 
that $(S_{\mathcal{F}_{2}(\Omega_{2})}\circ(T^{\K}\varepsilon\id_{E}))(g)=f$. Hence we obtain  
$h:=S_{\mathcal{F}_{1}(\Omega_{1})}(g)\in\mathcal{F}_{1}(\Omega_{1},E)$ and $T^{E}(h)=f$ by part a).

d) For $n=1,2$ the continuous linear injection (see \eqref{eq:tensor_into_eps_product})
\[
\Theta_{n}\colon \mathcal{F}_{n}(\Omega_{n})\otimes_{\pi} E \to \mathcal{F}_{n}(\Omega_{n})\varepsilon E,\; 
\sum_{j=1}^{k} f_{j}\otimes e_{j} \longmapsto \bigl[y\mapsto \sum_{j=1}^{k} y(f_{j})e_{j}\bigr],
\]
from the \gls{projtenprod} $\gls{FotimespiE}$ with the projective topology 
extends to a continuous linear map 
$\widehat{\Theta}_{n}\colon \mathcal{F}_{n}(\Omega_{n})\widehat{\otimes}_{\pi} E 
\to \mathcal{F}_{n}(\Omega_{n})\varepsilon E$
on the completion $\gls{FotimespiEcompl}$ 
of $\mathcal{F}_{n}(\Omega_{n})\otimes_{\pi} E$. 
The map $\widehat{\Theta}_{n}$ is also a topological isomorphism since $\mathcal{F}_{n}(\Omega_{n})$ is nuclear 
for $n=1,2$ in case (i) resp.\ $E$ is nuclear in case (ii).
Furthermore, $T^{\K}\otimes_{\pi}\operatorname{id}_{E}\colon 
\mathcal{F}_{1}(\Omega_{1})\otimes_{\pi}E\to \mathcal{F}_{2}(\Omega_{2})\otimes_{\pi}E$ 
is defined by the relation $\Theta_{2}\circ (T^{\K} \otimes_{\pi}\operatorname{id}_{E})
=(T^{\K}\varepsilon \operatorname{id}_{E})\circ \Theta_{1}$. We denote by 
$T^{\K}\, \widehat{\otimes}_{\pi}\operatorname{id}_{E}$ the continuous linear extension of 
$T^{\K} \otimes_{\pi}\operatorname{id}_{E}$ to the completion 
$\mathcal{F}_{1}(\Omega_{1})\widehat{\otimes}_{\pi} E$. 
Moreover, $\mathcal{F}_{n}(\Omega_{n})$ for $n=1,2$ and $E$ are Fr\'echet spaces, 
$T^{\K}$ and $\id_{E}$ are linear, continuous and surjective,
so $T^{\K}\,\widehat{\otimes}_{\pi}\operatorname{id}_{E}$ is surjective by \cite[10.24 Satz, p.\ 255]{Kaballo}.
We observe that 
\[
T^{\K}\varepsilon \operatorname{id}_{E}
=\widehat{\Theta}_{2}\circ(T^{\K}\, \widehat{\otimes}_{\pi}\operatorname{id}_{E})
  \circ\widehat{\Theta}^{\,-1}_{1}
\]
and deduce that $T^{\K}\varepsilon \operatorname{id}_{E}$ is surjective. 
Now, we apply part c), which proves part d). 

e) Throughout this proof we use the notation $X'':=(X_{b}')_{b}'$ for a locally convex Hausdorff space $X$ 
and $T:=T^{\K}$.
The space $\mathcal{F}_{1}(\Omega_{1})$ is a Fr\'echet space 
and so its closed subspace $\ker T$ as well. 
Further, $\mathcal{F}_{n}(\Omega_{n})$ is a Montel space for $n=1,2$ and $\ker T$ nuclear, 
thus they are reflexive. 
The sequence
\begin{equation}\label{thm19.1}
0\to\ker T\overset{i}{\to}\mathcal{F}_{1}(\Omega_{1})
\overset{T}{\to}\mathcal{F}_{2}(\Omega_{2})\to 0,
\end{equation}
where $i$ means the inclusion, is a topologically exact sequence of Fr\'echet spaces because 
$T$ is surjective by assumption. Let us denote by $J_{0}\colon\ker T\to(\ker T)''$ and 
$J_{n}\colon\mathcal{F}_{n}(\Omega_{n})\to\mathcal{F}_{n}(\Omega_{n})''$ for $n=1,2$ the canonical embeddings 
which are topological isomorphisms since $\ker T$ and
$\mathcal{F}_{n}(\Omega_{n})$ are reflexive for $n=1,2$. 
Then the exactness of \eqref{thm19.1} implies that
\begin{equation}\label{thm19.2}
0\to(\ker T)''\overset{i_{0}}{\to}\mathcal{F}_{1}(\Omega_{1})''
\overset{T_{1}}{\to}\mathcal{F}_{2}(\Omega_{2})''\to 0,
\end{equation}
where $i_{0}:=J_{0}\circ i\circ J^{-1}_{0}$ and $T_{1}:=J_{2}\circ T\circ J^{-1}_{1}$, 
is an exact topological sequence. 
This exact sequence is topological because the (strong) bidual of a Fr\'echet space is again a Fr\'echet space 
by \cite[Corollary 25.10, p.\ 298]{meisevogt1997}.

(i) Let $E=F_{b}'$ where $F$ is a Fr\'echet space with $(DN)$. 
Then $\operatorname{Ext}^{1}(F,(\ker T)'')=0$ by \cite[5.1 Theorem, p.\ 186]{vogt1987} 
since $\ker T$ is nuclear and satisfies $(\Omega)$ and 
therefore $(\ker T)''$ as well. 
Combined with the exactness of \eqref{thm19.2} this implies that the sequence
\[
0\to L(F,(\ker T)'')\overset{i^{\ast}_{0}}{\to}L(F,\mathcal{F}_{1}(\Omega_{1})'')
\overset{T^{\ast}_{1}}{\to}L(F,\mathcal{F}_{2}(\Omega_{2})'')\to 0
\]
is exact by \cite[Proposition 2.1, p.\ 13--14]{palamodov1971} where $i^{\ast}_{0}(B):=i_{0}\circ B$ and
$T^{\ast}_{1}(D):=T_{1}\circ D$ for 
$B\in L(F,(\ker T)'')$ and $D\in L(F,\mathcal{F}_{1}(\Omega_{1})'')$. 
In particular, we obtain that
\begin{equation}\label{thm19.3}
T^{\ast}_{1}\colon L(F,\mathcal{F}_{1}(\Omega_{1})'')\to
L(F,\mathcal{F}_{2}(\Omega_{2})'')
\end{equation}
is surjective. 
Via $E=F_{b}'$ and \prettyref{prop:isom_op_spaces} ($X=\mathcal{F}_{n}(\Omega_{n})$ and $Y=F$) we have 
the isomorphisms into
\begin{gather*}
\psi_{n}:=S_{\mathcal{F}_{n}(\Omega_{n})}\circ{^{\mathrm{t}}(\cdot)}\colon L(F,\mathcal{F}_{n}(\Omega_{n})'')\to 
\mathcal{F}_{n}(\Omega_{n},E),\\
\psi_{n}(u)=\bigl(S_{\mathcal{F}_{n}(\Omega_{n})}\circ{^{\mathrm{t}}(\cdot)}\bigr)(u)
           =\bigl[x\mapsto{^{\mathrm{t}}u}(\delta_{x})\bigr],
\end{gather*}
for $n=1,2$ and the inverse
\[ 
 \psi^{-1}_{2}(f)
=(S\circ{^{\mathrm{t}}(\cdot)})^{-1}(f)
=({^{t}(\cdot)}\circ S_{\mathcal{F}_{2}(\Omega_{2})}^{-1})(f)
={^{t}(\mathcal{J}^{-1}\circ R_{f}^{t})},\quad f\in\mathcal{F}_{2}(\Omega_{2},E).
\]
Let $g\in\mathcal{F}_{2}(\Omega_{2},E)$. 
Then $\psi^{-1}_{2}(g)\in L(F,\mathcal{F}_{2}(\Omega_{2})'')$ and by the surjectivity of \eqref{thm19.3} 
there is $u\in L(F,\mathcal{F}_{1}(\Omega_{1})'')$ such that 
$T^{\ast}_{1}u=\psi^{-1}_{2}(g)$. 
So we get $\psi_{1}(u)\in\mathcal{F}_{1}(\Omega_{1},E)$. 
Next, we show that $T^{E}\psi_{1}(u)=g$ is valid. 
Let $y\in F$ and $x\in\Omega_{2}$. Then 
\[
T^{E}(\psi_{1}(u))(x)={^{\mathrm{t}}u}(\delta_{x}\circ T)
\]
by consistency and
\begin{align*}
 T^{E}(\psi_{1}(u))(x)(y)
&={^{\mathrm{t}}u}(\delta_{x}\circ T)(y)
 =u(y)(\delta_{x}\circ T)
 =\langle\delta_{x}\circ T,J^{-1}_{1}(u(y))\rangle\\
&=\langle\delta_{x},TJ^{-1}_{1}(u(y))\rangle
 =\langle[J_{2}\circ T\circ J^{-1}_{1}](u(y)),\delta_{x}\rangle
 =\langle(T_{1}\circ u)(y),\delta_{x}\rangle\\
&=\langle(T^{\ast}_{1}u)(y),\delta_{x}\rangle
 =\psi^{-1}_{2}(g)(y)(\delta_{x})
 ={^{t}(\mathcal{J}^{-1}\circ R_{g}^{t})}(y)(\delta_{x})\\
&=(\mathcal{J}^{-1}\circ R_{g}^{t})(\delta_{x})(y)
 =\mathcal{J}^{-1}(\mathcal{J}(g(x))(y)
 =g(x)(y).
\end{align*}
Thus $T^{E}(\psi_{1}(u))(x)=g(x)$ for every $x\in\Omega_{2}$, which proves the surjectivity.

(ii) Let $E$ be an ultrabornological PLS-space satisfying $(PA)$. 
Since the nuclear Fr\'echet space $\ker T$ is also a Schwartz space, 
its strong dual $(\ker T)_{b}'$ is a DFS-space. 
By \cite[Theorem 4.1, p.\ 577]{Dom1} we obtain $\operatorname{Ext}^{1}_{PLS}((\ker T)_{b}',E)=0$ 
as the bidual $(\ker T)''$ satisfies $(\Omega)$, $E$ is a PLS-space satisfying $(PA)$ 
and condition (c) in the theorem is fulfilled because $(\ker T)_{b}'$ is the strong dual of 
a nuclear Fr\'echet space.
Moreover, we have $\operatorname{Proj}^{1} E=0$ due to \cite[Corollary 3.3.10, p.\ 46]{wengenroth2003} 
because $E$ is an ultrabornological PLS-space. 
Then the exactness of the sequence \eqref{thm19.2}, \cite[Theorem 3.4, p.\ 567]{Dom1} 
and \cite[Lemma 3.3, p.\ 567]{Dom1} 
(in the lemma the same condition (c) as in \cite[Theorem 4.1, p.\ 577]{Dom1} is fulfilled 
and we choose $H=(\ker T)''$,
$F=\mathcal{F}_{1}(\Omega_{1})''$ and $G=\mathcal{F}_{2}(\Omega_{2})''$), imply that the sequence
\[
0\to L(E_{b}',(\ker T)'')
\overset{i^{\ast}_{0}}{\to}L(E_{b}',\mathcal{F}_{1}(\Omega_{1})'')
\overset{T^{\ast}_{1}}{\to}L(E_{b}',\mathcal{F}_{2}(\Omega_{2})'')\to 0
\]
is exact. The maps $i^{\ast}_{0}$ and $T^{\ast}_{1}$ are defined as in part (i). 
Especially, we get that
\begin{equation}\label{thm19.4}
T^{\ast}_{1}\colon L(E_{b}',\mathcal{F}_{1}(\Omega_{1})'')\to
L(E_{b}',\mathcal{F}_{2}(\Omega_{2})'')
\end{equation}
is surjective. 

By \cite[Remark 4.4, p.\ 1114]{D/L} we have $L_{b}(\mathcal{F}_{n}(\Omega_{n})_{b}',E'')
\cong L_{b}(E_{b}',\mathcal{F}_{n}(\Omega_{n})'')$ for $n=1,2$ via taking adjoints 
since $\mathcal{F}_{n}(\Omega_{n})$, being a Fr\'echet--Schwartz space, is a PLS-space and hence 
its strong dual an LFS-space, which is regular by 
\cite[Corollary 6.7, $10.\Leftrightarrow 11.$, p.\ 114]{wengenroth2003}, 
and $E$ is an ultrabornological PLS-space, in particular, reflexive by \cite[Theorem 3.2, p.\ 58]{domanski2004}. 
In addition, the map
\[
P\colon L_{b}(\mathcal{F}_{n}(\Omega_{n})_{b}',E'')\to 
L_{b}(\mathcal{F}_{n}(\Omega_{n})_{b}',E),
\]
defined by $P(u)(y):=\mathcal{J}^{-1}(u(y))$ for $u\in L(\mathcal{F}_{n}(\Omega_{n})_{b}',E'')$ 
and $y\in \mathcal{F}_{n}(\Omega_{n})'$, is an isomorphism because $E$ is reflexive.
Due to \prettyref{prop:isom_op_spaces} b) with $X=\mathcal{F}_{n}(\Omega_{n})$ we obtain 
the isomorphisms into
\begin{gather*}
\psi_{n}:=S\circ\mathcal{J}^{-1}\circ{^{\mathrm{t}}(\cdot)}\colon L_{b}(E_{b}',\mathcal{F}_{n}(\Omega_{n})'')
\to\mathcal{F}_{n}(\Omega_{n},E),\\
\psi_{n}(u)=[S_{\mathcal{F}_{n}(\Omega_{n})}\circ\mathcal{J}^{-1}\circ{^{\mathrm{t}}(\cdot)}](u)
           =\bigl[x\mapsto\mathcal{J}^{-1}({^{\mathrm{t}}u}(\delta_{x}))\bigr],
\end{gather*}
for $n=1,2$ and the inverse given by 
\begin{align*}
 \psi^{-1}_{2}(f)
&=(S_{\mathcal{F}_{2}(\Omega_{2})}\circ\mathcal{J}^{-1}\circ{^{\mathrm{t}}(\cdot)})^{-1}(f)
 =[{^{t}(\cdot)}\circ\mathcal{J}\circ S^{-1}_{\mathcal{F}_{2}(\Omega_{2})}](f)
 ={^{t}(\mathcal{J}\circ\mathcal{J}^{-1}\circ R_{f}^{t})}\\
&={^{t}(R_{f}^{t})}
\end{align*}
for $f\in\mathcal{F}_{2}(\Omega_{2},E)$.

Let $g\in\mathcal{F}_{2}(\Omega_{2},E)$. 
Then $\psi^{-1}_{2}(g)\in L(E_{b}',\mathcal{F}_{2}(\Omega_{2})'')$ 
and by the surjectivity of \eqref{thm19.4} there exists $u\in L(E_{b}',\mathcal{F}_{1}(\Omega_{1})'')$ 
such that $T^{\ast}_{1}u=\psi^{-1}_{2}(g)$. 
So we have $\psi_{1}(u)\in\mathcal{F}_{1}(\Omega_{1},E)$. 
The last step is to show that $T^{E}\psi_{1}(u)=g$. 
As in part (i) we gain for every $x\in\Omega_{2}$
\[	
T^{E}(\psi_{1}(u))(x)=\mathcal{J}^{-1}({^{\mathrm{t}}u}(\delta_{x}\circ T))
\]
by consistency and for every $y\in E'$
\begin{align*}
  {^{\mathrm{t}}u}(\delta_{x}\circ T)(y)
&=u(y)(\delta_{x}\circ T)
 =(T^{\ast}_{1}u)(y)(\delta_{x})
 =\psi^{-1}_{2}(g)(y)(\delta_{x})
 ={^{t}(R_{g}^{t})}(y)(\delta_{x})\\
&=\delta_{x}(y\circ g)
 =y(g(x))
 =\mathcal{J}(g(x))(y).
\end{align*}
Thus we have ${^{\mathrm{t}}u}(\delta_{x}\circ T)=\mathcal{J}(g(x))$ and therefore
$T^{E}(\psi_{1}(u))(x)=g(x)$ for all $x\in\Omega_{2}$.
\end{proof}

\prettyref{thm:eps_prod_surj_inj} d) and e) are generalisations of \cite[Corollary 4.3, p.\ 2689]{kruse2018_5}
and \cite[Theorem 5, p.\ 7--8]{kruse2019_1} where 
$T^{\C}$ is the Cauchy--Riemann operator $\overline{\partial}$ on certain weighted spaces 
$\mathcal{CV}^{\infty}(\Omega)$ of smooth functions.
Our next result is the well-known application of tensor product theory and splitting 
theory to linear partial differential operators we already mentioned in the introduction. 

\begin{cor}\label{cor:surjectivity_hypo_elliptic}
Let $E$ be a locally complete lcHs, $\Omega_{1}\subset\R^{d}$ open and $P(\partial)^{\K}$ 
be a linear partial differential operator with $\mathcal{C}^{\infty}$-smooth coefficients.
Then the following holds:
\begin{enumerate}
\item[a)] $P(\partial)^{E}=S_{\mathcal{C}^{\infty}(\Omega_{1})}\circ(P(\partial)^{\K}\varepsilon\id_{E})
\circ S^{-1}_{\mathcal{C}^{\infty}(\Omega_{1})}$.
\item[b)] If $\K=\C$, $P(D):=P(D)^{\C}:=P(-\iu\partial)^{\C}$ has constant coefficients and is
\begin{enumerate}
\item[(i)] elliptic, or
\item[(ii)] hypoelliptic and $\Omega_{1}$ convex,
\end{enumerate}
and 
\begin{enumerate}
\item[(iii)] $E$ is a Fr\'echet space, or
\item[(iv)] $E=F_{b}'$ where $F$ is a Fr\'echet space satisfying $(DN)$, $d\geq 2$, or
\item[(v)] $E$ is an ultrabornological PLS-space satisfying $(PA)$, $d\geq 2$,
\end{enumerate}
then $P(D)^{E}\colon\mathcal{C}^{\infty}(\Omega_{1},E)\to\mathcal{C}^{\infty}(\Omega_{1},E)$ is surjective. 
\end{enumerate}
\end{cor}
\begin{proof}
Part a) follows from \prettyref{thm:eps_prod_surj_inj} a), \prettyref{ex:diff_usual} a) 
and the consistency of $(P(\partial)^{E},P(\partial)^{\K})$
because \eqref{eq:cons_partial_diff} holds for $u\in \mathcal{CW}^{\infty}(\Omega_{1})\varepsilon E$ as well. 

Let us turn to part b). The inverse of $S_{\mathcal{C}^{\infty}(\Omega_{1})}$ 
is given by $R^{t}$ by \prettyref{ex:diff_usual} a). The map 
$P(D)=P(D)^{\C}\colon\mathcal{C}^{\infty}(\Omega_{1})\to\mathcal{C}^{\infty}(\Omega_{1})$ 
is surjective by \cite[Corollary 10.6.8, p.\ 43]{H2} and \cite[Theorem 10.6.2, p.\ 41]{H2} in case (ii) 
resp.\ by \cite[Corollary 10.8.2, p.\ 51]{H2} in case (i). 
The space $\mathcal{CW}^{\infty}(\Omega_{1})$, i.e.\ 
$\mathcal{C}^{\infty}(\Omega_{1})$ with its usual topology (see \prettyref{ex:weighted_smooth_functions} b)), 
is a nuclear Fr\'echet space and thus 
its closed subspace $\ker P(D)^{\C}$ as well. In case (i) $\ker P(D)^{\C}$ has 
$(\Omega)$  due to \cite[Proposition 2.5 (b), p.\ 173]{vogt1983} and in case (ii) 
due to \cite[4.5 Corollary (a), p.\ 202]{petzsche1980}. Hence the surjectivity of 
$P(D)^{E}$ follows from \prettyref{thm:eps_prod_surj_inj} d)+e).  
\end{proof}

Recently, it was shown in \cite[Theorem 4.2, p.\ 13]{debrouwere_kalmes2022} that in the case 
that $\Omega_{1}$ is convex, $\ker P(D)^{\C}$ has property $(\Omega)$ for any $P(D)$ with constant coefficients 
(so without the assumption of $P(D)$ being hypoelliptic). Hence we may replace the assumption of $P(D)$ 
being hypoelliptic in (ii) by the assumption that 
$P(D)^{\C}\colon\mathcal{C}^{\infty}(\Omega_{1})\to\mathcal{C}^{\infty}(\Omega_{1})$ 
is surjective. Even more recently, a necessary and sufficient condition for the surjectivity of $P(D)^{\C}$ and 
$\ker P(D)^{\C}$ having property $(\Omega)$ for $P(D)$ with constant coefficients and general open $\Omega_{1}\subset\R^{d}$ 
was derived in \cite[Theorem 1.1. (a), p.\ 3]{debrouwere_kalmes2023} using shifted fundamental solutions. 

Even though \prettyref{cor:surjectivity_hypo_elliptic} b) is known, 
it is often proved without using tensor products or splitting theory 
(see e.g.\ \cite[Theorem 10.10, p.\ 240]{Kaballo}) 
or it is phrased as the surjectivity of $P(D)^{\C}\widehat{\otimes}_{\pi}\id_{E}$ 
(see e.g.\ \cite[Eq.\ (52.4), p.\ 541]{Treves}) and the proof of the relation 
\[
 P(D)^{E}
=S_{\mathcal{C}^{\infty}(\Omega_{1})}\circ
 (\widehat{\Theta}_{1}\circ(P(D)^{\C}\, \widehat{\otimes}_{\pi}\operatorname{id}_{E})\circ\widehat{\Theta}^{\,-1}_{1})
 \circ S^{-1}_{\mathcal{C}^{\infty}(\Omega_{1})}
\]
for Fr\'echet spaces $E$ is omitted (see e.g.\ \cite[p.\ 545--546]{Treves}), 
or only the surjectivity of $T_{1}^{\ast}=P(D)_{1}^{\ast}$ in part e) of \prettyref{thm:eps_prod_surj_inj} 
is actually shown and it is only stated but not proved
that this implies the surjectivity of $P(D)^{E}$ (see e.g.\ the statement of surjectivity of $P(D)^{E}$ in 
\cite[p.\ 168]{vogt1983} for elliptic $P(D)$ and $E=F_{b}'$ for a Fr\'echet space $F$ with $(DN)$
and that it is `only' shown that $P(D)_{1}^{\ast}$ is surjective by 
\cite[Proposition 2.5 (b), p.\ 173]{vogt1983} and \cite[Theorem 2.4 (b), p.\ 173]{vogt1983} where the symbol
$P(D)^{\ast}$ is used instead of $P(D)_{1}^{\ast}$ in \cite[p.\ 172]{vogt1983} since the isomorphism 
$J_{1}=J_{2}$ is omitted). So, apart from being the probably most classical application of tensor products or
splitting theory, that is the reason why we still included \prettyref{cor:surjectivity_hypo_elliptic}.
 
Let us give another application of \prettyref{thm:eps_prod_surj_inj} d) and e), namely, 
a vector-valued \gls{BorelRitt}.

\begin{thm}\label{thm:Borel_Ritt}
Let $E$ be an lcHs and $(x_{n})_{n\in\N_{0}}$ a sequence in $E$. If
\begin{enumerate}
\item[(i)] $E$ is a Fr\'echet space, or 
\item[(ii)] $E=F_{b}'$ where $F$ is a Fr\'echet space satisfying $(DN)$, or
\item[(iii)] $E$ is an ultrabornological PLS-space satisfying $(PA)$, 
\end{enumerate}
then there is $f\in\mathcal{C}^{\infty}_{2\pi}(\R,E)$ such that $(\partial^{n})^{E}f(0)=x_{n}$ 
for all $n\in\N_{0}$.
\end{thm}
\begin{proof}
By the Borel--Ritt theorem \cite[Satz 9.12, p.\ 206]{Kaballo} the map 
\[
T^{\K}\colon\mathcal{C}^{\infty}_{2\pi}(\R)\to\K^{\N_{0}},\;
T^{\K}(f):=\bigl((\partial^{n})^{\K}f(0)\bigr)_{n\in\N_{0}},
\]
is surjective and obviously linear and continuous as well. 
Now, we define the map $T^{E}\colon\mathcal{C}^{\infty}_{2\pi}(\R,E)\to E^{\N_{0}}$ 
by replacing $\K$ by $E$ in the definition of $T^{\K}$. 
Due to \prettyref{ex:space_of_all_functions} $\K^{\N_{0}}$ and $E^{\N_{0}}$ are $\varepsilon$-compatible and 
the inverse of $S_{\K^{\N_{0}}}$ is given by $R^{t}$.
In addition, $\mathcal{C}^{\infty}_{2\pi}(\R)$ and $\mathcal{C}^{\infty}_{2\pi}(\R,E)$
are $\varepsilon$-compatible by \prettyref{ex:smooth_periodic_eps_compat} 
as in all three cases $E$ is complete. 
We observe that $(T^{E},T^{\K})$ is consistent by \prettyref{prop:diff_cons_barrelled} c).
The spaces $\K^{\N_{0}}$ and $\mathcal{C}^{\infty}_{2\pi}(\R)$ are nuclear Fr\'echet spaces. 
The first by \cite[Theorem 51.1, p.\ 526]{Treves}
and the second because it is a subspace of the nuclear space $\mathcal{C}^{\infty}(\R)$ 
by \cite[Examples 28.9 (1), p.\ 349--350]{meisevogt1997} and
\cite[Proposition 28.6, p.\ 347]{meisevogt1997}. 
Hence in case (i) our statement follows from \prettyref{thm:eps_prod_surj_inj} d). 
Moreover, $\ker T^{\K}$ is nuclear since $\mathcal{C}^{\infty}_{2\pi}(\R)$ is nuclear. By the proof 
of \cite[Lemma 31.3, p.\ 392--393]{meisevogt1997} $\ker T^{\K}$ is isomorphic to $s(\N_{0})$. The 
space $s(\N_{0})$ has $(\Omega)$ by \cite[Lemma 29.11 (3), p.\ 368]{meisevogt1997} and 
thus $\ker T^{\K}$ as well because $(\Omega)$ is a linear topological invariant 
by \cite[Lemma 29.11 (1), p.\ 368]{meisevogt1997}. Therefore our statement in case (ii) and (iii) 
follows from \prettyref{thm:eps_prod_surj_inj} e). 
\end{proof} 
 
We close this section with an application of \prettyref{thm:eps_prod_surj_inj} b) 
to the \gls{Fourier_transformation} on the Beurling--Bj\"orck spaces $\mathcal{S}_{\mu}(\R^{d},E)$ 
from \prettyref{ex:Bjoerck}.

\begin{thm}\label{thm:Bjoerck_Fourier}
Let $E$ be a locally complete lcHs over $\C$ and $\mu\colon\R^{d}\to[0,\infty)$ continuous such that 
$\mu(x)=\mu(-x)$ for all $x\in\R^{d}$ and condition $(\gamma)$ is fulfilled. 
\begin{enumerate}
\item[(i)] If $E$ has metric ccp, or
\item[(ii)] if $\mu\in\mathcal{C}^{1}(\R^{d})$ and there are $k\in\N_{0}$, $C>0$ such that 
$|\partial^{e_{n}}\mu(x)|\leq C\e^{k\mu(x)}$ for all $x\in\R^{d}$ and $1\leq n\leq d$,
\end{enumerate}
then 
$\mathfrak{F}^{E}\colon\mathcal{S}_{\mu}(\R^{d},E)\to\mathcal{S}_{\mu}(\R^{d},E)$ is an isomorphism 
with $\mathfrak{F}^{E}=S\circ(\mathfrak{F}^{\C}\varepsilon\id_{E})\circ S^{-1}$.
\end{thm}
\begin{proof}
Due to \prettyref{ex:Bjoerck} $\mathcal{S}_{\mu}(\R^{d})$ and $\mathcal{S}_{\mu}(\R^{d},E)$ are 
$\varepsilon$-compatible. 
The Fourier transformation $\mathfrak{F}^{\C}\colon\mathcal{S}_{\mu}(\R^{d})\to\mathcal{S}_{\mu}(\R^{d})$ 
is a well-defined isomorphism by the definition of $\mathcal{S}_{\mu}(\R^{d})$ and since
$(\mathfrak{F}^{\C}\circ\mathfrak{F}^{\C})(f)(x)=f(-x)$ for all $f\in\mathcal{S}_{\mu}(\R^{d})$ 
as well as $\mu(x)=\mu(-x)$ for all $x\in\R^{d}$. 
Due to \eqref{eq:bjoerck_fourier_eps} with $\beta=0$ 
we have that $(\mathfrak{F}^{E},\mathfrak{F}^{\C})$
is a consistent family for $(\mathcal{S}_{\mu},E)$ and thus it follows from 
\prettyref{thm:eps_prod_surj_inj} b) that $\mathfrak{F}^{E}$ is an isomorphism 
and $\mathfrak{F}^{E}=S\circ(\mathfrak{F}^{\C}\varepsilon\id_{E})\circ S^{-1}$, 
which completes the proof.
\end{proof}
\section{Extension of vector-valued functions}
\label{sect:extension}
We study the problem of extending vector-valued functions via the existence of weak extensions in this section. 
The precise description of this problem reads as follows.
Let $E$ be a locally convex Hausdorff space over the field $\K$ of real or complex numbers and 
$\F:=\mathcal{F}(\Omega,\K)$ a locally convex Hausdorff space of 
$\K$-valued functions on a set $\Omega$. 
Suppose that the point evaluations $\delta_{x}$ belong to the dual $\F'$ for every $x\in\Omega$ and that 
there is a locally convex Hausdorff space $\FE$ of $E$-valued functions on $\Omega$ such 
that the map 
\begin{equation}\label{eq:intro}
S\colon \F\varepsilon E \to \FE,\;u\longmapsto [x\mapsto u(\delta_{x})],
\end{equation}
is an isomorphism into, i.e.\ $\F$ and $\FE$ are $\varepsilon$-into-compatible. 
Thus $\F\varepsilon E$ is a linearisation of a subspace of $\FE$.
Linearisations that are based on the Dixmier--Ng theorem were used by 
Bonet, Doma\'nski and Lindstr\"{o}m in \cite[Lemma 10, p.\ 243]{BonDomLind2001} resp.\ 
Laitila and Tylli in \cite[Lemma 5.2, p.\ 14]{Laitila2006} to describe the space of weakly holomorphic resp.\ harmonic
functions on the unit disc $\Omega=\D\subset\C$ with values in a (complex) Banach space $E$ (see also \cite{kruse_2022}).

\begin{que}\label{que:weak_strong}
Let $\Lambda$ be a subset of $\Omega$ and $G$ a linear subspace of $E'$.
Let $f\colon\Lambda\to E$ be such that for every $e'\in G$,
the function $e'\circ f\colon\Lambda\to \K$ has an extension in $\F$.
When is there an extension $F\in\FE$ of $f$, i.e.\ $F_{\mid \Lambda}=f$\ ?
\end{que}

An affirmative answer for $\Lambda=\Omega$ and $G=E'$ is called a \emph{\gls{weak_strong_p}}. 
For weighted continuous functions on a completely regular Hausdorff space $\Omega$ with values in 
a semi-Montel or Schwartz space $E$ a weak-strong principle is given by Bierstedt in \cite[2.10 Lemma, p.\ 140]{B2}. 
Weak-strong principles for holomorphic functions on open subsets $\Omega\subset\C$ were shown 
by Dunford in \cite[Theorem 76, p.\ 354]{Dunford1938} for Banach spaces $E$ and by
Grothendieck in \cite[Th\'{e}or\`{e}me 1, p.\ 37--38]{Grothendieck1953} for quasi-complete $E$. 
For a wider class of function spaces weak-strong principles are due to Grothendieck, mainly, in the case 
that $\F$ is nuclear and $E$ complete (see \cite[Chap.\ II, \S3, n$^\circ$3, Th\'{e}or\`{e}me 13, p.\ 80]{Gro}), 
which covers the case that $\F$ is the space $\mathcal{C}^{\infty}(\Omega)$ of smooth functions on 
an open set $\Omega\subset\R^{d}$ (with its usual topology).
 
Gramsch \cite{Gramsch1977} analysed the weak-strong principles of Grothendieck and realised that they can be used 
to extend functions if $\Lambda$ is a set of uniqueness, i.e.\ from $f\in\F$ and $f(x)=0$ for all $x\in\Lambda$
follows that $f=0$, and $\F$ a semi-Montel space, $E$ complete and $G=E'$ (see \cite[0.1, p.\ 217]{Gramsch1977}).
An extension result for holomorphic functions where $G=E'$ and $E$ is sequentially complete 
was shown by Bogdanowicz in \cite[Corollary 3, p.\ 665]{Bogdanowicz1969}.

Grosse-Erdmann proved for holomorphic functions on $\Lambda=\Omega$ in \cite[5.2 Theorem, p.\ 35]{grosse-erdmann1992}
that it is sufficient to test locally bounded functions $f$ with values 
in a locally complete space $E$ with functionals from a weak$^{\star}$-dense subspace $G$ of $E'$.  
Arendt and Nikolski \cite{Arendt2000,Arendt2006} shortened his proof in the case that $E$ is a 
Fr\'{e}chet space (see \cite[Theorem 3.1, p.\ 787]{Arendt2000} and \cite[Remark 3.3, p.\ 787]{Arendt2000}).
Arendt gave an affirmative answer in \cite[Theorem 5.4, p.\ 74]{Arendt2016} for harmonic functions 
on an open subset $\Lambda=\Omega\subset\R^{d}$ where the 
range space $E$ is a Banach space and $G$ a weak$^{\star}$-dense subspace of $E'$. 

In \cite{Gramsch1977} Gramsch also derived extension results for a large class of Fr\'{e}chet--Montel spaces $\F$ 
in the case that $\Lambda$ is a special set of uniqueness, 
$E$ sequentially complete and $G$ strongly dense in $E'$ (see \cite[3.3 Satz, p.\ 228--229]{Gramsch1977}). 
He applied it to the space of holomorphic functions and Grosse-Erdmann \cite{grosse-erdmann2004} 
expanded this result to the case of $E$ being $B_{r}$-complete and $G$ only a weak$^{\star}$-dense subspace of $E'$ 
(see \cite[Theorem 2, p.\ 401]{grosse-erdmann2004} and \cite[Remark 2 (a), p.\ 406]{grosse-erdmann2004}). 
In a series of papers \cite{B/F/J,F/J,F/J/W,jorda2005,jorda2013} 
these results were generalised and improved by Bonet, Frerick, Jord\'{a} and Wengenroth 
who used \eqref{eq:intro} to obtain extensions for vector-valued functions via extensions of linear operators. 
In \cite{jorda2005,jorda2013} this was done by Jord\'{a} for holomorphic functions on a domain (i.e.\ open and connected) 
$\Omega\subset\C$ and weighted holomorphic functions on a domain $\Omega$ in a Banach space.
In \cite{B/F/J} this was done by Bonet, Frerick and Jord\'{a} for closed subsheaves $\F$ of the sheaf of 
smooth functions $\mathcal{C}^{\infty}(\Omega)$ on a domain $\Omega\subset\R^{d}$. 
Their results implied some consequences on the work of Bierstedt and Holtmanns \cite{Bierstedt2003} as well.
Further, in \cite{F/J} this was done by Frerick and Jord\'{a} for closed subsheaves $\F$ of smooth functions 
on a domain $\Omega\subset\R^{d}$ which are closed in the sheaf $\mathcal{C}(\Omega)$ of continuous functions
and in \cite{F/J/W} by the first two authors and Wengenroth in the case that $\F$ is the space of bounded functions 
in the kernel of a hypoelliptic linear  partial differential operator, 
in particular, the spaces of bounded holomorphic or harmonic functions. 

In this section we present a unified approach to the extension problem for a large class of function spaces. 
The spaces we treat are usually of the kind 
that $\F$ belongs to the class of semi-Montel spaces, Fr\'echet--Schwartz spaces or Banach spaces.
Even quite general weighted spaces $\F$ are treated, at least, if $E$ is a semi-Montel space. 
Our approach is based on three ideas. First, it is based on the representation of (a subspace of) $\FE$ as a space 
of continuous linear operators via the map $S$ from \eqref{eq:intro}. 
We note that almost all our examples of such spaces $\FE$ are actually of the form of a general weighted space $\FVE$ 
from \prettyref{def:weighted_space}. 
Second, it is based on the idea to consider a set of uniqueness $\Lambda$ not necessarily as a subset of $\Omega$ 
but rather as a set of functionals acting on $\F$. 
In the definition of a set of uniqueness given above one may identify $\Lambda$ with the set of functionals 
$\{\delta_{x}\;|\; x\in\Lambda\}$ and this shift of perspective allows us to consider certain sets of functionals 
of the form $T^{\K}_{m,x}$ as sets of uniqueness for $\F$ 
(see \prettyref{def:set_uniqueness}).
Third, the generalised concept of consistency and strength of a family of operators 
$(T^{E}_{m},T^{\K}_{m})_{m\in M}$ acting on $(\FE,\F)$ from \prettyref{def:cons_strong}
enables us to generalise \prettyref{que:weak_strong} and affirmatively answer this generalised question. 

These three ideas are used to extend the mentioned results and we always have to balance the sets $\Lambda$ 
from which we extend our functions and the subspaces $G\subset E'$ 
with which we test. The case of `thin' sets $\Lambda$ and `thick' subspaces $G$ is handled 
in \prettyref{sub:thin_0}, the converse case of `thick' sets $\Lambda$ and `thin' subspaces $G$ 
in \prettyref{sub:thick}.
\subsection{Extension from thin sets}\label{sub:thin_0}
Using the functionals $T^{\K}_{m,x}$, we extend the definition of a set of uniqueness 
and a space of restrictions given in \cite[Definition 4, 5, p.\ 230]{B/F/J}. This prepares the ground for a generalisation 
of \prettyref{que:weak_strong} using a strong, consistent family $(T^{E}_{m},T^{\K}_{m})_{m\in M}$.

\begin{defn}[{set of uniqueness}]\label{def:set_uniqueness}
Let $\Omega$ be a non-empty set, $\F\subset\K^{\Omega}$ an lcHs, $(\omega_m)_{m\in M}$ be a family of non-empty sets 
and $T^{\K}_{m}\colon \F\to \K^{\omega_{m}}$ be linear for all $m\in M$.
Then $U\subset\bigcup_{m\in M}(\{m\}\times\omega_{m})$ is called 
a \emph{\gls{set_uniqueness}} for $(T^{\K}_{m},\mathcal{F})_{m\in M}$ if
\begin{enumerate}
\item [(i)] $\forall\; (m,x)\in U:\; T^{\K}_{m,x}\in \F' $,
\item [(ii)] $\forall\; f\in\F: \;[\forall\,(m,x)\in U:\;T^{\K}_{m}(f)(x)=0]\;\;\Rightarrow\;\; f=0$.
\end{enumerate}
We omit the index $m$ in $\omega_{m}$ and $T^{\K}_{m}$ if $M$ is a singleton and consider $U$ as a subset of $\omega$.
\end{defn}

If $U$ is a set of uniqueness for $(T^{\K}_{m},\mathcal{F})_{m\in M}$, 
then $\operatorname{span}\{T^{\K}_{m,x}\;|\;(m,x)\in U\}$ is
dense in $\F_{\sigma}'$ (and $\F_{\kappa}'$) by the bipolar theorem. 

\begin{rem}\label{rem:Schauder_coeff_set_uni} 
Let $\Omega$ be a non-empty set and $\F\subset\K^{\Omega}$ an lcHs.
\begin{enumerate}
 \item[a)] A simple set of uniqueness for $(\id_{\K^{\Omega}},\mathcal{F})$ is given by $U:=\Omega$ 
 if $\delta_{x}\in\F'$ for all $x\in\Omega$.
 \item[b)] If $\F$ has a Schauder basis $(f_{n})_{n\in\N}$  with associated sequence of coefficient functionals 
 $T^{\K}:=(T^{\K}_{n})_{n\in\N}$, then $U:=\N$ is a set of uniqueness for $(T^{\K},\mathcal{F})$.
\end{enumerate}
\end{rem}

An example for b) is the space of holomorphic functions on an open disc $\D_{r}(z_{0})\subset\C$ 
with radius $0<r\leq\infty$ and center $z_0\in\C$. 
If we equip this space with the topology of compact convergence, then it has the shifted monomials $((\cdot-z_{0})^{n})_{n\in\N_{0}}$ 
as a Schauder basis 
with the point evaluations $(\delta_{z_0}\circ\partial_{\C}^{n})_{n\in\N_{0}}$ 
given by $(\delta_{z_0}\circ\partial_{\C}^{n})(f):=f^{(n)}(z_{0})$ 
as associated sequence of coefficient functionals. We will explore further sets of uniqueness for concrete function spaces in the upcoming examples 
and come back to b) in \prettyref{sect:sequence_space}. 

\begin{defn}[{\gls{rest_space}}]
Let $G\subset E'$ be a separating subspace and $U$ a set of uniqueness for $(T^{\K}_{m},\mathcal{F})_{m\in M}$.
Let $\gls{F_GUE}$ be the space of functions $f\colon U\to E$ such that for every $e'\in G$ there is 
$f_{e'}\in\F$ with $T^{\K}_{m}(f_{e'})(x)=(e'\circ f)(m,x)$ for all $(m,x)\in U$.
\end{defn}

\begin{rem}\label{rem:R_f}
Since $U$ is a set of uniqueness, the functions $f_{e'}$ are unique and the map 
$\gls{R_f_curly}\colon E'\to \F,$ $\mathscr{R}_{f}(e'):=f_{e'}$, is well-defined and linear. 
The map $\mathscr{R}_{f}$ resembles the map $R_{f}$ defined above \prettyref{lem:strong_is_weak}.
\end{rem}

\begin{rem}\label{rem:R_well-defined}
Let $\F$ and $\FE$ be $\varepsilon$-into-compatible.
Consider a set of uniqueness $U$ for $(T^{\K}_{m},\mathcal{F})_{m\in M}$, a separating subspace $G\subset E'$ and 
a strong, consistent family $(T^{E}_{m},T^{\K}_{m})_{m\in M}$ for $(\mathcal{F},E)$.
For $u\in \F\varepsilon E$ set $f:=S(u)$. Then $f\in\FE$ by the $\varepsilon$-into-compatibility 
and we set $\widetilde{f}\colon U\to E$, $\widetilde{f}(m,x):=T^{E}_{m}(f)(x)$. It follows that
\[
(e'\circ \widetilde{f})(m,x)= (e'\circ T^{E}_{m}(f))(x)=T^{\K}_{m}(e'\circ f)(x)
\]
for all $(m,x)\in U$ and $f_{e'}:=e'\circ f\in\F$ for all $e'\in E'$ by the strength of the family.
We conclude that $\widetilde{f}\in\mathcal{F}_{G}(U,E)$.
\end{rem} 

\begin{rem}
If $U$ is a set of uniqueness for $(T^{\K}_{m},\mathcal{F})_{m\in M}$, then 
the existence of operators $(T^{E}_{m})_{m\in M}$ such that
$(T^{E}_{m},T^{\K}_{m})_{m\in M}$ is a strong, consistent family for $(\mathcal{F},E)$ 
is often guaranteed by the Riesz--Markov--Kakutani representation theorems 
in \prettyref{sect:riesz_markov_kakutani}.
\end{rem}

Under the assumptions of \prettyref{rem:R_well-defined} the map
\[
R_{U,G}\colon S(\F\varepsilon E)\to \mathcal{F}_{G}(U,E),\;f\mapsto (T^{E}_{m}(f)(x))_{(m,x)\in U},
\]
is well-defined. The map $R_{U,G}$ is also linear since $T^{E}_{m}$ is linear for all $m\in M$. 
Further, the strength of the defining family guarantees that $R_{U,G}$ is injective.

\begin{prop}\label{prop:injectivity}
Let $\F$ and $\FE$ be $\varepsilon$-into-compatible, 
$G\subset E'$ a separating subspace and $U$ a set of uniqueness for $(T^{\K}_{m},\mathcal{F})_{m\in M}$.
If $(T^{E}_{m},T^{\K}_{m})_{m\in M}$ is a strong family for $(\mathcal{F},E)$, 
then the map 
\[
T^{E}\colon \FE\to E^{U}, \;f\mapsto (T^{E}_{m}(f)(x))_{(m,x)\in U},
\]
is injective, in particular, $R_{U,G}$ is injective.
\end{prop}
\begin{proof}
Let $f\in \FE$ with $T^{E}(f)=0$. Then
\[ 
0=(e'\circ T^{E}(f))(m,x)=(e'\circ T^{E}_{m}(f))(x)=T^{\K}_{m}(e'\circ f)(x),\quad (m,x)\in U,
\]
and $e'\circ f \in\F$ for all $e'\in E'$ by the strength of the family. 
Since $U$ is a set of uniqueness, we get that $e'\circ f=0$ for all $e'\in E'$, which implies $f=0$.
\end{proof}

\begin{que}\label{que:surj_restr_set_unique}
Let $\F$ and $\FE$ be $\varepsilon$-into-compatible,
$G\subset E'$ a separating subspace, 
$(T^{E}_{m},T^{\K}_{m})_{m\in M}$ a strong family for $(\mathcal{F},E)$
and $U$ a set of uniqueness for $(T^{\K}_{m},\mathcal{F})_{m\in M}$.
When is the injective restriction map 
\[
R_{U,G}\colon S(\F\varepsilon E)\to \mathcal{F}_{G}(U,E),\;f\mapsto (T^{E}_{m}(f)(x))_{(m,x)\in U},
\]
surjective?
\end{que}

The \prettyref{que:weak_strong} is a special case of this question if there is a 
set of uniqueness $U$ for $(T^{\K}_{m},\mathcal{F})_{m\in M}$ with 
$\{T^{\K}_{m,x}\;|\; (m,x)\in U\}=\{\delta_{x}\;|\;x\in\Lambda\}$, $\Lambda\subset\Omega$. 
We observe that a positive answer to the surjectivity of $R_{\Omega,G}$ results in the following weak-strong principle.

\begin{prop}\label{prop:weak_strong_principle}
Let $\F$ and $\FE$ be $\varepsilon$-into-compatible, $G\subset E'$ a separating subspace such that 
$e'\circ f\in \F$ for all $e'\in G$ and $f\in\FE$. 
If 
\[
R_{\Omega,G}\colon S(\F\varepsilon E)\to\mathcal{F}_{G}(\Omega,E),\;f\mapsto f,
\]
with the set of uniqueness $\Omega$ for $(\id_{\K^{\Omega}},\mathcal{F})$ is surjective, then 
\[
\F\varepsilon E\cong \FE \quad\text{via}\;S
 \quad\text{and}\quad
\FE=\{f\colon\Omega\to E\;|\;\forall\;e'\in G:\;e'\circ f\in\F \}.
\]
\end{prop}
\begin{proof}
From the $\varepsilon$-into-compatibility and the surjectivity of $R_{\Omega,G}$ we obtain
\[
 \{f\colon\Omega\to E\;|\;\forall\;e'\in G:\;e'\circ f\in\F\}=\mathcal{F}_{G}(\Omega,E)=S(\F\varepsilon E)\subset \FE. 
\]
Further, the assumption that $e'\circ f\in \F$ for all $e'\in G$ and $f\in\FE$, 
implies that $\FE$ is a subspace of the space on the left-hand side, which proves our statement, 
in particular, the surjectivity of $S$.
\end{proof}

To answer \prettyref{que:surj_restr_set_unique} for general sets of uniqueness we have to restrict to a certain class
of separating subspaces of $E'$.

\begin{defn}[{determine  boundedness \cite[p.\ 230]{B/F/J}}]
A linear subspace $G\subset E'$ \emph{\gls{det_bound}} if every $\gls{sigmaEG}$-bounded set $B\subset E$ is 
already bounded in $E$.
\end{defn}

In \cite[p.\ 139]{fernandez1989} such a space $G$ is called uniform boundedness deciding by Fern\'andez et al.\ and 
in \cite[p.\ 63]{nygaard2006} $w^{\ast}$-thick by Nygaard if $E$ is a Banach space.

\begin{rem} 
\begin{enumerate}
\item[a)] Let $E$ be an lcHs. Then $G:=E'$ determines boundedness 
by \cite[Mackey's theorem 23.15, p.\ 268]{meisevogt1997}.
\item[b)] Let $X$ be a barrelled lcHs, $Y$ an lcHs and $E:=L_{b}(X,Y)$. For $x\in X$ and $y'\in Y'$ we set 
$\delta_{x,y'}\colon L(X,Y)\to \K$, $T\mapsto y'(T(x))$, and $G:=\{\delta_{x,y'}\;|\;x\in X,\,y'\in Y'\}\subset E'$. 
Then the span of $G$ determines boundedness (in $E$) by Mackey's theorem and the uniform boundedness principle. 
For Banach spaces $X$,$Y$ this is already observed in \cite[Remark 11, p.\ 233]{B/F/J} and, if in addition $Y=\K$, 
in \cite[Remark 1.4 b), p.\ 781]{Arendt2000}.
\item[c)] Further examples and a characterisation of subspaces $G\subset E'$ that determine boundedness can be found 
in \cite[Remark 1.4, p.\ 781--782]{Arendt2000}, \cite[Theorem 1.5, p.\ 63--64]{nygaard2006}
and \cite[Theorem 2.3, 2.4, p.\ 67--68]{nygaard2006} in the case that $E$ is a Banach space. 
\end{enumerate}
\end{rem}

\addtocontents{toc}{\SkipTocEntry}
\subsection*{\texorpdfstring{$\F$}{F(Omega)} a semi-Montel space and \texorpdfstring{$E$}{E} (sequentially) complete}

Our next results are in need of spaces $\F$ such that closed graph theorems hold with Banach spaces as domain spaces and 
$\F$ as the range space. Let us formally define this class of spaces. 

\begin{defn}[{BC-space \cite[p.\ 395]{powell1955}}]
We call an lcHs $F$ a \emph{\gls{BC_space}} if for every Banach space $X$ and every linear map $f\colon X\to F$ 
with closed graph in $X\times F$, one has that $f$ is continuous. 
\end{defn}

A characterisation of BC-spaces is given by Powell in \cite[6.1 Corollary, p.\ 400--401]{powell1955}. 
Since every Banach space is ultrabornological and barrelled, 
the \cite[Closed graph theorem 24.31, p.\ 289]{meisevogt1997} of de Wilde 
and the Pt\'{a}k--K\={o}mura--Adasch--Valdivia closed graph theorem \cite[\S34, 9.(7), p.\ 46]{Koethe} 
imply that webbed spaces and $B_{r}$-complete spaces are BC-spaces. 
We recall that an lcHs $F$ is said to be \emph{\gls{Brcomplete}} 
if every $\sigma(F',F)$-dense 
$\sigma^{f}(F',F)$-closed linear subspace of $F'$ equals $F'$ 
where $\sigma^{f}(F',F)$ is the finest topology coinciding with $\sigma(F',F)$ on all equicontinuous sets in $F'$ 
(see \cite[\S34, p.\ 26]{Koethe}). An lcHs $F$ is called \emph{\gls{Bcomplete}} 
if every $\sigma^{f}(F',F)$-closed linear subspace of $F'$ is weakly closed.
In particular, $B$-complete spaces are $B_{r}$-complete and every $B_{r}$-complete space is complete 
by \cite[\S34, 2.(1), p.\ 26]{Koethe}.
These definitions are equivalent to the original definitions of $B_{r}$- and $B$-completeness 
by Pt\'{a}k \cite[Definition 2, 5, p.\ 50, 55]{ptak1958} due to \cite[\S34, 2.(2), p.\ 26--27]{Koethe} 
and we note that they are also called \emph{infra-Pt\'{a}k spaces} and \emph{Pt\'{a}k spaces}, respectively. In particular, Fr\'{e}chet spaces are $B$-complete by \cite[9.5.2 Krein--\u{S}mulian Theorem, p.\ 184]{Jarchow} but we will encounter non-Fr\'echet $B$-complete spaces as well.

The following proposition is a modification of \cite[Satz 10.6, p.\ 237]{Kaballo} and 
uses the map $\mathscr{R}_{f}\colon e'\mapsto f_{e'}$ from \prettyref{rem:R_f}. 

\begin{prop}\label{prop:ext_F_semi_M}
Let $U$ be a set of uniqueness for $(T^{\K}_{m},\mathcal{F})_{m\in M}$ and $\F$ a BC-space.
Then $\mathscr{R}_{f}(B_{\alpha}^{\circ})$ is bounded in $\F$ for every $f\in\mathcal{F}_{E'}(U,E)$ 
and $\alpha\in\mathfrak{A}$ where $B_{\alpha}:=\{x\in E\;|\; p_{\alpha}(x)<1\}$. 
In addition, if $\F$ is a semi-Montel space, then $\mathscr{R}_{f}(B_{\alpha}^{\circ})$ is relatively compact in $\F$.
\end{prop}
\begin{proof}
Let $f\in\mathcal{F}_{E'}(U,E)$ and $\alpha\in\mathfrak{A}$. The polar $B_{\alpha}^{\circ}$ 
is compact in $E_{\sigma}'$ and thus $E_{B_{\alpha}^{\circ}}'$ is a Banach space 
by \cite[Corollary 23.14, p.\ 268]{meisevogt1997}.
We claim that the restriction of $\mathscr{R}_{f}$ to $E_{B_{\alpha}^{\circ}}'$ has closed graph. Indeed, let 
$(e_{\tau}')$ be a net in $E_{B_{\alpha}^{\circ}}'$ converging to $e'$ in $E_{B_{\alpha}^{\circ}}'$ and  
$\mathscr{R}_{f}(e_{\tau}')$ converging to $g$ in $\F$. For $(m,x)\in U$ we note that
\begin{align*}
T^{\K}_{m,x}\bigl(\mathscr{R}_{f}(e_{\tau}')\bigr)&=T^{\K}_{m}(f_{e_{\tau}'})(x)
=(e_{\tau}'\circ f)(m,x)
\to (e'\circ f)(m,x)=T^{\K}_{m}(f_{e'})(x)\\
&=T^{\K}_{m}\bigl(\mathscr{R}_{f}(e')\bigr)(x).
\end{align*}
The left-hand side converges to $T^{\K}_{m,x}(g)$ since $T^{\K}_{m,x}\in\F'$ for all 
$(m,x)\in U$. Hence we have $T^{\K}_{m}(g)(x)=T^{\K}_{m}\bigl(\mathscr{R}_{f}(e')\bigr)(x)$ 
for all $(m,x)\in U$. From $U$ being a set of uniqueness follows that $g=\mathscr{R}_{f}(e')$. 
Thus the restriction of $\mathscr{R}_{f}$ to $E_{B_{\alpha}^{\circ}}'$ has closed graph and is continuous 
since $\F$ is a BC-space. This yields that $\mathscr{R}_{f}(B_{\alpha}^{\circ})$ is bounded 
as $B_{\alpha}^{\circ}$ is bounded in $E_{B_{\alpha}^{\circ}}'$.
If $\F$ is also a semi-Montel space, then $\mathscr{R}_{f}(B_{\alpha}^{\circ})$ is even relatively compact.
\end{proof}

Now, we are ready to prove our first extension theorem. Its proof of surjectivity of $R_{U,E'}$ 
is just an adaptation of the proof of 
surjectivity of $S$ given in \prettyref{thm:full_linearisation}. 
Let $U$ be a set of uniqueness for $(T^{\K}_{m},\mathcal{F})_{m\in M}$. 
For $f\in\mathcal{F}_{E'}(U,E)$ we consider the dual map 
\[
\mathscr{R}_{f}^{t}\colon \F' \to E'^{\star},\;\mathscr{R}_{f}^{t}(y)(e'):=y(f_{e'}),
\]
where $E'^{\star}$ is the algebraic dual of $E'$. Further, we recall the notation $\mathcal{J}\colon E\to E'^{\star}$ 
for the canonical injection.

\begin{thm}\label{thm:ext_F_semi_M}
Let $\F$ and $\FE$ be $\varepsilon$-into-compatible, 
$(T^{E}_{m},T^{\K}_{m})_{m\in M}$ a strong, consistent family for $(\mathcal{F},E)$, 
$\F$ a semi-Montel BC-space and $U$ a set of uniqueness for $(T^{\K}_{m},\mathcal{F})_{m\in M}$. 
If
\begin{enumerate}
\item [(i)] $E$ is complete, or
\item [(ii)] $E$ is sequentially complete and for every $f\in\mathcal{F}_{E'}(U,E)$ and $f'\in\F'$ there is 
a sequence $(f_{n}')_{n\in\N}$ in $\F'$ converging to $f'$ in $\F_{\kappa}'$ such that 
$\mathscr{R}_{f}^{t}(f_{n}')\in\mathcal{J}(E)$ for every $n\in\N$, 
\end{enumerate}
then the restriction map $R_{U,E'}\colon S(\F\varepsilon E)\to \mathcal{F}_{E'}(U,E)$ is surjective.
\end{thm}
\begin{proof} 
Let $f\in\mathcal{F}_{E'}(U,E)$. As in \prettyref{thm:full_linearisation} 
we equip $\mathcal{J}(E)$ with the system of seminorms given by 
\begin{equation}\label{eq1:ext_F_semi_M}
p_{B^{\circ}_{\alpha}}(\mathcal{J}(x)):=\sup_{e'\in B^{\circ}_{\alpha}}|\mathcal{J}(x)(e')|=p_{\alpha}(x),\quad x\in E,
\end{equation}
for all $\alpha\in \mathfrak{A}$ where $B_{\alpha}:=\{x\in E\;|\; p_{\alpha}(x)<1\}$.
We claim $\mathscr{R}_{f}^{t}\in L(\F_{\kappa}',\mathcal{J}(E))$. 
Indeed, we have for $y\in \F'$ 
\begin{equation}\label{eq2:ext_F_semi_M}
p_{B^{\circ}_{\alpha}}\bigl(\mathscr{R}_{f}^{t}(y)\bigr)
=\sup_{e'\in B^{\circ}_{\alpha}}|y(f_{e'})|
=\sup_{x\in \mathscr{R}_{f}(B^{\circ}_{\alpha})}|y(x)|
\leq\sup_{x\in K_{\alpha}}|y(x)|
\end{equation}
where $K_{\alpha}:=\overline{\mathscr{R}_{f}(B^{\circ}_{\alpha})}$. 
Due to \prettyref{prop:ext_F_semi_M} the set $\mathscr{R}_{f}(B^{\circ}_{\alpha})$ is absolutely convex 
and relatively compact, implying that $K_{\alpha}$ is absolutely convex and compact in 
$\F$ by \cite[6.2.1 Proposition, p.\ 103]{Jarchow}. 
Further, we have for all $e'\in E'$ and $(m,x)\in U$
\begin{equation}\label{eq3:ext_F_semi_M}
\mathscr{R}_{f}^{t}(T^{\K}_{m,x})(e')=T^{\K}_{m,x}(f_{e'})=(e'\circ f)(m,x)=\mathcal{J}\bigl(f(m,x)\bigr)(e')
\end{equation}
and thus $\mathscr{R}_{f}^{t}(T^{\K}_{m,x})\in\mathcal{J}(E)$. 

First, let condition (i) be satisfied, i.e.\ let $E$ be complete, and $f'\in \F'$. 
The span of $\{T^{\K}_{m,x}\;|\; (m,x)\in U\}$ is dense in
$\F_{\kappa}'$ since $U$ is a set of uniqueness for $\F$. Thus there is a net $(f_{\tau}')$ 
converging to $f'$ in $\FV_{\kappa}'$ with $\mathscr{R}_{f}^{t}(f_{\tau}')\in\mathcal{J}(E)$ 
and 
\begin{equation}\label{eq4:ext_F_semi_M}
p_{B^{\circ}_{\alpha}}\bigl(\mathscr{R}_{f}^{t}(f_{\tau}')-\mathscr{R}_{f}^{t}(f')\bigr)
 \underset{\eqref{eq2:ext_F_semi_M}}{\leq} \sup_{x\in K_{\alpha}}|(f_{\tau}'-f')(x)|\to 0
\end{equation}
for all $\alpha\in \mathfrak{A}$. We gain that $(\mathscr{R}_{f}^{t}(f_{\tau}'))$ is a Cauchy net 
in the complete space $\mathcal{J}(E)$.
Hence it has a limit $g\in\mathcal{J}(E)$ which coincides with $\mathscr{R}_{f}^{t}(f')$ since
\begin{align*}
\qquad p_{B^{\circ}_{\alpha}}\bigl(g-\mathscr{R}_{f}^{t}(f')\bigr)
&\leq p_{B^{\circ}_{\alpha}}\bigl(g-\mathscr{R}_{f}^{t}(f_{\tau}')\bigr)
 +p_{B^{\circ}_{\alpha}}\bigl(\mathscr{R}_{f}^{t}(f_{\tau}')-\mathscr{R}_{f}^{t}(f')\bigr)\\
&\;\;\mathclap{\underset{\eqref{eq4:ext_F_semi_M}}{\leq}}\;\;\; 
 p_{B^{\circ}_{\alpha}}\bigl(g-\mathscr{R}_{f}^{t}(f_{\tau}')\bigr)
 + \sup_{x\in K_{\alpha}}\bigl|(f_{\tau}'-f')(x)\bigr|\to 0
\end{align*}
for all $\alpha\in \mathfrak{A}$. We conclude that $\mathscr{R}_{f}^{t}(f')\in\mathcal{J}(E)$ for every $f'\in \F'$. 

Second, let condition (ii) be satisfied and $f'\in \F'$. Then there is 
a sequence $(f_{n}')$ in $\F'$ converging to $f'$ in $\F_{\kappa}'$ such that 
$\mathscr{R}_{f}^{t}(f_{n}')\in\mathcal{J}(E)$ for every $n\in\N$. From \eqref{eq2:ext_F_semi_M} we derive 
that $(\mathscr{R}_{f}^{t}(f_{n}'))$ is a Cauchy sequence in the sequentially complete 
space $\mathcal{J}(E)$ converging to $\mathscr{R}_{f}^{t}(f')\in\mathcal{J}(E)$.

Therefore we obtain in both cases that $\mathscr{R}_{f}^{t}\in L(\F_{\kappa}',\mathcal{J}(E))$.
So we get for all $\alpha\in \mathfrak{A}$ and $y\in \F'$ 
\[
p_{\alpha}\bigl((\mathcal{J}^{-1}\circ \mathscr{R}_{f}^{t})(y)\bigr)
\underset{\eqref{eq1:ext_F_semi_M}}{=}
  p_{B^{\circ}_{\alpha}}\bigl(\mathcal{J}((\mathcal{J}^{-1}\circ \mathscr{R}_{f}^{t})(y))\bigr)
= p_{B^{\circ}_{\alpha}}\bigl(\mathscr{R}_{f}^{t}(y)\bigr)
\underset{\eqref{eq2:ext_F_semi_M}}{\leq}\sup_{x\in K_{\alpha}}|y(x)|.
\]
This implies $\mathcal{J}^{-1}\circ \mathscr{R}_{f}^{t}\in L(\F_{\kappa}', E)=\F\varepsilon E$ (as linear spaces).
We set $F:=S(\mathcal{J}^{-1}\circ \mathscr{R}_{f}^{t})$ and obtain from consistency that
\[
 T^{E}_{m}(F)(x)
=T^{E}_{m}S(\mathcal{J}^{-1}\circ \mathscr{R}_{f}^{t})(x)
=\mathcal{J}^{-1}\bigl(\mathscr{R}_{f}^{t}(T^{\K}_{m,x})\bigr)
\underset{\eqref{eq3:ext_F_semi_M}}{=}\mathcal{J}^{-1}\bigl(\mathcal{J}(f(m,x))\bigr)
=f(m,x)
\]
for every $(m,x)\in U$, which means $R_{U,E'}(F)=f$.
\end{proof}

If $E$ is complete and $U$ a set of uniqueness for $(T^{\K}_{m},\mathcal{F})_{m\in M}$ with 
$\{T^{\K}_{m,x}\;|\; (m,x)\in U\}=\{\delta_{x}\;|\;x\in\Lambda\}$, $\Lambda\subset\Omega$, then 
we get \cite[0.1, p.\ 217]{Gramsch1977} as a special case. 
Condition (i) and (ii) are adaptations of \prettyref{cond:surjectivity_linearisation} a) and c) 
from $\mathcal{FV}(\Omega,E)$ and $R_{f}$ to $\mathcal{F}_{E'}(U,E)$ and $\mathscr{R}_{f}$. 
We also treat an adaptation of \prettyref{cond:surjectivity_linearisation} e) in \prettyref{thm:fix_topo_E_semi_M}. 
\prettyref{cond:surjectivity_linearisation} b) and d) may be adapted as well 
but we restrict to the ones we actually apply. 
First, we apply \prettyref{thm:ext_F_semi_M} to the space of bounded zero-solutions of a
hypoelliptic linear  partial differential operator equipped with the strict topology $\beta$ from
\prettyref{prop:strict_top_isomorphism}.

\begin{prop}\label{prop:strict_topo_hypo_B_complete_semi_Montel}
Let $\Omega\subset\R^{d}$ be open and $P(\partial)^{\K}$ a hypoelliptic linear partial differential operator. 
Then $(\mathcal{C}^{\infty}_{P(\partial),b}(\Omega),\beta)$ is a B-complete semi-Montel space.
\end{prop}
\begin{proof}
Due to the proof of \prettyref{prop:strict_top_isomorphism} we know that $\beta$ coincides with the mixed topology 
$\gamma(\tau_{c},\|\cdot\|_{\infty})$. 
It is easy to check that the closed $\|\cdot\|_{\infty}$-unit ball $B_{\|\cdot\|_{\infty}}$ 
is $\tau_{c}$-compact in $\mathcal{C}^{\infty}_{P(\partial),b}(\Omega)$.
Thus \cite[Section I.1, 1.13 Proposition, p.\ 11]{cooper1978} yields 
that $(\mathcal{C}^{\infty}_{P(\partial),b}(\Omega),\beta)$ is a semi-Montel space. 
From \cite[2.9 Theorem, p.\ 185]{ruess1977} it follows that the space is $B$-complete.
\end{proof}

\begin{cor}
Let $\Omega\subset\R^{d}$ be open, $E$ a complete lcHs, 
$P(\partial)^{\K}$ a hypoelliptic linear partial differential operator,
$(T^{E}_{m},T^{\K}_{m})_{m\in M}$ a strong, consistent family 
for $((\mathcal{C}^{\infty}_{P(\partial),b}(\Omega),\beta),E)$
and $U$ a set of uniqueness for $(T^{\K}_{m},(\mathcal{C}^{\infty}_{P(\partial),b}(\Omega),\beta))_{m\in M}$. 
If $f\colon U\to E$ is a function such that there is $f_{e'}\in\mathcal{C}^{\infty}_{P(\partial),b}(\Omega)$ 
for each $e'\in E'$ with $T^{\K}_{m}(f_{e'})(x)=(e'\circ f)(m,x)$ for all $(m,x)\in U$, 
then there is a unique $F\in\mathcal{C}^{\infty}_{P(\partial),b}(\Omega,E)$ 
with $T^{E}_{m}(F)(x)=f(m,x)$ for all $(m,x)\in U$.
\end{cor}
\begin{proof}
The space $(\mathcal{C}^{\infty}_{P(\partial),b}(\Omega),\beta)$ is a semi-Montel BC-space by 
\prettyref{prop:strict_topo_hypo_B_complete_semi_Montel}.
Moreover, $(\mathcal{C}^{\infty}_{P(\partial),b}(\Omega),\beta)$ and 
$(\mathcal{C}^{\infty}_{P(\partial),b}(\Omega,E),\beta)$ 
are $\varepsilon$-compatible by \prettyref{prop:strict_top_isomorphism}, 
yielding our statement by \prettyref{thm:ext_F_semi_M} (i) and \prettyref{prop:injectivity}. 
\end{proof}

Especially, for any $m\in\N_{0}$ the family 
$((\partial^{\beta})^{E},(\partial^{\beta})^{\K})_{\beta\in\N_{0}^{d},|\beta|\leq m}$ 
is strong and consistent for $((\mathcal{C}^{\infty}_{P(\partial),b}(\Omega),\beta),E)$ 
by the proof of \prettyref{prop:strict_top_isomorphism}. 
It is always possible to construct a strong, consistent family $(T^{E}_{m},T^{\K}_{m})_{m\in M}$ 
for $((\mathcal{C}^{\infty}_{P(\partial),b}(\Omega),\beta),E)$
from a given set of uniqueness $(T^{\K}_{m},(\mathcal{C}^{\infty}_{P(\partial),b}(\Omega),\beta))_{m\in M}$ 
due to \prettyref{rem:representation_subspaces} b) and c).

Similarly, we may apply \prettyref{thm:ext_F_semi_M} to the space 
$\mathcal{E}^{\{M_{p}\}}(\Omega,E)$ of ultradifferentiable functions of class $\{M_{p}\}$ of Roumieu-type 
from \prettyref{ex:weighted_smooth_functions} f).
$\mathcal{E}^{\{M_{p}\}}(\Omega)$ is a projective limit of a countable sequence of DFS-spaces by 
\cite[Theorem 2.6, p.\ 44]{Kom7} and thus webbed because being webbed is stable under the formation 
of projective and inductive limits of countable sequences by 
\cite[5.3.3 Corollary, p.\ 92]{Jarchow}. 
Further, if the sequence $(M_{p})_{p\in\N_{0}}$ satisfies Komatsu's conditions (M.1) and (M.3)', 
then $\mathcal{E}^{\{M_{p}\}}(\Omega)$ is a Montel space by \cite[Theorem 5.12, p.\ 65--66]{Kom7}. 
The spaces $\mathcal{E}^{\{M_{p}\}}(\Omega)$ and $\mathcal{E}^{\{M_{p}\}}(\Omega,E)$ are $\varepsilon$-compatible 
if (M.1) and (M.3)' hold and $E$ is complete by \prettyref{ex:ultradifferentiable} b). 
Hence \prettyref{thm:ext_F_semi_M} (i) is applicable.

\begin{rem}
We remark that \prettyref{rem:R_well-defined} and \prettyref{thm:ext_F_semi_M} still hold 
if the map $S\colon\F\varepsilon E\to\FE$ is only a linear isomorphism into, 
i.e.\ an isomorphism into of linear spaces, since the topological nature of $\varepsilon$-into-compatibility 
is not used in the proof. In particular, this means that it can be applied to the space $\gls{MOE}$ 
of meromorphic functions on an open, connected set $\Omega\subset\C$ with values in an lcHs $E$ over $\C$ 
(see \cite[p.\ 356]{Bonet2002}). 
The space $\mathcal{M}(\Omega)$ is a Montel LF-space, thus webbed by 
\cite[5.3.3 Corollary (b), p.\ 92]{Jarchow}, due to the proof of 
\cite[Theorem 3 (a), p.\ 294--295]{grosse-erdmann1995} if it is equipped with the locally convex topology 
$\tau_{ML}$ given in \cite[p.\ 292]{grosse-erdmann1995}. 
By \cite[Proposition 6, p.\ 357]{Bonet2002} the map $S\colon\mathcal{M}(\Omega)\varepsilon E\to\mathcal{M}(\Omega,E)$ 
is a linear isomorphism if $E$ is locally complete and does not contain the space $\C^{\N}$. 
Therefore we can apply \prettyref{thm:ext_F_semi_M} if $E$ is complete and does not contain $\C^{\N}$. 
This augments \cite[Theorem 12, p.\ 12]{jorda2005} where $E$ is assumed to be locally complete with 
suprabarrelled strong dual and $(T^{E},T^{\C})=(\id_{E^{\Omega}},\id_{\C^{\Omega}})$.
\end{rem}

\addtocontents{toc}{\SkipTocEntry}
\subsection*{\texorpdfstring{$\F$}{F(Omega)} a Fr\'{e}chet--Schwartz space and \texorpdfstring{$E$}{E} locally complete}

We recall the following abstract extension result. 

\begin{prop}[{\cite[Proposition 7, p.\ 231]{B/F/J}}]\label{prop:ext_FS_set_uni}
Let $E$ be a locally complete lcHs, $Y$ a Fr\'{e}chet--Schwartz space, $X\subset Y_{b}'(=Y_{\kappa}')$ dense 
and $\mathsf{A}\colon X\to E$ linear.
Then the following assertions are equivalent:
\begin{enumerate}
\item [a)] There is a (unique) extension $\widehat{\mathsf{A}}\in Y\varepsilon E$ of $\mathsf{A}$.
\item [b)] $(\mathsf{A}^{t})^{-1}(Y)$ $(=\{e'\in E'\;|\; e'\circ \mathsf{A}\in Y\})$ determines boundedness in $E$.
\end{enumerate}
\end{prop}

Next, we generalise \cite[Theorem 9, p.\ 232]{B/F/J} using the preceding proposition. The proof 
of the generalisation is simply obtained by replacing the set of uniqueness in the proof of 
\cite[Theorem 9, p.\ 232]{B/F/J} by our more general set of uniqueness. 

\begin{thm}\label{thm:ext_FS_set_uni}
Let $E$ be a locally complete lcHs, $G\subset E'$ determine boundedness and $\F$ and $\FE$ 
be $\varepsilon$-into-compatible. 
Let $(T^{E}_{m},T^{\K}_{m})_{m\in M}$ be a strong, consistent family for $(\mathcal{F},E)$,
$\F$ a Fr\'{e}chet--Schwartz space and $U$ a set of uniqueness for $(T^{\K}_{m},\mathcal{F})_{m\in M}$. 
Then the restriction map $R_{U,G}\colon S(\F\varepsilon E) \to \mathcal{F}_{G}(U,E)$ is surjective.
\end{thm}
\begin{proof}
Let $f\in \mathcal{F}_{G}(U,E)$. 
We choose $X:=\operatorname{span}\{T^{\K}_{m,x}\;|\;(m,x)\in U\}$ and $Y:=\F$.
Let $\mathsf{A}\colon X\to E$ be the linear map generated by $\mathsf{A}(T^{\K}_{m,x}):=f(m,x)$. 
The map $\mathsf{A}$ is well-defined since $G$ is $\sigma(E',E)$-dense. 
Let $e'\in G$ and $f_{e'}$ be the unique element in $\F$ such that 
$T^{\K}_{m}(f_{e'})(x)=(e'\circ \mathsf{A})(T^{\K}_{m,x})$ for all $(m,x)\in U$. This equation 
allows us to consider $f_{e'}$ as a linear form on $X$ 
(by setting $f_{e'}(T^{\K}_{m,x}):=(e'\circ \mathsf{A})(T^{\K}_{m,x})$), 
which yields $e'\circ \mathsf{A}\in\F$ for all $e'\in G$. It follows that
$G\subset(\mathsf{A}^{t})^{-1}(Y)$, implying that $(\mathsf{A}^{t})^{-1}(Y)$ determines boundedness. 
Applying \prettyref{prop:ext_FS_set_uni}, there is an extension $\widehat{\mathsf{A}}\in\F\varepsilon E$ of 
$\mathsf{A}$ and we set $F:=S(\widehat{\mathsf{A}})$. We note that 
\[
T^{E}_{m}(F)(x)=T^{E}_{m}S(\widehat{\mathsf{A}})(x)=\widehat{\mathsf{A}}(T^{\K}_{m,x})
=\mathsf{A}(T^{\K}_{m,x})=f(m,x)
\]
for all $(m,x)\in U$ by consistency, yielding $R_{U,G}(F)=f$.
\end{proof}

Let us apply the preceding theorem to our weighted spaces of continuously partially differentiable functions 
and its subspaces from \prettyref{ex:weighted_smooth_functions} and \prettyref{ex:diff_vanish_at_infinity}.

\begin{cor}\label{cor:weak_strong_CV}
Let $E$ be a locally complete lcHs, $G\subset E'$ determine boundedness,
$\mathcal{V}^{\infty}$ a directed family of weights which is locally bounded away from zero on 
an open set $\Omega\subset\R^{d}$, 
let $\mathcal{F}(\Omega)$ be a Fr\'{e}chet--Schwartz space and $U\subset\N_{0}^{d}\times\Omega$ 
a set of uniqueness for $(\partial^{\beta},\mathcal{F})_{\beta\in\N_{0}^{d}}$ 
where $\mathcal{F}$ stands for $\mathcal{CV}^{\infty}$, $\mathcal{CV}^{\infty}_{0}$, 
$\mathcal{CV}^{\infty}_{P(\partial)}$ or $\mathcal{CV}^{\infty}_{P(\partial),0}$. 
Then the following holds:
\begin{enumerate}
\item[a)] If $f\colon U\to E$ is a function such that there is $f_{e'}\in\mathcal{F}(\Omega)$ for each $e'\in G$ 
with $\partial^{\beta}f_{e'}(x)=(e'\circ f)(\beta,x)$ for all $(\beta,x)\in U$, 
then there is a unique $F\in\mathcal{F}(\Omega,E)$ with $(\partial^{\beta})^{E}F(x)=f(\beta,x)$ 
for all $(\beta,x)\in U$.
\item[b)] If $U\subset\Omega$ and $f\colon U\to E$ is a function such that 
$e'\circ f$ admits an extension $f_{e'}\in\mathcal{F}(\Omega)$ for every $e'\in G$, 
then there is a unique extension $F\in\mathcal{F}(\Omega,E)$ of $f$. 
\item[c)] $\FE=\{f\colon\Omega\to E\;|\;\forall\;e'\in G:\;e'\circ f\in\F \}$.
\end{enumerate}
\end{cor}
\begin{proof}
The strength and consistency of $((\partial^{\beta})^{E},\partial^{\beta})_{\beta\in\N_{0}^{d}}$ 
for $(\mathcal{F},E)$ and the $\varepsilon$-compatibility of 
$\F$ and $\FE$ follow from \prettyref{ex:weighted_diff} e)+f) 
and \prettyref{ex:diff_vanish_at_infinity} c)+f).
This implies that part a) and its special case part b) hold by \prettyref{thm:ext_FS_set_uni} 
and \prettyref{prop:injectivity}. 
Part c) follows from part b) and \prettyref{prop:weak_strong_principle} since 
$U:=\Omega$ is a set of uniqueness for $(\id_{\K^{\Omega}},\mathcal{F})$. 
\end{proof}

\begin{rem}\label{rem:set_unique_CV}
Let $\mathcal{V}^{\infty}$ be a directed family of weights which is locally bounded away from zero 
on an open set $\Omega\subset\R^{d}$. 
\begin{enumerate}
\item[a)] Then any dense set $U\subset\Omega$ is a set of uniqueness for $(\id_{\K^\Omega},\mathcal{F})$ with 
$\mathcal{F}=\mathcal{CV}^{\infty}$, $\mathcal{CV}^{\infty}_{0}$, $\mathcal{CV}^{\infty}_{P(\partial)}$ 
or $\mathcal{CV}^{\infty}_{P(\partial),0}$ due to continuity.
\item[b)] Let $\Omega$ be connected and $x_{0}\in\Omega$. 
Then $U:=\{(e_{n},x)\;|\;1\leq n\leq d,\,x\in\Omega\}\cup\{(0,x_{0})\}$ 
is a set of uniqueness for $(\partial^{\beta},\mathcal{F})_{\beta\in\N_{0}}$ by the mean value theorem 
with $\mathcal{F}$ from a).
\item[c)] Let $\K:=\R$, $d:=1$, $\Omega:=(a,b)\subset\R$, $g\colon (a,b)\to \N$ and $x_{0}\in(a,b)$. 
Then $U:=\{(g(x),x)\;|\;x\in(a,b)\}\cup\{(n,x_{0})\;|\;n\in\N_{0}\}$ is a set of uniqueness 
for $(\partial^{\beta},\mathcal{F})_{\beta\in\N_{0}}$ with $\mathcal{F}$ from a). 
Indeed, if $f\in\mathcal{F}(\Omega)$ and $0=\partial^{g(x)}f(x)$ for all $x\in(a,b)$, then $f$ is a polynomial 
by \cite[Chap.\ 11, Theorem, p.\ 53]{donoghue1969}. If, in addition, $0=\partial^{n}f(x_{0})$ for all $n\in\N_{0}$, 
then the polynomial $f$ must vanish on the whole interval $\Omega$.
\item[d)] Let $\Omega\subset\C$ be connected. Then any set $U\subset\Omega$ with an accumulation point in 
$\Omega$ is a set of uniqueness for $(\id_{\C^\Omega},\mathcal{CV}^{\infty}_{\overline{\partial}})$ 
by the identity theorem for holomorphic functions.
\item[e)] Let $\Omega\subset\C$ be connected and $z_{0}\in\Omega$. 
Then $U:=\{(n,z_{0})\;|\;n\in\N_{0}\}$ is a set of uniqueness 
for $(\partial^{n}_{\C},\mathcal{CV}^{\infty}_{\overline{\partial}})_{n\in\N_{0}}$ by local power series expansion and the identity theorem.
\item[f)] Let $\Omega\subset\R^{d}$ be connected. 
Then any non-empty open set $U\subset\Omega$ is a set of uniqueness for 
$(\id_{\K^\Omega},\mathcal{CV}^{\infty}_{\Delta})$ by the identity theorem for harmonic functions 
(see e.g.\ \cite[Theorem 5, p.\ 218]{gustin1948}).
\item[g)] Further examples of sets of uniqueness for $(\id_{\K^\Omega},\mathcal{CV}^{\infty}_{\Delta})$ 
are given in \cite{kokurin2015}.
\end{enumerate}
\end{rem}

In part e) a special case of \prettyref{rem:Schauder_coeff_set_uni} b) is used, namely, 
that $\mathcal{CW}^{\infty}_{\overline{\partial}}(\D_{r}(z_{0}))$ 
has a Schauder basis with associated coefficient functionals $(\delta_{z_{0}}\circ\partial^{n}_{\C})_{n\in\N_{0}}$ 
where $0<r\leq\infty$ is such that $\D_{r}(z_{0})\subset\Omega$.
In order to obtain some sets of uniqueness which are more sensible w.r.t.\ the family of weights $\mathcal{V}^{\infty}$, 
we turn to entire and harmonic functions fulfilling some growth conditions. 
For a family $\mathcal{V}:=(\nu_{j})_{j\in\N}$ of continuous weights on $\R^{d}$ set 
$\mathcal{V}^{\infty}:=(\nu_{j,m})_{j\in\N,m\in\N_{0}}$ where 
$\nu_{j,m}\colon\{\beta\in\N_{0}^{d}\;|\;|\beta|\leq m\}\times\R^{d}\to [0,\infty)$, $\nu_{j,m}(\beta,x):=\nu_{j}(x)$.
We know that $\mathcal{CV}_{P(\partial)}(\R^{d},E)=\mathcal{CV}_{P(\partial)}^{\infty}(\R^{d},E)$ 
as locally convex spaces and that $\mathcal{CV}_{P(\partial)}(\R^{d})$ is a nuclear Fr\'echet space 
for $P(\partial)=\overline{\partial}$ or $P(\partial)=\Delta$ by \prettyref{prop:AV_CV_coincide_hol_har} 
if $E$ is a locally complete lcHs and $\mathcal{V}$ fulfils \prettyref{cond:weights}. 
In particular, \prettyref{cond:weights} is fulfilled if 
$\nu_{j}(x):=\exp(-(\tau+\tfrac{1}{j})|x|)$, $x\in\R^{d}$, for all $j\in\N$ and some $0\leq \tau<\infty$ 
and thus we can apply \prettyref{cor:weak_strong_CV} to the spaces 
$A_{\overline{\partial}}^{\tau}(\C,E)=\mathcal{CV}_{\overline{\partial}}(\C,E)$ 
of entire and $A_{\Delta}^{\tau}(\R^{d},E)=\mathcal{CV}_{\Delta}(\R^{d},E)$ of harmonic functions of exponential 
type $\tau$ by \prettyref{rem:ex_NF_cond_weights}.
Hence we may complement our list in \prettyref{rem:set_unique_CV} by some more examples for spaces of functions 
of exponential type $0\leq\tau<\infty$.

\begin{rem}\label{rem:ex_set_uni_AP}
The following sets $U\subset\C$ are sets of uniqueness for $(\id_{\C^{\C}},A_{\overline{\partial}}^{\tau})$.
\begin{enumerate}
\item[a)] If $\tau<\pi$, then $U:=\N_{0}$ is a set of uniqueness by \cite[9.2.1 Carlson's theorem, p.\ 153]{boas1954}. 
\item[b)] Let $\delta>0$ and $(\lambda_{n})_{n\in\N}\subset(0,\infty)$ such that $\lambda_{n+1}-\lambda_{n}>\delta$ for all $n\in\N$. 
Then $U:=(\lambda_{n})_{n\in\N}$ is a set of uniqueness if $\limsup_{r\to\infty}r^{-2\tau/\pi}\psi(r)=\infty$ where 
$\psi(r):=\exp(\sum_{\lambda_{n}<r}\lambda_{n}^{-1})$, $r>0$, by \cite[9.5.1 Fuchs's theorem, p.\ 157--158]{boas1954}.
\end{enumerate}
The following sets $U$ are sets of uniqueness for $(\partial^{n}_{\C},A_{\overline{\partial}}^{\tau})_{n\in\N_{0}}$.
\begin{enumerate}
\item[c)] Let $(\lambda_{n})_{n\in\N_{0}}\subset\C$ with $|\lambda_{n}|<1$ for all $n\in\N_{0}$. If $\tau<\ln(2)$, 
then $U:=\{(n,\lambda_{n})\;|\;n\in\N_{0}\}$ is a set of uniqueness by \cite[9.11.1 Theorem, p.\ 172]{boas1954}. 
If $\tau<\ln(2+\sqrt{3})$, then $U:=\{(2n+1,0)\;|\;n\in\N_{0}\}\cup\{(2n,\lambda_{n})\;|\;n\in\N_{0}\}$ is 
a set of uniqueness by \cite[9.11.3 Theorem, p.\ 173]{boas1954}.
\item[d)] Let $(\lambda_{n})_{n\in\N_{0}}\subset\C$ with $\limsup_{n\to\infty}n^{-1}\sum_{k=1}^{n}|\lambda_{k}|\leq 1$. 
If $\tau<\e^{-1}$, then $U:=\{(n,\lambda_{n})\;|\;n\in\N_{0}\}$ is a set of uniqueness by \cite[9.11.4 Theorem, p.\ 173]{boas1954}.
\end{enumerate}
The following sets $U\subset\R^{d}$ are sets of uniqueness for $(\id_{\R^{\R^{d}}},A_{\Delta}^{\tau})$.
\begin{enumerate}
\item[e)] Let $d:=2$. If there is $k\in\N$ with $\tau<\pi/k$, then $U:=\Z\cup(\Z+\iu k)$ is a set of uniqueness by \cite[Theorem 1, p.\ 425]{boas1972}.
\item[f)] Let $d:=2$. If $\tau<\pi$ and $\theta\notin\pi\Q$, then $U:=\Z\cup(\e^{\iu\theta}\Z)$ is a set of uniqueness 
by \cite[Theorem 2, p.\ 426]{boas1972}.
\item[g)] If $\tau<\pi$, then $U:=\{0,1\}\times\Z^{d-1}$ is a set of uniqueness by \cite[Corollary 1.8, p.\ 312]{rao1974}. 
\item[h)] If $\tau<\pi$ and $a\in\R$ with $|a|\leq\sqrt{1/(d-1)}$, then $U:=\Z^{d-1}\times\{0,a\}$ is a set of uniqueness by 
\cite[Theorem A, p.\ 335]{zeilberger1976}.
\item[i)] Further examples of sets of uniqueness can be found in \cite{armitage1979}.
\end{enumerate}
The following sets $U$ are sets of uniqueness for $((\partial^{\beta})^{\R},A_{\Delta}^{\tau})_{\beta\in\N_{0}^{d}}$.
\begin{enumerate}
\item[j)] If $\tau<\pi$, then $U:=\{(\beta,(x,0))\;|\;\beta\in\{0,e_{d}\},\,x\in\Z^{d-1}\}$ is 
a set of uniqueness by \cite[Theorem B, p.\ 335]{zeilberger1976}. Further examples can be found in \cite{armitage1979}.
\end{enumerate}
\end{rem}

We need the following weak-strong principle in our last section for the space $\mathcal{E}_{0}(E)$ 
of $E$-valued infinitely continuously partially differentiable functions on $(0,1)$ such that all derivatives 
can be continuously extended to the boundary and vanish at $1$.

\begin{cor}\label{cor:weak_strong_E_0}
Let $E$ be a locally complete lcHs and $G\subset E'$ determine boundedness. Then 
$\mathcal{E}_{0}(E)=\{f\colon(0,1)\to E\;|\;\forall\;e'\in G:\;e'\circ f\in\mathcal{E}_{0} \}$.
\end{cor}
\begin{proof}
By \prettyref{ex:E_0} $\mathcal{E}_{0}$ is a Fr\'echet--Schwartz space and 
$\mathcal{E}_{0}$ and $\mathcal{E}_{0}(E)$ are $\varepsilon$-compatible.
We derive our statement from \prettyref{thm:ext_FS_set_uni} and \prettyref{prop:weak_strong_principle} 
with $(T^{E},T^{\K}):=(\id_{E^{U}},\id_{\K^{U}})$ and $U:=(0,1)$.
\end{proof}

\addtocontents{toc}{\SkipTocEntry}
\subsection*{\texorpdfstring{$\Fv$}{Fv(Omega)} a Banach space and \texorpdfstring{$E$}{E} locally complete}

In this subsection we consider function spaces $\FE$ with a certain structure, namely, spaces $\FVE$ from 
\prettyref{def:weighted_space} where the family of weights $\mathcal{V}=(\nu_{j,m})_{j\in J,m\in M}$ 
only consists of one weight function, i.e.\ the sets $J$ and $M$ can be chosen as singletons. 
So for two non-empty sets $\Omega$ and $\omega$, a weight $\nu\colon\omega\to (0,\infty)$, 
a linear operator $T^{E}\colon E^{\Omega}\supset\dom T^{E}\to E^{\omega}$ and 
a linear subspace $\operatorname{AP}(\Omega,E)$ of $E^{\Omega}$ we consider the space
\[
\gls{FvE}=\bigl\{f\in \fe\;|\; 
 \forall\;\alpha\in \mathfrak{A}:\; |f|_{\alpha}<\infty\bigr\}
\]
where 
\[
\fe=\operatorname{AP}(\Omega,E)\cap \dom T^{E}
\]
and
\[
|f|_{\alpha}=|f|_{\Fv,\alpha}=\sup_{x \in \omega}p_{\alpha}\bigl(T^{E}(f)(x)\bigr)\nu(x).
\]

For instance, if $\Omega:=\omega$, $T^{E}:=\id_{E^{\Omega}}$ and $\nu:=1$ on $\Omega$, then 
$\FvE$ is the linear subspace of $\fe$ consisting of bounded functions.
We use the methods developed in \cite{F/J/W,jorda2013} where, in particular, the special case that 
$\Fv$ is the space of bounded smooth functions on an open set $\Omega\subset\R^{d}$ in the kernel of a 
hypoelliptic linear partial differential operator resp.\ a weighted space of holomorphic functions on 
an open subset $\Omega$ of a Banach space is treated. The lack of compact subsets of an infinite dimensional 
Banach space $\Fv$ is compensated in \cite{F/J/W,jorda2013} by equipping $\f$ 
with a locally convex Hausdorff topology such that the closed unit ball of $\Fv$ is compact in $\f$. 
Among others, the space $\fe:=(\mathcal{O}(\Omega,E),\tau_{c})$ of holomorphic functions 
on an open set $\Omega\subset\C$ with values in a locally complete space $E$ equipped 
with topology $\tau_{c}$ of compact convergence is used in \cite{F/J/W} and 
the space $\FvE:=\gls{H_inftyOE}$ of $E$-valued bounded holomorphic functions on $\Omega$.

\begin{prop}\label{prop:mingle-mangle}
Let $\f$ and $\fe$ be $\varepsilon$-into-compatible, 
$(T^{E},T^{\K})$ a consistent family for $(F,E)$ and a generator for $(\mathcal{F}\nu,E)$ and
the map $i\colon\Fv\to\f$, $f\mapsto f$, continuous. We set 
\[
 \gls{Feps}:=S(\{u\in\f\varepsilon E\;|\;u(B_{\Fv}^{\circ \f'})\;\text{is bounded in}\; E\})
\]
where $\gls{B_Fv_circ_f'}:=\{y'\in\f'\;|\;\forall\;f\in B_{\Fv}:\;|y'(f)|\leq 1\}$ 
and $B_{\Fv}$ is the closed unit ball of $\Fv$.
Then the following holds: 
\begin{enumerate}
\item[a)] $\Fv$ is a $\dom$-space.
\item[b)] Let $u\in\f\varepsilon E$. Then 
\[
\sup_{y'\in B_{\Fv}^{\circ \f'}}p_{\alpha}(u(y'))=|S(u)|_{\Fv,\alpha},\quad\alpha\in\mathfrak{A}.
\]
In particular, 
\[
 \Feps=S(\{u\in\f\varepsilon E\;|\;\forall\;\alpha\in\mathfrak{A}:\;|S(u)|_{\Fv,\alpha}<\infty\}).
\]
\item[c)] $S(\Fv\varepsilon E)\subset\Feps\subset\FvE$ as linear spaces. If $\f$ and $\fe$ are even 
$\varepsilon$-compatible, then $\Feps=\FvE$.
\item[d)] If $\FvE$ is Hausdorff, then 
\begin{enumerate}
\item[(i)] $(T^{E},T^{\K})$ is a consistent generator for $(\mathcal{F}\nu,E)$.
\item[(ii)] $\Fv$ and $\FvE$ are $\varepsilon$-into-compatible.
\item[(iii)] $(T^{E},T^{\K})$ is a strong generator for $(\mathcal{F}\nu,E)$ if it is a strong family for $(F,E)$.
\end{enumerate}
\end{enumerate}
\end{prop}
\begin{proof}
Part a) follows from the continuity of the map $i$ and the $\varepsilon$-into-compatibility of $\f$ and $\fe$. 
Let us turn to part b). As in \prettyref{lem:topology_eps} it follows from the bipolar theorem that 
\[
 B_{\Fv}^{\circ \f'}=\oacx\{T^{\K}_{x}(\cdot)\nu(x)\;|\;x\in\omega\},
\]
where $\oacx$ denotes the closure w.r.t.\ $\kappa(\f',\Fv)$ of the absolutely convex hull $\acx$ of the set 
$D:=\{T^{\K}_{x}(\cdot)\nu(x)\;|\;x\in\omega\}$ on the right-hand side, 
and that 
\begin{align*}
  \sup_{y'\in B_{\Fv}^{\circ \f'}}p_{\alpha}(u(y'))
&=\sup_{y'\in\acx(D)}p_{\alpha}(u(y')) 
 =\sup_{y'\in D}p_{\alpha}(u(y'))=\sup_{x\in\omega}p_{\alpha}\bigl(u(T^{\K}_{x})\bigr)\nu(x)\\
&=\sup_{x\in\omega}p_{\alpha}\bigl(T^{E}(S(u))(x)\bigr)\nu(x)
 =|S(u)|_{\Fv,\alpha}
\end{align*}
by consistency, which proves part b).

Let us address part c). The continuity of the map $i$ implies 
the continuity of the inclusion $\Fv\varepsilon E\hookrightarrow\f\varepsilon E$ and 
thus we obtain $u_{\mid \f'}\in \f\varepsilon E$ for every $u\in\Fv\varepsilon E$.
If $u\in\Fv\varepsilon E$ and $\alpha\in\mathfrak{A}$, 
then there are $C_{0},C_{1}>0$ and an absolutely convex compact set $K\subset\Fv$ 
such that $K\subset C_{1}B_{\Fv}$ and 
\[
 \sup_{y'\in B_{\Fv}^{\circ \f'}}p_{\alpha}(u(y'))
 \leq C_{0}\sup_{y'\in B_{\Fv}^{\circ \f'}}\sup_{f\in K}|y'(f)|
 \leq C_{0}C_{1},
\]
which implies $S(\Fv\varepsilon E)\subset\Feps$. 
If $f:=S(u)\in\Feps$ and $\alpha\in\mathfrak{A}$, then $S(u)\in\fe$ and
\[
 |f|_{\Fv,\alpha}=\sup_{x\in\omega}p_{\alpha}(u(T^{\K}_{x})\nu(x))<\infty
\]
by consistency, yielding $\Feps\subset\FvE$. If $\f$ and $\fe$ are even 
$\varepsilon$-compatible, then $S(\f\varepsilon E)=\fe$, which yields $\Feps=\FvE$ by part b).

Let us turn to part d). By part a) $\Fv$ is a $\dom$-space. Since $\FvE$ is Hausdorff, it is also a 
$\dom$-space due to \prettyref{rem:weights_Hausdorff_directed} c).
We have $u_{\mid \f'}\in \f\varepsilon E$ for every $u\in\Fv\varepsilon E$ and 
\[
S_{\Fv}(u)(x)=u(\delta_{x})=u_{\mid \f'}(\delta_{x})
=S_{\f}(u_{\mid \f'})(x),\quad x\in\Omega.
\]
In combination with $S(\f\varepsilon E)\subset\fe$ and the consistency of $(T^{E},T^{\K})$ for $(F,E)$ 
this yields that $(T^{E},T^{\K})$ is a consistent generator for $(\mathcal{F}\nu,E)$. 
Thus part (i) holds and implies part (ii) by \prettyref{thm:linearisation}. 
If $(T^{E},T^{\K})$ is in addition a strong family for $(F,E)$, 
then the inclusion $\FvE\subset\fe$ implies that $e'\circ f\in\fe$ and 
$T^{\K}(e'\circ f)(x)=(e'\circ T^{E}(f))(x)$ for all $e'\in E'$, $f\in\FvE$ and $x\in\omega$. 
It follows that $(T^{E},T^{\K})$ is a strong generator for $(\mathcal{F}\nu,E)$. 
\end{proof}

The canonical situation in part c) is that $\Feps$ and $\FvE$ coincide as linear spaces for 
locally complete $E$ as we will encounter in the forthcoming examples, 
e.g.\ if $\FvE:=H^{\infty}(\Omega,E)$ and $\fe:=(\mathcal{O}(\Omega,E),\tau_{c})$ for an open set $\Omega\subset\C$. 
That all three spaces in part c) coincide is usually only guaranteed by \prettyref{cor:full_linearisation} (iii)
if $E$ is a semi-Montel space. Therefore the `mingle-mangle' space $\Feps$ is a good replacement for 
$S(\Fv\varepsilon E)$ for our purpose. 

\begin{rem}\label{rem:R_well-defined_Banach}
Let $(T^{E},T^{\K})$ be a strong, consistent family for $(F,E)$ and a generator for $(\mathcal{F}\nu,E)$.
Let $\f$ and $\fe$ be $\varepsilon$-into-compatible and the inclusion $\Fv\hookrightarrow\f$ continuous. 
Consider a set of uniqueness $U$ for $(T^{\K},\mathcal{F}\nu)$ and a separating subspace $G\subset E'$.
For $u\in \f\varepsilon E$ such that $u(B_{\Fv}^{\circ \f'})$ is bounded in $E$, 
i.e.\ $S(u)\in \Feps$, we set $f:=S(u)$. 
Then $f\in\fe$ by the $\varepsilon$-into-compatibility 
and we define $\widetilde{f}\colon U\to E$, $\widetilde{f}(x):=T^{E}(f)(x)$. This yields
\begin{equation}\label{eq:injective}
(e'\circ \widetilde{f})(x)= (e'\circ T^{E}(f))(x)=T^{\K}(e'\circ f)(x)
\end{equation}
for all $x\in U$ and $f_{e'}:=e'\circ f\in\f$ for each $e'\in E'$ by the strength of the family. 
Moreover, $T^{\K}_{x}(\cdot)\nu(x)\in B_{\Fv}^{\circ \f'}$ for every $x\in\omega$, 
which implies that for every $e'\in E'$ there are $\alpha\in\mathfrak{A}$ and $C>0$ such that
\[
|f_{e'}|_{\Fv}=\sup_{x\in\omega}\bigl|e'\bigl(u(T^{\K}_{x}(\cdot)\nu(x)\bigr)\bigr|
\leq C\sup_{y'\in B_{\Fv}^{\circ \f'}}p_{\alpha}(u(y'))<\infty
\]
by strength and consistency. Hence $f_{e'}\in\Fv$ for every $e'\in E'$ and $\widetilde{f}\in\mathcal{F}\nu_{G}(U,E)$.
\end{rem}

Under the assumptions of \prettyref{rem:R_well-defined_Banach} the map
\begin{equation}\label{eq:well_def_B_unique}
R_{U,G}\colon \Feps\to \mathcal{F}\nu_{G}(U,E),\;f\mapsto (T^{E}(f)(x))_{x\in U}, 
\end{equation}
is well-defined and linear. In addition, we derive from \eqref{eq:injective} that $R_{U,G}$ is injective since 
$U$ is a set of uniqueness and $G\subset E'$ separating. 
The replacement of \prettyref{que:surj_restr_set_unique} reads as follows.

\begin{que}\label{que:surj_restr_set_unique_banach}
Let the assumptions of \prettyref{rem:R_well-defined_Banach} be fulfilled.
When is the injective restriction map 
\[
R_{U,G}\colon \Feps\to \mathcal{F}\nu_{G}(U,E),\;f\mapsto (T^{E}(f)(x))_{x\in U},
\]
surjective?
\end{que}

Due to \prettyref{prop:mingle-mangle} c) the \prettyref{que:weak_strong} is a special case of this question 
if $\Lambda\subset\Omega=:\omega$ and $U:=\Lambda$ is a set of uniqueness for $(\id_{\K^{\Omega}},\mathcal{F}\nu)$.
We recall the following extension result for continuous linear operators. 

\begin{prop}[{\cite[Proposition 2.1, p.\ 691]{F/J/W}}]\label{prop:ext_B_set_uni}
Let $E$ be a locally complete lcHs, $G\subset E'$ determine boundedness, $Z$ a Banach space whose closed unit ball $B_{Z}$ is a 
compact subset of an lcHs $Y$ and $X\subset Y'$ be a $\sigma(Y',Z)$-dense subspace. 
If $\mathsf{A}\colon X\to E$ is a $\sigma(X,Z)$-$\sigma(E,G)$-continuous linear map, 
then there exists a (unique) extension $\widehat{\mathsf{A}}\in Y\varepsilon E$ of $\mathsf{A}$ 
such that $\widehat{\mathsf{A}}(B_{Z}^{\circ Y'})$ is bounded in $E$ where 
$B_{Z}^{\circ Y'}:=\{y'\in Y'\;|\;\forall\;z\in B_{Z}:\; |y'(z)|\leq 1\}$.
\end{prop}

Now, we are able to generalise \cite[Theorem 2.2, p.\ 691]{F/J/W} and \cite[Theorem 10, p.\ 5]{jorda2013}.

\begin{thm}\label{thm:ext_B_unique}
Let $E$ be a locally complete lcHs, $G\subset E'$ determine boundedness 
and $\f$ and $\fe$ be $\varepsilon$-into-compatible. 
Let $(T^{E},T^{\K})$ be a generator for $(\mathcal{F}\nu,E)$ 
and a strong, consistent family for $(F,E)$, $\Fv$ a Banach space whose closed unit ball $B_{\Fv}$ 
is a compact subset of $\f$ and $U$ a set of uniqueness for $(T^{\K},\mathcal{F}\nu)$.
Then the restriction map 
\[
 R_{U,G}\colon \Feps \to \mathcal{F}\nu_{G}(U,E) 
\]
is surjective.
\end{thm}
\begin{proof}
Let $f\in \mathcal{F}\nu_{G}(U,E)$. 
We set $X:=\operatorname{span}\{T^{\K}_{x}\;|\;x\in U\}$, $Y:=\f$ and $Z:=\Fv$. 
The consistency of $(T^{E},T^{\K})$ for $(F,E)$ yields that $X\subset Y'$.
From $U$ being a set of uniqueness of $Z$ follows that $X$ is 
$\sigma(Z',Z)$-dense. Since $B_{Z}$ is a compact subset of $Y$, it follows that $Z$ is a linear subspace of $Y$ 
and the inclusion $Z\hookrightarrow Y$ is continuous, which yields $y'_{\mid Z}\in Z'$ for every $y'\in Y'$. 
Thus $X$ is $\sigma(Y',Z)$-dense. 
Let $\mathsf{A}\colon X\to E$ be the linear map determined by $\mathsf{A}(T^{\K}_{x}):=f(x)$. 
The map $\mathsf{A}$ is well-defined since $G$ is $\sigma(E',E)$-dense. 
Due to
\[
e'(\mathsf{A}(T^{\K}_{x}))=(e'\circ f)(x)=T^{\K}_{x}(f_{e'})
\]
for every $e'\in G$ and $x\in U$ we have that $\mathsf{A}$ is $\sigma(X,Z)$-$\sigma(E,G)$-continuous. 
We apply \prettyref{prop:ext_B_set_uni} and gain an extension $\widehat{\mathsf{A}}\in Y\varepsilon E$ 
of $\mathsf{A}$ such that $\widehat{\mathsf{A}}(B_{Z}^{\circ Y'})$ is bounded in $E$. 
We set $\widetilde{F}:=S(\widehat{\mathsf{A}})\in\Feps$ and get for all $x\in U$ that
\[
T^{E}(\widetilde{F})(x)=T^{E}S(\widehat{\mathsf{A}})(x)=\widehat{\mathsf{A}}(T^{\K}_{x})=f(x)
\]
by consistency for $(F,E)$, implying $R_{U,G}(\widetilde{F})=f$.
\end{proof}

Let $\Omega\subset\R^{d}$ be open, $E$ an lcHs and 
$P(\partial)^{E}\colon\mathcal{C}^{\infty}(\Omega,E)\to\mathcal{C}^{\infty}(\Omega,E)$ 
a linear partial differential operator which is hypoelliptic if $E=\K$. 
We consider the weighted space $\mathcal{CV}_{P(\partial)}(\Omega,E)$ of zero solutions from 
\prettyref{prop:frechet_bierstedt} where the familiy of weights $\mathcal{V}$ only consists of one 
continuous weight $\nu\colon\Omega\to(0,\infty)$, i.e.\ the space
\[
 \mathcal{C}\nu_{P(\partial)}(\Omega,E)
=\{f\in\mathcal{C}^{\infty}_{P(\partial)}(\Omega,E)\;|\;\forall\;\alpha\in\mathfrak{A}:\;
   |f|_{\nu,\alpha}:=\sup_{x\in\Omega}p_{\alpha}(f(x))\nu(x)<\infty\}.
\]

\begin{cor}\label{cor:hypo_weighted_ext_unique}
Let $E$ be a locally complete lcHs, $G\subset E'$ determine boundedness, $\Omega\subset\R^{d}$ open, 
$P(\partial)^{\K}$ a hypoelliptic linear partial differential operator, $\nu\colon\Omega\to(0,\infty)$ continuous and 
$U$ a set of uniqueness for $(\id_{\K^{\Omega}},\mathcal{C}\nu_{P(\partial)})$.
If $f\colon U\to E$ is a function such that $e'\circ f$ admits an extension 
$f_{e'}\in\mathcal{C}\nu_{P(\partial)}(\Omega)$ 
for every $e'\in G$, then there exists a unique extension $F\in\mathcal{C}\nu_{P(\partial)}(\Omega,E)$ of $f$.
\end{cor}
\begin{proof}
We choose $\f:=(\mathcal{C}^{\infty}_{P(\partial)}(\Omega),\tau_{c})$ 
and $\fe:=(\mathcal{C}^{\infty}_{P(\partial)}(\Omega,E),\tau_{c})$. Then we  
have $\mathcal{F}\nu(\Omega)=\mathcal{C}\nu_{P(\partial)}(\Omega)$ and 
$\mathcal{F}\nu(\Omega,E)=\mathcal{C}\nu_{P(\partial)}(\Omega,E)$
with the generator $(T^{E},T^{\K}):=(\operatorname{\id}_{E^{\Omega}},\operatorname{\id}_{\K^{\Omega}})$ 
for $(\mathcal{F}\nu,E)$.
We note that $\f$ and $\fe$ are $\varepsilon$-compatible 
and $(T^{E},T^{\K})$ is a strong, consistent family for $(F,E)$ by \prettyref{prop:co_top_isomorphism}. 
We observe that $\mathcal{F}\nu(\Omega)$ is a Banach space by \prettyref{prop:frechet_bierstedt} and 
for every compact $K\subset\Omega$ we have 
\[
\sup_{x\in K}|f(x)|\underset{\eqref{eq:hypo_weighted_Frechet}}{\leq}\sup_{z\in K}\nu(z)^{-1}|f|_{\nu} 
\leq \sup_{z\in K}\nu(z)^{-1},
\quad f\in B_{\mathcal{F}\nu(\Omega)},
\]
yielding that $B_{\mathcal{F}\nu(\Omega)}$ is bounded in $\f$.  
The space $\f=(\mathcal{C}^{\infty}_{P(\partial)}(\Omega),\tau_{c})$ is a Fr\'echet--Schwartz space, 
thus a Montel space, and it is easy to check that $B_{\mathcal{F}\nu(\Omega)}$ is $\tau_{c}$-closed. 
Hence the bounded and $\tau_{c}$-closed set $B_{\mathcal{F}\nu(\Omega)}$ is compact in $\f$. 
Finally, we remark that the $\varepsilon$-compatibility of $\f$ and $\fe$ 
in combination with the consistency of $(\id_{E^{\Omega}},\id_{\K^{\Omega}})$ for $(F,E)$ gives
$\Feps=\FvE$ as linear spaces by \prettyref{prop:mingle-mangle} c). 
From \prettyref{thm:ext_B_unique} follows our statement.
\end{proof}

If $\Omega=\D\subset\C$ is the open unit disc, $P(\partial)=\overline{\partial}$ the Cauchy--Riemann operator 
and $\nu=1$ on $\D$, then $\mathcal{C}\nu_{P(\partial)}(\Omega,E)=H^{\infty}(\D,E)$ and 
a sequence $U:=(z_{n})_{n\in\N}\subset\D$ of distinct elements is a set of uniqueness for $(\id_{\C^{\D}},H^{\infty})$ 
if and only if it satisfies the Blaschke condition $\sum_{n\in\N}(1-|z_{n}|)=\infty$ 
(see e.g.\ \cite[15.23 Theorem, p.\ 303]{rudin1970}).

For a continuous function $\nu\colon\D\to(0,\infty)$ and a complex lcHs $E$ we define the \emph{\gls{Bloch_type_space}} 
\[
\gls{BnuDE}:=\{f\in\mathcal{O}(\D,E)\;|\;\forall\;\alpha\in\mathfrak{A}:\;|f|_{\nu,\alpha}<\infty\}
\]
with 
\[
|f|_{\nu,\alpha}:=\max\bigl(p_{\alpha}(f(0)),\sup_{z\in\D}p_{\alpha}((\partial_{\C}^{1})^{E}f(z))\nu(z)\bigr).
\]
If $E=\C$, we write $f'(z):=(\partial_{\C}^{1})^{\C}f(z)$ for $z\in\D$ and $f\in\mathcal{O}(\D)$. 

\begin{prop}\label{prop:Bloch_Banach}
If $\nu\colon\D\to(0,\infty)$ is continuous, then $\mathcal{B}\nu(\D)$ is a Banach space.
\end{prop}
\begin{proof}
Let $f\in\mathcal{B}\nu(\D)$. From the estimates 
\begin{align*}
|f(z)|&\leq |f(0)|+\bigl|\int_{0}^{z}f'(\zeta)\d\zeta\bigr|
\leq |f(0)|+\frac{|z|}{\min_{\xi\in[0,z]}\nu(\xi)}\sup_{\zeta\in[0,z]}|f'(\zeta)|\nu(\zeta)\\
&\leq 2\max\Bigl(1,\frac{|z|}{\min_{\xi\in[0,z]}\nu(\xi)}\Bigr)|f|_{\nu}
\end{align*}
for every $z\in\D$ and 
\begin{equation}\label{eq:Bloch}
\max_{|z|\leq r}|f(z)|\leq  2\max\Bigl(1,\frac{r}{\min_{|z|\leq r}\nu(z)}\Bigr)|f|_{\nu}
\end{equation}
for all $0<r<1$ and $f\in\mathcal{B}\nu(\D)$ it follows that 
$\mathcal{B}\nu(\D)$ is a Banach space 
by using the completeness of $(\mathcal{O}(\D),\tau_{c})$ analogously to the proof 
of \prettyref{prop:frechet_bierstedt}.
\end{proof}

\begin{prop}\label{prop:complex_diff_cons_strong}
Let $\Omega\subset\C$ be open and $E$ a locally complete lcHs over $\C$. 
Then $((\partial_{\C}^{n})^{E},(\partial_{\C}^{n})^{\C})_{n\in\N_{0}}$ 
is a strong, consistent family for $((\mathcal{O}(\Omega),\tau_{c}),E)$.
\end{prop}
\begin{proof}
We recall from \eqref{eq:complex.real.deriv} that the real and complex derivatives are related by
\begin{equation}\label{eq:complex.real.deriv_1}
 (\partial^{\beta})^{E}f(z)=\iu^{\beta_{2}}(\partial^{|\beta|}_{\C})^{E}f(z),\quad z\in\Omega,
\end{equation}
for every $f\in\mathcal{O}(\Omega,E)$ and $\beta=(\beta_{1},\beta_{2})\in\N_{0}^{2}$.  
Further, the Fr\'{e}chet space $(\mathcal{O}(\Omega),\tau_{c})$ is barrelled. 
Due to \prettyref{prop:diff_cons_barrelled} c) and 
\eqref{eq:complex.real.deriv_1} we have for all $u\in(\mathcal{O}(\Omega),\tau_{c})\varepsilon E$ 
\[
(\partial^{n}_{\C})^{E}S(u)(z)=u(\delta_{z}\circ(\partial^{n}_{\C})^{\C}),\quad n\in\N_{0},\,z\in\Omega,
\]
which means that $((\partial_{\C}^{n})^{E},(\partial_{\C}^{n})^{\C})_{n\in\N_{0}}$ is consistent. 

Moreover, we have 
\[
(\partial^{n}_{\C})^{\C}(e'\circ f)(z)=e'\bigl((\partial^{n}_{\C})^{E}f(z)),\quad n\in\N_{0},\,z\in\Omega,
\]
for all $e'\in E'$ and $f\in\mathcal{O}(\Omega,E)$, implying the strength 
of $((\partial_{\C}^{n})^{E},(\partial_{\C}^{n})^{\C})_{n\in\N_{0}}$.
\end{proof}

Let $E$ be an lcHs and $\nu\colon\D\to(0,\infty)$ be continuous. 
We set $\omega:=\{0\}\cup\{(1,z)\;|\;z\in\D\}$, 
define the operator $T^{E}\colon \mathcal{O}(\D,E)\to E^{\omega}$ by 
\[
T^{E}(f)(0):=f(0)\quad\text{and}\quad T^{E}(f)(1,z):=(\partial_{\C}^{1})^{E}f(z),\;z\in\D,
\] 
and the weight $\nu_{\ast}\colon\omega\to (0,\infty)$ by 
\[
\nu_{\ast}(0):=1\quad\text{and}\quad\nu_{\ast}(1,z):=\nu(z),\;z\in\D.
\]
Then we have for every $\alpha\in\mathfrak{A}$ that
\[
 |f|_{\nu,\alpha}=\sup_{x\in\omega}p_{\alpha}\bigl(T^{E}(f)(x)\bigr)\nu_{\ast}(x),
 \quad f\in\mathcal{B}\nu(\D,E),
\]
and with $F(\D,E):=\mathcal{O}(\D,E)$ we observe that 
$\mathcal{F}\nu_{\ast}(\D,E)=\mathcal{B}\nu(\D,E)$ with generator $(T^{E},T^{\C})$.

\begin{cor}\label{cor:Bloch_ext_unique}
Let $E$ be a locally complete lcHs, $G\subset E'$ determine boundedness, $\nu\colon\D\to(0,\infty)$ continuous 
and $U_{\ast}\subset\D$ have an accumulation point in $\D$. 
If $f\colon \{0\}\cup(\{1\}\times U_{\ast})\to E$ is a function such that there is 
$f_{e'}\in\mathcal{B}\nu(\D)$ for each $e'\in G$ with $f_{e'}(0)=e'(f(0))$ and 
$f_{e'}'(z)=e'(f(1,z))$ for all $z\in U_{\ast}$, 
then there exists a unique $F\in\mathcal{B}\nu(\D,E)$ with $F(0)=f(0)$ and $(\partial_{\C}^{1})^{E}F(z)=f(1,z)$ 
for all $z\in U_{\ast}$.
\end{cor}
\begin{proof}
We take $F(\D):=(\mathcal{O}(\D),\tau_{c})$ 
and $F(\D,E):=(\mathcal{O}(\D,E),\tau_{c})$. Then 
we have $\mathcal{F}\nu_{\ast}(\D)=\mathcal{B}\nu(\D)$ 
and $\mathcal{F}\nu_{\ast}(\Omega,E)=\mathcal{B}\nu(\D,E)$ 
with the weight $\nu_{\ast}$ and generator $(T^{E},T^{\C})$
for $(\mathcal{F}\nu_{\ast},E)$ described above. 
The spaces $F(\D)$ and $F(\D,E)$ are $\varepsilon$-compatible by 
\prettyref{prop:co_top_isomorphism} in combination with \eqref{eq:holomorphic_coincide_1}, 
and the generator is a strong, consistent family for $(F,E)$ 
by \prettyref{prop:complex_diff_cons_strong}. 
Due to \prettyref{prop:Bloch_Banach} $\mathcal{F}\nu_{\ast}(\D)=\mathcal{B}\nu(\D)$ is a Banach space
and we deduce from \eqref{eq:Bloch} that $B_{\mathcal{F}\nu_{\ast}(\D)}$ is compact 
in the Montel space $(\mathcal{O}(\D),\tau_{c})$.
We note that the $\varepsilon$-compatibility of $\f$ and $\fe$ 
in combination with the consistency of $(T^{E},T^{\C})$ for $(F,E)$ gives
$\mathcal{F}_{\varepsilon}\nu_{\ast}(\D,E)=\mathcal{F}\nu_{\ast}(\D,E)$ as linear spaces by 
\prettyref{prop:mingle-mangle} c). In addition, $U:=\{0\}\cup\{(1,z)\;|\;z\in U_{\ast}\}$ is a set of uniqueness 
for $(T^{\C},\mathcal{F}\nu_{\ast})$ by the identity theorem, proving our statement by \prettyref{thm:ext_B_unique}.
\end{proof}

\addtocontents{toc}{\SkipTocEntry}
\subsection*{\texorpdfstring{$E$}{E} a Fr\'echet space}

In this section we restrict to the case that $E$ is a Fr\'echet space and $G\subset E'$ is generated 
by a sequence that \emph{fixes the topology} in $E$.

\begin{defn}[{\cite[Definition 12, p.\ 8]{B/F/J}}]\label{def:fix_top_2}
 Let $Y$ be a Fr\'{e}chet space. An increasing sequence $(B_{n})_{n\in\N}$ of bounded subsets of $Y_{b}'$ \emph{\gls{fix_topo}} in $Y$ if $(B_{n}^{\circ})_{n\in\N}$ is a fundamental system of zero neighbourhoods of $Y$.
\end{defn}

\begin{rem}\label{rem:ex_almost_norming}
Let $Y$ be a Banach space. If $B\subset Y_{b}'$ is bounded, i.e.\ bounded w.r.t.\ the operator norm, 
such that $B$ fixes the topology in $Y$, i.e.\ $B^{\circ}$ is bounded in $Y$, 
then $B$ is called an \emph{\gls{almost_norming}} subset.
Examples of almost norming subspaces are given in \cite[Remark 1.2, p.\ 780--781]{Arendt2000}. 
For instance, the set of point evaluations $B:=\{\delta_{1/n}\;|\; n\in\N\}$ is almost norming for
the $Y:=H^{\infty}(\D):=\mathcal{C}^{\infty}_{\overline{\partial},b}(\D)$.
\end{rem}

\begin{defn}[{\gls{sb_rest_space}}]
Let $E$ be a Fr\'{e}chet space, $(B_{n})$ fix the topology in $E$ and $G:=\operatorname{span}(\bigcup_{n\in\N} B_{n})$. 
Let $\FV$ be a $\dom$-space, $U$ a set of uniqueness for $(T^{\K}_{m},\mathcal{FV})_{m\in M}$ and
set
\[
\gls{FV_GUE_sb}:=\{f\in\mathcal{FV}_{G}(U,E)\;|\;\forall\;n\in\N:\; \{f_{e'}\;|\;e'\in B_{n}\}\;
\text{is bounded in}\;\FV\}.
\]
\end{defn}

Let $E$ be a Fr\'{e}chet space, $(B_{n})$ fix the topology in $E$, $G:=\operatorname{span}(\bigcup_{n\in\N} B_{n})$,
$(T^{E}_{m},T^{\K}_{m})_{m\in M}$ be a strong, consistent generator for $(\mathcal{FV},E)$ 
and $U$ a set of uniqueness for $(T^{\K}_{m},\mathcal{FV})_{m\in M}$. 
For $u\in \FV\varepsilon E$ we have $R_{U,G}(f)\in\mathcal{FV}_{G}(U,E)$ with $f:=S(u)$
by \prettyref{rem:R_well-defined} and for $j\in J$ and $m\in M$ 
\[
\sup_{e'\in B_{n}}|f_{e'}|_{j,m}=\sup_{e'\in B_{n}}\sup_{x\in\omega_{m}}|e'(T^{E}_{m}(f)(x)\nu_{j,m}(x))|
=\sup_{e'\in B_{n}}\sup_{y\in N_{j,m}(f)}|e'(y)|
\]
with $N_{j,m}(f):=\{T^{E}_{m}(f)(x)\nu_{j,m}(x)\;|\;x\in\omega_{m}\}$. This set is bounded in $E$ since 
\[
\sup_{y\in N_{j,m}(f)}p_{\alpha}(f)=|f|_{j,m,\alpha}<\infty
\]
for all $\alpha\in\mathfrak{A}$, implying $\sup_{e'\in B_{n}}|f_{e'}|_{j,m}<\infty$ and 
$R_{U,G}(f)\in\mathcal{FV}_{G}(U,E)_{sb}$. 
Hence the injective linear map
\[
R_{U,G}\colon S(\FV\varepsilon E)\to \mathcal{FV}_{G}(U,E)_{sb},\;f\mapsto (T^{E}_{m}(f)(x))_{(m,x)\in U},
\]
is well-defined.

\begin{que}\label{que:surjective_sb_rest_space}
Let $E$ be a Fr\'{e}chet space, $(B_{n})$ fix the topology in $E$ and $G:=\operatorname{span}(\bigcup_{n\in\N} B_{n})$. 
Let $(T^{E}_{m},T^{\K}_{m})_{m\in M}$ be a strong, consistent generator for $(\mathcal{FV},E)$
and $U$ a set of uniqueness for $(T^{\K}_{m},\mathcal{FV})_{m\in M}$.
When is the injective restriction map 
\[
R_{U,G}\colon S(\FV\varepsilon E)\to \mathcal{FV}_{G}(U,E)_{sb},\;f\mapsto (T^{E}_{m}(f)(x))_{(m,x)\in U},
\]
surjective?
\end{que}

\begin{rem}
Let $E$ be a Fr\'{e}chet space with increasing system of seminorms 
$(p_{\alpha_{n}})_{n\in\N}$, $B_{n}:=B_{\alpha_{n}}^{\circ}$ where 
$B_{\alpha_{n}}:=\{x\in E\;|\;p_{\alpha_{n}}(x)<1\}$, 
$(T^{E}_{m},T^{\K}_{m})_{m\in M}$ a strong, consistent generator for $(\mathcal{FV},E)$
and $U$ a set of uniqueness for $(T^{\K}_{m},\mathcal{FV})_{m\in M}$. If 
$\FV$ is a BC-space, then $\mathcal{FV}_{E'}(U,E)_{sb}=\mathcal{FV}_{E'}(U,E)$ by \prettyref{prop:ext_F_semi_M}. 
Hence \prettyref{thm:ext_F_semi_M} (i) answers \prettyref{que:surjective_sb_rest_space} in this case. 
\end{rem}

Let us turn to the case where $G$ need not coincide with $E'$.

\addtocontents{toc}{\SkipTocEntry}
\subsection*{\texorpdfstring{$\FV$}{FV(Omega)} a Fr\'{e}chet--Schwartz space and \texorpdfstring{$E$}{E} a Fr\'echet space}

We recall the following result.

\begin{prop}[{\cite[Lemma 9, p.\ 504]{F/J}}]\label{prop:seq_bound}
 Let $E$ be a Fr\'{e}chet space, $(B_{n})$ fix the topology in $E$, $Y$ a Fr\'{e}chet--Schwartz space and 
 $X\subset Y_{b}'(=Y_{\kappa}')$ a dense subspace. Set $G:=\operatorname{span}(\bigcup_{n\in\N} B_{n})$ and let
 $\mathsf{A}\colon X\to E$ be a linear map which is 
 $\sigma(X,Y)$-$\sigma(E,G)$-continuous and satisfies that $\mathsf{A}^{t}(B_{n})$ is bounded in $Y$ 
 for each $n\in\N$. 
 Then $\mathsf{A}$ has a (unique) extension $\widehat{\mathsf{A}}\in Y\varepsilon E$.
\end{prop}

Next, we improve \cite[Theorem 1 ii), p.\ 501]{F/J}.

\begin{thm}\label{thm:ext_FS_set_uni_seq_bounded}
Let $E$ be a Fr\'{e}chet space, $(B_{n})$ fix the topology in $E$ and $G:=\operatorname{span}(\bigcup_{n\in\N} B_{n})$,
$(T^{E}_{m},T^{\K}_{m})_{m\in M}$ a strong, consistent generator for $(\mathcal{FV},E)$, 
$\FV$ a Fr\'{e}chet--Schwartz space and $U$ a set of uniqueness for $(T^{\K}_{m},\mathcal{FV})_{m\in M}$. 
Then the restriction map $R_{U,G}\colon S(\FV\varepsilon E)\to \mathcal{FV}_{G}(U,E)_{sb}$ is surjective.
\end{thm}
\begin{proof}
Let $f\in \mathcal{FV}_{G}(U,E)_{sb}$. We set $X:=\operatorname{span}\{T^{\K}_{m,x}\;|\;(m,x)\in U\}$ and 
$Y:=\FV$. Let $\mathsf{A}\colon X\to E$ be the linear map determined by
$\mathsf{A}(T^{\K}_{m,x}):=f(m,x)$ which is well-defined since $G$ is $\sigma(E',E)$-dense. 
From 
\[
e'(\mathsf{A}(T^{\K}_{m,x}))=(e'\circ f)(m,x)=T^{\K}_{m,x}(f_{e'})
\]
for every $e'\in G$ and $(m,x)\in U$ it follows that $\mathsf{A}$ is $\sigma(X,Y)$-$\sigma(E,G)$-continuous and
\[
\sup_{e'\in B_{n}}|\mathsf{A}^{t}(e')|_{j,k}=\sup_{e'\in B_{n}}|f_{e'}|_{j,k}<\infty
\]
for all $j\in J$, $k\in M$ and $n\in\N$. Due to \prettyref{prop:seq_bound} there is 
an extension $\widehat{\mathsf{A}}\in\FV\varepsilon E$ of $\mathsf{A}$. 
We set $F:=S(\widehat{\mathsf{A}})$ and get for all $(m,x)\in U$ that
\[
T^{E}_{m}(F)(x)=T^{E}_{m}S(\widehat{\mathsf{A}})(x)=\widehat{\mathsf{A}}(T^{\K}_{m,x})=f(m,x)
\]
by consistency, which means $R_{U,G}(F)=f$.
\end{proof}

\begin{cor}
Let $E$ be a Fr\'{e}chet space, $(B_{n})$ fix the topology in $E$ and $G:=\operatorname{span}(\bigcup_{n\in\N} B_{n})$. 
Let $\mathcal{V}:=(\nu_{j})_{j\in\N}$ be an increasing family of weights which is locally bounded away from zero 
on an open set $\Omega\subset\R^{d}$, $P(\partial)^{\K}$ a hypoelliptic linear partial differential operator, 
$\mathcal{CV}_{P(\partial)}(\Omega)$ a Schwartz space and 
$U\subset\Omega$ a set of uniqueness for $(\id_{\K^{\Omega}},\mathcal{CV}_{P(\partial)})$. 
If $f\colon U\to E$ is a function such that $e'\circ f$ admits an extension 
$f_{e'}\in\mathcal{CV}_{P(\partial)}(\Omega)$ for each $e'\in G$ 
and $\{f_{e'}\;|\;e'\in B_{n}\}$ is bounded in $\mathcal{CV}_{P(\partial)}(\Omega)$ for each $n\in\N$, 
then there is a unique extension $F\in\mathcal{CV}_{P(\partial)}(\Omega,E)$ of $f$.
\end{cor}
\begin{proof}
$\mathcal{CV}_{P(\partial)}(\Omega)$ is a Fr\'echet--Schwartz space and $(\id_{E^{\Omega}},\id_{\K^{\Omega}})$ a 
strong, consistent generator for $(\mathcal{CV}_{P(\partial)},E)$ by \prettyref{prop:frechet_bierstedt}
and the proof of \prettyref{ex:subspace_bierstedt} b). 
Now, \prettyref{thm:ext_FS_set_uni_seq_bounded} and \prettyref{prop:injectivity} prove our statement.
\end{proof}

We already mentioned examples of families of weights $\mathcal{V}$ such that $\mathcal{CV}_{P(\partial)}(\R^{d})$ 
is a nuclear Fr\'echet space and sets of uniqueness for $(\id_{\K^{\R^{d}}},\mathcal{CV}_{P(\partial)})$ 
in \prettyref{rem:ex_NF_cond_weights} and \prettyref{rem:ex_set_uni_AP} and 
if $P(\partial)=\overline{\partial}$ or $P(\partial)=\Delta$. Further sets of uniqueness are given in \prettyref{rem:ex_fix_top_AP}. 
If $E$ is a Banach space, then an almost norming set fixes the topology and examples can be found 
via \prettyref{rem:ex_almost_norming}.

\addtocontents{toc}{\SkipTocEntry}
\subsection*{\texorpdfstring{$\Fv$}{Fv(Omega)} a Banach space and \texorpdfstring{$E$}{E} a Fr\'echet space}

Let $E$ be a Fr\'{e}chet space, $(B_{n})$ fix the topology in $E$ 
and recall the assumptions of \prettyref{rem:R_well-defined_Banach}.
Let $(T^{E},T^{\K})$ be a strong, consistent family for $(F,E)$ and a generator for $(\mathcal{F}\nu,E)$. 
Let $\f$ and $\fe$ be $\varepsilon$-into-compatible and the inclusion $\Fv\hookrightarrow\f$ continuous.
Consider a set of uniqueness $U$ for $(T^{\K},\mathcal{F}\nu)$ and $G:=\operatorname{span}(\bigcup_{n\in\N} B_{n})\subset E'$.
For $u\in\f\varepsilon E$ such that $u(B_{\Fv}^{\circ \f'})$ is bounded in $E$
we have $R_{U,G}(f)\in\mathcal{F}\nu_{G}(U,E)$ with $f:=S(u)\in\Feps$ by \eqref{eq:well_def_B_unique}. 
We note that
\[
 \sup_{e'\in B_{n}}|f_{e'}|_{\Fv}
=\sup_{e'\in B_{n}}\sup_{x\in\omega}|e'(T^{E}(f)(x)\nu(x))|
=\sup_{e'\in B_{n}}\sup_{y\in N_{\omega}(f)}|e'(y)|
\]
with the bounded set $N_{\omega}(f):=\{T^{E}(f)(x)\nu(x)\;|\;x\in\omega\}\subset E$, 
implying $R_{U,G}(f)\in\mathcal{FV}_{G}(U,E)_{sb}$. 
Thus the injective linear map
\[
R_{U,G}\colon \Feps \to \mathcal{F}\nu_{G}(U,E)_{sb},\;f\mapsto (T^{E}(f)(x))_{x\in U},
\]
is well-defined. 

\begin{que}
Let the assumptions of \prettyref{rem:R_well-defined_Banach} be fulfilled, 
$E$ be a Fr\'{e}chet space, $(B_{n})$ fix the topology in $E$ and $G:=\operatorname{span}(\bigcup_{n\in\N} B_{n})$. 
When is the injective restriction map 
\[
R_{U,G}\colon \Feps \to \mathcal{F}\nu_{G}(U,E)_{sb},\;f\mapsto (T^{E}(f)(x))_{x\in U},
\]
surjective?
\end{que}

Now, we can generalise \cite[Corollary 2.4, p.\ 692]{F/J/W} and \cite[Theorem 11, p.\ 5]{jorda2013}.

\begin{cor}\label{cor:ext_B_unique_seq_bound}
Let $E$ be a Fr\'{e}chet space, $(B_{n})$ fix the topology in $E$, 
set $G:=\operatorname{span}(\bigcup_{n\in\N} B_{n})$ and 
let $\f$ and $\fe$ be $\varepsilon$-into-compatible. 
Let $(T^{E},T^{\K})$ be a generator for $(\mathcal{F}\nu,E)$ and a strong, consistent family for $(F,E)$,
$\Fv$ a Banach space whose closed unit ball $B_{\Fv}$ is a compact subset of $\f$ 
and $U$ a set of uniqueness for $(T^{\K},\mathcal{F}\nu)$.
Then the restriction map 
\[
 R_{U,G}\colon\Feps \to \mathcal{F}\nu_{G}(U,E)_{sb} 
\]
is surjective.
\end{cor}
\begin{proof}
Let $f\in\mathcal{F}\nu_{G}(U,E)_{sb}$. Then $\{f_{e'}\;|\;e'\in B_{n}\}$ is bounded in $\Fv$ for each $n\in\N$. 
We deduce for each $n\in\N$, $(a_{k})_{k\in\N}\in\ell^{1}$ and $(e_{k}')_{k\in\N}\subset B_{n}$ that 
$(\sum_{k\in\N}a_{k}e_{k}')\circ f$ admits the extension $\sum_{k\in\N}a_{k}f_{e_{k}'}$ in $\Fv$. 
Due to \cite[Proposition 7, p.\ 503]{F/J} the LB-space 
$E'((B_{n})_{n\in\N}):=\lim\limits_{\substack{\longleftarrow\\n\in\N}}E'(B_{n})$, where 
\[
E'(B_{n}):=\{\sum_{k\in\N}a_{k}e_{k}'\;|\;(a_{k})_{k\in\N}\in\ell^{1},\,(e_{k}')_{k\in\N}\subset B_{n}\}
\]
is endowed with its Banach space topology for $n\in\N$, determines boundedness in $E$. 
Hence we conclude that $f\in \mathcal{F}\nu_{E'((B_{n})_{n\in\N})}(U,E)$, which yields that there 
is $u\in\f\varepsilon E$ with bounded $u(B_{\Fv}^{\circ \f'})\subset E$ 
such that $R_{U,G}(S(u))=f$ by \prettyref{thm:ext_B_unique}.
\end{proof}

As an application we directly obtain the following two corollaries of \prettyref{cor:ext_B_unique_seq_bound} 
since its assumptions are fulfilled by the proof of 
\prettyref{cor:hypo_weighted_ext_unique} and \prettyref{cor:Bloch_ext_unique}, respectively.

\begin{cor}\label{cor:hypo_weighted_ext_unique_seq_bound}
Let $E$ be a Fr\'{e}chet space, $(B_{n})$ fix the topology in $E$ and $G:=\operatorname{span}(\bigcup_{n\in\N} B_{n})$,
$\Omega\subset\R^{d}$ open, $P(\partial)^{\K}$ a hypoelliptic linear partial differential operator, 
$\nu\colon\Omega\to(0,\infty)$ continuous and $U$ a set of uniqueness 
for $(\id_{\K^{\Omega}},\mathcal{C}\nu_{P(\partial)})$.
If $f\colon U\to E$ is a function such that $e'\circ f$ admits an extension 
$f_{e'}\in\mathcal{C}\nu_{P(\partial)}(\Omega)$ for each $e'\in G$ and $\{f_{e'}\;|\;e'\in B_{n}\}$ is 
bounded in $\mathcal{C}\nu_{P(\partial)}(\Omega)$ for each $n\in\N$, 
then there exists a unique extension $F\in\mathcal{C}\nu_{P(\partial)}(\Omega,E)$ of $f$.
\end{cor}

\begin{cor}\label{cor:Bloch_ext_unique_seq_bound}
Let $E$ be a Fr\'{e}chet space, $(B_{n})$ fix the topology in $E$ and $G:=\operatorname{span}(\bigcup_{n\in\N} B_{n})$, 
$\nu\colon\D\to(0,\infty)$ continuous and $U_{\ast}\subset\D$ have an accumulation point in $\D$. 
If $f\colon \{0\}\cup(\{1\}\times U_{\ast})\to E$ is a function such that there is 
$f_{e'}\in\mathcal{B}\nu(\D)$ for each $e'\in G$ with $f_{e'}(0)=e'(f(0))$ and $f_{e'}'(z)=e'(f(1,z))$ 
for all $z\in U_{\ast}$ and $\{f_{e'}\;|\;e'\in B_{n}\}$ is bounded in $\mathcal{B}\nu(\D)$ for each $n\in\N$, 
then there exists a unique $F\in\mathcal{B}\nu(\D,E)$ with $F(0)=f(0)$ and $(\partial_{\C}^{1})^{E}F(z)=f(1,z)$ 
for all $z\in U_{\ast}$.
\end{cor}
\subsection{Extension from thick sets}\label{sub:thick}
In order to obtain an affirmative answer to \prettyref{que:surj_restr_set_unique} 
for general separating subspaces of $E'$ we have to restrict to the spaces $\FV$ from \prettyref{def:weighted_space}
and a certain class of sets of uniqueness.

\begin{defn}[{fix the topology}]\label{def:fix_top_1}
Let $\FV$ be a $\dom$-space. We say that $U\subset\bigcup_{m\in M}(\{m\}\times\omega_{m})$ \emph{\gls{fix_topo}} 
in $\FV$ if for every $j\in J$ and $m\in M$ there are $i\in J$, $k\in M$ and $C>0$ such that 
\[
|f|_{j,m}\leq C \sup_{\substack{x\in\omega_{k}\\(k,x)\in U}}|T^{\K}_{k}(f)(x)|\nu_{i,k}(x),\quad
f\in \FV .
\]
\end{defn}

In particular, $U$ is a set of uniqueness if it fixes the topology. The present definition of fixing 
the topology is a generalisation of \cite[Definition 13, p.\ 234]{B/F/J}. 
Sets that fix the topology appear under several different notions. 
Rubel and Shields call them dominating in \cite[4.10 Definition, p.\ 254]{rubelshields1966} in the context 
of bounded holomorphic functions. 
In the context of the space of holomorphic functions with the topology of compact convergence studied 
by Grosse-Erdmann \cite[p.\ 401]{grosse-erdmann2004} they are said to determine locally uniform convergence. 
Ehrenpreis \cite[p.\ 3,4,13]{ehrenpreis1970} (cf.\ \cite[Definition 3.2, p.\ 166]{schneider1974}) 
refers to them as sufficient sets when he considers inductive limits of weighted spaces of entire resp.\ 
holomorphic functions, including the case of Banach spaces. 
In the case of Banach spaces sufficient sets coincide with weakly sufficient sets defined 
by Schneider \cite[Definition 2.1, p.\ 163]{schneider1974} (see e.g.\ \cite[\S7, 1), p.\ 547]{korobenik1987}) 
and these notions are extended beyond spaces of holomorphic functions 
by Korobe{\u{\i}}nik \cite[p.\ 531]{korobenik1987}.
Seip \cite[p.\ 93]{seip1992b} uses the term sampling sets in the context of weighted Banach spaces 
of holomorphic functions whereas Beurling uses the term balayage in \cite[p.\ 341]{beurling1989} 
and \cite[Definition, p.\ 343]{beurling1989}.
Leibowitz \cite[Exercise 4.1.4, p.\ 53]{leibowitz1970}, Stout \cite[7.1 Definition, p.\ 36]{stout1971} 
and Globevnik \cite[p.\ 291--292]{globevnik1979} 
call them boundaries in the context of subalgebras of the algebra $\mathcal{C}(\Omega,\C)$ 
of complex-valued continuous functions on a compact Hausdorff space $\Omega$ with sup-norm. 
Fixing the topology is also connected to the notion of frames used by Bonet et al.\ in \cite{bonet2017}. 
Let us set
\begin{equation}\label{eq:frame} 
\ell\mathcal{V}(U,E):=\{f\colon U\to E\;|\;\forall\;j\in J,m\in M,\alpha\in\mathfrak{A}:\;\|f\|_{j,m,\alpha}<\infty\}
\end{equation}
with
\[
\|f\|_{j,m,\alpha}:=\sup_{\substack{x\in\omega_{m}\\(m,x)\in U}}p_{\alpha}(f(m,x))\nu_{j,m}(x)
\]
for an lcHs $E$ and a set $U$ which fixes the topology in $\FV$. If $M$ is a singleton, 
$\omega_{m}=\Omega=U$, then $\ell\mathcal{V}(U,E)$ coincides with the space defined right
above \prettyref{ex:sequence_vanish_infty}.
If $U$ is countable, then the inclusion $\ell\mathcal{V}(U)\hookrightarrow\K^{U}$ continuous 
where $\K^{U}$ is equipped with the topology of pointwise convergence and $\ell\mathcal{V}(U)$ contains 
the space of sequences (on $U$) with compact support as a linear subspace, 
then $(T^{\K}_{k,x})_{(k,x)\in U}$ is an $\ell\mathcal{V}(U)$-frame in the sense of 
\cite[Definition 2.1, p.\ 3]{bonet2017}.

\begin{defn}[{\gls{lb_rest_space}}]
Let $\FV$ be a $\dom$-space, $U$ fix the topology in $\FV$ and $G\subset E'$ a separating subspace. We set 
\[
N_{U,i,k}(f):=\{f(k,x)\nu_{i,k}(x)\;|\;x\in\omega_{k},\,(k,x)\in U\}
\]
for $i\in J$, $k\in M$ and $f\in\mathcal{FV}_{G}(U,E)$ and 
\begin{align*}
   \gls{FV_GUE_lb}
:=&\{f\in\mathcal{FV}_{G}(U,E)\;|\;\forall\;i\in J,\,k\in M:\;N_{U,i,k}(f)\;\text{bounded in}\; E\}\\
 =&\mathcal{FV}_{G}(U,E)\cap\ell\mathcal{V}(U,E).
\end{align*}
\end{defn}

Consider a set $U$ which fixes the topology in $\FV$, a separating subspace $G\subset E'$ and 
a strong, consistent family $(T^{E}_{m},T^{\K}_{m})_{m\in M}$ for $(\mathcal{FV},E)$.
For $u\in \FV\varepsilon E$ set $f:=S(u)\in\FVE$ by \prettyref{thm:linearisation}. Then 
we have $R_{U,G}(f)\in \mathcal{FV}_{G}(U,E)$ with $f:=S(u)$ 
by \prettyref{rem:R_well-defined} and for $i\in J$ and $k\in M$
\[
\sup_{y\in N_{U,i,k}(R_{U,G}(f))}p_{\alpha}(y)
=\sup_{\substack{x\in\omega_{k}\\(k,x)\in U}}p_{\alpha}(T^{E}_{k}(f)(x))\nu_{i,k}(x)
\leq |f|_{i,k,\alpha}<\infty
\]
for all $\alpha\in\mathfrak{A}$, implying the boundedness of $N_{U,i,k}(R_{U,G}(f))$ in $E$. 
Thus $R_{U,G}(f)\in\mathcal{FV}_{G}(U,E)_{lb}$ and the injective linear map
\[
R_{U,G}\colon S(\FV\varepsilon E)\to \mathcal{FV}_{G}(U,E)_{lb},\;f\mapsto (T^{E}_{m}(f)(x))_{(m,x)\in U},
\]
is well-defined. 

\begin{que}
Let $G\subset E'$ be a separating subspace, 
$(T^{E}_{m},T^{\K}_{m})_{m\in M}$ a strong, consistent generator for $(\mathcal{FV},E)$
and $U$ fix the topology in $\FV$.
When is the injective restriction map 
\[
R_{U,G}\colon S(\FV\varepsilon E)\to \mathcal{FV}_{G}(U,E)_{lb},\;f\mapsto (T^{E}_{m}(f)(x))_{(m,x)\in U},
\]
surjective?
\end{que}

If $G\subset E'$ determines boundedness and $U$ fixes the topology in $\FV$, then the preceding question 
and \prettyref{que:surj_restr_set_unique} coincide.

\begin{rem}\label{rem:rest_spaces_coincide}
Let $G\subset E'$ determine boundedness, 
$(T^{E}_{m},T^{\K}_{m})_{m\in M}$ a strong, consistent generator for $(\mathcal{FV},E)$
and $U$ fix the topology in $\FV$. Then 
\[
\mathcal{FV}_{G}(U,E)_{lb}=\mathcal{FV}_{G}(U,E).
\]
\end{rem}
\begin{proof}
We only need to show that the inclusion `$\supset$' holds. Let $f\in\mathcal{FV}_{G}(U,E)$. Then there is $f_{e'}\in\FV$ with 
$T^{\K}_{m}(f_{e'})(x)=(e'\circ f)(m,x)$ for all $(m,x)\in U$ and 
\[
\sup_{y\in N_{U,i,k}(f)}|e'(y)|=\sup_{\substack{x\in\omega_{k}\\(k,x)\in U}}|(e'\circ f)(k,x)|\nu_{i,k}(x)\leq |f_{e'}|_{i,k}<\infty
\]
for each $e'\in G$, $i\in J$ and $k\in M$. Since $G\subset E'$ determines boundedness, this means that $N_{U,i,k}(f)$ is bounded in $E$ 
and hence $f\in\mathcal{FV}_{G}(U,E)_{lb}$.
\end{proof}

\addtocontents{toc}{\SkipTocEntry}
\subsection*{\texorpdfstring{$\FV$}{FV(Omega)} arbitrary and \texorpdfstring{$E$}{E} a semi-Montel space}

\begin{defn}[{generalised Schwartz space}]\label{def:general_schwartz}
We call an lcHs $E$ a \emph{\gls{gen_Schwartz_space}} if every bounded set in $E$ is already precompact.
\end{defn}

In particular, semi-Montel spaces and Schwartz spaces are generalised Schwartz spaces 
by \cite[10.4.3 Corollary, p.\ 202]{Jarchow}. Conversely, a generalised Schwartz space is a Schwartz space 
if it is quasi-normable by \cite[10.7.3 Corollary, p.\ 215]{Jarchow}. Moreover, 
looking at the proof of \prettyref{lem:FVE_rel_comp} b), we see that this lemma not only holds for 
semi-Montel or Schwartz spaces but for all generalised Schwartz spaces. 

\begin{prop}\label{prop:ext_E_semi_M}
Let $E$ be an lcHs,
$\FV$ a $\dom$-space and $U$ fix the topology in $\FV$.
Then $\mathscr{R}_{f}\in L(E_{b}',\FV)$ 
and $\mathscr{R}_{f}(B_{\alpha}^{\circ})$ is bounded in $\FV$  
for every $f\in\mathcal{FV}_{E'}(U,E)_{lb}$ and $\alpha\in\mathfrak{A}$ 
where $B_{\alpha}:=\{x\in E\;|\;p_{\alpha}(x)<1\}$ and $\mathscr{R}_{f}$ 
is the map from \prettyref{rem:R_f}. In addition, if $E$ is a generalised Schwartz space, 
then $\mathscr{R}_{f}\in L(E_{\gamma}',\FV)$ 
and $\mathscr{R}_{f}(B_{\alpha}^{\circ})$ is relatively compact in $\FV$.
\end{prop}
\begin{proof}
Let $f\in \mathcal{FV}_{E'}(U,E)_{lb}$, $j\in J$ and $m\in M$. Then there are $i\in J$, $k\in M$ and $C>0$ such that 
for every $e'\in E'$
\begin{align*}
|\mathscr{R}_{f}(e')|_{j,m}&=|f_{e'}|_{j,m}
\leq C\sup_{\substack{x\in\omega_{k}\\(k,x)\in U}}|T^{\K}_{k}(f_{e'})(x)|\nu_{i,k}(x)\\
&=C \sup_{\substack{x\in\omega_{k}\\(k,x)\in U}}|(e'\circ f)(k,x)|\nu_{i,k}(x)
=C\sup_{y\in N_{U,i,k}(f)}|e'(y)|,
\end{align*}
which proves the first part because $N_{U,i,k}(f)$ is bounded in $E$. Let us consider the second part.
The bounded set $N_{U,i,k}(f)$ is already precompact in $E$ because $E$ is a generalised Schwartz space. 
Therefore we have $\mathscr{R}_{f}\in L(E_{\gamma}',\FV)$. The polar $B_{\alpha}^{\circ}$ is relatively compact 
in $E_{\gamma}'$ for every $\alpha\in\mathfrak{A}$ by the Alao\u{g}lu--Bourbaki theorem 
and thus $\mathscr{R}_{f}(B_{\alpha}^{\circ})$ in $\FV$ as well.
\end{proof}

\begin{thm}\label{thm:fix_topo_E_semi_M}
Let $E$ be a semi-Montel space, 
$(T^{E}_{m},T^{\K}_{m})_{m\in M}$ a strong, consistent generator for $(\mathcal{FV},E)$
and $U$ fix the topology in $\FV$. 
Then the restriction map $R_{U,E'}\colon S(\FV\varepsilon E)\to \mathcal{FV}_{E'}(U,E)_{lb}$ is surjective.
\end{thm}
\begin{proof}
Let $f\in \mathcal{FV}_{E'}(\Omega,E)_{lb}$ and $e'\in E'$. 
For every $f'\in \FV'$ there are $j\in J$, $m\in M$ and $C_{0}>0$ with
\[
|\mathscr{R}_{f}^{t}(f')(e')|=|f'(f_{e'})|\leq C_{0} |f_{e'}|_{j,m}.
\]
By the proof of \prettyref{prop:ext_E_semi_M} there are $i\in J$, $k\in M$ and $C>0$ such that
\[
     |\mathscr{R}_{f}^{t}(f')(e')|
\leq C_{0}C\sup_{y\in N_{U,i,k}(f)}|e'(y)|
\leq C_{0}C\sup_{y\in \oacx(N_{U,i,k}(f))}|e'(y)|.
\]
The set $\oacx(N_{U,i,k}(f))$ is absolutely convex and compact by \cite[6.2.1 Proposition, p.\ 103]{Jarchow} 
and \cite[6.7.1 Proposition, p.\ 112]{Jarchow} because $E$ is a semi-Montel space. 
Therefore $\mathscr{R}_{f}^{t}(f')\in (E'_{\kappa})'=\mathcal{J}(E)$ by the Mackey--Arens theorem.
As in \prettyref{thm:ext_F_semi_M} we obtain $\mathcal{J}^{-1}\circ\mathscr{R}_{f}^{t}\in \FV\varepsilon E$ 
by \eqref{eq1:ext_F_semi_M}, \eqref{eq2:ext_F_semi_M} and \prettyref{prop:ext_E_semi_M}. 
Setting $F:=S(\mathcal{J}^{-1}\circ\mathscr{R}_{f}^{t})$, 
we conclude $T^{E}_{m}(F)(x)=f(m,x)$ for all $(m,x)\in U$ by \eqref{eq3:ext_F_semi_M} and so $R_{U,E'}(F)=f$.
\end{proof}

\begin{rem}
Let $E$ be a Fr\'{e}chet space with increasing system of seminorms 
$(p_{\alpha_{n}})_{n\in\N}$, $B_{n}:=B_{\alpha_{n}}^{\circ}$ where 
$B_{\alpha_{n}}:=\{x\in E\;|\;p_{\alpha_{n}}(x)<1\}$, 
$(T^{E}_{m},T^{\K}_{m})_{m\in M}$ a strong, consistent generator for $(\mathcal{FV},E)$
and $U$ a set of uniqueness for $(T^{\K}_{m},\mathcal{FV})_{m\in M}$. 
If $U$ fixes the topology of $\FV$,
then $\mathcal{FV}_{E'}(U,E)_{sb}=\mathcal{FV}_{E'}(U,E)$ by \prettyref{rem:rest_spaces_coincide} and \prettyref{prop:ext_E_semi_M}. Hence \prettyref{thm:fix_topo_E_semi_M} answers 
\prettyref{que:surjective_sb_rest_space} if $E$ is a Fr\'echet--Montel space. 
\end{rem}

Our first application of \prettyref{thm:fix_topo_E_semi_M} concerns the space $\mathcal{C}_{bu}(\Omega,E)$
of bounded uniformly continuous functions from a metric space $\Omega$ to an lcHs $E$ 
from \prettyref{ex:uniformly_cont}.

\begin{cor}\label{cor:uniformly_cont}
Let $\Omega$ be a metric space, $U\subset\Omega$ a dense subset and $E$ a semi-Montel space.
If $f\colon U\to E$ is a function such that $e'\circ f$ admits an extension $f_{e'}\in\mathcal{C}_{bu}(\Omega)$ 
for each $e'\in E'$, then there is a unique extension $F\in\mathcal{C}_{bu}(\Omega,E)$ of $f$. In particular, 
\[
\mathcal{C}_{bu}(\Omega,E)=\{f\colon\Omega\to E\;|\;\forall\;e'\in E':\;e'\circ f\in\mathcal{C}_{bu}(\Omega)\}.
\]
\end{cor}
\begin{proof}
$(\id_{E^{\Omega}},\id_{\K^{\Omega}})$ is a strong, consistent generator for $(\mathcal{C}_{bu},E)$ and 
we have $\mathcal{C}_{bu}(\Omega)\varepsilon E\cong \mathcal{C}_{bu}(\Omega,E)$ via $S$ by \prettyref{ex:uniformly_cont}.
Due to \prettyref{thm:fix_topo_E_semi_M}, \prettyref{prop:injectivity} and \prettyref{rem:rest_spaces_coincide} 
with $G=E'$ the extension $F$ exists and is unique because the dense set $U\subset\Omega$ fixes the topology 
in $\mathcal{C}_{bu}(\Omega)$. The rest follows from \prettyref{prop:weak_strong_principle}. 
\end{proof}

Next, we consider the space $\mathcal{A}(\overline{\Omega},E)$ of continuous functions
from $\overline{\Omega}$ to an lcHs $E$ over $\C$ which are holomorphic on an open and bounded set 
$\Omega\subset\C$ from \prettyref{ex:disc_algebra}.

\begin{cor}\label{cor:disc_algebra}
Let $\Omega\subset\C$ be open and bounded, $U\subset\overline{\Omega}$ fix the topology 
in $\mathcal{A}(\overline{\Omega})$ and $E$ a semi-Montel space over $\C$. 
If $f\colon U\to E$ is a function such that $e'\circ f$ admits an extension 
$f_{e'}\in\mathcal{A}(\overline{\Omega})$ for each $e'\in E'$, 
then there is a unique extension $F\in\mathcal{A}(\overline{\Omega},E)$ of $f$. In particular, 
\[
 \mathcal{A}(\overline{\Omega},E)
=\{f\colon\overline{\Omega}\to E\;|\;\forall\;e'\in E':\;e'\circ f\in\mathcal{A}(\overline{\Omega})\}.
\]
\end{cor}
\begin{proof}
$(\id_{E^{\overline{\Omega}}},\id_{\C^{\overline{\Omega}}})$ is a strong, consistent generator for $(\mathcal{A},E)$ and 
$\mathcal{A}(\overline{\Omega})\varepsilon E\cong \mathcal{A}(\overline{\Omega},E)$ via $S$ 
by \prettyref{ex:disc_algebra}.
Due to \prettyref{thm:fix_topo_E_semi_M}, \prettyref{prop:injectivity} and \prettyref{rem:rest_spaces_coincide} 
with $G=E'$ the extension $F$ exists and is unique. The remaining part follows 
from \prettyref{prop:weak_strong_principle}. 
\end{proof}

If $\Omega\subset\C$ is connected, then the boundary $\partial\Omega$ of $\Omega$ fixes the topology 
in $\mathcal{A}(\overline{\Omega})$ by the maximum principle. 
If $\Omega=\D$, then $\partial\D$ is the intersection of all sets 
that fix the topology in $\mathcal{A}(\overline{\D})$ by \cite[7.7 Example, p.\ 39]{stout1971}. 

If $E$ is a generalised Schwartz space which is not a semi-Montel space, 
we do not know whether the extension results in \prettyref{cor:uniformly_cont} and \prettyref{cor:disc_algebra}
hold but we still have a weak-strong principle due to the following observation which is based on 
\cite[Chap.\ 3, \S9, Proposition 2, p.\ 231]{horvath} with $\sigma(E,E')$ replaced by $\sigma(E,G)$. 

\begin{prop}\label{prop:general_Schwartz}
If
\begin{enumerate}
\item[(i)] $E$ is a semi-Montel space and $G\subset E'$ a separating subspace, or 
\item[(ii)] $E$ is a generalised Schwartz space and $G\subset \widehat{E}'$ a separating subspace, 
i.e.\ separates the points of the completion $\widehat{E}$,
\end{enumerate}
then the initial topology of $E$ and the topology $\sigma(E,G)$ coincide on the bounded sets of $E$.
\end{prop}
\begin{proof}
(i) Let $B\subset E$ be a bounded set. If $E$ is a semi-Montel space, then the closure $\overline{B}$ is compact 
in $E$. The topology induced by $\sigma(E,G)$ on $\overline{B}$ is Hausdorff and weaker than 
the initial topology induced by $E$. Thus the two topologies coincide on $\overline{B}$ and so on $B$ by 
the remarks above \cite[Chap.\ 3, \S9, Proposition 2, p.\ 231]{horvath}.

(ii) Let $B\subset E$ be a bounded set. If $E$ is a generalised Schwartz space, then $B$ is precompact in $E$ 
and relatively compact in the completion $\widehat{E}$ by \cite[3.5.1 Theorem, p.\ 64]{Jarchow}. 
Hence the closure $\overline{B}$ is compact in $\widehat{E}$. 
The topology induced by $\sigma(\widehat{E},G)$ on $\overline{B}$ is Hausdorff 
and weaker than the initial topology induced by $\widehat{E}$, implying
that the two topologies coincide on $\overline{B}$ as in part (i).
This yields that $\sigma(E,G)$ and the initial topology of $E$ coincide on $B$ because 
$\sigma(E,G)=\sigma(\widehat{E},G)$ on $B$ and the initial topologies of $E$ and $\widehat{E}$ 
coincide on $B$ as well.
\end{proof}

Concerning (ii), we note that a separating subspace $G\subset E'$ of $E$ need not separate the points of $\widehat{E}$ 
by \cite[5.4 Example, p.\ 36]{grosse-erdmann1992} (even though $E'=\widehat{E}'$ by \cite[3.4.2 Theorem, p.\ 61--62]{Jarchow}). 
Next, we apply \prettyref{prop:general_Schwartz} to the space $\mathcal{A}(\overline{\Omega},E)$.

\begin{rem}\label{rem:bierstedt_weak_strong}
Let $E$ be an lcHs over $\C$ and $\Omega\subset\C$ open and bounded. If
\begin{enumerate}
\item[(i)] $E$ is a semi-Montel space and $G\subset E'$ determines boundedness, or 
\item[(ii)] $E$ is a generalised Schwartz space and $G\subset\widehat{E}'$ a separating subspace which 
determines boundedness in $E$, 
\end{enumerate}
then 
\[
 \mathcal{A}(\overline{\Omega},E)
=\{f\colon\overline{\Omega}\to E\;|\;\forall\;e'\in G:\;e'\circ f\in\mathcal{A}(\overline{\Omega})\}.
\]
Indeed, let us denote the right-hand side by $\mathcal{A}(\overline{\Omega},E)_{\sigma}$ 
and set $E_{\sigma}:=(E,\sigma(E,G))$. 
Then $\mathcal{A}(\overline{\Omega},E)_{\sigma}=\mathcal{A}(\overline{\Omega},E_{\sigma})$ 
and $f(\overline{\Omega})$ is bounded for every $f\in\mathcal{A}(\overline{\Omega},E)_{\sigma}$ 
as $G$ determines boundedness in $E$. 
The initial topology of $E$ and $\sigma(E,G)$ coincide on the bounded range $f(\overline{\Omega})$ 
of $f\in\mathcal{A}(\overline{\Omega},E)_{\sigma}$ by \prettyref{prop:general_Schwartz}. Hence we deduce that
\[
\mathcal{A}(\overline{\Omega},E)_{\sigma}=\mathcal{A}(\overline{\Omega},E_{\sigma})=\mathcal{A}(\overline{\Omega},E).
\]
\end{rem}

In this way Bierstedt proves his weak-strong principles for weighted continuous functions 
in \cite[2.10 Lemma, p.\ 140]{B2} with $G=E'=\widehat{E}'$.

\addtocontents{toc}{\SkipTocEntry}
\subsection*{\texorpdfstring{$\FV$}{FV(Omega)} a Fr\'{e}chet--Schwartz space and \texorpdfstring{$E$}{E} locally complete}

\begin{defn}[{chain-structured}]\label{def:chain_structured} 
Let $\FV$ be a $\dom$-space. We say that $U\subset\bigcup_{m\in \N}(\{m\}\times\omega_{m})$ is \emph{chain-structured} if 
\begin{enumerate}
\item[(i)] $(k,x)\in U\;\;\Rightarrow\;\;\forall\;m\in \N,\,m\geq k:\;(m,x)\in U$,
\item[(ii)] $\forall\;(k,x)\in U,\,m\in \N,\,m\geq k,\,f\in\FV:\;T^{\K}_{k}(f)(x)=T^{\K}_{m}(f)(x)$.
\end{enumerate}
\end{defn}

\begin{rem} 
Let $\Omega\subset\R^{d}$ be open and $\mathcal{V}^{\infty}$ a directed family of weights. 
Concerning the operators $(T^{\K}_{m})_{m\in\N_{0}}$ 
of $\mathcal{CV}^{\infty}(\Omega)$ from \prettyref{ex:weighted_smooth_functions} a)
where $\omega_{m}:=\{\beta\in\N_{0}^{d}\;|\;|\beta|\leq m\}\times\Omega$ resp.\ 
$\omega_{m}:=\N_{0}^{d}\times\Omega$, we have for all $k\in\N_{0}$ and $f\in\mathcal{CV}^{\infty}(\Omega)$ 
that
\[
T^{\K}_{k}(f)(\beta,x)=\partial^{\beta}f(x)=T^{\K}_{m}(f)(\beta,x),
\quad \beta\in\N_{0}^{d},\;|\beta|\leq k,\;x\in\Omega,
\]
for all $m\in\N_{0}$, $m\geq k$. Hence condition (ii) of \prettyref{def:chain_structured} is fulfilled for any 
$U\subset\bigcup_{m\in \N_{0}}(\{m\}\times\omega_{m})$ in this case. 
Condition (i) says that once a `link' $(k,\beta,x)$ belongs to $U$ for some 
order $k$, then the `link' $(m,\beta,x)$ belongs to $U$ for any higher order $m$ as well. 
\end{rem}

\begin{defn}[{diagonally dominated, increasing}] 
We say that a family $\mathcal{V}:=(\nu_{j,m})_{j,m\in\N}$ of weights on $\Omega$ 
is \emph{diagonally dominated and increasing} if $\omega_{m}\subset \omega_{m+1}$ for all $m\in \N$ and 
$\nu_{j,m}\leq \nu_{\max(j,m),\max(j,m)}$ on $\omega_{\min(j,m)}$ for all $j,m\in\N$ as well as 
$\nu_{j,j}\leq \nu_{j+1,j+1}$ on $\omega_{j}$ for all $j\in\N$. 
\end{defn}

\begin{rem}\label{rem:fix_top_1=2}
Let $\FV$ be a $\dom$-space, $U\subset\bigcup_{m\in \N}(\{m\}\times\omega_{m})$ chain-structured, 
$G\subset E'$ a separating subspace and $\mathcal{V}$ diagonally dominated and increasing.  
\begin{enumerate}
 \item [a)] If $U$ fixes the topology in $\FV$, then 
 \[
 \mathcal{FV}_{G}(U,E)_{lb}=\{f\in\mathcal{FV}_{G}(U,E)\;|\;\forall\;i\in \N:\;N_{U,i}(f)\;\text{bounded in}\; E\}
 \]
 with $N_{U,i}(f):=N_{U,i,i}(f)$.
 \item [b)] Let $\FV$ be a Fr\'{e}chet space. We set $U_{m}:= \{(m,x)\in U\;|\; x\in \omega_{m}\}$ and 
 $B_{j}:=\bigcup_{m=1}^{j}\{T^{\K}_{m,x}(\cdot)\nu_{m,m}(x)\;|\; (m,x)\in U_{m}\}\subset \FV'$ for $j\in\N$. 
 Then $U$ fixes the topology in $\FV$ in the sense of \prettyref{def:fix_top_1} if and only if 
 the sequence $(B_{j})_{j\in\N}$ fixes the topology in $\FV$ in the sense of \prettyref{def:fix_top_2}. 
\end{enumerate}
\end{rem}
\begin{proof}
Let us begin with a). We only need to show that the inclusion `$\supset$' holds. 
Let $f$ be an element of the right-hand side and $i,k\in\N$. 
We set $m:=\max(i,k)$ and observe that for $(k,x)\in U$ we have $(m,x)\in U$ by (i) and
\[
(e'\circ f)(k,x)=T^{\K}_{k}(f_{e'})(x)\underset{\text{(ii)}}{=}T^{\K}_{m}(f_{e'})(x)=(e'\circ f)(m,x)
\]
for each $e'\in G$ with (i) and (ii) from the definition of $U$ being chain-structured. 
Since $G$ is separating, it follows that $f(k,x)=f(m,x)$. Hence we get for all $\alpha\in\mathfrak{A}$
\begin{align*}
\sup_{y\in N_{U,i,k}(f)}p_{\alpha}(y)&=\sup_{\substack{x\in\omega_{k}\\(k,x)\in U}}p_{\alpha}(f(k,x))\nu_{i,k}(x)
\underset{\text{(i)}}{\leq} \sup_{\substack{x\in\omega_{m}\\(m,x)\in U}}p_{\alpha}(f(k,x))\nu_{m,m}(x)\\
&= \sup_{\substack{x\in\omega_{m}\\(m,x)\in U}}p_{\alpha}(f(m,x))\nu_{m,m}(x)<\infty
\end{align*}
using that $\omega_{k}\subset\omega_{m}$ and $\mathcal{V}$ is diagonally dominated.

Let us turn to part b). `$\Rightarrow$': Let $j\in\N$ and $A\subset \FV$ be bounded. Then 
\[
 \sup_{y\in B_{j}}\sup_{f\in A}|y(f)|
=\sup_{\substack{1\leq m\leq j\\(m,x)\in U_{m}}}\sup_{f\in A}|T^{\K}_{m}(f)(x)|\nu_{m,m}(x)
\leq\sup_{f\in A}\sup_{1\leq m\leq j}|f|_{m,m}<\infty
\]
since $A$ is bounded, implying that $B_{j}$ is bounded in $\FV_{b}'$. Further, $(B_{j})$ is increasing by definition. 
Additionally, for all $j\in\N$
\begin{align*}
  B_{j}^{\circ}
&=\bigcap_{m=1}^{j}\{f\in\FV\;|\;\sup_{\substack{x\in\omega_{m}\\(m,x)\in U}}|T^{\K}_{m}(f)(x)|\nu_{m,m}(x)\leq 1\}\\
&=\{f\in\FV\;|\;\sup_{\substack{x\in\omega_{j}\\(j,x)\in U}}|T^{\K}_{j}(f)(x)|\nu_{j,j}(x)\leq 1\}
\end{align*}
because $U$ is chain-structured and $\mathcal{V}$ increasing. 
Thus $(B_{j}^{\circ})$ is a fundamental system of zero neighbourhoods of $\FV$ if $U$ fixes the topology.\\
`$\Leftarrow$': Let $j,m\in\N$. Then there are $i\in\N$ and $\varepsilon>0$ such that 
\[
\varepsilon B_{i}^{\circ}\subset \{f\in\FV\;|\;|f|_{j,m}\leq 1\}=:D_{j,m}
\]
which follows from fixing the topology in the sense of \prettyref{def:fix_top_2}.
Let $f\in D_{j,m}$ and set
\[
|f|_{U_{i}}:=\sup_{(i,x)\in U_{i}}|T^{\K}_{i}(f)(x)|\nu_{i,i}(x).
\]
If $|f|_{U_{i}}= 0$, then $tf\in\varepsilon B_{i}^{\circ}$ for all $t>0$ and 
hence $t|f|_{j,m}=|tf|_{j,m}\leq 1$ for all $t>0$, which yields 
$|f|_{j,m}=0=|f|_{U_{i}}$.
If $|f|_{U_{i}}\neq 0$, then $\tfrac{f}{|f|_{U_{i}}}\in B_{i}^{\circ}$ and 
thus $\varepsilon\tfrac{f}{|f|_{U_{i}}}\in D_{j,m}$, implying
\[
 |f|_{j,m}
=\frac{1}{\varepsilon}|f|_{U_{i}}\bigl|\varepsilon\frac{f}{|f|_{U_{i}}}\bigr|_{j,m}
\leq \frac{1}{\varepsilon}|f|_{U_{i}}.
\]
The inequality $|f|_{j,m}\leq\tfrac{1}{\varepsilon}|f|_{U_{i}}$ still holds 
if $|f|_{U_{i}}=0$.
\end{proof}

\begin{thm}[{\cite[Theorem 16, p.\ 236]{B/F/J}}]\label{thm:fix_top}
 Let $Y$ be a Fr\'{e}chet--Schwartz space, $(B_{j})_{j\in\N}$ fix the topology in $Y$ and 
 $\mathsf{A}\colon X:=\operatorname{span}(\bigcup_{j\in\N} B_{j})\to E$ be a linear map 
 which is bounded on each $B_{j}$. If
 \begin{enumerate}
  \item [a)] $(\mathsf{A}^{t})^{-1}(Y)$ is dense in $E_{b}'$ and $E$ locally complete, or 
  \item [b)] $(\mathsf{A}^{t})^{-1}(Y)$ is dense in $E_{\sigma}'$ and $E$ is $B_{r}$-complete,
 \end{enumerate}
 then $\mathsf{A}$ has a (unique) extension $\widehat{\mathsf{A}}\in Y\varepsilon E$.
\end{thm}

Now, we generalise \cite[Theorem 17, p.\ 237]{B/F/J}.

\begin{thm}\label{thm:ext_FS_fix_top}
 Let $E$ be an lcHs and $G\subset E'$ a separating subspace. Let 
 $(T^{E}_{m},T^{\K}_{m})_{m\in M}$ be a strong, consistent generator for $(\mathcal{FV},E)$, 
 $\FV$ a Fr\'{e}chet--Schwartz space, $\mathcal{V}$ diagonally dominated and increasing and
 $U$ be chain-structured and fix the topology in $\FV$.
 If
 \begin{enumerate}
  \item [a)] $G$ is dense in $E_{b}'$ and $E$ locally complete, or 
  \item [b)] $E$ is $B_{r}$-complete,
 \end{enumerate}
 then the restriction map $R_{U,G}\colon S(\FV\varepsilon E)\to \mathcal{FV}_{G}(U,E)_{lb}$ is surjective.
\end{thm}
\begin{proof}
Let $f\in \mathcal{FV}_{G}(U,E)_{lb}$. We set $X:=\operatorname{span}(\bigcup_{j\in\N} B_{j})$ 
with $B_{j}$ from \prettyref{rem:fix_top_1=2} b) and $Y:=\FV$. 
Let $\mathsf{A}\colon X\to E$ be the linear map determined by 
\[
\mathsf{A}(T^{\K}_{m,x}(\cdot)\nu_{m,m}(x)):=f(m,x)\nu_{m,m}(x) 
\]
for $1\leq m\leq j$ and $(m,x)\in U_{m}$ with $U_{m}$ from \prettyref{rem:fix_top_1=2} b).
The map $\mathsf{A}$ is well-defined since $G$ is $\sigma(E',E)$-dense, and 
bounded on each $B_{j}$ because $\mathsf{A}(B_{j})=\bigcup_{m=1}^{j}N_{U,m}(f)$.
Let $e'\in G$ and $f_{e'}$ be the unique element in $\FV$ such that 
$T^{\K}_{m}(f_{e'})(x)=(e'\circ f)(m,x)$ for all $(m,x)\in U$, which implies 
$T^{\K}_{m}(f_{e'})(x)\nu_{m,m}(x)=(e'\circ \mathsf{A})(T^{\K}_{m,x}(\cdot)\nu_{m,m}(x))$ for all $(m,x)\in U_{m}$.
This equation allows us to consider $f_{e'}$ as a linear form on $X$ 
(by $f_{e'}(T^{\K}_{m,x}(\cdot)\nu_{m,m}(x)):=(e'\circ \mathsf{A})(T^{\K}_{m,x}(\cdot)\nu_{m,m}(x))$), 
which yields $e'\circ \mathsf{A}\in\FV$ for all $e'\in G$. It follows that
$G\subset(\mathsf{A}^{t})^{-1}(Y)$. Noting that $G$ is $\sigma(E',E)$-dense, we 
apply \prettyref{thm:fix_top} and obtain an extension $\widehat{\mathsf{A}}\in\FV\varepsilon E$ of 
$\mathsf{A}$. We set $F:=S(\widehat{\mathsf{A}})$
and observe that for all $(m,x)\in U$ there is $j\in\N$, $j\geq m$, such that $(j,x)\in U_{j}$ and 
$\nu_{j,j}(x)>0$ by \eqref{loc3} and because $U$ is chain-structured and $\mathcal{V}$ diagonally dominated 
and increasing. Due to the proof of \prettyref{rem:fix_top_1=2} a) we have $f(j,x)=f(m,x)$ and thus
\begin{align*}
T^{E}_{m}(F)(x)&=T^{E}_{m}S(\widehat{\mathsf{A}})(x)=\widehat{\mathsf{A}}(T^{\K}_{m,x})
=\frac{1}{\nu_{j,j}(x)}\widehat{\mathsf{A}}(T^{\K}_{m,x}(\cdot)\nu_{j,j}(x))\\
&=\frac{1}{\nu_{j,j}(x)}\widehat{\mathsf{A}}(T^{\K}_{j,x}(\cdot)\nu_{j,j}(x))
=f(j,x)=f(m,x)
\end{align*}
by consistency, yielding $R_{U,G}(F)=f$.
\end{proof}

In particular, condition a) is fulfilled if $E$ is semi-reflexive. 
Indeed, if $E$ is semi-reflexive, then $E$ is quasi-complete by \cite[Chap.\ IV, 5.5, Corollary 1, p.\ 144]{schaefer}
and $\overline{G}^{b(E',E)}=\overline{G}^{\tau(E',E)}=E'$ by \cite[11.4.1 Proposition, p.\ 227]{Jarchow} 
and the bipolar theorem. 
For instance, condition b) is satisfied if $E$ is a Fr\'echet space or 
$E=(\mathcal{C}^{\infty}_{\overline{\partial},b}(\D),\beta)$ 
which is a $B_{r}$-complete space by \prettyref{prop:strict_topo_hypo_B_complete_semi_Montel} and 
is not a Fr\'echet space by \prettyref{rem:strict_top_non_barrelled}.

As stated, our preceding theorem generalises \cite[Theorem 17, p.\ 237]{B/F/J} 
where $\FV$ is a closed subspace of $\mathcal{CW}^{\infty}(\Omega)$ 
for open, connected $\Omega\subset\R^{d}$. A characterisation of sets 
that fix the topology in the space $\mathcal{CW}^{\infty}_{\overline{\partial}}(\Omega)$ 
of holomorphic functions on an open, connected set $\Omega\subset\C$ is given in \cite[Remark 14, p.\ 235]{B/F/J}. 
The characterisation given in \cite[Remark 14 (b), p.\ 235]{B/F/J} is still valid and applied 
in \cite[Corollary 18, p.\ 238]{B/F/J} for closed subspaces of $\mathcal{CW}^{\infty}_{P(\partial)}(\Omega)$ 
where $P(\partial)^{\K}$ is a hypoelliptic linear partial differential operator which satisfies the maximum principle, 
namely, that $U\subset\Omega$ fixes the topology 
if and only if there is a sequence $(\Omega_{n})_{n\in\N}$ of relatively compact, open subsets of $\Omega$ 
with $\bigcup_{n\in\N}\Omega_{n}=\Omega$ such that 
$\partial\Omega_{n}\subset\overline{U\cap\Omega_{n+1}}$ for all $n\in\N$. 
Among the hypoelliptic operators $P(\partial)^{\K}$ satisfying the maximum principle are the Cauchy--Riemann operator 
$\overline{\partial}$ and the Laplacian $\Delta$. 
Further examples can be found in \cite[Corollary 3.2, p.\ 33]{gilbarg_trudinger2001}.
The statement of \cite[Corollary 18, p.\ 238]{B/F/J} for the space of holomorphic functions 
is itself a generalisation of \cite[Theorem 2, p.\ 401]{grosse-erdmann2004} with 
\cite[Remark 2 (a), p.\ 406]{grosse-erdmann2004} where $E$ is $B_{r}$-complete
and of \cite[Theorem 6, p.\ 10]{jorda2005} where $E$ is semi-reflexive. 
The case that $G$ is dense in $E_{b}'$ and $E$ is sequentially complete 
is covered by \cite[3.3 Satz, p.\ 228--229]{Gramsch1977}, not only for spaces of holomorphic functions, 
but for several classes of function spaces. 

Let us turn to other families of weights than $\mathcal{W}^{\infty}$. 
Due to \prettyref{prop:AV_CV_coincide_hol_har} we already know that $U:=\{0\}\times\C$ fixes the topology in 
$\mathcal{CV}^{\infty}_{\overline{\partial}}(\C)=\mathcal{CV}_{\overline{\partial}}(\C)$ 
and $U:=\{0\}\times\R^{d}$ in $\mathcal{CV}^{\infty}_{\Delta}(\R^{d})=\mathcal{CV}_{\Delta}(\R^{d})$ 
if $\mathcal{V}:=(\nu_{j})_{j\in\N}$ fulfils \prettyref{cond:weights} and 
$\mathcal{V}^{\infty}:=(\nu_{j,m})_{j\in\N,m\in\N_{0}}$ where 
$\nu_{j,m}\colon\{\beta\in\N_{0}^{d}\;|\;|\beta|\leq m\}\times\R^{d}\to [0,\infty)$, $\nu_{j,m}(\beta,x):=\nu_{j}(x)$.

\begin{cor}
Let $E$ be an lcHs, $G\subset E'$ a separating subspace, $\mathcal{V}:=(\nu_{j})_{j\in\N}$ 
an increasing family of weights which is locally bounded away from zero on an open set $\Omega\subset\R^{d}$,  
$P(\partial)^{\K}$ a hypoelliptic linear partial differential operator, $\mathcal{CV}_{P(\partial)}(\Omega)$ 
a Schwartz space and $U\subset\Omega$ fix the topology of $\mathcal{CV}_{P(\partial)}(\Omega)$. If 
 \begin{enumerate}
  \item [a)] $G$ is dense in $E_{b}'$ and $E$ locally complete, or
  \item [b)] $E$ is $B_{r}$-complete,
 \end{enumerate}
and $f\colon U\to E$ is a function in $\ell\mathcal{V}(U)$ such that $e'\circ f$ admits an extension 
$f_{e'}\in\mathcal{CV}_{P(\partial)}(\Omega)$ for each $e'\in G$, 
then there is a unique extension $F\in\mathcal{CV}_{P(\partial)}(\Omega,E)$ of $f$.
\end{cor}
\begin{proof}
The existence of $F$ follows from \prettyref{prop:frechet_bierstedt}, \prettyref{ex:subspace_bierstedt} b) 
and \prettyref{thm:ext_FS_fix_top} with $(T^{E}_{m},T^{\K}_{m})_{m\in M}:=(\id_{E^{\Omega}},\id_{\K^{\Omega}})$. 
The uniqueness of $F$ is a result of \prettyref{prop:injectivity}.
\end{proof}

We have the following sufficient conditions on a family of weights $\mathcal{V}$ 
which guarantee the existence of a countable set 
$U\subset\C$ that fixes the topology of $\mathcal{CV}_{\overline{\partial}}(\C)$ due to Abanin and Varziev \cite{abanin2013}.

\begin{prop}\label{prop:AV_fix_top_Abanin}
Let $\mathcal{V}:=(\nu_{j})_{j\in\N}$ where $\nu_{j}(z):=\exp(a_{j}\mu(z)-\varphi(z))$, $z\in\C$, 
with some continuous, subharmonic function $\mu\colon\C\to[0,\infty)$, a continuous function $\varphi\colon\C\to\R$ 
and a strictly increasing, positive sequence $(a_{j})_{j\in\N}$ 
with $a:=\lim_{j\to\infty} a_{j}\in (0,\infty]$. Let there be
\begin{enumerate}
\item[(i)] $s\geq 0$ and $C>0$ such that $|\varphi(z)-\varphi(\zeta)|\leq C$ and $|\mu(z)-\mu(\zeta)|\leq C$ 
for all $z,\zeta\in\C$ with $|z-\zeta|\leq (1+|z|)^{-s}$,
\item[(ii)] $\max(\varphi(z),\mu(z))\leq |z|^{q}+C_{0}$ for some $q,C_{0}>0$ and
\item[(iii)] $\ln(|z|)=O(\mu(z))$ as $|z|\to\infty$ if $a=\infty$, or $\ln(|z|)=o(\mu(z))$ as $|z|\to\infty$ 
if $0<a<\infty$. 
\end{enumerate}
Let $(\lambda_{k})_{k\in\N}$ be the sequence of simple zeros of a function 
$L\in\mathcal{C}\widetilde{\mathcal{V}}_{\overline{\partial}}(\C)$ having no other zeros 
where $\widetilde{\mathcal{V}}:=(\nu_{j}^{2}/\nu_{m_{j}})_{j\in\N}$ for some sequence $(m_{j})_{j\in\N}$ in $\N$. 
Suppose that there are $j_{0}\in\N$ and 
a sequence of circles $\{z\in\C\;|\;|z|=R_{m}\}$ with $R_{m}\nearrow\infty$ such that 
\[
|L(z)|\nu_{j_{0}}(z)\geq C_{m},\quad m\in\N,\,z\in\C,\,|z|=R_{m},
\]
for some $C_{m}\nearrow\infty$ and 
\[
|L'(\lambda_{k})|\nu_{j_{0}}(\lambda_{k})\geq 1\quad\text{for all sufficiently large}\;k\in\N.
\]
Then $\mathcal{CV}_{\overline{\partial}}(\C)$ is a nuclear Fr\'echet space for all $a\in(0,\infty]$ 
and $U:=(\lambda_{k})_{k\in\N}$ fixes the topology 
of $\mathcal{CV}_{\overline{\partial}}(\C)$ if $a=\infty$. 
If $\mu$ is a radial function, i.e.\ $\mu(z)=\mu(|z|)$, $z\in\C$, with $\mu(2z)\sim\mu(z)$ as $|z|\to\infty$, then 
$U$ fixes the topology of $\mathcal{CV}_{\overline{\partial}}(\C)$ for all $a\in(0,\infty]$.
\end{prop}
\begin{proof}
First, we check that \prettyref{cond:weights} is satisfied, which implies that 
$\mathcal{CV}_{\overline{\partial}}(\C)$ is a nuclear Fr\'echet space by \prettyref{prop:AV_CV_coincide_hol_har}. 
We set $k:=\max(s,2)$ and observe that (i) is also fulfilled with $k$ instead of $s$. 
Let $z\in\C$ and $\|\zeta\|_{\infty},\|\eta\|_{\infty}\leq (1/\sqrt{2})(1+|z|)^{-k}=:r(z)$. 
From $|\cdot|\leq \sqrt{2}\|\cdot\|_{\infty}$ and (i) it follows
\[
|\mu(z+\zeta)-\mu(z+\eta)|\leq |\mu(z+\zeta)-\mu(z)|+|\mu(z)-\mu(z+\eta)|\leq 2C
\] 
and thus $\mu(z+\zeta)\leq 2C+\mu(z+\eta)$. In the same way we obtain $-\varphi(z+\zeta)\leq 2C-\varphi(z+\eta)$. 
Hence we have
\[
a_{j}\mu(z+\zeta)-\varphi(z+\zeta)\leq 2C(a_{j}+1)+a_{j}\mu(z+\eta)-\varphi(z+\eta)
\]
for $j\in\N$, implying	
\[
\nu_{j}(z+\zeta)\leq \e^{2C(a_{j}+1)}\nu_{j}(z+\eta),
\]
which means that $(\alpha.1)$ of \prettyref{cond:weights} holds. By (iii) there are $\varepsilon>0$ and $R>0$ 
such that $\ln(|z|)\leq \varepsilon \mu(z)$ for all $z\in\C$ with $|z|\geq R$ if $a=\infty$. 
This yields for all $|z|\geq \max(2,R)$ that
\[
a_{j}\mu(z)+k\ln(1+|z|)\leq a_{j}\mu(z)+k\ln(|z|^2) = a_{j}\mu(z)+2k\ln(|z|)\leq a_{j}\mu(z)+2k\varepsilon\mu(z).
\] 
Since $a=\infty$, there is $n\in\N$ such that $a_{n}\geq a_{j}+2k\varepsilon$, resulting in 
\[
a_{j}\mu(z)+k\ln(1+|z|)\leq a_{n}\mu(z)
\] 
for all $|z|\geq \max(2,R)$. Therefore we derive 
\begin{equation}\label{eq:AV_fix_top}
a_{j}\mu(z)+k\ln(1+|z|)\leq a_{n}\mu(z)+k\ln(1+\max(2,R))
\end{equation}
for all $z\in\C$, which means that $(\alpha.2)$ and $(\alpha.3)$ hold with $\psi_{j}(z):=r(z)$. 
If $0<a<\infty$, for every $\varepsilon>0$ there is $R>0$ such that $\ln(|z|)\leq \varepsilon \mu(z)$ for all 
$z\in\C$ with $|z|\geq R$ by (iii). Thus we may choose $\varepsilon>0$ such that $a_{j+1}-a_{j}\geq 2k\varepsilon>0$ 
because $(a_{j})$ is strictly increasing. We deduce that \eqref{eq:AV_fix_top} with $n:=j+1$ 
holds in this case as well and $(\alpha.2)$ and $(\alpha.3)$, too.

Observing that the condition that $U=(\lambda_{k})_{k\in\N}$ is the sequence of simple zeros of a function 
$L\in\mathcal{C}\widetilde{\mathcal{V}}_{\overline{\partial}}(\C)$ means that 
$L\in\mathscr{L}(\Phi^{a}_{\varphi,\mu};U)$ and (i) that $\varphi$ and $\mu$ vary slowly w.r.t.\ $r(z):=(1+|z|)^{-s}$ 
in the notation of \cite[Definition, p.\ 579, 584]{abanin2013} and \cite[p.\ 585]{abanin2013}, respectively, 
the statement that $U$ fixes the topology is a consequence of \cite[Theorem 2, p.\ 585--586]{abanin2013}.
\end{proof}

\begin{rem}\fakephantomsection\label{rem:ex_fix_top_AP}
\begin{enumerate}
\item[a)] Let $D\subset\C$ be convex, bounded and open with $0\in D$. 
Let $\varphi(z):=H_{D}(z):=\sup_{\zeta\in D}\re(z\zeta)$, $z\in\C$, be the supporting 
function of $D$, $\mu(z):=\ln(1+|z|)$, $z\in\C$, and $a_{j}:=j$, $j\in\N$. 
Then $\varphi$ and $\mu$ fulfil the conditions of \prettyref{prop:AV_fix_top_Abanin} with $a=\infty$ 
by \cite[p.\ 586]{abanin2013} and the existence of an entire function $L$ which fulfils the conditions 
of \prettyref{prop:AV_fix_top_Abanin} is guaranteed by \cite[Theorem 1.6, p.\ 1537]{abanin2011}. 
Thus there is a countable set $U:=(\lambda_{k})_{k\in\N}\subset\C$ which fixes the topology 
in $A^{-\infty}_{D}:=\mathcal{CV}_{\overline{\partial}}(\C)$ with 
$\mathcal{V}:=(\exp(a_{j}\mu-\varphi))_{j\in\N}$.
\item[b)] An explicit construction of a set $U:=(\lambda_{k})_{k\in\N}\subset\C$ 
which fixes the topology in $A^{-\infty}_{D}$ is given 
in \cite[Algorithm 3.2, p.\ 3629]{abanin2010}. This construction does not rely on the entire function $L$. 
In particular (see \cite[p.\ 15]{bonet2017}), for $D:=\D$ we have $\varphi(z)=|z|$, 
for each $k\in\N$ we may take $l_{k}\in\N$, $l_{k}>2\pi k^2$, 
and set $\lambda_{k,j}:=kr_{k,j}$, $1\leq j\leq l_{k}$, where $r_{k,j}$ denote the $l_{k}$-roots of unity. 
Ordering $\lambda_{k,j}$ in a sequence of one index appropriately, 
we obtain a sequence which fixes the topology of $A^{-\infty}_{\D}$.
\item[c)] Let $\mu\colon\C\to[0,\infty)$ be a continuous, subharmonic, radial function 
which increases with $|z|$ and satisfies 
\begin{enumerate}
\item[(i)] $\sup_{\zeta\in\C,\|\zeta\|_{\infty}\leq r(z)}\mu(z+\zeta)\leq 
\inf_{\zeta\in\C,\|\zeta\|_{\infty}\leq r(z)}\mu(z+\zeta)+C$ 
for some continuous function $r\colon\C\to (0,1]$ and $C>0$,
\item[(ii)] $\ln(1+|z|^2)=o(\mu(|z|))$ as $|z|\to\infty$,
\item[(iii)] $\mu(2|z|)=O(\mu(|z|))$ as $|z|\to\infty$. 
\end{enumerate}
Then $\mathcal{V}:=(\exp(-(1/j)\mu))_{j\in\N}$ fulfils \prettyref{cond:weights} where $(\alpha.1)$ follows from (i) 
and $(\alpha.2)$, $(\alpha.3)$ as in the proof of \prettyref{prop:AV_fix_top_Abanin}. 
Thus $\mathcal{CV}_{\overline{\partial}}(\C)$ is a nuclear Fr\'echet space by \prettyref{prop:AV_CV_coincide_hol_har}.
If $\mu(|z|)=o(|z|^2)$ as $|z|\to\infty$ or $\mu(|z|)=|z|^2$, $z\in\C$, 
then $U:=\{\alpha n+\iu\beta m\;|\;n,m\in\Z\}$ fixes 
the topology in the space $A_{\mu}^{0}:=\mathcal{CV}_{\overline{\partial}}(\C)$ for any $\alpha,\beta>0$ 
by \cite[Corollary 4.6, p.\ 20]{bonet2017} and \cite[Proposition 4.7, p.\ 20]{bonet2017}, respectively. 
\item[d)] For instance, the conditions on $\mu$ in c) are  fulfilled for $\mu(z):=|z|^{\gamma}$, $z\in\C$, 
with $0<\gamma\leq 2$ by \cite[1.5 Examples (3), p.\ 205]{meise1987}. 
If $\gamma=1$, then $A_{\mu}^{0}=A^{0}_{\overline{\partial}}(\C)$ is the space 
of entire functions of exponential type zero (see \prettyref{rem:ex_NF_cond_weights}).
\item[e)] More general characterisations of countable sets that fix the topology of 
$\mathcal{CV}_{\overline{\partial}}(\C)$ can be found in \cite[Theorem 1, p.\ 580]{abanin2013} 
and \cite[Theorem 4.5, p.\ 17]{bonet2017}.
\end{enumerate}
\end{rem}

The spaces $A_{\mu}^{0}$ from c) are known as \emph{H\"ormander algebras} and the space $A^{-\infty}_{D}$ 
considered in a) is isomorphic to the strong dual of the \emph{Korenblum space} $A^{-\infty}(D)$ 
via Laplace transform by \cite[Proposition 4, p.\ 580]{melikhov2004}.

\addtocontents{toc}{\SkipTocEntry}
\subsection*{\texorpdfstring{$\Fv$}{Fv(Omega)} a Banach space and \texorpdfstring{$E$}{E} locally complete}

For a $\dom$-space $\Fv$, a set $U$ that fixes the topology in $\Fv$ and a separating subspace $G\subset E'$ 
we have
\begin{align*}
\mathcal{F}\nu_{G}(U,E)_{lb}=&\{f\in\mathcal{F}\nu_{G}(U,E)\;|\;N_{U}(f)\;\text{bounded in}\; E\}\\
=&\mathcal{F}\nu_{G}(U,E)\cap\ell\nu(U,E)
\end{align*}
where $N_{U}(f):=\{f(x)\nu(x)\;|\;x\in U\}$.
Let us recall the assumptions of \prettyref{rem:R_well-defined_Banach} but now $U$ fixes the topology.
Let $(T^{E},T^{\K})$ be a strong, consistent family for $(F,E)$ 
and a generator for $(\mathcal{F}\nu,E)$.
Let $\f$ and $\fe$ be $\varepsilon$-into-compatible and the inclusion $\Fv\hookrightarrow\f$ continuous. 
Consider a set $U$ which fixes the topology in $\Fv$ 
and a separating subspace $G\subset E'$.
For $u\in \f\varepsilon E$ such that $u(B_{\Fv}^{\circ \f'})$ is bounded in $E$
we have $R_{U,G}(f)\in\mathcal{F}\nu_{G}(U,E)$ with $f:=S(u)\in\Feps$ by \eqref{eq:well_def_B_unique}.
Further, $T^{\K}_{x}(\cdot)\nu(x)\in B_{\Fv}^{\circ \f'}$ for every $x\in\omega$, 
which implies that 
\[
\sup_{x\in U}p_{\alpha}(R_{U,G}(f)(x))\nu(x)
=\sup_{x\in U}p_{\alpha}\bigl(u(T^{\K}_{x}(\cdot)\nu(x))\bigr)\\
\leq \sup_{y'\in B_{\Fv}^{\circ \f'}}p_{\alpha}(u(y'))<\infty
\]
for all $\alpha\in\mathfrak{A}$ by consistency. Hence $R_{U,G}(f)\in\mathcal{F}\nu_{G}(U,E)_{lb}$.
Therefore the injective linear map
\[
R_{U,G}\colon \Feps \to \mathcal{F}\nu_{G}(U,E)_{lb},\;f\mapsto (T^{E}(f)(x))_{x\in U},
\]
is well-defined and the question we want to answer is:

\begin{que}
Let the assumptions of \prettyref{rem:R_well-defined_Banach} be fulfilled and $U$ fix the topology in $\Fv$.
When is the injective restriction map 
\[
R_{U,G}\colon \Feps\to \mathcal{F}\nu_{G}(U,E)_{lb},\;f\mapsto (T^{E}(f)(x))_{x\in U},
\]
surjective?
\end{que}

\begin{prop}[{\cite[Proposition 3.1, p.\ 692]{F/J/W}}]\label{prop:ext_B_fix_top}
Let $E$ be a locally complete lcHs, $G\subset E'$ a separating subspace and $Z$ a Banach space 
whose closed unit ball $B_{Z}$ is a 
compact subset of an lcHs $Y$. Let $B_{1}\subset B_{Z}^{\circ Y'}$ such that 
$B_{1}^{\circ Z}:=\{z\in Z\;|\;\forall\;y'\in B_{1}:\;|y'(z)|\leq 1\}$ is bounded in $Z$.  
If $\mathsf{A}\colon X:=\operatorname{span}B_{1}\to E$ is a linear map which is bounded on $B_{1}$ 
such that there is a $\sigma(E',E)$-dense subspace $G\subset E'$ with $e'\circ \mathsf{A}\in Z$ for all $e'\in G$, 
then there exists a (unique) extension $\widehat{\mathsf{A}}\in Y\varepsilon E$ of $\mathsf{A}$ 
such that $\widehat{\mathsf{A}}(B_{Z}^{\circ Y'})$ is bounded in $E$.
\end{prop}

The following theorem is a generalisation of \cite[Theorem 3.2, p.\ 693]{F/J/W} and \cite[Theorem 12, p.\ 5]{jorda2013}.

\begin{thm}\label{thm:ext_B_fix_top}
Let $E$ be a locally complete lcHs, $G\subset E'$ a separating subspace and $\f$ 
and $\fe$ be $\varepsilon$-into-compatible. 
Let $(T^{E},T^{\K})$ be a generator for $(\mathcal{F}\nu,E)$ 
and a strong, consistent family for $(F,E)$, $\Fv$ a Banach space
whose closed unit ball $B_{\Fv}$ 
is a compact subset of $\f$ and $U$ fix the topology in $\Fv$.
Then the restriction map 
\[
 R_{U,G}\colon \Feps \to \mathcal{F}\nu_{G}(U,E)_{lb} 
\]
is surjective.
\end{thm}
\begin{proof}
Let $f\in \mathcal{F}\nu_{G}(U,E)_{lb}$. 
We set $B_{1}:=\{T^{\K}_{x}(\cdot)\nu(x)\;|\;x\in U\}$, 
$X:=\operatorname{span}B_{1}$, $Y:=\f$ and $Z:=\Fv$. 
We have $B_{1}\subset Y'$ since $(T^{E},T^{\K})$ is a consistent family for $(F,E)$. 
If $f\in B_{Z}$, then
\[
|T^{\K}_{x}(f)\nu(x)|\leq |f|_{\Fv}\leq 1
\]
for all $x\in U$ and thus $B_{1}\subset B_{Z}^{\circ Y'}$. 
Further on, there is $C>0$ such that for all $f\in B_{1}^{\circ Z}$
\[
|f|_{\Fv}\leq C\sup_{x\in U}|T^{\K}_{x}(f)|\nu(x)
\leq C
\]
as $U$ fixes the topology in $Z$, implying the boundedness of $B_{1}^{\circ Z}$ in $Z$.
Let $\mathsf{A}\colon X\to E$ be the linear map determined by 
\[
\mathsf{A}(T^{\K}_{x}(\cdot)\nu(x)):=f(x)\nu(x). 
\]
The map $\mathsf{A}$ is well-defined since $G$ is $\sigma(E',E)$-dense, and 
bounded on $B_{1}$ because $\mathsf{A}(B_{1})=N_{U}(f)$.
Let $e'\in G$ and $f_{e'}$ be the unique element in $\mathcal{F}\nu(\Omega)$ such that 
$T^{\K}(f_{e'})(x)=(e'\circ f)(x)$ for all $x \in U$, which implies 
$T^{\K}(f_{e'})(x)\nu(x)=(e'\circ \mathsf{A})(T^{\K}_{x}(\cdot)\nu(x))$.
Again, this equation allows us to consider $f_{e'}$ as a linear form on $X$ 
(by setting $f_{e'}(T^{\K}_{x}(\cdot)\nu(x)):=(e'\circ \mathsf{A})(T^{\K}_{x}(\cdot)\nu(x))$), 
which yields $e'\circ \mathsf{A}\in\mathcal{F}\nu(\Omega)=Z$ for all $e'\in G$. 
Hence we can apply \prettyref{prop:ext_B_fix_top} and obtain an extension 
$\widehat{\mathsf{A}}\in Y\varepsilon E$ of $\mathsf{A}$ 
such that $\widehat{\mathsf{A}}(B_{Z}^{\circ Y'})$ is bounded in $E$. 
We set $\widetilde{F}:=S(\widehat{\mathsf{A}})\in\Feps$ and get for all $x\in U$ that
\[
T^{E}(\widetilde{F})(x)=T^{E}S(\widehat{\mathsf{A}})(x)=\widehat{\mathsf{A}}(T^{\K}_{x})
=\frac{1}{\nu(x)}\mathsf{A}(T^{\K}_{x}(\cdot)\nu(x))=f(x)
\]
by consistency for $(F,E)$, yielding $R_{U,G}(\widetilde{F})=f$.
\end{proof}

\begin{cor}\label{cor:hypo_weighted_ext_fix_top}
Let $E$ be a locally complete lcHs, $G\subset E'$ a separating subspace, $\Omega\subset\R^{d}$ open, 
$P(\partial)^{\K}$ a hypoelliptic linear partial differential operator, $\nu\colon\Omega\to(0,\infty)$ continuous and 
$U$ fix the topology in $\mathcal{C}\nu_{P(\partial)}(\Omega)$.
If $f\colon U\to E$ is a function in $\ell\nu(U,E)$ such that $e'\circ f$ 
admits an extension $f_{e'}\in\mathcal{C}\nu_{P(\partial)}(\Omega)$ 
for every $e'\in G$, then there exists a unique extension $F\in\mathcal{C}\nu_{P(\partial)}(\Omega,E)$ of $f$.
\end{cor}
\begin{proof}
Observing that $f\in\mathcal{F}\nu_{G}(U,E)_{lb}$ with $\mathcal{F}\nu(\Omega)=\mathcal{C}\nu_{P(\partial)}(\Omega)$, 
our statement follows directly from \prettyref{thm:ext_B_fix_top} 
whose conditions are fulfilled by the proof of \prettyref{cor:hypo_weighted_ext_unique}. 
\end{proof}

Sets that fix the topology in $\mathcal{C}\nu_{P(\partial)}(\Omega)$ for different weights $\nu$ 
are well-studied if $P(\partial)=\overline{\partial}$ is the Cauchy--Riemann operator.
If $\Omega\subset\C$ is open, $P(\partial)=\overline{\partial}$ and $\nu=1$, 
then $\mathcal{C}\nu_{P(\partial)}(\Omega)=H^{\infty}(\Omega)$ is the space of bounded holomorphic functions 
on $\Omega$. 
Brown, Shields and Zeller characterise the countable discrete sets $U:=(z_{n})_{n\in\N}\subset\Omega$ that 
fix the topology in $H^{\infty}(\Omega)$ 
with $C=1$ and equality in \prettyref{def:fix_top_1} for Jordan domains $\Omega$ in 
\cite[Theorem 3, p.\ 167]{brownshieldszeller1960}.
In particular, they prove for $\Omega=\D$ that a discrete $U=(z_{n})_{n\in\N}$ fixes the topology in 
$H^{\infty}(\D)$ if and only if almost 
every boundary point is a non-tangential limit of a sequence in $U$. 
Bonsall obtains the same characterisation for bounded harmonic functions, i.e.\ $P(\partial)=\Delta$ and 
$\nu=1$, on $\Omega=\D$ by \cite[Theorem 2, p.\ 473]{bonsall1987}. 
An example of such a set $U=(z_{n})_{n\in\N}\subset\D$ is constructed in
\cite[Remark 6, p.\ 172]{brownshieldszeller1960}.
Probably the first example of a countable discrete set $U\subset\D$ that fixes the topology in $H^{\infty}(\D)$ 
is given by Wolff in \cite[p.\ 1327]{wolff1921} (cf.\ \cite[Theorem (Wolff), p.\ 402]{grosse-erdmann2004}). 
In \cite[4.14 Theorem, p.\ 255]{rubelshields1966} Rubel and Shields give a charaterisation of sets $U\subset\Omega$ 
that fix the topology in $H^{\infty}(\Omega)$ by means of bounded complex measures where $\Omega\subset\C$ is open 
and connected. The existence of a countable $U$ fixing the topology in $H^{\infty}(\Omega)$ is shown in 
\cite[4.15 Proposition, p.\ 256]{rubelshields1966}.
In the case of several complex variables the existence of such a countable $U$ is treated by Sibony in 
\cite[Remarques 4 b), p.\ 209]{sibony1975}
and by Massaneda and Thomas in \cite[Theorem 2, p.\ 838]{massaneda2000}.

If $\Omega=\C$ and $P(\partial)=\overline{\partial}$, 
then $\mathcal{C}\nu_{P(\partial)}(\Omega)=:F_{\nu}^{\infty}(\C)$ is 
a generalised $L^{\infty}$-version of the Bargmann--Fock space. 
In the case that $\nu(z)=\exp(-\alpha |z|^{2}/2)$, $z\in\C$, for some $\alpha>0$, Seip and Wallst\'en show 
in \cite[Theorem 2.3, p.\ 93]{seip1992b} that a countable discrete set $U\subset\C$ fixes the topology 
in $F_{\nu}^{\infty}(\C)$ if and only if $U$ contains a uniformly discrete subset $U'$ with lower uniform 
density $D^{-}(U')>\alpha/\pi$ (the proof of sufficiency is given in \cite{seip1992c} and the result was announced 
in \cite[Theorem 1.3, p.\ 324]{seip1992a}). 
A generalisation of this result using lower angular densities is given by Lyubarski{\u{\i}} and Seip 
in \cite[Theorem 2.2, p.\ 162]{lyubarskij1994} to weights of the form $\nu(z)=\exp(-\phi(\arg z)|z|^{2}/2)$, $z\in\C$,
with a $2\pi$-periodic $2$-trigonometrically convex function $\phi$ such that $\phi\in\mathcal{C}^{2}([0,2\pi])$ 
and $\phi(\theta)+(1/4)\phi''(\theta)>0$ for all $\theta\in[0,2\pi]$.
An extension of the results in \cite{seip1992b} to weights of the form $\nu(z)=\exp(-\phi(z))$, $z\in\C$, 
with a subharmonic function $\phi$ such that $\Delta\phi(z)\sim 1$ is given in 
\cite[Theorem 1, p.\ 249]{ortegacerda1998} by Ortega-Cerd{\`a} and Seip. 
Here, $f(x)\sim g(x)$ for two functions $f,g\colon\Omega\to\R$ means 
that there are $C_{1},C_{2}>0$ such that $C_{1}g(x)\leq f(x)\leq C_{2}g(x)$ for all $x\in\Omega$.  
Marco, Massaneda and Ortega-Cerd{\`a} describe sets that fix the topology in $F_{\nu}^{\infty}(\C)$ with 
$\nu(z)=\exp(-\phi(z))$, $z\in\C$, for some subharmonic function $\phi$ whose Laplacian $\Delta\phi$ is a 
doubling measure (see \cite[Definition 5, p.\ 868]{marco2003}), 
e.g.\ $\phi(z)=|z|^{\beta}$ for some $\beta>0$, in \cite[Theorem A, p.\ 865]{marco2003}.
The case of several complex variables is handled by Ortega-Cerd{\`a}, Schuster and Varolin 
in \cite[Theorem 2, p.\ 81]{ortegacerda2006}.

If $\Omega=\D$ and $P(\partial)=\overline{\partial}$, then 
$\mathcal{C}\nu_{P(\partial)}(\Omega)=:A_{\nu}^{\infty}(\D)$ is 
a generalised $L^{\infty}$-version of the weighted Bergman space (and of $H^{\infty}(\D)$). 
For $\nu(z)=(1-|z|^{2})^{n}$, $z\in\D$, for some $n\in\N$, Seip proves that a countable discrete set 
$U\subset\D$ fixes the topology in $A_{\nu}^{\infty}(\D)$ if and only if $U$ contains a uniformly discrete subset 
$U'$ with lower uniform density $D^{-}(U')>n$ by \cite[Theorem 1.1, p.\ 23]{seip1993}, and gives a typical example 
in \cite[p.\ 23]{seip1993}.
Later on, this is extended by Seip in \cite[Theorem 2, p.\ 718]{seip1998} to weights $\nu(z)=\exp(-\phi(z))$, $z\in\D$, 
with a subharmonic function $\phi$ such that $\Delta\phi(z)\sim(1-|z|^{2})^{-2}$, 
e.g.\ $\phi(z)=-\beta\ln(1-|z|^2)$, $z\in\D$, for some $\beta>0$.
Doma\'nski and Lindstr\"{o}m give necessary resp.\ sufficient conditions for fixing the topology 
in $A_{\nu}^{\infty}(\D)$ in the case that $\nu$ is an essential weight on $\D$, i.e.\ there is $C>0$ with
$\nu(z)\leq\widetilde{\nu}(z)\leq C\nu(z)$ for each $z\in\D$ where 
$\widetilde{\nu}(z):=(\sup\{|f(z)|\;|\;f\in B_{A_{\nu}^{\infty}(\D)}\})^{-1}$ is the associated weight. 
In \cite[Theorem 29, p.\ 260]{DomLind2002} they describe necessary resp.\ sufficient conditions for fixing the topology 
if the upper index $U_{\nu}$ is finite (see \cite[p.\ 242]{DomLind2002}), and 
necessary and sufficient conditions in \cite[Corollary 31, p.\ 261]{DomLind2002} 
if $0<L_{\nu}=U_{\nu}<\infty$ holds where $L_{\nu}$ is the lower index (see \cite[p.\ 243]{DomLind2002}), 
which for example can be applied to $\nu(z)=(1-|z|^{2})^{n}(\ln(\tfrac{\e}{1-|z|}))^{\beta}$, $z\in\D$, 
for some $n>0$ and $\beta\in\R$. 
The case of simply connected open $\Omega\subset\C$ is considered in \cite[Corollary 32, p.\ 261--262]{DomLind2002}.

Borichev, Dhuez and Kellay treat $A_{\nu}^{\infty}(\D)$ and $F_{\nu}^{\infty}(\C)$ simultaneously.  
Let $\Omega_{R}:=\D$, if $R=1$, and $\Omega_{R}:=\C$ if $R=\infty$. They take $\nu(z)=\exp(-\phi(z))$, $z\in\Omega_{R}$, 
where $\phi\colon[0,R)\to[0,\infty)$ is an increasing function such that $\phi(0) = 0$, 
$\lim_{r\to R}\phi(r)=\infty$, $\phi$ is extended to $\Omega_{R}$ by $\phi(z):=\phi(|z|)$, 
$\phi\in\mathcal{C}^{2}(\Omega_{R})$, and, in addition $\Delta\phi(z)\geq 1$ if $R=\infty$ 
(see \cite[p.\ 564--565]{borichev2007}). 
Then they set $\rho\colon[0,R)\to\R$, $\rho(r):=[\Delta\phi(r)]^{-1/2}$, and suppose that $\rho$
decreases to $0$ near $R$, $\rho'(r)\to 0$, $r\to R$, and either
$(I_{\D})$ the function $r\mapsto\rho(r)(1-r)^{-C}$ increases for some $C\in\R$ and for $r$ close to $1$, resp.\
$(I_{\C})$ the function $r\mapsto\rho(r)r^{C}$ increases for some $C\in\R$ and for large $r$,
or $(II_{\Omega_{R}})$ that $\rho'(r)\ln(1/\rho(r))\to 0$, $r\to R$ (see \cite[p.\ 567--569]{borichev2007}).
Typical examples for $(I_{\D})$ are 
\[
\phi(r)=\ln(\ln(\tfrac{1}{1-r}))\ln(\tfrac{1}{1-r})\quad\text{or}\quad\phi(r)=\tfrac{1}{1-r},
\]
a typical example for $(II_{\D})$ is $\phi(r)=\exp(\tfrac{1}{1-r})$, for $(I_{\C})$
\[
 \phi(r)=r^2\ln(\ln(r))\quad\text{or}\quad\phi(r)=r^p,\;\text{for some}\; p>2,
\]
and a typical example for $(II_{\C})$ is $\phi(r)=\exp(r)$.
Sets that fix the topology in $A_{\nu}^{\infty}(\D)$ are described by densities 
in \cite[Theorem 2.1, p.\ 568]{borichev2007} and sets that fix the topology in $F_{\nu}^{\infty}(\C)$ 
in \cite[Theorem 2.5, p.\ 569]{borichev2007}.

Wolf uses sets that fix the topology in $A_{\nu}^{\infty}(\D)$ for the characterisation of weighted composition 
operators on $A_{\nu}^{\infty}(\D)$ with closed range in \cite[Theorem 1, p.\ 36]{wolf2011} for bounded $\nu$.

\begin{cor}\label{cor:Bloch_ext_fix_top}
Let $E$ be a locally complete lcHs, $G\subset E'$ a separating subspace, $\nu\colon\D\to(0,\infty)$ continuous and 
$U:=\{0\}\cup(\{1\}\times U_{\ast})$ fix the topology in $\mathcal{B}\nu(\D)$ with $U_{\ast}\subset\D$. 
If $f\colon U\to E$ is a function in $\ell\nu_{\ast}(U,E)$ such that there is 
$f_{e'}\in\mathcal{B}\nu(\D)$ for each $e'\in G$ with $f_{e'}(0)=e'(f(0))$ and $f_{e'}'(z)=e'(f(1,z))$ 
for all $z\in U_{\ast}$, then there exists a unique $F\in\mathcal{B}\nu(\D,E)$ with $F(0)=f(0)$ and 
$(\partial_{\C}^{1})^{E}F(z)=f(1,z)$ for all $z\in U_{\ast}$.
\end{cor}
\begin{proof}
As in \prettyref{cor:hypo_weighted_ext_fix_top} but with $\mathcal{F}\nu_{\ast}(\D)=\mathcal{B}\nu(\D)$ 
and \prettyref{cor:Bloch_ext_unique} instead of \prettyref{cor:hypo_weighted_ext_unique}.
\end{proof}

Sets that fix the topology in $\mathcal{B}\nu(\D)$ play an important role in the characterisation of 
composition operators on $\mathcal{B}\nu(\D)$ with closed range. 
Chen and Gauthier give a characterisation in \cite{chengauthier2008} for weights of the form 
$\nu(z)=(1-|z|^{2})^{\alpha}$, $z\in\D$, for some $\alpha\geq 1$. 
We recall the following definitions which are needed to phrase this characterisation.
For a continuous function $\nu\colon\D\to(0,\infty)$ and a non-constant holomorphic function $\phi\colon\D\to\D$ we set 
\[
\tau^{\nu}_{\phi}(z):=\frac{\nu(z)|\phi'(z)|}{\nu(\phi(z))},\;z\in\D,\quad\text{and}\quad
\Omega^{\nu}_{\varepsilon}:=\{z\in\D\;|\;\tau^{\nu}_{\phi}(z)\geq\varepsilon\},\;\varepsilon>0,
\]
and define the \emph{pseudohyperbolic distance}
\[
\rho(z,w):=\Bigl|\frac{z-w}{1-\overline{z}w}\Bigr|,\;z,w\in\D.
\]
For $0<r<1$ a set $B\subset\D$ is called a \emph{pseudo} $r$\emph{-net} 
if for every $w\in\D$ there is $z\in B$ with $\rho(z,w)\leq r$ (see \cite[p.\ 195]{chengauthier2008}). A set $U_{\star}\subset \D$ is a sampling set for $\mathcal{B}\nu(\D)$ 
with $\nu$ as above in the sense of \cite[p.\ 198]{chengauthier2008} if and only if 
$\{0\}\cup(\{1\}\times U_{\star})$ fixes the topology in $\mathcal{B}\nu(\D)$ 
(see the definitions above \prettyref{cor:Bloch_ext_unique}).

\begin{thm}[{\cite[Theorem 3.1, p.\ 199, Theorem 4.3, p.\ 202]{chengauthier2008}}]\label{thm:Bloch_sampling}
Let $\phi\colon\D\to\D$ be a non-constant holomorphic function and $\nu(z)=(1-|z|^{2})^{\alpha}$, $z\in\D$, 
for some $\alpha\geq 1$. Then the following statements are equivalent.
\begin{enumerate}
\item[(i)] The composition operator $C_{\phi}\colon\mathcal{B}\nu(\D)\to\mathcal{B}\nu(\D)$, $C_{\phi}(f):=f\circ\phi$, is bounded below 
(i.e.\ has closed range).
\item[(ii)] There is $\varepsilon>0$ such that $\{0\}\cup(\{1\}\times \phi(\Omega^{\nu}_{\varepsilon}))$ fixes the topology in $\mathcal{B}\nu(\D)$.
\item[(iii)] There are $\varepsilon>0$ and $0<r<1$ such that $\phi(\Omega^{\nu}_{\varepsilon})$ is a pseudo $r$-net.
\end{enumerate}
\end{thm}

This theorem has some predecessors. The implications (i)$\Rightarrow$(iii) and 
(iii), $r<1/4\Rightarrow$(i) 
for $\alpha=1$ are due to Ghatage, Yan and Zheng by \cite[Proposition 1, p.\ 2040]{ghatage2001} 
and \cite[Theorem 2, p.\ 2043]{ghatage2001}. 
This was improved by Chen to (i)$\Leftrightarrow$(iii) for $\alpha=1$ by removing the restriction $r<1/4$ 
in \cite[Theorem 1, p.\ 840]{chen2003}. 
The proof of the equivalence (i)$\Leftrightarrow$(ii) given in \cite[Theorem 1, p.\ 1372]{ghatage2005} for $\alpha=1$ 
is due to Ghatage, Zheng and Zorboska. 
A non-trivial example of a sampling set for $\alpha=1$ can be found in \cite[Example 2, p.\ 1376]{ghatage2005} 
(cf.\ \cite[p.\ 203]{chengauthier2008}).
In the case of several complex variables a characterisation corresponding to \prettyref{thm:Bloch_sampling}
is given by Chen in \cite[Theorem 2, p.\ 844]{chen2003} and Deng, Jiang and Ouyang 
in \cite[Theorem 1--3, p.\ 1031--1032, 1034]{deng2007} where $\Omega$ is the unit ball of $\C^{d}$.
Gim{\'e}nez, Malav{\'e} and Ramos-Fern{\'a}ndez extend \prettyref{thm:Bloch_sampling} by 
\cite[Theorem 3, p.\ 112]{gimenez2010} and \cite[Corollary 1, p.\ 113]{gimenez2010} to more general weights 
of the form $\nu(z)=\mu(1-|z|^{2})$ with some continuous function $\mu\colon (0,1]\to(0,\infty)$ 
such that $\mu(r)\to 0$, $r\to 0\vcenter{\hbox{${\scriptstyle{+}}$}}$, 
which can be extended to a holomorphic function $\mu_{0}$ on $\D_{1}(1)$ without zeros 
in $\D_{1}(1)$ and fulfilling $\mu(1-|1-z|)\leq C|\mu_{0}(z)|$ for all $z\in\D_{1}(1)$ and some $C>0$ 
(see \cite[p.\ 109]{gimenez2010}). 
Examples of such functions $\mu$ are $\mu_{1}(r):=r^{\alpha}$, $\alpha>0$, $\mu_{2}:=r\ln(2/r)$ 
and $\mu_{3}(r):=r^{\beta}\ln(1-r)$, $\beta>1$, for $r\in(0,1]$ (see \cite[p.\ 110]{gimenez2010}) 
and with $\nu(z)=\mu_{1}(1-|z|^{2})=(1-|z|^{2})^{\alpha}$, $z\in\D$, one gets \prettyref{thm:Bloch_sampling} 
back for $\alpha\geq 1$. For $0<\alpha<1$ and $\nu(z)=\mu_{1}(1-|z|^{2})$, $z\in\D$, 
the equivalence (i)$\Leftrightarrow$(ii) is given in \cite[Proposition 4.4, p.\ 14]{yoneda2018} of Yoneda as well and 
a sufficient condition implying (ii) in \cite[Corollary 4.5, p.\ 15]{yoneda2018}.
Ramos-Fern{\'a}ndez generalises the results given in \cite{gimenez2010} to bounded essential weights $\nu$ on $\D$ by 
\cite[Theorem 4.3, p.\ 85]{fernandez2011} and \cite[Remark 4.2, p.\ 84]{fernandez2011}.
In \cite[Theorem 2.4, p.\ 3106]{pirasteh2018} Pirasteh, Eghbali and Sanatpour use sets that fix the topology 
in $\mathcal{B}\nu(\D)$ for radial essential $\nu$ to characterise Li--Stevi\'{c} integral-type operators 
on $\mathcal{B}\nu(\D)$ with closed range instead of composition operators. 
The composition operator on the harmonic variant of the Bloch type space $\mathcal{B}\nu(\D)$ 
with $\nu(z)=(1-|z|^{2})^{\alpha}$, $z\in\D$, for some $\alpha>0$ is considered by Esmaeili, Estaremi and Ebadian, 
who give a corresponding result in \cite[Theorem 2.8, p.\ 542]{esmaeili2018}.
\section{Weak-strong principles for differentiability of finite order}
\label{sect:weak_strong_finite_order}
This section is dedicated to $\mathcal{C}^{k}$-weak-strong principles for differentiable functions. 
So the question is:

\begin{que}
Let $E$ be an lcHs, $G\subset E'$ a separating subspace, $\Omega\subset\R^{d}$ open 
and $k\in\N_{0}\cup\{\infty\}$. If $f\colon\Omega\to E$ is such that $e'\circ f\in\mathcal{C}^{k}(\Omega)$ 
for each $e'\in G$, does $f\in\mathcal{C}^{k}(\Omega,E)$ hold?
\end{que}

An affirmative answer to the preceding question is called a $\mathcal{C}^{k}$-weak-strong principle. 
It is a result of Bierstedt \cite[2.10 Lemma, p.\ 140]{B2} that for $k=0$ the $\mathcal{C}^{0}$-weak-strong 
principle holds if $\Omega\subset\R^{d}$ is open (or more general a Hausdorff $k_{\R}$-space), $G=E'$ and 
$E$ is such that every bounded set is already precompact in $E$, i.e.\ $E$ is a generalised Schwartz space (see \prettyref{def:general_schwartz} and \prettyref{rem:bierstedt_weak_strong}). For instance, the last condition is fulfilled if $E$ is 
a semi-Montel or Schwartz space. The $\mathcal{C}^{0}$-weak-strong principle does not hold for general $E$ by 
\cite[Beispiel, p.\ 232]{Kaballo}.

Grothendieck sketches in a footnote \cite[p.\ 39]{Grothendieck1953} 
(cf.\ \cite[Chap.\ 3, Sect.\ 8, Corollary 1, p.\ 134]{gr73}) 
the proof that for $k<\infty$ a weakly-$\mathcal{C}^{k+1}$ function $f\colon\Omega\to E$ on an open set 
$\Omega\subset\R^{d}$ with values in a quasi-complete lcHs $E$ is already $\mathcal{C}^{k}$, i.e.\ that from 
$e'\circ f\in\mathcal{C}^{k+1}(\Omega)$ for all $e'\in E'$ it follows $f\in\mathcal{C}^{k}(\Omega,E)$. 
A detailed proof of this statement is given by Schwartz in \cite{Schwartz1955}, simultaneously weakening 
the condition from quasi-completeness of $E$ to sequential completeness and from 
weakly-$\mathcal{C}^{k+1}$ to weakly-$\mathcal{C}^{k,1}_{loc}$.

\begin{thm}[{\cite[Appendice, Lemme II, Remarques 1$^0$), p.\ 146--147]{Schwartz1955}}]\label{thm:schwartz_weak_strong}
Let $E$ be a sequentially complete lcHs, $\Omega\subset\R^{d}$ open and $k\in\N_{0}$. 
\begin{enumerate}
\item [a)] If $f\colon\Omega\to E$ is such that $e'\circ f\in \mathcal{C}^{k,1}_{loc}(\Omega)$ for all $e'\in E'$, 
then $f\in\mathcal{C}^{k}(\Omega,E)$. 
\item [b)] If $f\colon\Omega\to E$ is such that $e'\circ f\in \mathcal{C}^{k+1}(\Omega)$ for all $e'\in E'$, 
then $f\in\mathcal{C}^{k}(\Omega,E)$. 
\end{enumerate}
\end{thm} 

Here $\mathcal{C}^{k,1}_{loc}(\Omega)$ denotes the space of functions in $\mathcal{C}^{k}(\Omega)$ 
whose partial derivatives of order $k$ are locally Lipschitz continuous. 
Part b) clearly implies a $\mathcal{C}^{\infty}$-weak-strong principle for 
open $\Omega\subset\R^{d}$, $G=E'$ and sequentially complete $E$. This can be generalised to locally complete $E$.
Waelbroeck has shown in \cite[Proposition 2, p.\ 411]{waelbroeck1967_2} and 
\cite[Definition 1, p.\ 393]{waelbroeck1967_1} 
that the $\mathcal{C}^{\infty}$-weak-strong principle holds if $\Omega$ is a manifold, $G=E'$ 
and $E$ is locally complete.
It is a result of Bonet, Frerick and Jord\'a that the $\mathcal{C}^{\infty}$-weak-strong principle still holds by 
\cite[Theorem 9, p.\ 232]{B/F/J} if $\Omega\subset\R^{d}$ is open, $G\subset E'$ determines boundedness 
and $E$ is locally complete.
Due to \cite[2.14 Theorem, p.\ 20]{kriegl} of Kriegl and Michor an lcHs $E$ is locally complete if and 
only if the $\mathcal{C}^{\infty}$-weak-strong principle holds for $\Omega=\R$ and $G=E'$. 

One of the goals of this section is to improve \prettyref{thm:schwartz_weak_strong}. 
We recall the following definition from \prettyref{ex:diff_ext_boundary}. 
For $k\in\N_{0}$ the space of $k$-times continuously 
partially differentiable $E$-valued functions on an open set $\Omega\subset\R^{d}$ whose partial derivatives 
up to order $k$ are continuously extendable to the boundary of $\Omega$ is
\[
 \mathcal{C}^{k}(\overline{\Omega},E)=\{f\in\mathcal{C}^{k}(\Omega,E)\;|\;(\partial^{\beta})^{E}f\;
 \text{cont.\ extendable on}\;\overline{\Omega}\;\text{for all}\;\beta\in\N^{d}_{0},\,|\beta|\leq k\}
\]
equipped with the system of seminorms given by 
\[
 |f|_{\mathcal{C}^{k}(\overline{\Omega}),\alpha}=\sup_{\substack{x\in \Omega\\ \beta\in\N^{d}_{0}, |\beta|\leq k}}
 p_{\alpha}((\partial^{\beta})^{E}f(x)), \quad f\in\mathcal{C}^{k}(\overline{\Omega},E),\, \alpha\in\mathfrak{A}.
\]
The space of functions in $\mathcal{C}^{k}(\overline{\Omega},E)$ such 
that all its $k$-th partial derivatives are $\gamma$-H\"older continuous with $0<\gamma\leq 1$ is given by  
\[
\gls{Ckgamma}:=
\bigl\{f\in\mathcal{C}^{k}(\overline{\Omega},E)\;|\;\forall\;\alpha\in\mathfrak{A}:\;
|f|_{\mathcal{C}^{k,\gamma}(\overline{\Omega}),\alpha}<\infty\bigr\}
\]
where
\[
|f|_{\mathcal{C}^{k,\gamma}(\overline{\Omega}),\alpha}:=\max\Bigl(|f|_{\mathcal{C}^{k}(\overline{\Omega}),\alpha},
\sup_{\beta\in\N^{d}_{0}, |\beta|=k}|(\partial^{\beta})^{E}f|_{\mathcal{C}^{0,\gamma}(\Omega),\alpha}\Bigr)
\]
with
\[
|f|_{\mathcal{C}^{0,\gamma}(\Omega),\alpha}:=\sup_{\substack{x,y\in\Omega\\x\neq y}}
\frac{p_{\alpha}(f(x)-f(y))}{|x-y|^{\gamma}}.
\]
We set $\mathcal{C}^{k,\gamma}(\overline{\Omega}):=\mathcal{C}^{k,\gamma}(\overline{\Omega},\K)$ and 
\[
\omega_{1}:=\{\beta\in\N_{0}^{d}\;|\;|\beta|\leq k\}\times\Omega\quad\text{and}\quad
\omega_{2}:=\{\beta\in\N_{0}^{d}\;|\;|\beta|=k\}\times (\Omega^{2}\setminus\{(x,x)\;|\;x\in\Omega\})
\]
as well as $\omega:=\omega_{1}\cup\omega_{2}$. We define the operator 
$T^{E}\colon\mathcal{C}^{k}(\Omega,E)\to E^{\omega}$ by
\begin{align*}
T^{E}(f)(\beta,x):=&(\partial^{\beta})^{E}(f)(x) &&,\;(\beta,x)\in\omega_{1},\\
T^{E}(f)(\beta,(x,y)):=&(\partial^{\beta})^{E}(f)(x)-(\partial^{\beta})^{E}(f)(y) &&,\;(\beta,(x,y))\in\omega_{2}.
\end{align*}
and the weight $\nu\colon\omega\to (0,\infty)$ by 
\[
\nu(\beta,x):=1,\;(\beta,x)\in\omega_{1},\quad\text{and}\quad \nu(\beta,(x,y)):=\frac{1}{|x-y|^{\gamma}},
\;(\beta,(x,y))\in\omega_{2}.
\]
By setting $\fe:=\mathcal{C}^{k}(\overline{\Omega},E)$ and observing that
\[
|f|_{\mathcal{C}^{k,\gamma}(\overline{\Omega}),\alpha}=\sup_{x\in\omega}p_{\alpha}(T^{E}(f)(x))\nu(x),
\quad f\in\mathcal{C}^{k,\gamma}(\overline{\Omega},E),\,\alpha\in\mathfrak{A},
\]
we have $\mathcal{F}\nu(\Omega,E)=\mathcal{C}^{k,\gamma}(\overline{\Omega},E)$ with generator $(T^{E},T^{\K})$. 

\begin{cor}\label{cor:ext_B_unique}
Let $E$ be a locally complete lcHs, $G\subset E'$ determine boundedness, $\Omega\subset\R^{d}$ open and bounded, 
$k\in\N_{0}$ and $0<\gamma\leq 1$. In the case $k\geq 1$, assume additionally that $\Omega$ has Lipschitz boundary. 
If $f\colon\Omega\to E$ is such that $e'\circ f\in \mathcal{C}^{k,\gamma}(\overline{\Omega})$ for all $e'\in G$, 
then $f\in\mathcal{C}^{k,\gamma}(\overline{\Omega},E)$. 
\end{cor}
\begin{proof}
We take $\f:=\mathcal{C}^{k}(\overline{\Omega})$ and $\fe:=\mathcal{C}^{k}(\overline{\Omega},E)$ 
and have $\mathcal{F}\nu(\Omega)=\mathcal{C}^{k,\gamma}(\overline{\Omega})$ 
and $\mathcal{F}\nu(\Omega,E)=\mathcal{C}^{k,\gamma}(\overline{\Omega},E)$ with the weight 
$\nu$ and generator $(T^{E},T^{\K})$ for $(\mathcal{F}\nu,E)$ described above.
Due to the proof of \prettyref{ex:diff_ext_boundary} and \prettyref{thm:linearisation} 
the spaces $\f$ and $\fe$ are $\varepsilon$-into-compatible for any lcHs $E$ 
(the condition that $E$ has metric ccp in \prettyref{ex:diff_ext_boundary} is only needed for 
$\varepsilon$-compatibility). 
Another consequence of \prettyref{ex:diff_ext_boundary} is that 
\[
T^{E}(S(u))(\beta,x)=(\partial^{\beta})^{E}(S(u))(x)=u(\delta_{x}\circ(\partial^{\beta})^{\K})=u(T^{\K}_{\beta,x}),
\quad (\beta,x)\in\omega_{1},
\]
holds for all $u\in\f\varepsilon E$, implying 
\begin{align*}
T^{E}(S(u))(\beta,(x,y))&=T^{E}(S(u))(\beta,x)-T^{E}(S(u))(\beta,y)=u(T^{\K}_{\beta,x})-u(T^{\K}_{\beta,y})\\
&=u(T^{\K}_{\beta,(x,y)}),\quad(\beta,(x,y))\in\omega_{2}.
\end{align*}
Thus $(T^{E},T^{\K})$ is a consistent family for $(F,E)$ and its strength is easily seen.
In addition, $\mathcal{F}\nu(\Omega)=\mathcal{C}^{k,\gamma}(\overline{\Omega})$ is a Banach space 
by \cite[Theorem 9.8, p.\ 110]{driver2004} (cf.\ \cite[1.7 H\"olderstetige Funktionen, p.\ 46]{alt2012}) 
whose closed unit ball is compact in $\f=\mathcal{C}^{k}(\overline{\Omega})$ 
by \cite[8.6 Einbettungssatz in H\"older-R\"aumen, p.\ 338]{alt2012}. 
Moreover, the $\varepsilon$-into-compatibility of $\f$ and $\fe$ 
in combination with the consistency of $(T^{E},T^{\K})$ for $(F,E)$ implies 
$\Feps\subset\mathcal{F}\nu(\Omega,E)$ as linear spaces by \prettyref{prop:mingle-mangle} c). 
Hence our statement follows from \prettyref{thm:ext_B_unique}
with the set of uniqueness $U:=\{0\}\times\Omega$ for $(T^{\K},\mathcal{F}\nu)$.
\end{proof}

\begin{rem}
We point out that \prettyref{cor:ext_B_unique} corrects our result \cite[Corollary 5.3, p.\ 16]{kruse2019_3} 
by adding the missing assumption that $\Omega$ should additionally
have Lipschitz boundary in the case $k\geq 1$. This is needed to deduce that the closed unit ball of 
$\mathcal{C}^{k,\gamma}(\overline{\Omega})$ is compact in $\mathcal{C}^{k}(\overline{\Omega})$ 
by \cite[8.6 Einbettungssatz in H\"older-R\"aumen, p.\ 338]{alt2012} (in the notation of \cite{gilbarg_trudinger2001} 
$\Omega$ having Lipschitz boundary means that it is a $\mathcal{C}^{0,1}$ domain, see 
\cite[Lemma 6.36, p.\ 136]{gilbarg_trudinger2001} and the comments below and above this lemma). 
This additional assumption is missing in \cite[Theorem 14.32, p.\ 232]{driver2004}, 
which is our main reference in \cite{kruse2019_3} for the compact embedding, 
but it is needed due to \cite[U8.1 Gegenbeispiel zu Einbettungss\"atzen, p.\ 365]{alt2012} 
(cf.\ \cite[p.\ 53]{gilbarg_trudinger2001}).  
However, this only affects the result \cite[Corollary 6.3, p.\ 21--22]{kruse2019_3} 
where we have to add this missing assumption as well 
(see \prettyref{cor:Hoelder_Blaschke} for this). The other results of \cite{kruse2019_3} 
derived from \cite[Corollary 5.3, p.\ 16]{kruse2019_3} are not affected by this missing assumption since 
they are all a consequence of \cite[Corollary 5.4, p.\ 17]{kruse2019_3} and \cite[Corollary 6.4, p.\ 22]{kruse2019_3}, 
whose proofs can be adjusted without additional 
assumptions (see \prettyref{cor:ext_B_unique_loc_Hoelder} and \prettyref{cor:loc_Hoelder_Blaschke} for this).
\end{rem}

Next, we use the preceding corollary to generalise the theorem of Grothendieck and Schwartz on weakly 
$\mathcal{C}^{k+1}$-functions.
For $k\in\N_{0}$ and $0<\gamma\leq 1$ we define the space of $k$-times continuously partially differentiable 
$E$-valued functions with locally $\gamma$-H\"older continuous partial derivatives of $k$-th order
on an open set $\Omega\subset\R^{d}$ by
\[
 \gls{Ckgammaloc}:=\{f\in\mathcal{C}^{k}(\Omega,E)\;|\;
 \forall\;K\subset\Omega\;\text{compact},\,\alpha\in\mathfrak{A}:\;|f|_{K,\alpha}<\infty\}
\]
where 
\[
|f|_{K,\alpha}:=\max\Bigl(|f|_{\mathcal{C}^{k}(K),\alpha},\,\sup_{\beta\in\N^{d}_{0}, |\beta|= k}
|(\partial^{\beta})^{E}f|_{\mathcal{C}^{0,\gamma}(K),\alpha}\Bigr)
\]
with
\[
 |f|_{\mathcal{C}^{k}(K),\alpha}:=\sup_{\substack{x\in K\\ \beta\in\N^{d}_{0}, |\beta|\leq k}}
 p_{\alpha}((\partial^{\beta})^{E}f(x))
\]
and 
\[
|f|_{\mathcal{C}^{0,\gamma}(K),\alpha}:=\sup_{\substack{x,y\in K\\x\neq y}}
\frac{p_{\alpha}(f(x)-f(y))}{|x-y|^{\gamma}}.
\]

Further, we set $\mathcal{C}^{k,\gamma}_{loc}(\Omega):=\mathcal{C}^{k,\gamma}_{loc}(\Omega,\K)$. 
Using \prettyref{cor:ext_B_unique}, we are able to improve \prettyref{thm:schwartz_weak_strong} 
to the following form.  

\begin{cor}\label{cor:ext_B_unique_loc_Hoelder}
Let $E$ be a locally complete lcHs, $G\subset E'$ determine boundedness, $\Omega\subset\R^{d}$ open, 
$k\in\N_{0}$ and $0<\gamma\leq 1$. 
\begin{enumerate}
\item [a)] If $f\colon\Omega\to E$ is such that $e'\circ f\in \mathcal{C}^{k,\gamma}_{loc}(\Omega)$ 
for all $e'\in G$, then $f\in\mathcal{C}^{k,\gamma}_{loc}(\Omega,E)$. 
\item [b)] If $f\colon\Omega\to E$ is such that $e'\circ f\in \mathcal{C}^{k+1}(\Omega)$ for all $e'\in G$, 
then $f\in\mathcal{C}^{k,1}_{loc}(\Omega,E)$. 
\end{enumerate}
\end{cor}
\begin{proof}
Let us start with a). Let $f\colon\Omega\to E$ be such that $e'\circ f\in \mathcal{C}^{k,\gamma}_{loc}(\Omega)$ 
for all $e'\in G$. Let $(\Omega_{n})_{n\in\N}$ be an exhaustion of $\Omega$ with open, relatively compact sets 
$\Omega_{n}\subset\Omega$ with Lipschitz boundaries $\partial\Omega_{n}$ (e.g.\ choose each $\Omega_{n}$ as the interior 
of a finite union of closed axis-parallel cubes, see the proof of \cite[Theorem 1.4, p.\ 7]{stein2005} for the construction) 
that satisfies $\Omega_{n}\subset\Omega_{n+1}$ for all $n\in\N$. 
Then the restriction of $e'\circ f$ to $\Omega_{n}$ is an element of $\mathcal{C}^{k,\gamma}(\overline{\Omega}_{n})$ 
for every $e'\in G$ and $n\in\N$. Due to \prettyref{cor:ext_B_unique} we obtain that 
$f\in\mathcal{C}^{k,\gamma}(\overline{\Omega}_{n},E)$ for every $n\in\N$. 
Thus $f\in\mathcal{C}^{k,\gamma}_{loc}(\Omega,E)$ 
since differentiability is a local property and for each compact $K\subset\Omega$ there is $n\in\N$ 
such that $K\subset \Omega_{n}$.

Let us turn to b), i.e.\ let $f\colon\Omega\to E$ be such that $e'\circ f\in\mathcal{C}^{k+1}(\Omega)$ 
for all $e'\in G$. 
Since $\Omega\subset\R^{d}$ is open, for every $x\in\Omega$ there is $\varepsilon_{x}>0$ such that 
$\overline{\mathbb{B}_{\varepsilon_{x}}(x)}\subset\Omega$.
For all $e'\in G$, $\beta\in\N_{0}^{d}$ with $|\beta|=k$ 
and $w,y\in\overline{\mathbb{B}_{\varepsilon_{x}}(x)}$, $w\neq y$, it holds that
\[
 \frac{|(\partial^{\beta})^{\K}(e'\circ f)(w)-(\partial^{\beta})^{\K}(e'\circ f)(y)|}{|w-y|}
\leq C_{d}\max_{1\leq n\leq d}
 \max_{z\in\overline{\mathbb{B}_{\varepsilon_{x}}(x)}}|(\partial^{\beta+e_{n}})^{\K}(e'\circ f)(z)|
\]
by the mean value theorem applied to the real and imaginary part 
where $C_{d}:=\sqrt{d}$ if $\K=\R$, and $C_{d}:=2\sqrt{d}$ if $\K=\C$. 
Thus $e'\circ f\in \mathcal{C}^{k,1}_{loc}(\Omega)$ for all $e'\in G$ 
since for each compact set $K\subset\Omega$ there are $m\in\N$ and $x_{i}\in\Omega$, $1\leq i\leq m$, 
such that $K\subset \bigcup_{i=1}^{m}\mathbb{B}_{\varepsilon_{x_{i}}}(x_{i})$.
It follows from part a) that $f\in\mathcal{C}^{k,1}_{loc}(\Omega,E)$. 
\end{proof}

If $\Omega=\R$, $\gamma=1$ and $G=E'$, then part a) of \prettyref{cor:ext_B_unique_loc_Hoelder} is already known 
by \cite[2.3 Corollary, p.\ 15]{kriegl}. 
A `full' $\mathcal{C}^{k}$-weak-strong principle for $k<\infty$, i.e.\ the conditions of part b) 
imply $f\in\mathcal{C}^{k+1}(\Omega,E)$, does not hold in general (see \cite[p.\ 11--12]{kriegl}) 
but it holds if we restrict the class of admissible lcHs $E$.

\begin{thm}\label{thm:weak_strong_finite_order}
Let $E$ be a semi-Montel space, $G\subset E'$ determine boundedness, $\Omega\subset\R^{d}$ open and
$k\in\N$. If $f\colon\Omega\to E$ is such that $e'\circ f\in \mathcal{C}^{k}(\Omega)$ for all $e'\in G$, 
then $f\in\mathcal{C}^{k}(\Omega,E)$. 
\end{thm}
\begin{proof}
Let $f\colon\Omega\to E$ be such that $e'\circ f\in \mathcal{C}^{k}(\Omega)$ for all $e'\in G$. 
Due to \prettyref{cor:ext_B_unique_loc_Hoelder} b) we already know that $f\in\mathcal{C}^{k-1,1}_{loc}(\Omega,E)$ since 
semi-Montel spaces are quasi-complete and thus locally complete. 
Now, let $x\in\Omega$, $\varepsilon_{x}>0$ such that $\overline{\mathbb{B}_{\varepsilon_{x}}(x)}\subset\Omega$, 
$\beta\in\N_{0}^{d}$ with $|\beta|=k-1$ and $n\in\N$ with $1\leq n\leq d$. 
The set
\[
B:=\Bigl\{\frac{(\partial^{\beta}f)^{E}(x+he_{n})-(\partial^{\beta}f)^{E}f(x)}{h}\;|\;
          h\in\R,\,0<h\leq\varepsilon_{x}\Bigr\}
\]
is bounded in $E$ because $f\in\mathcal{C}^{k-1,1}_{loc}(\Omega,E)$. As $E$ is a semi-Montel space, 
the closure $\overline{B}$ is compact in $E$. 
Let $(h_{m})_{m\in\N}$ be a sequence in $\R$ such that $0<h_{m}\leq\varepsilon_{x}$ for all $m\in\N$. 
From the compactness of $\overline{B}$ we deduce that there is a subnet $(h_{m_{\iota}})_{\iota\in I}$ of $(h_{m})_{m\in\N}$ and $y_{x}\in\overline{B}$ with 
\[
y_{x}=\lim_{\iota\in I}\frac{(\partial^{\beta}f)^{E}(x+h_{m_{\iota}}e_{n})-(\partial^{\beta}f)^{E}f(x)}{h_{m_{\iota}}}
=:\lim_{\iota\in I}y_{\iota}.
\]
Further, we note that the limit
\begin{equation}\label{eq:weak_limit_1}
 (\partial^{\beta+e_{n}})^{\K}(e'\circ f)(x)
=\lim_{\substack{h\to 0\\ h\in\R, h\neq 0}}
 \frac{\partial^{\beta}(e'\circ f)(x+he_{n})-\partial^{\beta}(e'\circ f)(x)}{h}
\end{equation}
exists for all $e'\in G$ and that $(e'(y_{\iota}))_{\iota\in I}$ is a subnet of the net of difference quotients 
on the right-hand side of \eqref{eq:weak_limit_1} as $(\partial^{\beta})^{\K}(e'\circ f)=e'\circ (\partial^{\beta})^{E} f$.
Therefore 
\begin{align}\label{eq:weak_limit_2}
  (\partial^{\beta+e_{n}})^{\K}(e'\circ f)(x)
&=\lim_{\substack{h\to 0\\ h\in\R, h\neq 0}}
   e'\Bigl(\frac{(\partial^{\beta})^{E}f(x+he_{n})-(\partial^{\beta})^{E}f(x)}{h}\Bigr)\notag\\
&=\lim_{\substack{h\to 0\\ h\in\R, 0<h\leq\varepsilon_{x}}}
   e'\Bigl(\frac{(\partial^{\beta})^{E}f(x+he_{n})-(\partial^{\beta})^{E}f(x)}{h}\Bigr)
 =\lim_{\iota\in I}e'(y_{\iota})\notag\\
&=e'(y_{x})
\end{align}
for all $e'\in G$. By \prettyref{prop:general_Schwartz} (i) the topology $\sigma(E,G)$ 
and the initial topology of $E$ coincide on $\overline{B}$. Combining this fact with
\eqref{eq:weak_limit_2}, we deduce that 
\[
 (\partial^{\beta+e_{n}})^{E}f(x)
=\lim_{\substack{h\to 0\\ h\in\R, h\neq 0}}\frac{(\partial^{\beta})^{E}f(x+he_{n})-(\partial^{\beta})^{E}f(x)}{h}
=y_{x}.
\]
In addition, $e'\circ(\partial^{\beta+e_{n}})^{E}f=(\partial^{\beta+e_{n}})^{\K}(e'\circ f)$ is continuous on 
$\overline{\mathbb{B}_{\varepsilon_{x}}(x)}$ for all $e'\in G$, meaning that the restriction of 
$(\partial^{\beta+e_{n}})^{E}f$ on $\overline{\mathbb{B}_{\varepsilon_{x}}(x)}$ to $(E,\sigma(E,G))$
is continuous, and the range $(\partial^{\beta+e_{n}})^{E}f(\overline{\mathbb{B}_{\varepsilon_{x}}(x)})$ is bounded 
in $E$. As before we use that $\sigma(E,G)$ and the initial topology of $E$ coincide on 
$(\partial^{\beta+e_{n}})^{E}f(\overline{\mathbb{B}_{\varepsilon_{x}}(x)})$, 
which implies that the restriction of $(\partial^{\beta+e_{n}})^{E}f$ on $\overline{\mathbb{B}_{\varepsilon_{x}}(x)}$ 
is continuous w.r.t.\ the initial topology of $E$. 
Since continuity is a local property and $x\in\Omega$ is arbitrary, we conclude that $(\partial^{\beta+e_{n}})^{E}f$ 
is continuous on $\Omega$.
\end{proof}

In the special case that $\Omega=\R$, $G=E'$ and $E$ is a Montel space, i.e.\ a barrelled semi-Montel space, 
a different proof of the preceding weak-strong principle can be found in the proof of 
\cite[Lemma 4, p.\ 15]{carroll1961}. 
This proof uses the Banach--Steinhaus theorem and needs the barrelledness of the Montel space $E_{b}'$. 
Our weak-strong principle \prettyref{thm:weak_strong_finite_order} does not need the barrelledness of $E$, 
hence can be applied to $E=(\mathcal{C}^{\infty}_{\overline{\partial},b}(\D),\beta)$ 
which is a non-barrelled semi-Montel space by \prettyref{rem:strict_top_non_barrelled}
and \prettyref{prop:strict_topo_hypo_B_complete_semi_Montel}.

Besides the `full' $\mathcal{C}^{k}$-weak-strong principle for $k<\infty$ and semi-Montel $E$, part b) 
of \prettyref{cor:ext_B_unique_loc_Hoelder} also suggests an `almost' $\mathcal{C}^{k}$-weak-strong principle 
in terms of \cite[3.1.6 Rademacher's theorem, p.\ 216]{federer1969}, which we prepare next.

\begin{defn}[{generalised Gelfand space}]
We say that an lcHs $E$ is a \emph{\gls{gen_Gelfand_space}}
if every Lipschitz continuous map $f\colon [0,1]\to E$ is differentiable almost everywhere w.r.t\ to 
the one-dimensional Lebesgue measure. 
\end{defn}

If $E$ is a real Fr\'echet space ($\K=\R$), then this definition coincides with 
the definition of a Fr\'echet--Gelfand space given in \cite[Definition 2.2, p.\ 17]{mankiewicz1973}. In particular, 
every real nuclear Fr\'echet lattice (see \cite[Theorem 6, Corollary, p.\ 375, 378]{grosse-erdmann1991}) 
and more general every real Fr\'echet--Montel space is a generalised Gelfand space 
by \cite[Theorem 2.9, p.\ 18]{mankiewicz1973}. 
If $E$ is a Banach space, then this definition coincides with the definition of a Gelfand space given 
in \cite[Definition 4.3.1, p.\ 106--107]{diesteluhl1977} by \cite[Proposition 1.2.4, p.\ 18]{arendt2011}. 
A Banach space is a Gelfand space if and only if it has the Radon--Nikod\'ym 
property by \cite[Theorem 4.3.2, p.\ 107]{diesteluhl1977}. 
Thus separable duals of Banach spaces, reflexive Banach spaces and $\ell^{1}(\Gamma)$ 
for any set $\Gamma$ are generalised Gelfand spaces by \cite[Theorem 3.3.1 (Dunford-Pettis), p.\ 79]{diesteluhl1977}, 
\cite[Corollary 3.3.4 (Phillips), p.\ 82]{diesteluhl1977} and \cite[Corollary 3.3.8, p.\ 83]{diesteluhl1977}. 
The Banach spaces $c_{0}$, $\ell^{\infty}$, $\mathcal{C}([0,1])$, $\mathcal{L}^{1}([0,1])$ and $\mathcal{L}^{\infty}([0,1])$ do not have the 
Radon--Nikod\'ym property and hence are not generalised Gelfand spaces by \cite[Proposition 1.2.9, p.\ 20]{arendt2011}, 
\cite[Example 1.2.8, p.\ 20]{arendt2011} and \cite[Proposition 1.2.10, p.\ 21]{arendt2011}.

\begin{cor}\label{cor:allmost_weak_strong_finite_order}
Let $E$ be a locally complete generalised Gelfand space, $G\subset E'$ determine boundedness, $\Omega\subset\R$ open and
$k\in\N$. If $f\colon\Omega\to E$ is such that $e'\circ f\in \mathcal{C}^{k}(\Omega)$ for all $e'\in G$, 
then $f\in\mathcal{C}^{k-1,1}_{loc}(\Omega,E)$ and the derivative 
$(\partial^{k})^{E}f(x)$ exists for Lebesgue almost all $x\in\Omega$.
\end{cor}
\begin{proof}
The first part follows from \prettyref{cor:ext_B_unique_loc_Hoelder} b). 
Now, let $[a,b]\subset\Omega$ be a bounded interval. 
We set $F\colon[0,1]\to E$, $F(x):=(\partial^{k-1})^{E}f(a+x(b-a))$. 
Then $F$ is Lipschitz continuous as $f\in\mathcal{C}^{k-1,1}_{loc}(\Omega,E)$. 
This yields that $F$ is differentiable on $[0,1]$ almost everywhere because $E$ is a generalised Gelfand space, 
implying that $(\partial^{k-1})^{E}f$ is differentiable on $[a,b]$ almost everywhere. 
Since the open set $\Omega\subset\R$ can be written as a countable union of disjoint open intervals 
$I_{n}$, $n\in\N$, and each $I_{n}$ is a countable union of closed bounded intervals 
$[a_{m},b_{m}]$, $m\in\N$, our statement follows from the fact that the countable union of null sets is a null set.
\end{proof}

To the best of our knowledge there are still some open problems for continuously partially differentiable functions 
of finite order. 

\begin{que}
\begin{enumerate}
\item[(i)] Are there other spaces than semi-Montel spaces $E$ for which the `full' 
$\mathcal{C}^{k}$-weak-strong principle \prettyref{thm:weak_strong_finite_order} with $k<\infty$ is true? 
For instance, if $k=0$, then it is still true if $E$ is 
a generalised Schwartz space by \cite[2.10 Lemma, p.\ 140]{B2}. 
Does this hold for $0<k<\infty$ as well?
\item[(ii)] Does the `almost' $\mathcal{C}^{k}$-weak-strong principle \prettyref{cor:allmost_weak_strong_finite_order} 
also hold for $d>1$? 
\item[(iii)] For every $\varepsilon>0$ does there exist a function $g\in\mathcal{C}^{k}(\R,E)$ such that 
$\lambda(\{x\in\Omega\;|\;f(x)\neq g(x)\})<\varepsilon$ in \prettyref{cor:allmost_weak_strong_finite_order} 
where $\lambda$ is the one-dimensional Lebesgue measure. 
In the case that $E=\R^{n}$ this is true by \cite[Theorem 3.1.15, p.\ 227]{federer1969}.
\item[(iv)] Is there a `Radon--Nikod\'ym type' characterisation of generalised Gelfand spaces as in the Banach case? 
\end{enumerate}
\end{que}
\section{Vector-valued Blaschke theorems}
\label{sect:blaschke}
In this section we prove several convergence theorems for Banach-valued functions in the 
spirit of Blaschke's convergence theorem \cite[Theorem 7.4, p.\ 219]{burckel1979} as it is done in 
\cite[Theorem 2.4, p.\ 786]{Arendt2000} and \cite[Corollary 2.5, p.\ 786--787]{Arendt2000} 
for bounded holomorphic functions and more general in \cite[Corollary 4.2, p.\ 695]{F/J/W} 
for bounded functions in the kernel of a hypoelliptic linear partial differential operator. 
\emph{\gls{Blaschke}} says that if $(z_{n})_{n\in\N}\subset\D$ is a sequence of distinct elements
with $\sum_{n\in\N}(1-|z_{n}|)=\infty$ and if $(f_{k})_{k\in\N}$ is a bounded sequence in $H^{\infty}(\D)$ 
such that $(f_{k}(z_{n}))_{k}$ converges in $\C$ for each $n\in\N$, then there is $f\in H^{\infty}(\D)$ such that 
$(f_{k})_{k}$ converges uniformly to $f$ on the compact subsets of $\D$, i.e.\ w.r.t.\ to $\tau_{c}$.

\begin{prop}[{\cite[Proposition 4.1, p.\ 695]{F/J/W}}]\label{prop:Blaschke_operator}
Let $(E,\|\cdot\|)$ be a Banach space, $Z$ a Banach space whose closed unit ball $B_{Z}$ is a 
compact subset of an lcHs $Y$ and let $(\mathsf{A}_{\iota})_{\iota\in I}$ be a net in $Y\varepsilon E$ 
such that
\[
\sup_{\iota\in I}\{\|\mathsf{A}_{\iota}(y)\|\;|\;y\in B_{Z}^{\circ Y'}\}<\infty.
\] 
Assume further that there exists a $\sigma(Y',Z)$-dense subspace $X\subset Y'$ such that 
$\lim_{\iota}\mathsf{A}_{\iota}(x)$ exists for each $x\in X$. Then there is 
$\mathsf{A}\in Y\varepsilon E$ with $\mathsf{A}(B_{Z}^{\circ Y'})$ bounded and 
$\lim_{\iota}\mathsf{A}_{\iota}=\mathsf{A}$ uniformly on the equicontinuous subsets of $Y'$, 
i.e.\ for all equicontinuous $B\subset Y'$ and $\varepsilon>0$ there exists $\varsigma\in I$ such that 
\[
\sup_{y\in B}\|\mathsf{A}_{\iota}(y)-\mathsf{A}(y)\|<\varepsilon
\]
for each $\iota\geq\varsigma$.
\end{prop}

Next, we generalise \cite[Corollary 4.2, p.\ 695]{F/J/W}.

\begin{cor}\label{cor:Blaschke_vector_valued}
Let $(E,\|\cdot\|)$ be a Banach space and $\f$ and $\fe$ be $\varepsilon$-into-compatible. 
Let $(T^{E},T^{\K})$ be a generator for $(\mathcal{F}\nu,E)$ 
and a strong, consistent family for $(F,E)$, $\Fv$ a Banach space 
whose closed unit ball $B_{\mathcal{F}\nu(\Omega)}$ 
is a compact subset of $\f$ and $U$ a set of uniqueness for $(T^{\K},\mathcal{F}\nu)$. 

If $(f_{\iota})_{\iota\in I}\subset\Feps$ is a bounded net in $\FvE$ such that 
$\lim_{\iota}T^{E}(f_{\iota})(x)$ exists for all $x\in U$, then there is $f\in\Feps$ such that 
$(f_{\iota})_{\iota\in I}$ converges to $f$ in $\fe$.
\end{cor}
\begin{proof}
We set $X:=\operatorname{span}\{T^{\K}_{x}\;|\;x\in U\}$, $Y:=\f$ and $Z:=\Fv$.  
As in the proof of \prettyref{thm:ext_B_unique} we observe that $X$ is $\sigma(Y',Z)$-dense in $Y'$. 
From $(f_{\iota})_{\iota\in I}\subset\Feps$ follows that there are $\mathsf{A}_{\iota}\in\f\varepsilon E$ 
with $S(\mathsf{A}_{\iota})=f_{\iota}$ for all $\iota\in I$. Since $(f_{\iota})_{\iota\in I}$ is a bounded net 
in $\FvE$, we note that 
\begin{align*}
  \sup_{\iota\in I}\sup_{x\in\omega}\|\mathsf{A}_{\iota}(T^{\K}_{x}(\cdot)\nu(x))\|
 &=\sup_{\iota\in I}\sup_{x\in\omega}\|T^{E}S(\mathsf{A}_{\iota})(x)\|\nu(x)
 =\sup_{\iota\in I}\sup_{x\in\omega}\|T^{E}f_{\iota}(x)\|\nu(x)\\
 &=\sup_{\iota\in I}|f_{\iota}|_{\FvE}<\infty
\end{align*}
by consistency. Further, $\lim_{\iota}S(\mathsf{A}_{\iota})(T^{\K}_{x})=\lim_{\iota}T^E(f_{\iota})(x)$
exists for each $x\in U$, implying the existence of $\lim_{\iota}S(\mathsf{A}_{\iota})(x)$ for each $x\in X$ 
by linearity. 
We apply \prettyref{prop:Blaschke_operator} and obtain $f:=S(\mathsf{A})\in\Feps$ 
such that $(\mathsf{A}_{\iota})_{\iota\in I}$ converges to $\mathsf{A}$ in $\f\varepsilon E$. 
From $\f$ and $\fe$ being $\varepsilon$-into-compatible it follows that 
$(f_{\iota})_{\iota\in I}$ converges to $f$ in $\fe$.
\end{proof}
 
First, we apply the preceding corollary to the space $\mathcal{C}^{[\gamma]}_{z}(\Omega,E)$ of $\gamma$-H\"older continuous functions on 
$\Omega$ that vanish at a fixed point $z\in\Omega$ from \prettyref{ex:hoelder} a).
We recall that for a metric space $(\Omega,\d)$, $z\in\Omega$, an lcHs $E$ and 
$0<\gamma\leq 1$ we have
\[
\mathcal{C}^{[\gamma]}_{z}(\Omega,E)=\{f\in E^{\Omega}\;|\;f(z)=0\;\text{and}\;\forall\;\alpha\in \mathfrak{A}:\; 
|f|_{\mathcal{C}^{0,\gamma}(\Omega),\alpha}<\infty\}.
\]
Further, we set $\omega:=\Omega^{2}\setminus\{(x,x)\;|\;x\in\Omega\}$, 
$\fe:=\{f\in\mathcal{C}(\Omega,E)\;|\;f(z)=0\}$ 
and $T^{E}\colon\fe\to E^{\omega}$, $T^{E}(f)(x,y):=f(x)-f(y)$, and 
\[
\nu\colon\omega\to [0,\infty),\;\nu(x,y):=\frac{1}{\d(x,y)^{\gamma}}.
\]
Then we have for every $\alpha\in\mathfrak{A}$ that
\[
 |f|_{\mathcal{C}^{0,\gamma}(\Omega),\alpha}=\sup_{x\in\omega}p_{\alpha}\bigl(T^{E}(f)(x)\bigr)\nu(x),
 \quad f\in\mathcal{C}^{[\gamma]}_{z}(\Omega,E),
\]
and observe that $\FvE=\mathcal{C}^{[\gamma]}_{z}(\Omega,E)$ with generator $(T^{E},T^{\K})$. 

\begin{cor}\label{cor:Hoelder_vanish_Blaschke}
Let $E$ be a Banach space, $(\Omega,\d)$ a metric space, $z\in\Omega$ and $0<\gamma\leq 1$.
If $(f_{\iota})_{\iota\in I}$ is a bounded net in $\mathcal{C}^{[\gamma]}_{z}(\Omega,E)$ such that 
$\lim_{\iota}f_{\iota}(x)$ exists for all $x$ in a dense subset $U\subset\Omega$,
then there is $f\in\mathcal{C}^{[\gamma]}_{z}(\Omega,E)$ such that 
$(f_{\iota})_{\iota\in I}$ converges to $f$ in $\mathcal{C}(\Omega,E)$ uniformly on 
compact subsets of $\Omega$.
\end{cor}
\begin{proof}
We choose $\f:=\{f\in\mathcal{C}(\Omega)\;|\;f(z)=0\}$ 
and $\fe:=\{f\in\mathcal{C}(\Omega,E)\;|\;f(z)=0\}$. Then we have 
$\Fv=\mathcal{C}^{[\gamma]}_{z}(\Omega)$ 
and $\FvE=\mathcal{C}^{[\gamma]}_{z}(\Omega,E)$ with the weight $\nu$ and generator $(T^{E},T^{\K})$ 
for $(\mathcal{F}\nu,E)$ described above.
Due to \cite[3.1 Bemerkung, p.\ 141]{B2} the spaces $\f$ and $\fe$, equipped 
with the topology $\tau_{c}$ of compact convergence, 
are $\varepsilon$-compatible. Obviously, $(T^{E},T^{\K})$ is a strong, consistent family for $(\mathcal{F},E)$.
In addition, $\Fv=\mathcal{C}^{[\gamma]}_{z}(\Omega)$ is a Banach space 
by \cite[Proposition 1.6.2, p.\ 20]{Weaver}. For all $f$ from the closed unit ball $B_{\Fv}$ 
of $\Fv$ we have
\[
|f(x)-f(y)|\leq \d(x,y)^{\gamma},\quad x,y\in\Omega,
\]
and 
\[
|f(x)|=|f(x)-f(z)|\leq \d(x,z)^{\gamma},\quad x\in\Omega.
\]
It follows that $B_{\Fv}$ is (uniformly) equicontinuous and 
$\{f(x)\;|\;f\in B_{\Fv}\}$ is bounded in $\K$ for all $x\in\Omega$. 
Ascoli's theorem (see e.g.\ \cite[Theorem 47.1, p.\ 290]{munkres2000}) implies the compactness of 
$B_{\Fv}$ in $\f$ (see also \cite[3.7 Theorem (a), p.\ 10]{kruse_2022}).
Furthermore, the $\varepsilon$-compatibility of $\f$ and $\fe$ 
in combination with the consistency of $(T^{E},T^{\K})$ for $(F,E)$ gives $\Feps=\FvE$
as linear spaces by \prettyref{prop:mingle-mangle} c). 
We note that $\lim_{\iota}f_{\iota}(x)=\lim_{\iota}T^{E}(f_{\iota})(x,z)$ for all $x$ in $U$,
proving our claim by \prettyref{cor:Blaschke_vector_valued}.
\end{proof}

The space $\mathcal{C}^{[\gamma]}_{z}(\Omega)$ is named $\operatorname{Lip}_{0}(\Omega^{\gamma})$ 
in \cite{Weaver} (see \cite[Definition 1.6.1 (b), p.\ 19]{Weaver} and \cite[Definition 1.1.2, p.\ 2]{Weaver}). 
\prettyref{cor:Hoelder_vanish_Blaschke} generalises \cite[Proposition 2.1.7, p.\ 38]{Weaver} 
(in combination with \cite[Proposition 1.2.4, p.\ 5]{Weaver}) where $\Omega$ is compact, $U=\Omega$ and $E=\K$. 

\begin{cor}\label{cor:Hoelder_Blaschke}
Let $E$ be a Banach space, $\Omega\subset\R^{d}$ open and bounded, $k\in\N_{0}$ and $0<\gamma\leq 1$. 
In the case $k\geq 1$, assume additionally that $\Omega$ has Lipschitz boundary. 
If $(f_{\iota})_{\iota\in I}$ is a bounded net in $\mathcal{C}^{k,\gamma}(\overline{\Omega},E)$ such that 
\begin{enumerate}
\item [(i)] $\lim_{\iota}f_{\iota}(x)$ exists for all $x$ in a dense subset $U\subset\Omega$, or if
\item [(ii)] $\lim_{\iota}(\partial^{e_{n}})^{E}f_{\iota}(x)$ exists for all $1\leq n\leq d$ and $x$ in 
a dense subset $U\subset\Omega$, $\Omega$ is connected and there is $x_{0}\in\overline{\Omega}$ such that 
$\lim_{\iota}f_{\iota}(x_{0})$ exists and $k\geq 1$, 
\end{enumerate}
then there is $f\in\mathcal{C}^{k,\gamma}(\overline{\Omega},E)$ such that 
$(f_{\iota})_{\iota\in I}$ converges to $f$ in $\mathcal{C}^{k}(\overline{\Omega},E)$.
\end{cor}
\begin{proof}
As in \prettyref{cor:ext_B_unique} we take $\f:=\mathcal{C}^{k}(\overline{\Omega})$ 
and $\fe:=\mathcal{C}^{k}(\overline{\Omega},E)$ 
as well as $\Fv:=\mathcal{C}^{k,\gamma}(\overline{\Omega})$ 
and $\FvE:=\mathcal{C}^{k,\gamma}(\overline{\Omega},E)$ with the weight $\nu$ and generator $(T^{E},T^{\K})$ 
for $(\mathcal{F}\nu,E)$ described above of \prettyref{cor:ext_B_unique}.
By the proof of \prettyref{cor:ext_B_unique} all conditions of \prettyref{cor:Blaschke_vector_valued} are satisfied, 
which implies our statement.
\end{proof}

We recall that $\mathcal{CW}^{k}(\Omega,E)$ is the space $\mathcal{C}^{k}(\Omega,E)$ equipped with its usual 
topology for an open set $\Omega\subset\R^{d}$, $k\in\N_{\infty}\cup\{0\}$ and an lcHs $E$ 
(see \prettyref{ex:weighted_smooth_functions} b) for $k\in\N_{\infty}$ and the definition 
above \prettyref{prop:diff_cons_barrelled} for $k=0$).

\begin{cor}\label{cor:loc_Hoelder_Blaschke}
Let $E$ be a Banach space, $\Omega\subset\R^{d}$ open, $k\in\N_{0}$ and $0<\gamma\leq 1$.
If $(f_{\iota})_{\iota\in I}$ is a bounded net in $\mathcal{C}^{k,\gamma}_{loc}(\Omega,E)$ such that 
\begin{enumerate}
\item [(i)] $\lim_{\iota}f_{\iota}(x)$ exists for all $x$ in a dense subset $U\subset\Omega$, or if
\item [(ii)] $\lim_{\iota}(\partial^{e_{n}})^{E}f_{\iota}(x)$ exists for all $1\leq n\leq d$ and $x$ in 
a dense subset $U\subset\Omega$, $\Omega$ is connected and there is $x_{0}\in\Omega$ such that 
$\lim_{\iota}f_{\iota}(x_{0})$ exists and $k\geq 1$, 
\end{enumerate}
then there is $f\in\mathcal{C}^{k,\gamma}_{loc}(\Omega,E)$ such that 
$(f_{\iota})_{\iota\in I}$ converges to $f$ in $\mathcal{CW}^{k}(\Omega,E)$.
\end{cor}
\begin{proof}
Let $(\Omega_{n})_{n\in\N}$ be an exhaustion of $\Omega$ with open, relatively compact sets $\Omega_{n}\subset\Omega$ 
such that $\Omega_{n}$ has Lipschitz boundary, $\Omega_{n}\subset\Omega_{n+1}$ for all $n\in\N$ and, 
in addition, $x_{0}\in\Omega_{1}$ and $\Omega_{n}$ is connected for each $n\in\N$ in case (ii) 
(see the proof of \prettyref{cor:ext_B_unique_loc_Hoelder}).
The restriction of $(f_{\iota})_{\iota\in I}$ to $\Omega_{n}$ is a bounded net in 
$\mathcal{C}^{k,\gamma}(\overline{\Omega}_{n},E)$ for each $n\in\N$. By \prettyref{cor:Hoelder_Blaschke} there 
is $F_{n}\in\mathcal{C}^{k,\gamma}(\overline{\Omega}_{n},E)$ for each $n\in\N$ such that 
the restriction of $(f_{\iota})_{\iota\in I}$ to $\Omega_{n}$ converges to $F_{n}$ 
in $\mathcal{C}^{k}(\overline{\Omega}_{n},E)$ since $U\cap\Omega_{n}$ is dense in $\Omega_{n}$ 
due to $\Omega_{n}$ being open and $x_{0}$ being an element of the connected set $\Omega_{n}$ in case (ii). 
The limits $F_{n+1}$ and $F_{n}$ coincide on $\Omega_{n}$ for each $n\in\N$. Thus the 
definition $f:=F_{n}$ on $\Omega_{n}$ for each $n\in\N$ gives a well-defined 
function $f\in\mathcal{C}^{k,\gamma}_{loc}(\Omega,E)$, 
which is a limit of $(f_{\iota})_{\iota\in I}$ in $\mathcal{CW}^{k}(\Omega,E)$.
\end{proof}

\begin{cor}\label{cor:k+1_smooth_Blaschke}
Let $E$ be a Banach space, $\Omega\subset\R^{d}$ open and $k\in\N_{0}$.
If $(f_{\iota})_{\iota\in I}$ is a bounded net in $\mathcal{C}^{k+1}(\Omega,E)$ such that 
\begin{enumerate}
\item [(i)] $\lim_{\iota}f_{\iota}(x)$ exists for all $x$ in a dense subset $U\subset\Omega$, or if
\item [(ii)] $\lim_{\iota}(\partial^{e_{n}})^{E}f_{\iota}(x)$ exists for all $1\leq n\leq d$ and $x$ in 
a dense subset $U\subset\Omega$, $\Omega$ is connected and there is $x_{0}\in\Omega$ such that 
$\lim_{\iota}f_{\iota}(x_{0})$ exists, 
\end{enumerate}
then there is $f\in\mathcal{C}^{k,1}_{loc}(\Omega,E)$ such that 
$(f_{\iota})_{\iota\in I}$ converges to $f$ in $\mathcal{CW}^{k}(\Omega,E)$.
\end{cor}
\begin{proof}
By \prettyref{cor:ext_B_unique_loc_Hoelder} b) $(f_{\iota})_{\iota\in I}$ is a bounded net 
in $\mathcal{C}^{k,1}_{loc}(\Omega,E)$. 
Hence our statement is a consequence of \prettyref{cor:loc_Hoelder_Blaschke}.
\end{proof}

The preceding result directly implies a $\mathcal{C}^{\infty}$-smooth version.

\begin{cor}
Let $E$ be a Banach space and $\Omega\subset\R^{d}$ open.
If $(f_{\iota})_{\iota\in I}$ is a bounded net in $\mathcal{C}^{\infty}(\Omega,E)$ such that 
\begin{enumerate}
\item [(i)] $\lim_{\iota}f_{\iota}(x)$ exists for all $x$ in a dense subset $U\subset\Omega$, or if
\item [(ii)] $\lim_{\iota}(\partial^{e_{n}})^{E}f_{\iota}(x)$ exists for all $1\leq n\leq d$ and $x$ in 
a dense subset $U\subset\Omega$, $\Omega$ is connected and there is $x_{0}\in\Omega$ such that 
$\lim_{\iota}f_{\iota}(x_{0})$ exists, 
\end{enumerate}
then there is $f\in\mathcal{C}^{\infty}(\Omega,E)$ such that 
$(f_{\iota})_{\iota\in I}$ converges to $f$ in $\mathcal{CW}^{\infty}(\Omega,E)$.
\end{cor}

Now, we turn to weighted kernels of hypoelliptic linear partial differential operators.

\begin{cor}\label{cor:weighted_hypo_Blaschke}
Let $E$ be a Banach space, $\Omega\subset\R^{d}$ open, 
$P(\partial)^{\K}$ a hypoelliptic linear partial differential operator, 
$\nu\colon\Omega\to(0,\infty)$ continuous and $U\subset\Omega$ a set of uniqueness for 
$(\id_{\K^{\Omega}},\mathcal{C}\nu_{P(\partial)})$. 
If $(f_{\iota})_{\iota\in I}$ is a bounded net in $(\mathcal{C}\nu_{P(\partial)}(\Omega,E),|\cdot|_{\nu})$ 
such that $\lim_{\iota}f_{\iota}(x)$ exists for all $x\in U$, 
then there is $f\in\mathcal{C}\nu_{P(\partial)}(\Omega,E)$ such that 
$(f_{\iota})_{\iota\in I}$ converges to $f$ in $(\mathcal{C}^{\infty}_{P(\partial)}(\Omega,E),\tau_{c})$.
\end{cor}
\begin{proof}
Our statement follows from \prettyref{cor:Blaschke_vector_valued} since by the proof of 
\prettyref{cor:hypo_weighted_ext_unique} all conditions needed are fulfilled. 
\end{proof}

For $\nu=1$ on $\Omega$ the preceding corollary is included in \cite[Corollary 4.2, p.\ 695]{F/J/W} but then 
an even better result is available, whose proof we prepare next. 
We recall the definition of the space $(\mathcal{C}^{\infty}_{P(\partial),b}(\Omega,E),\beta)$ 
with the strict topology $\beta$ from \prettyref{prop:strict_top_isomorphism}.
For an open set $\Omega\subset\R^{d}$, an lcHs $E$ and a linear partial differential operator 
$P(\partial)^{E}$ which is hypoelliptic if $E=\K$ the space of bounded zero solutions is
\[
 \mathcal{C}^{\infty}_{P(\partial),b}(\Omega,E)
 =\{f\in\mathcal{C}^{\infty}_{P(\partial)}(\Omega,E)\;|\;\forall\;\alpha\in\mathfrak{A}:\;
   \|f\|_{\infty,\alpha}=\sup_{x\in\Omega}p_{\alpha}(f(x))<\infty\}.
\]
We equip this space with strict topology $\beta$ induced by the seminorms 
\[
|f|_{\widetilde{\nu},\alpha}:=\sup_{x\in\Omega}p_{\alpha}(f(x))|\widetilde{\nu}(x)|,
\quad f\in\mathcal{C}^{\infty}_{P(\partial),b}(\Omega,E),
\]
for $\widetilde{\nu}\in\mathcal{C}_{0}(\Omega)$.
Now, we phrase for $\mathcal{C}^{\infty}_{P(\partial),b}(\Omega,E)=\mathcal{C}\nu_{P(\partial)}(\Omega,E)$ 
with $\nu=1$ on $\Omega$ the improved version of \prettyref{cor:weighted_hypo_Blaschke}.

\begin{cor}\label{cor:strict_hypo_Blaschke}
Let $E$ be a Banach space, $\Omega\subset\R^{d}$ open, $P(\partial)^{\K}$ a hypoelliptic linear 
partial differential operator and $U\subset\Omega$ a set of uniqueness 
for $(\id_{\K^{\Omega}},\mathcal{C}^{\infty}_{P(\partial),b})$. 
If $(f_{\iota})_{\iota\in I}$ is a bounded net in $(\mathcal{C}^{\infty}_{P(\partial),b}(\Omega,E),\|\cdot\|_{\infty})$ 
such that $\lim_{\iota}f_{\iota}(x)$ exists for all $x\in U$, 
then there is $f\in\mathcal{C}^{\infty}_{P(\partial),b}(\Omega,E)$ such that 
$(f_{\iota})_{\iota\in I}$ converges to $f$ in $(\mathcal{C}^{\infty}_{P(\partial),b}(\Omega,E),\beta)$.
\end{cor}
\begin{proof}
We take $\f:=(\mathcal{C}^{\infty}_{P(\partial),b}(\Omega),\beta)$ 
and $\fe:=(\mathcal{C}^{\infty}_{P(\partial),b}(\Omega,E),\beta)$ 
as well as $\Fv:=(\mathcal{C}^{\infty}_{P(\partial),b}(\Omega),\|\cdot\|_{\infty})$ 
and $\FvE:=(\mathcal{C}^{\infty}_{P(\partial),b}(\Omega,E),\|\cdot\|_{\infty})$ 
with the weight $\nu(x):=1$, $x\in\Omega$, and generator $(\id_{E^{\Omega}},\id_{\Omega^{\K}})$ 
for $(\mathcal{F}\nu,E)$. The generator is strong and consistent for $(F,E)$ 
and $\f$ and $\fe$ are $\varepsilon$-compatible by \prettyref{prop:strict_top_isomorphism}. 
The space $\Fv$ is a Banach space as a closed subspace of the Banach space 
$(\mathcal{C}_{b}(\Omega),\|\cdot\|_{\infty})$. Its closed unit ball $B_{\Fv}$ is $\tau_{c}$-compact because 
$(\mathcal{C}^{\infty}_{P(\partial)}(\Omega),\tau_{c})$ is a Fr\'echet--Schwartz space, in particular, a Montel space. 
Thus $B_{\Fv}$ is $\|\cdot\|_{\infty}$-bounded and $\tau_{c}$-compact, which implies that 
it is also $\beta$-compact by \cite[Proposition 1 (viii), p.\ 586]{cooper1971} 
and \cite[Proposition 3, p.\ 590]{cooper1971}. 
In addition, the $\varepsilon$-compatibility of $\f$ and $\fe$ 
in combination with the consistency of $(\id_{E^{\Omega}},\id_{\K^{\Omega}})$ for $(F,E)$ gives $\Feps=\FvE$
as linear spaces by \prettyref{prop:mingle-mangle} c), verifying our statement 
by \prettyref{cor:Blaschke_vector_valued}.
\end{proof}

A direct consequence of \prettyref{cor:strict_hypo_Blaschke} is the following remark.

\begin{rem}
Let $E$ be a Banach space, $\Omega\subset\R^{d}$ open, $P(\partial)^{\K}$ a hypoelliptic linear partial 
differential operator and $(f_{\iota})_{\iota\in I}$ a bounded net in the space
$(\mathcal{C}^{\infty}_{P(\partial),b}(\Omega,E),\|\cdot\|_{\infty})$. 
Then the following statements are equivalent:
\begin{enumerate} 
\item[(i)] $(f_{\iota})$ converges pointwise,
\item[(ii)] $(f_{\iota})$ converges uniformly on compact subsets of $\Omega$,
\item[(iii)] $(f_{\iota})$ is $\beta$-convergent.  
\end{enumerate}
\end{rem}

In the case of complex-valued bounded holomorphic functions of one variable, i.e.\ $E=\C$, $\Omega\subset\C$ is open and 
$P(\partial)=\overline{\partial}$ is the Cauchy--Riemann operator, convergence w.r.t.\ $\beta$ is known as 
bounded convergence (see \cite[p.\ 13--14, 16]{rubel1971}) and
the preceding remark is included in \cite[3.7 Theorem, p.\ 246]{rubelshields1966} for connected sets $\Omega$.

A similar improvement of \prettyref{cor:Hoelder_vanish_Blaschke} for the space 
$\mathcal{C}^{[\gamma]}_{z}(\Omega,E)$ of $\gamma$-H\"older continuous functions on a metric space $(\Omega,\d)$ 
that vanish at a given point $z\in\Omega$ is possible, using the \gls{strict_topology} $\beta$ on 
$\mathcal{C}^{[\gamma]}_{z}(\Omega)$ given by the seminorms
\[
|f|_{\nu}:=\sup_{\substack{x,y\in\Omega\\x\neq y}}
\frac{|f(x)-f(y)|}{|x-y|^{\gamma}}|\nu(x,y)|,\quad f\in\mathcal{C}^{[\gamma]}_{z}(\Omega),
\]
for $\nu\in\mathcal{C}_{0}(\omega)$ with $\omega=\Omega^{2}\setminus\{(x,x)\;|\;x\in\Omega\}$. 
If $\Omega$ is compact and $E$ a Banach space, this follows as in \prettyref{cor:strict_hypo_Blaschke} 
from the observation that $\beta$ is the mixed topology 
$\gamma(|\cdot|_{\mathcal{C}^{0,\gamma}(\Omega)},\tau_{c})$ 
by \cite[Theorem 3.3, p.\ 645]{vargas2018}, that a set is $\beta$-compact if and only if it 
is $|\cdot|_{\mathcal{C}^{0,\gamma}(\Omega)}$-bounded and $\tau_{c}$-compact 
by \cite[Theorem 2.1 (6), p.\ 642]{vargas2018}, the $\varepsilon$-compatibility
$(\mathcal{C}^{[\gamma]}_{z}(\Omega),\beta)\varepsilon E\cong(\mathcal{C}^{[\gamma]}_{z}(\Omega,E),\gamma\tau_{\gamma})$ 
by \cite[Theorem 4.4, p.\ 648]{vargas2018} where the topology 
$\gamma\tau_{\gamma}$ is described in \cite[Definition 4.1, p.\ 647]{vargas2018} 
and coincides with $\beta$ if $E=\K$ by \cite[Proposition 4.3 (i), p.\ 647]{vargas2018}.

Let us turn to Bloch type spaces. The result corresponding to \prettyref{cor:weighted_hypo_Blaschke} 
for Bloch type spaces reads as follows.

\begin{cor}
Let $E$ be a Banach space, $\nu\colon\D\to(0,\infty)$ continuous and $U_{\ast}\subset\D$ 
have an accumulation point in $\D$. 
If $(f_{\iota})_{\iota\in I}$ is a bounded net in $\mathcal{B}\nu(\D,E)$ 
such that $\lim_{\iota}f_{\iota}(0)$ and $\lim_{\iota}(\partial_{\C}^{1})^{E}f_{\iota}(z)$ 
exist for all $z\in U_{\ast}$, 
then there is $f\in\mathcal{B}\nu(\D,E)$ such that 
$(f_{\iota})_{\iota\in I}$ converges to $f$ in $(\mathcal{O}(\D,E),\tau_{c})$. 
\end{cor}
\begin{proof}
Due to the proof of \prettyref{cor:Bloch_ext_unique} all conditions needed 
to apply \prettyref{cor:Blaschke_vector_valued} are fulfilled, which proves our statement.
\end{proof}
\section{Wolff type results}
\label{sect:wolff_type}
The following theorem gives us a Wolff type description of the dual of $\f$ 
and generalises \cite[Theorem 3.3, p.\ 693]{F/J/W} and \cite[Corollary 3.4, p.\ 694]{F/J/W} whose proofs only need a bit of adaptation. \emph{\gls{Wolff}} \cite[p.\ 1327]{wolff1921} (cf.\ \cite[Theorem (Wolff), p.\ 402]{grosse-erdmann2004}) 
phrased in a functional analytic way (see \cite[p.\ 240]{F/J/W}) says: if $\Omega\subset\C$
is a domain (i.e.\ open and connected), then for each $\mu\in\mathcal{O}(\Omega)'$ there are a sequence 
$(z_{n})_{n\in\N}$ which is relatively compact in $\Omega$ and a sequence
$(a_{n})_{n\in\N}$ in $\ell^{1}$ such that $\mu=\sum_{n=1}^{\infty}a_{n}\delta_{z_{n}}$.

\begin{thm}\label{thm:wolff}
Let $\f$ and $\fe$ be $\varepsilon$-into-compatible, 
$(T^{E},T^{\K})$ be a generator for $(\mathcal{F}\nu,E)$ and a strong, consistent family for $(F,E)$ 
for every Banach space $E$. 
Let $\f$ be a nuclear Fr\'echet space and $\Fv$ a Banach space 
whose closed unit ball $B_{\Fv}$ 
is a compact subset of $\f$ and $(x_{n})_{n\in\N}$ fixes the topology in $\Fv$.
\begin{enumerate}
\item[a)] Then there is $0<\lambda\in\ell^1$, i.e.\ $\lambda\in\ell^{1}$ and $\lambda_{n}>0$ for all $n\in\N$, such that for every bounded $B\subset\f_{b}'$ there 
is $C\geq 1$ with 
\[
       \{\mu_{\mid\Fv}\;|\;\mu\in B\}
\subset\{\sum_{n=1}^{\infty}a_{n}\nu(x_{n})T^{\K}_{x_{n}}\in\Fv'
         \;|\;a\in\ell^{1},\;\forall\;n\in\N:\;|a_{n}|\leq C\lambda_{n}\}.
\]
\item[b)] Let $(\|\cdot\|_{k})_{k\in\N}$ denote the system of seminorms generating the topology of $\f$. 
Then there is a decreasing zero sequence $(\varepsilon_{n})_{n\in\N}$ such that for all $k\in\N$ 
there is $C\geq 1$ with 
\[
\|f\|_{k}\leq C \sup_{n\in\N}|T^{\K}(f)(x_{n})|\nu(x_{n})\varepsilon_{n},\quad f\in\Fv. 
\]
\end{enumerate}
\end{thm}
\begin{proof}
We start with part a). Let $B_{1}:=\{T^{\K}_{x_{n}}(\cdot)\nu(x_{n})\;|\;n\in\N\}\subset\f'$, 
$X:=\operatorname{span} B_{1}$, $Y:=\f$, 
$Z:=\Fv$ and $E_{1}:=\{\sum_{n=1}^{\infty}a_{n}\nu(x_{n})T^{\K}_{x_{n}}\;|\;a\in\ell^{1}\}$. 
From
\[
|j_{1}(a)(f)|:=|\sum_{n=1}^{\infty}a_{n}\nu(x_{n})T^{\K}_{x_{n}}(f)|\leq\sup_{n\in\N}|T^{\K}(f)(x_{n})|\nu(x_{n})\|a\|_{\ell^{1}}
\leq |f|_{\Fv}\|a\|_{\ell^{1}}
\]
for all $f\in\Fv$ and $a\in\ell^{1}$ it follows that $E_{1}$ is a linear subspace of $\Fv'$ 
and the continuity of the map $j_{1}\colon \ell^{1}\to\Fv'$ where $\Fv'$ is equipped 
with the operator norm. In addition, we deduce that the linear map 
$j\colon \ell^{1}/\ker j_{1}\to\Fv'$, $j([a]):=j_{1}(a)$, where 
$[a]$ denotes the equivalence class of $a\in\ell^{1}$ in the quotient space $\ell^{1}/\ker j_{1}$,
is continuous w.r.t.\ the quotient norm since
\[
\|j([a])\|_{\Fv'}\leq \inf_{b\in\ell^{1},[b]=[a]}\|b\|_{\ell^{1}}=\|[a]\|_{\ell^{1}/\ker j_{1}}.
\]
By setting $E:=j(\ell^{1}/\ker j_{1})$ and $\|j([a])\|_{E}:=\|[a]\|_{\ell^{1}/\ker j_{1}}$, $a\in\ell^{1}$, 
and observing that $\ell^{1}/\ker j_{1}$ is a Banach space, we obtain that $E$ is also a Banach space, 
which is continuously embedded in $\Fv'$.

We denote by $\mathsf{A}\colon X\to E$ the restriction to $Z=\Fv$ determined by 
\[
\mathsf{A}(T^{\K}_{x_{n}}(\cdot)\nu(x_{n})):=T^{\K}_{x_{n}}(\cdot)_{\mid\Fv}\nu(x_{n})=j([e_{n}])
\]
where $e_{n}$ is the $n$-th unit sequence in $\ell^{1}$. We consider $\Fv$ as a subspace 
of $E'$ via 
\[
f(j([a])):=j([a])(f)=\sum_{n=1}^{\infty}a_{n}\nu(x_{n})T^{\K}(f)(x_{n}),\quad a\in\ell^{1},
\]
for $f\in\Fv$. The space $G:=\Fv$ clearly separates the points of $E$, 
thus is $\sigma(E',E)$-dense and  
\[
 (f\circ \mathsf{A})(T^{\K}_{x_{n}}(\cdot)\nu(x_{n})) 
=\mathsf{A}(T^{\K}_{x_{n}}(\cdot)\nu(x_{n}))(f)
=j([e_{n}])(f)
=f(j([e_{n}]))
\]
for all $n\in\N$. Hence we may consider $f\circ \mathsf{A}$ by identification with $f$ as an element of 
$Z=\Fv$ for all $f\in G=\Fv$. It follows from \prettyref{prop:ext_B_fix_top} that there is 
a unique extension $\widehat{\mathsf{A}}\in\f\varepsilon E$ of $\mathsf{A}$ such that 
$S(\widehat{\mathsf{A}})\in\Feps$. 

For each $e'\in E'$ there are $C_{0},C_{1}>0$ and an absolutely convex compact set $K\subset\f$ such that
\[
|(e'\circ\widehat{\mathsf{A}})(\mu)|\leq C_{0}\|\widehat{\mathsf{A}}(\mu)\|_{E}\leq C_{0}C_{1}\sup_{f\in K}|\mu(f)|
\]
for all $\mu\in\f'$, implying $e'\circ \widehat{\mathsf{A}}\in(\f_{b}')'$. 
Due to the reflexivity of the nuclear Fr\'echet space $\f$ we obtain 
$e'\circ \widehat{\mathsf{A}}\in\f$ for each $e'\in E'$. 
Further, for each $e'\in E'$ we have 
\begin{align*}
  \|e'\circ\widehat{\mathsf{A}}\|_{\Fv}
&=\sup_{x\in\omega}|T^{\K}(e'\circ\widehat{\mathsf{A}})(x)|\nu(x)
 =\sup_{x\in\omega}|(e'\circ\widehat{\mathsf{A}})(T^{\K}_{x}(\cdot)\nu(x))|\\
&\leq C_{0}\sup_{x\in\omega}\|\widehat{\mathsf{A}}(T^{\K}_{x}(\cdot)\nu(x))\|_{E}<\infty
\end{align*}
since $\widehat{\mathsf{A}}(B_{\Fv}^{\circ \f'})$ is bounded in $E$. 
This yields $e'\circ\widehat{\mathsf{A}}\in\Fv$ for each $e'\in E'$. In particular, we get that
$\widehat{\mathsf{A}}$ is $\sigma(\f',\Fv)$-$\sigma(E,E')$ continuous. 
The restriction $r\colon\f'\to \Fv'$, $r(\mu):=\mu_{\mid \Fv}$, 
is $\sigma(\f',\Fv)$-$\sigma(\Fv',\Fv)$ continuous 
and coincides with $\widehat{\mathsf{A}}$ on the $\sigma(\f',\Fv)$-dense subspace 
$X=\operatorname{span} B_{1}\subset\f'$. 
Therefore $\widehat{\mathsf{A}}(\mu)=r(\mu)=\mu_{\mid \Fv}$ for all $\mu\in\f'$.

Let $B$ be an absolutely convex, closed and bounded subset of $\f_{b}'$. 
We endow $W:=\operatorname{span} B$ with the Minkowski functional of $B$. 
Due to the nuclearity of $\f$, there are an absolutely convex, closed and bounded subset 
$V\subset\f_{b}'$, $(w_{k}')_{k\in\N}\subset B_{W'}$, $(\mu_{k})_{k\in\N}\subset V$ 
and $0\leq \gamma\in\ell^{1}$, i.e.\ $\gamma\in\ell^{1}$ and $\gamma_{n}\geq 0$ 
for all $n\in\N$, such that
\[
\mu=\sum_{k=1}^{\infty}\gamma_{k}w_{k}'(\mu)\mu_{k},\quad \mu\in B,
\]
by \cite[2.9.1 Theorem, p.\ 134, 2.9.2 Definition, p.\ 135]{bogachev2017}. 
The boundedness of $\widehat{\mathsf{A}}(V)$ in $E$ and the definition of $E$ give us a bounded sequence $([\beta^{(k)}])_{k\in\N}\subset E$ 
with 
\[
{\mu_{k}}_{\mid \Fv}=\widehat{\mathsf{A}}(\mu_{k})=\sum_{n=1}^{\infty}\beta^{(k)}_{n}\nu(x_{n})T^{\K}_{x_{n}}
\]
for all $k\in\N$. The sequence $(\beta^{(k)})_{k\in\N}\subset\ell^{1}$ is also bounded by 
\cite[Remark 5.11, p.\ 36]{meisevogt1997} 
and we set $\rho_{n}:=\sum_{k=1}^{\infty}\gamma_{k}|\beta^{(k)}_{n}|$ for $n\in\N$. 
With $\rho:=(\rho_{n})_{n\in\N}$ we have
\[
\|\rho\|_{\ell^{1}}=\sum_{n=1}^{\infty}\sum_{k=1}^{\infty}\gamma_{k}|\beta^{(k)}_{n}|
\leq \sum_{n=1}^{\infty}\sup_{l\in\N}|\beta^{(l)}_{n}|\sum_{k=1}^{\infty}\gamma_{k}
=\sup_{l\in\N}\|\beta^{(l)}\|_{\ell^{1}}\|\gamma\|_{\ell^{1}}<\infty,
\]
which means that $\rho\in\ell^{1}$. 
For every $\mu\in B$ we set $a_{n}:=\sum_{k=1}^{\infty}\gamma_{k}w_{k}'(\mu)\beta^{(k)}_{n}$, $n\in\N$, 
and conclude that $a\in\ell^{1}$ with $|a_{n}|\leq\rho_{n}$ for all $n\in\N$ and
\begin{equation}\label{eq:wolff}
\mu_{\mid \Fv}=\sum_{n=1}^{\infty}a_{n}\nu(x_{n})T^{\K}_{x_{n}}.
\end{equation}
The strong dual $\f_{b}'$ of the Fr\'echet--Schwartz space $\f$ is a DFS-space 
and thus there is a fundamental sequence of bounded (closed, absolutely convex) sets $(B_{l})_{l\in\N}$ in $\f_{b}'$ 
by \cite[Proposition 25.19, p.\ 303]{meisevogt1997}. Due to our preceding results there is $\rho^{(l)}\in\ell^{1}$ 
with \eqref{eq:wolff} for each $l\in\N$. Finally, part a) follows from choosing $0<\lambda\in\ell^{1}$ such that 
each $\rho^{(l)}$ is componentwise smaller than a multiple of $\lambda$, i.e.\ we choose $\lambda$ in a way that 
for each $l\in\N$ there is $C_{l}\geq 1$ with $\rho^{(l)}_{n}\leq C_{l}\lambda_{n}$ for all $n\in\N$ 
(w.l.o.g.\ we may assume (the worst case) that $\lim_{n\to\infty}\rho^{(l+1)}_{n}/\rho^{(l)}_{n}=\infty$ for each $l\in\N$. Then the construction of a suitable $0<\lambda\in\ell^{1}$ is given in 
\cite[Chap.\ IX, \S41, 7., p.\ 301--302]{knopp1951}: set $c_{n}^{(l)}:=\rho^{(l)}_{n}$ for all $l,n\in\N$ 
and define $\lambda_{n}:=c_{n}+\frac{1}{n^2}$ for all $n\in\N$ with the $(c_{n})\in\ell^1$ constructed there. 
Then set $C_{1}:=1$ and $C_{l}:=(\max\{c_{n}^{(l)}\;|\;1\leq n\leq n_{l-1}\}/\min\{\lambda_{n}\;|\;1\leq n\leq n_{l-1}\})+1$, $l\geq 2$, for the sequence of indices $(n_{l})_{l\in\N}$ from the construction of $(c_{n})$.).

Let us turn to part b). We choose $\lambda\in\ell^{1}$ from part a) and a decreasing zero sequence 
$(\varepsilon_{n})_{n\in\N}$ such that $(\tfrac{\lambda_{n}}{\varepsilon_{n}})_{n\in\N}$ still belongs to $\ell^{1}$ 
(e.g.\ take $\varepsilon_{n}:=(\sum_{k=n}^{\infty}\lambda_{k})^{1/2}$ for $n\in\N$ by 
\cite[Chap.\ IX, \S39, Theorem of Dini, p.\ 293]{knopp1951}).
For $k\in\N$ we set 
\[
\widetilde{B}_{k}:=\{f\in\f\;|\;\|f\|_{k}\leq 1\}
\]
and note that the polar $\widetilde{B}_{k}^{\circ}$ is bounded in $\f_{b}'$. 
Due to part a) there exists $C\geq 1$ such that 
\[
\widehat{\mathsf{A}}(\widetilde{B}_{k}^{\circ})
\subset
\{\sum_{n=1}^{\infty}a_{n}\nu(x_{n})T^{\K}_{x_{n}}\in\Fv'
         \;|\;a\in\ell^{1},\;\forall\;n\in\N:\;|a_{n}|\leq C\lambda_{n}\}.
\]
By \cite[Proposition 22.14, p.\ 256]{meisevogt1997} the formula 
\[
\|f\|_{k}=\sup_{y'\in \widetilde{B}_{k}^{\circ}}|y'(f)|,\quad f\in\f,
\]
is valid and hence 
\begin{align*}
  \|f\|_{k}
&=\sup_{y'\in \widetilde{B}_{k}^{\circ}}|r(y')(f)|
 =\sup_{y'\in \widetilde{B}_{k}^{\circ}}|\widehat{\mathsf{A}}(y')(f)|
 \leq C\sup_{\substack{a\in\ell^{1}\\ |a_{n}|\leq \lambda_{n}}}
      \bigl|\sum_{n=1}^{\infty}a_{n}\nu(x_{n})T^{\K}(f)(x_{n})\bigr|\\
&\leq C\Bigl\|\Bigl(\frac{\lambda_{n}}{\varepsilon_{n}}\Bigr)_{n}\Bigr\|_{\ell^{1}}
       \sup_{n\in\N}|T^{\K}(f)(x_{n})|\nu(x_{n})\varepsilon_{n}
\end{align*}
for all $f\in\Fv$.
\end{proof}

\begin{rem}
The proof of \prettyref{thm:wolff} shows it is not needed that the assumption 
that $\f$ and $\fe$ are $\varepsilon$-into-compatible, 
$(T^{E},T^{\K})$ is a generator for $(\mathcal{F}\nu,E)$ and a strong, consistent family for $(F,E)$ 
is fulfilled for every Banach space $E$. It is sufficient that it is fulfilled for the Banach space 
$E:=j(\ell^{1}/\ker j_{1})$.
\end{rem}

We recall from \eqref{eq:frame} that for a positive sequence $\nu:=(\nu_{n})_{n\in\N}$ and an lcHs $E$ we have
\[
\ell\nu(\N,E)=\{x=(x_{n})_{n\in\N}\in E^{\N}\;|\;\forall\;\alpha\in\mathfrak{A}:\;
\|x\|_{\alpha}=\sup_{n\in\N}p_{\alpha}(x_{n})\nu_{n}<\infty\}.
\]
Further, we equip the space $E^{\N}$ of all sequences in $E$ from \prettyref{ex:space_of_all_functions} 
with the topology of pointwise convergence, i.e.\ the topology generated by the seminorms 
\[
|x|_{k,\alpha}:=\sup_{1\leq n\leq k}p_{\alpha}(x_{n}),\quad x=(x_{n})_{n\in\N}\in E^{\N},
\]
for $k\in\N$ and $\alpha\in\mathfrak{A}$.

\begin{cor}
Let $\nu:=(\nu_{n})_{n\in\N}$ be a positive sequence. 
\begin{enumerate}
\item[a)] Then there is $0<\lambda\in\ell^1$ such that for every bounded $B\subset(\K^{\N})_{b}'$ there 
is $C\geq 1$ with 
\[
       \{\mu_{\mid\ell\nu(\N)}\;|\;\mu\in B\}
\subset\{\sum_{n=1}^{\infty}a_{n}\nu_{n}\delta_{n}\in\ell\nu(\N)'
         \;|\;a\in\ell^{1},\;\forall\;n\in\N:\;|a_{n}|\leq C\lambda_{n}\}.
\]
\item[b)] Then there is a decreasing zero sequence $(\varepsilon_{n})_{n\in\N}$ such that for all $k\in\N$ 
there is $C\geq 1$ with 
\[
\sup_{1\leq n\leq k}|x_{n}|\leq C \sup_{n\in\N}|x_{n}|\nu_{n}\varepsilon_{n},\quad x=(x_{n})_{n\in\N}\in\ell\nu(\N). 
\]
\end{enumerate}
\end{cor}
\begin{proof}
We take $F(\N):=\K^{\N}$ and $F(\N,E):=E^{\N}$ as well as $\mathcal{F}\nu(\N):=\ell\nu(\N)$ and 
$\mathcal{F}\nu(\N,E):=\ell\nu(\N,E)$ 
where $(T^{E},T^{\K}):=(\operatorname{\id}_{E^{\N}},\operatorname{\id}_{\K^{\N}})$ is the generator 
for $(\mathcal{F}\nu,E)$.
We remark that $F(\N)$ and $F(\N,E)$ are $\varepsilon$-compatible 
and $(T^{E},T^{\K})$ is a strong, consistent family for $(F,E)$ by \prettyref{ex:space_of_all_functions}
for every Banach space $E$. Moreover, $\mathcal{F}\nu(\N)=\ell\nu(\N)$ is a Banach space by 
\cite[Lemma 27.1, p.\ 326]{meisevogt1997} since $\ell\nu(\N)=\lambda^{\infty}(A)$ 
with the K\"othe matrix $A:=(a_{n,j})_{n,j\in\N}$ given by $a_{n,j}:=\nu_{n}$ for all $n,j\in\N$. 
In addition, we have for every $k\in\N$
\[
\sup_{1\leq n\leq k}|x_{n}|\leq \sup_{1\leq n\leq k}\nu_{n}^{-1}|x|_{\nu} \leq \sup_{1\leq n\leq k}\nu_{n}^{-1},
\quad x=(x_{n})_{n\in\N}\in B_{\mathcal{F}\nu(\N)},
\]
which means that $B_{\mathcal{F}\nu(\N)}$ is bounded in $F(\N)$.  
The space $F(\N)=\K^{\N}$ is a nuclear Fr\'echet space and 
$B_{\mathcal{F}\nu(\N)}$ is obviously closed in $\K^{\N}$. 
Thus the bounded and closed set $B_{\mathcal{F}\nu(\N)}$ is compact in $F(\N)$,
implying our statement by \prettyref{thm:wolff}.
\end{proof}

\begin{cor}
Let $\Omega\subset\R^{d}$ be open, $P(\partial)^{\K}$ a hypoelliptic linear partial differential operator, 
$\nu\colon\Omega\to(0,\infty)$ continuous and $(x_{n})_{n\in\N}$ fix the topology 
in $\mathcal{C}\nu_{P(\partial)}(\Omega)$.
\begin{enumerate}
\item[a)] Then there is $0<\lambda\in\ell^1$ such that for every bounded 
$B\subset(\mathcal{C}^{\infty}_{P(\partial)}(\Omega),\tau_{c})_{b}'$ there is $C\geq 1$ with 
\[
       \{\mu_{\mid\mathcal{C}\nu_{P(\partial)}(\Omega)}\;|\;\mu\in B\}
\subset\{\sum_{n=1}^{\infty}a_{n}\nu(x_{n})\delta_{x_{n}}\in\mathcal{C}\nu_{P(\partial)}(\Omega)'
         \;|\;a\in\ell^{1},\;\forall\;n\in\N:\;|a_{n}|\leq C\lambda_{n}\}.
\]
\item[b)] Then there is a decreasing zero sequence $(\varepsilon_{n})_{n\in\N}$ such that 
for all compact $K\subset\Omega$ there is $C\geq 1$ with 
\[
\sup_{x\in K}|f(x)|\leq C \sup_{n\in\N}|f(x_{n})|\nu(x_{n})\varepsilon_{n},
\quad f\in\mathcal{C}\nu_{P(\partial)}(\Omega). 
\]
\end{enumerate}
\end{cor}
\begin{proof}
Due to the proof of \prettyref{cor:hypo_weighted_ext_unique} and the observation that the space 
$\f=(\mathcal{C}^{\infty}_{P(\partial)}(\Omega),\tau_{c})$ is a nuclear Fr\'echet space 
all conditions of \prettyref{thm:wolff} are fulfilled, which yields our statement.
\end{proof}
\section[Series representation of vector-valued functions]{Series representation of vector-valued functions via Schauder decompositions}
\label{sect:schauder}
The purpose of this section is to lift series representations known from scalar-valued functions 
to vector-valued functions and its underlying idea was derived from the classical example of 
the (local) power series representation of a holomorphic function. 
We recall that a $\C$-valued function $f$ on the open disc $\D_{r}(0)$
around zero with radius $r>0$ belongs to the space $\mathcal{O}(\D_{r}(0))$ of 
holomorphic functions on $\D_{r}(0)$ if the limit
\begin{equation}\label{intro:holom}
 f^{(1)}(z):=\lim_{\substack{h\to 0\\ h\in\C,h\neq 0}}\frac{f(z+h)-f(z)}{h},\quad z\in \D_{r}(0),
\end{equation}
exists in $\C$. It is well-known that every $f\in\mathcal{O}(\D_{r}(0))$ can be written as 
\[
 f(z)=\sum_{n=0}^{\infty}\frac{f^{(n)}(0)}{n!}z^{n},\quad z\in\D_{r}(0),
\]
where the power series on the right-hand side converges uniformly on every compact subset of $\D_{r}(0)$ and 
$f^{(n)}(0)$ is the $n$-th complex derivative of $f$ at $0$ which is defined from \eqref{intro:holom} 
by the recursion $f^{(0)}:=f$ and $f^{(n)}:=(f^{(n-1)})^{(1)}$ for $n\in\N$. 
By \cite[2.1 Theorem and Definition, p.\ 17--18]{grosse-erdmann1992} and 
\cite[5.2 Theorem, p.\ 35]{grosse-erdmann1992}, this series representation remains valid if 
$f$ is a holomorphic function on $\D_{r}(0)$ with values in a locally complete locally convex Hausdorff space $E$ 
over $\C$ where holomorphy means that the limit \eqref{intro:holom} exists in $E$ 
and the higher complex derivatives are defined recursively as well. 
Analysing this example, we observe that $\mathcal{O}(\D_{r}(0))$, equipped with the topology $\tau_{c}$ of uniform convergence on 
compact subsets of $\D_{r}(0)$, is a Fr\'{e}chet space, in particular, barrelled, with a Schauder basis formed 
by the monomials $z\mapsto z^{n}$. Further, the formulas for the complex derivatives of a 
$\C$-valued resp.\ an $E$-valued function $f$ on $\D_{r}(0)$ 
are built up in the same way by \eqref{intro:holom} (see \prettyref{chap:notation}).

Our goal is to derive a mechanism which uses these observations and transfers known series representations for 
other spaces of scalar-valued functions to their vector-valued counterparts. Let us describe the general setting. 
We recall from \cite[14.2, p.\ 292]{Jarchow} that a sequence $(f_{n})$ in a locally convex Hausdorff space $F$ over a 
field $\K$ is called a \emph{\gls{top_basis}}, 
or simply a basis, if for every $f\in F$ there is a unique sequence of coefficients 
$(\lambda^{\K}_{n}(f))$ in $\K$ such that 
\begin{equation}\label{intro:basis}
f=\sum_{n=1}^{\infty}\lambda^{\K}_{n}(f)f_{n}
\end{equation}
where the series converges in $F$. Due to the uniqueness of the coefficients the map 
$\lambda^{\K}_{n}\colon f\mapsto \lambda^{\K}_{n}(f)$ is well-defined, linear and 
called the $n$\emph{-th \gls{coeff_func} associated to} $(f_{n})$. Further, for each $k\in\N$ the 
map 
\[
 P_{k}\colon F\to F,\;P_{k}(f):=\sum_{n=1}^{k}\lambda^{\K}_{n}(f)f_{n}, 
\]
is a linear projection whose range is $\operatorname{span}\{f_{1},\ldots,f_{n}\}$ and it is called 
the $k$\emph{-th \gls{exp_op} associated to} $(f_{n})$. 
A basis $(f_{n})$ of $F$ is called \emph{\gls{equicontinuous}} if the expansion operators $P_{k}$ 
form an equicontinuous sequence in the linear space $L(F,F)$ of continuous linear maps from $F$ to $F$ 
(see \cite[14.3, p.\ 296]{Jarchow}). 
A basis $(f_{n})$ of $F$ is called a \emph{\gls{Schauder_basis}} if the coefficient functionals are continuous, 
i.e.\ $\lambda^{\K}_{n}\in F'$ for each $n\in\N$.
In particular, this is already fulfilled if $F$ is a Fr\'{e}chet space by 
\cite[Corollary 28.11, p.\ 351]{meisevogt1997}. If $F$ is barrelled, then 
a Schauder basis of $F$ is already equicontinuous and $F$ has the (bounded) approximation property 
by the uniform boundedness principle.

The starting point for our approach is equation \eqref{intro:basis}. 
Let $F$ and $E$ be non-trivial locally convex Hausdorff spaces over a field $\K$ 
where $F$ has an equicontinuous Schauder basis $(f_{n})$ 
with associated coefficient functionals $(\lambda^{\K}_{n})$.
The expansion operators $(P_{k})$ form a so-called \emph{\gls{Schauder_decomp}} of $F$ (see \cite[p.\ 77]{Bonet2007}), i.e.\ 
they are continuous projections on $F$ such that
\begin{enumerate}
 \item [(i)] $P_{k}P_{j}=P_{\min(j,k)}$ for all $j,k\in\N$,
 \item [(ii)] $P_{k}\neq P_{j}$ for $k\neq j$,
 \item [(iii)] $(P_{k}f)$ converges to $f$ for each $f\in F$.
\end{enumerate}
This operator theoretic definition of a Schauder decomposition is equivalent to the usual definition in terms of closed subspaces of $F$ 
given in \cite[p.\ 377]{kalton_1970} (see \cite[p.\ 219]{Lotz1985}).
In our main \prettyref{thm:schauder_decomp} of this section we prove that $(P_{k}\varepsilon\id_{E})$ is a Schauder decomposition 
of Schwartz' $\varepsilon$-product $F\varepsilon E$ 
and each $u\in F\varepsilon E$ has the series representation
\[
u(f')=\sum_{n=1}^{\infty}u(\lambda_{n}^{\K})f'(f_{n}),\quad f'\in F'.
\]
Now, suppose that $F=\F$ is a space of $\K$-valued functions on a set $\Omega$ with a 
topology such that the point-evaluation functionals $\delta_{x}$, $x\in\Omega$, belong to $\F'$ and that there is a locally convex Hausdorff space $\FE$ of 
functions from $\Omega$ to $E$ such that the map 
\[
S\colon \F\varepsilon E\to \FE,\; u\longmapsto [x\mapsto u(\delta_{x})],
\]  
is an isomorphism, i.e.\ suppose that $\F$ and $\FE$ are $\varepsilon$-compatible.
Assuming that for each $n\in\N$ and $u\in\F\varepsilon E$ there is 
$\lambda^{E}_{n}(S(u))\in E$ with
\begin{equation}\label{eq:intro_consistent}
\lambda^{E}_{n}(S(u))=u(\lambda^{\K}_{n}),
\end{equation}
i.e.\ $(\lambda^{E},\lambda^{\K})$ is consistent, we obtain in \prettyref{cor:schauder_decomp} 
that $(S\circ (P_{k}\varepsilon\id_{E})\circ S^{-1})_{k}$ 
is a Schauder decomposition of $\FE$ and
\[
f=\lim_{k\to\infty}(S\circ (P_{k}\varepsilon\id_{E})\circ S^{-1})(f)=\sum_{n=1}^{\infty}\lambda^{E}_{n}(f)f_{n},
\quad f\in\FE,
\]
which is the desired series representation in $\FE$. 
In particular, the consistency condition \eqref{eq:intro_consistent} guarantees in the 
case of $E$-valued holomorphic functions 
on $\D_{r}(0)$ that the complex derivatives at $0$ appear in the Schauder decomposition of $\mathcal{O}(\D_{r}(0),E)$ 
since $(\partial_{\C}^{n})^{E}S(u)(0)=u(\delta_{0}\circ(\partial_{\C}^{n})^{\C})$ for all $u\in\mathcal{O}(\D_{r}(0))\varepsilon E$ and $n\in\N_{0}$ by \prettyref{prop:complex_diff_cons_strong} if $E$ is locally complete.
We apply our result to sequence spaces, spaces of continuously differentiable functions on 
a compact interval, the space of holomorphic functions, the Schwartz space and the space of 
smooth functions which are $2\pi$-periodic in each variable. 

As a byproduct of \prettyref{thm:schauder_decomp} we obtain  
that every element of the completion $F\,\widehat{\otimes}_{\varepsilon}E$ of the injective tensor product 
$F\,\otimes_{\varepsilon}E$ has a series representation as well if $F$ is a complete space 
with an equicontinuous Schauder basis and $E$ is complete. 
Concerning series representation in $F\,\widehat{\otimes}_{\varepsilon}E$, little seems to be known
whereas for the completion $F\,\widehat{\otimes}_{\pi}E$ of the projective tensor product $F\otimes_{\pi}E$ 
of two metrisable locally convex spaces $F$ and $E$ it is well-known that every $f\in F\,\widehat{\otimes}_{\pi}E$ 
has a series representation 
\[
f=\sum_{n=1}^{\infty}a_{n}f_{n}\otimes e_{n}
\]
where $(a_{n})\in \ell^{1}$, i.e.\ $(a_{n})$ is absolutely summable, 
and $(f_{n})$ and $(e_{n})$ are null sequences in $F$ and $E$, respectively 
(see e.g.\ \cite[Chap.\ I, \S2 , n$^{\circ}$1, Th\'{e}or\`{e}me 1, p.\ 51]{Gro} 
or \cite[15.6.4 Corollary, p.\ 334]{Jarchow}). 
If $F$ and $E$ are metrisable and one of them is nuclear, then the isomorphism
$F\,\widehat{\otimes}_{\pi}E \cong F\,\widehat{\otimes}_{\varepsilon}E$ holds and we trivially have a series 
representation of the elements of $F\,\widehat{\otimes}_{\varepsilon}E$ as well.
Other conditions on the existence of series representations of the elements of $F\,\widehat{\otimes}_{\varepsilon}E$
can be found in \cite[Proposition 4.25, p.\ 88]{ryan2002}, where $F$ and $E$ are Banach spaces 
and both of them have a Schauder basis, and in \cite[Theorem 2, p.\ 283]{Joiner1970}, where $F$ and $E$ 
are locally convex Hausdorff spaces and both of them have an equicontinuous Schauder basis.
\subsection{Schauder decomposition}
Let us start with our main theorem on Schauder decompositions of $\varepsilon$-products. We recall 
from \eqref{eq:tensor_into_eps_product}
that we consider the tensor product $F\otimes E$ as a linear subspace of $F\varepsilon E$ for two locally convex 
Hausdorff spaces $F$ and $E$ by means of the linear injection 
\[
\Theta\colon F\otimes E\to F\varepsilon E,\; 
\sum^{k}_{n=1}{f_{n}\otimes e_{n}}\longmapsto\bigl[y\mapsto \sum^{k}_{n=1}{y(f_{n}) e_{n}}\bigr].
\]

The next theorem is essentially due to Jos\'e Bonet, improving a previous version of us which became 
\prettyref{cor:schauder_decomp}.

\begin{thm}\label{thm:schauder_decomp}
Let $F$ and $E$ be lcHs, $(f_{n})_{n\in\N}$ an equicontinuous Schauder basis of $F$ 
with associated coefficient functionals $(\lambda_{n})_{n\in\N}$ and set $Q_{n}\colon F\to F$, 
$Q_{n}(f):=\lambda_{n}(f)f_{n}$ for every $n\in\N$. Then the following holds:
\begin{enumerate}
\item [a)] The sequence $(P_{k})_{k\in\N}$ given by $P_{k}:=\bigl(\sum_{n=1}^{k}Q_{n}\bigr)\varepsilon\id_{E}$ is 
a Schauder decomposition of $F\varepsilon E$.
\item [b)] Each $u\in F\varepsilon E$ has the series representation
\[
u(f')=\sum_{n=1}^{\infty}u(\lambda_{n})f'(f_{n}),\quad f'\in F'.
\]
\item [c)] $F\otimes E$ is sequentially dense in $F\varepsilon E$. 
\end{enumerate}
\end{thm}
\begin{proof}
Since $(f_{n})$ is a Schauder basis of $F$, the sequence $(\sum_{n=1}^{k}Q_{n})$ converges to $\id_{F}$ in $L_{\sigma}(F,F)$. 
Thus we deduce from the equicontinuity of $(f_{n})$ that $(\sum_{n=1}^{k}Q_{n})$ converges to $\id_{F}$ in $L_{\kappa}(F,F)$ 
by \cite[Theorem 8.5.1 (b), p.\ 156]{Jarchow}. For $f'\in F'$ and $f\in F$ it holds
\begin{align*}
  (Q_{n}^{t}\circ Q_{m}^{t})(f')(f)
&=Q_{m}^{t}(f')(Q_{n}(f))=Q_{m}^{t}(f')(\lambda_{n}(f)f_{n})=f'(\lambda_{m}(\lambda_{n}(f)f_{n})f_{m})\\
&=\lambda_{m}(f_{n})\lambda_{n}(f)f'(f_{m})
 =\begin{cases}
   \lambda_{n}(f)f'(f_{n})&,\; m=n,\\
   0&,\; m\neq n,
  \end{cases}
\end{align*}
due to the uniqueness of the coefficient functionals $(\lambda_{n})$ (see \cite[14.2.1 Proposition, p.\ 292]{Jarchow}) and
it follows for $k,j\in\N$ that 
\[
(\sum_{n=1}^{j}Q_{n}^{t}\circ\sum_{m=1}^{k}Q_{m}^{t})(f')(f)=\sum_{n=1}^{\min(j,k)}\lambda_{n}(f)f'(f_{n})
=\sum_{n=1}^{\min(j,k)}Q_{n}^{t}(f')(f).
\]
This implies that
\[
(P_{k}P_{j})(u)=u\circ\sum_{n=1}^{j}Q_{n}^{t}\circ\sum_{m=1}^{k}Q_{m}^{t}
=u\circ\sum_{n=1}^{\min(j,k)}Q_{n}^{t}=P_{\min(j,k)}(u)
\]
for all $u\in F\varepsilon E$. If $k\neq j$, w.l.o.g.\ $k>j$, we choose $x\in E$, $x\neq 0$,\footnote{The lcHs $E$ is non-trivial by our assumptions in \prettyref{chap:notation}.} and 
consider $f_{k}\otimes x$ as an element of $F\varepsilon E$ via the map $\Theta$. Then 
\[
(P_{k}-P_{j})(f_{k}\otimes x)=\sum_{n=j+1}^{k}(f_{k}\otimes x)\circ Q_{n}^{t}=f_{k}\otimes x\neq 0
\]
since 
\[
((f_{k}\otimes x)\circ Q_{n}^{t})(f')=(f_{k}\otimes x)(\lambda_{n}(\cdot)f'(f_{n}))=\lambda_{n}(f_{k})f'(f_{n})x
 =\begin{cases}
   (f_{k}\otimes x)(f')&\hspace{-0.1cm},\; n=k,\\
   0&\hspace{-0.1cm},\; n\neq k.
  \end{cases}
\]
It remains to prove that for each $u\in F\varepsilon E$ 
\[
\lim_{k\to\infty}P_{k}(u)=u
\]
in $F\varepsilon E$. Let $(q_{\beta})_{\beta\in\mathfrak{B}}$ denote the system of seminorms inducing 
the locally convex topology of $F$. Let $u\in F\varepsilon E$ and $\alpha\in\mathfrak{A}$. 
Due to the continuity of $u$ there are an absolutely convex compact set $K=K(u,\alpha)\subset F$ 
and $C_{0}=C_{0}(u,\alpha)>0$ such that for each $f'\in F'$ we have 
\begin{align*}
 p_{\alpha}\bigl((P_{k}(u)-u)(f')\bigr)
&=p_{\alpha}\bigl(u\bigl((\sum_{n=1}^{k}Q_{n}^{t}-\id_{F'})(f')\bigr)\bigr)
\leq C_{0}\sup_{f\in K}\bigl|(\sum_{n=1}^{k}Q_{n}^{t}-\id_{F'})(f')(f)\bigr|\\
&=C_{0}\sup_{f\in K}\bigl|f'(\sum_{n=1}^{k}Q_{n}f-f)\bigr|.
\end{align*}
Let $V$ be an absolutely convex zero neighbourhood in $F$. As a consequence of the equicontinuity of 
the polar $V^{\circ}$ there are $C_{1}>0$ and $\beta\in\mathfrak{B}$ such that 
\[ 
\sup_{f'\in V^{\circ}}p_{\alpha}\bigl((P_{k}(u)-u)(f')\bigr)
\leq C_{0}C_{1}\sup_{f\in K}q_{\beta}(\sum_{n=1}^{k}Q_{n}f-f).
\]
In combination with the convergence of $(\sum_{n=1}^{k}Q_{n})$ to $\id_{F}$ in $L_{\kappa}(F,F)$ 
this yields the convergence of $(P_{k}(u))$ to $u$ in $F\varepsilon E$ and settles part a). 

Let us turn to b) and c). Since 
\[
P_{k}(u)(f')=u\bigl(\sum_{n=1}^{k}Q_{n}^{t}(f')\bigr)=\sum_{n=1}^{k}u(\lambda_{n})f'(f_{n})
\]
for every $f'\in F'$, we note that the range of $P_{k}(u)$ is contained in $\Span\{u(\lambda_{n})\;|\;1\leq n\leq k\}$ 
for each $u\in F\varepsilon E$ and $k\in\N$. Hence $P_{k}(u)$ has finite rank and thus belongs to $F\otimes E$, 
implying the sequential density of $F\otimes E$ in $F\varepsilon E$ and the desired series representation by part a).
\end{proof}

The index set of the equicontinuous Schauder basis of $F$ 
in \prettyref{thm:schauder_decomp} need not be $\N$ (or $\N_{0}$) but may be any other countable index set 
as long as the equicontinuous Schauder basis is unconditional which is, for instance, always fulfilled 
if $F$ is nuclear by \cite[21.10.1 Dynin-Mitiagin Theorem, p.\ 510]{Jarchow}.

\begin{rem}\label{rem:schauder_decomp}
If $F$ and $E$ are complete, we have under the assumption of \prettyref{thm:schauder_decomp} that 
$F\,\widehat{\otimes}_{\varepsilon}E\cong F\varepsilon E$ by c) since $F\varepsilon E$ is complete 
by \cite[Satz 10.3, p.\ 234]{Kaballo} and $F\,\widehat{\otimes}_{\varepsilon}E$ is the closure 
of $F\otimes E$ in $F\varepsilon E$. 
Thus each element of $F\,\widehat{\otimes}_{\varepsilon}E$ has a series representation. 
\end{rem}

Let us apply the preceding theorem to spaces of Lebesgue integrable functions. 
We consider the measure space $([0,1], \mathscr{L}([0,1]), \lambda)$ of Lebesgue measurable sets 
and use the notation $\mathcal{L}^{p}[0,1]$ for 
the space of (equivalence classes) of Lebesgue $p$-integrable functions on $[0,1]$.
The \emph{Haar system} $h_{n}\colon [0,1]\to\R$, $n\in\N$, given by 
$h_{1}(x):=1$ for all $x\in [0,1]$ and 
\[
h_{2^{k}+j}(x):=
\begin{cases}
\phantom{-}1 &, (2j-2)/2^{k+1}\leq x<(2j-1)/2^{k+1},\\
-1 &, (2j-1)/2^{k+1}\leq x<2j/2^{k+1},\\
\phantom{-}0 &, \text{else},
\end{cases}
\]
for $k\in\N_{0}$ and $1\leq j\leq 2^{k}$ forms a Schauder basis 
of $\mathcal{L}^{p}[0,1]$ for every $1\leq p<\infty$ and the associated coefficient 
functionals are 
\[
\lambda_{n}(f):=\int_{[0,1]}f(x)h_{n}(x)\d\lambda(x), \quad f\in \mathcal{L}^{p}[0,1],\;n\in\N,
\]
(see \cite[Satz I, p.\ 317]{Schauder1928}). Because $\mathcal{L}^{p}[0,1]$ is Banach space and thus barrelled, 
its Schauder basis $(h_{n})$ is equicontinuous and 
we directly obtain from \prettyref{thm:schauder_decomp} the following corollary. 

\begin{cor}\label{cor:L_p_>1}
Let $E$ be an lcHs and $1\leq p<\infty$. $(\sum_{n=1}^{k}\lambda_{n}(\cdot)h_{n}\varepsilon\id_{E})_{k\in\N}$ is 
a Schauder decomposition of $\mathcal{L}^{p}[0,1]\varepsilon E$ and for
each $u\in\mathcal{L}^{p}[0,1]\varepsilon E$ it holds
\[
u(f')=\sum_{n=1}^{\infty}u(\lambda_{n})f'(h_{n}),\quad f'\in \mathcal{L}^{p}[0,1]'.
\]
\end{cor}

Defining $\mathcal{L}^{p}([0,1],E):=\mathcal{L}^{p}[0,1]\varepsilon E$, we can read the corollary 
above as a statement on series representations in the vector-valued version of $\mathcal{L}^{p}[0,1]$. 
However, in many cases of spaces $\F$ of scalar-valued functions there 
is a more natural way to define the vector-valued version $\FE$ of $\F$, 
namely, that $\F$ and $\FE$ are $\varepsilon$-compatible.

\begin{rem}\label{rem:identification_subspaces}
If $\F$ and $\FE$ are $\varepsilon$-into-compatible, then we get by identification of isomorphic subspaces
 \[
  \F\otimes_{\varepsilon} E\subset \F\varepsilon E \subset \FE
 \]
and the embedding $ \F\otimes E\hookrightarrow \FE$ is given by $f\otimes e\longmapsto [x\mapsto f(x)e]$.
\end{rem}
\begin{proof}
The inclusions obviously hold and $\F\varepsilon E$ and $\FE$ 
induce the same topology on $\F\otimes E$. 
Further, we have
 \[
 f\otimes e \overset{\Theta}{\longmapsto}[y\mapsto y(f)e]
 \overset{S}{\longmapsto}[x\longmapsto [y\mapsto y(f) e](\delta_{x})]
 =[x\mapsto f(x)e].\qedhere
 \]
\end{proof}

\begin{cor}\label{cor:schauder_decomp}
Let $\F$ and $\FE$ be $\varepsilon$-compatible, 
$(f_{n})_{n\in\N}$ an equi\-continuous Schauder basis of $\F$ 
with associated coefficient functionals $\lambda^{\K}:=(\lambda_{n}^{\K})_{n\in\N}$. 
Let there be $\lambda^{E}\colon \FE\to E^{\N}$ such that $(\lambda^{E},\lambda^{\K})$ 
is a consistent family for $(\mathcal{F},E)$,
and set $Q_{n}^{E}\colon \FE\to \FE$, $Q_{n}^{E}(f):=\lambda_{n}^{E}(f)f_{n}$ for every $n\in\N$. 
Then the following holds:
\begin{enumerate}
\item [a)] The sequence $(P_{k}^{E})_{k\in\N}$ given by $P_{k}^{E}:=\sum_{n=1}^{k}Q_{n}^{E}$ is 
a Schauder decomposition of $\FE$.
\item [b)] Each $f\in\FE$ has the series representation
\[
f=\sum_{n=1}^{\infty}\lambda_{n}^{E}(f)f_{n}.
\]
\item [c)] $\F\otimes E$ is sequentially dense in $\FE$. 
\end{enumerate}
\end{cor}
\begin{proof}
For each $u\in\F\varepsilon E$ and $x\in\Omega$ we note that with $P_{k}$ from \prettyref{thm:schauder_decomp} it holds
 \begin{align*}
    (S\circ P_{k})(u)(x)
  &=u\bigl(\sum_{n=1}^{k}Q_{n}^{t}(\delta_{x})\bigr) 
   =u\bigl(\sum_{n=1}^{k}\lambda_{n}^{\K}(\cdot)f_{n}(x)\bigr) 
   =\sum_{n=1}^{k}u(\lambda_{n}^{\K})f_{n}(x)\\
  &=\sum_{n=1}^{k}\lambda_{n}^{E}(S(u))f_{n}(x)
   =(P_{k}^{E}\circ S)(u)(x),
 \end{align*}
which means that $S\circ P_{k}=P_{k}^{E}\circ S$. This implies part a) and b) 
by \prettyref{thm:schauder_decomp} a) since $S$ is an isomorphism. 
Part c) is a direct consequence of \prettyref{thm:schauder_decomp} c) and 
the isomorphism $\F\varepsilon E\cong\FE$.
\end{proof}

In the preceding corollary we used the isomorphism $S$ to obtain a Schauder decomposition. 
On the other hand, if $S$ is an isomorphism into, which is often the case (see \prettyref{thm:linearisation}), 
we can use a Schauder decomposition of $\FE$ to 
prove the surjectivity of $S$.

\begin{prop}\label{prop:reverse_Schauder}
Let $\F$ and $\FE$ be $\varepsilon$-into-compatible.
Let there be $(f_{n})_{n\in\N}$ in $\F$ and 
for every $f\in\FE$ a sequence $(\lambda_{n}^{E}(f))_{n\in\N}$ in $E$ such that
 \[
  f=\sum_{n=1}^{\infty}\lambda_{n}^{E}(f)f_{n}, \quad f\in\FE.
 \]
Then the following holds:
\begin{enumerate}
\item [a)] $\F\otimes E$ is sequentially dense in $\FE$.
\item [b)] If $\F$ and $E$ are sequentially complete, then
 \[
  \FE\cong\F\varepsilon E.
 \] 
\item [c)] If $\F$ and $E$ are complete, then
 \[
  \FE\cong\F\varepsilon E\cong\F\widehat{\otimes}_{\varepsilon}E.
 \]
\end{enumerate}
\end{prop}
\begin{proof}
Let $f\in\FE$ and observe that
\[
P^{E}_{k}(f):=\sum_{n=1}^{k}\lambda_{n}^{E}(f)f_{n}=\sum_{n=1}^{k}f_{n}\otimes \lambda_{n}^{E}(f)\in \F\otimes E
\]
for every $k\in\N$ by \prettyref{rem:identification_subspaces}.
Due to our assumption we have the convergence $P^{E}_{k}(f)\to f$ in $\FE$. 
Thus $\F\otimes E$ is sequentially dense in $\FE$. 

Let us turn to part b). If $\F$ and $E$ are sequentially complete, 
then $\F\varepsilon E$ is sequentially complete 
by \cite[Satz 10.3, p.\ 234]{Kaballo}. Since $S$ is an isomorphism into and
\[
S(\Theta(\sum_{n=q}^{k}f_{n}\otimes \lambda_{n}^{E}(f)))=\sum_{n=q}^{k}\lambda_{n}^{E}(f)f_{n}
\]
for all $k,q\in\N$, $k>q$, we get that $(\Theta(\sum_{n=1}^{k}f_{n}\otimes \lambda_{n}^{E}(f))$ is a Cauchy sequence 
in $\F\varepsilon E$ and thus convergent. Hence we deduce that
\[
S(\lim_{k\to\infty}\Theta(\sum_{n=1}^{k}f_{n}\otimes \lambda_{n}^{E}(f)))
=\lim_{k\to\infty}\sum_{n=1}^{k}(S\circ\Theta)(f_{n}\otimes \lambda_{n}^{E}(f))
=\sum_{n=1}^{\infty}\lambda_{n}^{E}(f)f_{n}=f,
\]
which proves the surjectivity of $S$. 

If $\F$ and $E$ are complete, then $\F\widehat{\otimes}_{\varepsilon}E$ 
is the closure of $\F\otimes_{\varepsilon}E$ in the complete space 
$\F\varepsilon E$ by \cite[Satz 10.3, p.\ 234]{Kaballo}.
As $\lim_{k\to\infty}\Theta(\sum_{n=1}^{k}f_{n}\otimes \lambda_{n}^{E}(f))$ is an element of the closure, 
we obtain part c). 
\end{proof}
\subsection{Examples of Schauder decompositions}
\label{sub:Schauder_examples}
\subsubsection*{\textbf{Sequence spaces}}

For our first application we recall the definition of some sequence spaces.
For an lcHs $E$ and a K\"othe matrix $A:=\left(a_{k,j}\right)_{k,j\in\N}$ we define 
the topological subspace of $\lambda^{\infty}(A,E)$ from \prettyref{cor:sequence_vanish_infty} a) by
\[
 \gls{c0AE}:=\{x=(x_{k})\in E^{\N}\;|\;\forall\;j\in\N:\;\lim_{k\to\infty}x_{k}a_{k,j}=0\}.
\]
In particular, the space $c_{0}(\N,E)$ of null-sequences in $E$ is obtained as 
$c_{0}(\N,E)=c_{0}(A,E)$ with $a_{k,j}:=1$ for all $k,j\in\N$. 
The space of convergent sequences in $E$ is defined by
\[
\gls{cNE}:=\{x\in E^{\N}\;|\; x=(x_{k})\;\text{converges in}\;E\}
\]
and equipped with the system of seminorms
\[
|x|_{\alpha}:=\sup_{k\in\N}p_{\alpha}(x_{k}),\quad x\in c(\N,E),
\]
for $\alpha\in\mathfrak{A}$. 
Further, we set $c_{0}(A):=c_{0}(A,\K)$, $c_{0}(\N):=c_{0}(\N,\K)$ and 
$c(\N):=c(\N,\K)$. 
Furthermore, we equip the space $E^{\N}$ with the system of seminorms given by 
\[
 \|x\|_{l,\alpha}:=\sup_{k\in\N}p_{\alpha}(x_{k})\chi_{\{1,\ldots,l\}}(k),\quad x=(x_{k})\in E^{\N},
\]
for $l\in\N$ and $\alpha\in\mathfrak{A}$. For a non-empty set $\Omega$ we define for $n\in\Omega$ the $n$-th unit function by 
\[
 \varphi_{n,\Omega}\colon \Omega\to \K,\;
 \varphi_{n,\Omega}(k):=\begin{cases} 1 &,\;k=n, \\ 0 &,\;\text{else},\end{cases}
\]
and we simply write $\varphi_{n}$ instead of $\varphi_{n,\Omega}$ if no confusion seems to be likely. 
Further, we set
$
 \varphi_{\infty}\colon \N\to \K,\;
 \varphi_{\infty}(k):=1,
$
and $x_{\infty}:=\delta_{\infty}(x):=\lim_{k\to\infty}x_{k}$ for $x\in c(\N,E)$.  
For series representations of the elements in these sequence spaces we do not need \prettyref{cor:schauder_decomp} 
due to the subsequent proposition but we can use the representation to obtain the surjectivity of $S$ 
for sequentially complete $E$.

\begin{prop}\label{prop:vector_valued_seq_spaces}
 Let $E$ be an lcHs and $\ell(\Omega,E)$ one of the spaces $c_{0}(A,E)$, $E^{\N}$, 
 $s(\N^{d},E)$, $s(\N_{0}^{d},E)$ or $s(\Z^{d},E)$. 
 \begin{enumerate}
 \item [a)] Then $(\sum_{n\in\Omega,|n|\leq k}\delta_{n} \varphi_{n})_{k\in\N}$ is a Schauder decomposition 
 of $\ell(\Omega,E)$ and 
  \[
  x=\sum_{n\in\Omega} x_{n}\varphi_{n},\quad x\in \ell(\Omega,E).
  \]
 \item [b)] Then $(\delta_{\infty}\varphi_{\infty}+\sum_{n=1}^{k}(\delta_{n}-\delta_{\infty}) \varphi_{n})_{k\in\N}$ 
 is a Schauder decomposition of $c(\N,E)$ and 
 \[
  x=x_{\infty}\varphi_{\infty} +\sum_{n=1}^{\infty}(x_{n}-x_{\infty})\varphi_{n},\quad x\in c(\N,E).
 \]
 \end{enumerate}
\end{prop}
\begin{proof}
Let us begin with a). First, we note that $(\varphi_{n})_{n\in\Omega}$ is an unconditional equicontinuous 
Schauder basis of $s(\Omega)$, $\Omega=\N^{d}$, $\N_{0}^{d}$, $\Z^{d}$, 
since $s(\Omega)$ is a nuclear Fr\'echet space. 
Now, for $x=(x_{n})\in \ell(\Omega,E)$ let $(P_{k}^{E})$ be the sequence 
in $\ell(\Omega,E)$ given by $P_{k}^{E}(x):=\sum_{|n|\leq k} x_{n}\varphi_{n}$. 
It is easy to see that $P_{k}^{E}$ is a continuous projection on $\ell(\Omega,E)$, 
$P_{k}^{E}P_{j}^{E}=P_{\min(k,j)}^{E}$ for all $k,j\in\N$ and $P_{k}^{E}\neq P_{j}^{E}$ for $k\neq j$.
Let $\varepsilon>0$, $\alpha\in\mathfrak{A}$ and $j\in\N$. 
For $x\in c_{0}(A,E)$ there is $N_{0}\in\N$ such that 
$p_{\alpha}(x_{n}a_{n,j})<\varepsilon$ for all $n\geq N_{0}$. 
Hence we have for $x\in c_{0}(A,E)$
\[
|x-P_{k}^{E}(x)|_{j,\alpha}=\sup_{n>k}p_{\alpha}(x_{n})a_{n,j}\leq\sup_{n\geq N_{0}}p_{\alpha}(x_{n})a_{n,j}\leq\varepsilon
\]
for all $k\geq N_{0}$. For $x\in E^{\N}$ and $l\in\N$ we have 
\[
\left\|x-P_{k}^{E}(x)\right\|_{l,\alpha}=0<\varepsilon
\]
for all $k\geq l$. For $x\in s(\Omega,E)$, $\Omega=\N^{d}$, $\N_{0}^{d}$, $\Z^{d}$, 
we notice that there is $N_{1}\in\N$ such that for all $n\in\Omega$ with $|n|\geq N_{1}$ we have
\[
\frac{(1+|n|^{2})^{j/2}}{(1+|n|^{2})^{j}}=(1+|n|^{2})^{-j/2}<\varepsilon.
\]
Thus we deduce for $|n|\geq N_{1}$
\[
p_{\alpha}(x_{n})(1+|n|^{2})^{j/2}< \varepsilon p_{\alpha}(x_{n})(1+|n|^{2})^{j}
\leq \varepsilon |x|_{2j,\alpha}
\]
and hence
\[
|x-P_{k}^{E}(x)|_{j,\alpha}=\sup_{|n|>k}p_{\alpha}(x_{n})(1+|n|^{2})^{j/2}\leq\sup_{|n|\geq N_{1}}p_{\alpha}(x_{n})(1+|n|^{2})^{j/2}
\leq\varepsilon |x|_{2j,\alpha}
\]
for all $k\geq N_{1}$. Therefore $(P_{k}^{E}(x))$ converges to $x$ in $\ell(\Omega,E)$ and 
\[
x=\lim_{k\to\infty}P_{k}^{E}(x)=\sum_{n\in\Omega}x_{n}\varphi_{n}.
\]

Now, we turn to b). For $x=(x_{n})\in c(\N,E)$ let $(\widetilde{P}_{k}^{E}(x))$ be the sequence in $c(\N,E)$ given by 
$\widetilde{P}_{k}^{E}(x):=x_{\infty}\varphi_{\infty}+\sum_{n=1}^{k} (x_{n}-x_{\infty}) \varphi_{n}$. 
Again, it is easy to see that $\widetilde{P}_{k}^{E}$ is a continuous projection on $c(\N,E)$, 
$\widetilde{P}_{k}^{E}\widetilde{P}_{j}^{E}=\widetilde{P}_{\min(k,j)}^{E}$ for all $k,j\in\N$ 
and $\widetilde{P}_{k}^{E}\neq\widetilde{P}_{j}^{E}$ for $k\neq j$.
Let $\varepsilon>0$ and $\alpha\in\mathfrak{A}$. Then there is $N_{2}\in\N$ such that 
$p_{\alpha}(x_{n}-x_{\infty})<\varepsilon$ for all $n\geq N_{2}$. Thus we obtain 
\[
|x-\widetilde{P}_{k}^{E}(x)|_{\alpha}=\sup_{n>k}p_{\alpha}(x_{n}-x_{\infty})\leq \sup_{n\geq N_{2}}p_{\alpha}(x_{n}-x_{\infty})
\leq\varepsilon
\]
for all $k\geq N_{2}$, implying that $(\widetilde{P}_{k}^{E}(x))$ converges to $x$ in $c(\N,E)$ and 
\[
x=\lim_{k\to\infty}\widetilde{P}_{k}^{E}(x)=x_{\infty}\varphi_{\infty}+\sum_{n=1}^{\infty}(x_{n}-x_{\infty})\varphi_{n}.\qedhere
\]
\end{proof}

\begin{thm}\label{thm:sequence.spaces}
 Let $E$ be a sequentially complete lcHs and $\ell(\Omega,E)$ one of the spaces $c_{0}(A,E)$, $E^{\N}$, 
 $s(\N^{d},E)$, $s(\N_{0}^{d},E)$ or $s(\Z^{d},E)$. Then
 \[
  (i)\;\ell(\Omega,E)\cong \ell(\Omega)\varepsilon E,\quad (ii)\;c(\N,E)\cong c(\N)\varepsilon E.
 \]
\end{thm}
\begin{proof} 
The map $S_{\ell(\Omega)}$ is an isomorphism into by \prettyref{thm:linearisation} and, 
in addition, by \prettyref{prop:van.at.inf0} (i) if $\ell(\Omega,E)=c_{0}(A,E)$.
Considering $c(\N,E)$, we observe that for $x\in c(\N)$ 
\[
\delta_{n}(x)=x_{n}\to x_{\infty}=\delta_{\infty}(x),
\]
which implies the convergence $\delta_{n}\to \delta_{\infty}$ 
in $c(\N)_{\gamma}'$ by the Banach--Steinhaus theorem since $c(\N)$ is a Banach space. 
Hence we get 
\[
u(\delta_{\infty})=\lim_{n\to\infty}u(\delta_{n})
=\lim_{n\to\infty}S(u)(n)=\delta_{\infty}(S(u))
\]
for every $u\in c(\N)\varepsilon E$, which implies that $S_{c(\N)}$ is an isomorphism into 
by \prettyref{thm:linearisation}.
From \prettyref{prop:vector_valued_seq_spaces} and \prettyref{prop:reverse_Schauder} we deduce our statement.
\end{proof}

More general, we note that \prettyref{thm:sequence.spaces} holds 
for any lcHs $E$ if $\ell(\Omega,E)=E^{\N}$ by \prettyref{ex:space_of_all_functions}, 
for $E$ with metric ccp if $\ell(\Omega,E)=c_{0}(A,E)$ by \prettyref{ex:cont_loc_comp} (ii), 
and for locally complete $E$ if $\ell(\Omega,E)=s(\Omega,E)$ 
with $\Omega=\N^{d}$, $\N_{0}^{d}$, $\Z^{d}$ by \prettyref{cor:sequence_vanish_infty} b).

\subsubsection*{\textbf{Continuous and differentiable functions on a compact interval}}

We start with continuous functions on compact sets. 
Let $E$ be an lcHs and $\Omega\subset\R^{d}$ compact. 
We equip the space $\mathcal{C}(\Omega,E)$ of continuous functions on $\Omega$ with values in $E$ 
with the system of seminorms given by
\[
|f|_{\alpha}:=\sup_{x\in\Omega}p_{\alpha}(f(x)),\quad f\in\mathcal{C}(\Omega,E),
\]
for $\alpha\in\mathfrak{A}$.
We want to apply our preceding results to intervals.
Let $-\infty<a<b<\infty$ and $T:=(t_{j})_{0\leq j\leq n}$ be a partition of 
the interval $[a,b]$, i.e.\ $a=t_{0}<t_{1}<\ldots<t_{n}=b$. 
The \emph{\gls{hat_function}} $h_{t_{j}}^{T}\colon [a,b]\to\R$ for the partition $T$ are given by 
\[
h_{t_{j}}^{T}(x):=
\begin{cases}
\frac{x-t_{j}}{t_{j}-t_{j-1}} &, t_{j-1}\leq x\leq t_{j},\\
\frac{t_{j+1}-x}{t_{j+1}-t_{j}} &,t_{j}< x\leq t_{j+1},\\
0 &, \text{else},
\end{cases}
\]
for $2\leq j\leq n-1$ and 
\[
h_{a}^{T}(x):=
\begin{cases}
\frac{t_{1}-x}{t_{1}-a} &, a\leq x\leq t_{1},\\
0 &, \text{else},
\end{cases}
\quad 
h_{b}^{T}(x):=
\begin{cases}
\frac{x-t_{n-1}}{b-t_{n-1}} &, t_{n-1}\leq x\leq b,\\
0 &, \text{else}.
\end{cases}
\]
Let $\mathcal{T}:=(t_{n})_{n\in\N_{0}}$ be a dense sequence in $[a,b]$ with $t_{0}=a$, $t_{1}=b$ 
and $t_{n}\neq t_{m}$ for $n\neq m$. For $T^{n}:=\{t_{0},\ldots,t_{n}\}$ there is a (unique) enumeration 
$\sigma\colon \{0,\ldots,n\}\to\{0,\ldots,n\}$ of $T^{n}$ such that $T_{n}:=(t_{\sigma(j)})_{0\leq j\leq n}$ 
is a partition of $[a,b]$ with $T^{n}=\{t_{\sigma(1)},\ldots,t_{\sigma(n)}\}$.
The functions $\varphi^{\mathcal{T}}_{0}:=h^{T_{1}}_{t_{0}}$, 
$\varphi^{\mathcal{T}}_{1}:=h^{T_{1}}_{t_{1}}$ and $\varphi^{\mathcal{T}}_{n}:=h^{T_{n}}_{t_{\sigma(j)}}$ 
with $j=\sigma^{-1}(n)$ for $n\geq 2$ are called \emph{\gls{Schauder_hat_function}} for the sequence $\mathcal{T}$ 
and form a Schauder basis of $\mathcal{C}([a,b])$ with associated coefficient functionals given by 
$\lambda^{\K}_{0}(f):=f(t_{0})$, $\lambda^{\K}_{1}(f):=f(t_{1})$ and 
\[
\lambda^{\K}_{n+1}(f):=f(t_{n+1})
-\sum_{k=0}^{n}\lambda^{\K}_{k}(f)\varphi^{\mathcal{T}}_{k}(t_{n+1}),\quad f\in\mathcal{C}([a,b]),\; n\geq 1,
\]
by \cite[2.3.5 Proposition, p.\ 29]{Semadeni1982}. Looking at the coefficient functionals, we 
see that the right-hand sides even make sense if $f\in\mathcal{C}([a,b],E)$ and thus we define 
$\lambda^{E}_{n}$ on $\mathcal{C}([a,b],E)$ for $n\in\N_{0}$ accordingly. 

\begin{thm}\label{thm:cont.func.interval}
 Let $E$ be an lcHs with metric ccp and $\mathcal{T}:=(t_{n})_{n\in\N_{0}}$ a dense sequence in $[a,b]$ 
 with $t_{0}=a$, $t_{1}=b$ and $t_{n}\neq t_{m}$ for $n\neq m$. 
 Then $(\sum_{n=0}^{k}\lambda_{k}^{E}\varphi^{\mathcal{T}}_{n})_{k\in\N_{0}}$ is a Schauder decomposition of 
 $\mathcal{C}([a,b],E)$ and 
 \[
 f=\sum_{n=0}^{\infty}\lambda^{E}_{n}(f)\varphi^{\mathcal{T}}_{n},\quad f\in \mathcal{C}([a,b],E).
 \]
\end{thm}
\begin{proof}
The spaces $\mathcal{C}([a,b])$ and $\mathcal{C}([a,b],E)$ are $\varepsilon$-compatible by 
\prettyref{ex:cont_usual} if $E$ has metric ccp. 
$\mathcal{C}([a,b])$ is a Banach space and thus barrelled, implying 
that its Schauder basis $(\varphi^{\mathcal{T}}_{n})$ is equicontinuous. 
We note that for all $u\in\mathcal{C}([a,b])\varepsilon E$ and $x\in [a,b]$
\[
\lambda_{n}^{E}(S(u))(x)=u(\delta_{t_{n}})=u(\lambda^{\K}_{n}),\quad n\in\{0,1\},
\] 
and by induction 
\begin{align*}
  \lambda_{n+1}^{E}(S(u))(x)&=u(\delta_{t_{n+1}})-\sum_{k=0}^{n}\lambda^{E}_{k}(S(u))\varphi^{\mathcal{T}}_{k}(t_{n+1})
 =u(\delta_{t_{n+1}})-\sum_{k=0}^{n}u(\lambda^{\K}_{k})\varphi^{\mathcal{T}}_{k}(t_{n+1})\\
&=u(\lambda_{n+1}^{\K}),\quad n\geq 1.
\end{align*}
Thus $(\lambda^{E},\lambda^{\K})$ is consistent, proving our claim by \prettyref{cor:schauder_decomp}.
\end{proof}

If $a=0$, $b=1$ and $\mathcal{T}$ is the sequence of dyadic numbers given in 
\cite[2.1.1 Definitions, p.\ 21]{Semadeni1982}, then $(\varphi^{\mathcal{T}}_{n})$ is the 
so-called Faber--Schauder system. Using the Schauder basis and coefficient functionals of 
the space $\mathcal{C}_{0}(\R)$ of continuous functions vanishing at infinity given 
in \cite[2.7.1, p.\ 41--42]{Semadeni1982} and \cite[2.7.4 Corollary, p.\ 43]{Semadeni1982} 
and that $S_{\mathcal{C}_{0}(\R)}$ is an isomorphism by \prettyref{ex:cont_loc_comp} (ii) 
if $E$ has metric ccp, the corresponding result for the $E$-valued 
counterpart $\mathcal{C}_{0}(\R,E)$ holds as well by a similar reasoning. 
Another corresponding result holds for the space $C^{[\gamma]}_{0,0}([0,1],E)$, 
$0<\gamma<1$, of $\gamma$-H\"{o}lder continuous functions on $[0,1]$ with values in $E$ 
that vanish at zero and at infinity if one uses 
the Schauder basis and coefficient functionals of $C^{[\gamma]}_{0,0}([0,1])$
from \cite[Theorem 2, p.\ 220]{Ciesieski1960} and \cite[Theorem 3, p.\ 230]{Ciesieski1959}.
This result is a bit weaker since \prettyref{ex:hoelder} only guarantees that 
$S_{\mathcal{C}^{[\gamma]}_{0,0}([0,1])}$ is an isomorphism if $E$ is quasi-complete.

Now, we turn to the spaces $\mathcal{C}^{k}([a,b],E)$ of continuously differentiable functions on an 
interval $(a,b)$ with values in an lcHs $E$ such that
all derivatives can be continuously extended to the boundary 
from \prettyref{ex:diff_ext_boundary}. We set $f^{(k)}(x):=(\partial^{k})^{\K}f(x)$ for $x\in(a,b)$ and 
$f\in\mathcal{C}^{k}([a,b])$.
From the Schauder hat functions $(\varphi^{\mathcal{T}}_{n})$ for a dense sequence 
$\mathcal{T}:=(t_{n})_{n\in\N_{0}}$ in $[a,b]$ with $t_{0}=a$, $t_{1}=b$ and 
$t_{n}\neq t_{m}$ for $n\neq m$ and the associated coefficient functionals 
$\lambda_{n}^{\K}$ we can easily get 
a Schauder basis for the space $\mathcal{C}^{k}([a,b])$, $k\in\N$, by 
applying $\int_{a}^{(\cdot)}$ $k$-times to the series representation 
\[
f^{(k)}=\sum_{n=0}^{\infty}\lambda_{n}^{\K}(f^{(k)})\varphi^{\mathcal{T}}_{n},
\quad f\in\mathcal{C}^{k}([a,b]),
\]
where we identified $f^{(k)}$ with its continuous extension. 
The resulting Schauder basis $f_{n}^{\mathcal{T}}\colon [a,b]\to \R$ and associated coefficient 
functionals $\mu_{n}^{\K}\colon \mathcal{C}^{k}([a,b])\to \K$, $n\in\N_{0}$, are 
\begin{align*}
f_{n}^{\mathcal{T}}(x)&=\frac{1}{n!}(x-a)^{n}, &&\mu_{n}^{\K}(f)=f^{(n)}(a), & 0\leq n\leq k-1,\\
f_{n}^{\mathcal{T}}(x)&=\int_{a}^{x}\int_{a}^{s_{k-1}}\cdots\int_{a}^{s_{2}}\int_{a}^{s_{1}}
\varphi^{\mathcal{T}}_{n-k}\d s\d s_{1}\dots \d s_{k-1}, 
&&\mu_{n}^{\K}(f)=\lambda_{n-k}^{\K}(f^{(k)}), & n\geq k,
\end{align*}
for $x\in[a,b]$ and $f\in\mathcal{C}^{k}([a,b])$ (see e.g.\ \cite[p.\ 586--587]{schonefeld1969}, 
\cite[2.3.7, p.\ 29]{Semadeni1982}). 
Again, the mapping rule for the coefficient functionals still makes sense if $f\in\mathcal{C}^{k}([a,b],E)$ 
and so we define $\mu^{E}_{n}$ on $\mathcal{C}^{k}([a,b],E)$ for $n\in\N_{0}$ accordingly.

\begin{thm}\label{thm:cont.diff.func.interval}
 Let $E$ be an lcHs with metric ccp, $k\in\N$, $\mathcal{T}:=(t_{n})_{n\in\N_{0}}$ a dense sequence in $[a,b]$ 
 with $t_{0}=a$, $t_{1}=b$ and $t_{n}\neq t_{m}$ for $n\neq m$. 
 Then $(\sum_{n=0}^{l}\mu_{n}^{E}f^{\mathcal{T}}_{n})_{l\in\N_{0}}$ is a Schauder decomposition of 
 $\mathcal{C}^{k}([a,b],E)$ and 
 \[
 f=\sum_{n=0}^{\infty}\mu^{E}_{n}(f)f^{\mathcal{T}}_{n},\quad f\in \mathcal{C}^{k}([a,b],E).
 \]
\end{thm}
\begin{proof}
The spaces $\mathcal{C}^{k}([a,b])$ and $\mathcal{C}^{k}([a,b],E)$ are $\varepsilon$-compatible by 
\prettyref{ex:diff_ext_boundary} if $E$ has metric ccp. 
The Banach space $\mathcal{C}^{k}([a,b])$ is barrelled giving the equicontinuity of its Schauder basis. 
Due to \prettyref{prop:diff_cons_barrelled} c) we have for all $u\in\mathcal{C}^{k}([a,b])\varepsilon E$, 
$\beta\in\N_{0}$, $\beta\leq k$, and $x\in(a,b)$
\[
(\partial^{\beta})^{E}S(u)(x)=u(\delta_{x}\circ(\partial^{\beta})^{\K}).
\]
Further, for every sequence $(x_{n})$ in $(a,b)$ converging to $t\in\{a,b\}$ we obtain by 
\prettyref{prop:cont_ext} in combination with \prettyref{lem:cont_ext}
applied to $T:=(\partial^{\beta})^{\K}$
\[
\lim_{n\to\infty}(\partial^{\beta})^{E}S(u)(x_{n})=u(\lim_{n\to\infty}\delta_{x_{n}}\circ(\partial^{\beta})^{\K}).
\]
From these observations we deduce that $\mu_{n}^{E}(S(u))=u(\mu_{n}^{\K})$ for all $n\in\N_{0}$, 
i.e.\ $(\mu^{E},\mu^{\K})$ is consistent. 
Therefore our statement is a consequence of \prettyref{cor:schauder_decomp}.
\end{proof}

\subsubsection*{\textbf{Holomorphic functions}}

In this short subsection we show how to get the result on power series expansion of holomorphic functions 
from the introduction. 
Let $E$ be an lcHs over $\C$, $z_{0}\in\C$, $r\in(0,\infty]$ and 
equip $\mathcal{O}(\D_{r}(z_{0}),E)$ with the topology $\tau_{c}$ of compact convergence. 

\begin{thm}\label{thm:powerseries}
  Let $E$ be a locally complete lcHs over $\C$, $z_{0}\in\C$ and $r\in(0,\infty]$. 
  Then $(f\mapsto\sum_{n=0}^{k}\frac{(\partial^{n}_{\C})^{E}f(z_{0})}{n!}(\cdot - z_{0})^{n})_{k\in\N_{0}}$ 
  is a Schauder decomposition 
  of $\mathcal{O}(\D_{r}(z_{0}),E)$ and 
   \[
    f=\sum_{n=0}^{\infty}\frac{(\partial^{n}_{\C})^{E}f(z_{0})}{n!}(\cdot - z_{0})^{n},
    \quad f\in\mathcal{O}(\D_{r}(z_{0}),E).
   \]
\end{thm}
\begin{proof}
The spaces $\mathcal{O}(\D_{r}(z_{0}))$ and $\mathcal{O}(\D_{r}(z_{0}),E)$ are $\varepsilon$-compatible 
by \prettyref{prop:co_top_isomorphism} and \eqref{eq:holomorphic_coincide_1} 
(cf.\ \cite[Theorem 9, p.\ 232]{B/F/J}) if $E$ is locally complete. 
Further, the Schauder basis $((\cdot - z_{0})^{n})$ of 
$\mathcal{O}(\D_{r}(z_{0}))$ is equicontinuous 
since the Fr\'{e}chet space $\mathcal{O}(\D_{r}(z_{0}))$ is barrelled. 
Due to \prettyref{prop:complex_diff_cons_strong} we have for all $u\in\mathcal{O}(\D_{r}(z_{0}))\varepsilon E$
\[
(\partial^{n}_{\C})^{E}S(u)(z)=u(\delta_{z}\circ(\partial^{n}_{\C})^{\C}),\quad n\in\N_{0},\,z\in\D_{r}(z_{0}),
\]
which yields that $(\lambda^{E},\lambda^{\C})$ is consistent where 
$\lambda^{E}\colon\mathcal{O}(\D_{r}(z_{0}),E)\to E^{\N_{0}}$ is given by 
$\lambda^{E}_{n}(f):=\frac{(\partial^{n}_{\C})^{E}f(z_{0})}{n!}$ for $n\in\N_{0}$ 
(and analogously for $E$ replaced by $\C$). 
Hence \prettyref{cor:schauder_decomp} implies our statement.
\end{proof}

\prettyref{thm:powerseries} holds for holomorphic functions in several variables as well 
(see \cite[Theorem 5.7, p.\ 264]{kruse2019_4}).

\subsubsection*{\textbf{Fourier expansions}}

In this subsection we turn our attention to \gls{Fourier_expansion} in the 
Schwartz space $\mathcal{S}(\R^{d},E)$ and in the space $\mathcal{C}^{\infty}_{2\pi}(\R^{d},E)$ of smooth functions that are $2\pi$-periodic in each variable. 

We recall the definition of the Hermite functions. For $n\in\N_{0}$ we set 
\[
h_{n}\colon\R\to\R,\;
h_{n}(x):=(2^{n}n!\sqrt{\pi})^{-1/2}\Bigl(x-\frac{d}{dx}\Bigr)^{n}\e^{-x^{2}/2}=(2^{n}n!\sqrt{\pi})^{-1/2}H_{n}(x)\e^{-x^{2}/2},
\]
with the \emph{\gls{Hermite_poly}} $H_{n}$ of degree $n$ which can be computed recursively by 
\[
 H_{0}(x)=1,\; H_{n+1}(x)=2xH_{n}(x)-H_{n}'(x)\;\text{and}\; H_{n}'(x)=2nH_{n-1}(x), \quad 
 x\in\R,\; n\in\N_{0}.
\]
For $n=(n_{k})\in\N_{0}^{d}$ we define the $n$\emph{-th \gls{Hermite_function}} by
\[
\gls{h_n}\colon\R^{d}\to\R,\;
h_{n}(x):=\prod_{k=1}^{d} h_{n_{k}}(x_{k}),\quad\text{and}\quad
\gls{H_n}\colon\R^{d}\to\R,\;
H_{n}(x):=\prod_{k=1}^{d} H_{n_{k}}(x_{k}).
\]

\begin{prop}\label{prop:Schwartz_pettis_int}
Let $E$ be a locally complete lcHs, $f\in\mathcal{S}(\R^{d},E)$ and $n\in\N_{0}^{d}$. 
Then $fh_{n}$ is Pettis-integrable on $\R^{d}$.
\end{prop}
\begin{proof}
First, we set $\psi\colon\R^{d}\to\R$, $\psi(x):=\e^{-|x|^{2}/2}$, as well as 
$g\colon\R^{d}\to [0,\infty)$, $g(x):=\e^{|x|^{2}/2}$. 
Then $\psi\in\mathcal{L}^{1}(\R^{d},\lambda)$ and $\psi g=1$. 
Moreover, let $u\colon\R^{d}\to E$, $u(x):=f(x)h_{n}(x)g(x)$, and note that
\[
 (\partial^{e_{j}})^{E}u(x)
=(\partial^{e_{j}})^{E}f(x)h_{n}(x)g(x)
 +f(x)g(x)\partial^{e_{j}}h_{n}(x)
 +f(x)h_{n}(x)g(x)x_{j}
\]
where
\begin{align*}
  \partial^{e_{j}}h_{n}(x)
&=(2^{n_{j}}n_{j}!\sqrt{\pi})^{-1/2}(H_{n_{j}}'(x_{j})\e^{-x_{j}^{2}/2}-H_{n_{j}}(x_{j})x_{j}\e^{-x_{j}^{2}/2})
  \prod_{k=1,k\neq j}^{d} h_{n_{k}}(x_{k})\\
&= (2^{n_{j}}n_{j}!\sqrt{\pi})^{-1/2}(2n_{j}H_{n_{j}-1}(x_{j})-x_{j}H_{n_{j}}(x_{j}))\e^{-x_{j}^{2}/2}
  \prod_{k=1,k\neq j}^{d} h_{n_{k}}(x_{k})
\end{align*}
for all $x=(x_{k})\in\R^{d}$ and $1\leq j\leq d$.
We set $C_{n}:=(\prod_{i=1}^{d}2^{n_{i}}n_{i}!\sqrt{\pi})^{-1/2}$ and observe that 
\[
 g(x)\partial^{e_{j}}h_{n}(x)=\e^{|x|^{2}/2}\partial^{e_{j}}h_{n}(x)
 =C_{n}(2n_{j}H_{n_{j}-1}(x_{j})-x_{j}H_{n_{j}}(x_{j}))\prod_{k=1,k\neq j}^{d} H_{n_{k}}(x_{k})
\]
is a polynomial in $d$ variables. The functions given by
\[
h_{n}(x)g(x)=\e^{|x|^{2}/2}h_{n}(x)=C_{n}H_{n}(x)\quad\text{and}\quad h_{n}(x)g(x)x_{j}=C_{n}H_{n}(x)x_{j}
\]
are polynomials in $d$ variables as well. Thus there are $m\in\N$ and $C>0$ such that 
\[
\max\bigl(|h_{n}(x)g(x)|,|g(x)\partial^{e_{j}}h_{n}(x)|,|h_{n}(x)g(x)x_{j}|\bigr)\leq C(1+|x|^{2})^{m/2}
\]
for all $x\in\R^{d}$ and $1\leq j\leq d$, which implies 
\[
     p_{\alpha}((\partial^{e_{j}})^{E}u(x))\\
\leq C \bigl(p_{\alpha}((\partial^{e_{j}})^{E}f(x))(1+|x|^{2})^{m/2}
 +2p_{\alpha}(f(x))(1+|x|^{2})^{m/2}\bigr)
\]
for all $\alpha\in\mathfrak{A}$ and hence 
\[
\sup_{\substack{x\in\R^{d}\\ \beta\in\N_{0}^{d},|\beta|\leq 1}}p_{\alpha}((\partial^{\beta})^{E}u(x))
\leq 3C|f|_{\mathcal{S}(\R^{d}),m,\alpha}.
\]
Therefore $u=fh_{n}g$ is (weakly) $\mathcal{C}^{1}_{b}$, which yields
$u\in\mathcal{C}_{b}^{[1]}(\R^{d},E)$ by \prettyref{prop:abs_conv_comp_C_1_b}. 
Further, we set $h\colon \R^{d}\to (0,\infty)$, $h(x):=1+|x|^2$, and observe that 
\[
\sup_{x\in\R^{d}}p_{\alpha}(u(x)h(x))\leq\sup_{x\in\R^{d}}p_{\alpha}(f(x))|h_{n}(x)g(x)h(x)|\leq C|f|_{m+2,\alpha}<\infty
\]
for all $\alpha\in\mathfrak{A}$. In addition, we remark that for every $\varepsilon>0$ there is $r>0$ 
such that $1\leq\varepsilon h(x)$ for all $x\notin\overline{\mathbb{B}_{r}(0)}=:K$. 
We deduce from \prettyref{prop:pettis.ccp.to.loc.complete} (iii) that $fh_{n}$ 
is Pettis-integrable on $\R^{d}$.
\end{proof}

Due to the previous proposition we can define the $n$-th Fourier coefficient of 
$f\in\mathcal{S}(\R^{d},E)$ by
\[
  \widehat{f}(n):=\mathscr{F}^{E}_{n}(f)
:=\int_{\R^{d}}f(x)\overline{h_{n}(x)}\d x=\int_{\R^{d}}f(x)h_{n}(x)\d x, \quad n\in\N_{0}^{d},
\]
if $E$ is locally complete.
We know that the map 
\[
 \mathscr{F}^{\K}\colon\mathcal{S}(\R^{d})\to s(\N_{0}^{d}),\;
 \mathscr{F}^{\K}(f):=\bigl(\widehat{f}(n)\bigr)_{n\in\N_{0}^{d}},
\]
is an isomorphism (see e.g.\ \cite[Satz 3.7, p.\ 66]{Kaballo}). 
We improve this result to locally complete $E$ 
and derive a Schauder decomposition of $\mathcal{S}(\R^{d},E)$ as well.

\begin{thm}\label{thm:fourier.rap.dec}
Let $E$ be a locally complete lcHs. Then the following holds:
\begin{enumerate}
\item [a)] $(\sum_{n\in\N_{0}^{d}, |n|\leq k}\mathscr{F}^{E}_{n}h_{n})_{k\in\N}$ 
is a Schauder decomposition of $\mathcal{S}(\R^{d},E)$ and 
\[
f=\sum_{n\in\N_{0}^{d}}\widehat{f}(n)h_{n}, \quad f\in\mathcal{S}(\R^{d},E).
\] 
\item [b)] The map
 \[
  \mathscr{F}^{E}\colon \mathcal{S}(\R^{d},E)\to s(\N_{0}^{d},E),
  \;\mathscr{F}^{E}(f):=\bigl(\widehat{f}(n)\bigr)_{n\in\N_{0}^{d}},
 \]
 is an isomorphism and
 \[
 \mathscr{F}^{E}=S_{s(\N_{0}^{d})}\circ(\mathscr{F}^{\K}\varepsilon\id_{E})\circ S_{\mathcal{S}(\R^{d})}^{-1}.
 \]
\end{enumerate}
\end{thm}
\begin{proof}
Let us begin with part a). 
Due to \prettyref{cor:Schwartz} the spaces $\mathcal{S}(\R^{d})$ and $\mathcal{S}(\R^{d},E)$ 
are $\varepsilon$-compatible and the inverse of the isomorphism 
$S\colon\mathcal{S}(\R^{d})\varepsilon E\to \mathcal{S}(\R^{d},E)$ is given by the map 
$R^{t}\colon\mathcal{S}(\R^{d},E)\to\mathcal{S}(\R^{d})\varepsilon E$, $f\mapsto\mathcal{J}^{-1}\circ R_{f}^{t}$,
according to \prettyref{thm:full_linearisation}. Moreover, $\mathcal{S}(\R^{d})$ is a nuclear Fr\'{e}chet space, 
thus barrelled, and hence its Schauder basis $(h_{n})$ is equicontinuous and unconditional.
From the Pettis-integrability of $fh_{n}$ 
by \prettyref{prop:Schwartz_pettis_int} and \prettyref{prop:pettis_consistent} 
with $(T^{E}_{0},T^{\K}_{0}):=(h_{n}\id_{E^{\R^{d}}},h_{n}\id_{\K^{\R^{d}}})$ 
we obtain that $(\mathscr{F}^{E}, \mathscr{F}^{\K})$ is consistent. 
Hence we conclude our statement from \prettyref{cor:schauder_decomp}.

Let us turn to part b). First, we show that the map $\mathscr{F}^{E}$ 
is well-defined. Let $f\in\mathcal{S}(\R^{d},E)$. 
Then $e'\circ f\in \mathcal{S}(\R^{d})$ and 
\[
 \langle e',\mathscr{F}^{E}(f)_{n}\rangle=\langle e',\widehat{f}(n)\rangle
 =\widehat{e'\circ f}(n)=\mathscr{F}^{\K}(e'\circ f)_{n}
\]
for every $n\in\N_{0}^{d}$ and $e'\in E'$. Thus we have $\mathscr{F}^{\K}(e'\circ f)\in s(\N_{0}^{d})$ 
for every $e'\in E'$, which implies by \cite[Mackey's theorem 23.15, p.\ 268]{meisevogt1997}
that $\mathscr{F}^{E}(f)\in s(\N_{0}^{d},E)$ and that $\mathscr{F}^{E}$ is well-defined. 
Due to \prettyref{cor:Schwartz} and \prettyref{cor:sequence_vanish_infty} the maps 
$S_{\mathcal{S}(\R^{d})}$ and $S_{s(\N_{0}^{d})}$ are isomorphisms, which implies that 
$\mathscr{F}^{E}$ is also an isomorphism with 
$\mathscr{F}^{E}=S_{s(\N_{0}^{d})}\circ(\mathscr{F}^{\K}\varepsilon\id_{E})\circ S_{\mathcal{S}(\R^{d})}^{-1}$
by \prettyref{thm:eps_prod_surj_inj} b).
\end{proof}

Our last example of this subsection is devoted to Fourier expansions in the space $\mathcal{C}^{\infty}_{2\pi}(\R^{d},E)$. 
We recall that $\mathcal{C}^{\infty}_{2\pi}(\R^{d},E)$ denotes the topological subspace of 
$\mathcal{CW}^{\infty}(\R^{d},E)$ consisting of the functions 
which are $2\pi$-periodic in each variable. 
Due to \prettyref{lem:pettis.loc.complete} we are able to define the $n$-th Fourier coefficient of 
$f\in\mathcal{C}^{\infty}_{2\pi}(\R^{d},E)$ by
\[
\widehat{f}(n):=\mathfrak{F}^{E}_{n}(f):=(2\pi)^{-d}\int_{[-\pi,\pi]^{d}}f(x)\e^{-\iu\langle n,x\rangle}\d x, \quad n\in\Z^{d},
\]
where $\langle\cdot,\cdot\rangle$ is the usual scalar product on $\R^{d}$, if $E$ is locally complete. 
We know that the map 
\[
 \mathfrak{F}^{\C}\colon\mathcal{C}^{\infty}_{2\pi}(\R^{d})\to s(\Z^{d}),\;
 \mathfrak{F}^{\C}(f):=\bigl(\widehat{f}(n)\bigr)_{n\in\Z^{d}},
\]
is an isomorphism (see e.g.\ \cite[Satz 1.7, p.\ 18]{Kaballo}), which we lift to the $E$-valued case.

\begin{thm}\label{thm:fourier_periodic}
Let $E$ be a locally complete lcHs over $\C$.
\begin{enumerate}
\item[a)] Then $(\sum_{n\in\Z^{d}, |n|\leq k}\mathfrak{F}^{E}_{n}\e^{\iu\langle n,\cdot\rangle})_{k\in\N}$
is a Schauder decomposition of $\mathcal{C}^{\infty}_{2\pi}(\R^{d},E)$ and
\[
f=\sum_{n\in\Z^{d}}\widehat{f}(n)\e^{\iu\langle n,\cdot\rangle}, \quad f\in\mathcal{C}^{\infty}_{2\pi}(\R^{d},E).
\]
\item[b)] The map 
\[
\mathfrak{F}^{E}\colon\mathcal{C}^{\infty}_{2\pi}(\R^{d},E)\to s(\Z^{d},E),\;
\mathfrak{F}^{E}(f):=\bigl(\widehat{f}(n)\bigr)_{n\in\Z^{d}},
\] 
is an isomorphism and 
\[
 \mathfrak{F}^{E}
=S_{s(\Z^{d})}\circ (\mathfrak{F}^{\C}\varepsilon\id_{E})\circ S^{-1}_{\mathcal{C}^{\infty}_{2\pi}(\R^{d})}.
\]
\end{enumerate}
\end{thm}
\begin{proof}
The spaces $\mathcal{C}^{\infty}_{2\pi}(\R^{d})$ and $\mathcal{C}^{\infty}_{2\pi}(\R^{d},E)$ 
are $\varepsilon$-compatible by \prettyref{ex:smooth_periodic_eps_compat}.

The space $\mathcal{C}^{\infty}_{2\pi}(\R^{d})$ is barrelled since it is a nuclear Fr\'{e}chet space 
and thus its Schauder basis $(\e^{\iu\langle n,\cdot\rangle})$ is equicontinuous and unconditional. 
By \prettyref{thm:full_linearisation} the inverse of $S_{\mathcal{C}^{\infty}_{2\pi}(\R^{d})}$ is given by 
$R^{t}\colon \mathcal{C}^{\infty}_{2\pi}(\R^{d},E) \to \mathcal{C}^{\infty}_{2\pi}(\R^{d}) \varepsilon E$,
$f\mapsto \mathcal{J}^{-1}\circ R_{f}^{t}$.
From the Pettis-integrability of $f\e^{-\iu\langle n,\cdot\rangle}$ and \prettyref{prop:pettis_consistent} 
with $(T^{E}_{0},T^{\K}_{0}):=(\e^{-\iu\langle n,\cdot\rangle}\id_{E^{\R^{d}}},
\e^{-\iu\langle n,\cdot\rangle}\id_{\C^{\R^{d}}})$ 
we obtain that $(\mathfrak{F}^{E}, \mathfrak{F}^{\C})$ is consistent. 
Hence we conclude part a) from \prettyref{cor:schauder_decomp}.

Let us turn to part b). As in \prettyref{thm:fourier.rap.dec} it follows from 
\cite[Mackey's theorem 23.15, p.\ 268]{meisevogt1997} that the map $\mathfrak{F}^{E}$ is well-defined.
Due to \prettyref{cor:sequence_vanish_infty} and \prettyref{ex:smooth_periodic_eps_compat} the maps $S_{s(\Z^{d})}$
and $S_{\mathcal{C}^{\infty}_{2\pi}(\R^{d})}$ are isomorphisms, which implies that $\mathfrak{F}^{E}$ 
is an isomorphism as well with 
$\mathfrak{F}^{E}
=S_{s(\Z^{d})}\circ(\mathscr{F}^{\C}\varepsilon\id_{E})\circ S_{\mathcal{C}^{\infty}_{2\pi}(\R^{d})}^{-1}$
by \prettyref{thm:eps_prod_surj_inj} b).
\end{proof}

For quasi-complete $E$ \prettyref{thm:fourier_periodic} is already known by \cite[Satz 10.8, p.\ 239]{Kaballo}.
\section{Representation by sequence spaces}
\label{sect:sequence_space}
Our last section is dedicated to the representation of weighted spaces of $E$-valued functions by weighted
spaces of $E$-valued sequences if there is a counterpart of this representation in the scalar-valued case 
involving the coefficient functionals associated to a Schauder basis (see \prettyref{rem:Schauder_coeff_set_uni} b)).
We only touched upon this problem in \prettyref{sect:schauder} 
for special cases like $\mathcal{S}(\R^{d},E)$ and $\mathcal{C}^{\infty}_{2\pi}(\R^{d},E)$ in
\prettyref{thm:fourier.rap.dec} b) and \prettyref{thm:fourier_periodic} b).
We solve this problem in a different way by an application of our extension results from \prettyref{sect:extension}.
As an example we treat the space $\mathcal{O}(\D_{R}(0),E)$ of holomorphic functions 
and the multiplier space $\mathcal{O}_{M}(\R,E)$ of the Schwartz space 
(see \prettyref{cor:Schauder_coeff_space_multiplier}).

\begin{thm}\label{thm:Schauder_coeff_space}
Let $E$ be a locally complete lcHs, $G\subset E'$ determine boundedness and $\F$ and $\FE$ resp.\ 
$\ell(\N)$ and $\ell(\N,E)$ be $\varepsilon$-into-compatible with $e'\circ g\in\ell(\N)$ for all $e'\in E'$ 
and $g\in\ell(\N,E)$.
Let $(f_{n})_{n\in\N}$ be an equicontinuous Schauder basis of $\F$ 
with associated coefficient functionals $(T^{\K}_{n})_{n\in \N}$ such that 
\[
T^{\K}\colon \F\to\ell(\N),\;T^{\K}(f):=(T^{\K}_{n}(f))_{n\in \N},
\]
is an isomorphism and let there be $T^{E}\colon \FE\to E^{\N}$ such that $(T^{E},T^{\K})$ 
is a strong, consistent family for $(\mathcal{F},E)$. 
If 
\begin{enumerate}
\item[(i)] $\F$ is a Fr\'{e}chet--Schwartz space, or 
\item[(ii)] $E$ is sequentially complete, $G=E'$ and $\F$ is a semi-Montel BC-space,
\end{enumerate}
then the following holds:
\begin{enumerate}
\item[a)] $\mathcal{F}_{G}(\N,E)=\ell(\N,E)$.
\item[b)] $\ell(\N)$ and $\ell(\N,E)$ are $\varepsilon$-compatible, in particular, $\ell(\N)\varepsilon E\cong\ell(\N,E)$.
\item[c)] The map 
\[
T^{E}\colon \FE\to\ell(\N,E),\;T^{E}(f):=(T^{E}_{n}(f))_{n\in \N},
\]
is a well-defined isomorphism, $\F$ and $\FE$ are $\varepsilon$-compatible, in particular, 
$\F\varepsilon E\cong\FE$, and $T^{E}=S_{\ell(\N)}\circ(T^{\K}\varepsilon\id_{E})\circ S_{\F}^{-1}$.
\end{enumerate}
\end{thm}
\begin{proof}
a)(1) First, we remark that $\N$ is a set of uniqueness for $(T^{\K},\mathcal{F})$. 
Let $u\in\F\varepsilon E$ and $n\in\N$. Then 
\begin{align}\label{eq:Schauder_coeff_space}
 R_{\N,G}(S_{\F}(u))(n)&=(T^{E}\circ S_{\F})(u)(n)=T^{E}_{n}(S_{\F}(u))=u(T^{\K}_{n})=u(\delta_{n}\circ T^{\K})\notag\\
&=(u\circ (T^{\K})^{t})(\delta_{n})
 =(T^{\K}\varepsilon\id_{E})(u)(\delta_{n})\notag\\
&=\bigl(S_{\ell(\N)}\circ (T^{\K}\varepsilon\id_{E})\bigr)(u)(n)
\end{align}
by consistency and the $\varepsilon$-into-compatibility, yielding $\mathcal{F}_{G}(\N,E)\subset \ell(\N,E)$ once we have shown that $R_{\N,G}$ is surjective, which we postpone to part b).

a)(2) Let $g\in\ell(\N,E)$. Then $e'\circ g\in\ell(\N)$ for all $e'\in E'$ and $g_{e'}:=(T^{\K})^{-1}(e'\circ g)\in\F$. 
We note that $T^{\K}_{n}(g_{e'})=(e'\circ g)(n)$ for all $n\in\N$, which implies 
$\ell(\N,E)\subset \mathcal{F}_{G}(\N,E)$. 

b) We only need to show that $S_{\ell(\N)}$ is surjective. Let $g\in\ell(\N,E)$, which implies 
$g\in\mathcal{F}_{G}(\N,E)$ by part a)(2). 

We claim that $R_{\N,G}$ is surjective. In case (i) this follows directly from \prettyref{thm:ext_FS_set_uni}. 
Let us turn to case (ii) and denote by $(f_{n})_{n\in\N}$ the equicontinuous Schauder basis of $\F$ 
associated to $(T^{\K}_{n})_{n\in \N}$. We check that condition (ii) of \prettyref{thm:ext_F_semi_M} 
is fulfilled. Let $f'\in\F'$ and set 
\[
f_{k}'\colon \F\to\K,\;f_{k}'(f):=\sum_{n=1}^{k}T^{\K}_{n}(f)f'(f_{n}),
\]
for $k\in\N$. Then $f_{k}'\in\F'$ for every $k\in\N$ and $(f_{k}')$ converges to $f'$ in $\F_{\sigma}'$ 
since $(\sum_{n=1}^{k}T^{\K}_{n}(f)f_{n})$ converges to $f$ in $\F$. From the equicontinuity of the 
Schauder basis we deduce that $(f_{k}')$ converges to $f'$ in $\F_{\kappa}'$ 
by \cite[8.5.1 Theorem (b), p.\ 156]{Jarchow}. 
Let $f\in\mathcal{F}_{E'}(\N,E)$. For each $e'\in E'$ and $k\in\N$ we have 
\[
\mathscr{R}^{t}_{f}(f_{k}')(e')=f_{k}'(f_{e'})=\sum_{n=1}^{k}T^{\K}_{n}(f_{e'})f'(f_{n})
=e'(\sum_{n=1}^{k}f(n)f'(f_{n}))
\]
since $f\in\mathcal{F}_{E'}(\N,E)$, implying $\mathscr{R}^{t}_{f}(f_{k}')\in\mathcal{J}(E)$. 
Hence we can apply \prettyref{thm:ext_F_semi_M} (ii) and obtain that $R_{\N,E'}$ is surjective, finishing the proof of part a)(1). 

Thus there is $u\in\F\varepsilon E$ such that $R_{\N,E'}(S_{\F}(u))=g$ in both cases.
Then $(T^{\K}\varepsilon\id_{E})(u)\in\ell(\N)\varepsilon E$ and from \eqref{eq:Schauder_coeff_space} we derive 
\[
S_{\ell(\N)}((T^{\K}\varepsilon\id_{E})(u))=R_{\N,G}(S_{\F}(u))=g,
\]
proving the surjectivity of $S_{\ell(\N)}$. 

c) First, we note that the map $T^{E}$ is well-defined. Indeed, we have $(e'\circ T^{E})(f)=T^{\K}(e'\circ f)\in\ell(\N)$ 
for all $f\in\FE$ and $e'\in E'$ by the strength of the family. 
Part a) implies that $T^{E}(f)\in\mathcal{F}_{G}(\N,E)=\ell(\N,E)$ 
and thus the map $T^{E}$ is well-defined and its linearity follows from the linearity of the $T^{E}_{n}$ for $n\in\N$.
Next, we prove that $T^{E}$ is surjective. 
Let $g\in\ell(\N,E)$. Since $T^{\K}\varepsilon\id_{E}$ is an isomorphism 
and $S_{\ell(\N)}$ by part b) as well, we obtain that 
$u:=((T^{\K}\varepsilon\id_{E})^{-1}\circ S_{\ell(\N)}^{-1})(g)\in \F\varepsilon E$. 
Therefore $S_{\F}(u)\in\FE$ and from \eqref{eq:Schauder_coeff_space} we get
\[
T^{E}(S_{\F}(u))=(T^{E}\circ S_{\F})(u)=\bigl(S_{\ell(\N)}\circ (T^{\K}\varepsilon\id_{E})\bigr)(u)=g,
\]
which means that $T^{E}$ is surjective. The injectivity of $T^{E}$ by \prettyref{prop:injectivity}, implies that 
\[
S_{\F}=(T^{E})^{-1}\circ\bigl(S_{\ell(\N)}\circ (T^{\K}\varepsilon\id_{E})\bigr),
\]
yielding the surjectivity of $S_{\F}$ and thus the $\varepsilon$-compatibility of $\F$ and $\FE$. 
Furthermore, we have $T^{E}=S_{\ell(\N)}\circ (T^{\K}\varepsilon\id_{E})\circ S_{\F}^{-1}$, resulting in 
$T^{E}$ being an isomorphism. 
\end{proof}

We note that one should not confuse the coefficient space $\ell(\N)$ of the Schauder
series expansion of functions from $\F$ in the theorem above with the space $\ell^{1}=\ell^{1}(\N)$ 
of absolutely summable sequences. 
We remark again (see \prettyref{thm:schauder_decomp}) that the index set of the equicontinuous Schauder basis of $\F$ 
in \prettyref{thm:Schauder_coeff_space} need not be $\N$ (or $\N_{0}$) but may be any other countable index set 
as long as the equicontinuous Schauder basis is unconditional which is, for instance, always fulfilled 
if $\F$ is nuclear by \cite[21.10.1 Dynin-Mitiagin Theorem, p.\ 510]{Jarchow}.

\prettyref{thm:Schauder_coeff_space} (i) gives another proof of \prettyref{thm:fourier.rap.dec} b) and 
\prettyref{thm:fourier_periodic} b). 
Let us demonstrate an application of the preceding theorem which relates the space of 
$\mathcal{O}(\D_{R}(0),E)$, $0<R\leq\infty$, of holomorphic functions on $\D_{R}(0)$ with values in a complex 
locally complete lcHs $E$ (see \prettyref{thm:powerseries}) and the K\"othe space $\lambda^{\infty}(A_{R},E)$ 
with K\"othe matrix $A_{R}:=(r_{j}^{k})_{k\in\N_{0},j\in\N}$ 
for some strictly increasing sequence $(r_{j})_{j\in\N}$ in $(0,R)$ converging to $R$ 
(see \prettyref{cor:sequence_vanish_infty}), using the sequence of Taylor coefficients of a holomorphic function.

\begin{cor}\label{cor:Schauder_coeff_space_holom}
Let $E$ be a locally complete lcHs over $\C$, $0<R\leq\infty$ and define 
the K\"othe matrix $A_{R}:=(r_{j}^{k})_{k\in\N_{0},j\in\N}$ 
for some strictly increasing sequence $(r_{j})_{j\in\N}$ in $(0,R)$ converging to $R$. 
Then $\lambda^{\infty}(A_{R})\varepsilon E\cong\lambda^{\infty}(A_{R},E)$
and
 \[
 \lambda^{E}\colon \mathcal{O}(\D_{R}(0),E)\to \lambda^{\infty}(A_{R},E),\;
 \lambda^{E}(f):=\Bigl(\frac{(\partial^{k}_{\C})^{E}f(0)}{k!}\Bigr)_{k\in\N_{0}},
 \]
is an isomorphism with 
$\lambda^{E}=S_{\lambda^{\infty}(A_{R})}\circ(\lambda^{\C}\varepsilon \id_{E})\circ S_{\mathcal{O}(\D_{R}(0))}^{-1}$.
\end{cor}
\begin{proof}
By \prettyref{prop:co_top_isomorphism} and \eqref{eq:holomorphic_coincide_1} the spaces $\mathcal{O}(\D_{R}(0))$ and $\mathcal{O}(\D_{R}(0),E)$ are $\varepsilon$-compatible. 
Moreover, $\lambda^{\infty}(A_{R})$ and $\lambda^{\infty}(A_{R},E)$ are $\varepsilon$-compatible by 
\prettyref{cor:sequence_vanish_infty} as $\lim_{k\to\infty}(\frac{r_{j}}{r_{j+1}})^{k}=0$ for any $j\in\N$. 
Clearly, we have $e'\circ x\in\lambda^{\infty}(A_{R})$ for all $e'\in E'$ and $x\in\lambda^{\infty}(A_{R},E)$. 
The space $\mathcal{O}(\D_{R}(0))$ with the topology $\tau_{c}$ of compact convergence 
is a nuclear Fr\'echet space and thus a Fr\'echet--Schwartz space. 
In particular, this space is barrelled and its Schauder basis of monomials 
$(z \mapsto z^{k})_{k\in\N_{0}}$ is equicontinuous. The corresponding coefficient functionals 
are given by $\lambda^{\C}_{k}$ and the map $\lambda^{\C}$ is an isomorphism 
by \cite[Example 27.27, p.\ 341--342]{meisevogt1997}. By the proof of \prettyref{thm:powerseries} the family 
$(\lambda^{E},\lambda^{\C})$ is consistent for $(\mathcal{O},E)$ 
and its strength follows from \prettyref{prop:complex_diff_cons_strong}. 
Now, we can apply \prettyref{thm:Schauder_coeff_space} (i), yielding our statement.
\end{proof}

Let us present another application of \prettyref{thm:Schauder_coeff_space} 
to the space $\mathcal{O}_{M}(\R^{d},E)$ of 
multipliers for the Schwartz space from \prettyref{ex:weighted_smooth_functions} d). 
For simplicity we restrict to the case $d=1$.
Fix a compactly supported test function $\varphi\in\mathcal{C}^{\infty}_{c}(\R)$ with 
$\varphi(x)=1$ for $x\in[0,\frac{1}{4}]$ and $\varphi(x)=0$ for $x\geq\frac{1}{2}$.   
For $f\in\mathcal{C}^{\infty}(\R,E)$ we set 
\[
f_{j}(x):=f(x+j)-\sum_{k=0}^{\infty}a_{k}\varphi(-2^k(x-1))f(-2^k(x-1+j)+1) ,\;x\in[0,1],\,j\in\Z,
\]
where 
\[
 a_{k}:=\prod_{j=0,j\neq k}^{\infty}\frac{1+2^j}{2^j-2^k},\;k\in\N_{0}.
\]
Fixing $x\in[0,1)$, we observe that $f_{j}(x)$ is well-defined for each $j\in\Z$ 
since there are only finitely many summands due to the compact support of $\varphi$ 
and $-2^k(x-1)\to\infty$ for $k\to\infty$. For $x=1$ we have $f_{j}(1)=0$ for each $j$ 
and the convergence of the series in $E$ follows from the uniform continuity of $f$ on $[0,1]$, $f(0)=0$ 
and $\sum_{k=0}^{\infty}a_{k}=1$ by the case $n=0$ in \cite[Lemma (iii), p.\ 625]{seeley1964}. 
For each $e'\in E'$ and $j\in\Z$ we note that 
\[
e'(f_{j}(x))=(e'\circ f)(x+j)-\sum_{k=0}^{\infty}a_{k}\varphi(-2^k(x-1))(e'\circ f)(-2^k(x-1+j)+1),\;
x\in[0,1],
\]
which implies that $e'\circ f_{j}\in \mathcal{E}_{0}$ by \cite[Proposition 3.2, p.\ 15]{bargetz2014}. 
Using the weak-strong principle \prettyref{cor:weak_strong_E_0}, we obtain 
that $f_{j}\in\mathcal{E}_{0}(E)$ for all $j\in\Z$ if $E$ is locally complete. Setting 
\[
 \rho\colon\R\to[0,1],\;\rho(x):=1-\cos(\arctan(x))=1-\frac{1}{\sqrt{1+x^2}},
\]
we deduce from the proof and with the notation of \cite[Proposition 2.2, p.\ 1494]{bargetz2015} that 
$e'\circ f_{j}\circ\rho=(\Phi_{2}^{-1}\circ\Phi_{1})(e'\circ f_{j})$ is an element of the Schwartz space 
$\mathcal{S}(\R)$ for each $e'\in E'$. The weak-strong principle \prettyref{cor:weak_strong_CV} c) 
yields that $f_{j}\circ\rho\in\mathcal{S}(\R,E)$ 
if $E$ is locally complete. Hence $(f_{j}\circ\rho)\cdot h_{2n}$ is Pettis-integrable on $\R$  
for every $j\in\Z$ and $n\in\N_{0}$ by \prettyref{prop:Schwartz_pettis_int}  
if $E$ is locally complete where $h_{n}$ is the $n$-th Hermite function.
Therefore the Pettis-integral
\[
 b_{n,j}(f):=\langle f_{j}\circ\rho, h_{2n}\rangle_{\mathcal{L}^{2}}:=\int_{\R}f_{j}(\rho(x))h_{2n}(x)\d x,\;j\in\Z,
 \,n\in\N_{0},
\]
is a well-defined element of $E$ by \prettyref{prop:Schwartz_pettis_int} if $E$ is locally complete. 
By \cite[Theorem 2.1, p.\ 1496--1497]{bargetz2015} (cf.\ \cite[Theorem 3, p.\ 478]{valdivia1981}) the map 
 \[
  \Phi^{\K}\colon \mathcal{O}_{M}(\R)\to s(\N)_{b}'\widehat{\otimes}_{\pi}s(\N),\;
  \Phi^{\K}(f):=(b_{\sigma(n,j)}(f))_{(n,j)\in\N^2},
 \]
is an isomorphism where $\sigma\colon\N^{2}\to\N_{0}\times\Z$ is the enumeration given 
by $\sigma(n,j):=(n-1,(j-1)/2)$ if $j$ is odd, and $\sigma(n,j):=(n-1,-j/2)$ if $j$ is even. 
Here, we have to interpret $\Phi^{\K}(f)$ as an element of $s(\N)_{b}'\widehat{\otimes}_{\pi}s(\N)$ 
by identification of isomorphic spaces. Namely, 
\[
s(\N)_{b}'\widehat{\otimes}_{\pi}s(\N)\cong s(\N)\widehat{\otimes}_{\pi}s(\N)_{b}'
\cong s(\N) \varepsilon s(\N)_{b}'\cong s(\N,s(\N)_{b}')
\]
holds where the first isomorphism is due to the commutativity of $\widehat{\otimes}_{\pi}$, the second due to 
the nuclearity of $s(\N)$ and the last due to \prettyref{cor:sequence_vanish_infty} b) via $S_{s(\N)}$. 
Then we interpret $\Phi^{\K}(f)$ as an element of $s(\N,s(\N)_{b}')$ by means of
\[
j\in\N \longmapsto [a\in s(\N)\mapsto \sum_{n\in\N}a_{n}b_{\sigma(n,j)}]
\]
(see also \eqref{eq:Schauder_coeff_space_multiplier_1} below).

\begin{cor}\label{cor:Schauder_coeff_space_multiplier}
If $E$ is a sequentially complete lcHs, then the map 
 \[
 \Phi^{E}\colon \mathcal{O}_{M}(\R,E)\to s(\N,L_{b}(s(\N),E)),\;\Phi^{E}(f):=(b_{\sigma(n,j)}(f))_{(n,j)\in\N^2},
 \]
is an isomorphism where we interpret $\Phi^{E}(f)$ as an element of $s(\N,L_{b}(s(\N),E))$.
\end{cor}
\begin{proof}
The spaces $\mathcal{O}_{M}(\R)$ and $\mathcal{O}_{M}(\R,E)$ are $\varepsilon$-compatible 
by \prettyref{cor:Schwartz} with the inverse of $S_{\mathcal{O}_{M}(\R)}$ given by the map 
$R^{t}\colon \mathcal{O}_{M}(\R,E)\to\mathcal{O}_{M}(\R)\varepsilon E$, 
$f\mapsto \mathcal{J}^{-1}\circ R^{t}_{f}$, according to \prettyref{thm:full_linearisation}. 
The barrelled nuclear space $\mathcal{O}_{M}(\R)$ has the equicontinuous unconditional Schauder basis 
$(\psi_{\sigma(n,j)})_{(n,j)\in\N^{2}}$ with associated coefficient functionals 
$\delta_{n,j}\circ\Phi^{\K}=b_{\sigma(n,j)}$ given in \cite[Proposition 3.2, p.\ 1499]{bargetz2015}.
Next, we show that $(\Phi^{E},\Phi^{\K})$ 
is a strong, consistent family for $(\mathcal{O}_{M},E)$. 
Let $f\in\mathcal{O}_{M}(\R,E)$. For each $e'\in E'$ and $(n,j)\in\N^2$ we have
\begin{align*}
  \delta_{n,j}\circ\Phi^{\K}(e'\circ f)
&=b_{\sigma(n,j)}(e'\circ f)
 =\int_{\R}(e'\circ f)_{(j-1)/2}(\rho(x))h_{2(n-1)}(x)\d x \\
&=\langle e',\int_{\R}f_{(j-1)/2}(\rho(x))h_{2(n-1)}(x)\d x\rangle
 =\langle e',\delta_{n,j}\circ\Phi^{E}(f)\rangle\\
&=e'(b_{\sigma(n,j)}(f))
\end{align*}
if $j$ is odd since $(f_{(j-1)/2}\circ\rho)\cdot h_{2(n-1)}$ is Pettis-integrable on $\R$. 
The analogous result holds for even $j$ as well. This implies the strength of the family. 
Due to \prettyref{prop:pettis_consistent} with $(T^{E}_{0},T^{\K}_{0})$ given by 
$T^{E}_{0}(f):=(f_{j}\circ\rho)h_{2n}$, $f\in\mathcal{O}_{M}(\R,E)$, and 
$T^{\K}_{0}(f):=(f_{j}\circ\rho)h_{2n}$, $f\in\mathcal{O}_{M}(\R)$, the family 
$(\Phi^{E},\Phi^{\K})$ is consistent.

In order to apply \prettyref{thm:Schauder_coeff_space} we need spaces 
$\ell\mathcal{V}(\N^2)$ and $\ell\mathcal{V}(\N^2,E)$ of  sequences with values in $\K$ and $E$, respectively. 
In addition, the space $\ell\mathcal{V}(\N^2)$ has to be isomorphic 
to $s(\N,s(\N)_{b}')$ so that $\Phi^{\K}\colon\mathcal{O}_{M}(\R)\to s(\N,s(\N)_{b}')\cong\ell\mathcal{V}(\N^2)$ 
becomes the isomorphism we need for \prettyref{thm:Schauder_coeff_space}. We set
\[
\ell\mathcal{V}(\N^2,E):=\{x=(x_{n,j})\in E^{\N^2}\;|\;\forall\;k\in\N,\,B\subset s(\N)\;\text{bounded},\,\alpha\in\mathfrak{A}:\;
\|x\|_{k,B,\alpha}<\infty\}
\]
where 
\[
\|x\|_{k,B,\alpha}:=\sup_{(j,a)\in\omega_{B}}p_{\alpha}(T^{E}(x)(j,a))\nu_{k,B}(j,a)
\]
with $\omega_{B}:=\N\times B$ and $\nu_{k,B}\colon \omega_{B}\to [0,\infty)$, $\nu_{k,B}(j,a):=(1+j^2)^{k/2}$, and 
\[
T^{E}(x)(j,a):=\sum_{n\in\N}a_{n}x_{n,j}.
\]
We claim that the map 
\begin{equation}\label{eq:Schauder_coeff_space_multiplier_1}
T^{E}\colon\ell\mathcal{V}(\N^2,E)\to s(\N,L_{b}(s(\N),E)),\; x\mapsto (T^{E}(x)(j,\cdot))_{j\in\N},
\end{equation}
is an isomorphism. We remark for each $k\in\N$, bounded $B\subset s(\N)$ and $\alpha\in\mathfrak{A}$ that
\[
|T^{E}(x)|_{s(\N),k,(B,\alpha)}=\sup_{j\in\N}\sup_{a\in B}p_{\alpha}(T^{E}(x)(j,a))(1+j^2)^{k/2}=\|x\|_{k,B,\alpha}
\]
for all $x\in\ell\mathcal{V}(\N^2,E)$, implying that $T^{E}$ is an isomorphism into.  
Let $y:=(y_{j})\in s(\N,L_{b}(s(\N),E))$. Then $y_{j}\in L_{b}(s(\N),E)$ for $j\in\N$ and 
we set $x_{n,j}:=y_{j}(e_{n})$ for $n\in\N$ where $e_{n}$ is the $n$-th unit sequence in $s(\N)$. We note 
that with $x:=(x_{n,j})_{(n,j)\in\N^2}$ 
\[
T^{E}(x)(j,a)=\sum_{n\in\N}a_{n}x_{n,j}=\sum_{n\in\N}a_{n}y_{j}(e_{n})=y_{j}(\sum_{n\in\N}a_{n}e_{n})=y_{j}(a)
\]
holds for all $j\in\N$ and $a:=(a_{n})\in s(\N)$ since $(e_{n})$ is a Schauder basis of $s(\N)$ 
with associated coefficient functionals $a\mapsto a_{n}$. It follows that $x\in\ell\mathcal{V}(\N^2,E)$ 
and the surjectivity of $T^{E}$. 

The next step is to prove that $\ell\mathcal{V}(\N^2)$ and $\ell\mathcal{V}(\N^2,E)$ 
are $\varepsilon$-into-compatible. 
Due to \prettyref{thm:linearisation} we only need to show that $(T^{E},T^{\K})$ is a consistent generator 
for $(\ell\mathcal{V},E)$. Let $u\in\ell\mathcal{V}(\N^2)\varepsilon E$. Then 
\begin{equation}\label{eq:Schauder_coeff_space_multiplier_2}
 \sum_{n=1}^{m}a_{n}S_{\ell\mathcal{V}(\N^{2})}(u)(j,n)=\sum_{n=1}^{m}a_{n}u(\delta_{j,n})
=u(\sum_{n=1}^{m}a_{n}\delta_{j,n})
\end{equation}
for all $m\in\N$ and $a:=(a_{n})\in s(\N)$. Since 
\[
 (\sum_{n=1}^{m}a_{n}\delta_{j,n})(x)=\sum_{n=1}^{m}a_{n}x_{j,n}\to T^{\K}(x)(j,a)
=T^{\K}_{(j,a)}(x),\quad m\to\infty,
\]
for all $x\in\ell\mathcal{V}(\N^2)$, we deduce that $(\sum_{n=1}^{m}a_{n}\delta_{j,n})_{m}$ converges 
to $T^{\K}_{(j,a)}(x)$ in $\ell\mathcal{V}(\N^2)_{\kappa}'$ by the Banach--Steinhaus theorem, 
which is applicable as $\ell\mathcal{V}(\N^2)\cong s(\N,s(\N)_{b}')\cong\mathcal{O}_{M}(\R)$ is barrelled. 
We conclude that 
\[
 u(T^{\K}_{(j,a)})
=\lim_{m\to\infty}u(\sum_{n=1}^{m}a_{n}\delta_{j,n})
\underset{\eqref{eq:Schauder_coeff_space_multiplier_2}}{=}
 \sum_{n=1}^{\infty}a_{n}S_{\ell\mathcal{V}(\N^{2})}(u)(j,n)
=T^{E}S_{\ell\mathcal{V}(\N^{2})}(u)(j,a)
\]
and thus the consistency of $(T^{E},T^{\K})$ for $(\ell\mathcal{V},E)$.

Furthermore, we clearly have $e'\circ x\in\ell\mathcal{V}(\N^2)$ for all $x\in\ell\mathcal{V}(\N^2,E)$ and the map 
$\Phi\colon \mathcal{O}_{M}(\R)\to s(\N)_{b}'\widehat{\otimes}_{\pi}s(\N)\cong\ell\mathcal{V}(\N^2)$ 
is an isomorphism by \cite[Theorem 2.1, p.\ 1496--1497]{bargetz2015} 
and \eqref{eq:Schauder_coeff_space_multiplier_1}. 
Due to \cite[Chap.\ II, \S4, n$^\circ$4, Th\'{e}or\`{e}me 16, p.\ 131]{Gro} the dual 
$\mathcal{O}_{M}(\R)_{b}'$ is an LF-space and thus $\mathcal{O}_{M}(\R)\cong(\mathcal{O}_{M}(\R)_{b}')_{b}'$ 
is the strong dual of an LF-space by reflexivity and therefore webbed by \cite[Satz 7.25, p.\ 165]{Kaballo}. 
Finally, we can apply \prettyref{thm:Schauder_coeff_space} (ii), yielding our statement. 
\end{proof}

\begin{rem}\label{rem:Schauder_coeff_space_multiplier}
The actual isomorphism in \prettyref{cor:Schauder_coeff_space_multiplier} (without the interpretation) is given by
$\widetilde{\Phi}^{E}:=T^{E}\circ\Phi^{E}$ with $T^{E}$ from \eqref{eq:Schauder_coeff_space_multiplier_1} 
and we have
\[
 \widetilde{\Phi}^{E}=T^{E}\circ\Phi^{E}
=T^{E}\circ S_{\ell\mathcal{V}(\N^{2})}\circ (\Phi^{\K}\varepsilon\id_{E})\circ S^{-1}_{\mathcal{O}_{M}(\R)}.
\]
Furthermore, \prettyref{cor:Schauder_coeff_space_multiplier} is valid for locally complete $E$ as well. Indeed, 
similar to \prettyref{ex:sequence_vanish_infty} we may show that 
$\ell\mathcal{V}(\N^2,E)\cong \ell\mathcal{V}(\N^2)\varepsilon E$ for locally complete $E$. 
In combination with \prettyref{cor:Schwartz} and \prettyref{thm:eps_prod_surj_inj} b) this proves 
\prettyref{cor:Schauder_coeff_space_multiplier} for locally complete $E$ as in 
\prettyref{thm:fourier.rap.dec} b).
\end{rem}

\appendixpage
\noappendicestocpagenum
\addappheadtotoc 
\renewcommand{\thechapter}{\Alph{chapter}}  
\renewcommand{\thesection}{\Alph{chapter}.\arabic{section}}
\begin{appendix}
\chapter[Compactness of closed abs.\ convex hulls \& Pettis-integrals]{Compactness of closed absolutely convex hulls and Pettis-integrals}
\label{app:appendix}
\section[Compactness of closed absolutely convex hulls]{Compactness of closed absolutely convex hulls}
\label{app:clos_abs_conv_compact}
In this section of the appendix we treat the question for which functions 
$f\colon\Omega\to E$, subsets $K\subset\Omega$ and lcHs $E$ sets like $\oacx(\overline{f(K)})$ are compact or sets like 
\[
N_{j,m}(f):=\{T^{E}_{m}(f)(x)\nu_{j,m}(x)\;|\;x\in\omega_{m}\},\quad j\in J,\,m\in M,
\] 
for $f\in\FVE$ are contained in an absolutely convex compact set. This is useful in connection with $\varepsilon$-compatibility 
due to \prettyref{cor:full_linearisation} (iv) and also relevant in connection with the Pettis-integrability of a vector-valued function due to the Mackey--Arens theorem.

We recall that the space of c\`{a}dl\`{a}g functions on a set $\Omega\subset\R$ with values in an lcHs $E$ is defined by 
\[
  D(\Omega,E):=\{f\in E^{\Omega}\;|\;\forall\;x\in\Omega:\;\lim_{w\to x\rlim}f(w)=f(x)\;\text{and}\;
  f(x\llim):=\lim_{w\to x\llim}f(w)\;\text{exists}\}.\footnote{We recall that for $x\in\Omega$ we only demand $\lim_{w\to x\rlim}f(w)=f(x)$ 
  if $x$ is an accumulation point of $[x,\infty)\cap\Omega$, and the existence of the limit $\lim_{w\to x\llim}f(w)$ if $x$ is an 
  accumulation point of $(-\infty,x]\cap\Omega$.}
\]

\begin{prop}\label{prop:cadlag_precomp}
Let $\Omega\subset\R$, $K\subset\Omega$ be compact and $E$ an lcHs. Then $f(K)$ is precompact 
for every $f\in D(\Omega,E)$. If $E$ is quasi-complete, then $\oacx(\overline{f(K)})$ is compact.
\end{prop}
\begin{proof}
Let $f\in D(\Omega,E)$, $\alpha\in \mathfrak{A}$ and $\varepsilon>0$. 
We recall and define
\[
\mathbb{B}_{r}(x)=\{w\in\R\;|\;|w-x|<r\}\quad\text{and}\quad
B_{\varepsilon,\alpha}(y):=\{w\in E\;|\;p_{\alpha}(w-y)<\varepsilon\}
\]
for every $x\in\Omega$, $y\in E$ and $r>0$. 
Let $x\in\Omega$. Then there is $r_{x\llim}>0$ such that $p_{\alpha}(f(w)-f(x\llim))<\varepsilon$ for all
$w\in\mathbb{B}_{r_{x\llim}}(x)\cap(-\infty,x)\cap\Omega$ if $x$ is an accumulation point of $(-\infty,x]\cap\Omega$. 
Further, there is $r_{x\rlim}>0$ such that $p_{\alpha}(f(w)-f(x))<\varepsilon$ 
for all $w\in\mathbb{B}_{r_{x\rlim}}(x)\cap[x,\infty)\cap\Omega$ if $x$ is an accumulation point of $[x,\infty)\cap\Omega$. 
If $x$ is an accumulation point of $(-\infty,x]\cap\Omega$ and $[x,\infty)\cap\Omega$, we choose $r_{x}:=\min(r_{x\llim},r_{x\rlim})$. 
If $x$ is an accumulation point of $(-\infty,x]\cap\Omega$ but not of $[x,\infty)\cap\Omega$, we choose $r_{x}:=r_{x\llim}$. 
If $x$ is an accumulation point of $[x,\infty)\cap\Omega$ but not of $(-\infty,x]\cap\Omega$, we choose $r_{x}:=r_{x\rlim}$. 
If $x$ is neither an accumulation point of $(-\infty,x]\cap\Omega$ nor of $[x,\infty)\cap\Omega$, then there is 
$r_{x}>0$ such that $\mathbb{B}_{r_{x}}(x)\cap\Omega=\{x\}$. 
 
Setting $V_{x}:=\mathbb{B}_{r_{x}}(x)\cap\Omega $, we note that the sets $V_{x}$ are open in $\Omega$ with respect to the topology induced by $\R$ and $K\subset \bigcup_{x\in K} V_{x}$. 
Since $K$ is compact, there are $n\in\N$ and $x_{1},\ldots,x_{n}\in K$ such that 
$K\subset \bigcup_{i=1}^{n} V_{x_{i}}$. W.l.o.g.\ each $x_{i}$ is an accumulation point of 
$(-\infty,x_{i}]\cap\Omega$ and $[x_{i},\infty)\cap\Omega$. Then we have 
$f(w)\in (B_{\varepsilon,\alpha}(f(x_{i}\llim))\cup B_{\varepsilon,\alpha}(f(x_{i})))$ for all $w\in V_{x_{i}}$ and get 
\[
f(K)\subset\bigcup_{i=1}^{n} f(V_{x_{i}})\subset \bigcup_{i=1}^{n}\bigl(B_{\varepsilon,\alpha}(f(x_{i}\llim))\cup B_{\varepsilon,\alpha}(f(x_{i}))\bigr),
\]
which means that $f(K)$ is precompact.

If $E$ is quasi-complete, then the precompact set $f(K)$ is relatively compact 
by \cite[3.5.3 Proposition, p.\ 65]{Jarchow}. Hence $\oacx(\overline{f(K)})$ is compact as 
quasi-complete spaces have ccp.
\end{proof}

For $f\in D(\Omega,E)$ we define the \emph{\gls{jump_function}} $\gls{Delta_astf}(x):=f(x)-f(x\llim)$, $x\in\Omega$, 
where we set $f(x\llim):=0$ if $x$ is not an accumulation point of $(-\infty,x]\cap\Omega$.

\begin{prop}\label{prop:cadlag_difference_precomp}
Let $\Omega\subset\R$, $K\subset\Omega$ be compact and $E$ an lcHs. Then $\Delta_{\ast} f(K)$ is precompact 
for every $f\in D(\Omega,E)$. If $E$ is quasi-complete, then the set $\oacx(\overline{\Delta_{\ast} f(K)})$ 
is compact.
\end{prop}
\begin{proof}
If $K$ is a finite set, then $\Delta_{\ast} f(K)$ is finite, thus compact, and we are done. 
So let us assume that $K$ is not finite. Let $\alpha\in\mathfrak{A}$ and $\varepsilon>0$. 
We define $\Delta_{\varepsilon,\alpha}:=\{x\in K\;|\;p_{\alpha}(\Delta_{\ast} f(x))\geq\varepsilon\}$ and claim that $\Delta_{\alpha,\varepsilon}$ is a finite set. Let us assume the contrary. 
Then there is an infinite sequence $(x_{n})$ in 
$\Delta_{\varepsilon,\alpha}\subset K$. 
Due to the compactness of $K$ there is a subsequence of $(x_{n})$ which converges to some $x\in K$. W.l.o.g.\ 
this subsequence is strictly increasing and we call this 
subsequence again $(x_{n})$. Since $f$ has left limits (in left-accumulation points), for every $n\in\N$, $n\geq 2$, there 
is $w_{n}\in(x_{n-1},x_{n})$ such that $p_{\alpha}(f(x_{n}\llim)-f(w_{n}))\leq \varepsilon/2$ (if $x_{n}$ is not an accumulation 
point of $(-\infty,x_{n}]\cap\Omega$, then there is $w_{n}\in(x_{n-1},x_{n})$ with $w_{n}\notin\Omega$ and we set $f(w_{n}):=0$). 
Hence we have 
\begin{align*}
     p_{\alpha}(f(x_{n})-f(w_{n}))
&\geq p_{\alpha}(f(x_{n})-f(x_{n}\llim))-p_{\alpha}(f(x_{n}\llim)-f(w_{n}))\\
&= p_{\alpha}(\Delta_{\ast} f (x_{n}))-p_{\alpha}(f(x_{n}\llim)-f(w_{n}))
\geq \varepsilon/2
\end{align*}
for all $n\geq 2$. But this is a contradiction because 
\[
\lim_{n\to\infty}f(x_{n})=\lim_{n\to\infty}f(w_{n})=f(x\llim),
\]
which proves our claim. 

Next, we note that
\[
\Delta_{\ast} f(K)
\subset \bigl(B_{\varepsilon,\alpha}(0)\cup \Delta_{\ast} f(\Delta_{\varepsilon,\alpha})\bigr)
\subset \bigcup_{z\in\{0\}\cup \Delta_{\ast} f(\Delta_{\varepsilon, \alpha})}z+B_{\varepsilon,\alpha}(0),
\]
which implies that $\Delta_{\ast} f(K)$ is precompact as 
$\{0\}\cup \Delta_{\ast} f(\Delta_{\varepsilon, \alpha})$ is finite.

If $E$ is quasi-complete, then the precompact set $\Delta_{\ast} f(K)$ is relatively compact 
by \cite[3.5.3 Proposition, p.\ 65]{Jarchow}. Hence $\oacx(\overline{\Delta_{\ast} f(K)})$ is compact as 
quasi-complete spaces have ccp.
\end{proof}

\prettyref{prop:cadlag_precomp} and \prettyref{prop:cadlag_difference_precomp} are known in the case that $\Omega=[0,1]$ and 
$E=\K$ (see the comments after \cite[Chap.\ 3, Sect.\ 14, Lemma 1, p.\ 110]{billingsley1968}) since precompactness is equivalent 
to boundedness if $E=\K$.

\begin{prop}\label{prop:abs_conv_comp_C_0}
Let $\Omega$ be a locally compact topological Hausdorff space
and $f\in\mathcal{C}_{0}(\Omega,E)$. If
\begin{enumerate}
\item[(i)] $E$ is an lcHs with ccp, or 
\item[(ii)] $E$ is an lcHs with metric ccp and $\Omega$ second-countable,
\end{enumerate}
then $\oacx{(f(\Omega))}$ is compact.
\end{prop}
\begin{proof}
Let $\Omega$ be compact, then $f(\Omega)$ is compact in $E$ as $f$ is continuous. 
If $\Omega$ is even second-countable, then 
$\Omega$ is metrisable by \cite[Chap.\ XI, 4.1 Theorem, p.\ 233]{Dugundji1966} and thus $f(\Omega)$ as well 
by \cite[Chap.\ IX, \S2.10, Proposition 17, p.\ 159]{bourbakiII}. 
This yields that $\oacx{(f(\Omega))}$ is compact in both cases.

Let $\Omega$ be non-compact and $\Omega^{\ast}$ denote the one-point compactification of $\Omega$. Since 
$f\in\mathcal{C}_{0}(\Omega,E)$, it has a unique continuous extension $\widehat{f}$ 
to $\Omega^{\ast}$ with $\widehat{f}(\infty)=0$. 
Hence $K:=\widehat{f}(\Omega^{\ast})$ is a compact set in $E$ 
as $\Omega^{\ast}$ is compact and $\widehat{f}$ continuous. 
If $\Omega$ is even second-countable, then $\Omega^{\ast}$ is metrisable by 
\cite[Chap.\ XI, 8.6 Theorem, p.\ 247]{Dugundji1966} and thus $K$ as well 
by \cite[Chap.\ IX, \S2.10, Proposition 17, p.\ 159]{bourbakiII}.
This yields that $\oacx{(K)}$ is compact in both cases and thus the closed subset $\oacx{(f(\Omega))}$, too.
\end{proof}

We note that $\mathcal{C}_{0}(\Omega,E)=\mathcal{C}(\Omega,E)$ if $\Omega$ is compact. For our next proposition we define the space of bounded $\gamma$-H\"older continuous functions, $0<\gamma\leq 1$, 
from a metric space $(\Omega,\d)$ to an lcHs $E$ by 
\[
\gls{Cbgamma}:=\Bigl\{f\in E^{\Omega}\;|\;\forall\alpha\in\mathfrak{A}:\;\sup_{x\in \Omega}p_{\alpha}(f(x))<\infty\;\text{and}\;\sup_{\substack{x,y\in \Omega\\x\neq y}}\frac{p_{\alpha}(f(x)-f(y))}{\d (x,y)^{\gamma}}<\infty\Bigr\}.
\]

\begin{prop}\label{prop:abs_conv_comp_hoelder}
Let $(\Omega,\d)$ be a metric space, 
$E$ a locally complete lcHs and $f\in\mathcal{C}_{b}^{[\gamma]}(\Omega,E)$ for some $0<\gamma\leq 1$. 
If there is $h\colon\Omega\to (0,\infty)$ such that $fh$ is bounded on $\Omega$ and 
with $N:=\{x\in\Omega\;|\;f(x)=0\}$ it holds that 
\[
\forall\;\varepsilon>0\;\exists\;K\subset\Omega\;\text{compact}\;
\forall\;x\in\Omega\setminus (K\cup N) :\;1\leq\varepsilon h(x),
\]
then $\oacx{(f(\Omega))}$ is compact.
\end{prop}
\begin{proof}
Since $f\in\mathcal{C}_{b}^{[\gamma]}(\Omega,E)$, the sets $f(\Omega)$ and 
\[
B_{1}:=\Bigl\{\frac{f(z)-f(t)}{\d (z,t)^{\gamma}}\;|\;z,t\in \Omega,\;z\neq t\Bigr\}
\]
are bounded in $E$. Further, the range $(fh)(\Omega)$ is bounded in $E$ by assumption.
Thus $B:=\oacx(B_{1}\cup f(\Omega)\cup (fh)(\Omega))$ is a closed disk and $E_{B}$ a Banach space 
with the norm $\|x\|_{B}:=\inf\{r>0\;|\;x\in rB\}$, $x\in E_{B}$, as $E$ is locally complete. 
Next, we show that $f(\Omega)$ is precompact in $E_{B}$. Let $V$ be a zero neighbourhood in $E_{B}$. 
Then there is $\varepsilon>0$ such that 
$U_{\varepsilon}:=\{x\in E_{B}\;|\; \|x\|_{B}\leq\varepsilon\}\subset V$. 
Moreover, there is a compact set $K\subset\Omega$ such that $1\leq\varepsilon h(x)$ for all 
$x\in\Omega\setminus (K\cup N)$. 
The map $f\colon \Omega\to E_{B}$ is well-defined and uniformly continuous because 
$\|f(z)-f(t)\|_{B}\leq \d (z,t)^{\gamma}$ for all $z,t\in\Omega$, which follows from $B_{1}\subset B$.
We deduce that $f(K)$ is compact in $E_{B}$. 
We note that 
\[
f(x)=f(x)h(x)\frac{1}{h(x)},\quad x\in \Omega\setminus N,
\] 
which implies that $\|f(x)\|_{B}\leq \frac{1}{h(x)}$ as $(fh)(\Omega)\subset B$. Hence we have 
\[
\|f(x)\|_{B}\leq \frac{1}{h(x)}\leq \varepsilon,\quad x\in \Omega\setminus (K\cup N),
\]
and the estimate $0=\|f(x)\|_{B}\leq \varepsilon$ is still valid for $x\in N$, 
yielding $f(\Omega\setminus K)\subset U_{\varepsilon}$. Since $f(K)$ is compact in $E_{B}$, 
it is also precompact and so there is a finite set $P\subset E_{B}$ such that $f(K)\subset P+V$. 
We derive that 
\[
 f(\Omega)
=\bigl(f(K)\cup f(\Omega\setminus K)\bigr) 
\subset\bigl( (P+V) \cup U_{\varepsilon}\bigr)\subset \bigl((P\cup\{0\})+V\bigr),
\]
which means that $f(\Omega)$ is precompact in $E_{B}$ and thus $\oacx{(f(\Omega))}$ as well 
by \cite[6.7.1 Proposition, p.\ 112]{Jarchow}.
Therefore the set $\oacx{(f(\Omega))}$ is compact in the Banach space $E_{B}$ and 
also compact in the weaker topology of $E$.
\end{proof}

The underlying idea of \prettyref{prop:abs_conv_comp_hoelder} is taken from 
\cite[Lemma 1, Proposition 2, p.\ 354]{Bonet2002}.

\begin{prop}\label{prop:abs_conv_comp_C_1_b}
Let $\Omega\subset\R^{d}$ be an open convex set, $E$ an lcHs over $\K$ and 
$f\colon\Omega\to E$ weakly $C^{1}_{b}$, i.e.\ $e'\circ f\in\mathcal{C}_{b}^{1}(\Omega)$ for each $e'\in E'$. 
Then $f\in\mathcal{C}_{b}^{[1]}(\Omega,E)$.
\end{prop}
\begin{proof} Let $z,t\in\Omega$, $z\neq t$. By the mean value theorem we have 
\[
\frac{|(e'\circ f)(z)-(e'\circ f)(t)|}{|z-t|}\leq C_{d}\max_{1\leq n\leq d}\sup_{x\in\Omega}|(\partial^{e_{n}})^{\K}(e'\circ f)(x)|\leq C_{d}|e'\circ f|_{\mathcal{C}_{b}^{1}(\Omega)}<\infty
\]
for all $e'\in E'$ where $C_{d}:=\sqrt{d}$ if $\K=\R$, and $C_{d}:=2\sqrt{d}$ if $\K=\C$.
It follows from \cite[Mackey's theorem 23.15, p.\ 268]{meisevogt1997} that $f$ is Lipschitz continuous 
and bounded as well, thus $f\in\mathcal{C}_{b}^{[1]}(\Omega,E)$.  
\end{proof}

\begin{prop}\label{prop:vanish_at_infinity_precomp}
Let $\FVE$ be a $\dom$-space, let there be a set $X$, a family $\mathfrak{K}$ of sets and a map 
$\pi\colon\bigcup_{m\in M}\omega_{m} \to X$ such that $\bigcup_{K\in\mathfrak{K}}K\subset X$. 
If $f\in\FVE$ fulfils
\begin{align*}
&\forall\;\varepsilon >0,\, j\in J,\, m\in M,\,\alpha\in\mathfrak{A}\;\exists\;K\in\mathfrak{K}:\\
&(i)\; \sup_{\substack{x\in\omega_{m},\\ \pi(x)\notin K}}p_{\alpha}\bigl(T^{E}_{m}(f)(x)\bigr)\nu_{j,m}(x)<\varepsilon,\\
&(ii)\; N_{\pi\subset K,j,m}(f):=\{T^{E}_{m}(f)(x)\nu_{j,m}(x)\;|\;x\in\omega_{m},\,\pi(x)\in K\}\;\text{is precompact in}\;E,\notag
\end{align*}
then the set $N_{j,m}(f)$ is precompact in $E$ for every $j\in J$ and $m\in M$. 
If $E$ is quasi-complete, then $\oacx(\overline{N_{j,m}(f)})$ is compact.
\end{prop}
\begin{proof}
Let $V$ be a zero neighbourhood in $E$. Then there are $\alpha\in\mathfrak{A}$ and $\varepsilon>0$ such 
that $B_{\varepsilon,\alpha}\subset V$ where $B_{\varepsilon,\alpha}:=\{x\in E\;|\;p_{\alpha}(x)<\varepsilon\}$. 
Let $j\in J$ and $m\in M$. Due to (i) there is $K\in\mathfrak{K}$ such that the set 
\[
\qquad\;\; N_{\pi\nsubset K,j,m}(f):=\{T^{E}_{m}(f)(x)
 \nu_{j,m}(x)\;|\;x\in\omega_{m},\;\pi(x)\notin K\}
\]
is contained in $B_{\varepsilon,\alpha}$. Further, the precompactness of $N_{\pi\subset K,j,m}(f)$ by (ii) 
implies that there exists a finite set $P\subset E$ such that $N_{\pi\subset K,j,m}(f) \subset P+V$. 
Hence we conclude
\begin{align*}
 N_{j,m}(f)&= \bigl( N_{\pi\nsubset K,j,m}(f)\cup N_{\pi\subset K,j,m}(f)\bigr) \\
 &\subset \bigl(B_{\varepsilon,\alpha}\cup (P+V)\bigr)
 \subset\bigl( V\cup (P+V)\bigr)
 =(P\cup\{0\})+V,
\end{align*}
which means that $N_{j,m}(f)$ is precompact. 

The second part of the statement follows from the fact that 
a precompact set in a quasi-complete space is relatively compact 
by \cite[3.5.3 Proposition, p.\ 65]{Jarchow} and that
quasi-complete spaces have ccp.
\end{proof}

The most common case is that $\mathfrak{K}$ consists of the compact subsets of 
$\Omega$ and $\pi$ is a projection on $X:=\Omega$ (see e.g.\ \prettyref{ex:cont_loc_comp}, \prettyref{ex:subspace_bierstedt} and \prettyref{ex:diff_vanish_at_infinity}).
\section[The Pettis-integral]{The Pettis-integral}
\label{app:pettis}
We start with the definition of the Pettis-integral which we use to define 
Fourier transformations of vector-valued functions 
(see \prettyref{prop:Fourier-trafo_Bjoerck}, \prettyref{thm:fourier.rap.dec} and \prettyref{thm:fourier_periodic}) and for 
Riesz--Markov--Kakutani theorems in \prettyref{sect:riesz_markov_kakutani}.

Let $\Sigma$ be a $\sigma$-algebra on a set $X$. A function $\mu\colon\Sigma\to\K$ is 
called \emph{\gls{K_valued_meas}} if $\mu(\varnothing)=0$ and $\mu$ is \emph{countably additive}, 
i.e.\ for any sequence $(A_{n})_{n\in\N}$ of pairwise disjoint sets in $\Sigma$ it holds that
\[
\mu\bigl(\bigcup_{n\in\N}A_{n}\bigr)=\sum_{n\in\N}\mu(A_{n})\in\K .
\]
If $\K=\R$, $\mu$ is also called a \emph{signed measure}, and if $\K=\C$ a \emph{complex measure}.
If $\K$ is replaced by $[0,\infty]$, we say that $\mu$ is a \emph{\gls{pos_meas}}.
For a $\K$-valued measure $\mu$ its \emph{\gls{total_var}} $\gls{mu_abs}$ given by 
\[
|\mu|(A):=\sup\bigl\{\sum_{n\in\N}|\mu(A_{n})|\;|\;A_{n}\in\Sigma, A_{m}\cap A_{n}=\varnothing\;\text{if}\; m\neq n, A=\bigcup_{n\in\N}A_{n}\bigr\},\quad A\in\Sigma,
\] 
is a well-defined positive measure by \cite[6.2 Theorem, p.\ 117]{rudin1970} 
and it is \emph{finite} by \cite[6.4 Theorem, p.\ 118]{rudin1970}, 
i.e.\ $|\mu|(X)<\infty$. Obviously, a $\K$-valued measure $\mu$ is positive if and only if $|\mu|=\mu$.
For a positive measure $\mu$ on $X$ and $1\leq p<\infty$ let 
\[
\mathfrak{L}^{p}(X,\mu):=\{f\colon X\to\K\;\text{measurable}\;|\;
q_{p}(f):=\int_{X}|f(x)|^{p}\d\mu(x)<\infty\}
\]
and define the quotient space of $p$-integrable functions by 
$\gls{L_pXmu}:=\mathfrak{L}^{p}(X,\mu)/\{f\in\mathfrak{L}^{p}(X,\mu)\;|\;q_{p}(f)=0\}$,
which becomes a Banach space if it is equipped with the norm 
$\|f\|_{p}:=\|f\|_{\mathcal{L}^{p}}:=q_{p}(F)^{1/p}$, $f=[F]\in \mathcal{L}^{p}(X,\mu)$. 
From now on we do not distinguish between equivalence classes and 
their representatives anymore.

For a $\K$-valued measure $\mu$ there is a unique $h\in\mathcal{L}^{1}(X,|\mu|)$ with
$\d\mu = h\d|\mu|$ by the Radon--Nikod\'ym theorem (see \cite[6.12 Theorem, p.\ 124]{rudin1970})
and $h$ can be chosen such that $|h|=1$, 
i.e.\ has a representative with modulus equal to $1$. Now, we say that $f\in\mathcal{L}^{p}(X,\mu)$ 
if $f\cdot h\in \mathcal{L}^{p}(X,|\mu|)$. For $f\in\mathcal{L}^{1}(X,\mu)$ we define 
the \emph{integral of} $f$ \emph{on} $X$ \emph{w.r.t.}\ $\mu$ by
\[
\int_{X}f(x)\d\mu(x):=\int_{X}f(x)h(x)\d|\mu|(x).
\]

For a measure space $(X, \Sigma, \mu)$ and $f\colon X\to \K$ we say that 
$f$ is \emph{integrable on} $\Lambda\in\Sigma$ and write $f\in\mathcal{L}^{1}(\Lambda,\mu)$ 
if $\chi_{\Lambda}f\in \mathcal{L}^{1}(X,\mu)$. Then we set 
\[
\int\limits_{\Lambda} f(x)  \d\mu(x):=\int\limits_{X} \chi_{\Lambda}(x)f(x)  \d\mu(x).
\]

\begin{defn}[Pettis-integral]\label{def:integral}
 	Let $(X, \Sigma, \mu)$ be a measure space and $E$ an lcHs. 
 	A function $f\colon X \to E$ is called \emph{\gls{weakly_measurable}} if the function 
 	$e'\circ  f  \colon X\to \K$, $ (e'\circ f )(x) := \langle e' , f(x) \rangle:=e'(f(x)),$ 
 	is measurable for all $e' \in E'$.  
 	A weakly measurable function is said to be \emph{\gls{weakly_integrable}} 
 	if $e' \circ f \in \mathcal{L}^{1}(X,\mu)$. 
 	A function $f\colon X\to E$ is called \emph{\gls{Pettis_integrable}} on $\Lambda\in\Sigma$ 
 	if it is weakly integrable on $\Lambda$ and
 	\[
 	\exists\; e_{\Lambda} \in E \; \forall e' \in E': \langle e' , e_{\Lambda} \rangle 
 	= \int\limits_{\Lambda} \langle e' , f(x) \rangle \d\mu(x). 
 	\]  
 	In this case $e_{\Lambda}$ is unique due to $E$ being Hausdorff and 
 	we set the \emph{\gls{Pettis_integral}}
 	\[
 	 \int\limits_{\Lambda} f(x) \d\mu(x):=e_{\Lambda}.
 	\]	
\end{defn}	 

If we consider the measure space $(X, \mathscr{L}(X), \lambda)$ of Lebesgue measurable sets 
for $X\subset\R^{d}$, we just write $\d x:=\d\lambda(x)$.

\begin{lem}\label{lem:pettis.loc.complete}
 Let $E$ be a locally complete lcHs, $\Omega\subset\R^{d}$ open and $f\colon\Omega\to E$.
 If $f$ is weakly $\mathcal{C}^{1}$, i.e.\ $e'\circ f\in\mathcal{C}^{1}(\Omega)$ 
 for every $e'\in E'$, then $f$ is Pettis-integrable on every compact subset 
 $K\subset\Omega$ with respect to any locally finite positive measure $\mu$ on $\Omega$ and
 \[
  p_{\alpha}\bigl(\int_{K}f(x)\d\mu(x)\bigr)\leq \mu(K)\sup_{x\in K}p_{\alpha}(f(x)),\quad\alpha\in\mathfrak{A}.
 \]
\end{lem}
\begin{proof}
 Let $K\subset\Omega$ be compact and $(\Omega,\Sigma,\mu)$ a measure space with locally finite measure $\mu$, 
 i.e.\ $\Sigma$ contains the Borel $\sigma$-algebra $\mathcal{B}(\Omega)$ on $\Omega$ and 
 for every $x\in\Omega$ there is a neighbourhood $U_{x}\subset\Omega$ of $x$ such that $\mu(U_{x})<\infty$. 
 Since the map $e'\circ f$ is differentiable for every $e'\in E'$, thus Borel-measurable, 
 and $\mathcal{B}(\Omega)\subset \Sigma$, it is measurable. 
 We deduce that $e'\circ f\in \mathcal{L}^{1}(K,\mu)$ for every $e'\in E'$ because locally finite measures 
 are finite on compact sets. Therefore the map 
 \[
  I\colon E'\to \mathbb{K},\;I(e'):=\int_{K}\langle e',f(x)\rangle\d\mu(x)
 \]
 is well-defined and linear. We estimate
 \[
  |I(e')|\leq |\mu(K)| \sup_{x\in f(K)}|e'(x)|\leq \mu(K) \sup_{x\in \oacx(f(K))}|e'(x)|,\quad e'\in E'.
 \]
 Due to $f$ being weakly $\mathcal{C}^{1}$ and \cite[Proposition 2, p.\ 354]{Bonet2002} 
 the absolutely convex set $\oacx(f(K))$ is compact, 
 yielding $I\in(E_{\kappa}')'\cong E$ by the theorem of Mackey--Arens, which means that there is $e_{K}\in E$ such that
 \[
  \langle e',e_{K}\rangle=I(e')=\int_{K}\langle e',f(x)\rangle\d\mu(x),\quad e'\in E'.
 \]
 Hence $f$ is Pettis-integrable on $K$ w.r.t.\ $\mu$.
 For $\alpha\in\mathfrak{A}$ we set $B_{\alpha}:=\{x\in E\;|\;p_{\alpha}(x)<1\}$ and observe that
\begin{align*}
 p_{\alpha}\bigl(\int_{K}f(x)\d\mu(x)\bigr)
 &=\sup_{e'\in B_{\alpha}^{\circ}}\bigl|\langle e',\int_{K}f(x)\d\mu(x)\rangle\bigr|
 =\sup_{e'\in B_{\alpha}^{\circ}}\bigl|\int_{K}e'(f(x))\d\mu(x)\bigr|\\
 &\leq \mu(K)\sup_{e'\in B_{\alpha}^{\circ}}\sup_{x\in K}|e'(f(x))|
 = \mu(K)\sup_{x\in K}p_{\alpha}(f(x))
\end{align*}
where we used \cite[Proposition 22.14, p.\ 256]{meisevogt1997} in the first and last equation 
to get from $p_{\alpha}$ to $\sup_{e'\in B_{\alpha}^{\circ}}$ and back.
\end{proof}

\begin{lem}\label{lem:pettis.seq.complete}
 Let $E$ be a sequentially complete lcHs, $\Omega\subset\R^{d}$ open, 
 $(\Omega,\Sigma,\mu)$ a measure space with locally finite positive measure $\mu$ and $f\colon\Omega\to E$.
 If $f$ is weakly $\mathcal{C}^{1}$ and there are $\psi\in\mathcal{L}^{1}(\Omega,\mu)$ 
 and $g\colon\Omega\to[0,\infty)$ measurable such that $\psi g\geq 1$ and $fg$ is bounded on $\Omega$, 
 then $f$ is Pettis-integrable on $\Omega$ and
 \[
p_{\alpha}\bigl(\int_{\Omega}f(x)\d \mu(x)\bigr)\leq \|\psi\|_{1}\sup_{x\in\Omega}p_{\alpha}(f(x)g(x)),
\quad\alpha\in\mathfrak{A}.
\]
\end{lem}
\begin{proof}
Let $(K_{n})_{n\in\N}$ be a compact exhaustion of $\Omega$.
Due to \prettyref{lem:pettis.loc.complete} the Pettis-integral 
\[
e_{n}:=\int_{K_{n}}f(x)\d\mu(x)
\]
is a well-defined element of $E$ for every $n\in\N$.
Next, we show that $(e_{n})$ is a Cauchy sequence 
in $E$. Let $\alpha\in\mathfrak{A}$, $m\in\N_{0}$ and $k,n\in\N$ with $k>n$. 
We set $B_{\alpha}:=\{x\in E\;|\;p_{\alpha}(x)<1\}$ and $Q_{k,n}:=K_{k}\setminus K_{n}$ 
and note that
\begin{align}\label{eq:pettis.seq.complete}
  p_{\alpha}(e_{k}-e_{n})
&=\sup_{e'\in B_{\alpha}^{\circ}}|e'(e_{k}-e_{n})|
 =\sup_{e'\in B_{\alpha}^{\circ}}\bigl| \int_{Q_{k,n}}e'(f(x))\d\mu(x)\bigr|\notag\\
&\leq \int_{Q_{k,n}}|\psi(x)|\d\mu(x)
      \sup_{e'\in B_{\alpha}^{\circ}}\sup_{x\in \Omega}|e'(f(x)g(x))|\notag\\
&=\int_{Q_{k,n}}|\psi(x)|\d\mu(x)
  \sup_{x\in \Omega}p_{\alpha}(f(x)g(x))
\end{align}
where we used \cite[Proposition 22.14, p.\ 256]{meisevogt1997} to switch from 
$p_{\alpha}$ to $\sup_{e'\in B_{\alpha}^{\circ}}$ and back.
Since $\psi\in\mathcal{L}^{1}(\Omega,\mu)$, we have that $(e_{n})$ is a Cauchy sequence 
in the sequentially complete space $E$. 
Thus $e_{\Omega}:=\lim_{n\to\infty}e_{n}$ exists in $E$ and the dominated convergence theorem implies 
\[
e'(e_{\Omega})=\lim_{n\to\infty}e'(e_{n})
=\lim_{n\to\infty}\int_{K_{n}}e'(f(x)) \d \mu(x)
=\int_{\Omega} e'(f(x))\d \mu(x), \quad e'\in E'.
\]
Hence $f$ is Pettis-integrable on $\Omega$ 
with $\int_{\Omega}f(x)\d \mu(x)= e_{\Omega}$. As in \eqref{eq:pettis.seq.complete} we have 
\[
 p_{\alpha}(e_{n})
 \leq \int_{K_{n}}|\psi(x)|\d\mu(x) 
  \sup_{x\in \Omega}p_{\alpha}(f(x)g(x))
 \leq \|\psi\|_{1}\sup_{x\in\Omega}p_{\alpha}(f(x)g(x)) 
\]
for every $n\in\N$. Letting $n\to\infty$, we derive the estimate in our statement. 
\end{proof}

\begin{rem}\label{rem:pettis.loc.complete}
Let $\mu$ be a $\K$-valued measure and $\Sigma$ contain $\mathcal{B}(\Omega)$. 
Then \prettyref{lem:pettis.loc.complete} is still valid with $\mu(K)$ replaced by $|\mu|(K)$ 
due to the definition of the integral w.r.t.\ a $\K$-valued measure and 
as $|\mu|(K)\leq|\mu|(\Omega)<\infty$. Thus \prettyref{lem:pettis.seq.complete} holds in this case as well.
\end{rem}

The following definition is analogous to the definition of the Pettis-integral.

\begin{defn}[Pettis-summable]\label{def:pettis_summable} 
 	Let $I$ be a non-empty set and $E$ an lcHs. 
 	A family $(f_{i})_{i\in I}$ in $E$ is called \emph{\gls{weakly_summable}} 
 	if $(\langle e' , f_{i} \rangle)_{i\in I} \in \ell^{1}(I,\K)$ for all $e'\in E'$. 
 	A family $(f_{i})_{i\in I}$ in $E$ is called \emph{\gls{Pettis_summable}} 
 	if it is weakly summable and
 	\[
 	\exists\; e_{I} \in E \; \forall e' \in E': \langle e' , e_{I} \rangle 
 	= \sum\limits_{i\in I} \langle e' , f_{i} \rangle. 
 	\]  
 	In this case $e_{I}$ is unique due to $E$ being Hausdorff and we set 
 	\[
 	 \sum\limits_{i\in I}  f_{i} :=e_{I}.
 	\]	
\end{defn}	 

For the elements $f$ of the 
space $D([0,1],E)$ of $E$-valued c\`{a}dl\`{a}g functions on $[0,1]$ 
and their jump functions $\Delta_{\ast} f$ we have the following result.
 
\begin{prop}\label{prop:cadlag_pettis}
Let $E$ be a quasi-complete lcHs, $\mu$ a $\K$-valued Borel measure on $[0,1]$ and $\psi\in\ell^{1}([0,1],\K)$.
Then $f\in D([0,1],E)$ is Pettis-integrable on $[0,1]$ and 
\[
p_{\alpha}\bigl(\int_{[0,1]}f(x)\d \mu(x)\bigr)\leq|\mu|([0,1]) \sup_{x\in[0,1]}p_{\alpha}(f(x)),
\quad\alpha\in\mathfrak{A},
\]
and $(\Delta_{\ast} f)\psi$ is Pettis-summable on $[0,1]$ and 
\[
     p_{\alpha}\bigl(\sum_{x\in[0,1]}(\Delta_{\ast} f)(x)\psi(x)\bigr)
\leq \|\psi\|_{\ell^{1}}\sup_{x\in[0,1]}p_{\alpha}(\Delta_{\ast} f(x)),
\quad\alpha\in\mathfrak{A}.
\]
\end{prop}
\begin{proof}
By \cite[Chap.\ 3, Sect.\ 14, Lemma 1, p.\ 110]{billingsley1968} $e'\circ f$ 
is Borel measurable for every $e'\in E'$ and integrable due to its boundedness on $[0,1]$. 
Thus the map 
\[
I\colon E'\to\K,\;I(e'):=\int_{[0,1]}e'(f(x))\d\mu(x),
\]
is well-defined and linear. 
It follows from \prettyref{prop:cadlag_precomp} that 
$\oacx(\overline{f([0,1])})$ is absolutely convex and compact in $E$.
In combination with the estimate
\begin{align*}
|I(e')|&\leq|\mu|([0,1])
        \sup_{x\in f([0,1])}|e'(x)|
        \leq|\mu|([0,1])\sup_{x\in\oacx(\overline{f([0,1])})}|e'(x)|
\end{align*} 
for every $e'\in E'$ we deduce that $I\in(E_{\kappa}')'\cong E$ by the theorem of Mackey--Arens, 
which implies that there is $e_{[0,1]}\in E$ such that
\[
\langle e', e_{[0,1]}\rangle=I(e')=\int_{[0,1]}e'(f(x))\d\mu(x),\quad e'\in E'.
\]
Thus $f$ is Pettis-integrable on $[0,1]$. 

Since $\psi\in\ell^{1}([0,1])$ and $e'\circ \Delta_{\ast}f$ bounded on $[0,1]$ for every $e'\in E'$, the map 
\[
I_{0}\colon E'\to\K,\;I_{0}(e'):=\sum_{x\in[0,1]}e'((\Delta_{\ast} f)(x)\psi(x)),
\]
is well-defined and linear. Moreover, the set $\oacx(\overline{\Delta_{\ast} f([0,1])})$ 
is absolutely convex and compact 
by \prettyref{prop:cadlag_difference_precomp}. Again, the estimate
\begin{align*}
|I_{0}(e')|&\leq\sum_{x\in[0,1]}|\psi(x)|
        \sup_{x\in \Delta_{\ast} f([0,1])}|e'(x)|
        \leq\|\psi\|_{\ell^{1}}\sup_{x\in\oacx(\overline{\Delta_{\ast} f([0,1])})}|e'(x)|
\end{align*} 
for every $e'\in E'$, implies our statement. 
The remaining estimates are deduced analogously to \prettyref{lem:pettis.loc.complete}.
\end{proof}

\begin{prop}\label{prop:pettis.ccp.to.loc.complete}
Let $E$ be an lcHs, $\Omega$ a topological Hausdorff space 
and $(\Omega,\Sigma,\mu)$ a measure space. 
If $f\colon\Omega\to E$ is weakly integrable and 
there are $\psi\in\mathcal{L}^{1}(\Omega,\mu)$ and $g\colon\Omega\to[0,\infty)$ 
measurable such that $\psi g\geq 1$
and 
\begin{enumerate}
\item[(i)] $E$ has ccp, $\Omega$ is locally compact and $fg\in\mathcal{C}_{0}(\Omega,E)$, or 
\item[(ii)] $E$ has metric ccp, $\Omega$ is locally compact and second-countable, and $fg\in\mathcal{C}_{0}(\Omega,E)$, or
\item[(iii)] $E$ is locally complete, $\Omega$ a metric space, $fg\in\mathcal{C}_{b}^{[\gamma]}(\Omega,E)$ 
for some $0<\gamma\leq 1$ and there is $h\colon\Omega\to (0,\infty)$ such that $fgh$ is bounded on $\Omega$ 
and with $N:=\{x\in\Omega\;|\;f(x)g(x)=0\}$ it holds that
\[
\forall\;\varepsilon>0\;\exists\;K\subset\Omega\;\text{compact}\;
\forall\;x\in\Omega\setminus (K\cup N) :\;1\leq\varepsilon h(x),
\]
\end{enumerate} 
then $f$ is Pettis-integrable on $\Omega$ and 
\[
p_{\alpha}\bigl(\int_{\Omega}f(x)\d \mu(x)\bigr)\leq \|\psi\|_{1}\sup_{x\in\Omega}p_{\alpha}(f(x)g(x)),
\quad\alpha\in\mathfrak{A}.
\]
\end{prop}
\begin{proof}
Since $f$ is weakly integrable, the map 
\[
I\colon E'\to\K,\;I(e'):=\int_{\Omega}e'(f(x))\d\mu(x),
\]
is well-defined and linear.
It follows from \prettyref{prop:abs_conv_comp_C_0} in case (i)-(ii) 
and from \prettyref{prop:abs_conv_comp_hoelder} in case (iii) 
that $\oacx(fg(\Omega))$ is absolutely convex and compact in $E$.  
If $\mu$ is a positive measure, i.e.\ $[0,\infty]$-valued, we observe that
\begin{align*}
|I(e')|&\leq\int_{\Omega}|\psi(x)|\d\mu(x)
        \sup_{x\in fg(\Omega)}|e'(x)|
        \leq \|\psi\|_{1}\sup_{x\in\oacx(fg(\Omega))}|e'(x)|
\end{align*} 
for every $e'\in E'$. 
If $\mu$ is a $\K$-valued measure, then the same estimate holds with $\mu$ replaced by $|\mu|$. 
We deduce from this estimate that $I\in(E_{\kappa}')'\cong E$ by the theorem of Mackey--Arens,
which implies that there is $e_{\Omega}\in E$ such that
\[
\langle e', e_{\Omega}\rangle=I(e')=\int_{\Omega}e'(f(x))\d\mu(x),\quad e'\in E'.
\]
Thus $f$ is Pettis-integrable on $\Omega$. 
The remaining estimate is deduced analogously to \prettyref{lem:pettis.loc.complete}.
\end{proof}

The idea how to prove \prettyref{prop:pettis.ccp.to.loc.complete} (ii) 
for $\Omega=\R^{d}$ is due to 
an anonymous reviewer of \cite{kruse2017} but did not make it into \cite{kruse2017} 
because of page limits. 
\end{appendix}
\printunsrtglossary[type=symbols,style=alttreegroup,title={List of Symbols}]
\printunsrtglossary[type=index,style=mcolindex,title={Index}]
\bibliography{biblio}
\bibliographystyle{plainnat}
\end{document}